\documentclass[11pt,parskip=half,x11names, twoside=false, open=any, a4paper]{scrbook}
\usepackage[utf8]{inputenc}
\usepackage[T1]{fontenc}
\usepackage{lmodern}
\usepackage{amssymb, amsmath, amstext, amsopn, amsthm, amscd, amsxtra, amsfonts, xcolor, empheq}
\usepackage[most]{tcolorbox}
\usepackage[ilines, headsepline]{scrlayer-scrpage}
\setheadwidth[0pt]{textwithmarginpar}
\usepackage{tikz}
\usetikzlibrary{cd}
\usepackage{colonequals}

\usepackage{caption}
\usepackage{subcaption}

\usepackage{enumitem}
\setlist[enumerate]{label=\textup{(\alph*)}}

\usepackage{hyperref}%
\hypersetup{
urlcolor = red,colorlinks = true,
linkcolor = blue,
citecolor = blue,
linktocpage = true,
pdftitle = {Infinite-dimensional geometry},
pdfauthor = {A. Schmeding},
bookmarksopen = true,
bookmarksopenlevel = 1,
unicode = true,
hypertexnames =true
}%
\usepackage{cleveref}
\usepackage{xcolor}
\colorlet{lightcyan}{cyan!40!white}

\usepackage{chngcntr}
\usepackage{stackengine}

\usepackage{makeidx}
\makeindex
\usepackage{ifthen}
\newboolean{firstanswerofthechapter}

\usepackage[lastexercise,answerdelayed]{exercise}
\counterwithin{Exercise}{chapter}
\counterwithin{Answer}{chapter}
\renewcounter{Exercise}[chapter]

\renewcommand{\AnswerName}{Exercise}
\renewcommand{\AnswerHeader}{\ifthenelse{\boolean{firstanswerofthechapter}}%
    {\bigskip\noindent\textcolor{cyan}{\textbf{CHAPTER \thechapter}}\newline\newline%
        \noindent\bfseries{\AnswerName\ \ExerciseHeaderNB}\smallskip}
    {\noindent\bfseries{\AnswerName\ \ExerciseHeaderNB}\smallskip}}
\swapnumbers
\newtheoremstyle{standard}
{16pt} 
{16pt} 
{} 
{} 
{\bfseries}
{} 
{ } 
{{\thmname{#1~}}{\thmnumber{#2.}}\thmnote{~(#3)}} 

\newtheoremstyle{kursiv}
{16pt} 
{16pt} 
{\itshape} 
{} 
{\bfseries}
{} 
{ } 
{{\thmname{#1~}}{\thmnumber{#2.}}\thmnote{~(#3)}} 

\theoremstyle{standard}
\newtheorem{defn} [subsection]{Definition}
\newtheorem{ex} [subsection]{Example}

\newtheorem{rem} [subsection]{Remark}

\newtheorem{setup} [subsection]{}

\theoremstyle{kursiv}
\newtheorem{thm}[subsection]{Theorem}

\newtheorem{prop} [subsection]{Proposition}
\newtheorem{cor} [subsection]{Corollary}
\newtheorem{lem} [subsection]{Lemma}

\newcommand{\cA}{\ensuremath{\mathcal{A}}}

\newcommand{\cG}{\ensuremath{\mathcal{G}}}

\newcommand{\cU}{\ensuremath{\mathcal{U}}}

\newcommand{\Lf}{\ensuremath{\mathbf{L}}}
\newcommand{\bX}{\ensuremath{\mathbf{X}}}
\newcommand{\bY}{\ensuremath{\mathbf{Y}}}
\newcommand{\LB}[1][\cdot \hspace{1pt} , \cdot]{[\hspace{1pt} #1 \hspace{1pt} ]}

\newcommand{\SSS}{\ensuremath{\mathbb{S}}}

\newcommand{\bC}{\ensuremath{\mathbb{C}}}

\newcommand{\R}{\ensuremath{\mathbb{R}}}

\newcommand{\N}{\ensuremath{\mathbb{N}}}

\newcommand{\naive}{na{\"i}ve} 

\DeclareMathOperator{\id}{id}
\DeclareMathOperator{\Diff}{Diff}
\DeclareMathOperator{\Imm}{Imm}
\DeclareMathOperator{\Sub}{Sub}
\DeclareMathOperator{\Emb}{Emb}

\DeclareMathOperator{\Comp}{Comp}
\DeclareMathOperator{\one}{\mathbf{1}}
\DeclareMathOperator{\ev}{\mathrm{ev}}
\DeclareMathOperator{\Evol}{\mathrm{Evol}}
\DeclareMathOperator{\evol}{\mathrm{evol}}
\DeclareMathOperator{\dist}{dist}
\DeclareMathOperator{\src}{\mathbf{s}}
\DeclareMathOperator{\trg}{\mathbf{t}}
\DeclareMathOperator{\divr}{div}
\DeclareMathOperator{\grad}{grad}
\DeclareMathOperator{\pr}{pr}
\DeclareMathOperator{\Ad}{Ad}
\newcommand\opn{\ensuremath{\mathrel{\mathpalette\opncls\circ}}}

\newcommand{\opncls}[2]{%
  \ooalign{$#1\subseteq$\cr
  \hidewidth\raisefix{#1}\hbox{$#1{\stylefix{#1}#2}\mkern2mu$}\cr}}

\def\raisefix#1{%
  \ifx#1\displaystyle
    \raise.39ex
  \else
    \ifx#1\textstyle
      \raise.39ex
    \else
      \ifx#1\scriptstyle
        \raise.275ex
      \else
        \raise.150ex
      \fi
    \fi
  \fi
}
\def\stylefix#1{%
  \ifx#1\displaystyle
    \scriptstyle
  \else
    \ifx#1\textstyle
      \scriptstyle
    \else
      \ifx#1\scriptstyle
        \scriptscriptstyle
      \else
        \scriptscriptstyle
      \fi
    \fi
  \fi
}

\newcommand{\Frechet}{Fr\'{e}chet }
\newcommand{\coloneq}{\colonequals}

\newcommand{\toto}{\ensuremath{\nobreak\rightrightarrows\nobreak}}
\DeclareMathOperator{\Bis}{\ensuremath{Bis}}
\newcommand{\vBis}{\ensuremath{\operatorname{vBis}}}

\makeatletter
\providecommand*{\shuffle}{%
  \mathbin{\mathpalette\shuffle@{}}%
}
\newcommand*{\shuffle@}[2]{%
  \sbox0{$#1\vcenter{}$}%
  \kern .15\ht0 
  \rlap{\vrule height .25\ht0 depth 0pt width 2.5\ht0}%
  \raise.1\ht0\hbox to 2.5\ht0{%
    \vrule height 1.75\ht0 depth -.1\ht0 width .17\ht0 %
    \hfill
    \vrule height 1.75\ht0 depth -.1\ht0 width .17\ht0 %
    \hfill
    \vrule height 1.75\ht0 depth -.1\ht0 width .17\ht0 %
  }%
  \kern .15\ht0 
}
\makeatother
\clearscrheadfoot
\ihead{\headmark}
\ohead{\pagemark}
\addtokomafont{pagenumber}{\bfseries\color{LightBlue4}}
\addtokomafont{pagehead}{\color{LightBlue4}}

\newsavebox{\tempbox}

\newcommand\copyrighttext{%
  \footnotesize This material will be published by Cambridge University Press as \emph{An Introduction to Infinite-Dimensional Differential Geometry} by Alexander Schmeding. 
This pre-publication version is free to view and download for personal use only. Not for re-distribution, re-sale or use in derivative works. \textregistered Alexander Schmeding 2021}
\newcommand\copyrightnotice{%
\begin{tikzpicture}[remember picture,overlay]
\node[anchor=south,yshift=10pt] at (current page.south) {\fbox{\parbox{\dimexpr\textwidth-\fboxsep-\fboxrule\relax}{\copyrighttext}}};
\end{tikzpicture}%
}  
  
  \title{An Introduction to Infinite-Dimensional Differential Geometry}
  \author{Alexander Schmeding}

\begin{document}

\begin{titlepage}
    \vspace*{7em}
    \begin{center}
        {\Huge An introduction to infinite-dimensional differential geometry}
            \\[2em]
        {\scriptsize A. Schmeding}
    \end{center}\vfill
\begin{tcolorbox}[colback=white,colframe=blue!75!black,title=Book \textbf{Draft} Version 1.01]
The present document is the draft version of an introduction to infinite-dimensional differential geometry. It is based on the notes for the course \begin{center}   MAT331: Topics in Analysis (infinite-dimensional geometry)\end{center} at the University of Bergen in Fall 2020.

If you find any errors please send a message to \begin{center}                                       
\url{alexander.schmeding@nord.no}
                                                \end{center}
\copyrighttext
\end{tcolorbox}
    \end{titlepage}
\pagestyle{empty}
\pagenumbering{gobble}

\tableofcontents
\newpage
\pagestyle{scrheadings}
\pagenumbering{arabic}
\setcounter{page}{1}

\section*{Preface}\addcontentsline{toc}{chapter}{Preface} \copyrightnotice
Analysis and geometry on infinite-dimensional spaces is an active research field with many applications in mathematics and physics. Examples for applications arise naturally even when one is interested in problems which on first sight seem genuinely finite-dimensional. You might have heard that it is impossible to accurately predict the weather over a long time. It turns out that this can be explained by studying the curvature of certain infinite-dimensional manifolds, \cite{MR202082}. This example shows that everyday phenomena are intricately linked to geometric objects residing on infinite-dimensional manifolds. In recent years the list of novel applications for infinite-dimensional (differential) geometry has broadened considerably. Among the more surprising novelties are applications in stochastic and rough analysis (rough path theory a la T.~Lyons leads to spaces of paths in infinite-dimensional groups, see \cite{FaH20}) and renormalisation of stochastic partial differential equations via Hairer's regularity structures (see \cite{BaS18}). 

The aim of this book is to give an introduction to infinite-dimensional (differential) geometry. Differential geometry in infinite dimensions comes in many flavours, such as Riemannian and symplectic geometry. One can study Lie groups and their actions as well as K\"{a}hler manifolds, or Finsler geometry. As should already be apparent from this very incomplete list, it is simply not possible to cover a sizeable portion of the diverse topics subsumed under the label (infinite-dimensional) differential geometry. Hence the present book will focus on two main areas: Riemannian geometry and Lie groups. These topics are arguably the most prominent and well-studied of the above list. Moreover, certain basic but important examples in both topics can be approached based on basic results on manifolds of (differentiable) mappings. However, it is important to stress that the focus of this book is introductory in nature. We will usually refrain from discussing results in their most general form if this allows us to avoid lengthy technical discussions. Before we outline the program of the book further, let us highlight two applications of infinite-dimensional geometric structures:

\textbf{Shape analysis}
 Shapes are unparametrised curves in a vector space/on a manifold. Mathematically these are modelled by considering spaces of differentiable functions and quotienting out an appropriate action of a diffeomorphism group (modelling the reparametrisation). Spaces of differentiable functions can only be modelled using infinitely many parameters, whence they are prime examples of infinite-dimensional spaces. Notice that when talking about spaces of differentiable functions, the functions themselves are thought of as points in the infinite-dimensional space. This is a subtle point as a path between two points (aka functions) in this space will be a curve of curves. The aim of shape analysis is now to compare and transform shapes. For the comparison task one is interested to compute shortest distances between shapes. In the language of Riemannian geometry one thus wants to compute geodesics (curves which are locally of shortest length) in these spaces. From the geometric formulation, a plethora of applications ranging from medical imaging to computer graphics can be dealt with. As a sample application we would like to highlight the processing of motion capturing data. \begin{figure}
     \centering
     \begin{subfigure}[h]{0.45\textwidth}
         \centering
         \includegraphics[width=7cm]{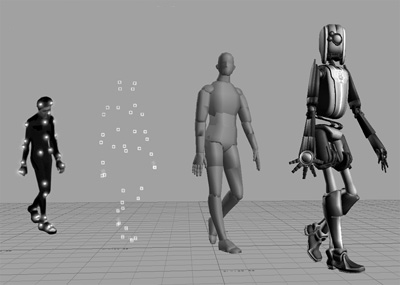}
         \caption{Motion capturing using active markers to animate a figure.}
         \label{mocap1}
     \end{subfigure}
     \hfill
     \begin{subfigure}[h]{0.45\textwidth}
         \centering
         \includegraphics[width=6.5cm]{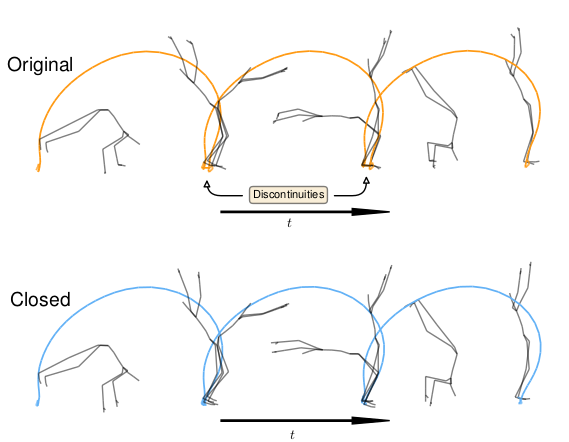}
         \caption{Removing discontinuities from animations via geometric methods}
        \label{fig:mocap}
     \end{subfigure}
     \caption{{\tiny (a) Based on the Wikimedia commons picture Activemarker2, public domain, see \url{https://commons.wikimedia.org/wiki/File:Activemarker2.PNG}, accessed: 07.12.21, (b) reprinted from \cite{CaEaS16} with permission of AIMS.}} 
\end{figure}
The data from the motion capturing process leads to a skeletal wireframe. Since motion capturing is in general expensive, one would like to create algorithms which interpolate between different movements or remove discontinuities occurring in the looping of motions (see \Cref{fig:mocap} for a graphical example of the problem and its (numerical) solution). 

\textbf{Current groups and diffeomorphism groups}
Groups arising from problems related to differential geometry, fluid dynamics and the symmetry of evolution equations are often naturally infinite-dimensional manifolds with smooth group operations. 
Prime examples are the diffeomorphism groups $\Diff (K)$ for $K$ a smooth and compact manifold. We encountered them already as reparametrisation groups in the shape analysis example. However, they are also of independent interest, for example in fluid dynamics: If $K$ is a three dimensional torus, the motion of a particle in a fluid corresponds, under periodic boundary conditions, to a curve in $\Diff(K)$, see \cite{EM70}. This observation led Arnold to his discovery of a general method by which (certain) partial differential equations (PDE) can be lifted to ordinary differential equations on diffeomorphism groups. Typically these PDE arise in the context of hydrodynamics. Many of the applications of this technique, nowadays known as \textbf{Euler-Arnold theory for PDE}, come from geometric hydrodynamics (see \cite{MR2456522} or \cite{modin2019geometric} for an introduction). The already mentioned relation of weather forecasts to infinite-dimensional manifolds arises in a similar fashion. 
To deal with these examples, one frequently needs to leave the theory of differential geometry on Banach manifolds \cite{Lang}, since diffeomorphism groups can not be modelled as Lie groups on Banach spaces. To understand this, consider the canonical action\index{diffeomorphism group!canonical action} of the diffeomorphisms $\Diff(K)$ on a connected compact manifold $K$, i.e. 
$$\alpha\colon \Diff (K) \times K \rightarrow K , \quad (\varphi,k) \mapsto \varphi(k).$$
This action is effective (i.e.\, if $\varphi(k)=\psi(k)$ for all $k$, then $\varphi=\psi$), and with some work one can show that the action $\alpha$ is transitive (i.e., for every pair of points there is a diffeomorphism mapping the first to the second, see \cite{MR1319441}). We would expect that any canonical Lie group structure turns $\alpha$ into a smooth map. However Omori established the following \cite{MR0579603}:
\begin{tcolorbox}[colback=white,colframe=white!50!black,title=Theorem (Omori '78)] \index{Omori's theorem}
If a (connected) Banach-Lie group $G$ acts effectively, transitively and smoothly on a compact manifold, then $G$ must be a finite-dimensional Lie group.
\end{tcolorbox}
Thus, since there is no way for the diffeomorphism group being described only by finitely many parameters, the diffeomorphism group $\Diff (K)$ can not be a Lie group. To treat these examples one thus has to look beyond the realm of Banach spaces.

Another example are the so called current groups $C^\infty (K,G)$ of smooth mappings from a compact manifold to a Lie group $G$. The group structure is here given by pointwise multiplication of functions. In physics current groups describe symmetries in Yang-Mills theories. For example in the theory of electrodynamics, the gauge transformations for Maxwell's equations form a (Lie) subgroup of a current group. Note that the manifold $K$ in physically relevant theories typically models space-time and is non-compact. However, we will restrict ourselves in this book to current groups (and function spaces) on compact manifolds $K$. This allows us to avoid an overly technical discussion while special properties of the main examples remain accessible. As an example, we mention the ``lifting'' of geometric features from finite-dimensional target Lie groups to the infinite-dimensional current groups. This procedure works particularly well in the special case where $K = \SSS^1$ is the unit circle. Then the current group is better known under the name \emph{loop group} $LG \coloneq C^\infty (\SSS^1,G)$, cf.\ \cite{MR900587}. Loop groups and current groups are great examples for a general property of spaces of smooth functions which we shall encounter often: Under suitable assumptions the infinite-dimensional geometry arises from ``lifting the finite-dimensional geometry'' of the target spaces. 

A common topic of the applications and mathematical topics mentioned above is an intimate connection between finite and infinite-dimensional geometry. Indeed while infinite-dimensional differential geometry might seem like an arcane topic from the perspective of the finite-dimensional geometer, there are many intimate and (sometimes) surprising connections between the realms of finite and infinite-dimensional geometry. We already mentioned Euler-Arnold theory and Arnold's insights relating curvature on infinite-dimensional manifolds to weather forecasts. Another example is Duistermaat and Kolk's proof of the third Lie theorem (``\emph{every finite-dimensional Lie algebra is the Lie algebra of a Lie group}'', \cite[Section 1.14]{DaK00}) or Klingenberg's investigation of closed geodesics (cf.\ \cite{MR1330918}). In both cases infinite-dimensional techniques on path spaces were leveraged to solve the finite-dimensional problems. While many of the examples just mentioned can be studied staying in the realm of finite-dimensional geometry, the link to infinite-dimensional geometry should not be disregarded as a purely academic exercise. A case in point might be the theory of rough paths discussed in \Cref{sect:rough}. Here it generally suffices to consider truncated and thus finite-dimensional geometric settings. However, the infinite-dimensional limiting objects hint at deeper geometric insights hidden in the finite-dimensional perspective. It is my view that the infinite-dimensional perspective provides not only a convenient framework for these examples but exhibits important underlying structures and principles. These are worth exploring both for their own sake and for the connected applications.

In the context of the present book we will explore the connection between the finite and infinite-dimensional realm in the context of manifolds of differentiable mappings. These manifolds allow the lifting of geometric structures such as Lie groups and Riemannian metrics to interesting infinite-dimensional structures. For finite orders of differentiability manifolds of differentiable mappings can be modelled on Banach spaces (see \cite{Pal68}). Thus they are within the grasp of differential geometry on Banach spaces, \cite{Lang}. However, since certain constructions (such as the exponential law) become much more technical for finite orders of differentiability, we shall in the present book study exclusively manifolds and spaces of smooth (i.e.~infinitely often differentiable) mappings. This places us immediately outside of the realm of Banach manifolds but enables important constructions such as the Lie group structure for diffeomorphism groups.

Finally, I would like to draw the readers attention to the fact that besides many links and connections between finite and infinite-dimensional geometry, there are quite severe differences between both settings of differential geometry. Many of the strong structural statements from finite-dimensional geometry simply cease to be valid when passing to manifolds modelled on infinite-dimensional spaces (many statements already break down in the more familiar Banach setting). These pitfalls are well known to experts but often surprise beginners in the field. In the spirit of providing an introduction to infinite-dimensional geometry I have taken care to emphasise the differences and illustrate them with examples where possible. For example, the following problems are frequently encountered when passing to the infinite-dimensional setting:
\begin{itemize}
 \item Infinite-dimensional spaces can \textbf{not be locally compact} and they do not support a unique vector topology. Thus arguments building on compactness are not available. (cf.\, \Cref{AppA})
 \item Smooth bump functions do not need to exist (see \Cref{smooth:bumpfun}). Hence the usual local to global arguments become unavailable.
 \item Beyond Banach spaces there is \textbf{no general solution theory for ordinary differential equations} and no general inverse function theorem (\Cref{invfunction}).
 \item Dual spaces become more difficult to handle. As a consequence it is impossible to define differential forms as sections in dual bundles in general (cf.\, \Cref{rem:dualTM}).
 \item Equivalent definitions from finite-dimensional differential geometry are often not equivalent in the infinite-dimensional setting. This happens for example for submersions and immersions, compare \Cref{sect:subm}.  
\end{itemize}
 The structure of the book is as follows. First of all, we shall provide the necessary foundational material for differential geometry in the general infinite-dimensional setting in the first two chapters. Here we emphasise 
\begin{enumerate}
 \item calculus and manifolds on infinite-dimensional spaces beyond the Banach setting, and
 \item manifolds of differentiable functions. 
\end{enumerate}
Beyond Banach spaces there are several choice as to how one can generalise the concept of differentiability. We adopted the so called Bastiani calculus based on iterated directional derivatives. This choice is different from the popular ``convenient calculus''. We shall compare the calculi in later chapters.
Furthermore, we discuss several foundational topics such as locally convex spaces in a series of appendices. The material covered there should allow the reader to grasp most of the basics needed to follow the main part of the book. 
Moreover, the material in the appendices gives some insight into the common problems arising in the passage from finite to infinite-dimensional settings already mentioned. 
While the setting in which we will be working is quite general, we will often not exhibit the most general definitions, results and settings which can possibly be treated in the framework. For example, only manifolds of mappings on a compact manifold are considered here. While the non-compact case is interesting and deserving of attention, the associated theory is much more involved and technical. Our philosophy here is that if the simplified case already admits applications and exhibits the character and problems of the infinite-dimensional setting, we will restrict our attention to the simple case. However, we shall comment on the more general case and provide pointers to the literature. Armed with the knowledge provided by the present book, the reader should be able to quickly learn the more general case should she so desire.

Having dealt with the general setting, we shall then study the main objects of interest in this book:
\begin{enumerate}
 \item (infinite-dimensional) Lie groups and Lie algebras, \Cref{Chap:Liegp} and 
 \item (weak) Riemannian geometry, \Cref{RiemGeo}.
\end{enumerate} 
Based on these building blocks we shall explore several applications of infinite-dimensional geometry in the next chapters. These range from shape analysis to connections with higher geometry (in the guise of Lie groupoids) to Euler-Arnold theory for PDE and the geometry of rough path spaces. These later chapters can be read mostly independently from each other. I have selected the topics for the advanced chapters with a view towards developments within the field over the last years. The broad selection of topics and the introductory nature of the present book prevent an in-depth discussion of these advanced topics. Therefore these chapters gracefully omit many of the more technical and subtle points. However, the reader will be able to gain an impression of the role infinite-dimensional differential geometry plays in these applications. Moreover, there are ample references to the literature which should enable the interested reader to follow up on a topic after perusing the respective chapter.

Before we begin, let us set some conventions which will be in effect for all that is to come. 
\begin{tcolorbox}[colback=white,colframe=blue!75!black,title=Conventions]
 \begin{itemize}
  \item Write $\N \coloneq \{1,2,3,\ldots\}$ for the natural numbers and $\N_0 \coloneq \N \cup \{0\}$.
  \item All topological spaces are assumed to be Hausdorff, i.e.~for every pair of two (non equal) points in the space there are disjoint open neighborhoods of these points.
  \item If nothing else is said, products of topological spaces will always carry the product topology.
 \item 
 If $U \subseteq X$ is an open subset of a topological space, we also write shorter $U \opn X$.
 \end{itemize}
 We shall exclusively work over the real numbers $\R$, all vector spaces are to be understood over the real numbers. However, many results carry over to complex vector spaces. For example it is no problem to define the notion of complex differentiability (see e.g.\, \cite{MR1911979}).
\end{tcolorbox}

\textbf{Acknowledgements}
 We acknowledge support by Nord University and the Trond Mohn foundation in the ``Pure Mathematics in Norway'' program, cf.\ \url{https://www.puremath.no/} for more information.
 In particular I am thankful for Nord Universities commitment to make the book available in open access format.
 
 I have profited from an early draft of the upcoming \cite{GNprep} by H.~Gl\"ockner and K.-H.~Neeb on infinite-dimensional Lie groups. Their lucid presentation of calculus directly influenced \Cref{sect:calculus}. Material from the thesis by A.~Mykleblust (``Limits of Banach spaces'', UiB 2020, advised by A.~Schmeding) was included in \Cref{AppA}.
 
 I am indebted to Martin Bauer, Rafael Dahmen, Gard O.~Helle, Cy Maor, Klas Modin, Torstein Nilssen and David M.~Roberts for many useful suggestions and discussions while I was writing this book. Further thanks go to Rafael Dahmen, Klas Modin, Torstein Nilssen, David Prinz and Nikolas Tapia for reading the manuscript and suggesting many improvements.
 Finally, I thank the participants of the course MAT331 and Roberto Rubio who brought many errors in previous versions to my attention. 
 
\chapter*{Recommended further reading}\addcontentsline{toc}{chapter}{Recommended further reading}
 As mentioned in the introduction, infinite-dimensional differential geometry is a vast topic which I could not hope to cover in this book. Thus the book concentrates on an introduction to infinite-dimensional Lie groups, (weak) Riemannian metrics and their interplay. As further reading on the topics of this book I would like to recommend the following presentations, which have also influenced the presentation in the present book: 
 \begin{itemize}
  \item \textbf{(infinite-dimensional) Lie theory} \cite{NeeMon} lecture notes, \cite{Neeb06} extensive survey, and once it becomes available \cite{GNprep},
  \item \textbf{Manifolds of mappings} \cite{WocInf}, on non-compact domains \cite{Mic},
  \item \textbf{(weak) Riemannian geometry}  \cite{BruL2,BruSC}.
 \end{itemize}
 Furthermore, there was a host of topics which could not be covered in this book. Albeit there is again no hope that I may do the vast body of work on these topics justice, I would like to mention a few works which are either introductory or exhibit new phenomena which are genuinely infinite-dimensional in nature:
 \begin{itemize}
  \item \textbf{Symplectic geometry} \cite{MR960687} (Banach setting) and \cite[Section 48]{KM97}.
  \item \textbf{Sub-Riemannian geometry} \cite{zbMATH06551456,zbMATH05649883}. 
  \item \textbf{K\"{a}hler geometry} \cite{Serg20}.
  \item \textbf{Poisson geometry} \cite{BaGaT18}.
  \item \textbf{Finsler geometry} \cite{Laro19} on diffeomorphism groups, \cite{EaP2} in the context of bounded Fr\'{e}chet geometry.
 \end{itemize}
\chapter{Calculus in locally convex spaces}\label{sect:calculus} \copyrightnotice

It is well known that \textquotedblleft multidimensional calculus\textquotedblright, aka \textquotedblleft\Frechet calculus\textquotedblright, carries over to the realm of Banach spaces and Banach manifolds (see e.g.\ \cite{Lang}). 
As we have seen in the introduction, Banach spaces are often not sufficient for our purposes. To generalise derivatives we will, as a minimum, need vector spaces with an amenable topology (which need not be induced by a norm).

\begin{defn}\label{defn:tvs}
 Consider a vector space $E$. A topology $\mathcal{T}$ on $E$ making addition $+\colon E \times E \rightarrow E$ and scalar multiplication $\cdot \colon \R \times E \rightarrow E$ continuous is called a \emph{vector topology}\index{vector topology} (where $\R$ carries the usual norm topology). We then say that $(E,\mathcal{T})$ (or $E$ for short) is a \emph{topological vector space} (or \emph{TVS} for short)\index{space!topological vector space}. 
\end{defn}

\begin{ex}\label{ex:TVS}
\begin{enumerate}
 \item Every normed space and in particular every finite dimensional vector space is a topological vector space. 
 \item For a more interesting example, fix $0 < p < 1$. Two measurable functions $\gamma,\eta \colon [0,1] \rightarrow \R$ are equivalent $\gamma \sim \eta$ if and only if $\int_0^1 |\gamma(s)-\eta(s)| \mathrm{d}s =0$. Denote by $L^p [0,1]$ the vector space of all equivalence classes $[\gamma]$ of functions such that $\int_0^1 |\gamma(s)|^p \mathrm{d}s < \infty$. Topologise $L^p[0,1]$ via the metric topology induced by 
 $$d([\gamma],[\eta]) \coloneq \int_0^1|\gamma(s)-\eta(s)|^p\mathrm{d} s.$$
 In a metric space, we can test continuity of the vector space operations using sequences. For this, pick $\lambda_n \rightarrow \lambda \in \R$ and $[\gamma_n] \rightarrow [\gamma], [\eta_n] \rightarrow [\eta]$ (with respect to $d$) and use the triangle inequality to obtain: 
 $$d(\lambda_n[\gamma_n]+[\eta_n], \lambda[\gamma]+[\eta]) \leq |\lambda_n-\lambda|^p d([\gamma_n], [0]) + |\lambda|^p d([\gamma_n],[\gamma])+ d([\eta_n],[\eta]).$$
 This shows that the vector space operations are continuous, i.e.\ $L^p[0,1]$ is a TVS.
\end{enumerate}
\end{ex}

In topological vector spaces differentiable curves can be defined as follows.

\begin{defn}\label{defn: smoothcurve}
 Let $E$ be a topological vector space. A continuous mapping $\gamma \colon I \rightarrow E$ from a non-degenerate interval\footnote{i.e.\ $I$ has more then one point. In the following we will always assume this when talking about intervals.} $I \subseteq \R$, is called $C^0$-curve. A $C^0$-curve is called $C^1$-\emph{curve} if the limit 
 $$\gamma' (s) \coloneq \lim_{t \rightarrow 0} \frac{1}{t} (\gamma(s+t)-\gamma(s))$$
 exists for all $s \in I^\circ$ (interior of $I$) and extends to a continuous map $\frac{\mathrm{d}}{\mathrm{d}t}\gamma\coloneq \gamma' \colon I \rightarrow E, s \mapsto \gamma' (s)$. Recursively for $k \in \N$, we call $\gamma$ a $C^k$-\emph{curve}\index{curve!smooth} if $\gamma$ is a $C^{k-1}$-curve and $\frac{\mathrm{d}^{k-1}}{\mathrm{d}t^{k-1}}\gamma$ is a $C^1$-curve. Then $\frac{\mathrm{d}^{k}}{\mathrm{d}t^{k}}\gamma \coloneq \left(\frac{\mathrm{d}^{k-1}}{\mathrm{d}t^{k-1}}\gamma\right)'$.   
 If $\gamma$ is a $C^k$-curve for every $k \in \N_0$, we also say that $\gamma$ is \emph{smooth} or of class $C^\infty$.
\end{defn}

Unfortunately, calculus on topological vector spaces is in general ill behaved. The next exercise shows that derivatives may fail to give us meaningful information:

\begin{Exercise}\label{Ex:count0}
 Given $0 < p < 1$ we let $L^p [0,1]$ be the topological vector space from \Cref{ex:TVS} (b). Recall that the topology on $L^p[0,1]$ is induced by the metric $d([\gamma],[\eta]) \coloneq \int_0^1|\gamma(s)-\eta(s)|^p\mathrm{d} s$. For a set $A \subseteq [0,1]$ write $\mathbf{1}_A$ for the characteristic function and define 
 $$\beta \colon [0,1] \rightarrow L^p [0,1] , \beta (t):=[\mathbf{1}_{[0,t[}].$$
 Show that $\beta$ is an injective $C^1$-curve with $\beta'(t) = 0, \forall t \in [0,1]$.
\end{Exercise}

\setboolean{firstanswerofthechapter}{true}
\begin{Answer}[number={\ref{Ex:count0}}] 
 \emph{Let $0<p<1$. We prove that the curve $\beta \colon [0,1] \rightarrow L^p [0,1] , \beta (t):=[\mathbf{1}_{[0,t[}].$ is an injective $C^1$-curve with $\beta'(t) = 0, \forall t \in [0,1]$. }
 \\[.15em]
 
 Let us show that for every $x \in ]0,1[$ the derivative of $\beta$ vanishes. Then the claim follows by continuity also for the boundary points. Consider for $h$ small the differential quotient $[h^{-1} (\beta (x+h)-\beta(x))] = h^{-1}[\mathbf{1}_{[x,x+h[}]$ converges to $0$ with respect to the $L^p$-metric: 
 $$d([h^{-1} (\beta (x+h)-\beta(x))],[0]) = \int_0^1|h^{-1}\mathbf{1}_{[x,x+h[}(s)|^p\mathrm{d} s = |h|^{-p}\int_0^1\mathbf{1}_{[x,x+h[}(s)\mathrm{d} s = |h|^{1-p}.$$
 Taking the limit $h \rightarrow 0$, we see that the derivative must be $0$ and in particular $\beta$ is a $C^1$-function.  Let now $x <y$, then $\beta(y)-\beta(x) = [\mathbf{1}_{[x,y[}\neq [0]$, so $\beta$ is injective.
\end{Answer}
\setboolean{firstanswerofthechapter}{false}

Obviously we would like to avoid this defect, whence we have to strengthen the assumptions on our vector spaces.

\section{Curves in locally convex spaces}

Calculus in topological vector spaces exhibits pathologies which can be avoided by strengthening the requirements on the underlying space. This leads to locally convex spaces, whose topology is induced by so called seminorms. See also \Cref{AppA} for more information on locally convex spaces.
\begin{defn}
 Let $E$ be a vector space. A map $p \colon E \rightarrow [0,\infty[$ is called a \emph{seminorm}\index{seminorm} if it satisfies the following \begin{enumerate}
  \item $p(\lambda x) = |\lambda |p(x), \forall \lambda \in \R , x \in E$                                                                                                                                \item $p(x+y) \leq p(x)+p(y)$
 \end{enumerate}
\end{defn}
 Note that contrary to the definition of a norm, we did not require that $p(x)=0$ if and only if $x=0$. The next definition uses the notion of an initial topology, which we recall for the readers convenience in \Cref{app:topo}.

\begin{defn}\label{defn:lcvx}
 A topological vector space $(E,\mathcal{T})$ is called \emph{locally convex space}\index{space!locally convex} if there is a family $\{p_i \colon E \rightarrow [0,\infty[ \mid i \in I\}$ of continuous seminorms for some index set $I$, such that 
     \begin{enumerate}
     \item $\mathcal{T}$ is the initial topology with respect to the canonical projections \\ $\{q_i \colon E \rightarrow E/p_i^{-1}(0)\}_{i \in I}$ onto the normed spaces $E/p_i^{-1}(0)$.
     \item if $x \in E$ with $p_i (x) = 0$ for all $i \in I$, then $x = 0$. Thus the seminorms separate the points, i.e.\ $\mathcal{T}$ has the Hausdorff property.\footnote{Some authors do not require separation of points, whence our locally convex spaces are Hausdorff locally convex spaces in their terminology.}
    \end{enumerate}
  We then say that the topology $\mathcal{T}$ is \emph{generated by the family of seminorms}\index{vector topology!generated by seminorms} $\{p_i\}_{i \in I}$ and call this family a \emph{generating family of seminorms}.\index{seminorm!generating family of}
  Usually we suppress $\mathcal{T}$ and write $(E, \{p_i\}_{i\in I})$ or simply $E$ instead of $(E,\mathcal{T})$. 
  \\
  \textbf{Alternative to (a)}: We will see in \Cref{AppA} that equivalent to (a), we can define $\mathcal{T}$ to be the unique vector topology determined by the basis of $0$-neighbourhoods given by (finite) intersections of the balls $B_{i,\varepsilon} (0) =\{x \in E \mid p_i(x)< \varepsilon\}$, where $p_i$ runs through a generating family of seminorms. These balls are all convex, thus justifying the name ``locally convex space''. 
  
   A locally convex space $(E, \{p_i\}_{i\in \N})$ with a countable systems of seminorms is \emph{metrisable}\index{space!metrisable} (i.e.\ its topology is induced by a metric, see Exercise \ref{EX11} 1.) and if $E$ is complete, it is called \emph{Fr\'{e}chet space}.\index{Fr\'{e}chet space}
\end{defn}

\begin{ex}\label{example:lcs}
 \begin{enumerate}
  \item Every normed space $(E,\lVert \cdot \rVert)$ is a locally convex space, where the family of seminorms consists only of the norm $\lVert \cdot \rVert$. 
  \item Consider the space $C^\infty ([0,1], \R)$ of all smooth functions from the interval $[0,1]$ to $\R$ (with pointwise addition and scalar multiplication). This space is not naturally a normed space.\footnote{For any normed topology, the differential operator $D \colon C^\infty ([0,1],\R) \rightarrow C^\infty ([0,1],\R), D(f)=f'$ must be discontinuous (which is certainly undesirable). To see this, recall that a continuous linear map on a normed space has bounded spectrum, but $D$ has arbitrarily large eigenvalues (consider $f_n(t)\coloneq \exp (n t), n\in \N$).} We define a family of seminorms on it via 
  $$\lVert f \rVert_n \coloneq \sup_{0\leq k \leq n} \left\lVert \frac{\mathrm{d}^k}{\mathrm{d}t^k}f\right\rVert_\infty = \sup_{0\leq k \leq n} \sup_{t \in [0,1]} \left|\frac{\mathrm{d}^k}{\mathrm{d}t^k}f(t)\right|, n\in \N_0.$$
  The topology generated by the seminorms is called the \emph{compact-open $C^\infty$-topology}\index{compact open $C^\infty$-topology} and turns $C^\infty([0,1],\R)$ into a locally convex space, which is even a \Frechet space (Exercise \ref{EX11} 2.). 
 \end{enumerate}
\end{ex}

Locally convex spaces have many good properties, for example they admit enough continuous linear functions to separate the points, i.e.\ the following holds. 

\begin{thm}[{Hahn-Banach, \cite[Proposition 22.12]{MaV97}}] \label{thm:HahnBanach} \index{Hahn-Banach theorem}
For a locally convex space $E$ the continuous linear functionals separate the points, i.e.\ for each pair $x,y \in E$ there exists a continuous linear $\lambda \colon E \rightarrow \R$ such that $\lambda(x) \neq \lambda (y)$. 
\end{thm}

\begin{defn}
 Let $E$ be a locally convex space, then we denote by $E'=L(E,\R)$ the continuous linear maps from $E$ to $\R$. The space $E'$ is the so called \emph{dual space} of $E$,\index{dual space}. There are several ways to turn $E'$ into a locally convex space, \cite[p.63 f.]{Rudin}, but in general we will not need a topology beyond the special case if $E$ is a Banach space and $E'$ carries the operator norm topology. 
\end{defn}

With the help of the Hahn-Banach theorem we can avoid the pathologies observed for topological vector spaces. To this end we need the notion of a weak integral

\begin{defn}
 Let $\gamma \colon I \rightarrow E$ be a $C^0$-curve in a locally convex space $E$ and $a,b \in I$.  If there exists $z \in E$ such that 
 $$\lambda (z) = \int_a^b \lambda(\gamma(t)) \mathrm{d}t, \forall \lambda \in E',$$
 then $z \in E$ is called the \emph{weak integral of}\index{weak integral} $\gamma$ from $a$ to $b$ and denoted $\int_a^b \gamma(t) \mathrm{d}t \coloneq z$.
\end{defn}

Note that weak integrals (if they exist) are uniquely determined due to the Hahn-Banach theorem. 

\begin{prop}[First part of the fundamental theorem of calculus]\label{prop:funthm} \index{Fundamental theorem of calculus}
 Let $\gamma \colon I \rightarrow E$ be a $C^1$-curve in a locally convex space $E$ and $a,b \in I$ then 
 $$\gamma(b) - \gamma(a) = \int_a^b \gamma'(t) \mathrm{d}t.$$
\end{prop}

\begin{proof}
 Let $\lambda \in E'$. It is easy to see that $\lambda \circ \gamma \colon I \rightarrow \R$ is a $C^1$-curve with $(\lambda \circ \gamma)' = \lambda \circ (\gamma')$. The standard fundamental theorem of Calculus yields
 $$\lambda(\gamma(b)-\gamma(a))=\lambda (\gamma(b))-\lambda(\gamma(a)) = \int_a^b (\lambda\circ \gamma)'(s) \mathrm{d}s = \int_a^b\lambda(\gamma'(s))\mathrm{d}s.$$
 Hence $z = \gamma(b)-\gamma(a)$ satisfies the defining property of the weak integral.
\end{proof}

Note that \Cref{prop:funthm} implies that $L^p[0,1]$ can not be a locally convex space for $0<p<1$, cf.\ \cite[1.47]{Rudin} for an elementary proof of this fact.
\begin{rem}
 Also the second part of the fundamental theorem of calculus is true in our setting. Thus if $\gamma \colon I \rightarrow E$ is a $C^0$-curve, $a \in I$ and the weak integral
 $$\eta(t) \coloneq \int_a^t \gamma(s) \mathrm{d}s$$
 exists for all $t \in I$. Then $\eta \colon I \rightarrow E$ is a $C^1$-curve in $E$, and $\eta'=\gamma$.
 
 The proof however needs more techniques based on convex sets which we do not wish to go into (cf.\ \cite{GNprep}).
\end{rem}

The reader may wonder now, when do weak integrals of curves exist? One can prove that weak integrals of continuous curves always exist \emph{in the completion of a locally convex space}. The key point is that the integrals can be defined using Riemann sums, but these do not necessarily converge in the space itself, \cite[Lemma 2.5]{KM97}. Thus weak integrals exist for suitably complete spaces. To avoid getting bogged down with the discussion of completeness properties, we define:

\begin{defn}\label{Mackeycomplete}
 A locally convex space $E$ is \emph{Mackey complete}\index{Mackey complete} if for each smooth curve $\gamma \colon [0,1] \rightarrow E$ there exists a smooth curve $\eta\colon [0,1] \rightarrow E$ with $\eta' = \gamma$.
\end{defn}

Due to the fundamental theorem of calculus this implies that $\eta(s) -\eta(0) = \int_0^s \gamma(t) \mathrm{d}t$. Thus the weak integral of smooth curves exists in Mackey complete spaces. 

\begin{rem}\label{Mackey-remark}
 Mackey completeness is a very weak completeness condition, in particular, sequential completeness (i.e. Cauchy sequences converge in the space) implies Mackey completeness. This is evident from the alternative characterisation of Mackey completeness using sequences, see \Cref{def:TVS}. 
 Note however that it is not entirely trivial to find examples of Mackey complete but not sequentially complete spaces. We mention here that the space $\mathcal{K}(E,F)$ of compact operators between two (infinite-dimensional) Banach spaces $E,F$ with the strong operator topology is not sequentially complete but Mackey complete (see \cite{Voi92}).
 
 However, in metrisable locally convex spaces (e.g.\ in normed spaces) Mackey completeness is equivalent to completeness, cf.\ \cite[10.1.4]{Jar81}. We refer to  \cite[I.2]{KM97} for more information on Mackey completeness. In particular, \cite[Theorem 2.14]{KM97} shows that integrals exist for $C^1$-curves in Mackey complete spaces. 
\end{rem}

So far we have defined differentiable curves with values in locally convex spaces. The next step is to consider differentiable mappings between locally convex spaces. Here a different notion of calculus is needed. It turns out that (even on \Frechet spaces) there are many generalisations of \Frechet calculus (see \cite{MR0440592}) without a uniquely preferable choice. In the next section we present a simple and versatile notion called Bastiani calculus. Another popular approach to calculus in locally convex spaces, the so called convenient calculus, is discussed in \Cref{convenient:calculus}.

\begin{Exercise}\label{EX11}
        \vspace{-\baselineskip}
        \Question Let $(E,\{p_n\}_{n \in \N})$ be a locally convex space whose topology is generated by a countable set of seminorms. Prove that $d(x,y) := \sum_{n\in \N} 2^{-n} \tfrac{p_n (x-y)}{p_n(x-y) +1}$ is a metric on $E$ and the metric topology coincides with the locally convex topology.  
        \Question Consider $C^\infty ([0,1],\R)$ with the compact open $C^\infty$-topology (see \Cref{example:lcs}). 
        \subQuestion Show that a sequence $(f_k)_{k\in \N}$ converges to $f$ in this topology if and only if for all $\ell \in \N_0$ $(\tfrac{\mathrm{d}^\ell}{\mathrm{d}t^\ell} f_k)_k$ converges uniformly to $\tfrac{\mathrm{d}^\ell}{\mathrm{d}t^\ell} f$.\\
        {\tiny \textbf{Hint:} The uniform limit of a sequence of continuous functions is continuous. If a function sequence and the sequence of (first) derivatives converges, the limit of the sequence is differentiable.}
        \subQuestion Deduce that every Cauchy-sequence in the compact open $C^\infty$-topology converges to a smooth function. As $C^\infty ([0,1],\R)$ is a metric space by Exercise \ref{EX11} 1., this implies that the space is complete, i.e.\ a \Frechet space.
        \subQuestion Show that the differential operator $D \colon C^\infty ([0,1],\R) \rightarrow C^\infty ([0,1],\R), f\mapsto f'$ is continuous linear. {\tiny \textbf{Hint:} \Cref{lem:linmapzero}.}
        \Question Let $(E,\{p_i\}_I)$ be a locally convex space whose topology is generated by a \emph{finite} set of seminorms. Show that $p(x)=\max_{i\in I} p_i(x)$ defines a norm on $E$ which induces the same topology as the family $\{p_i\}$. In this case we call $E$ normable.
        \Question Establish the following properties of weak integrals
        \begin{enumerate}
         \item If the weak integrals of $\gamma, \eta \colon I \rightarrow E$ from $a$ to $b$ exist and $s \in \R$, then also the weak integral of $\gamma + s \eta$ exists and $\int_a^b (\gamma(t)+s\eta(t))\mathrm{d}t = \int_a^b\gamma(t)\mathrm{d}t + s \int_a^b \eta(t) \mathrm{d}t$
         \item If $\gamma \colon I \rightarrow E$ is constant, $\gamma(t) \equiv K$, then $\int^b_a \gamma(t) \mathrm{d}t$ exists and equals $(b-a)K$.
         \item $\int_a^c \gamma(t) \mathrm{d}t = \int_a^b \gamma(t)\mathrm{d}t + \int_b^c\gamma(t)\mathrm{d}t$ (if the integrals exist)
        \end{enumerate}
        \Question Let $\gamma \colon I \rightarrow E$ be a $C^k$-curve ($k \in \N$) and $\lambda \colon E \rightarrow F$ be continuous linear for $E,F$ locally convex. Show that $\lambda \circ \gamma$ is $C^k$ such that $\frac{\mathrm{d}^\ell}{\mathrm{d}t^\ell} (\lambda \circ \gamma) = \lambda \circ \left(\frac{\mathrm{d}^\ell}{\mathrm{d}t^\ell} \gamma \right),\ 1 \leq \ell \leq k$.
        \Question Endow a vector space $E$ with a topology $\mathcal{T}$ generated by seminorms as in \Cref{defn:lcvx}. Show that $(E,\mathcal{T})$ is a topological vector space (whence requiring that locally convex spaces are topological vector spaces was superfluous).
\end{Exercise}

\begin{Answer}[number={\ref{EX11} 2 c)}] 
 \emph{We will show that $D \colon C^\infty ([0,1],\R) \rightarrow C^\infty ([0,1],\R), c \mapsto c^\prime$ is continuous linear.}
 \\[.15em]
 
 Clearly the differential operator is linear with respect to pointwise addition of functions. Let now $0$ be the constant $0$-function. Then linearity of $D$ together with \Cref{lem:linmapzero} implies that $D$ will be continuous if the preimage of every $0$-neighbourhood $U \subseteq C^\infty ([0,1],\R)$ is a $0$-neighbourhood. Thus we pick an open $0$-neighbourhood $U$. Shrinking $U$ we may assume that $U$ is a ball $B_r^n (0)$ of radius $r>0$ for the seminorms $\lVert \cdot \rVert_n$ where we have chosen suitable $n\in \N_0$ (and have exploited that these seminorms form a fundamental system by \Cref{ex:co_seminorm}). 
 Now as the $k$th derivative of a function coincides with the $k-1$st derivative of its derivative we observe that$D(B_{n+1}^r (0)) \subseteq B_n^r(0)$. In other words $B_{n+1}^r(0) \subseteq D^{-1}(U)$ and the preimage is a $0$-neighbourhood. We deduce that $D$ is continuous.
\end{Answer}

\section{Bastiani calculus}

Bastiani calculus\index{Bastiani calculus} (also called Keller's $C^k_c$-theory, \cite{MR0440592}), introduced in \cite{bastiani}, builds a calculus around directional derivatives and their continuity. It is the basis of our investigation as this calculus works in locally convex spaces beyond the Banach setting.

\begin{defn}\label{defn: der:Bast} Let $E, F$ be locally convex spaces, $U \opn E$,
$f \colon U \rightarrow F$ a map and $r \in \N_{0} \cup \{\infty\}$. If it
exists, we define for $(x,h) \in U \times E$ the \emph{directional derivative}\index{derivative!directional}
$$df(x;h) \coloneq D_h f(x) \coloneq \lim_{\R \setminus \{0\} \ni t\rightarrow 0} t^{-1} \big(f(x+th) -f(x)\big).$$
We say that $f$ is $C^r$ if the \emph{iterated directional derivatives}\index{derivative!iterated directional}
\begin{displaymath}
d^{k}f (x;y_1,\ldots , y_k) \coloneq (D_{y_k} D_{y_{k-1}} \cdots D_{y_1}
f) (x)
\end{displaymath}
exist for all $k \in \N_0$ such that $k \leq r$, $x \in U$ and
$y_1,\ldots , y_k \in E$ and define continuous maps
$d^{k} f \colon U \times E^k \rightarrow F$ (where $f^{0}\coloneq f$). If $f$ is $C^k$ for all $k\in \N_0$ we say that $f$ is \emph{smooth} or $C^\infty$. Note that $df = d^{1} f$ and for curves $c \colon I \rightarrow E$ we have $c^\prime (t) = dc(t;1)$.
\end{defn}

\begin{rem}\label{rem:easyproof}
 Note that the iterated directional derivatives are only taken with respect to the first variable (i.e.\ of the map $x \mapsto df(x;v)$, where $v$ is supposed to be fixed). One can alternatively define iterated differentials to derivate with respect to all variables, but this leads to the same differentiability concept (see \cite{MR1911979} for a detailed explanation).
 The following observations are easily proved from the definitions:
 \begin{enumerate}
  \item $d^{2}f(x;v,w) = \lim_{t\rightarrow 0} t^{-1}(df(x+tw;v)-df(x;v))$
  \item $d^{k}f(x;v_1,\ldots , v_k)= \left.\frac{\mathrm{d}}{\mathrm{d}t}\right|_{t=0} d^{k-1}f(x+tv_k;v_1,\ldots, v_{k-1})$
  \item $f$ is $C^k$ if and only if $f$ is $C^{k-1}$ and $d^{k-1}f$ is $C^1$. Then $d^kf=d(d^{k-1}f)$.
 \end{enumerate}
 Finally, there is a version of \emph{Schwarz' theorem}\index{Schwarz' theorem} which states that the order of directions $v_1,\ldots,v_k$ in  $d^kf(x;v_1,\ldots ,v_k)$ is irrelevant (see Exercise \ref{Ex:Bastiani} 3.).
\end{rem}

\begin{ex}
 Let $A \colon E \rightarrow F$ be a continuous linear map between locally convex spaces. Then $A$ is $C^1$, as we can exploit
 $$dA (x;v) = \lim_{t\rightarrow 0} t^{-1}(A(x+tv)-A(x)) = \lim_{t \rightarrow 0} A(v) = A(v).$$
 In particular since $A$ is continuous, so is the first derivative and we see that $A$ is a $C^1$-map.
 Computing the second derivative we use that the first derivative is constant in $x$ (but not in $v$!) to
 obtain 
 \begin{align*}d^2 A(x;v,w) =& D_w(dA(x;v))= \lim_{t\rightarrow 0} t^{-1}(dA(x+tw;v) - dA(x;v)) \\
   =& \lim_{t\rightarrow 0} t^{-1}(A(v)-A(v)) =0.
 \end{align*}
 In conclusion $A$ is a $C^2$-map (obviously even a $C^\infty$-map) whose higher derivatives vanish.
\end{ex}

\begin{lem}\label{lem:homogeneity}
 Let $f \colon E \supseteq U \rightarrow F$ be a $C^1$-map. Then $df(x;\cdot)$ is homogeneous, i.e.\ $df(x;sv)=sdf(x;v)$ for all $x \in U, v \in E$ and $s \in \R$.
\end{lem}

\begin{proof}
 As $df(x;0v)=df(x;0)=0=0df(x;v)$ we may assume that $s\neq 0$ and thus $df(x;sv)= \displaystyle \lim_{t\rightarrow 0} t^{-1} (f(x+tsv)-f(x)) = s\lim_{t\rightarrow 0} (st)^{-1}(f(x+tsv)-f(x))=sdf(x;v)$
\end{proof}

\begin{prop}[Mean value theorem on locally convex spaces]\label{prop:mean_value}
Let $E,F$ be locally convex spaces, $f \colon U \rightarrow F$ a $C^1$-map on $U \opn E$.\index{mean value theorem}
Then 
\begin{align}\label{eq:meanvalue}
f(y)-f(x) = \int_0^1 df(x+t(y-x);y-x)\mathrm{d}t
\end{align}
for all $x,y\in U$ such that $U$ contains the line segment $\overline{xy}:=\{tx+(1-t)y\mid t\in [0,1]\}$.
\end{prop}

\begin{proof}
 Note that the curve $\gamma \colon [0,1] \rightarrow F , \gamma(t)\coloneq f(x+t(y-x))$ is differentiable at each $t\in [0,1]$.
 Its derivative is 
 \begin{align*}
  \gamma'(t) = \lim_{s\rightarrow 0} s^{-1}(\gamma(t+s)-\gamma(t)) = df(x+t(y-x),y-x),
 \end{align*}
 whence $\gamma'$ is continuous (as $df$ is) and thus a $C^1$-curve. Apply now the Fundamental theorem \ref{prop:funthm} to $\gamma'$ to obtain \eqref{eq:meanvalue}.
\end{proof}

On a locally convex space, every point has arbitrarily small convex neighborhoods. Convex neighborhoods contain all line segments between points in the neighborhood, whence \Cref{prop:mean_value} is available on these neighborhoods. As a consequence we obtain:

\begin{cor}\label{cor:loc_const}
 If $f\colon U \rightarrow F$ is a $C^1$-map with $df\equiv 0$, then $f$ is locally constant.
\end{cor}

\begin{proof}
 For $x \in U$ choose a convex neighbourhood $x \in V \subseteq U$ (cf.\ Appendix \ref{AppA}). For each $y \in V$ the line segment connecting $x$ and $y$ is contained in $V$, whence the vanishing of the derivative with \eqref{eq:meanvalue} implies $f(x)=f(y)$ and $f$ is constant on $V$. 
\end{proof}

\begin{prop}[Rule on partial differentials]\label{prop:rpd}
 Let $E_1,E_2,F$ be locally convex spaces, $U \opn E_1\times E_2$ and $f \colon U \rightarrow F$ continuous. Then $f$ is $C^1$ if and only if the limits \index{rule on partial differentials}
 \begin{align*}
  d_1 f(x,y;v_1) &\coloneq \lim_{t \rightarrow 0} t^{-1}(f(x+tv_1,y)-f(x,y)),\\
  d_2 f(x,y;v_2) &\coloneq \lim_{t \rightarrow 0} t^{-1}(f(x,y+tv_2)-f(x,y))
 \end{align*}
exist for all $(x,y) \in U$ and $(v_1,v_2) \in E_1\times E_2$ and extend to continuous mappings $d_i f \colon U \times E_i \rightarrow F, i=1,2$. In this case 
\begin{align}\label{eq:partialderivatives}
 df(x,y;v_1,v_2) = d_1 f(x,y;v_1) + d_2 f(x,y;v_2), \quad \forall (x,y) \in U, (v_1,v_2) \in E_1 \times E_2.
\end{align}
\end{prop}

\begin{proof}
 If $f$ is $C^1$ the mappings $d_if$ clearly exist and are continuous. Conversely, let us assume that the mappings $d_if$ exist and are continuous. For $(x,y) \in U$, $(v_1,v_2) \in E\times F$ we fix $\varepsilon >0$ such that $(x,y)+t (v_1,v_2) \in U$, whenever $|t|<\varepsilon$. Now if we fix the $i$th component of $f$ we obtain a $C^1$-mapping (by hypothesis, the derivative is $d_if$). Therefore \Cref{prop:mean_value} with \Cref{lem:homogeneity} yields 
 \begin{align}
  &\frac{f((x,y)+t(v_1,v_2))-f(x,y)}{t} \notag \\ =& \frac{f(x+tv_1,y+tv_2)-f(x+tv_1,y)}{t}+\frac{f(x+tv_1,y)-f(x,y)}{t} \notag\\
  =& \int_0^1 d_2f(x+tv_1,y+stv_2;v_2)\mathrm{d}s+\int_0^1d_1f(x+stv_1,y;v_1)\mathrm{d}s \label{integrals}
 \end{align}
 The integrals \eqref{integrals} make sense also for $t=0$, whence they define maps $I_i\colon ]-\varepsilon,\varepsilon[ \rightarrow H$. Due to continuous dependence on the parameter $t$,\footnote{We are cheating here, the continuous dependence of weak integrals on parameters has not been established in this book. See \cite[I Theorem 2.1.5]{MR656198} for a proof which carries over to our setting.} the right hand side of \eqref{integrals} converges for $t\rightarrow 0$. We deduce that the limit $df$ exist and satisfies \eqref{eq:partialderivatives} which is continuous, whence $f$ is $C^1$. 
\end{proof}

The following alternative characterisation of $C^1$-maps will turn the proof of the chain rule into a triviality. However, we shall only sketch the proof to avoid discussing convergence issues of the weak integral involved.

\begin{lem}\label{lem:altchar}
 A map $f \colon E\supseteq U \rightarrow F$ is of class $C^1$ if and only if there exists a continuous mapping, the difference quotient map,
 $$f^{[1]} \colon U^{[1]}\coloneq \{(x,v,s) \in U \times E \times \R \mid x+sv \in U\} \rightarrow F$$
 such that $f(x+sv)-f(x)=sf^{[1]}(x,v,s)$ for all $(x,v,s) \in U^{[1]}$.
\end{lem}

\begin{proof}
Let us assume first that $f^{[1]}$ exists and is continuous. Note that $U^{[1]} \opn U \times E \times \R$. Then $df(x;v) = f^{[1]}(x,v,0)$ exists and is continuous as a partial map of $f^{[1]}$. So $f$ is $C^1$. Conversely, if $f$ is $C^1$, the map 
$$f^{[1]}(x,v,s) \coloneq \begin{cases}
                                                     s^{-1}(f(x+sv)-f(x)) & (x,v,s) \in U^{[1]}, s \neq 0 \\
                                                     df (x;v) & (x,v,s) \in U^{[1]}, s=0
                                                    \end{cases}$$
 is continuous on the open set $U^{[1]} \setminus \{(x,v,s) \in U^{[1]} \mid s = 0\}$. That $f^{[1]}$ extends to a continuous map on all of $U^{[1]}$ follows from continuity of parameter dependent weak integrals, see \cite[Proposition 7.4]{MR2069671} for details. 
\end{proof}

\begin{lem}
 If $f \colon E \supseteq U \rightarrow F$ is $C^1$, then $df(x;\cdot) \colon E \rightarrow F$ is a continuous linear map for each $x \in U$.
\end{lem}

\begin{proof}
 Fix $x \in U$ and note that $df(x;\cdot)$ is continuous as a partial map of the continuous $df$. We have already seen in \Cref{lem:homogeneity} that $df(x;sv)=sdf(x;v)$ for all $s \in \R$, whence we only have to prove additivity.
 Choosing $v,w \in E$ we compute
 \begin{align*}
  t^{-1}(f(x+t(v+w)-f(x))&=t^{-1}(f(x+t(v+w)-f(x+tv))+t^{-1}(f(x+tv)-f(x))\\ &= f^{[1]}(x+tv,w,t)+f^{[1]}(x,v,t)
 \end{align*}
 The right hand side makes also sense for $t=0$ and is continuous, whence passing to the limit we get $df(x;v+w)=df(x;v)+df(x;w)$.
\end{proof}

\begin{prop}[Chain rule]\label{chainrule}\index{chain rule}
Let $f \colon E \supseteq U \rightarrow F$ and $g \colon F\supseteq W \rightarrow K$ be $C^1$-maps with $f(U)\subseteq W$. Then $g\circ f$ is a $C^1$-map with derivative given by
$$d(g\circ f) (x;v) =dg(f(x);df(x,v)) \quad (\text{i.e.} (g\circ f)'(x)=g'(f(x))\circ f'(x).$$
\end{prop}

\begin{proof}
 We use the notation from Lemma \ref{lem:altchar} and write for $(x,y,t) \in U^{[1]}$ with $t\neq 0$:
 \begin{equation} \label{uffda}\begin{aligned}
  t^{-1}(g(f(x+ty))-g(f(x))) &=t^{-1}\left(g\left(f(x)+t\frac{f(x+ty)-f(x)}{t}\right)-g(f(x))\right)\\
  &= g^{[1]}(f(x)f^{[1]}(x,y,t),t) 
 \end{aligned}
 \end{equation}
The function $h\colon U^{[1]} \rightarrow K, h(x,y,t)\coloneq g^{[1]}(f(x)f^{[1]}(x,y,t),t)$ is continuous and extends the right hand side of \eqref{uffda}. Hence Lemma \ref{lem:altchar} shows that $g\circ f$ is $C^1$, with 
$$(g\circ f)^{[1]} (x,y,t)=g^{[1]}(f(x)f^{[1]}(x,y,t),t) \text{ for all } (x,y,t) \in U^{[1]}.$$
Thus $d(g\circ f)(x;y) = (g\circ f)^{[1]} (x,y,0) = g^{[1]}(f(x)f^{[1]}(x,y,0),0) =dg(f(x);df(x;y))$.
\end{proof}

The chain rule, is the basis to transport concepts from differential geometry such as manifolds, tangent spaces etc.\ to our setting. Later chapters will define these objects. 

\begin{lem}\label{lem:prodspace}
 Let $E, (F_i)_{i \in I}$ be locally convex spaces and $f \colon E \supseteq U \rightarrow \prod_{i \in I} F_i$ a map on an open subset. Let $k \in \N_0 \cup \{\infty\}$ and set $f_i \colon \text{pr}_i \circ f \colon U \rightarrow F_j$ (where $\text{pr}_i$ is the $j$th \emph{coordinate projection}\index{coordinate projection}). Then $f$ is $C^k$ if and only if every $f_i, i\in I$ is $C^k$ and 
 \begin{equation}\label{productformula}
  df(x;v)=(df_i(x;v))_{i\in I},\quad x \in U, v\in E
 \end{equation}
\end{lem}

\begin{proof}
 If $f$ is $C^k$, we note that every $f_i=\text{pr}_i \circ f$ is $C^k$ by the chain rule as the projections are continuous linear, whence smooth. Further, $df_i = \text{pr}_i df$ which establishes \eqref{productformula}.

 For the converse, note that it suffices to assume that $k \in \N$. We argue by induction starting with $k=1$.
 Then $t^{-1} f(x+tv)-f(x))= (t^{-1}(f_i(x+tv)-f_i(x))_{i\in I}.$ Now the components converge to $df_i(x;v)$ for $t\rightarrow 0$, whence the difference quotient converges to the limit in \eqref{productformula}. Since every $df_i$ is continuous, also $df=(df_i)_{i\in I} \colon U \times E \rightarrow \prod_{i\in I}F_i$ is continuous and $f$ is $C^1$. For the induction step we notice that by induction $\text{pr}_i \circ df =df_i$ is $C^{k-1}$. By induction $df$ is $C^{k-1}$ and thus $f$ is $C^k$.  
\end{proof}

A subset $Y \subseteq X$ of a topological space $X$ is called \emph{sequentially closed}\index{space!sequentially closed} if $\lim_{n \rightarrow \infty} x_n \in A$ for each sequence $(x_n)_\N \subseteq A$ which converges in $X$. The following lemma will be useful in the discussion of submanifolds in Section \ref{sect:mfd}. 

\begin{lem}\label{lem:seq-closed}
Let $f \colon E \supseteq U \rightarrow F$ be a continuous map from an open subset of a locally convex space to a locally convex space, $F_0 \subseteq F$ a sequentially closed vector subspace such that $f(U)\subseteq F_0$. Let $k \in \N \cup \{\infty\}$, then $f$ is $C^k$ if and only if the corestriction $f|^{F_0} \colon U \rightarrow F_0$ is $C^k$. 
\end{lem}

\begin{proof}
 If $f|^{F_0}$ is $C^k$, so is $f = \iota \circ f|^{F_0}$, where $\iota \colon F_0 \rightarrow F$ is the continuous linear (whence smooth) inclusion. 
 
 Conversely, we argue by induction and assume first that $f$ is $C^1$. For $x\in U$ and $v \in E$ pick a sequence $t_n \rightarrow 0$ such that $x+t_nv \in U$ for each $n\in \N$. Then $df(x;y) = \lim_{n \rightarrow \infty} (f(x+t_nv)-f(x)) \in F_0$ by sequential closedness. Hence the limit exists in $F_0$. Further, as a map $U \times E \rightarrow F_0 , (x,v) \mapsto df(x;v)$ is continuous. We conclude that $f|^{F_0}$ is $C^1$. If $f$ is $C^k$, $d(f|^{F_0})=(df)|^{f_0}$ is $C^{k-1}$ by induction and hence $f$ is $C^k$.
\end{proof}

In \Cref{Dahmensexample} we will see that in \Cref{lem:seq-closed} sequential closedness is a necessary assumption for the validity of the statement.

\begin{Exercise}\label{Ex:Bastiani}
\vspace{-\baselineskip}%
 \Question Check the details for \Cref{rem:easyproof} (for Schwarz' theorem see below).
 \Question Let $E_1,E_2, F$ be locally convex spaces and $\beta \colon E_1 \times E_2 \rightarrow F$ be a continuous bilinear map. Show that $\beta$ is $C^1$ with first derivative $d\beta(x_1,x_2;y_1,y_2) = \beta(x_1,y_2)+\beta(y_1,x_2)$. Compute all higher derivatives of $\beta$ and show that $\beta$ is smooth.
 \Question \textbf{Schwarz' Theorem}: If $f \colon E \supseteq U \rightarrow F$ is a $C^k$-map, and $x \in U$, prove that $d^{r}f(x;\cdot) \colon E^r \rightarrow F$ is symmetric for all $2\leq r\leq k$ (i.e.\ the order of arguments is irrelevant to the function value).\index{Schwarz' theorem}
 \\
 {\tiny \textbf{Hint:} By the Hahn-Banach theorem it suffices to consider $\lambda \circ d^{r}f(x;\cdot) = d^{r}\lambda\circ f(x;\cdot)$, whence without loss of generality $F=\R$. Now for fixed $v_1,\ldots, v_r$ the property can be checked in the finite dimensional subspace generated by $v_1,\ldots,v_r$.}
 \Question The continuous functions $C([0,1],\R)$ form a Banach space with respect to the supremum norm $\lVert \cdot \rVert_\infty$ (the resulting topology is the compact open topology.)\index{compact open topology}.
 \begin{enumerate}
     \item[(a)] Let $f\colon \R \rightarrow \R$ be continuous and $(\gamma_n)_{n \in \N} \subseteq C([0,1],\R)$ be a uniformly convergent sequence of functions with limit $\gamma$. Exploit that $f$ is uniformly continuous on each ball and show that $f \circ \gamma_n \rightarrow f \circ \gamma$ uniformly.
     Deduce that the pushforward\index{pushforward} $f_* \colon C([0,1],\R) \rightarrow C([0,1],\R ), \eta \mapsto f \circ \eta$ is continuous                                                                                           \end{enumerate}
 Assume that $f \in C^1(\R,\R)$. Our aim will be to see that $f_*$ is then $C^1$.
 \begin{enumerate}
  \item[(b)] Assume that the limit 
  \begin{equation}\label{eq:pfwd}
   df_* (\gamma;\eta) \coloneq \lim_{t \rightarrow 0} t^{-1} (f_*(\gamma + t\eta)-f_*(\gamma))
  \end{equation}
 exists. The point evaluation $\ev_x \colon C([0,1],\R) \rightarrow \R, \eta \mapsto \eta(x)$ is continuous linear for each $x \in [0,1]$. Apply $\ev_x$ to both sides of \eqref{eq:pfwd} and find the only possible candidate $\psi(\gamma,\eta)$ for $df_*(\gamma;\eta)$,
 \item[(c)] Use point evaluations to verify that $t^{-1}(f_*(\gamma + t\eta)-f_*(\gamma)) = \int_0^1 \psi (\gamma +st\eta,\eta) \mathrm{d}s$
 \item[(d)] Show that $\psi(\gamma,\eta)$ from (b) is indeed the directional derivative $df_*(\gamma;\eta)$.
 \item[(e)] Verify that $df_*$ is continuous, whence $f_*$ is $C^1$. 
 \end{enumerate}
 \Question Let $E,F,H$ be locally convex spaces, $U\opn E, V\opn H$, $f\colon U \rightarrow F$ a $C^2$-map and $g\colon V \rightarrow U$ and $h \colon V \rightarrow E$ be $C^1$. Prove that the differential of the $C^1$-map $\phi \coloneq df \circ (g,h) \colon V\rightarrow F, \phi(x)=df(g(x);h(x))$ is given by
 \begin{align}\label{ident:doublederiv}
  d\phi (x;y) =d^2f(g(x);h(x),dg(x;y))+df(g(x);dh(x;y)), \forall x\in V, y \in H
 \end{align}
\end{Exercise}

\begin{Answer}[number={\ref{Ex:Bastiani} 3. Schwarz' theorem}] 
 \emph{We will show that for a $C^k$-map $f \colon E \supseteq U \rightarrow F$ and $x \in U$ the map $d^{r}f(x;\cdot) \colon E^r \rightarrow F$ is symmetric for every $2\leq r \leq k$ (and $r<\infty$).}\\
 {\tiny \textbf{Remark:} There are several possibilities to prove this, for example, it suffices to show that the directional derivatives $D_{v_1}$ and $D_{v_2}$ commute for all $v_1,v_2 \in E$ (then the general case follows from $d^r f(x;v_1,\ldots,v_r) = D_{v_r} \cdots D_{v_1}f(x)$).
 However, we will reduce the problem to the well known finite-dimensional case.}
 \\[.15em]
 
 Use the Hahn-Banach theorem: It suffices to prove that $d^r (\lambda \circ f)(x;\cdot) = \lambda (d^rf (x;\cdot))$ is symmetric for every continuous linear functional $\lambda$. Hence without loss of generality $F=\R$.
 Now pick $2 \leq r \leq k$ $v_1,\ldots , v_r \in E$. Then there is $\varepsilon >0$ such that $x+\sum_{i=1}^r t_i v_i$ makes sense for all $|t_i|<\varepsilon$. Thus we can define the auxiliary function
 $$h \colon \R^r \supseteq ]-\varepsilon,\varepsilon[^r \rightarrow \R,\quad (t_1,\ldots,t_r) \mapsto f(x+t_1v_1+\cdots t_rv_r).$$
 By the chain rule $h$ is $C^k$. Note that by the finite-dimensional version of Schwarz' theorem the partial derivatives of $h$ commute. Now the statement for $f$ follows from the chain rule and the observation (cf.\ \Cref{rem:easyproof}) that 
 $$\left.\frac{\partial^r}{\partial_{t_1} \cdots \partial_{t_r}}\right|_{t_1,\ldots,t_r=0} h(t_1,\ldots,t_r) = d^r f(x;v_1,\ldots,v_r).$$
\end{Answer}

\begin{Answer}[number={\ref{Ex:Bastiani} 5.}]
\emph{For a $C^2$ map $f$ and $C^1$ maps $g,h$ we derive a formula for the derivative of $\phi \coloneq df \circ (g,h)$.}\\[.15em]
 
 This is an exercise in applying the chain rule and the rule on partial differentials (\Cref{prop:rpd}):
 \begin{align*}
  d\phi (x;y) &= d(df \circ (g,h))(x;y) = (d_1 df) (g(x),h(x);dg(x;y)) + (d_2df) (g(x),h(x);dh(x;y))\\
  &= d^2 f (g(x);h(x),dg(x;y))+df(g(x);dh(x;y))
 \end{align*}
 Here we have used that the derivative of $df$ with respect to the first component is just $d^2f$ and that $df(g(x);\cdot)$ is continuous linear.
\end{Answer}
 
\subsubsection*{Bastiani vs. \Frechet calculus on Banach spaces}

On Banach spaces one usually defines differentiability in terms of the so called \Frechet derivative. We briefly recall the definitions which should be familiar from basic courses on calculus. While Bastiani calculus is somewhat weaker then \Frechet calculus, the gap between those two can be quantified. We collect these results in the present section.

\begin{defn}\label{defn:Frechet}
A map $f \colon E \supseteq U \rightarrow F$ from an open subset of a normed space $(E,\lVert \cdot \rVert_E)$ to a normed space $(F,\lVert \cdot\rVert_F)$ is \emph{continuous \Frechet differentiable} (or $FC^1$) if for each $x \in U$ there exists a continuous linear map $A_x \in L(E,F)$ such that $$f(x+h)-f(x)=A_x.h +R_x(h) \text{ with }\lim_{h \rightarrow 0}\lVert R_x(h)\rVert_F /\lVert h\rVert_E=0$$ and the mapping $Df \colon U \rightarrow L(E,F), x \mapsto A_x$ is continuous (where the right hand side carries the operator norm). Inductively we define for $k \in \N$ that $f$ is a \emph{$k$-times continuous \Frechet differentiable} map (or $FC^k$-map) if it is $FC^1$ and $Df$ is $FC^{k-1}$. Moreover, $f$ is (\Frechet-)smooth or $FC^\infty$ if $f$ is an $FC^k$ map for every $k \in \N$.\index{Fr\'{e}chet calculus}
\end{defn}

The reader may wonder why the notion of \Frechet differentiability can not be generalised beyond the setting of normed spaces. The reason for this is, that the continuity of the derivative can not be formulated as there is no suitable topology on the spaces $L(E,F)$. Indeed, there is no locally convex topology making evaluation and composition on the spaces $L(E,F)$ continuous. We refer to \Cref{prop:dual_norm} for an example of this pathology in the context of dual spaces. 

From the definition it is apparent (Exercise \ref{Ex:BFrecht} 2.) that if $f$ is $FC^1$, then $Df (x)(h) = df(x;h)$ and thus every $FC^1$-map is automatically $C^1$ in the Bastiani sense. However, one learns in basic calculus courses, existence of directional derivatives is weaker than the existence of derivatives in the \Frechet sense. The next example exhibits this.

\begin{ex}[Bastiani $C^1$ is weaker than \Frechet $C^1$, Milnor '82]\label{ex:Frechetbast}
 Consider the Banach space $\ell^1 = \{(x_n)_{n \in \N} \mid x_n \in \R, \forall n \in \N, \lVert (x_n)\rVert_{\ell^1} = \sum_{n\in \N}|x_n| < \infty \}$. Let $\varphi (u) \coloneq \log(1+u^2)$ and $\psi(u) \coloneq \frac{\mathrm{d}}{\mathrm{d}u} \varphi(u) = \frac{2u}{1+u^2}$. We observe that $|\psi(u)|\leq 1$, whence $|\varphi(u)|\leq |u|$ and we obtain a well-defined map
 $$f \colon \ell^1 \rightarrow \R, f((x_n))\coloneq \sum_{n \in \N} \frac{\varphi(nx_n)}{n}.$$
 Observe that $|f((x_n))|\leq \lVert (x_n)\rVert_{\ell^1}$. In Exercise \ref{Ex:BFrecht} 3. we will show that $f$ is $C^1$ with differential $df((x_n);(v_n)) = \sum_{n\in \N} v_n \psi(nx_n)$ but not \Frechet differentiable.
\end{ex}

\begin{setup}[Bastiani vs.\ \Frechet calculus] While Bastiani $C^k$ is weaker than $FC^k$, \cite[Appendix A.3]{MR2952176} shows that there is only a mild loss of differentiability. In particular, $FC^\infty =C^\infty$. The proofs are somewhat technical induction arguments involving the operator norm. Hence we only summarise the relation between the calculi \textbf{on Banach spaces} in the following diagram (arrows denote implications between conditions):
\begin{displaymath}
\begin{tikzcd}
C^{k+1} \arrow[rr, Rightarrow]
& & FC^k \arrow[rr, Rightarrow] & &  C^{k} \ar[ll, bend left=25, Rightarrow, "\text{dim}\, E < \infty"{below}]
\end{tikzcd} 
\end{displaymath} 
\end{setup}

\begin{Exercise}\label{Ex:BFrecht}  \vspace{-\baselineskip}
\Question Let $(E,\lVert \cdot \rVert_E), (F,\lVert \cdot \rVert_F)$ be normed spaces and the space $L(E,F)$ of continuous linear maps be endowed with the operator norm $\lVert A \rVert_{\text{op}} = \sup_{x \in E\setminus \{0\}} \frac{\lVert A(x)\rVert_F}{\lVert x\rVert_E}$. Show that the evaluation map $\varepsilon \colon L(E,F) \times E \rightarrow F, \varepsilon (A,x)=A(x)$ is continuous.
\Question Let $f\colon U \rightarrow F$ be an $FC^1$-map on $U\opn E$, where $E,F$ are Banach spaces.
\subQuestion Show that the \Frechet derivative satisfies $Df(x)(h) =df(x;h)$ for every $x\in U$, $h\in E$ and deduce that $f$ is $C^1$ in the Bastiani sense.
\subQuestion Use induction to prove that every $FC^k$-map is already $C^k$ by showing that the $k$th-\Frechet derivative gives rise to the $k$th derivative in the Bastiani sense.
 \Question We fill in the details for \Cref{ex:Frechetbast}. Notation is as in the example. Prove that
 \subQuestion $|\psi(u)|\leq 1, \forall u \in \R$.
 \subQuestion $df((x_n);(v_n)) = \sum_{n \in \N} v_n \psi(nx_n)$, whence continuous and thus $f$ is $C^1$.
 \subQuestion $\lVert df((x_n);\cdot)\rVert_{\text{op}}$ equals $0$ if $x=0$ but is $\geq 1$ if $x_n=1/n$ for some $n\in \N$. 
 \subQuestion for $\delta_n \coloneq \begin{cases} 0 & \text{if } m \neq n \\ 1/n & \text{if } m=n\end{cases}$ the expression $\lVert df(\delta_n;\cdot)\rVert_{\text{op}}$ does not converge to $0$. Deduce that $f$ is not $FC^1$.
\end{Exercise}

\section{Infinite-dimensional manifolds}\label{sect:mfd}

In this section we recall the basic notions of manifolds modelled on locally convex spaces. Most of these definitions should be very familiar from the finite-dimensional setting.

\begin{setup}
 In this section we will write $g\circ f$ as a shorthand for $g\circ f|_{f^{-1}(A)}\colon f^{-1}(A)\rightarrow B$ if it helps to avoid clumsy notation.
\end{setup}

\begin{defn}[Charts and atlas]
 Let $M$ be a Hausdorff topological space and $E$ be a locally convex space. A \emph{chart}\index{manifold chart} for $M$ is a homeomorphism $\varphi \colon U_\varphi \rightarrow V_\varphi$ from $U_\varphi \opn M$ onto $V_\varphi \opn E_\varphi$, where $E_\varphi$ is a locally convex space. Let $r \in \N_0 \cup \{\infty\}$. A $C^r$\emph{-atlas}\index{manifold atlas} for $M$ is a set $\mathcal{A}$ of charts for $M$ satisfying the following
 \begin{enumerate}
  \item $M = \bigcup_{\varphi \in \mathcal{A}} U_\varphi$
  \item for all $\varphi, \psi \in \mathcal{A}$ the \emph{change of charts}\index{change of charts} $\varphi\circ \psi^{-1}$ (which are mappings between open subsets of locally convex spaces) are $C^r$.\footnote{Formally, if the charts do not intersect, the change of charts is the empty map $\emptyset \rightarrow \emptyset$ which is $C^r$.}
 \end{enumerate}
 Two $C^r$-atlases $\mathcal{A},\mathcal{A}'$ for $M$ are \emph{equivalent} if their union $\mathcal{A}\cup \mathcal{A}'$ is a $C^r$-atlas for $M$. This is an equivalence relation.
\end{defn}
 
 \begin{defn}
  A $C^r$ manifold $(M,\mathcal{A})$ is a Hausdorff topological space with an equivalence class of $C^r$-atlases $\mathcal{A}$. (If the equivalence class $\mathcal{A}$ is clear we simply write $M$.)
 \end{defn}

 \begin{rem}
  Contrary to the finite-dimensional case we do not require manifolds to be paracompact or second countable (as topological spaces).
  
In general, the manifolds we are interested in will not be modelled on a single locally convex space. For a $C^1$-atlas, the locally convex spaces in which charts take their image are necessarily isomorphic on each connected component. However, some examples we will encounter later on have a huge number of connected components. For each of these connected components the locally convex model spaces will in general not be isomorphic.
\end{rem}

\begin{ex}
Every locally convex space $E$ is a manifold with global chart given by the identity $\id_E$. Similarly, every $U \opn E$ is a manifold with global chart given by the inclusion $U \rightarrow E$.
\end{ex}

\begin{ex}[Hilbert sphere]\label{ex: HSphere} \index{Hilbert sphere}
  For a Hilbert space $(H,\langle \cdot , \cdot \rangle)$ the unit sphere $S_H \coloneq \{x \in H \mid \langle x,x\rangle =1\}$ is a $C^\infty$-manifold. To construct charts, we define the Hilbert space $H_{x_0} = \{y \in H \mid \langle y,x_0\rangle=0\} \subseteq H$ for $x_0 \in S_H$. (If dim $H < \infty$, $H_{x_0}$ is a proper subspace of $H$, but if $H$ is infinite-dimensional, it is isomorphic to $H$, cf.~\cite{Dob95}.) Then define the sets $U_{x_0} \coloneq \{x \in S_H \mid \langle x,x_0\rangle >0\}$ and $V_{x_0} \coloneq \{y\in H_{x_0} \mid \langle y ,y\rangle < 1\}$. 
  We can now define a chart
  $$\varphi_{x_0} \colon U_{x_0} \rightarrow V_{x_0},\quad x \mapsto x-\langle x,x_0\rangle x_0.$$
  (its inverse is given by the formula $\varphi_{x_0}^{-1} (y)=y+\sqrt{1-\langle y,y\rangle} x_0$). Applying these formulae, we see that the change of charts map for $x_0,z_0 \in S_H$ is a smooth map between open (possibly empty) subsets of Hilbert spaces:
  $$\varphi_{z_0} \circ \varphi_{x_0}^{-1} (y) =(y-\langle y,z_0\rangle z_0)+\sqrt{1-\langle y,y \rangle} (x_0-\langle x_0,z_0\rangle z_0 ).$$
\end{ex}

\begin{defn}\label{defn:submanifold}
 Let $M$ be a $C^r$-manifold and together with a sequentially closed vector subspace $F_\varphi \subseteq E_\varphi$ for each chart $\varphi \colon U_\varphi \rightarrow V_\varphi \opn E_\varphi$. A \emph{($C^r$-)submanifold of $M$} is a subset $N \subseteq M$ such that for each $x \in N$, there exists a chart $\phi \colon U_\phi \rightarrow V_\phi$ of $M$ around $x$ such that $\phi(U_\phi \cap N) = V_\phi \cap F_\varphi$. Then $\phi_N \coloneq \phi|_{U_\phi\cap N}^{V_\phi \cap F_\varphi}$ is a chart for $N$, called a \emph{submanifold chart}. \index{submanifold}
 Due to \Cref{lem:seq-closed}, the submanifold charts form a $C^r$-atlas for $N$.
 
If $N$ is a submanifold of $M$ such that all the sequentially closed subspaces $F_\phi$ are complemented subspaces of $E_\phi$ (see \Cref{sect:subm}), we call $N$ a \emph{split submanifold}\index{submanifold!split} of $M$.
\end{defn} 

\begin{defn}\label{defn:productmfd}
 Let $(M,\mathcal{A})$ and $(N,\mathcal{B})$ be $C^r$ manifolds. Then the product $M\times N$ becomes a $C^r$-manifold using the atlas $\mathcal{C} \coloneq \{\phi \times \psi \mid \phi \in \mathcal{A} , \psi \in \mathcal{B}\}$. We call the resulting $C^r$ manifold the \emph{(direct) product of $M$ and $N$}.\index{product manifold}
\end{defn}

\begin{defn}
 Let $r \in \N_0 \cup \{\infty\}$ and $M,N$ be $C^r$ manifolds. A map $f \colon M \rightarrow N$ is called $C^r$ if $f$ is continuous and, for every pair of charts $\phi, \psi$, the map 
 $$\psi \circ f \circ \phi^{-1} \colon E_\phi \supseteq \phi(f^{-1}(U_\psi)\cap U_\phi) \rightarrow F_\psi$$
 is a $C^r$ map. We write $C^r (M,N)$ for the set of all $C^r$-maps from $M$ to $N$.
\end{defn}

\begin{rem}\label{rem:insertchart}
 Let $f\colon M \rightarrow N$ be a continuous map between $C^r$-manifolds. Assume that for some charts $(U_\varphi,\varphi)$ and $(U_\psi,\psi)$ the composition $\varphi \circ f \circ \psi$ is $C^r, r \in \N \cup \{\infty\}$. Then for any other pair of charts $(U_\kappa, \kappa)$ and $(U_\lambda,\lambda)$ with $f(U_\lambda \cap U_\psi) \subseteq U_\varphi \cap  U_\kappa$ we have on $U_\lambda \cap U_\psi$ that 
 $$ \kappa \circ f \circ \lambda^{-1}|_{\lambda (U_\lambda \cap U_\psi} = (\kappa \circ \varphi^{-1}) \circ (\varphi \circ f \circ \psi^{-1}) \circ (\psi \circ \lambda^{-1}|_{\lambda (U_\lambda \cap U_\psi})$$
 where the mapping in the middle is $C^r$ by assumption and the other mappings are change of charts (which are $C^r$ by $M,N$ being $C^r$-manifolds). Hence $\kappa \circ f \circ \lambda^{-1}|_{\lambda (U_\lambda \cap U_\psi)}$ is also $C^r$ by the chain rule \Cref{chainrule}. This argument is called "insertion of charts" and we leave it from now on to the reader. With the insertion of charts argument, it is easy to see that:
 \begin{enumerate}
  \item it suffices to test the $C^r$-property with respect to any atlas of $M$ and $N$. 
  \item if $f \colon M \rightarrow N$ and $g \colon N \rightarrow L$ are $C^r$-maps, so is $g\circ f \colon M \rightarrow L$.
  \item using \Cref{lem:prodspace}: If $M, N_1,N_2$ are $C^r$-manifolds and $f_i \colon M \rightarrow N_i, i=1,2$ are mappings. Then $f \coloneq (f_1 , f_2) \colon M \rightarrow N_1 \times N_2$ is $C^r$ if and only if $f_1,f_2$ are $C^r$
 \end{enumerate} 
\end{rem}

\begin{lem}\label{lem:submfd:initial}
 Let $N$ be a ($C^r$-)submanifold of the $C^r$-manifold $M$. Then the inclusion $\iota\colon N \rightarrow M$ is $C^r$. Further, $f \colon P \rightarrow N$ is $C^r$ if and only if $\iota \circ f$ is $C^r$.
\end{lem}

\begin{proof}
 Due to \Cref{rem:insertchart} if suffices to check the $C^r$-property of $\iota$ in charts which cover $N$. Thus we may choose charts $(\phi,U_\phi)$ of $M$ which induce submanifold charts $\phi_N$ as in \Cref{defn:submanifold}. But then $\phi \circ \iota \circ \phi_N^{-1}$ is the inclusion map $V_\phi \colon F \rightarrow V$ which is $C^r$ (as restriction of a continuous linear map). 
 
 If $f$ is $C^r$, so is $\iota \circ f$ by \Cref{rem:insertchart}. Conversely, let $\iota \circ f$ be a $C^r$-map and $\phi$, $\phi_N$ as before and $\psi$ a chart for $P$. Then $\phi \circ \iota \circ f \circ \psi^{-1} \colon \psi(U_\psi \cap f^{-1}(U_\phi) \rightarrow E$ is $C^r$ with values in the sequentially closed subspace $F$ and thus $(\phi\circ \iota \circ f \circ \psi)|^F=\phi_N \circ f\circ \psi^{-1}$ is $C^r$ by \Cref{lem:seq-closed}. We conclude that $f$ is $C^r$.
\end{proof}

\begin{Exercise}\label{Ex:graph} \vspace{-\baselineskip}%
 \Question Let $f\colon M \rightarrow N$ be a $C^r$-map between $C^r$-manifolds. show that the graph $\text{graph} (f) \coloneq \{(m,f(m)) \mid m \in M\}$ is a split submanifold of $M\times N$.\\
 {\tiny \textbf{Hint:} Use the description of the graph to construct submanifold charts by hand.}
 \Question Verify the details of \Cref{ex: HSphere}: Check that the charts make sense as mappings from $U_{x_0}$ to $V_{x_0}$. Show that the change of charts $\varphi_{x_0} \circ \varphi_{z_0}^{-1}$ is smooth for all $x_0,z_0 \in S_H$ such that $U_{x_0} \cap U_{z_0} \neq \emptyset$. 
  \Question Let $M$ be a manifold and $U \opn M$. Let $\mathcal{A}$ be an atlas of $M$. Endow $U$ with the subspace topology and show that $\mathcal{A}_U \coloneq \{ \phi|_{U \cap U_\phi}\mid (\phi,U_\phi) \in \mathcal{A}\}$ is a manifold atlas for $U$ turning it into a submanifold of $M$. 
 \Question Check that the set $\mathcal{C}$ in \Cref{defn:productmfd} is a $C^r$ atlas for the product manifold.
 \Question Show that a compact manifold must be modelled on a finite-dimensional space.\\ {\tiny \textbf{Hint:} \Cref{prop:compfindim}}
\end{Exercise}

\begin{Answer}[number= {\ref{Ex:graph} 1.}]
 \emph{We construct charts turning }
 $\text{graph} (f) \coloneq \{(m,f(m)) \mid m \in M\}$ \emph{into a split $C^r$ submanifold of $M\times N$ for a $C^r$-function $f$}.
\\[.15em]

It suffices to construct submanifold charts for every point $(m,f(m))$. To this end pick $(U_\varphi,\varphi)$ a chart of $M$ and $(U_\psi,\psi)$ a chart of $N$ with $m \in U_\varphi$, $f(m) \in U_\psi$. Assume that $\varphi \times \psi$ is a mapping into the locally convex space $E \times F$. We will now construct a chart around $(m,f(m))$ mapping all elements in $\text{graph}(f) \cap U_\varphi \times U_\psi$ to the (complemented) subspace $E \times \{0\} \subseteq E\times F$.
Since the vector space operations are continuous (bi)linear, they are smooth in the Bastiani sense and we obtain a $C^r$-map for $M\times N$ via
$$\kappa \colon U_\varphi \times U_\psi \rightarrow E\times F,\quad (m,n) \mapsto (\varphi(m),\psi(n)-\psi(f(m))).$$
This mapping is a $C^r$-diffeomorphism as its inverse is given by the formula $\kappa^{-1} (x,y) \coloneq (\varphi^{-1}(x),\psi^{-1}(y+\psi (f(\varphi^{-1}(x)))))$ (we leave it as an exercise to show that this mapping is well defined on an open subset of $E \times F$). Thus $\kappa$ is a chart for $M \times N$. By construction $\kappa(m,f(m))= (\varphi(m),\psi(f(m))-\psi(f(m)))=(\varphi(m),0)$. Thus $\kappa$ is a submanifold chart for the graph. 
\end{Answer}
\begin{Answer}[number={\ref{Ex:graph} 4.}]
 \emph{We prove that a locally compact manifold $M$ is necessarily finite dimensional} (note that the exercise asks for compact manifolds, but the argument only requires local compactness).\\[.15em]
 
 Let $\varphi \colon U_\varphi \rightarrow V_\varphi \subseteq E$ be a chart for $M$. Since $M$ is locally compact, there exists a compact neighbourhood $C$ of $x \in U_\varphi$ such that $C \subseteq U_\varphi$. Then $\varphi(C) \subseteq E$ is a compact neighbourhood of $\varphi(x)$. Since translations are homeomorphisms in $E$, the translated set $\varphi(C)-\varphi(x) =\{y = m-\varphi(x) \mid m \in \varphi(C)\}$ is a compact $0$-neighbourhood in $E$. Thus $E$ is finite dimensional by \Cref{prop:compfindim}. Thus $M$ is a finite dimensional manifold.
\end{Answer}

\subsection*{Tangent spaces and the tangent bundle}

\begin{defn}
 Let $M$ be a $C^r$-manifold ($r\geq 1$) and $p \in M$. We say that a $C^1$-curve $\gamma$ \emph{passes through}\index{curve!passing through a point} $p$ if $\gamma(0) = p$. For two such curves $\gamma, \eta$ we define the relation 
 \begin{equation}\label{eq:equivalencerel}
 \gamma \sim \eta \Leftrightarrow (\phi \circ \gamma)'(0)=(\phi \circ \eta)'(0)
 \end{equation} 
 for some chart $\phi$ of $M$ around $p$. By the chain rule \eqref{eq:equivalencerel} holds for every chart around $p$ and defines an equivalence relation on the set of all curves passing through $p$.
 The equivalence class $[\gamma]$ is called \emph{(geometric) tangent vector}\index{tangent vector (geometric)} (of $M$ at $p$). Define the \emph{(geometric) tangent space of $M$ at $p$}\index{tangent space} as the set $T_p M$ of all geometric tangent vectors at $p$. 
\end{defn}

We recall now that the tangent space at a point can be turned into a locally convex space isomorphic to the modelling space at that point.
\begin{lem}\label{lem:tangentident}
 \begin{enumerate}
  \item Let $\phi$ be a chart of $M$ around $p$, set $p_\phi \coloneq \phi(p)$. Then 
  $$h_\phi \colon E_\phi \rightarrow T_pM,\quad h_\phi (y) \coloneq [t \mapsto \phi^{-1}(p_\phi+ty)]$$
  is a bijection with inverse $h_\phi^{-1} \colon T_pM \rightarrow E_\phi,\quad [\gamma] \mapsto (\phi\circ \gamma)'(0)$
  \item For all charts $\phi,\psi$ around $p$, we have $h_\psi^{-1}\circ h_\phi = d(\psi\circ \phi^{-1})(p_\phi; \cdot)$ which is an automorphism of the topological vector space $E$.
  \item $T_pM$ admits a unique locally convex space structure such that $h_\phi$ is an isomorphism of locally convex spaces for some (and hence all) charts $\phi$ of $M$ around $p$.
 \end{enumerate} 
\end{lem}

\begin{proof}
  \begin{enumerate}
   \item Note that $h_\phi$ and $h_\phi^{-1}$ are well defined and $h_\phi^{-1}$ is injective. For $y \in E_\phi$, $h^{-1}_\phi \circ h_\phi (y)) = \left.\frac{\mathrm{d}}{\mathrm{d}t}\right|_{t=0} \phi(\phi^{-1}(p_\phi + ty))=y$. Thus $h^{-1}_{\phi}$ is surjective and the inverse of $h_\phi$.
   \item[(b)-(c)] Compute for $y \in E$: $h_\psi \circ h_\phi^{-1}(y)=  \left.\frac{\mathrm{d}}{\mathrm{d}t}\right|_{t=0} \psi(\phi^{-1}(p_\phi + ty))= d(\psi\circ \phi^{-1}) (p_\phi; \cdot)$. Now (c) follows directly from the definition of the vector space structure.\qedhere
  \end{enumerate}
 \end{proof}

\begin{setup}\label{defn:tangent:open}
 Let $U \opn E$, $E$ a locally convex space and $f \colon U \rightarrow F$ a $C^1$-map. We define the mapping 
 $$Tf \colon U \times E \rightarrow F\times F,\quad (x,v) \mapsto (f(x), df(x;v)),$$
 and call this mapping the \emph{tangent map} of $f$. Note that the chain rule \Cref{chainrule} can now be written as $T(f\circ g) = Tf \circ Tg$.\index{chain rule}
\end{setup}

 \begin{defn}[tangent bundle]\label{defn:tangbun}
  Let $(M,\mathcal{A})$ be a $C^r$-manifold with $r \geq 1$. We call $TM \coloneq \bigcup_{p \in M} T_p M$ the \emph{tangent bundle}\index{tangent bundle} of $M$. Then $\pi_{M} \colon TM \rightarrow M, T_pM \ni v \mapsto p$ is called the \emph{bundle projection}. We equip $TM$ with the final topology with respect to the family $(T\phi^{-1})_{\phi \in \mathcal{A}}$ of mappings
  $$T\phi^{-1} \colon V_\phi \times E_\phi \rightarrow TM, \quad (x,v) \mapsto [t \mapsto \phi^{-1}(x+ty)] \in T_{\phi^{-1}(x)} M.$$
  Note that $TU_\phi = \pi_{M}^{-1}(U_\phi)$ is open in $TU$ for all $\phi \in \mathcal{A}$. and 
  $$T\phi \coloneq (T\phi^{-1})^{-1} \colon TU_\phi \rightarrow V_\phi \times E$$
  is a homeomorphism. Moreover, $\mathcal{B} \coloneq \{T\phi \mid \phi \in \mathcal{A}\}$ is a $C^{r-1}$-atlas for $TM$. Thus $TM$ becomes a $C^{r-1}$ manifold and $\pi_{M} \colon TM \rightarrow M$ a $C^{r-1}$-map.
\end{defn}

\begin{lem}
We check the details in \Cref{defn:tangbun}. Let $\varphi,\psi \in \mathcal{A}$:
 \begin{enumerate}
  \item We have $T(\psi \circ \phi^{-1})\circ T\phi = T\psi$,
  \item $TU_\varphi$ is open in $TM$ and $T\varphi \colon TU_\varphi \rightarrow V_\varphi \times E_\varphi$ is a homeomorphism.
  \item $\mathcal{B} = \{T\phi \mid \phi \in \mathcal{A}\}$ is a $C^{r-1}$-atlas, if $M$ is Hausdorff, so is $TM$,
  \item $\pi_{M}$ is a $C^{r-1}$-map,
  \item $TM$ induces on each tangent space $T_pM$ its natural topology
 \end{enumerate}
\end{lem}

\begin{proof}
 \begin{enumerate}
  \item This follows from \Cref{lem:tangentident} (b), we leave the details to the reader.
  \item If $\varphi, \psi \in \mathcal{A}$, we have $(T\psi^{-1})^{-1}(TU\varphi)=T\psi (TU_\varphi \cap TU_\psi)=\psi(U_\varphi\cap U_\psi) \times E_\psi \opn V_\psi \times E_\psi$. By the definition of the final topology $TU_\varphi$ is open in $TM$.
  
  By definition of the final topology $T\varphi^{-1}$ is continuous, whence $T\varphi$ is open for every $\varphi \in \mathcal{A}$. For continuity, pick $U \opn V_\varphi \times E_\varphi$ and let $\psi \in \mathcal{A}$. Now $W \coloneq U \cap \varphi(U_\varphi \cap U_\psi)\times E_\varphi) \opn V_\varphi \times E_\varphi$, whence a quick computation shows
  \begin{align*}
   (T\psi^{-1})^{-1}((T\varphi)^{-1}(U))=T(\psi\circ \varphi^{-1})(W) \opn V_\psi \times E_\psi
  \end{align*}
  as $T(\psi \circ \varphi^{-1})$ is a homeomorphism between open subsets of $V_\varphi \times E_\varphi$ and $V_\psi \times E_\psi$. We deduce that $(T\varphi)^{-1}(U)$ is open in $TM$, whence $T\varphi$ is continuous. 
  \item By (b) each $T\phi$ is a homeomorphism from an open subset of $TM$ onto an open subset of $E_\phi \times E_\phi$, whence it is a chart. Clearly the $TU_\varphi$ cover $TM$ and by (a) the transition maps are $T(\psi \circ \varphi^{-1})=(\psi \circ \varphi^{-1},d(\psi \circ \varphi^{-1}))$ whence $C^{r-1}$. The Hausdorff property is left as Exercise \ref{Ex:tangentmaps}.
  \item In every chart we have $\varphi \circ \pi_{M} = \text{pr}_1 \circ T\varphi$, where $\text{pr}_1\colon V_\varphi \times E_\varphi \rightarrow V_\varphi$ is the canonical projection. As the charts conjugate $\pi_{M}$ to a smooth map, it is of class $C^{r-1}$.
  \item By the definition of the vector topology in \Cref{lem:tangentident} (c) this follows from (b)  \qedhere
  \end{enumerate}
\end{proof}

\begin{rem}\label{rem:dualTM}
 In finite dimensional differential geometry (and on Banach manifolds), one often introduces the \emph{dual bundle}\index{vector bundle!dual bundle} or \emph{cotangent bundle}\index{cotangent bundle}
 $$T^*M \coloneq L(TM,\R) \coloneq \bigcup_{x \in M} (T_xM)' = \bigcup_{x \in M} L(T_xM,\R),$$
 where $(T_xM)' \coloneq L(T_xM,\R)$ is the space of continuous linear mappings $T_xM \rightarrow \R$ (with the operator norm topology if $T_xM$ is Banach, see \cite{MR1330918} or \cite[III.1]{Lang} for the bundle structure). In our setting, there is \textbf{in general} no canonical manifold structure which turns $T^*M$ into a vector bundle.
 
 This poses a problem for theories employing the dual bundle, e.g.\ symplectic geometry, and also differential forms can not be defined as sections of the dual bundle. However, in the case of differential forms, one can circumvent this problem and obtain a theory similar to finite-dimensional differential forms without the dual bundle. We briefly discuss this in \Cref{app:diff_form}.
\end{rem}

We will now introduce tangent mappings of differentiable mappings between manifolds (see \Cref{defn:tangent:open} for the case of an open subset of a locally convex space).

\begin{defn}[Tangent maps]\label{defn:tangentmap}
Let $f \colon M \rightarrow N$ be a $C^r$-map between $C^{r}$-manifolds. Then we define the mappings
$$T_pf \colon T_pM \rightarrow T_{f(p)}N,\quad [\gamma] \mapsto [f\circ \gamma],\quad p \in M$$
Then we define the \emph{tangent map}\index{tangent map} $Tf \colon TM \rightarrow TN, T_pM \ni [\gamma] \mapsto T_pf([\gamma])$.
Note that by construction $\pi_{N} \circ Tf = f \circ \pi_{M}$. Moreover, for each pair of charts $\psi$ of $N$ and $\phi$ of $M$ such that $f(U_\phi) \subseteq U_\psi$ the following diagram is commutative
\begin{displaymath}
 \begin{tikzcd}
  TU_\phi \ar[d, "T\phi"]  \ar[rr , "Tf|_{TU_\phi}^{TU_\psi}"] & & TU_\psi \ar[d, "T\psi"] \\
  V_\phi \times E_\phi \ar[rr, "T(\psi \circ f \circ \phi^{-1})"]  && V_\psi \times F_\psi 
 \end{tikzcd}
\end{displaymath}
Hence the tangent map $Tf$ is a $C^{r-1}$-map if $f$ is a $C^r$-map. 
\end{defn}

\begin{lem}[Chain rule on manifolds]\label{lem:chainrulemfd} \index{chain rule}
 Let $M,N,L$ be $C^r$-manifolds and $f \colon M \rightarrow N$, $g \colon N \rightarrow L$ be $C^r$ maps. Then $T(g\circ f) = Tg\circ Tf$.
\end{lem}

Note that we can of course iterate the tangent construction and form the higher tangent manifolds $T^kM \coloneq \underbrace{(T(T(\cdots(T}_{k \text{ times}}M)\cdots)$ if $M$ is a $C^\ell$-manifold and $k \leq \ell$. Similarly one defines higher tangent maps $T^kf\coloneq \underbrace{(T(T(\cdots(T}_{k \text{ times}}f)\cdots)$.

For later use we set a notation for derivatives of manifold valued curves.

\begin{defn}
 Let $M$ be a manifold and $c \colon J \rightarrow M$ be a $C^1$-map from some interval $J \subseteq \R$. Then we identify $T_t J = \R$ and define the mapping
 $$\dot{c} \colon J \rightarrow TM , \quad t \mapsto T c (t,1).$$
 Note that $\dot{c}$ is the manifold version of the curve differential $c^\prime$, in particular, if $(U,\varphi)$ is a chart of $M$, we have for each $t \in c^{-1}(U)$ the relation 
 $T\varphi \left( \dot{c}(t)\right) = (\varphi \circ c, (\varphi \circ c)^\prime)$.\index{derivative!of a manifold valued curve}
\end{defn}

\begin{Exercise}\label{Ex:tangentmaps} \vspace{-\baselineskip}%
 \Question Verify the details in \Cref{defn:tangentmap} and \Cref{lem:chainrulemfd}.  Show in particular that 
  \subQuestion $Tf$ is well defined and defines a $C^{r-1}$-map with the claimed properties. 
  \subQuestion Check that for $M = U$ an open subset of a locally convex space, both definitions of tangent mappings coincide.
  \subQuestion the manifold $TM$ is a Hausdorff topological space. \\ {\tiny \textbf{Hint:} Consider two cases for $v,w \in TM$: $\pi_{M} (v)=\pi_{M}(w)$ and $\pi_{M} (v)\neq \pi_{M}(w)$.}
 \Question Show inductively, that 
 \subQuestion If $E$ is a locally convex space and $U \opn E$, then $T^k U \cong U \times E^{2^k-1}$ 
 \subQuestion for $U \opn E, V \opn F$ in locally convex spaces, we have 
 $$d^kf(x;v_1,\ldots,v_k) = \text{pr}_{2^k}\left (T^kf(x,w_1,\ldots, w_{2^k-1}\right),$$
 where $\text{pr}_{2^k}$ is the projection onto the $2^k$th component and $w_{2^i+1}=v_{i+1}$ for $0\leq i \leq k-1$ and $w_i=0$ else.
 \Question Establish a manifold version of the rule on partial differentials \Cref{prop:rpd},\index{rule on partial differentials} i.e.\ show that: If $f \colon M_1 \times M_2 \rightarrow N$ is a $C^1$-map (between $C^1$-manifolds) and $p_i \colon M_1 \times M_2 \rightarrow M_i$ are the canonical projections, then for $p=(x,y) \in M_1\times M_2$,
 $$T_p f (v) = T_x f(\cdot, y)(Tp_1 (v)) + T_y f(x,\cdot) (Tp_2 (v)).$$
 Identifying $T(M_1\times M_2)$ with $TM_1\times TM_2$ and $v=(v_x,v_y)$ this formula becomes
 $$T_p f  = T_x f(\cdot, y)(v_x) + T_y f(x,\cdot) (v_y).$$
\end{Exercise}

\section{Elements of differential geometry: submersions and immersions}\label{sect:subm}

 Immersions and submersions are among the first tools students encounter in courses on differential geometry when asked to construct (sub-)manifolds. They can still serve this purpose in infinite-dimensions if one chooses ones definitions carefully. In this section we follow \cite[Section 4.4]{MR656198}\footnote{The smoothness assumption conveniently shortens the exposition but can of course be replaced by finite orders of differentiability, cf.\ \cite{Glofun}.} and the discussion requires the concept of complemented subspaces of a locally convex space (see Appendix \ref{App:complemented}).
 
 \begin{defn}\label{defn: imm/subm}
 Let $M, N$ be smooth manifolds modelled on locally convex spaces $E$ and $F$, respectively, and $\phi \colon M \rightarrow N$ smooth. We say that $\phi$ is 
 \begin{enumerate}
 \item an \emph{immersion}\index{immersion} if for every $x \in M$ there are manifold charts around $x$ and $\phi(x)$ such that $F\cong E \times H$ (as locally convex spaces, see \Cref{App:complemented}), the local representative of $\phi$ in these charts is the inclusion $E \rightarrow E \times H \cong F$,
 \item an \emph{embedding}\index{embedding} if $f$ is an immersion and a topological embedding (i.e.~a homeomorphism onto its image),
 \item a \emph{submersion}\index{submersion} if for every $x \in M$ there are manifold charts around $x$ and $\phi(x)$ such that $E \cong F \times H$  and the local representative of $\phi$ in these charts is the projection $E\cong F \times  H \rightarrow F$
 \end{enumerate}
 \end{defn} 

 \begin{lem}\label{lem:imm_loc_emb}
 If $f\colon M \rightarrow N$ is an immersion, then for every $x\in M$ there is an open neighbourhood $W_x$ such that $f|_{W_x}$ is an embedding.
\end{lem}

\begin{proof}
 Pick immersion charts around $x$ and $f(x)$, i.e.~charts $(U,\varphi), x \in U$ and $(W,\psi)$ such that $f(x)\in W$ and $\psi \circ f\circ \varphi^{-1} = j$, where $j \colon E_\varphi \rightarrow E_\psi \cong E_\varphi \times F$ is the inclusion of the complemented subspace $E_\varphi$ of $E_\psi$. Note that $j$ is a topological embedding onto its image, whence $f|_{U_\varphi} = \psi^{-1} \circ j \circ \varphi$ is a topological embedding (and thus an embedding).
\end{proof}
 
There are several alternative characterisations of submersions and immersions known in the finite dimensional setting. Many of these turn out to be weaker in our setting. 

\begin{lem}\label{alt:subm1}
 A smooth map $f \colon M \rightarrow N$ is a submersion if and only if for each $x \in M$ there are $(U,\varphi)$ of $M$ and $(V,\psi)$ of $N$ with $x\in U$ and $f(U)\subseteq V$ such that $\psi \circ f \circ \varphi^{-1} = \pi|_{\varphi(U)}$ for a continuous linear map $\pi$ with continuous linear right inverse $\sigma$ (i.e.~$\pi\circ\sigma =\id$).
\end{lem}

\begin{proof}
 If $f$ is a submersion and $x\in M$ we can pick submersion charts around $x$, i.e.~charts $\varphi \colon U_\varphi \rightarrow U_F \times U_H \subseteq F\times H \cong E$ with $U_F \opn F$ and $U_H \opn E$ and $\psi \colon V_\psi \rightarrow U_F$ such that $\psi \circ f \circ \varphi^{-1} (x,y) = x, \ \forall (x,y) \in U_F\times U_H$. Obviously, the projection is continuous linear with right inverse given by the inclusion $F \rightarrow F \times H \cong E$.
 
 Let us conversely assume that around $x \in M$ there are charts such that $\psi \circ f \circ \varphi^{-1} = \pi|_{\varphi(U)}$ holds for a continuous linear map $\pi \colon E \rightarrow F$ with continuous right inverse $\sigma \colon F \rightarrow E$. then $\pi|_{\sigma(F)} \sigma (F) \rightarrow F$ is an isomorphism of locally convex spaces with inverse $\sigma$. We obtain a new chart $\nu \coloneq \sigma \circ \psi$ of $N$. Now $\kappa \coloneq \sigma \circ \pi \colon E \rightarrow E$ is continuous linear and satisfies $\kappa \circ \kappa =\kappa$. We deduce from \Cref{lem:complemented} that $\sigma(F)$ is a complemented subspace of $E$ and can identify $E \cong \sigma(F) \times \text{ker} (\pi)$ such that $\kappa$ becomes the projection onto $\sigma (F)$. Shrinking $U$ and $V$ if necessary we may assume $\varphi (U) = A \times B$ and $\nu (V) = A$ for open sets $A \opn \sigma(F)$ and $B \opn \text{ker} \pi$. Then
 $$\nu \circ f \circ \varphi^{-1} = \sigma \circ \psi \circ f \circ \varphi^{-1}= \sigma \circ \pi|_{A \times B} = \kappa|_{A\times B}.\qedhere$$
\end{proof}

\begin{rem}\label{rem:altsubm}
The alternative characterisation of submersions from \Cref{alt:subm1} implies (see Exercise \ref{Ex:submersion} 5.) that a submersion admits local smooth sections. If the manifolds $M$ and $N$ are Banach manifolds\index{Banach manifold} (i.e.~modelled on Banach spaces), the inverse function theorem implies that the existence of local smooth sections is even equivalent to being a submersion (cf.~\cite[Proposition 4.1.13]{MR1173211}).

Finally, we mention that \textbf{if the submersion} $\varphi$ \textbf{is surjective} we obtain: \emph{A map $f \colon N \rightarrow L$ is $C^r$ if and only if $f \circ \varphi$ is a $C^r$-map} for $r \in \N_{0} \cup \{\infty\}$ (see Exercise \ref{Ex:submersion} 6.).
\end{rem}

In finite-dimensional differential geometry, the above definitions are usually not the definitions of submersions/immersions but one deduces them from "easier conditions" involving the tangent mappings such as the following:

\begin{defn}
Let $M, N$ be smooth manifolds and $\phi \colon M \rightarrow N$ smooth. We say that $\phi$ is 
 \begin{enumerate}
 \item \emph{infinitesimally injective (or surjective)}\index{infinitesimally injective (or surjective)} if the tangent map $T_x \phi \colon T_x M \rightarrow T_{\phi (x)} N$ is injective (or surjective, respectively) for every $x \in M$,
 \item a \emph{\naive\ immersion}\index{immersion!\naive} if for every $x \in M$ the tangent map $T_x \phi \colon T_x M \rightarrow T_{\phi (x)} N$ is a topological embedding onto a complemented subspace of $T_{\phi (x)} N$,
 \item a \emph{\naive\ submersion}\index{submersion!\naive} if for every $x\in M$ the map $T_x \phi \colon T_x M \rightarrow T_{\phi (x)}N$ has a continuous linear right inverse. 
 \end{enumerate}
 \end{defn} 

 \setbox\tempbox=\hbox{\begin{tikzcd} 0 \ar[r]& A \ar[r, "i"] & B \ar[r,"q"] & C \ar[r] & 0\end{tikzcd} }
In infinite-dimensions, none of the \naive , infinitesimal versions and the properties from Definition \ref{defn: imm/subm} are equivalent as the following (counter-)examples show:
\begin{setup}[Infinitesimal properties are weaker then \naive\ properties]
Consider the Banach space $c_0$ of all (real) sequences converging to $0$ as a subspace of the Banach space $\ell^\infty$ of all bounded real sequences. We obtain a \emph{short exact sequence}\index{space!short exact sequence}\footnote{In the category of locally convex spaces, a sequence 
\begin{displaymath}
\box\tempbox
\end{displaymath}
of continuous linear maps is \emph{exact} if it satisfies both of the following conditions
\begin{enumerate}
\item \emph{algebraically exact}, i.e.\ images of maps coincide with kernels of the next map,
\item \emph{topologically exact}, i.e.\ $i$ and $q$ are open mappings onto their images. 
\end{enumerate}
If $A$, $B$ and $C$ are \Frechet (or Banach) spaces topological exactness follows from algebraic exactness by virtue of the open mapping theorem \cite[I. 2.11]{Rudin}; for general locally convex spaces this is not the case.} of Banach spaces 
\begin{equation} \label{sq: non-split}
\begin{tikzcd}
0 \arrow[rr] & &  c_0 \arrow[rr, hookrightarrow, "i"] & & \ell^\infty  \arrow[rr, twoheadrightarrow, "q"] && \ell^\infty/c_0 \arrow[rr] && 0.
\end{tikzcd}
\end{equation}
As $i$ and $q$ are continuous linear we have $T_x i (y) = di(x,y) = i(y)$ and $T_x q (y) = q(y)$, whence $i$ is infinitesimally injective and $q$ is infinitesimally surjective. Since $c_0$ is not complemented, \Cref{ex: notcomplemented}, $i$ is not a \naive\ immersion. This implies that \eqref{sq: non-split} does not split, i.e.\ $q$ does not admit a continuous linear right inverse and thus can not be a \naive\ submersion. 
\end{setup}

A submersion as in Definition \ref{defn: imm/subm} turns out to be an open map \cite[Lemma 1.7]{Glofun}. In finite-dimensional differential geometry this is a consequence of the inverse function theorem (when applied to a \naive\ submersion). Going beyond Banach manifolds the submersion property is stronger then the \naive\ notion: 

\begin{setup}[{\naive\ submersions need not be submersions}]\label{setup:naiveimmersion}
Consider the space $A\coloneq C(\R,\R)$ of continuous functions from the reals to the reals. The pointwise operations and the compact open topology (see Appendix \ref{sect:cptopntop}) turn $A$ into a locally convex space (by \Cref{lem:co_vs_uniform}, one can even show that it is a \Frechet space). Then the map $\exp_A \colon A \rightarrow A, f \mapsto e^f$ is continuous by \Cref{lem:pushforward}. Recall that the point evaluations $\ev_x (f) \coloneq f(x)$ are continuous linear on $A$, whence we can use them to find a candidate for the derivative of $\exp_A$. Then  the finite-dimensional chain rule yields the only candidate for the derivative of $\exp_A$ to be $d\exp_A (f;g)(x) = g(x)\cdot \exp_A(f)(x)$. Computing in the seminorms one can show that indeed $d\exp_A(f;g) = g \cdot \exp_A(f)$ and by induction $\exp_A$ is smooth. 
Moreover, at the constant zero function $\mathbf{0}$ we have $d\exp_A (\mathbf{0};\cdot)= \id_A$, so $\exp_A$ is a \naive\ submersion. However, $\exp_A$ takes values in $C(\R,]0,\infty[)$ which does not contain a neighbourhood of $\exp_A(\mathbf{0})=\mathbf{1}$. Thus $\exp_A$ can not be a submersion (as all submersions are open mappings by \cite[Lemma 1.7]{Glofun}.

This example also shows that the Inverse Function Theorem fails in this case \cite{Eel66} (cf.\, also \Cref{invfunction}).
\end{setup}

However, as \cite{Glofun} shows, the following relations do hold:

\begin{setup}[Submersions, immersions vs. the \naive\ and infinitesimal concepts]\label{rel:subm:imm}
Consider a map $\phi \colon M \rightarrow N$ between manifolds modelled on locally convex spaces, then:
\begin{displaymath}
\begin{tikzcd}
\text{Immersion} \arrow[rr, Rightarrow]
& & \text{\naive\ immersion} \arrow[rr, Rightarrow] \ar[ll, bend left=25, Rightarrow, "M \text{ Banach manifold}"{below}]& & \text{infinitesimally injective} \ar[ll, bend left=25, Rightarrow, "\text{dim}\, M < \infty"{below}]
\\ \\ \\
\text{Submersion} \arrow[rr, Rightarrow]
& & \text{\naive\ submersion} \arrow[rr, Rightarrow]  \ar[ll, bend left=25, Rightarrow, "N \text{ Banach manifold}"{below}]& & \text{infinitesimally surjective} \ar[ll, bend left=25, Rightarrow, "\text{dim}\, N < \infty"{below}] 
\ar[ll, bend right=25, Rightarrow, "M \text{ Hilbert manifold and } N \text{ Banach manifold}"{above}]
\end{tikzcd} 
\end{displaymath}
\end{setup}

The reason one really would like the strong notions of submersions and immersions is that these notions are strong enough to carry over the usual statements on submersions and immersions to the setting of infinite dimensional manifolds. For example one can prove several useful statements on split submanifolds. Again we refer to \cite{Glofun} for more general results on submersions and immersions in infinite dimensions.

\begin{defn}
 Let $f \colon M \rightarrow N$ be smooth and $S \subseteq N$ be a split submanifold. Then $f$ is \emph{transversal over }$S$\index{transversal} if for each $m \in f^{-1}(S)$ and submanifold chart $\psi \colon V \rightarrow V_1 \times V_2$ with $\psi (f(m))=(0,0)$ and $\psi (S) \subseteq V_1 \times \{0\}$ there exists an open $m$-neighbourhood $U$ with $f(U)\subseteq V$ and 
 \begin{equation}\label{eq:transversal}
  \begin{tikzcd}
  U \ar[r,"f"] & V \ar[r,"\psi"] & V_1 \times V_2 \ar[r,"\text{pr}_2"] & V_2 
  \end{tikzcd}
 \end{equation}
is a submersion.
\end{defn}
Note that due to the fact that compositions of submersions are again submersion (cf.\ Exercise \ref{Ex:submersion} 1.), whence if $f$ is a submersion \eqref{eq:transversal} always is a submersion.

\begin{prop}
Let $\phi \colon M \rightarrow N$ be a smooth map.
If $S \subseteq N$ is a split submanifold\footnote{A more involved proof works for every submanifold (not only for split ones), see \cite[Theorem C]{Glofun}.} of $N$ such that $f$ is transversal over $S$, then $\phi^{-1} (S)$ is a submanifold of $M$. 
\end{prop}

\begin{proof}
 The map \eqref{eq:transversal} is a submersion. Shrinking $U, V_2$, there are charts $\varphi \colon U \rightarrow U_1 \times U_2$ and $\kappa \colon V_2 \rightarrow U_2$ such that $\varphi(m)=(0,0)$ and $\kappa(0)=0$ and the following commutes:
$$
  \begin{tikzcd}
    U \ar[r,"f"] \ar[d, "\varphi"] & V \ar[r,"\psi"] & V_1 \times V_2 \ar[r,"\text{pr}_2"] & V_2 \ar[d, "\kappa"]\\
    U_1 \times U_2  \ar[rrr, "\text{pr}_2"] &&& U_2
  \end{tikzcd}
 $$
 Now we will prove that $\varphi$ is a submanifold chart for $f^{-1}(S)$, i.e.\ $\varphi(U\cap f^{-1}(S)) = \varphi(U) \cap (U_1 \times \{0\})$.
 To see this note that since $\psi$ is a submanifold chart, we have for $x \in U$ that $f(x) \in S$ if and only if $\text{pr}_2 (\psi(f(x)))=0$. Now the commutativity of the diagram shows that this is the case if and only if $\varphi (x) \in \text{pr}_2^{-1}(0)=U_1 \times \{0\}$.
\end{proof}

\begin{cor}\label{cor:preimagesubmfd}
 If $f \colon M \rightarrow N$ is a submersion, $f^{-1}(n)$ is a split submanifold for $n \in N$.
\end{cor}

\begin{lem}\label{lem:pullbacksubm}
Let $f \colon M \rightarrow P$ and $g \colon N \rightarrow P$ be smooth maps and $g$ be a submersion. Then the \emph{fibre product}\index{fibre product (of manifolds)} $M \times_P N \coloneq \{(m,n) \in M\times N \mid f(m)=g(p)\}$ is a split submanifold of $M \times N$ and the projection $\mathrm{pr}_1 \colon M \times_P N \rightarrow M$ is a submersion.
\end{lem}
 
 \begin{proof}
 Let $(m,n) \in M\times_P N$ and pick submersion charts $\psi \colon U_\psi \rightarrow V_\psi \opn E_N$, $\kappa \colon U_\kappa \rightarrow V_\kappa \opn E_P$ for $g$ with $n \in U_\psi$. Recall that for the submersion charts $\kappa \circ g \circ \psi^{-1} = \pi$ for a continuous projection $\pi \colon E_N = E_P \times F \rightarrow E_P$ (where $F$ is the subspace complement) and we may assume that $V_\psi = V_\kappa \times \tilde{V}$. Finally, we pick a chart $(U_\varphi,\varphi)$ of $M$ such that $m \in U_\varphi$ and $f(U_\varphi) \subseteq U_\kappa$. Hence we obtain a commutative diagram
 \begin{displaymath}
  \begin{tikzcd}
   & U_\psi \ar[r,"\psi"] \arrow[d, "g"] & V_\psi \arrow[r, equals]   \arrow[rd, "\pi"] &  V_\kappa \times \tilde{V} \arrow[d,"\text{pr}_{V_\kappa}"]  \\ 
 U_\varphi \arrow[r,"f"]& U_\kappa \arrow[rr,"\kappa"] &  &V_\kappa &
  \end{tikzcd}
 \end{displaymath}
Denote by $\text{pr}_{\tilde{V}}$ the projection onto $\tilde{V}$. Then we construct a smooth map for $(m,n) \in U_\varphi \times U_\psi$ via
\begin{align*}
 \delta(m,n) \coloneq (\varphi(m),(\text{pr}_{V_\kappa} (\psi(n))-\kappa(f(m)),\text{pr}_{\tilde{V}}(\psi(n)))) \in V_\varphi \times ((V_\kappa-V_\kappa) \times \tilde{V}).  
\end{align*}
This mapping is smoothly invertible with inverse given by
$$\delta^{-1}(x,(y,z)) = (\varphi^{-1}(x),\psi^{-1}(y+\kappa(f(\varphi^{-1}(x))),z)).$$
We leave it again as an exercise to work out that the domain of the inverse is open. Note that due to the commutative diagram we see that
$(m,n) \in M \times_P N$ if and only if $(m,n)$ maps under $\delta$ to $E_M \times \{0\} \times F$ which is a complemented subspace of $E_M \times E_P \times F \cong E_M \times E_N$. Thus $M \times_P N$ is a split submanifold of $M \times N$.
Since the projection $\text{pr}_1 \colon M \times N \rightarrow N$ is smooth, so is its restriction to $M\times_P N$ by \Cref{lem:submfd:initial}. As $\text{pr}_1 \colon M \times_P N \rightarrow M$ is conjugated by $\delta$ and $\varphi$ to the projection $\text{pr}_{V_\kappa}$, it is a submersion. \end{proof}

Finally, there is a close connection between embeddings and split submanifolds. 
\begin{lem}\label{lem:splitembedding}
 Let $f \colon M\rightarrow N$ be smooth. The following conditions are equivalent\begin{enumerate}
                                                                                  \item $f$ is an embedding,
                                                                                  \item $f(M)$ is a split submanifold of $N$ and $f|^{f(M)}\colon M \rightarrow f(M)$ is a diffeomorphism.
                                                                                 \end{enumerate}
\end{lem}
\begin{proof}
 Let $E,F$ be the modelling spaces of $M$ and $N$, respectively.\\
 (a) $\Rightarrow$ (b): By assumption, $f$ is a topological embedding and a smooth immersion. Consider $y \in f(M)$ and $x \in M$ with $f(x)=y$ and pick charts $\varphi_x \colon U_x \rightarrow V_x \subseteq E$ and $\varphi_y \colon U_y \rightarrow V_y \subseteq F$ such that $x \in U_x$, $f(U_x)\subseteq U_y$ and $\varphi_y \circ f \circ \varphi_x^{-1} = j|_{V_x}$ for a linear topological embedding $j\colon E\rightarrow F$ onto a complemented subspace $j(E) \oplus H = F$. Since $j(V_x)$ is relatively open, we may adjust choices such that $j(V_x) = V_y \cap j(E)$. A quick computation then shows that $\varphi_y$ restricts to a submanifold chart and $f(M)$ becomes a split submanifold of $N$. Moreover, in the (sub)manifold charts we have $j|_{V_x}^{V_y \cap j(E)} = \varphi_y|_{f(M)\cap U_y} \circ f\circ \varphi_x^{-1}$ and this map is a diffeomorphism. Thus $f|^{f(M)}\colon M \rightarrow f(M)$ is a local diffeomorphism and a homeomorphism, whence a diffeomorphism.\\
 (b) $\Rightarrow$ (a) Let  $\iota \colon f(M) \rightarrow N$ be the inclusion map. Since $f|^{f(M)}$ is a diffeomorphism, $\iota \circ f|^{f(M)}$ is a topological embedding. Since $f(M)$ is a split submanifold, there is an isomorphism $\alpha \colon E \rightarrow \alpha(E) \subseteq F$ of locally convex spaces, such that $\alpha (E)$ is complemented in $F$. Pick charts $\varphi_x \colon U_x \rightarrow V_x$ and $\varphi_{f(x)} \colon U_{f(x)} \rightarrow V_{f(x)}$ with $x \in U_x$ and $f(U_x)\subseteq U_{f(x)}$. We may assume that $V_{f(x)}=P\times Q$ for $P \opn \alpha(E)$ and $\varphi_{f(x)}(U_x\cap f(M)) = V_{f(x)} \cap \alpha (E) = P$. Set now $W \coloneq \alpha^{-1}(P)$. then it is easy to see that $\theta \coloneq (\varphi_{f(x)}\circ f\circ \varphi^{-1}_x)^{-1} \circ \alpha|_W \colon W \rightarrow V_{x}$ makes sense and is a smooth diffeomorphism, whence $\theta^{-1}\circ \varphi_x \colon U_x \rightarrow W$ is a chart for $M$. By construction $\varphi_{f(x)}\circ f\circ (\theta^{-1}\circ \varphi_x)^{-1} = \alpha|_W$ is a linear topological embedding onto $\alpha(E)$. This shows that $f$ is an immersion.
\end{proof}

\begin{Exercise}\label{Ex:submersion} \vspace{-\baselineskip}%
 \Question Show that if $f \colon M \rightarrow N$ and $g \colon N \rightarrow L$ are submersions, so is $g\circ f \colon M \rightarrow L$.
 \begin{minipage}{7cm}
   {\tiny \textbf{Hint}: The submersion property is local, try constructing small enough neighbourhoods around $m \in U \subseteq M, f(m) \in V \subseteq N$ and $g(f(m)) \in W \subseteq L$ and charts such that:}
           \end{minipage}
           \begin{minipage}{7cm}
            $
 \begin{tikzcd}
U \arrow[r, "f"] \arrow[d] & V \ar[d] \arrow[r, "g"] & W \arrow[d]\\
F \times X \arrow[r, "\text{pr}_F"] & F \cong Y \times Z \arrow[r, "\text{pr}_Y"] & Y
\end{tikzcd}
$
           \end{minipage}
\Question Let $f \colon M \rightarrow N$ and $g\colon K \rightarrow L$ be submersions (immersions) show that then also $f\times g \colon M\times K \rightarrow N \times L, (m,k) \mapsto (f(m),g(k))$ is a submersion (immersion).
\Question Let $p \colon M \rightarrow N$ be a submersion and $n \in N$. From \Cref{cor:preimagesubmfd} we obtain a submanifold $P \coloneq p^{-1}(n)$. Show that for $x \in P$ one can identify the tangent space of the submanifold as $T_x P = \text{ker} T_x p = \{v \in T_xM \mid T_x p (v) =0\}$. 
\Question Let $(H,\langle \cdot , \cdot\rangle)$ be a Hilbert space. Prove that $\psi \colon H \setminus \{0\} \rightarrow \R, x \mapsto \langle x,x\rangle$ is a submersion and deduce that the Hilbert sphere $S_H =\psi^{-1}(1)$ is a submanifold of $H$. Then show that $T_x S_H = \{v \in H \mid \langle v,x\rangle =0\}$ for all $x \in S_H$ (i.e. the tangent space is the orthogonal complement to the base point $x$).
 \Question Let $\varphi \colon M \rightarrow N$ be a smooth submersion. 
  Show that $\varphi$ admits smooth local sections, i.e.~for every $x \in M$ there is $\varphi (x) \in U \opn N$ and a smooth map $\sigma \colon U \rightarrow M$ with $\sigma (\varphi(x))=x$ and $\varphi \circ \sigma = \id_U$. Deduce then that $\varphi$ is an open map. \\
 {\footnotesize \textbf{Hint:} Use the characterisation from \Cref{rem:altsubm}.\\ \textbf{Remark}: If $M,N$ are Banach manifolds, the existence of local sections is equivalent to $\varphi$ being a submersion, see \cite[Proposition 4.1.13]{MR1173211}.}
 \Question Let $\varphi \colon M\rightarrow N$ be a smooth surjective submersion. Show that $f \colon N \rightarrow L$ is $C^r$ if and only if $f\circ \varphi$ is $C^r$ for $r \in \N_0 \cup \{\infty\}$.
 {\footnotesize \textbf{Hint:} Use Exercise \ref{Ex:submersion} 5.}
 \Question Show that if $f \colon M \rightarrow N$ and $g\colon N \rightarrow L$ are immersions (embeddings), so is $g\circ f$.
 \Question Work out the details omitted in the proof of \Cref{lem:splitembedding}.
\end{Exercise}

\begin{Answer}[number={\ref{Ex:submersion} 1.}] 
 \emph{We show that the composition $g\circ f$ of the submersions $f \colon M \rightarrow N$ and $g \colon N \rightarrow L$ is a submersion.} \\[.15em]

 The submersion property is local, whence we can restrict to chart neighbourhoods of submersion charts around $m \in U \subseteq M, f(m) \in V \subseteq N$ and $g(f(m)) \in W \subseteq L$ such that:
            $$
 \begin{tikzcd}
U \arrow[r, "f"] \arrow[d, "\varphi"] & V \ar[d, "\psi_1"] \arrow[equal,r] & V \ar[d, "\psi_2"] \arrow[r, "g"]  & W \arrow[d, "\kappa"]\\
F \times X \arrow[r, "\text{pr}_F"] & F \ar[r, "\cong"] & Y \times Z \arrow[r, "\text{pr}_Y"] & Y
\end{tikzcd}
$$
Now via the typical insertion of charts: 
 \begin{align}\label{compositionconst}
  \kappa \circ g\circ f \circ \varphi^{-1} = \kappa \circ g \circ \psi_2^{-1} \circ \psi_2 \circ \psi_1^{-1} \circ \psi_1 \circ f \circ \varphi^{-1} = \text{pr}_Y \circ \psi_2 \circ \psi_1^{-1} \circ \text{pr}_F.
 \end{align}
Note that the change of charts $\psi_2 \circ \psi_1^{-1}$ is a diffeomorphism on its domain (which we will now call $O$). Shrinking $U$ we may assume that $\varphi (U) = O \times D$. We obtain a modified chart $\tilde{\varphi} \coloneq ((\psi_2 \circ \psi_1^{-1}) \times \id_W) \circ \varphi$. If we now insert $\tilde{\varphi}$ into \eqref{compositionconst} we see that 
$\kappa \circ g\circ f\circ \tilde{\varphi}^{-1} = \text{pr}_Y \colon F \times X \cong Y \times Z \times X \rightarrow Y$. In other words we have constructed submersion charts for the composition which thus turns out to be a submersion.
 \end{Answer}

\chapter{Spaces and manifolds of smooth maps}\label{sect:smoothmappingspaces} \copyrightnotice

In this chapter we consider spaces of differentiable mappings as infinite-dimensional spaces.
These spaces will then serve as the model spaces for manifolds of mappings, i.e. manifolds of differentiable mappings between manifolds.

  \section{Topological structure of spaces of differentiable mappings}\label{sect:top:smoothmaps}
  
  In this section, we denote by $M, N$ (possibly infinite-dimensional) manifolds.
  
  \begin{defn}
   Endow the space $C^\infty (M,N)$ with the initial topology with respect to the map 
   \begin{align*}
    \Phi \colon C^\infty (M,N) \rightarrow \prod_{k \in \N_0} C(T^kM,T^kN)_{\text{c.o.}}, \quad f \mapsto (T^kf)_{k \in \N_0}
   \end{align*}
  where the spaces on the right hand side carry the compact open topology (cf.\ \Cref{sect:cptopntop}). The resulting topology on $C^\infty (M,N)$ is called the \emph{compact open $C^\infty$-topology}.\index{compact open $C^\infty$-topology}
  \end{defn}

  \begin{rem}\label{rem:Cinftopology}
   \begin{enumerate}
    \item The map $\Phi$ is clearly injective (as $T^0f\coloneq f$). Therefore, $\Phi$ is a homeomorphism onto its image.
    \item Note that the compact open $C^\infty$ topology is also the initial topology with respect to the mappings
    $$T^k \colon C^\infty (M,N) \rightarrow C(T^kM,T^kN)_{\text{c.o}},\quad f \mapsto T^kf, \quad k \in \N_0.$$
    \item If $M \opn E, N \opn F$ for $E,F$ locally convex spaces, the compact open $C^\infty$-topology is the initial topology with respect to the mappings
    $$d^k \colon C^\infty (M,N) \rightarrow C(M \times E^k,N)_{\text{c.o.}},\quad f \mapsto d^kf, \quad k\in N_0.$$
    \item By construction the compact open $C^\infty$ topology is finer then the compact open topology (i.e. the topology induced by the inclusion $C^\infty (M,N) \rightarrow C(M,N)$. In particular, if $M$ is locally compact (i.e. $M$ is finite-dimensional) the evaluation $\ev \colon C^\infty (M,N) \times M \rightarrow N, (f,x) \mapsto f(x)$ is continuous by \Cref{lem:eval}.
   \end{enumerate}
  \end{rem}
 
 \begin{lem}\label{continuity_pf_pb}
  Let $h \colon L \rightarrow M$ and $f \colon N \rightarrow O$ be smooth. Then the pushforward\index{pushforward} and the pullback\index{pullback} 
  \begin{align*}
   f_* \colon C^\infty (M,N) &\rightarrow C^\infty (M,O),\quad  g \mapsto f\circ g\\
   h^* \colon C^\infty (M,N) &\rightarrow C^\infty (L,N),\quad  g \mapsto g \circ h
  \end{align*}
  are continuous.
 \end{lem}

\begin{proof}
 Since the compact open $C^\infty$-topology is initial with respect to the family $(T^k)_{k \in \N_0}$ it suffices to prove that $T^k \circ f_*$ and $T^k \circ h^*$ are continuous for each $k \in \N_0$.
 However, the chain rule yields for each $k$ commutative diagrams
 \begin{displaymath}
  \begin{tikzcd}
   C^\infty (M,N) \ar[r,"f_*"] \ar[d,"T^k"] & C^\infty (M,O) \ar[d,"T^k"] && C^\infty(M,N) \ar[r,"h^*"] \ar[d,"T^k"] & C^\infty (L,N) \ar[d,"T^k"] \\
   C(T^kM,T^kN)_{\text{c.o.}} \ar[r, "(T^kf)_*"] & C(M,O)_{\text{c.o.}} && C(T^kM,T^kN)_{\text{c.o.}} \ar[r, "(T^kh)^*"] & C(T^kL,T^kN)_{\text{c.o.}} 
  \end{tikzcd}
 \end{displaymath}
The pushforward and the pullback in the lower row are continuous by \Cref{lem:pushforward}. We conclude that $f_*$ and $h^*$ are continuous.
\end{proof}
 \begin{prop}\label{prop:lcvx_mappingsp}
  Let $E$ be a locally convex space. Then the compact open $C^\infty$-topology turns
  $C^\infty (M,E)$ with the pointwise operations into a locally convex space.
  \end{prop}
  
  \begin{proof}
   The compact open $C^\infty$-topology is initial with respect to the map 
   $$\Phi \colon C^\infty (M,E) \rightarrow \prod_{k\in \N_0} \underbrace{C(T^kM, T^kE)_{\text{c.o.}}}_{\cong C(T^kM, E^{2^k})_{\text{c.o.}}},\quad f \mapsto (T^kf)_{k \in \N_0}.$$
   Now the spaces $C(T^kM, E^{2^k})_{\text{c.o.}}$ are locally convex spaces by \Cref{lem:co_vs_uniform} since $E^{2^k}$ is a locally convex space. The product of locally convex spaces is again a locally convex space. Thus every linear subspace of the product becomes a locally convex space.
   Now it is easy to see that $\Phi$ is linear with respect to pointwise addition and scalar multiplication. Thus the image of $\Phi$ is a linear subspace and $\Phi$ is an isomorphism of locally convex spaces identifying $C^\infty (M,E)$ as a locally convex subspace of $\prod_{k \in \N_0} C(T^kM,E^{2^k})$. 
  \end{proof}
 
 \subsection*{Interlude: Certain open sets in the compact open $C^\infty$-topology}
 In this section we recall the classic arguments (see e.g.~\cite{Hir76}) that certain subsets of mappings are open in the compact open $C^\infty$-topology. Observe that one needs here (and we shall require it in the whole section) that the source manifold is compact. For non-compact source manifolds, the sets discussed here will in general not be open in the compact open $C^\infty$-topology. We will need the results collected in this subsection in \Cref{Chap:Liegp} when discussing the group of diffeomorphisms. 
  
 \begin{lem}\label{lem:scattering}
  Let $M,N$ be manifolds, $M$ compact and a finite open cover $U_1,\ldots,U_n$ of $M$. Then the map 
  $$\Psi \colon C^\infty (M,N) \rightarrow \prod_{i=1}^n C^\infty (U_i,N),\quad f \mapsto (f|_{U_i})_{i\in I}$$
  is a homeomorphism onto $\mathcal{I} \coloneq \{(f_i) \in \prod_{i=1}^n C^\infty (U_i,N) \mid f_i|_{U_i \cap U_j} \equiv f_j|_{U_i \cap U_j}, \forall i,j \in I\}$. Moreover, $\mathcal{I}$ is closed in $\prod_{i=1}^n C^\infty (U_i,N)$. 
  \end{lem}

  \begin{proof}
   To see that $\mathcal{I}$ is closed we introduce for every $i,j \in I\coloneq \{1,\ldots, n\}$ and $x \in U_i \cap U_j$ the map 
   $\ev_{x,i,j} \colon \prod_{i=1}^n C^\infty (U_i , N) \rightarrow N \times N, (\gamma_k)_{1\leq k\leq n} \mapsto (\gamma_i (x),\gamma_j (x))$. These are continuous since projections onto components in a product and the point evaluations are continuous (cf.\ \Cref{rem:pointeval}). Now we denote by $\Delta N \subseteq N\times N$ the diagonal (i.e. all elements of the form $(n,n)$), which is closed in $N \times N$ due to $N$ being Hausdorff.
   Then $\mathcal{I}$ is closed as the preimage:
   \begin{align*}
    \mathcal{I} = \bigcap_{x \in \sqcup_{i,j \in I}U_i \cap U_j} \ev_{x,i,j}^{-1}(\Delta N).
   \end{align*}
  Note that we can write the restriction $f \mapsto f|_{U_i}$ as the pullback $f|_{U_i}=f\circ \iota_{U_i} =(\iota_{U_i})^*(f)$ with the inclusion of $U_i$ into $M$. Hence the restriction map is continuous by \Cref{continuity_pf_pb}. and as a consequence $\Psi$ is continuous. Moreover, $\Psi$ is clearly injective and we only have to prove that $\Psi$ is an open mapping onto its image. To see this we need to check that finite intersections of sets of the form
  \begin{displaymath}
   \lfloor K,U,k\rfloor \coloneq \{f \in C^\infty (M,N) \mid T^kf(K)\subseteq U\},\quad K \subseteq T^k M \text{ compact}, U \opn T^kN, k\in \N.
  \end{displaymath}
  get mapped to open sets by $\Psi$ (recall that the topology is initial with respect to the $T^k$). Now since $M$ is compact, \cite[II, \S 3 Proposition 3.2]{Lang} implies that there are open sets $V_i \subseteq \overline{V}_i \subseteq U_i, 1\leq i \leq n$ and $M \subseteq \bigcup_{1\leq i \leq n} V_i$. Consider now $f \in \bigcap_{1\leq r \leq \ell} \lfloor K_r,U_r,k_r\rfloor$ and define $T^{k_r}\overline{V}_i \coloneq \pi_{k_r}^{-1}(\overline{V}_i)$, where $\pi_{k_r} \colon T^{k_r}U_i \rightarrow U_i$ is the bundle projection. Note that $T^{k_r}\overline{V}_i$ is closed in $T^{k_r}U_i$ for every $1\leq i \leq n$. Then clearly 
  \begin{align}\label{eq:localCR}
   \gamma|_{U_i} \in \bigcap_{1\leq r \leq \ell} \lfloor K_r\cap T^{k_r}\overline{V}_i,O_r,k_r\rfloor \quad \forall 1\leq i\leq n
  \end{align}
  Now let $(g_i)_{1\leq i \leq n} \in \mathcal{I}$ such that every $g_i$ satisfies \eqref{eq:localCR} for $i$. Since the $T^{k_r}V_i$ cover $K_r$, the unique map $g$ defined via $g|_{U_i}=g_i$ satisfies $g \in \bigcap_{1\leq r \leq \ell} \lfloor K_r,U_r,k_r\rfloor$. We conclude 
  \begin{align*}
   \prod_{i=1}^n \bigcap_{1\leq r\leq \ell}\lfloor K_r\cap T^{k_r}\overline{V}_i , O_r, k_r\rfloor \subseteq \Psi \left(\bigcap_{1\leq r \leq \ell} \lfloor K_r,U_r,k_r\rfloor\right) 
  \end{align*}
  and thus $\Psi$ is open onto its image. 
  \end{proof}
 
 \begin{lem}\label{lem:immsub_opn}
  Let $M$ be a compact manifold and $N$ be a finite-dimensional manifold. Then the sets\index{immersion!set of} \index{submersion!set of}
  \begin{align*}
   \Imm (M,N) = \{ f \in C^\infty (M,N) \mid f \text{ is an immersion}\} \\
   \Sub (M,N) = \{ f \in C^\infty (M,N) \mid f \text{ is a submersion}\}
  \end{align*}
 are open in the compact open $C^\infty$-topology
 \end{lem}

 \begin{proof}
  Since $M$ is compact (whence finite-dimensional) a map is an immersion (submersion) if and only if it is infinitesimally injective (surjective). We need to check that these properties define an open set in $C(TM,TN)$, whence they induce an open set in the compact open $C^\infty$-topology. 
  
  Consider a map $f \in  C^\infty (M,N)$ and pick a pair of charts $(U_\varphi,\varphi)$ of $M$ and $(U_\kappa,\kappa)$ of $N$ such that $f(U_\varphi) \subseteq U_\kappa$.
  In addition we pick a compact set $K_\varphi \subseteq U_\varphi$ with non-empty interior. By compactness of $M$, we can choose a finite set of pairs of charts and compact sets such that the interior of the $K_\varphi$ cover $M$. Apply \Cref{lem:scattering} to obtain an embedding $C^\infty (M,N) \rightarrow \prod_{i=1}^n C^\infty (U_{\varphi_i} ,N)$. 
  We will now construct for each $K_{\varphi_i}$ an open neighborhood in $C^\infty (U_{\varphi_i},N)$ which consists only of immersions (submersions) if $f$ has this property. Pulling back the product of these neighborhoods with the embedding then yields the desired neighborhood of $f$ in $C^\infty (M,N)$. 
  
  To this end, we consider the smooth map $g \coloneq \kappa \circ f\circ \varphi^{-1} \in C^\infty (V_\varphi,V_\kappa)$ where $V_\varphi \opn \R^d$ and $V_\kappa \opn \R^n$. Set $L \coloneq \varphi (K_\varphi)$ and observe that $g$ is an immersion (submersion) if and only if $f$ is an immersion (submersion). 
  Recall that on open subsets of vector spaces we have $Tg = (g,dg) \in C(TV_\varphi,TV_\kappa) =C(V_\varphi \times \R^d,V_\kappa \times \R^n)$. Let $e_1,\ldots, e_d$ be the standard basis of $\R^d$. For $x \in L$ we represent the Jacobian as $J_x (g) = \left[df (x;e_1) , \ldots , df(x,e_d)\right]$. If $g$ is a immersion (submersion) then the Jacobi matrix has for every $x$ maximal rank, i.e.~if $g$ is an immersion, the differential is injective if and only if $\text{rk}J_x (g) =d \leq n$. In particular, the rank of the Jacobi matrix is constant, say $\text{rk} J_x (g) =N$ for all $x \in L$. We can thus pick for every $x \in L$ a subset $I_x \subseteq \{1,\ldots,d\}$ of $N$ elements such that $\{dg(x;e_j)\}_{j \in I_x}$ is linearly independent (note that the indices will in general depend on $x$!)
  If $\{df(x;e_j)\}_{j \in I_x}$ is linearly independent, then there exists $\varepsilon_x >0$ such that every tuple $(x_1,\ldots,x_N) \in \prod_{j \in I_x}B_{\varepsilon}(dg(x;e_j))$ is linearly independent (where $B_\varepsilon (z)$ is the $\varepsilon$-ball in $\R^d$, see \cite[Lemma 1.6.7]{MR1173211}). By continuity of $dg$, there is for every $x \in L$ a compact neighborhood $N_x$ of $x$ such that $(dg(y,e_{j_1}),\ldots ,dg(y,e_{j_N})) \in \prod_{j \in I_x}B_{\varepsilon}(dg(x;e_j)$ for all $y \in \overline{N}_x$. Thus 
  \begin{align}
   dg \in \Omega (g,x) \coloneq \bigcap_{j\in I_{x}}\lfloor N_x \times \{e_j\} , B_{\varepsilon_{x}} (df(x,e_j))\rfloor .
  \end{align}
  By construction every $h \in C^\infty (V_\varphi,V_\kappa)$ with $dh \in \Omega (g,x)$ has a Jacobian of rank $N$ at every point in $N_x$. In other words every such $h$ is an immersion (submersion) on $N_x$ if $g$ is such a map. In Exercise \ref{Ex:cinftytop} 4. we shall now construct from $\Omega (g,x)$ an open neighborhood of $f$ in $C^\infty (U_\varphi , N)$ consisting only of maps which restrict to immersions (submersions) on $K_\varphi$ if $f$ is an immersion (submersions). 
  We conclude that $\Imm (M,N)$ and $\Sub (M,N)$ are neighborhoods of their points, whence open.
  \end{proof} 

 \begin{prop}
  Let $M$ be a compact manifold and $N$ be a finite-dimensional manifold. Then the set of embeddings \index{embedding!set of}
  $$\Emb (M,N) \coloneq \{f \in \Imm (M,N) \mid f \text{ is a topological embedding}\}$$
  is open in the compact open $C^\infty$-topology.
 \end{prop}

 \begin{proof}
  Let $f \in \Emb (M,N)$ and fix a finite family of charts
  \begin{enumerate}
   \item  $(U_i,\varphi_i)$ of $M$ and $(V_i,\psi_i)$ of $N$ such that $f(U_i) \subseteq V_i$, such that
   \item  for every $i$ there is a compact set $K_i \subseteq U_i$ and the interiors of the $K_i$ cover $M$.
  \end{enumerate}
 Recall that an embedding is in particular an injective immersion. Hence \Cref{lem:immsub_opn} allows us to choose an open neighborhood $O_f \opn C^\infty (M,N)$ of $f$ consisting only of immersions which satisfy also property (a). We shall now show that we can shrink $O_f$ to obtain an open neighborhood of $f$ consisting only of immersions.
 
 We have already seen in \Cref{lem:imm_loc_emb} that every immersion in $O_f$ restricts locally to an embedding. However, since $M,N$ are finite-dimensional manifolds, we can use the quantitative version of the inverse function theorem (see \cite[1.1]{Glofun} and the references there) to obtain a uniform estimate on the size of these neighborhoods: Shrinking $O_f$, we may assume that every $g\in O_f$ satisfies \begin{enumerate}
  \item[(c)] $g|_{U_i}$ is an embedding for every $i$. 
\end{enumerate}
 (an alternative proof of this fact using uniform estimates can be found in \cite[2. Lemma 1.3]{Hir76}).
 Now since $f$ is an embedding, we see that for every $i$ the compact sets $f(K_i)$ and $f(M\setminus U_i)$ are disjoint and we can thus find disjoint $A_i,B_i \opn N$ such that $f(K_i) \subseteq A_i$ and $f(M\setminus U_i) \subseteq B_i$. As there are only finitely many $i$, we can shrink $O_f$ further such that every $g \in O_f$ satisfies
 \begin{enumerate}
  \item[(d)]  $g(K_i) \subseteq A_i$ and $g(M\setminus U_i) \subseteq B_i$ for all $i$.
 \end{enumerate}
 We shall show now that $g \in O_f$ is injective. Let $x,y \in M$ be distinct points and $x \in K_i$. If $y \in U_i$, then $g(x)\neq g(y)$ by (c). If $y \in M\setminus U_i$, then $g(x)\in A_i$ and $g(y) \in B_i$ by (d), so again $g(x)\neq g(y)$. We conclude that $g$ is injective.
 
 Summing up, $O_f$ is an open set consisting entirely of injective immersions. However, since $M$ is compact, every injective immersion is an embedding (see e.g.~\cite[Proposition 3.3.4]{MR1173211}).
 We conclude that $O_f$ consists only of embeddings, whence $\Emb (M,N)$ is open.
\end{proof}

\begin{cor}\label{cor:diffopn}
 If $M$ is a compact manifold, the set of diffeomorphisms \index{diffeomorphism group}
 $$\Diff (M) \coloneq \{f \in C^\infty (M,M) \mid \exists g \in C^\infty (M,M) \text{ with } g\circ f = \id_M\}$$
 is open in $C^\infty (M,M)$ with the compact open $C^\infty$-topology.
 \end{cor}

 \begin{proof}
  A diffeomorphism $\phi$ permutes the connected components of $M$ and induces on every component a diffeomorphism onto another component. Since the components are compact, there is an open $\phi$-neighborhood in $C^\infty (M,M)$ whose elements map every component to the same component as $\phi$.
  Thus we may assume that $M$ is connected.
  
  A diffeomorphism $\phi \colon M \rightarrow M$ is in particular a map which is an embedding and a submersion. Assume conversely that $\psi \colon M \rightarrow M$ is a mapping which is a submersion and an embedding. Since the image of a submersions is open (Exercise \ref{Ex:submersion} 5.), the set $\phi (M)$ is open and closed in $M$, whence $\phi (M) =M$ by connectedness of $M$. Hence $\phi$ is a bijective map. Its inverse is smooth by Exercise \ref{Ex:submersion} 6. as $\phi$ is a submersion and $\id_M = \phi^{-1} \circ \phi$ (alternatively a bijective embedding is a diffeomorphism by \Cref{lem:splitembedding}).    
 \end{proof}

 \begin{rem}
  As was already mentioned for non-compact $M$, the subsets considered in this subsection will in general not form open subsets of $C^\infty (M,N)$ with respect to the compact open $C^\infty$-topology. The reason for this is that the compact open topology can only control a function's behaviour on compact sets. On a non-compact manifold this topology is too weak to control the behaviour of a function on all of $M$. For this reason one has to refine the topology if $M$ is non-compact. The Whitney type topologies are a common choice, see \cite{HaS17}. However, many results presented in the next sections do not hold (at least not in the generality stated) for the Whitney type topologies. A few examples of this behaviour for $M$ non-compact are:
  \begin{itemize}
   \item the pullback $h^\ast$ is in general discontinuous for the Whitney-topologies (whereas it is continuous in the compact open topology by \Cref{continuity_pf_pb}),
   \item the exponential law, \Cref{thm:explaw} below, is wrong.
  \end{itemize}
While one can develop a general theory for function spaces on non-compact manifolds (see e.g.\ \cite{Mic}), these examples show already that the resulting theory will require a much higher technical investment. We refrain from a discussion in the context of this book and refer the interested reader instead to the literature \cite{HaS17,Mic}.
 \end{rem}

 \begin{Exercise}\label{Ex:cinftytop}
  \vspace{-\baselineskip}
 \Question Fill in the details for \Cref{rem:Cinftopology}:
  \subQuestion Show that the compact open $C^\infty$-topology is the initial topology wrt.\, $(T^k)_{k\in \N}$.\\
  {\tiny \textbf{Hint:} A mapping into a product is continuous if and only if each component is continuous.}
  \subQuestion If $M,N$ are open subsets of locally convex spaces, show that the initial topologies wrt.~the families $(T^k)_{k\in \N_0}$ and $(d^k)_{k\in \N_0}$ coincide.\\
  {\tiny \textbf{Hint:} Exercise \ref{Ex:tangentmaps} 3. yields one inclusion of topologies. For the converse show inductively that $d^k \circ f, \forall k\in \N_0$ continuous, implies $T^k \circ f$ continuous for all $f \colon Z \rightarrow C^\infty (M,N)$.}
  \Question Let $\varphi \colon M\rightarrow N$ be a smooth map between smooth manifolds and $E$ a locally convex space. Show that the pullback $\varphi^* \colon C^\infty (N,E) \rightarrow C^\infty (M,E), f \mapsto f \circ \varphi$ is continuous linear. Deduce that if $\varphi$ is a diffeomorphism, then $\varphi^*$ is an isomorphism of locally convex spaces.
  \Question Let $K,L$ be compact manifolds and $M$ be a manifold. Show that the composition map $\Comp \colon C^\infty (K,M)\times C^\infty (L,K) \rightarrow C^\infty (L,M),\quad (f,g) \mapsto f\circ g$ is continuous. 
  \Question Work out the missing details in the proof of \Cref{lem:immsub_opn}. Show in particular that due to the compactness of $L$ the $\Omega (g,x)$ yield an open neighborhood of $g$ in $C^\infty (V_\varphi,V_\kappa)$ consisting only of mappings whose Jacobian has maximal rank on the compact set $L$. Moreover, construct a neighborhood of $f \in C^\infty (M,N)$ consisting entirely of immersions (submersions) if $f$ is an immersion (submersion).  
 \end{Exercise}

  \section{The exponential law and its consequences}\label{sect:explaw}
 In this section we prove a version of the exponential law \Cref{thm:explaw} for smooth mappings. Before we begin, let us observe a crucial fact about the compact open $C^\infty$-topology:
 
 Assume that $M$ is a \emph{compact} manifold and $E$  a locally convex space. Since the compact open $C^\infty$-topology is finer then the compact open topology, we see that for every 
 $$O \opn E, \qquad C^\infty (M,O) \coloneq \{f \in C^\infty (M,E) \mid f(M)\subseteq O\}$$ is an open subset. Now $C^\infty (M,E)$ is a locally convex space by \Cref{prop:lcvx_mappingsp} whence $C^\infty (M,O)$ becomes a manifold and it makes sense to consider differentiable mappings with values in $C^\infty (M,O)$.
  
 We now prepare the proof of the exponential law by providing several auxiliary results.

 \begin{lem}\label{lem:fwedge_vector}
  Let $E,F,H$ be locally convex spaces, $U \opn E$ and $V \opn F$. If $f\colon U \times V \rightarrow H$ is smooth, then so is $f^\vee \colon U \rightarrow C^\infty (V,H)$. Its derivative is given by
  \begin{align}\label{eq:wedgederiv}
   df^\vee(x;\cdot) = (d_1f)^\vee(x).
  \end{align}
 \end{lem}

 \begin{proof}
  \textbf{Step 1:} \emph{$f^\vee$ is continuous.}\\
  It suffices to prove that $d^k \circ f^\vee\colon U \rightarrow C(V\times F^k,H)_{\text{c.o.}}$ is continuous for every $k\in \N_0$ (cf.\ \Cref{rem:Cinftopology} and Exercise \ref{Ex:cinftytop}). For $k=0$ this was proved in \Cref{prop:C0expo}.
  We prove by induction that 
  \begin{align}\label{eq:diffwedge} 
   d^k \circ f^\vee (x) = d^k(f^\vee(x))=(d_2^kf)^\vee(x) \quad \forall x\in U.
  \end{align}
  The induction start for $k=0$ is trivial. For the induction step let $k>0$ and we compute
  \begin{align*}
   &d^k(f^\vee(x))(y;v_1,\ldots, v_k)\\
   =& \lim_{t\rightarrow 0} t^{-1}\left(d^{k-1}(f^\vee(x))(y+tv_k;v_1,\ldots,v_{k-1})-d^{k-1}(f^\vee(x))(y;v_1,\ldots,v_{k-1})\right) \\
   =&  \lim_{t\rightarrow 0} t^{-1}\left(d^{k-1}(f(x,y+tv_k;(0,v_1),\ldots,(0,v_{k-1}))-d^{k-1}(f(x,y;(0,v_1),\ldots,(0,v_{k-1}))\right)\\
  =& d^kf(x,y;(0,v_1),\ldots, (0,v_k)).
  \end{align*}
 Thus we have identified $d^k\circ f$ as $(d_2^kf)^\vee$ which is again continuous by \Cref{prop:C0expo}. We conclude that $f^\vee$ is continuous.
 
 \textbf{Step 2:} \emph{$f^\vee$ is $C^1$ and the derivative satisfies \eqref{eq:wedgederiv}.} Pick $x \in U$, $z \in E$ and $t \in \R$ small. We shall show that 
 $$\Delta (t,x,z) \coloneq t^{-1}(f^\vee (x+tz)-f^\vee (x)) \xrightarrow{t \rightarrow 0} (d_1f)^\vee(x,\cdot;z)$$
 in $C^\infty (V,H)$. Recall that the compact open $C^\infty$-topology is initial with respect to the family $(d^k \colon C^\infty (V,H) \rightarrow C(V\times F^k,H)_{\text{c.o}})_{k\in \N_0}$. Thus $\Delta(t,x,z)$ converges for $t\rightarrow 0$ if and only if $d^k \circ \Delta(t)$ converges. Therefore, we pick $k \in \N_0$ and a neighborhood $\lfloor K , U \rfloor$ of $d^k ((d_1 f)^\vee)$, i.e.\, $K \subseteq V \times F^k$ is compact such that $(d_1f)^\vee(x,\cdot,z) (K) \subseteq U \opn H$.
 Since higher differentials are symmetric by Schwartz' theorem (Exercise \ref{Ex:Bastiani} 3.) we have
 \begin{align*}
  d^k(d_1f)^\vee (x;z)(y;v_1,\ldots , v_k) &= d^{k+1}f(x,y;(z,0),(0,v_1),\ldots,(0,v_k))\\  
  &= d^{k+1}f(x,y;(0,v_1),\ldots,(0,v_k),(z,0))
 \end{align*}
 Now we have seen in \Cref{lem:altchar} that the difference quotient extends continuously to $t=0$ by the differential. We apply this to $d^kf$: For each $\overline{y} \coloneq (y_0,v_1,\ldots,v_k) \in K \subseteq V\times F^k$ there exists $\overline{y} \in N_{\overline{y}} \opn V \times F^k$ and $\varepsilon_{\overline{y}} >0$ such that 
 \begin{align*}
  N_{\overline{y}} \times& ]-\varepsilon_{\overline{y}} , \varepsilon_{\overline{y}}[\,\setminus \{0\} \rightarrow H \\
  (\overline{w}&,t)  \mapsto t^{-1}(d^k f(x+tz,w_0;(0,w_1),\ldots,(0,w_k))- d^k f(x,w_0;(0,w_1),\ldots,(0,w_k))
 \end{align*}
 (where $\overline{w} = (w_0,\ldots,w_k)$) takes values in $U$ and extends continuously to some function $N_{\overline{y}} \times ]-\varepsilon_{\overline{y}} , \varepsilon_{\overline{y}}[ \rightarrow U$. Exploiting compactness of $K$ we cover it by finitely many of these neighborhoods $N_{\overline{y}_1}, \ldots ,N_{\overline{y}_\ell}$. Hence if $|t| < \min_{1\leq i \leq \ell} \varepsilon_{\overline{y}_i}$, we see that 
 $$t^{-1}(d^kf(x+tz,v;(0,v_1),\ldots,(0,v_k))-d^kf(x,v;(0,v_1),\ldots,(0,v_k)) \in U.$$
 In other words, $d^n \Delta (t,x,z)(\overline{v}) \in U$ for all $t$ small enough and $\overline{v} \in K$, whence \eqref{eq:wedgederiv} holds. Since $(d_1f)^\vee$ is $C^0$ by Step 1, we see that $f^\vee$ is $C^1$.  
 
 \textbf{Step 3:} \emph{$f$ is $C^k$ for each $k\geq 2$.}
 Note that $h \colon (U\times E) \times V \rightarrow H, ((x,z),y) \mapsto d_1f(x,y;z)$ is smooth. Now we argue by induction, where Step 2 is the induction start. Also Step 2 shows that $df^\vee = h^\vee$. The induction hypothesis shows that $h^\vee$ is $C^{k-1}$, so $f^\vee$ must be $C^k$ and since $k$ was arbitrary $f^\vee$ is smooth.
  \end{proof}
  
 \begin{prop}\label{Prop:evsmooth}
  Let $E$ be a finite-dimensional space, $F$ be locally convex and $U \opn E$. Then the evaluation map 
  $\ev \colon C^\infty (U,F) \times U \rightarrow F$
  is smooth.\index{evaluation map}
 \end{prop}
 
 \begin{proof}
  We already know from \Cref{rem:Cinftopology} that $\ev$ is continuous. Moreover, $\ev$ is linear in the first component and thus $d_1 \ev (f,x;g) = \ev (g,x)$ (this implies in particular that all partial derivatives with respect to the first component exist).
  Let us now compute $d_2 \ev$. For $f \in C^\infty (U,F)$ and $x\in U$ $w \in E$ and small $t$ we have
  $$t^{-1}\ev (f,x+tw)-\ev(f,x)=t^{-1}(f(x+tw)-f(x)) \rightarrow df(x;w) \text{ as } t\rightarrow 0.$$
  Hence $d_2\ev (f,x;w)$ exists and is given by $d_2 \ev (f,x;w) = df(x;w)=\ev_1 (df,(x,w))$, where $\ev_1 \colon C^\infty (U\times E,F) \times (U\times E) \rightarrow F, (\gamma,z) \mapsto \gamma(z)$ is continuous.
  We conclude that 
  \begin{align}
   \label{deriv:eval}
d\ev (f,x;g,v) = d_1 \ev (f,x;g) + d_2 \ev(f,x;v) = \ev (g,x) + \ev_1 (df,(x,v))
  \end{align}
  exists and is continuous (the mapping $C^\infty (U,F) \rightarrow C^\infty (U \times E,F), f\mapsto df$ is clearly continuous and linear whence smooth). Thus $\ev$ (and also $\ev_1$) is $C^1$. We see that inductively $\ev$ is $C^k$ as its derivative is already $C^{k-1}$.
 \end{proof}
 
 We will now formulate and prove the exponential law, \Cref{thm:explaw}. To justify the name, denote the set of all functions from $X$ to $Y$ by $Y^X$. In this notation the exponential law for arbitrary maps becomes $(Z^Y)^X \cong Z^{X\times Y}$, whence the name exponential law.
  
   \begin{thm}[Exponential law]\label{thm:explaw}
 Let $M$ be a compact manifold and $O \opn E, U \opn F$ be open subsets of locally convex spaces.\index{exponential law} Then
 \begin{enumerate}
  \item If $f \colon U \times M \rightarrow O$ is smooth, so is $f^\vee \colon U \rightarrow C^\infty (M,O),\quad f^\vee (x)(y)\coloneq f(x,y)$.
  \item The evaluation map $\ev \colon C^\infty (M,O)\times M \rightarrow O , (\gamma,x) \mapsto \gamma(x)$ is smooth and the map 
        $C^\infty(U\times M,O) \rightarrow C^\infty (U,C^\infty (M,O), f \mapsto f^\vee$ is a bijection.
 \end{enumerate}
 In particular, the mapping $f$ is smooth if and only if $f^\vee$ is smooth.
 \end{thm} 
 
 \begin{proof}
  Since $C^\infty (M,O)$ is an open submanifold of $C^\infty (M,E)$, \Cref{lem:submfd:initial} shows that for the purpose of this proof we may assume without loss of generality that $O=E$.
  \begin{enumerate}
   \item Pick a finite atlas $\{(\varphi_i,U_i)\}_{1\leq i \leq n}$ of the compact manifold $M$ and obtain from \Cref{lem:scattering} a topological embedding with closed image
   $$\Psi \colon C^\infty (M,E) \rightarrow \prod_{i=1}^n C^\infty (U_i,E),\quad f \mapsto (f|_{U_i})_{1\leq i \leq n}.$$
   Now since the image of $\Psi$ is closed we apply \Cref{lem:seq-closed}: For any smooth $f \colon U \times M \rightarrow O$ the map $f^\vee \colon U \rightarrow C^\infty (M,O) \opn C^\infty (M,E)$ will be smooth if and only if $\Psi \circ f^\vee$ is smooth. In other words, $f^\vee$ is smooth if and only if for every $1\leq i \leq n$ the map $U \ni x \rightarrow f^\vee (x)|_{U_i} \in C^\infty (U_i,E)$ is smooth. Applying Exercise \ref{Ex:cinftytop} 2.\, this is equivalent to the smoothness of the mappings $U \ni x \rightarrow f^\vee (x)|_{U_i} \circ \varphi_i^{-1} \in C^\infty (\varphi_i(U_i),E)$ and these mappings are smooth by \Cref{lem:fwedge_vector}.
   \item The map $\ev \colon C^\infty (M,O)\times M \rightarrow O$ is smooth if locally around each point it is smooth. Hence we choose $(f,x) \in C^\infty (M,O)\times M$ and pick $\varphi \colon U \rightarrow V \opn \R^d$ a chart of $M$ around $x$. We obtain another evaluation map $\ev_\varphi \colon C^\infty (V, O) \times V \rightarrow O$ such that
   $$\ev (\eta , z) = \ev_\varphi ((\varphi^{-1})^* (\eta),\varphi (z)),\quad  (\eta, z) \in C^\infty (M,O)\times U.$$
   Note that $(\varphi^{-1})^* \colon C^\infty (M,O) \rightarrow C^\infty (V,O)$ is the restriction of the continuous linear (whence smooth) map $(\varphi^{-1})^*\colon C^\infty (M,E) \rightarrow C^\infty (V,E)$ to the open set $C^\infty (M,O)$ (cf.\ Exercise \ref{Ex:cinftytop} 2.). Hence $(\varphi^{-1})^*$ is smooth and $\ev$ will be smooth if $\ev_\varphi$ is smooth for each chart $\varphi$ of $M$. However, the smoothness of $\ev_\varphi$ was established in \Cref{Prop:evsmooth}.
   
   We finally have to check that $C^\infty(U\times M,O) \rightarrow C^\infty (U,C^\infty (M,O), f \mapsto f^\vee$ is bijective. Obviously it is injective, hence we need to check surjectivity. If $f \colon U \rightarrow C^\infty (M,O)$ is smooth, then the following map is smooth $$U \times M \rightarrow C^\infty (M,O) \times M, \quad (u,m) \mapsto (f(u),m).$$
   Composing with $\ev$, the map $f^\wedge \colon U \times M \rightarrow O, (u,m) \mapsto f(u)(m)$ is smooth and satisfies $(f^\wedge)^\vee =f$. This establishes surjectivity.
   \qedhere
  \end{enumerate}
 \end{proof}

 \begin{rem}
  One can even show that the map from \Cref{thm:explaw} (b) is a homeomorphism. Moreover, one can obtain similar results for finite orders of differentiability. We skip the details here and refer instead to \cite{MR3342623} for more information.
  
  Note that compactness of the manifold $M$ is a crucial ingredient and the statement of the exponential law becomes false for general non-compact $M$, see \cite{Mic}.
 \end{rem}
 
 \begin{Exercise} \label{ex:explaw}\vspace{-\baselineskip}
  \Question Let $M$ be compact and $E$ be a finite-dimensional vector space. Show that the evaluation map $\ev \colon C^\infty (M,E) \times M \rightarrow E$ is a submersion. (If you prefer a real challenge, prove this for a locally convex space $E$)
  \Question Let $f \colon U \times M \rightarrow O$ and $p \colon O \rightarrow N$ be smooth maps. As always, we denote by $p_* \colon C^\infty (M, O) \rightarrow C^\infty (M , N), h \mapsto p\circ h$ the pushforward and assume that the exponential law \Cref{thm:explaw} holds for the spaces $C^\infty (U \times M,O)$ and $C^\infty (U\times M,N)$. Prove that 
  $$p_* \circ (f^\vee) = (p\circ f)^\vee.$$
 \end{Exercise}
\setboolean{firstanswerofthechapter}{true}
\begin{Answer}[number={\ref{ex:explaw} 1.}] 
 \emph{We assume that the exponential law holds for all relevant function spaces in this exercise! If $f \in C^\infty (U\times M,O)$ and $p \in C^\infty (O,N)$, we will prove that the pushforward satisfies $p_*\circ (f^\vee) = (p\circ f)^\vee$} \\[.15em]
 
 Pick $(u,m)\in U \times M$, then $$p_*\circ (f^\vee)(u)(m)= p_* (f^\vee(u))(m) = (p_* \circ f(u,\cdot))(m) = p(f(u,m))=(p\circ f)^\vee (u)(m)$$
 and this proves the assertion. (Why did we need to assume that the exponential law holds if the calculation does not use it?)
 \end{Answer}
\setboolean{firstanswerofthechapter}{false}

\section{Manifolds of mappings}
In this section we discuss spaces of smooth mappings between manifolds as infinite-dimensional manifolds. We shall not directly construct the manifold structure for general spaces (see \Cref{App:canmfdmap} for a sketch).

\begin{tcolorbox}[colback=white,colframe=blue!75!black,title=General assumption] In this section 
$K$ will be a compact smooth manifold,
$M, N$ will be smooth (possibly infinite-dimensional) manifolds.
\end{tcolorbox}

\begin{defn}\label{defn:canonicalmfd}
A smooth manifold structure on $C^\infty(K,M)$ is \emph{canonical} if \index{canonical manifold of mappings}
\begin{itemize} 
 \item the underlying topology is the compact open $C^\infty$-topology, and
 \item For each (possibly infinite-dimensional) $C^\infty$-manifold $N$ and $f\colon N\rightarrow C^\infty(K,M)$, the map $f$ is~$C^\infty$
if and only if
\begin{displaymath}
f^\wedge\colon N\times K\rightarrow M,\quad (x,y)\mapsto f(x)(y) \quad \text{ is } C^\infty.
\end{displaymath}
\end{itemize}
\end{defn}

\begin{rem}
A canonical manifold structure enforces a suitable version of the exponential law, \Cref{thm:explaw}. This enables differentiability properties of mappings to be verified on the underlying manifolds. 

A similar notion of canonical manifold exists also for spaces of finitely often differentiable mappings, \cite{AaGaS18,GaS21}.
We hasten to remark that the usual constructions of manifolds of mappings yield canonical manifold structures (cf.\ \Cref{App:canmfdmap}).
\end{rem}

\begin{lem}\label{base-cano}
If $C^\infty (K,M)$ is endowed with a canonical manifold structure, then
\begin{enumerate}
\item the evaluation map $\ev\colon C^\infty(K,M)\times K\rightarrow M$,
$\ev(\gamma,x)\coloneq\gamma(x)$ is a $C^{\infty}$-map.
\item
Canonical manifold structures are unique:
Writing $C^\infty(K,M)'$ for $C^\infty(K,M)$ with another canonical manifold structure,
then $\id\colon C^\infty(K,M)\rightarrow C^\infty(K,M)'$, $\gamma\mapsto \gamma$
is a $C^\infty$-diffeomorphism.
\item
Let $N\subseteq M$ be a submanifold such that the set $C^\infty(K,N)$
is a submanifold of $C^\infty(K,M)$.
Then the submanifold structure on $C^\infty(K,N)$ is canonical.
\item
If $M_1$, and $M_2$ are smooth manifolds such that $C^\infty(K,M_1)$ and $C^\infty(K,M_2)$ have canonical
manifold structures, then the manifold structure
on the product manifold $C^\infty(K,M_1)\times C^\infty(K,M_2)$
$\cong  C^\infty(K,M_1\times M_2)$
is canonical.
\end{enumerate}
\end{lem}
\begin{proof}
\begin{enumerate}
 \item Since $\id\colon C^\infty(K,M)\rightarrow C^\infty(K,M)$ is~$C^\infty$ and $C^\infty(K,M)$
is endowed with a canonical manifold structure, it follows that $\id^\vee\colon
C^\infty(K,M)\times K\rightarrow M$, $(\gamma,x)\mapsto \id(\gamma)(x)=\gamma(x)=\ev(\gamma,x)$
is $C^{\infty}$.
 \item The map $f\coloneq \id\colon C^\infty(K,M)\rightarrow C^\infty(K,M)'$ satisfies
$f^\wedge=\ev$ and is thus $C^{\infty}$, by~(a). Since $C^\infty(K,M)'$ is endowed with a canonical manifold structure,
it follows that~$f$ is~$C^\infty$. By the same reasoning, $f^{-1}=\id\colon C^\infty(K,M)'\rightarrow C^\infty(K,M)$ is~$C^\infty$.
\item As $C^\infty(K,N)$ is a submanifold, the inclusion $\iota\colon C^\infty(K,N)\rightarrow C^\infty(K,M)$, $\gamma\mapsto\gamma$ is~$C^\infty$ (cf.\ \Cref{lem:submfd:initial}).
Likewise, the inclusion map $j\colon N\rightarrow M$ is~$C^\infty$. Let $L$ be a manifold and $f\colon L\rightarrow C^\infty(K,N)$ be a map.
If~$f$ is smooth, then $\iota\circ f$ is smooth, entailing that
$(\iota\circ f)^\wedge \colon L\times K\rightarrow M$, $(x,y)\mapsto f(x)(y)$ is~$C^{\infty}$.
As the image of this map is contained in~$N$, which is a submanifold of~$M$,
we deduce that $f^\wedge=(\iota\circ f)^\wedge|^N$ is~$C^{\infty}$.
If, conversely, $f^\wedge\colon L\times K\rightarrow N$ is $C^{\infty}$,
then also $(\iota\circ f)^\wedge=j\circ (f^\wedge)\colon L\times K\rightarrow M$ is $C^{\infty}$.
Hence $\iota\circ f\colon L\rightarrow C^\infty(K,M)$ is~$C^\infty$ (the manifold structure on the range being
canonical). As $\iota\circ f$ is a $C^\infty$-map map with image in $C^\infty(K,N)$ which is a submanifold
of $C^\infty(K,M)$, we deduce that~$f$ is~$C^\infty$.
\item If~$L$ is a manifold and $f=(f_1,f_2)\colon
L\rightarrow C^\infty(K,M_1)\times C^\infty(K,M_2)$
a map, then~$f$ is~$C^\infty$ if and only if~$f_1$ and $f_2$ are~$C^\infty$.
As the manifold structures are canonical, this holds if and only if
$f_1^\wedge\colon L\times K\rightarrow M_1$ and $f_2^\wedge\colon L\times K\rightarrow M_2$
are $C^{\infty}$, which holds if and only if $f^\wedge=(f_1^\wedge,f_2^\wedge)$
is $C^{\infty}$. \qedhere
\end{enumerate}
\end{proof}

\begin{prop}\label{fstar-gen}
Assume that $C^\infty(K,M)$ and $C^\infty(K,N)$ admit canonical manifold structures.
If $\Omega\subseteq K\times M$ is an open subset and $f\colon \Omega \rightarrow N$ is a $C^{\infty}$-map,
then $\Omega'\coloneq\{\gamma\in C^\infty(K,M)\mid \{(k,\gamma (k)) , k \in K\} \subseteq\Omega\}$
is an open subset of $C^\infty(K,M)$ and
\begin{displaymath}
f_\star\colon \Omega' \rightarrow C^\infty(K,N),\quad \gamma\mapsto f\circ (\id_K,\gamma)
\end{displaymath}
is a $C^\infty$-map.
\end{prop}

\begin{proof}
In Exercise \ref{Ex:cptopn} it was proved that $\Omega'$ is open in $C(K,M)_{\text{c.o}}$, whence it is open in the finer compact open $C^\infty$-topology. By \Cref{base-cano} (a), the evaluation map $\ev\colon C^\infty(K,M)\times K\rightarrow M$ is $C^{\infty}$, whence $C^\infty(K,M)\times K\rightarrow K\times M$, $(\gamma,x)\mapsto (x,\gamma(x))$
is $C^{\infty}$. Since $f$ is $C^{\infty}$, the Chain Rule shows that
\begin{displaymath}
(f_\star)^\wedge\colon \Omega' \times K\rightarrow N, \quad (\gamma,x)\mapsto f_\star(\gamma)(x)=f(x,\gamma(x))=f(x,\ev(\gamma,x))
\end{displaymath}
is $C^{\infty}$. So $f_\star$ is~$C^\infty$,
as the manifold structure on $C^\infty(K,N)$ is canonical.
\end{proof}

\begin{cor}\label{fstar-basic}
Assume that $C^\infty(K,M)$ and $C^\infty(K,N)$
admit canonical manifold structures.
If $f\colon K\times M\rightarrow N$ is a $C^{\infty}$-map,
then we obtain a smooth map 
\begin{displaymath}
f_\star\colon C^\infty(K,M)\rightarrow C^\infty(K,N),\quad \gamma\mapsto f\circ (\id_K,\gamma).
\end{displaymath}
\end{cor}
Applying Corollary~\ref{fstar-basic} with $f(x,y)\coloneq (y)$, we get:
\begin{cor}\label{la-reu}
Assume that $C^\infty(K,M)$ and $C^\infty(K,N)$
admit canonical manifold structures.
If $g\colon M\rightarrow N$ is a $C^{\infty}$-map,
then the \emph{pushforward}\index{pushforward} is smooth
\begin{displaymath}
g_*\colon C^\infty(K,M)\rightarrow C^\infty(K,N),\quad \gamma\mapsto g\circ \gamma.
\end{displaymath}
\end{cor}

To construct manifold structures on $C^\infty(K,M)$ one needs an additional structure on $M$. This so called \emph{local addition} replaces the vector space addition not present on $M$.
\begin{defn}
Let $M$ be a smooth manifold.
A \emph{local addition}\index{local addition} is a smooth map
\begin{displaymath}
\Sigma \colon U \rightarrow M,
\end{displaymath}
defined on an open neighborhood $U \opn TM$ of the \emph{zero-section}\index{zero-section} of the tangent bundle
$\mathbf{0}_M\coloneq \{0_p\in T_pM\mid p\in M\}$
such that $\Sigma(0_p)=p$ for all $p\in M$,
\begin{displaymath}
U'\coloneq\{(\pi_{M}(v),\Sigma(v))\mid v\in U\}
\end{displaymath}
is open in $M\times M$ and $\theta\coloneq (\pi_{TM},\Sigma)\colon U \rightarrow U'$ is a diffeomorphism.
\end{defn}
If $C^\infty (M,N)$ is canonical and we interpret a tangent vector as an equivalence class of smooth curves $[t\rightarrow c(t)]$, with $c\colon ]-\varepsilon,\varepsilon[ \rightarrow C^\infty (M,N)$, the derivative of $c$ can be identified with the partial derivative of the adjoint map $c^\wedge \colon ]-\varepsilon,\varepsilon[ \times M \rightarrow N$. This shows that as a set we should have $TC^\infty (M,N) \cong C^\infty (M,TN)$. In the presence of a local addition, the set $C^\infty (M,TN)$ turns also into a canonical manifold and the bijection becomes an isomorphism of vector bundles over the identity. Summing up, this identification yields the following result.

\begin{prop}\label{prop: can:locadd}
If~$M$ admits a local addition\footnote{One can show that every paracompact strong Riemannian manifold (see \Cref{RiemGeo}) and thus every finite-dimensional paracompact manifold, admits a local addition. Moreover, Lie groups, cf.\ \Cref{Chap:Liegp}, admit local additions, \Cref{setup:locadd:Lie}.}, then $C^\infty(K,M)$
admits a canonical manifold structure and the tangent bundle can be identified with $C^\infty (K,TM)$.
\end{prop}
We refer to \Cref{App:canmfdmap} for more information on the proof.
\begin{setup}\label{thetamp}
Assume that $M,N$ admit local additions and $f \colon M \rightarrow N$ is a $C^{\infty}$ map. Then the identification $TC^\infty (K,M)\cong C^\infty (K,TM)$ induces a commuting diagram (Exercise \ref{ex:canmfd} 6.):
\begin{equation}\label{eq:commdiag}
\begin{tikzcd}
TC^{\infty} (K,M) \arrow[r, "\cong"] \arrow[d, "T(f_*)"]
& C^\infty (K,TM)  \arrow[d, "(Tf)_*"]  \\
TC^\infty (K,N)  \arrow[r, "\cong"] & C^\infty (K,TN). 
\end{tikzcd}
\end{equation} 

\end{setup}

We have seen in \Cref{la-reu} and Exercise \ref{ex:canmfd} 1. that the pushforward and the pullback of smooth functions are smooth with respect to canonical manifolds of mappings. Viewing these mappings as partial mappings of the full composition map 
$$\Comp \colon C^\infty (K,M) \times C^\infty (L,K) \rightarrow C^\infty (L,M),\quad (f,g) \mapsto f\circ g,$$
we see that the full composition is separately smooth in its variables. This immediately prompts the question as to whether the full composition map is smooth. In the general case (of a possibly non-compact source manifold) when one has no exponential law available this is complicated, but in our situation it reduces to an easy observation

\begin{prop}\label{smooth:fullcomp}
Let $K,L$ be compact manifolds and assume that the manifolds $C^\infty (K,M),C^\infty (L,M)$ are canonical ($C^\infty (L,K)$ is automatically a canonical manifold as $L$ admits a local addition) then the composition map
$$\Comp \colon C^\infty (K,M) \times C^\infty (L,K) \rightarrow C^\infty (L,M),\quad (f,g) \mapsto f\circ g,$$
is smooth.
\end{prop}

\begin{proof}
 By the exponential law for canonical manifolds, $\Comp$ is smooth if and only if the adjoint map 
 $$\Comp^\wedge \colon C^\infty (K,M) \times C^\infty (L,K) \times L \rightarrow M,\quad (f,g,l) \mapsto f(g(l))$$
 is smooth. However, this shows that $\Comp^\wedge (f,g,l)=\ev(f,\ev(g,l))$ and since the evaluation mappings are smooth for canonical manifolds, also the adjoint map and thus the composition are smooth.
\end{proof}

We have already seen that certain properties "lift to the manifold of mappings". For example, if $f\colon M \rightarrow N$ is smooth, the pushforward $f_* \colon C^\infty(K,M) \rightarrow C^\infty (K,N), g \mapsto f\circ g$ is smooth. Another example of this is the following result whose proof is remarkably involved and technical (we omit the proof here and pose it as Exercise \ref{ex:canmfd} 7.):

\begin{lem}[{Stacey-Roberts Lemma \cite[Lemma 2.4]{AaS19}}]\label{lem:SR}
 Let $p \colon M \rightarrow N$ be a submersion between finite-dimensional manifolds. Endowing the function spaces with their canonical manifold structure, the pushforward $p_* \colon C^\infty (K,M) \rightarrow C^\infty(K,N)$ becomes a submersion.\index{Stacey-Roberts Lemma}
\end{lem}

In the next chapters we will study other structures from differential geometry which can be lifted from finite dimensions to spaces of differentiable functions. For Lie groups this leads to the so called current groups (whose most prominent examples are the loop groups). In the context of Riemannian geometry, the lifting procedure gives rise to the $L^2$-metric and more generally to the Sobolev type Riemannian metrics on function spaces. Some examples of Sobolev type metrics will be discussed in \Cref{sect:L2} and \Cref{sect:EAtheory}.

\begin{Exercise}\label{ex:canmfd}   \vspace{-\baselineskip}
 \Question Let $h \colon L \rightarrow K$ be a smooth map. Assume that $C^\infty (K,M)$ and $C^\infty (L,M)$ are canonical manifolds.
 \subQuestion Show that the pullback $h^* \colon C^\infty (K,M) \rightarrow C^\infty (L,M),\quad f \mapsto f\circ h$ is smooth.\index{pullback}
 \subQuestion Assume that $K,L$ are compact and $M$ admits a local addition. Then we identify $TC^\infty (K,M) \cong C^\infty (K,TM)$ (cf.\ \Cref{setup:curves:tan}). Show that this identifies $T(h^*)$ with $h^* \colon C^\infty (K,TM) \rightarrow C^\infty (L,TM)$.
 \Question Let $K$ be a compact manifold and $O\opn E$ in a locally convex space. Prove that, $C^\infty (K,O) \opn C^\infty (K,E)$ (\Cref{prop:lcvx_mappingsp}) is a canonical manifold.
 \Question Let $M$ be a finite-dimensional manifold and $K$ be a compact manifold. Endow $C^\infty (K,M)$ with the canonical manifold structure from \Cref{App:canmfdmap}.
  \subQuestion Show that for $x\in K$ the point evaluation $\ev_x \colon C^\infty (K,M) \rightarrow M, \gamma \mapsto \gamma(x)$ is a submersion.\footnote{It is also possible to prove that the evaluation map $\ev \colon C^\infty (K,M) \times K \rightarrow M, (\gamma,x) \mapsto \gamma(x)$ is a submersion, cf.\ \cite[Corollary 2.9]{SaW16}.}
  \subQuestion Deduce that the set $S(x,y) \coloneq \{f \in C^\infty (K,M) \mid f(x)=y\}$ for some fixed $x \in K, y \in M$ is a split submanifold of $C^\infty (K,M)$.
  \subQuestion Is the set $\bigcap_{1\leq i \leq n} S(x_i,y_i)$ also a submanifold of $C^\infty (K,M)$ if we pick points $x_i \in K, y_i \in M$ for $1\leq i \leq n$ and $n \in \N$?
   \Question Consider a compact manifold $K$ and $M$ a manifold with a local addition. We endow $C^\infty (K,M)$ with the canonical manifold structure induced by the local addition, see \Cref{app:constmfd}. 
 Compute the tangent map of the evaluation map $$\ev \colon C^\infty (K,M) \times K \rightarrow M,\quad (\varphi,m) \mapsto \varphi(m).$$
 {\tiny \textbf{Hint:} Apply the rule on partial differentials, Exercise \ref{Ex:tangentmaps} 3. To compute the derivative for the variable in $C^\infty (K,M)$ exploit that $TC^\infty (K,M)\cong C^\infty (K,TM)$, \Cref{setup:curves:tan}. After choosing a smooth curve $c \colon ]-\varepsilon ,\varepsilon[ \rightarrow C^\infty(K,M)$ apply the exponential law to carry out the computation.} 
 \Question Assume that $K,L$ are compact and $M$ admits a local addition. Compute a formula for the tangent map of the smooth map (cf.\, \Cref{smooth:fullcomp}) 
 $$\Comp \colon C^\infty (L,M) \times C^\infty (K,L) \rightarrow C^\infty (K,M),\quad (g,f) \mapsto g \circ f.$$
 \Question Use the identification $TC^\infty (M,N) \ni [t\mapsto c] \mapsto (x\mapsto \frac{\partial}{\partial t} c^\wedge (t,x)) \in C^\infty (M,TN)$ to establish the commutativity of the diagram \eqref{eq:commdiag}.
 \Question Let $K$ be a compact manifold and $p \colon M \rightarrow N$ be a submersion between finite-dimensional paracompact manifolds. Establish the Stacey-Roberts Lemma by showing that the pushforward $p_* \colon C^\infty (K,M) \rightarrow C^\infty(K,N)$ becomes a submersion.\\
 {\footnotesize \textbf{Hint:} This is an involved exercise in \emph{finite dimensional} geometry which should only be attempted if one is familiar with Riemannian exponential maps, parallel transport and horizontal distributions. The idea is to construct a horizontal distribution $\mathcal{H}$ together with local additions $\eta_M$, $\eta_N$ (constructed from suitable Riemannian exponential maps) such that the following diagram commutes:}
  \begin{displaymath}
  \begin{tikzcd}
 TM =  \mathcal{V} \oplus\mathcal{H} \ar[d,"0 \oplus Tp|_{\mathcal{H}}"] &   \ar[l,"\supseteq"]\Omega_M  \ar[r,"\eta_M"]  & X \ar[d,"p"]  \\
		TN &  \ar[l,"\supseteq"] \Omega_N  \ar[r,"\eta_N"]  &   M
  \end{tikzcd}
  \end{displaymath}
{\footnotesize  Using these local additions, the canonical charts of the manifold of mappings become submersion charts.}
\end{Exercise}

\begin{Answer}[number={\ref{ex:canmfd} 1.}] 
 \emph{Let $h \colon L \rightarrow K$ be a smooth map. Assume that $C^\infty (K,M)$ and $C^\infty (L,M)$ are canonical manifolds. We prove that \begin{enumerate}
 \item The pullback $h^* \colon C^\infty (K,M) \rightarrow C^\infty (L,M),\quad f \mapsto f\circ h$ is smooth.
 \item If $K,L$ are compact and $M$ admits a local addition, then $TC^\infty (K,M) \cong C^\infty (K,TM)$ (cf.\ \Cref{setup:curves:tan}). This identifies $T(h^*)$ with $h^* \colon C^\infty (K,TM) \rightarrow C^\infty (L,TM)$.
 \end{enumerate}}
 (a) The pullback is a partial map of the full composition, whence smooth by \Cref{smooth:fullcomp}. Alternatively, smoothness follows directly from the exponential law, as $h^*$ is smooth if and only if the adjoint $(h^*)^\wedge \colon C^\infty(K,M) \times L \rightarrow M, (f,\ell) \mapsto f(h(\ell)) = \ev(f,h(\ell))$ is smooth. Since $C^\infty (K,M)$ is canonical,  \Cref{base-cano} shows that the evaluation is smooth. Now smoothness of the adjoint follows, since $\ev$ and $h$ are smooth.\\
 (b) Note that we only need the assumptions to identify $TC^\infty (K,M) \cong C^\infty (K,TM)$. To compute the tangent, we pick $c \colon ]-\varepsilon, \varepsilon[ \rightarrow C^\infty (K,M)$ smooth with $c(0)=f$ and $\dot{c}(0)=V_f$. Under the identification we can interpret $V_f(x) = \left.\frac{\partial}{\partial t}\right|_{t=0} c^\wedge (t,x)$ as a function $K \rightarrow TM$. Now 
 $$Th^* (V_f)(x) = \left.\frac{\partial}{\partial t}\right|_{t=0} h^*(c^\wedge (t,x)) = \left.\frac{\partial}{\partial t}\right|_{t=0} c^\wedge (t,h(x)) = V_f(h(x)) = h^* (V_f).$$
 Thus we have identified the tangent map as $h^* \colon C^\infty (K,TM) \rightarrow C^\infty (L,TM)$.
\end{Answer}
\begin{Answer}[number={\ref{ex:canmfd} 4}] 
 \emph{For $K$ a compact manifold and $M$ a manifold with local addition we endow $C^\infty (K,M)$ with its canonical manifold structure and compute the tangent map of the evaluation $\ev \colon C^\infty (K,M) \times K \rightarrow M$.}
 \\[.15em]
 
 We apply the rule on partial differentials for manifolds (Exercise \ref{Ex:tangentmaps} 6.):
 $$T_{(\varphi,k)}\ev (v_\varphi,v_k) = T_\varphi \ev(\cdot,k)(v_\varphi) + T_k\ev(\varphi,\cdot)(v_k)$$
 To evaluate the first term, pick a curve $c \colon ]-\varepsilon,\varepsilon[ \rightarrow C^\infty (K,M)$ with $c(0)=\varphi$ and $\dot{c}(0)=v_\varphi$. If we identify $TC^\infty (K,M)\cong C^\infty (K,TM)$ via \Cref{setup:curves:tan} we can interpret $v_\varphi$ as the smooth mapping $\left.\frac{\partial}{\partial t}\right|_{t=0} c^\wedge \colon K \rightarrow TM$. Then we compute
 $$T_\varphi \ev(\cdot,k)(v_\varphi) = \left.\frac{\mathrm{d}}{\mathrm{d}t}\right|_{t=0} \ev(c(t),k)=\left.\frac{\partial}{\partial t}\right|_{t=0}c^\wedge (t,k) = v_\varphi(k) = \ev_k (v_\varphi).$$
 Here $\ev_k \colon C^\infty (K,TM) \rightarrow TM$ is the evaluation in $k$ and we have exploited the exponential law in the computation.
 For the second term in the sum, it is immediately clear that we get $T_k \varphi (v_k)$.
 Thus we get as a formula for the tangent mapping
 $$T\ev (v_\varphi,v_k) = \ev_k (v_\varphi) + T\varphi (v_k).$$
\end{Answer}

\chapter{Lifting geometry to mapping spaces I: Lie groups}\label{Chap:Liegp} \copyrightnotice

In this chapter, one aim is to study spaces of mappings taking their values in a Lie group. It will turn out that these spaces carry again a natural Lie group structure. However, before we prove this, let us recall the definition and basic properties of (infinite-dimensional) Lie groups.  

\section{(infinite-dimensional) Lie groups}\label{sect:lGP}
Our presentation of Lie groups modelled on infinite-dimensional spaces mostly follows \cite{Neeb06}. There are many literature accounts for finite-dimensional Lie theory (see for example \cite{HaN12}), infinite-dimensional Lie theory (beyond Banach spaces) is by comparison relatively young and in its modern form goes back to Milnor's seminal work \cite{Mil82,Mil84}.

\begin{defn}
 A  \emph{(locally convex) Lie group}\index{Lie group} $G$ is a manifold $G$ modelled on a locally convex space endowed with a group structure such that the multiplication map $m_G \colon G \times G \rightarrow G$ and the inversion map $\iota \colon G \rightarrow G$ are smooth.
 A morphism of Lie groups is a smooth group homomorphism. In the following we shall drop the adjective ''locally convex'' and simply say Lie group
\end{defn}

\begin{tcolorbox}[colback=white,colframe=blue!75!black,title=Standard notation]
Let us fix some standard notation for objects occurring frequently in conjunction with Lie groups. Let $G$ be a Lie group, we shall write 
\begin{itemize}
 \item $\one_G$ for the unit element (or shorter $\one$),
 \item $m_G$ for multiplication, $\iota_G$ for inversion,
 \item for $ g\in G$ we let $\lambda_g \colon G \rightarrow G, h \mapsto gh$ and $\rho_g \colon G \rightarrow G, h \mapsto hg$ the \emph{left (right) translation}\index{Lie group!left (right) translation}. (Observe that $\lambda_g (\rho_h(x))=gxh=\rho_h(\lambda_g (x)).$)
\end{itemize}
\end{tcolorbox}

\begin{ex}
 A locally convex space $E$ is a Lie group with respect to vector addition and the usual manifold structure.
\end{ex}

\begin{ex}\label{ex:findimLIE}
The following examples are the classical finite-dimensional examples encountered in a first course on Lie theory. We include them here for readers who are not familiar with Lie groups.
 \begin{enumerate}
  \item Let $M_n (\R)$ be the set of all $n \times n$ matrices and $\text{Gl}_n (\R) \coloneq \{A \in M_n(\R) \mid \det A \neq 0\}$ be the set of $n\times n$ invertible matrices. Using the determinant, one sees that $\text{Gl}_n (\R) \subseteq M_n (\R) \cong \R^{n^2}$ is an open subset, whence a manifold. Since multiplication of matrices is given by polynomials in the entries of matrices, the multiplication is smooth with respect to the manifold structure. For invertible matrices, Cramer's rule shows that inversion is also polynomial in the entries of the matrix, whence smooth. In conclusion, matrix multiplication and inversion turns $\text{Gl}_n (\R)$ into a Lie group.\footnote{Alternatively, this example can be seen as a special case of the unit group of a CIA, see \Cref{ex:CIAunit}.} 
  \item The \emph{orthogonal group} $\text{O}_n(\R) \coloneq \{A \in \text{Gl}_n(\R) \mid AA^\top = \id_{R^n}\}$ is a closed submanifold of $\text{Gl}_n (\R)$ and this structure turns it into a Lie group. Further, the \emph{special orthogonal group} $\text{SO}_n (\R) \coloneq \{A \in \text{O}_n (\R) \mid \det (A) =1\}$ is an open subset of the orthogonal group and thus also a Lie group.
  \item The unit circle $\mathbb{S}^1 \subseteq \R^2$ is a submanifold (as the unit sphere of the Hilbert space $\R^2$). Identifying $\R^2 \cong \bC$, complex multiplication induces a Lie group structure on $\mathbb{S}^1$ which is explicitly given by the formulae 
  $$(x,y)\cdot(a,b) \coloneq (xa-yb,xb+ay), \quad (x,y)^{-1} \coloneq (x,-y).$$
 \end{enumerate}
\end{ex}

\begin{ex}[Unit groups of continuous inverse algebras]\label{ex:CIAunit}
 To generalise the matrix group example to infinite-dimensions, we recall the notion of a \emph{continuous inverse algebra (CIA)}\index{continuous inverse algebra}\index{CIA!see continuous inverse algebra}: Let $A$ be a locally convex space with a continuous bilinear map $\beta \colon A \times A \rightarrow A$ (we write shorter $xy\coloneq \beta(x,y)$) such that the associativity law $(xy)z=x(yz)$ holds. Furthermore, we assume that there exists an element $\one \in A$ such that $\one x =x= x\one , \forall x \in A$ and define the set of all invertible elements
 $$A^\times \coloneq \{x\in A \mid \exists x^{-1}\in A, \text{ such that } xx^{-1}=\one = x^{-1}x\}.$$
 Then $A^\times$ is a group, under the multiplication, called the \emph{unit group}\index{unit group (of an algebra)} of $A$.
 If $A^\times$ is open in $A$ and inversion $\iota \colon A^\times \rightarrow A^\times, x\mapsto x^{-1}$ is continuous, we call $A$ a \emph{continuous inverse algebra (CIA)}.
 The unit group of a CIA is a Lie group, see Exercise \ref{Ex:Lgpex} 2. \index{continuous inverse algebra!unit group}
 
 CIAs generalise Banach algebras, e.g.\, the algebra of continuous linear operators $L(E,E)$ of a Banach space $E$ with the operator norm $\lVert \cdot \rVert_{\text{op}}$ is a CIA.
\end{ex}

Before we continue with the general infinite-dimensional Lie theory, we will discuss now one of the most important example classes of such Lie groups: the diffeomorphism groups. These groups will return as a running example in the later sections to illustrate the concepts of Lie algebra and regularity.

\begin{ex}\label{ex:Diffgp}
 Let $M$ be a compact manifold. Then $M$ possesses a local addition, whence $C^\infty (M,M)$ is a canonical manifold and  by \Cref{smooth:fullcomp} the composition map $\Comp \colon C^\infty (M,M) \times C^\infty (M,M) \rightarrow C^\infty (M,M)$ is smooth. Recall from \Cref{cor:diffopn} that the \emph{set of diffeomorphisms}\index{diffeomorphism group} $\Diff (M)$ is an open subset of $C^\infty (M,M)$ which forms a group under composition of smooth maps. 
 Hence $\Diff (M)$ is an open submanifold of $C^\infty (M,M)$ and the group product is smooth with respect to this structure.  
 
 We will now prove that inversion $\iota \colon \Diff (M) \rightarrow \Diff (M)$ is smooth. Applying the exponential law for a canonical manifold, $\iota$ is smooth if and only if the mapping $\iota^\wedge \colon \Diff (M) \times M \rightarrow M, (\varphi,m) \mapsto \varphi^{-1}(m)$ is smooth.
 Consider the implicit equation 
 \begin{align}\label{imp:eq}
  \ev (\phi, \iota^\wedge (\phi,m))=\phi (\iota^\wedge (\phi ,m)) = m,\quad \forall m \in M .
 \end{align}
 which takes values in a finite-dimensional manifold, but has an infinite-dimensional parameter ($\phi \in \Diff (M)$). However, $\ev \colon \Diff (M) \times M \rightarrow M$ is smooth, and we can compute its partial differential as $T_{(\phi,x)}\ev (0,z) = T\phi (z)$ (cf.~Exercise \ref{ex:canmfd} 4.). Since $\phi$ is a diffeomorphism, we see that the partial derivative of $\ev$ is indeed invertible for every $\phi \in \Diff (M)$. Now smoothness of $\iota^\wedge$ follows from a suitable implicit function theorem. Observe that due to the infinite-dimensional parameter, the usual implicit function theorem \cite[I. \S 5 Theorem 5.9]{Lang} is not applicable to \eqref{imp:eq}! However, we invoke the generalised implicit function theorem \cite[Theorem 2.3]{MR2269430} which can deal with parameters in locally convex spaces (as long as the target of the implicit equation is a Banach manifold). The usual application of the implicit function theorem to \eqref{imp:eq} shows then that $\iota^\wedge$ and thus $\iota$ is smooth. We conclude that $\Diff (M)$ is a Lie group. 
\end{ex}

\begin{rem}
 To establish smoothness of the inversion $\Diff (M)$ we needed the exponential law and a generalised implicit function theorem. Use of this machinery can be avoided: In \cite[Theorem 11.11]{Mic} differentiability of the inversion is directly verified (which is technical and requires the (non-trivial) verification of continuity first). 
 Our approach is inspired by the proof in the convenient setting \cite[Theorem 43.1]{KM97}. There the problem can be reduced to a finite-dimensional equation which circumvents the need for a generalised implicit function theorem. 
\end{rem}

 The Lie group $\Diff (M)$ comes already with a canonical action on the manifold $M$ which we describe after recalling the notion of a Lie group action.
 
\begin{defn}[Lie group action]
Let $M$ be a manifold and $G$ be a Lie group.\index{Lie group action} Then a smooth map  
$$\alpha \colon G \times M \rightarrow M, (g,m) \mapsto \alpha (g,m) \equalscolon g.m$$
is called a \emph{(left) Lie group action} if it satisfies the following
$$ \alpha (\one_G,m) =m, \qquad \alpha (g_1,\alpha(g_2,m)) = \alpha (g_1g_2,m), \quad \forall g_1,g_2 \in G, m \in M$$
A right action is a smooth map $\beta \colon M \times G \rightarrow M, (m,g) \mapsto \beta(g,m) \equalscolon m.g$ such that 
 $$\beta(m,\one_G) =m, \qquad \beta(\beta(m,g_1),g_2) = \beta(m,g_1g_2), \quad \forall g_1,g_2\in G, m\in M.$$
 If $\beta$ is a right action, then $\alpha \coloneq \beta \circ (\id_M , \iota)$ is a left action. Similarly, we can obtain right actions from left actions and there is no essential difference between both notions. 
\end{defn}

\begin{ex}\label{ex:candiffact}
 Let $M$ be a compact manifold. Then, going through the construction in \Cref{ex:Diffgp}, it is immediately clear that the evaluation map 
 $$\alpha \colon \Diff(M) \times M \rightarrow M,\quad  (\varphi,m) \mapsto \varphi(m)$$
 induces a (left) Lie group action, called the \emph{canonical action of the diffeomorphism group}.\index{diffeomorphism group!canonical action}
 Furthermore, there is also the right action of $\Diff (M)$ on smooth functions
 $$\beta \colon C^\infty (M,N) \times \Diff (M) \rightarrow C^\infty (M,N), \quad (f,\varphi) \mapsto \varphi^* (f)=f\circ \varphi.$$
 If $C^\infty (M,N)$ is a canonical manifold of mappings, then \Cref{smooth:fullcomp} shows that $\beta$ is a (right) Lie group action. The right action $\beta$ is connected to several geometric structures such as the symplectic structure of the loop space $C^\infty (\SSS^1,N)$, see \cite{Wurz95}. We will encounter it again in the context of shape analysis in \Cref{sect:shapeanalysis}.
\end{ex}

 For a left Lie group action $\alpha$ the canonical map
 $$\alpha^\vee \colon G \rightarrow \Diff (M),\quad g \mapsto \alpha (g,\cdot)$$
 makes sense and yields a group morphism. If $M$ is finite-dimensional, the exponential law shows that smoothness of the group action is equivalent to smoothness of $\alpha^\vee$. This breaks down for an infinite dimensional manifold $M$ as there is no smooth structure on $\Diff (M)$. Similarly, if a Lie group $G$ acts by linear mappings on a vector space $E$ (this is called a representation of $G$) the literature considers the smoothness of $\alpha^\vee$ as a mapping to $\text{Aut} (E)$ (the group of linear automorphisms of $E$). If $E$ is normable, the group $\text{Aut}(E)$ inherits a canonical Lie group structure from the operator norm topology such that smoothness of $\alpha$ is equivalent to smoothness of $\alpha^\vee$. Again this equivalence breaks down for locally convex spaces which are not normable as $\text{Aut} (E)$ does in general not carry a Lie group structure. Note that this is not a serious problem, as one can still check smoothness with respect to the product $G \times M$ (and in infinite-dimensional representation theory of Lie groups even weaker concepts of smoothness of representations are more appropriate for the theory, see e.g.\ \cite{Neeb10} and cf.\ also \cite[I.3.4]{NeeMon}. However, we shall not discuss representation theory and the finer points of these problems in this book.  

\begin{defn}\label{defn:liesubgp}
 Let $G$ be a Lie group. We call a submanifold $H \subseteq G$ a \emph{Lie subgroup}\index{Lie group!Lie subgroup} if it is a subgroup of $G$.\footnote{Note that due to \Cref{lem:submfd:initial} this structure turns $H$ into a Lie group.} If $H$ is in addition closed in $G$, we call $H$ closed Lie subgroup.
\end{defn}

\begin{ex}\label{ex:geom_subgroup}
 The diffeomorphism group $\Diff (M)$ of a compact manifold $M$ contains many important subgroups of diffeomorphisms which preserve geometric structures. If $\omega$ is a differential form on $M$, we say \emph{a diffeomorphism} $\phi \in \Diff (M)$ \emph{preserves the differential form},\index{differential form!preserved by a diffeomorphism} if $\phi^\ast \omega = \omega$ (where the pullback is defined as in \Cref{defn:pullback}). As the pullback commutes with function composition, we can consider the subgroup $\Diff_\omega (M)$ of diffeomorphisms preserving a given differential form $\omega$. The most important examples are the following subgroups:  
 \begin{enumerate}
  \item if $\omega = \mu$ is a volume form on $M$, we obtain the group $\Diff_\mu (M)$ of \emph{volume preserving diffeomorphims},\index{diffeomorphism group!volume preserving}
  \item for a symplectic form $\omega$, this yields the \emph{group of symplectomorphism},\index{diffeomorphism group!symplectomorphisms}
  \item $\Diff_\theta (M) = \{\phi \in \Diff (M) \mid \phi^\ast \theta = f \theta \text{ for some } f \in C^\infty (M,\R \},$ the \emph{group of contactomorphisms} for a contact form $\theta$.\index{diffeomorphism group!contactomorphisms}
 \end{enumerate}
In the three cases mentioned above, one can show that the subgroups are also submanifolds of $\Diff (M)$ and thus Lie subgroups of $\Diff (M)$. We refer to \cite[Section 3]{Smo07} for detailed proofs. See however, \Cref{volumepreser} for a sketch of the construction for volume preserving diffeomorphisms.
\end{ex}

In finite-dimensional Lie theory a useful result states that every closed subgroup of a finite-dimensional Lie group is a Lie subgroup, \cite[Theorem 9.3.2]{HaN12}. This is no longer true in infinite-dimensions as the next example shows (cf.\, also \cite[Remark V.2.4,(c)]{NeeMon}).

\begin{ex}[{Wockel, \cite{WocInf}}]\label{closedsubgrp}
 Consider the space $(\ell^2,+)$ of all real sequences which are square summable, \cite[Example 12.11]{MaV97}. This is a Hilbert space with respect to the inner product 
 $$\langle (x_n)_{n\in \N}, (y_n)_{n\in \N}\rangle \coloneq \sum_{n\in \N} x_ny_n \quad \left(\text{and norm } \lVert (x_n)_{n\in \N}\rVert =\sqrt{\sum_{n\in \N} x_n^2}\right).$$ 
 We consider $\ell^2$ as an abelian Lie group and define the subgroup 
 $$H \coloneq \left\{(x_n)_{n \in \N} \in \ell^2 \middle| x_n \in \frac{1}{n} \mathbb{Z}, n \in \N\right\}.$$
 As the projections $\pi_j \colon\ell^2 \rightarrow \R, (x_n)_{n \in \N}\mapsto x_j, j\in \N$ are continuous linear, $H = \bigcap_{n\in \N} \pi_n^{-1} (\tfrac{1}{n}\mathbb{Z})$ is a closed subgroup. However, we shall see in Exercise \ref{Ex:Lgpex} 5. that $H$ is not a submanifold (it is not a manifold with respect to the subspace topology).   
\end{ex} 

For Lie groups, the tangent bundle is again a Lie group and moreover, the tangent bundle is trivial (i.e.\ it splits as a product of a vector space and the base manifold).
\begin{lem}\label{tangentLie}
 Let $G$ be a Lie group. Identify $T(G\times G) \cong TG \times TG$. 
 \begin{enumerate}
  \item The tangent map of the multiplication 
  \begin{align}\notag
   Tm_G \colon T(G\times G) \cong TG\times TG&\rightarrow TG,\\ T_gG \times T_hG \ni (v_g,w_h) &\mapsto Tm_G(v_g,w_h) = T_g\rho_h(v_g)+T_h\lambda_g(w_h)\label{tmult}
  \end{align}
  induces a Lie group structure\index{Lie group!tangent Lie group} on $TG$ with identity element $0_{\one} \in T_{\one}G$ and inversion 
  \begin{align}\label{Tiota}
   T\iota_G \colon TG \rightarrow TG, T_g G \ni v \mapsto -T\rho_{g^{-1}}T\lambda_{g^{-1}}(v)=-T\lambda_{g^{-1}}T\rho_{g^{-1}}(v).
  \end{align}
 The projection $\pi_G \colon TG \rightarrow G$ becomes a morphism of Lie groups with kernel $(T_{\one}G,+)$ and the zero-section $\mathbf{0}\colon G \rightarrow TG, g \mapsto 0_g$ is a morphism of Lie groups with $\pi_G \circ \mathbf{0} = \id_G$.
  \item The map 
  $$\Phi \colon G \times T_{\one} G \rightarrow TG, \quad (g,v) \mapsto g.v \coloneq Tm_G(0_g,v)$$
  is a diffeomorphism.
 \end{enumerate}
\end{lem}

\begin{proof}
 \begin{enumerate}
  \item Since $m_G$ and $\iota_G$ are smooth, the same holds for their tangent maps. The group axioms for $TG$ follow from the ones for $G$ by virtue of the chain rule (and has the claimed unit element). Linearity of the tangent map implies \eqref{tmult}, we leave this formula and \eqref{Tiota} as Exercise \ref{Ex:Lgpex} 3.. From the definition of the zero section and the projection, the morphism properties follow. As a result of \eqref{tmult}, $T_{(\one, \one)}m_G(v_{\one},w_{\one})=v_{\one} + w_{\one}$. Hence the multiplication on the normal subgroup $\text{ker} \pi_G = T_{\one}G$ is the addition.
  \item Since $\Phi=Tm_G (0_g,v)=Tm_G (\mathbf{0}(g),v)$ and the zero-section $\mathbf{0}$ is smooth, smoothness of the multiplication shows that $\Phi$ is smooth. Now a computation shows that $\Phi^{-1}(v)=(\pi_G(v),T_{\pi_G(v)}\lambda_{\pi_G(v)^{-1}}(v))$, whence $\Phi$ is bijective and its inverse is smooth (as inversion and multiplication in $G$ are smooth and the projection $\pi_G$ is smooth).\qedhere
 \end{enumerate}
\end{proof}
\setbox\tempbox=\hbox{\begin{tikzcd}
0 \arrow[rr] & &  N \arrow[rr, hookrightarrow] & & G \arrow[rr, twoheadrightarrow, "p"] && H \ar[rr]&& \one.
\end{tikzcd}}
\begin{rem}
Note that \Cref{tangentLie} shows that $T_{\one}G$ is a normal Lie subgroup of $TG$ and $TG$ is as a Lie group a semidirect product $T_{\one}G \rtimes G$.\index{Lie group!semidirect product}\index{Lie group!split exact sequence}\footnote{A Lie group $G$ with normal Lie subgroup $N$ and Lie subgroup $H$ such that $N\cap H = \{\one_G\}$ is a semidirect product $N \rtimes H$ if there exists a split exact sequence of Lie group homomorphisms  \begin{displaymath}
 \box\tempbox
 \end{displaymath}
 where splitting means that $p|_H = \id_H$. Equivalently, $G \cong N \times H$ (as manifolds) and the group product is given by $(n,h)\cdot (\tilde{n},\tilde{h})=(n \tilde{h}\tilde{n}\tilde{h}^{-1},h\tilde{h})$. See \cite[2.2.2]{HaN12} for more alternative characterisations.} Moreover, instead of left multiplication one can as well use right multiplication to identify the tangent bundle (the two different choices are related by the adjoint action (see \Cref{ex:adj:act} below).
\end{rem}

The tangent space at the identity of a Lie group plays a special r\^{o}le. In the next section this tangent space will be endowed with an additional structure, the Lie bracket. 

\begin{Exercise}\label{Ex:Lgpex}   \vspace{-\baselineskip}
 \Question Verify that $\SSS^1$ is a Lie group with the structure described in \Cref{ex:findimLIE} (c).
 \Question In this exercise we verify that the unit group $A^\times$ of a CIA $(A,\beta)$ forms a Lie group. Note that the multiplication is smooth by Exercise \ref{Ex:Bastiani} 2. Hence it suffices to prove smoothness of inversion.
 \subQuestion Use the identity $b^{-1}-a^{-1}=b^{-1}(a-b)a^{-1}$ to deduce that the differential quotient $d\iota (x;y)$ exists and satisfies $d\iota (x;y)=-x^{-1}yx^{-1}$.
 \subQuestion Use the formula from (a) to prove that $\iota$ is $C^1$ and inductively is $C^k$ for all $k \in \N_0$.
 \Question Work out the missing details in the proof of \Cref{tangentLie}.
  \subQuestion Prove \eqref{tmult} and verify that the tangent maps induce a group structure on $TG$.
  \subQuestion Establish \eqref{Tiota}, i.e. $T_a \iota (v) = -T\lambda_{a^{-1}}T\rho_{a^{-1}} (v)=-T\rho_{a^{-1}} T\lambda_{a^{-1}} (v)$.\\
  {\tiny \textbf{Hint:} Let $\gamma \colon ]-\varepsilon , \varepsilon[ \rightarrow G$ be smooth with $\gamma(0)=a$. Differentiate the relation $\one = \gamma(t)(\gamma(t))^{-1}$.}
  \subQuestion Show that one can obtain a diffeomorphism $TG \cong T_{\one} G \times G$ using right multiplication instead of left multiplication in part (b). 
  \Question Let $(A,\beta)$ be a continuous inverse algebra (CIA) and let $C^\infty (K, A)$ be endowed with the compact open $C^\infty$-topology. 
  \subQuestion Show that then $C^\infty (K,A)$ with the pointwise product is a CIA. 
  \subQuestion Show that $C^\infty (K, A^\times) = C^\infty (K,A)^\times$ and thus the group $C^\infty (K,A^\times)$ with the pointwise product is a Lie group.
 \Question We supply the details to \Cref{closedsubgrp}: Let $H = \{(x_n)_{n \in \N} \in \ell^2 \mid x_n \in \frac{1}{n} \mathbb{Z}, n \in \N\}$. 
 \subQuestion Show that every $0$-neighborhood in the subspace topology of $H$ contains at least one non-zero element.
 \\ {\tiny \textbf{Hint:} It suffices to consider norm balls.}
 \subQuestion Show that there is no $0$-neighborhood in $H$ which contains a continuous path connecting $0$ with a non-zero element. Deduce that $H$ is not locally homeomorphic to an open subset of a locally convex space and thus is not a (sub)manifold.
 \Question Let $G$ be a Lie group, show that the map $L \colon G \times TG \rightarrow TG, \ (g,v_h) \mapsto T\lambda_g (v_h)$ is a left Lie group action. Dually right multiplication yields a right action $R \colon TG \times G \rightarrow TG$. Work out a formula relating $L(g,v_h)$ to $R(v_h,g)$.  
 \end{Exercise}

  \setboolean{firstanswerofthechapter}{true}
\begin{Answer}[number={\ref{Ex:Lgpex} 5.}] 
 \emph{Let $H = \{(x_n)_{n \in \N} \in \ell^2 \mid x_n \in \frac{1}{n} \mathbb{Z}, n \in \N\}$. Then we prove that 
 \begin{enumerate}  
  \item every $0$-neighborhood in the subspace topology of $H$ contains at least one non-zero element.
  \item there is no $0$-neighborhood in $H$ which contains a continuous path connecting $0$ with a non-zero element. Thus $H$ is not a (sub)manifold.
 \end{enumerate}
}
(a) It suffices to consider the intersection of norm balls with $H$, and in particular, we only need to find such elements in $B_{1/m}(0) \cap H$ for $m \in \N$. However, for such a ball, it is clear that the sequence $x^m_n = 0$ if $m\neq n$ and $x^m_m = 1/(2m)$ is contained in the intersection.\\
(b) Assume that there is a $0$-neighborhood in $H$ which is path-connected. Then it contains an element $(x_k)_{k \in \N} \neq 0$. Pick $\ell\in \N$ with $x_\ell \neq 0$. If $c \colon [0,1] \rightarrow H$ is a continuous path connecting $0$ and $(x_k)_{k \in \N}$, then $\pi_\ell \circ c$ is a continuous path in $\R$ connecting $0$ and $x_\ell \neq 0$. Since the path $c$ takes its values in $H$, $\pi_\ell \circ c$ can take only values in a discrete subset of $\R$. Contradiction! Thus there is no zero-neighborhood of $H$ in the subspace topology which is path-connected and $H$ is therefore not locally homeomorphic to a locally convex space. We conclude that it can not be a (sub-)manifold of $\ell^2$.\end{Answer}
\setboolean{firstanswerofthechapter}{false}
  
\section{The Lie algebra of a Lie group}
We associate now to a Lie group a Lie algebra. This construction allows one to reformulate many problems in Lie theory in terms of linear algebra.

\begin{defn}
 A \emph{Lie algebra}\index{Lie algebra} is a vector space $\mathfrak{g}$ together with a \emph{Lie bracket},\index{Lie bracket} i.e.\, a bilinear map $\LB \colon \mathfrak{g} \times \mathfrak{g} \rightarrow \mathfrak{g}$ such that
 \begin{enumerate}
  \item $\LB[x,\LB[y,z]] + \LB[y,\LB[z,x]]+\LB[z,\LB[x,y]] = 0, \forall x,y,z \in \mathfrak{g}$ (Jacobi identity)\index{Lie algebra!Jacobi identity}
  \item $\LB[x,x] = 0, \forall x \in \mathfrak{g}$ 
 \end{enumerate}
If $\mathfrak{g}$ is a locally convex space and the Lie bracket continuous, $\mathfrak{g}$ is a \emph{locally convex Lie algebra}.\index{Lie algebra!locally convex}
A (continuous) linear map $h \colon \mathfrak{g} \rightarrow \mathfrak{h}$ between (locally convex) Lie algebras is a morphism of (locally convex) Lie algebras if $h(\LB[x,y])=\LB[h(x),h(y)], \forall x,y \in \mathfrak{g}$.
\end{defn}

\begin{rem}
 It will be essential for us that the Lie bracket and Lie algebra morphisms are continuous (see e.g.~the proof of \Cref{LieII:reg} for an example where continuity is needed). From now on we will mostly work with locally convex Lie algebras, hence we drop the adjective ``locally convex'' and write only Lie algebra. 
\end{rem}

\begin{ex}\label{ex:assocLie}
 If $A$ is a continuous inverse algebra (or more generally a locally convex algebra), then the commutator $\LB[x,y] \coloneq xy-yx$ turns $A$ into a Lie algebra. Hence the Lie bracket measures commutativity of the algebra product.
\end{ex}

\begin{ex}
 Every locally convex space $E$ is a Lie algebra, called an \emph{abelian Lie algebra}, with the \emph{trivial bracket} $\LB[x,y]\coloneq 0$.\index{Lie bracket!trivial}
\end{ex}

\begin{setup}[The Lie algebra of vector fields]\label{setup:LieVF}
Let $M$ be a manifold, and let 
$$\mathcal{V} (M) \coloneq \{X \in C^\infty (M,TM) \mid \pi_M \circ X = \id_M\}$$
be the locally convex space of all vector fields (cf.\, \Cref{chapter:LA:VF}). If $f \in C^\infty (M,E)$ is smooth with values in some locally convex space $E$ and $X \in \mathcal{V}(M)$, then we obtain a smooth function
$$X.f \coloneq df \circ X \colon M \rightarrow E \quad (\text{recall } df = \text{pr}_2 \circ Tf).$$
For $X,Y \in \mathcal{V}(M)$, there exists a unique vector field $\LB[X,Y] \in \mathcal{V}(M)$ determined by the property that on each $U \opn M$ we have 
\begin{equation}\label{eq:VFLieb}
 \LB[X,Y].f = X.(Y.f)-Y.(X.f) \quad \forall f \in C^\infty (U,E).
\end{equation}
 Thus $\mathcal{V}(M)$ becomes a Lie algebra (the local case is checked in \Cref{chapter:LA:VF} and we discuss the general case in Exercise \ref{Ex:LA} 3.). If $M$ is finite-dimensional, \Cref{cor:VFlcvx:LA} shows that $(\mathcal{V}(M), \LB )$ is a locally convex Lie algebra.\index{Lie bracket!of vector fields}
\end{setup}

\begin{setup}
 Let $G$ be a Lie group. A vector field $X \in \mathcal{V}(G)$ is called \emph{(left) invariant}\index{vector field!(left) invariant} if $X$ is $\lambda_g$-related to itself for all $g \in G$ (i.e.~$X\circ \lambda_g = T\lambda_g \circ X$, cf.~\Cref{chapter:LA:VF}). We write $\mathcal{V}^\ell (G)$ for the set of left-invariant vector fields. Note that relatedness of vector fields is inherited by the Lie bracket due to Exercise \ref{Ex:LA} 3., whence $\mathcal{V}^\ell (G)$ is a Lie subalgebra of $\mathcal{V}(G)$.\index{Lie algebra!of (left) invariant vector fields}
\end{setup}

\begin{prop}\label{prop:LA}
 Let $G$ be a Lie group, then the map 
 $$\Theta \colon T_{\one} G \rightarrow \mathcal{V}^\ell (G),\quad v \mapsto (g \mapsto T_{\one}\lambda_g (v))$$
 is an isomorphism of locally convex spaces with inverse $\Theta^{-1}(X)=X(\one)$. Thus $\Lf(G)\coloneq T_{\one} G$ can be endowed with the Lie bracket 
 $\LB[v,w] \coloneq \Theta^{-1}(\LB[\Theta(v),\Theta(w)]) =\LB[\Theta(v),\Theta(w)](\one)$ turning it into a Lie algebra. We call $(\Lf(G),\LB )$ the \emph{Lie algebra associated to} $G$.\index{Lie algebra!associated to a Lie group}
\end{prop}

\begin{proof}
 As $\Theta(v)(hg)= T_{\one}\lambda_{hg} (v) = T_{g}\lambda_{h}T_{\one}\lambda_{g} (v) = T_g \lambda_h \Theta (v)(g)$, the map $\Theta$ makes sense and its image consists of left-invariant vector fields (cf.\ Exercise \ref{Ex:LA} 4.). Linearity of $\Theta$ follows directly from the linearity of the tangent map. For $X \in \mathcal{V}^\ell(G)$ we have $X(g) = X\circ \lambda_g(\one) = T_{\one} \lambda_g X(\one) =\Theta (X(\one))(g)$, $\Theta$ is surjective. As the translations $\lambda_g$ are diffeomorphisms, it is clear that only $0 \in T_{\one}G$ gets mapped to the zero-vector field. We conclude that $\Theta$ is a vector space isomorphism (its inverse is obviously evaluation in $\one$). Note that $\mathcal{V}^\ell (G)$ carries the subspace topology induced by $\mathcal{V}(G)$ from \Cref{lcvx:space}. This immediately shows that $\Theta^{-1}$ is continuous as point evaluations are continuous in this topology. The continuity of $\Theta$ is left as Exercise \ref{Ex:LA} 5.
 That $\LB$ is a Lie bracket on $T_{\one} G$ follows directly by trivial computations since $\mathcal{V}^\ell(G)$ is a Lie algebra.
\end{proof}

If the Lie group $G$ is finite-dimensional, the above discussion shows that $(T_{\one} G,\LB)$ is a locally convex Lie algebra. Here only the continuity of the Lie bracket is unclear in the general case.
We shall now prove that the Lie bracket on $T_{\one} G$ is always continuous, hence the Lie algebra $\Lf(G)$ associated to a Lie group is always a locally convex Lie algebra. To this end, we need a local model of the multiplication

\begin{setup}\label{setup:locmult}
 Let $G$ be a Lie group. Since $T_{\one} G$ is isomorphic to the model space of $G$, we can pick a chart $\varphi \colon G \supseteq U_\varphi \rightarrow V_\varphi \opn T_{\one} G$ such that $\one \in U_\varphi$ and $\varphi (\one) = 0$. Moreover, we may assume that $T_{\one} \varphi = \id_{T_{\one} G}$. Due to the continuity of the multiplication of $G$ there is an open $\one$-neighborhood $W$ with $W \cdot W \subseteq U_\varphi$ (here $W\cdot W$ denotes the set of all products of two elements in $W$). Hence we can define a local multiplication 
 $$\ast \colon \varphi (W) \times \varphi (W) \rightarrow V_\varphi,\quad  (x,y) \mapsto  x \ast y \coloneq \varphi(\varphi^{-1} (x)\varphi^{-1}(y)).$$
 By construction the local multiplication is smooth and $\ast(0,x) = x = \ast (x,0)$. Hence the construction gives rise to a so called local Lie group (cf.\ \cite[Remark III.1.14]{NeeMon}). 
 As with the Lie group $G$ we can compute a Lie bracket using a local version of left-invariant vector field. To distinguish the local operations from the Lie group operations let us introduce a new symbol for left translation $\ell_x \colon \varphi(W) \rightarrow T_{\one} G,\ y \mapsto x \ast y, \ x \in \varphi (W).$
 For any $v \in T_{\one} G$ we can thus define a left-invariant vector field with respect to the local product
 $$\Lambda^v \colon \varphi(W) \rightarrow T_{\one} G,\quad x \mapsto d\ell_x (0;v) = \left.\frac{d}{dt}\right|_{t=0} x \ast tv.$$
 We will see in Exercise \ref{Ex:LA} 6.\, that a left-invariant vector field $X$ with $X(\one)=v$ is $\varphi$-related to $\Lambda^v$. 
 Together with the properties of the chart this yields the identity
 \begin{align*}
  \LB[v,w] &= \LB[\Theta(v),\Theta(w)](\one) = \LB[\Lambda^v,\Lambda^w](0) = d\Lambda^w (0;\Lambda^v(0))-d\Lambda^v(0;\Lambda^w(0)) \\ &= \left( \left.\frac{d^2}{d t ds}\right|_{t,s=0} sv \ast tw -  \left.\frac{d^2}{d t d s}\right|_{t,s=0} tw\ast sv\right)
 \end{align*}
 This formula shows immediately that the Lie bracket $\LB$ is continuous on $T_{\one} G$.
\end{setup}

\begin{cor}
 The Lie algebra $(\Lf(G),\LB)$ associated to a Lie group is a locally convex Lie algebra.
\end{cor}

\begin{rem}[Left-right confusion]\label{rem:LR}
 The reader may wonder now, why one uses left-invariant vector fields to compute the Lie algebra. Instead one could as well use \emph{right-invariant vector fields}, i.e.\ $X(g) =T\rho_g X(\one)$. This would also lead to a Lie algebra structure on $T_{\one}G$, however the induced Lie bracket would have the opposite sign (see Exercise \ref{Ex:LA} 7.). Indeed there is no reason to prefer left-invariant vector fields over right-invariant ones, the choice for left-invariant fields is historically motivated and customary.
\end{rem}

\begin{rem}
 There are several alternative ways to introduce the Lie bracket on $T_{\one} G$. For example, one can use the adjoint action of the Lie group (which is briefly discussed in \Cref{ex:adj:act} below). In addition the construction of the Lie bracket using the local multiplication in  \Cref{setup:locmult} can be interpreted (cf.\, \cite[Lemma III.1.6.]{NeeMon}) as a computation of the antisymmetric part of the second order Taylor polynomial of the local multiplication at $(0,0)$. Moreover, the Lie bracket measures commutativity of the group multiplication (cf.\ Exercise \ref{Ex:LA} 8.). 
\end{rem}

\begin{ex}[The Lie algebra of $\Diff(M)$]\label{ex:Liealgdiffeo}
 Let $M$ be a compact manifold. In \Cref{ex:Diffgp} we have seen that identifying $\Diff(M)$ as an open subset of $C^\infty (M,M)$, it becomes a Lie group under composition of maps. Further, \Cref{prop: can:locadd} shows that $T_{\id}\Diff(M) = T_{\id} C^\infty (M,M) \cong \mathcal{V}(M)$. We shall now show that the Lie algebra is given as $(\mathcal{V}(M), -\LB )$ with the \emph{negative} of the usual bracket of vector fields (cf.\ \Cref{setup:LieVF}).
 Extend $X \in \mathcal{V}(M) = T_{\id} M$ to the right-invariant vector field $R_X \in \mathcal{V}^\rho(\Diff (M)), R_X(\varphi) \coloneq X \circ \varphi$ and consider the product vector field $R_X \times \mathbf{0}_M \in \mathcal{V}(\Diff(M)\times M) = \mathcal{V}(\Diff(M)) \times \mathcal{V}(M)$. 
 We exploit now the canonical action $\alpha \colon \Diff (M) \times M \rightarrow M, (\varphi, m) \mapsto \varphi(m)$, \Cref{ex:candiffact} and note that as a restriction of $\ev$, a manifold version of \eqref{deriv:eval} (see Exercise \ref{ex:canmfd} 4.) yields the tangent map of $\alpha$: 
 \begin{align*}
  T\alpha (R_X \times \mathbf{0}_M)(\varphi,m) &= T_{(\varphi,m)}\alpha (X\circ \varphi ,\mathbf{0}_M(m)) =  X\circ \varphi (m) + T\varphi (0_m)= X \circ \varphi (m)\\ & = X (\alpha (\varphi,m))
 \end{align*}
Hence the product vector field $R_X \times \mathbf{0}_M$ is $\alpha$-related to $X$. Thus for $X,Y \in \mathcal{V}(M)$ the bracket $\LB[X,Y]$ is $\alpha$-related to $\LB[R_X \times \mathbf{0}_M ,R_Y \times \mathbf{0}_M] = \LB[R_X,R_Y] \times \mathbf{0}_M$, whence $\LB[R_X,R_Y]$ is the negative of the usual bracket, cf.\ \Cref{rem:LR}.

If $M$ admits a volume form $\mu$, we have seen that the volume preserving diffeomorphisms $\Diff_\mu (M)$ form a Lie subgroup of $\Diff (M)$. Due to \Cref{volumepreser} the Lie algebra $\Lf (\Diff_\mu (M))$ of this subgroup can be identified as the Lie subalgebra of divergence free vector fields $\mathcal{V}_\mu (M) = \{X \in \mathcal{V}(M)\mid \divr X = 0 (\Leftrightarrow \mathcal{L}_X \mu = 0)\}$.
\end{ex}

\begin{ex}
 Consider the locally convex space $E$ as the abelian Lie group $(E,+)$. An easy computation shows that $\Lf(E) = E$ with the trivial Lie bracket. Thus the Lie algebra of this abelian Lie group is an abelian Lie algebra (see Exercise \ref{Ex:LA} 8.)
\end{ex}

We can associate to every Lie group morphism a morphism of (locally convex) Lie algebras as the following lemma shows.
\begin{lem}\label{Liealghom}
 If $f \colon G \rightarrow H$ is a Lie group morphism (i.e.\, a smooth group homomorphism) then the map $\Lf (f) \coloneq T_{\one} f \colon \Lf(G) \rightarrow \Lf (H)$ is a morphism of Lie algebras, i.e.\ $\Lf(f)(\LB[v,w])=\LB[\Lf(f)(v),\Lf(f)(w)], \forall v,w \in \Lf(G)$.
\end{lem}

\begin{proof}
 Let $v \in T_{\one} G$ and $\tilde{v} \coloneq T_{\one} f (v) \in T_{\one} H$. Since $f$ is a group morphism we have 
 $$\Theta (\tilde{v})(f(g))=T_{\one}\lambda_{f(g)} T_{\one} f(v) = T_{g} f (T_{\one} \lambda_g (v)) =T_g f \Theta(v).$$
 Hence for every $v \in T_{\one} G$ the left-invariant vector field $\Theta (v)$ is $f$-related to $\Theta (\tilde{v})$. As $f$-relatedness is inherited by the Lie bracket (Exercise \ref{Ex:LA} 4.) we see that 
 $$T_{\one} f \LB[\Theta (v),\Theta(w)] = \LB[\Theta (\tilde{v}), \Theta(\tilde{w}] \circ f$$
 evaluating in $\one$, we obtain the claimed formula since $f(\one_G)=\one_H$.
\end{proof}

\begin{ex}\label{ex:adj:act}
 Let $G$ be a Lie group and $g \in G$. Then the conjugation with $g$ is the Lie group morphism $c_g \colon G \rightarrow G, h \mapsto ghg^{-1}$. Hence for every $g$ we obtain a Lie algebra morphism $\text{Ad}_g \coloneq \Lf (c_g) \colon \Lf (G) \rightarrow \Lf (G)$, called the \emph{adjoint map} of $g$. This gives rise to a smooth mapping 
 $$\text{Ad} \colon G \times \Lf (G) \rightarrow \Lf (G),\quad  (g,x) \mapsto \text{Ad}_g (x),$$
 called the \emph{adjoint action} of $G$.\index{Lie group action!adjoint action}
\end{ex}

The upshot of this short repetition is that every Lie group comes with an associated Lie algebra and every Lie group morphism gives rise to a Lie algebra morphism. In finite dimensions, the interplay between these objects leads to the classical Lie theorems.

\begin{tcolorbox}[colback=white,colframe=green!50!black,title=The Lie theorems for \textbf{finite-dimensional} Lie groups and Lie algebras]\index{Lie theorems}
\begin{itemize}
 \item[Lie 2] If $G,H$ are Lie groups and $G$ is simply connected, then for every Lie algebra homomorphism $f \colon \Lf(G) \rightarrow \Lf (H)$ there exists a Lie group morphism with $f = \Lf(\varphi)$.
 \item[Lie 3] For every Lie algebra $\mathfrak{g}$ there exists a connected Lie group $G$ with $\Lf (G) = \mathfrak{g}$.
\end{itemize}
\end{tcolorbox}
We omitted the first Lie theorem as it is a purely local statement which does not admit a global formulation on the Lie group. It is well known that the third Lie theorem fails in infinite dimensions. For example one can show that:
\begin{setup}[{A Lie algebra without an associated Lie group \cite{MR630634}}]
 Let $M$ be a connected \emph{non-compact} finite-dimensional manifold. Then there exists no Lie group $G$ such that $\Lf(G) = \mathcal{V}(M)$.
\end{setup}
However, under certain assumptions one can salvage at least the second Lie theorem. The main issue here is that in infinite-dimensional spaces, even simple differential equations might not have solutions (see \Cref{Diffeq:beyond} for explicit counter examples). Thus in our general setting a major task is to establish the existence of solutions to differential equations relevant to Lie theory. These \emph{equations of Lie type}\index{equation of Lie type} will be discussed next.

\begin{Exercise}\label{Ex:LA}   \vspace{-\baselineskip}
\Question Let $(\mathfrak{g},\LB )$ be a Lie algebra. Show that $\LB[x,x] =0$ for all $x \in \mathfrak{g}$ is equivalent to skew-symmetricity of the Lie bracket, i.e.\ $\LB[x,y]=-\LB[y,x], \forall x,y \in \mathfrak{g}$.  
\Question Show that the commutator bracket (\Cref{ex:assocLie}) is a Lie bracket.
 \Question With the help of the material in \Cref{chapter:LA:VF} show that the bracket of vector fields from \Cref{setup:LieVF} turns $\mathcal{V}(M)$ into a Lie algebra.
 \subQuestion Let $X,Y \in \mathcal{V}(M)$ and $\mathcal{A}$ be an atlas of $M$. The brackets of the local representatives yield $(\LB[X_\phi,Y_\phi])_{\phi \in \mathcal{A}}$. Show that this induces a vector field $\LB[X,Y]$ with the properties from \Cref{setup:LieVF}.
 \subQuestion Prove that $\LB[X,X] = 0$ for all $X \in \mathcal{V}(M)$ and the Jacobi identity holds.
 \subQuestion Show that the bracket is continuous with respect to the topology from \Cref{lcvx:space}.\\
 {\tiny \textbf{Hint:} All assertions can be localised in charts where \Cref{lem:locLie} holds.}
  \Question Prove a global version of \Cref{lem:locrelated}, i.e.\, if $(X_i,Y_i) \in \mathcal{V}(M) \times \mathcal{V}(N), i=1,2$ are pairs of $f$-related vector fields, then $\LB[X_1,X_2]$ and $\LB[Y_1,Y_2]$ are $f$-related. 
  \Question We check several details in the proof of \Cref{prop:LA}. Show that 
  \subQuestion $X_v \colon G \rightarrow TG, g\mapsto T_{\one}\lambda_g (v)$ is a smooth left-invariant vector field for $v\in T_{\one}G$.
  \subQuestion $\Theta \colon T_{\one} G \rightarrow \mathcal{V}^\ell(M) \subseteq \mathcal{V}(M), v \mapsto X_v$ is continuous.
  \\
  {\tiny \textbf{Hint:} Combine \Cref{lcvx:space} with \Cref{lem:fwedge_vector}.} 
   \Question Let $G$ be a Lie group and $\varphi$ be a chart with the properties from \Cref{setup:locmult} with respect to which we define a local multiplication $\ast$. Let $X \in \mathcal{V}^\ell (G)$ be left-invariant with $X(\one)=v$. Prove that $X_\varphi|_{\varphi(W)} = \Lambda^v$ and deduce that the (principal part of) left-invariant vector fields are thus locally related to the left-invariant vector fields with respect to the local multiplication. 
   \Question Let $v \in \Lf(G)$ and denote by $L_v$, $R_v$ the left-(/right-)invariant vector field constructed from $v$. Show that $L_v$ and $-R_v$ are $\iota$-related and deduce $\LB[L_v,L_w] = -\LB[R_v,R_w]$.\\
  {\tiny \textbf{Hint:} Use Exercise \ref{Ex:Lgpex} 3.\, to show that $T_{\one}\iota (v) = -v$.}
  \Question Let $G$ be an abelian Lie group. Show that the Lie bracket of $\Lf(G)$ is trivial, i.e.\, the Lie algebra is abelian.\\
 {\tiny \textbf{Hint:} Since $G$ is abelian all left-invariant vector fields are also right-invariant...}
 \Question Let $A$ be a CIA. Show that the Lie algebra of $A^\times$ is $A$ with Lie bracket $\LB[a,b] = ab-ba$.\\ {\tiny \textbf{Hint:} Note that since $A^\times \opn A$, the equation for left invariance is $X_v (g) = gv (=\beta(g,v))$ in $A$} 
 \Question Let $\mathfrak{g},\mathfrak{h}$ be two Lie algebras. Show that there is a canonical way to turn the product $\mathfrak{g} \times \mathfrak{h}$ into a Lie algebra. Explain how this was exploited in \Cref{ex:Liealgdiffeo}.
 \Question Let $G$ be a Lie group with associated Lie algebra $(\Lf (G). \LB )$. In this exercise we consider the adjoint action from \Cref{ex:adj:act}. Show that
 \subQuestion $\Ad \colon G \times \Lf (G) \rightarrow \Lf (G)$ is a Lie group action.
 \subQuestion for $\varphi \colon G \rightarrow H$ a Lie group morphism, $\Ad_{\varphi (g)}(\Lf (\varphi) (x)) = \Lf (\varphi)(\Ad_g (x))$. 
 \subQuestion for $x,y \in \Lf(G)$ one has $\text{ad}_x (y) \coloneq T_{\one_G,y}\Ad (x,0_y) = \LB[x,y]$.
 \end{Exercise}

\section{Regular Lie groups and the exponential map}\label{sect:regularLie}
In this section, we discuss differential equations needed for advanced tools in Lie theory. 

 \begin{defn}
 Let $G$ be a Lie group with Lie algebra $\Lf (G)$. We say $G$ is \emph{semiregular}\index{Lie group!semiregular} if for each smooth curve $\eta \in C^\infty ([0,1],\Lf(G))$ the initial value problem
 \begin{align}\label{eq:semiregular}
  \begin{cases}
   \dot{\gamma} (t) &= \gamma(t).\eta(t) = T_{\one} \lambda_{\gamma(t)} (\eta(t))\\
   \gamma(0) &= \one
  \end{cases}
 \end{align}
 has a (unique) solution $\Evol (\eta) \coloneq \gamma \colon [0,1] \rightarrow G$. We also say that \eqref{eq:semiregular} is a \emph{Lie type equation}. The group $G$ is \emph{regular (in the sense of Milnor)}\index{Lie group!regular (in the sense of Milnor)} if $G$ is semiregular and the following \emph{evolution map}\index{evolution map} is smooth
 $$\evol \colon C^\infty ([0,1], \Lf(G)) \rightarrow G,\quad \eta \mapsto \Evol(\eta)(1).$$
\end{defn}

\begin{setup}\label{setup:logderiv}
 For a smooth curve $c \colon [a,b] \rightarrow G$, we can define the \emph{left logarithmic derivative}\index{derivative!left logarithmic} $\delta^\ell (c) \colon [a,b] \rightarrow \Lf (G), t \mapsto T\lambda_{c(t)}(\dot{c}(t))$. Note that the logarithmic derivative inverts the evolution, i.e.\ $\delta^\ell (\Evol (\eta)) = \eta$.
 There are many identities relating $\delta^\ell$, $\Evol$ and their counterparts defined via right multiplication (see \cite[Section 38]{KM97} for an account, also cf.~Exercise \ref{Ex:MC} 3.). The left (right) logarithmic derivative is closely connected to the left (right) Maurer-Cartan form on $G$, see \Cref{ex:MCform}.
\end{setup}

\begin{rem} \label{rem:regularity}
\begin{enumerate}
\item Solutions to \eqref{eq:semiregular} are automatically unique by \cite[38.3 Lemma]{KM97}.
\item For regular Lie groups, Lie's second theorem holds. Since the proof can most conveniently be formulated within the framework of differential forms, we defer it to \Cref{App:MC}.
\item Instead of using the left multiplication in \eqref{eq:semiregular} one can as well use right multiplication to define regularity. Similar to the definition of the Lie algebra, using inversion of the group shows that the two notions of regularity are the same.
\end{enumerate}
\end{rem}

\begin{rem}
 Due to the usual solution theory of ordinary differential equations, every Banach-Lie group (i.e. Lie group modelled on a Banach space) and thus every finite-dimensional Lie group is regular, cf.\ \cite{Neeb06}. 
\end{rem}

\begin{ex}
 Consider the locally convex space $E$ as a Lie group $(E,+)$. Note that its Lie algebra is again $E$ with the zero-bracket (Exercise \ref{ex:regular} 3.). For a smooth curve $\eta \colon [0,1] \rightarrow \Lf(E)=E$ we interpret the Lie type equation in $TE=E\times E$ and obtain $(\gamma(t),\gamma'(t)) = (\gamma (t),\eta(t))$. Hence a solution $\gamma$ of \eqref{eq:semiregular} satisfies $\gamma' = \eta$. Therefore, if $(E,+)$ is regular, then $E$ is Mackey complete, \Cref{Mackeycomplete}. Conversely, if $E$ is Mackey complete, then $\evol_E (\eta) \coloneq \int_0^1 \eta(s) \mathrm{d}s$ defines a continuous linear (whence smooth) map $\evol_E \colon C^\infty ([0,1],E) \rightarrow E$. Whence $(E,+)$ is regular if and only if it is Mackey complete.
 Note that one can show that any Lie group which is regular is necessarily modelled on a Mackey complete space, \cite[Remark II.5.3 (b)]{Neeb06}.
\end{ex}

\begin{ex}\label{ex:CIAregular}
 Let $(A, \cdot)$ be a continuous inverse algebra (CIA) which is Mackey complete. If the topology of $A$ is generated by a family of seminorms which are submultiplicative, i.e.\, $q(xy) \leq q(x)q(y), \forall x,y \in A$, then $(A^\times,\cdot)$ is regular, \cite{MR2997582}. In this case the solutions to the evolution equation are given by the \emph{Volterra series}\index{Volterra series} 
 \begin{align}\label{Volterraseries}
  \gamma (t) = \one + \sum_{n\in \N} \int_0^t \int_{0}^{t_{n-1}} \cdots +\int_0^{t_2} \eta(t_1)\cdots \eta (t_n) \mathrm{d}t_1 \cdots \mathrm{d}t_n.
 \end{align}
 This has interesting applications in physics, control theory and rough path theory as for certain CIAs the above series modells signatures of irregular paths, cf.\, \Cref{sect:rough}.
\end{ex}

\begin{ex}[$\Diff (M)$ is regular]\label{Diff:reg}
 In \Cref{ex:Diffgp} we saw that for a compact manifold $M$ the group $\Diff (M)$ is a Lie group with Lie algebra $\mathcal{V} (M)$, \Cref{ex:Liealgdiffeo}. If $c \colon [0,1] \rightarrow \mathcal{V} (M) \subseteq C^\infty (M,TM)$ is smooth, we use the exponential law to interpret $X \coloneq c^\wedge$ as a smooth time dependent vector field on $M$. Following \Cref{rem:regularity} (c) we can solve Lie type equations with respect to the right multiplication to establish regularity. As right multiplication $\rho_\phi$ in $\Diff (M)$ is the pullback with $\phi$, we deduce from Exercise \ref{ex:canmfd} 1. $T_{\one}\rho_{\phi} (c(t)) = c(t) \circ \phi$. Hence the exponential law, allows us to rewrite the Lie type equation \eqref{eq:semiregular} with respect to right multiplication on $\Diff (M)$ as a differential equation for the vector field $X_t$ (subscript denoting time dependence) 
 \begin{align}\label{eq:floweqtdep}
  \gamma' (t) = X_t(\gamma(t)), \qquad \gamma(0) = \id_M.
 \end{align}
  In other words, $\gamma$ solves the Lie type equation \eqref{eq:semiregular} (wrt.\ right multiplication) if and only if it satisfies \eqref{eq:floweqtdep}, whence $\gamma$ is the flow $\mathrm{Fl}^X$of the time dependent vector field $X$ on $[0,1]\times M$ (see \Cref{VF:flow}). Since $[0,1] \times M$ is compact, the usual (finite dimensional!) theory of ordinary differential equations \cite[IV. \S 2]{Lang} shows that the flow of such a vector field always exists. We deduce that $\Diff (M)$ is semiregular and can consider the evolution operator $\evol \colon C^\infty ([0,1],\mathcal{V}(M)) \rightarrow \Diff (M)$. 
 
 To see that the evolution is smooth, recall that $\Diff(M) \opn C^\infty (M,M)$. Now we exploit that $C^\infty(M,M)$ is a canonical manifold of mappings to deduce that $\evol$ is smooth if 
 $$\evol^\wedge \colon C^\infty ([0,1], \mathcal{V}(M)) \times M \rightarrow M,\quad (X,m) \mapsto \mathrm{Fl}^{X}_1(m) $$ is smooth. 
 Here $\mathrm{Fl}_1^X$ denotes the time $1$-flow of the time dependent vector field $X$. We view \eqref{eq:floweqtdep} now as an ODE whose right hand side $F(t,X,m) = \ev (X,t)(m)=X(t,m)$ depends smoothly on the parameter $X$. The theory of parameter dependent ordinary differential equations (on finite-dimensional manifolds) shows that $\mathrm{Fl}^X$ depends smoothly on $X$, see e.g.\ \cite[Proposition 5.13]{MR3342623} which allows the (infinite-dimensional!) space $C^\infty([0,1], \mathcal{V}(M))$ as a parameter space. Hence $\evol^\wedge$ is smooth and $\Diff (M)$ is regular.  
\end{ex}

As a consequence of the regularity of $\Diff(M)$, we can apply Lie's second theorem, \Cref{LieII:reg} to Lie algebra morphisms into the Lie algebra of vector fields on a compact manifold.
Milnor \cite{Mil82} used these observations to prove a restricted version of the Lie-Palais theorem, cf.\ \cite{Pal57}. In its general form, the Lie-Palais theorem asserts that every finite-dimensional Lie algebra of vector fields $\mathfrak{g}$ of vector fields on a finite-dimensional smooth manifold $M$ which is generated by complete vector fields\footnote{A vector field $X \in \mathcal{V}(M)$ is complete if its integral curves $\dot{\varphi}_x (t) = X(\varphi_x(t)),\ \varphi_x(0)= x$ exist for all $t\in \R$ and $x \in M$} consists of complete vector fields and can be integrated to a global action of a Lie group $G$ on $M$. We discuss a proof for a severely limited version of this result as Exercise \ref{ex:regular} 7.

Not all Lie groups are regular as one can construct some pathological examples modelled on \textbf{incomplete} spaces. However, for almost all naturally occurring classes of Lie groups, regularity has been established. Moreover, the following conjecture is still open.

\begin{tcolorbox}[colback=white,colframe=green!50!black,title=Conjecture (Milnor '83)]
Every Lie group modelled on a Mackey complete space is regular.
\end{tcolorbox}

For regular Lie groups, one can solve the important class of Lie type differential equations. This allows us to discuss the Lie group exponential function (which can be viewed as an abstraction of the matrix exponential function to Lie groups).

\begin{lem}\label{lem:completeVF}
 Let $G$ be a regular Lie group and $X \in \mathcal{V}^\ell (G)$. Then there exists a unique curve $\gamma_X \colon \R \rightarrow G$ with $\dot{\gamma}_X(t)=X(\gamma_X(t)), \forall t$ and $\gamma_X(0)=\one$. This implies that every left-invariant vector field is complete.\index{vector field!(left) invariant} \index{vector field!complete}
\end{lem}

\begin{proof}
 Consider the smooth (constant) curve $\eta_X \colon \R \rightarrow T_{\one} G, t \mapsto X (\one)$ (in the following we will frequently consider the restriction of $\eta_X$ to $[0,1]$ without further notice). By regularity, we obtain a unique solution $\gamma \colon [0,1] \rightarrow G$ such that $\gamma(0)=\one$ and $\dot{\gamma}(t) = T\lambda_{\gamma(t)}(\eta_X(t)) = X(\gamma(t)), t \in [0,1]$. Thus this curve is the flow of the invariant vector field. Since the (constant) curve $\eta_X$ makes sense for all $t \in \R$, we can consider the equation \eqref{eq:semiregular} for every $t \in \R$. We now extend the flow $\gamma$ to all of $\R$.\\
 \textbf{Step 1:} \emph{Smooth extension to $[-1,1]$.}
 Using right-regularity, we can construct a curve $\gamma_R \colon [0,1] \rightarrow G$ which satisfies $\gamma(0)=\one$ and $\dot{\gamma}(t) = T\rho_{\gamma (t)}(X(1))$. Let us show that the formula 
 $$\gamma_X(t) \coloneq \begin{cases} 
 \gamma_R^{-1}(-t)& t \in [-1,0] \\
 \gamma(t)& t \in [0,1]                                                                                                                                                                                                                                                                                                                                                                                                                                              \end{cases}
$$
 yields a smooth extension of $\gamma$ whose derivative is $X(\gamma_X(t))$ at every $t$. We compute with the formula for the derivative of the inverse \eqref{Tiota}, the derivative of $\gamma_X$ on $[-1,0]$:
 \begin{align*}
  \frac{\mathrm{d}}{\mathrm{d}t} \gamma_R^{-1}(-t) &=-T\lambda_{\gamma_R^{-1} (-t)}T\rho_{\gamma_R^{-1} (-t)} \dot{\gamma}_R (-t)\\
  & = T\lambda_{\gamma_R^{-1} (-t)} T\rho_{\gamma_R^{-1} (-t)} T\rho_{\gamma_R(-t)} (v) = X(\gamma_R^{-1}(-t))
 \end{align*}
 Hence $\gamma_X$ is a smooth integral curve of the left-invariant field $X$.\\
 \textbf{Step 2:} \emph{$\gamma_X$ extends to all of $\R$} Pick $0 < t_0 < 1$ and define $\gamma_{t_0} \colon [-1+t_0,1+t_0] \rightarrow G, t \mapsto \gamma(t_0)\gamma_X(t-t_0)$. Note that $\gamma_{t_0}(t_0) = \gamma_X(t_0)$ and furthermore,  
 $$\dot{\gamma}_{t_0}(t_0+t) =T\lambda_{\gamma(t_0)} (\dot{\gamma}_X(t)) = T\lambda_{\gamma(t_0)} (X(\gamma_X(t)))= X(\gamma(t_0)\gamma_X (t))=X(\gamma_{t_0}(t+t_0)),$$
 where we have used the left invariance of $X$. Uniqueness of the solution implies now that on their common domain of definition $\gamma_X$ and $\gamma_{t_0}$ coincide. Thus we can extend $\gamma_X$ to $[-1,1+t_0]$. Repeating the argument, the domain of $\gamma_X$ is not bounded from above. Choosing $-1 < t_0 <0$ a similar argument shows that the domain can not be bounded from below. Thus $\gamma_X$ can be continued for all of $\R$. 
\end{proof}

\begin{defn}
 Let $G$ be a regular Lie group. Then we define the \emph{Lie group exponential}\index{Lie group!exponential}
 $$\exp_G \colon \Lf(G) \rightarrow G,\quad v \mapsto \gamma_v (1),$$
 where $\gamma_v$ is the unique integral curve starting from $\one$ of the vector field $L_v (g) \coloneq T\lambda_g (v)$.
\end{defn}

\begin{rem}
 Note that $C \colon \Lf(G) \rightarrow C^\infty ([0,1],\Lf(G)), v\mapsto (t\mapsto v)$ is continuous linear, whence smooth. Thus $\exp_G = \evol \circ C$ is smooth for every regular Lie group $G$.
\end{rem}

\begin{lem}\label{lem:deriv:evol}
 Let $G$ be a regular Lie group. Define for $\eta \in C^\infty ([0,1],\Lf(G))$ and  $s \in [0,1]$ the curve $\eta_s (t) \coloneq \eta(st)$. Then  
 \begin{align*}
  \Evol  (s\eta_s)(t) = \Evol (\eta) (st)\ \forall t \in [0,1]\quad \text{ and } \quad \Evol (\eta) (s) = \evol (s\eta_s).    
 \end{align*}
 In particular, this implies $T_0 \exp = \id_{\Lf(G)}$.
\end{lem}

\begin{proof}
 The curve $t \mapsto \Evol (\eta) (st)$ takes $0$ to the identity in $G$ and its derivative is $\frac{\mathrm{d}}{\mathrm{d}t} \Evol (\eta)(st) = s \Evol (\eta)(st).\eta(st)$. Thus it solves the Lie type equation  for the curve $s\eta_s$. This proves the first identity, while we obtain the second for $t=1$. 
 Let now $\eta (t) \coloneq v$ be constant for $v \in \Lf (G)$. Then $\eta_s (t)=\eta(t)$ and we see that $\exp(sv)=\evol (s\eta_s) = \Evol (\eta)(s)$.
 Derivating at $s=0$ yields 
 $$T_0 \exp (v) = \left.\frac{\mathrm{d}}{\mathrm{d}s}\right|_{s=0} \exp(sv) = \left.\frac{\mathrm{d}}{\mathrm{d}s}\right|_{s=0} \Evol (\eta)(s) = \underbrace{\Evol(\eta)(0)}_{=\one_G}.\eta(0)=v.$$
\end{proof}

Unfortunately, the observation that $T_0 \exp = \id_{\Lf(G)}$ is not as useful as in the Banach setting, where the inverse function theorem implies that the Lie group exponential is a local diffeomorphism onto a neighborhood of the unit of the group (then the Lie group the exponential yields a canonical chart called ``exponential coordinates'').\index{Lie group!exponential coordinates}
In general the Lie group exponential need \textbf{neither} be locally injective \textbf{nor} locally surjective. Most prominently, this happens for the diffeomorphism group $\Diff (M)$ from \Cref{ex:Diffgp}. We discuss the special case for $M=\SSS^1$ below in  \Cref{ex:badexponential}. Lie groups for which the exponential function is well behaved thus deserve a special name:

\begin{defn}
 A regular Lie group $G$ is called \emph{locally exponential}\index{Lie group!locally exponential} if the exponential function $\exp \colon \Lf(G) \rightarrow G$ restricts to a local diffeomorphism between a neighborhood of $0 \in \Lf(G)$ and $\one_G \in G$. 
\end{defn}
Unfortunately, diffeomorphism groups are not locally exponential as the following classical example shows: 

\begin{ex}\label{ex:badexponential}
 Consider the unit circle $\SSS^1$. We will show that the image of the Lie group exponential of the diffeomorphism group $\Diff (\SSS^1)$ contains no identity neighborhood. 
 Recall from \Cref{Diff:reg} that the Lie group exponential is the map
 $$\exp \colon \Lf (\Diff (\SSS^1)) = \mathcal{V} (\SSS^1) \rightarrow \Diff(\SSS^1),\quad V \mapsto \text{Fl}^V_1$$
 assigning to a (time independent) vector field its time $1$-flow. 
 
 Recall that $\theta \colon \R \rightarrow \R/2\pi \cong \SSS^1, \theta \mapsto e^{i\theta}$ is a submersion. Composing the submersion with a vector field of the circle, we identify vector fields of the circle with $2\pi$-periodic maps $\R \rightarrow \R$. For a constant vector field $X^c(\theta ) \equiv c$ a quick computation shows that its flow $\exp(tX^c(\theta)) =e^{i(\theta +tc)}$ is for fixed $t$ a rotation of the circle. 
 
 \textbf{Step 1:} \emph{A diffeomorphism $\eta \in \Diff (\SSS^1)$ without fixed points is an exponential of $V \in \mathcal{V}(\SSS^1)$ if and only if $\eta$ is conjugate to a rotation}. 
 If $V \in \mathcal{V}(\SSS^1)$ has a zero, its exponential has a fixed point. Thus if $\eta = \exp(X)$, we must have $X (\theta) \neq 0$. We will now construct $\varphi \in \Diff (\SSS^1)$ and a constant vector field $X^c$ such that $\varphi \circ \eta = \exp(X^c) \circ \varphi$. Assume for a moment that for all $t \in [0,1]$ we have the identity $\varphi \circ \exp (tX) =\exp (tX^c) \circ \varphi$. Differentiating at $t=0$ yields with \Cref{lem:deriv:evol} the identity 
 $T_{\theta}\varphi (X(\theta)) = X^c(\varphi (\theta))$ for all $\theta \in \SSS^1$. For ease of computation identify now $\varphi$ and $X$ with periodic mappings $\R \rightarrow \R$. Then the equation reads $\varphi'(\theta)X(\theta) = c$ for all $\theta \in \R$. Integrating we obtain with $\varphi (0)=0$ and the fact that $X$ vanishes nowhere that $\varphi (\theta) = \int_0^\theta c / X(s)\mathrm{d}s$. We will now choose $c$ such that $\varphi (\theta +2\pi) - \varphi(\theta) = \int_t^{t+2\pi} c/X(s)\mathrm{d}s = 2\pi$ (because then $\varphi$ descends to a diffeomorphism of $\SSS^1$). Since $X$ is $2\pi$-periodic, we can simply choose $c = 2\pi \int_0^{2\pi} X(s)\mathrm{d}s$ 
 Since the flows of the $\varphi$-related vector fields $X$ and $T\varphi^{-1}\circ X \circ \varphi$ are conjugate by $\varphi$ we obtain $\varphi \circ\eta = \varphi \circ \exp(X) = \exp (X^c) \circ \varphi = R_c \circ \varphi$. Thus if a fixed point free diffeomorphism of $\SSS^1$ is the exponential of a vector field, it is conjugate to a rotation.
 
 \textbf{Step 2:} \emph{Diffeomorphisms near the identity which are not exponentials.} 
 We claim that there are diffeomorphisms $\varphi$, arbitrarily near the identity $\id_{\SSS^1}$ such that $\varphi$ 
 \begin{enumerate}
  \item has no fixed points
  \item there exists $\theta_0 \in \SSS^1$ and $n \in \N$, $n >1$ such that $\varphi^n (\theta_0) = \theta_0$, but $\varphi^n \neq \id_{\SSS^1}$. 
 \end{enumerate}
 If this were true, then we note that if $\varphi$ is an exponential, (a) and Step 1 imply that it must be conjugate to rotation. However, this is impossible by (b) as a rotation which has a point of period $n$ must be periodic with period $n$ itself. Thus the image of the exponential does not contain $\varphi$ and $\Diff(\SSS^1)$ can not be locally exponential.
 
 To construct a diffeomorphism with properties (a) and (b) consider for $n \in \N$ large enough and $0 <\varepsilon < 1/n$ the maps
 $$f_{n,\varepsilon} (\theta) = \theta + \frac{\pi}{n} +\varepsilon \sin^2(n\theta),$$
 descend to diffeomorphisms of the circle which satisfy (a)-(b) and can be made arbitrarily close to the identity (in the compact open $C^\infty$-topology). 
 We leave the details as Exercise \ref{ex:regular} 6.
 
 More generally one can prove that the diffeomorphism group of a compact manifold is not locally exponential.
\end{ex}

\begin{ex}
 \begin{enumerate}
  \item As a consequence of \Cref{lem:deriv:evol} and the inverse function theorem, every Banach-Lie group and thus in particular every finite-dimensional Lie group is locally exponential.
  \item The unit group of a Mackey complete CIA $A$ (cf.\, \Cref{ex:CIAregular}) is locally exponential, \cite{MR2997582}. 
  \item If $G$ is a locally exponential Lie group, we shall see in Exercise \ref{ex:currentgroup} 6. that the current groups $C^\infty (K,G)$ (see \Cref{sect:currentgp}) are locally exponential. Due to \cite{MR2353707} this generalises even to the group of gauge transformations of a principal bundle with locally exponential gauge group.
 \end{enumerate}
\end{ex}

\begin{rem}
 While in infinite dimensions not every closed subgroup of a (locally exponential) Lie group is again a Lie group (see \Cref{closedsubgrp}), there are conditions which ensure that a closed subgroup of a locally exponential Lie group is a Lie subgroup. We refer to \cite[IV.]{Neeb06} for more information.   
\end{rem}

\begin{Exercise} \label{ex:regular}\vspace{-\baselineskip}
 \Question A smooth map $\alpha \colon \R \rightarrow G$ to a Lie group is called \emph{$1$-parameter subgroup} of (a regular Lie group) $G$ if it is a group homomorphism, i.e.\ $\alpha(s+t)=\alpha (s)\alpha(t) \forall s,t \in \R$.
 \subQuestion Show that a $1$-parameter subgroup $\alpha$ is an integral curve of the left-invariant vector field $L_v$ and the right-invariant field $R_v$, generated by $v\coloneq \left.\frac{\mathrm{d}}{\mathrm{d}t}\right|_{t=0}\alpha(t)$, i.e.\ $\dot{\alpha}(t) = L_v (\alpha(t))$ and $\dot{\alpha}(t) = R_v (\alpha(t))$.
 \subQuestion Show that if $\alpha, \beta \colon \R \rightarrow G$ are smooth, $\alpha(0)=\one=\beta(0)$ with $\dot{\alpha}(t)=L_v(\alpha(t))$ and $\dot{\beta}(t) = R_v (\beta(t))$ (the associated right-invariant vector field) then $\alpha = \beta$ and these curves are a $1$-parameter subgroup.\\
 {\tiny \textbf{Hint:} To show equality differentiate $\alpha(t)\beta(-t)$, for the second statement derivate $\beta(t-s)\beta(s)$.}
 \Question Consider the multiplicative group $(\R^\times, \cdot)$ as a Lie group. Show that $\delta^\ell (f)(t) = f'(t)/f(t)$ and explain the name ``logarithmic derivative''.
 \Question Let $\gamma_i \colon [0,1] \rightarrow G, i=1,2$ be smooth curves. Show that $\delta^\ell (\gamma_1) = \delta^\ell (\gamma_2)$ if and only if $\gamma_1 (t) = g\gamma_2 (t), \forall t \in [0,1]$ and some fixed $g \in G$.
 \Question Consider a locally convex space $(E,+)$ as a Lie group. Show that the Lie bracket on $\Lf(E) = E$ vanishes and the Lie group exponential is $\exp_E = \id_E$. 
 \Question Let $\alpha \colon G \rightarrow H$ be a Lie group morphism between regular Lie groups. Show that 
 \subQuestion for $\gamma \in C^1 ([0,1],G)$ we have $\delta^\ell (\alpha \circ \gamma) = \Lf(\alpha) (\delta^\ell (\gamma))$.
 \subQuestion for $\eta \in C^\infty ([0,1],\Lf(G))$ one has $\Evol (\Lf(\alpha)\circ \eta) = \alpha \circ \Evol (\eta)$ (a similar formula holds for $\evol$). Further, prove the naturality of the Lie group exponential, i.e.\ 
 \begin{align}\label{naturalexp}
  \exp_H \circ \Lf (\alpha) = \alpha \circ \exp_G.
 \end{align}

 \Question Consider for $n \in \N$ and $0< \varepsilon < 1/n$ the maps $f_{n, \varepsilon} \colon \R \rightarrow \R, \theta \mapsto \theta + \pi/n + \varepsilon \sin^2 (n\theta).$
 \subQuestion Show that $f_{n,\varepsilon}$ descends via the submersion $\R \rightarrow \R/2\pi\mathbb{Z}$ to a diffeomorphism $\varphi_{n,\varepsilon}$ of $\SSS^1$. Moreover, prove that there are $\varphi_{n,\varepsilon}$ arbitrarily near the identity (where $\Diff (\SSS^1) \opn C^\infty (\SSS^1,\SSS^2)$ carries the compact open $C^\infty$-topology).\\
 {\footnotesize \textbf{Hint:} Control $f_{n,\varepsilon}$ in the compact open $C^\infty$-topology on $\R$.}
 \subQuestion Show that $\varphi_{n,\varepsilon}$ does not possess a fixed point, but since $f_{n,\varepsilon}^{2n}(0)=0$ (modulo $2\pi$) it has a periodic point of period $2n$.
 \subQuestion Show that the $2n$-periodic orbit $\varphi_{n,\varepsilon}^k (0), k=1,\ldots, 2n-1$ is unique, i.e.\ if $\theta$ is not contained in the orbit, then $\varphi_{n,\varepsilon}^{2n} (\theta) \neq \theta$. We deduce that $\varphi_{n,\varepsilon}^{2n} \neq \id_{\SSS^1}$. 
 \Question Let $\mathfrak{g}$ be a finite-dimensional Lie algebra. By Lie's third theorem, \cite[Theorem 9.4.11]{HaN12} there exists a connected, simply connected Lie group $G$ such that $\Lf(G)=\mathfrak{g}$. Assume that $M$ is a compact manifold such that $\phi \colon \mathfrak{g} \rightarrow \mathcal{V}(M)$ is a Lie algebra morphism (with respect to the negative of the usual bracket of vector fields). Prove that $\phi$ induces a smooth action of  $G$ on $M$ (this is a very restricted version of the Lie-Palais theorem \cite{Pal57}).  
\end{Exercise}

\begin{Answer}[number={\ref{ex:regular} 7.}] 
 \emph{(Mini Lie-Palais) Every Lie algebra morphism $\phi\colon \mathfrak{g} \rightarrow \mathcal{V}(M)$ from a finite-dimensional Lie algebra $\mathfrak{g}$ to the Lie algebra of vector fields of a compact manifold $M$ (with the negative of the usual bracket) gives rise to a Lie group action $G \times M \rightarrow M$ (with $\Lf (G)=\mathfrak{g}$). 
}\\[1.3em]

Note first that since $\mathfrak{g}$ is finite-dimensional and $\phi$ is linear, $\phi$ is automatically continuous, whence a morphism of locally convex Lie algebras. By Lie's third theorem, \cite[Theorem 9.4.11]{HaN12} there exists a connected, simply connected Lie group $G$ such that $\Lf(G)=\mathfrak{g}$. Now $G$ is finite-dimensional and thus regular and $\Diff (M)$ is regular by \Cref{Diff:reg}. Hence we can apply Lie's second theorem for regular Lie groups, \Cref{LieII:reg}, to integrate $\phi$ to a Lie group morphism $\Phi \colon G \rightarrow \Diff (M)$, i.e.\ $\Lf (\Phi) = \phi$. Exploiting that $\Diff (M) \opn C^\infty (M,M)$ and $C^\infty (M,M)$ is a canonical manifold, \Cref{la:cano:mfdmap}, the adjoint map $\Phi^\wedge \colon G \times M \rightarrow M$ is a Lie group action (naturally induced by $\phi$.
\end{Answer}

\section{The current groups}\label{sect:currentgp}
In this section we construct a class of infinite-dimensional Lie groups which occur naturally in theoretical physics: loop groups and current groups. As a first step in the construction we prove a useful local characterisation of Lie groups

\begin{tcolorbox}[colback=white,colframe=blue!75!black,title=Multiplicative notation for sets]
In the statement of \Cref{loc:Liegrp} below we use multiplicative notation for sets. Recall that $A\cdot B$ means the set of all elements which can be written as a product of elements in $A$ and $B$ (in that order).
\end{tcolorbox}

\begin{prop}[{Bourbaki, \cite[Ch.\, III, \S 1, No.\, 9 Proposition 18]{MR1728312}}] \label{loc:Liegrp} \index{Lie group!Bourbaki construction principle}
 Let $G$ be a group and $U,V \subseteq G$ such that $\one \in V = V^{-1} \coloneq \{g \in G \mid g^{-1}\in V\}$ and $V\cdot V \subseteq U$. Assume that $U$ is equipped with a (smooth) manifold structure such that $V \opn U$ and the mappings $\iota|_V^V \colon V \rightarrow V$ and $m_G|_{V\times V} \colon V \times V \rightarrow U$ are smooth. Then the following holds:
 \begin{enumerate}
  \item There is a unique manifold structure on the subgroup $$G_0 \coloneq \langle V\rangle \coloneq \{v_1\cdots v_k \mid v_i\in V, k\in \N\} $$ such that $G_0$ becomes a Lie group, $V\opn G_0$ and $G_0$ and $U$ induce the same manifold structure on $V$.
  \item Assume that for each $g$ in a generating set of $G$ there is $\one \in W_g \opn U$ such that $gW_gg^{-1} \subseteq U$ and $c_g \colon W_g \rightarrow U, h\mapsto ghg^{-1}$ is smooth. Then there is a unique manifold structure on $G$ turning $G$ into a Lie group such that $V\opn G$ and both $G$ and $U$ induce the same manifold structure on $V$.
 \end{enumerate}
\end{prop}

\begin{proof}
 \textbf{Step 1: Shifted open sets} If $A \opn V$ and $v_0 \in V$ with $v_0 A \subseteq V$, then $v_0A$ is open in $V$ as the preimage of $A$ under the smooth map $\delta_{v_0} \colon V \rightarrow U, g \mapsto v_0^{-1}g$.\\
 \textbf{Step 2: Shifted charts} 
 Pick $W \opn V$ such that $W=W^{-1}$, $W^3=W\cdot W \cdot W \subseteq U$ and there exists a manifold chart $(\varphi,W)$ of $U$. For $g \in G$ we consider the map $\varphi_g \colon gW \rightarrow \varphi(W), h \mapsto \varphi(g^{-1}h)$. The idea is now to construct an atlas from these shifted charts which is compatible with the manifold structure on $V$.\\
 \textbf{Step 3: Manifold induced by shifted charts}
 Observe that for $g_1W\cap g_2 W \neq \emptyset$ we have $g_2^{-1}g_1 \in W^2$ (this entails $g_1^{-1}g_2 \in W^2$). Then Step 1 yield $W \cap g^{-1}_2g_1W \opn W$, whence $D_{g_1,g_2}\coloneq \varphi_{g_1} (g_1W\cap g_2W)= \varphi (W\cap g^{-1}_2g_1W)$ is open in the model space. For $d \in D_{g_1,g_2}$ we have now 
 $$\varphi_{g_2}\circ \varphi_{g_1}^{-1} (d) = \varphi (g_2^{-1}g_1\varphi^{-1}(d)),$$
 which is a smooth mapping by choice of $W$, whence the change of charts are smooth. 
 Endow $G$ with the final topology with respect to the atlas $\mathcal{S} \coloneq (\varphi_g)_{g \in G}$ (this topology is Hausdorff as $W$ is Hausdorff). We conclude that $\mathcal{S}$ is a manifold atlas turning $G$ into a smooth manifold.
 Moreover, for each $g_0 \in G$ we have $\varphi_{gg_0} \circ \lambda_{g_0}|_{gW} = \varphi_g$, whence the left translation with elements in $G$ is smooth for this manifold structure.\\
 \textbf{Step 4: The manifold structures coincide on $V$.} By our assumptions on multiplication and inversion, we can find for every  $v_0 \in V$ a set $\one \in A\opn W$ such that $v_0 A \subseteq V$. Step 3 shows that $v_0 A$ is open in $G$ and $A \rightarrow v_0A, h \mapsto v_0h$ is a diffeomorphism with respect to the structure induced by $\mathcal{S}$, in particular, $V\opn G$. Now Step 1 implies that this holds for the structure induced by $U$, whence both coincide on $V$.
 
 We are now in a position to prove both claims of the proposition:
 \begin{enumerate}
  \item By definition we have $G_0 = \langle V \rangle = \bigcup_{k\in \N} V^k$. We can write $V^2 = \bigcup_{g \in V} gV \opn G$ and thus inductively, $V^k \opn G$ and $G_0 \opn G$. Denote by $(G,\mathcal{S})$ the manifold induced by the atlas $\mathcal{S}$. Then Step 4 shows that $\delta \colon V \times V \rightarrow G_0 \opn (G,\mathcal{S}), (g,h) \mapsto gh^{-1}$ is smooth and this is equivalent to the smoothness of multiplication and inversion on the set $V$. For $g_0,h_0, g,h\in G$ we let $c_{h_0}(a)\coloneq h_0ah_0^{-1}$ and obtain the identity
  \begin{align}
   \delta(g,h) &= (h_0g_0^{-1})^{-1} h_0(g_0^{-1}g)(h_0^{-1}h)^{-1}h_0^{-1} \notag \\ 
   &= \lambda_{g_0^{-1}h_0} \circ c_{h_0} \circ \delta (\lambda_{g_0^{-1}} (g) ,\lambda_{h_0^{-1}}(h)) \label{Lieloc}
  \end{align}
  We proceed now by induction on $k$ and show first that $\delta$ is smooth on $V^k\times V$. For $k=1$ we know that $\delta$ is smooth. However, note that also $c_{h_0} \colon V \rightarrow G_0$ is smooth as 
  \begin{align}\label{conj:ext}
   c_{h_0} (g) =\lambda_{h_0} \circ \delta (g,h_0)
  \end{align}
 and $\delta$ is smooth on $V\times V$. Moreover, if $g \in V^k$ we can pick $g_0 \in V^{k-1}$ such that for all $x$ in a $g$-neighborhood $g_0^{-1}x \in V$. Following again \eqref{conj:ext}, we see that 
 $$c_{h_0}(x)= \lambda_{h_0g_0} \delta (g_0^{-1}x,h_0) \text{ is smooth for all } h_0 \in H \text{ and $x$ near $g$.}$$
 We conclude that $c_{h_0} \colon G_0 \rightarrow G_0$ is smooth for each $h_0 \in V$. \\
 Thus for the induction step from $k$ to $k+1$ we see now that $\delta \colon V^{k+1}\times V \rightarrow G_0$ is given in a neighborhood of $(g,h)$ by the composition of smooth maps \eqref{Lieloc} for $h_0\coloneq h$ and $g_0 \in V$ with $g_0^{-1}g \in V^{k}$. It follows that $\delta \colon G_0 \times V \rightarrow G_0$ is smooth. Now a trivial induction together with \eqref{Lieloc} shows that $\delta \colon G_0 \times G_0 \rightarrow G_0$ is smooth. We conclude that $G_0$ is a Lie group.   
  \item Prove inductively that $\delta$ is smooth of an element $(g,h)$ where $g$ is arbitrary and $h = s_1 \cdots s_n$ for $s_i, i=1,\ldots ,n$ in the generating set $\mathcal{G}$. Starting with $n=1$ we choose $h_0 \coloneq s_1^{-1}$ and $g_0 \coloneq g^{-1}$. By (a) the map $\delta$ is smooth on $G_0$, whence there is an open $(g,h)$-neighborhood which gets mapped by $\delta \circ (\lambda_{g_0}, \lambda_{h_0})$ into $W_g$. Since $h_0 = s_1 \in \mathcal{G}$ by the assumption, \eqref{Lieloc} implies that $\delta$ is smooth in a neighborhood of $(g,h)$. Now by \eqref{conj:ext} conjugation with $h_0 \in \mathcal{G}$ is smooth in a $g$-neighborhood, whence on all of $G$. If $\delta$ is smooth on an open neighborhood of $G \times \bigcup_{1\leq k \leq n}\mathcal{G}^k$ and $h=s_1\cdots s_{n+1}$. Arguing as above with $h_0 = s_1^{-1}$ shows that $\delta$ is smooth in a neighborhood of $G \times \bigcup_{1\leq k \leq n+1}\mathcal{G}^k$ for all $n \in \N$. As smoothness of $\delta$ is equivalent to the smoothness of the group operations, we deduce that $G$ is a Lie group.
 \end{enumerate}
 The uniqueness of the manifold structure of the Lie groups $G_0$ and $G$ in (a) and (b) follows from the fact that any other manifold structure with these properties induces the same manifold structure on the open $\one$-neighborhood $V$. 
\end{proof}

We will now apply the Bourbaki result to construct the Lie group structure for the current group $C^\infty (K,G)$. The crucial insight here is that for a Lie group, an atlas of charts can be constructed by left translating a chart at the identity. We will see in the construction of the current group, that this allows us to construct a canonical manifold of mappings by left translating the pushforward of a suitable chart at the identity. 

\begin{tcolorbox}[colback=white,colframe=blue!75!black,title=General assumption]
For the rest of this section we let $K$ be a compact manifold, $G$ a (possibly infinite-dimensional) Lie group with Lie algebra $\Lf(G)$.
Then we choose and fix $ \one \in U_1 \opn G$ together with a chart $\varphi \colon U_1 \rightarrow U \subseteq \Lf(G)$ such that $\varphi(\one)=0$ and $T_{\one} \varphi = \id_{\Lf(G)}$. Further we pick $\one \in V_1 \opn U_1$ such that $V_1\cdot V_1 \subseteq U_1, V_1^{-1} = V_1$ and set $V \coloneq \varphi(V_1)$. Mapping spaces will as always be endowed with the compact open $C^\infty$-topology,
\end{tcolorbox}

Let us now exploit again the local formulation of the Lie group structure of $G$ in a chart around the identity (as in the identification of the Lie algebra, \Cref{setup:locmult}). In order to clearly distinguish the local group operations from the (globally defined) group operations and the pushforwards appearing on the function spaces we introduce new symbols for the local operations.

\begin{setup}\label{setup:locop}
 The mappings 
 $$\mu \colon V \times V \rightarrow U, \mu(x,y)\coloneq \varphi(\varphi^{-1}(x)\varphi^{-1}(y)), \quad \iota \colon V \rightarrow V, \iota(x) \coloneq \varphi(\varphi^{-1}(x)^{-1})$$
 are smooth. Moreover, $C^\infty (K,U_1)$ is an open subset of $C^\infty (K,G)$ and we equip it with the smooth manifold structure making the bijection $C^\infty (M,U_1) \rightarrow C^\infty (K,U), \gamma \mapsto \varphi \circ \gamma$ a diffeomorphism of smooth manifolds. Note that due to Exercise \ref{ex:canmfd} 2., the manifold $C^\infty (K,U)$ is canonical.
 Then the pushforwards $\mu_* \colon C^\infty (K,V\times V) = C^\infty (K,V) \times C^\infty(K,V) \rightarrow C^\infty (K,U)$ and $\iota_* \colon C^\infty(K,V)\rightarrow C^\infty(K,V)$ are smooth by \Cref{la-reu}.
\end{setup}

\begin{thm}\label{thm:currentgroup}
 The pointwise operations turn $C^\infty (K,G)$ into a Lie group, a \emph{current group}.\index{current group} Its Lie algebra is $C^\infty (K,\Lf(G))$ with the pointwise Lie bracket (a \emph{current algebra}).\index{Lie algebra!current algebra}
\end{thm}

\begin{proof}
 We check that the assumptions of \Cref{loc:Liegrp} are satisfied: The identity element in $C^{\infty}(K,G)$ is the constant map $\one_{C^\infty (K,G)}(k) = \one_G$. Thus \Cref{setup:locop}, shows that  $C^\infty (K,U) \opn C^\infty (K,G)$ is an open identity neighborhood such that multiplication and inversion are smooth on the smaller neighborhood $C^\infty (K,V)$. 
 Let now $\gamma \in C^\infty (K,G)$, by compactness of $\gamma(K) \subseteq G$ there are $\one \in W_1 \opn V_1$ and $\gamma(K) \subseteq P \opn G$ such that $pW_1p^{-1} \subseteq U_1$ for all $p\in P$. Set $W \coloneq \varphi(W_1)$ and note that since $C^\infty (K,W) \opn C^\infty (K,V)$ so is $C^\infty (K,W_1) \opn C^\infty (K,V_1)$. Since $c\colon P \times W_1 \rightarrow U_1, h(p,w)\coloneq pwp^{-1}$ is smooth, so is 
 $$h_\gamma = \varphi \circ h \circ (\gamma \times \varphi^{-1}) \colon K \times W \rightarrow U, h_\gamma(x,y)=\varphi(\gamma(x)\varphi^{-1}(y)\gamma(x)^{-1})$$
Since the manifolds $C^\infty (K,W)$ and $C^\infty (K,U)$ are canonical (\Cref{setup:locop}), \Cref{fstar-gen} yields a smooth map $(h_\gamma)_\star \colon C^\infty (K,W) \rightarrow C^\infty (K,U), \eta \mapsto h_\gamma \circ (\id_W\times \eta).$
Now conjugation by $\gamma$ coincides on $C^\infty (K,W)$ with $(h_\gamma)_\star$ and is thus smooth on $C^\infty (K,W) \opn C^\infty (K,G)$. Now \Cref{loc:Liegrp} provides a unique smooth Lie group structure for $C^\infty (K,G)$.

To identify the Lie algebra, note that $\varphi_* \colon C^\infty (K,U_1) \rightarrow C^\infty (K,U), \gamma\mapsto \varphi \circ \gamma$ is a chart around the identity of the Lie group $C^\infty (K,G)$. Exploiting that $T_{\one} \varphi = \id_{\Lf(G)}$ we use \Cref{thetamp} to identify $T_{\one_{C^\infty (K,G)}} \varphi_* = (T_{\one} \varphi)_* = (\id_{\Lf(G)})_*$ whence $\Lf(C^\infty (K,G) \cong C^\infty (K,\Lf(G))$. Now the point evaluations $\ev_x \colon C^\infty (K,G) \rightarrow G, \gamma \mapsto \gamma(x)$ are Lie group morphisms (Exercise \ref{ex:currentgroup}), whence $\Lf (\ev_x) = \ev_x \colon C^\infty (K,\Lf(G))$ is a Lie algebra morphism by \Cref{Liealghom}. This implies $\LB[\gamma,\eta](x)=\LB[\gamma(x),\eta(x)]$ and thus the bracket is given by the pointwise bracket. 
\end{proof}

\begin{rem}
 For non-compact manifolds $M$, the group $C^\infty (M,G)$ can in general not be made a Lie group. To see this consider $\N$ as a $0$-dimensional manifold. Then $C^\infty (\N,\SSS^1) \cong (\SSS^1)^\N \coloneq \prod_{n \in \N} \SSS^1$ is a compact topological group. However, since it is not locally contractible (cf.~Exercise \ref{ex:currentgroup} 5.) it can not be a manifold and thus it can not be a Lie group. 
\end{rem}

We shall now prove that regularity of the target Lie group $G$ is inherited by the current group $C^{\infty } (K,G)$. The idea is simple: The Lie type equation on the current group is just a parametric version of the Lie type equation on the target Lie group. In other words, if $\eta \colon [0,1] \rightarrow C^\infty (K,\Lf(G))$ is a smooth curve, then we obtain for each $k\in K$ a Lie type equation 
$$\frac{\partial}{\partial t} \gamma^\wedge (t,k) = (\dot{\gamma} (t))(k) = (\gamma(t).\eta(t)) (k) = \gamma^\wedge (t,k).\eta^\wedge (t,k).$$
Due to the regularity of the target Lie group $G$, we can solve the Lie type equation (uniquely) for each fixed $k$. Then the solutions to these equations glue back together to a solution on the current group (due to a suitable exponential law).

\begin{prop}\label{prop:reglift}
 If $G$ is a regular Lie group, then $C^\infty (K,G)$ is a regular Lie group.
\end{prop}

\begin{proof}
In view of the preliminary considerations, we see that the pointwise solutions to the Lie type equation must be the solution to the equation on the current group. It remains to apply the exponential law\footnote{Since $[0,1]$ has boundary, we can not use the exponential law \Cref{thm:explaw}. Instead we have to appeal to the stronger version \cite[Theorem A]{MR3342623}. Note that the result is the same and in particular the (partial) derivative of the adjoint map corresponds to the derivative.} to show that these solutions depend smoothly on the initial data.
 We obtain a canonical isomorphism $\Theta$:
 $$C^\infty ([0,1],C^\infty (K,\Lf(G)) \cong C^\infty ([0,1]\times K , \Lf(G)) \cong C^\infty (K,C^\infty ([0,1],\Lf(G)).$$
 The map $h \coloneq (\evol_G)_* \circ \Theta \colon C^\infty ([0,1],C^\infty(K,\Lf(G)) \rightarrow C^\infty (K,G)$ is smooth by \Cref{la-reu} as $\evol_G \colon C^\infty ([0,1],\Lf(G)) \rightarrow G$ is smooth. Let $\ev_x \colon C^\infty (K,G) \rightarrow G$ and $e_x \colon C^\infty (K,\Lf(G)) \rightarrow \Lf(G)$ be the point evaluations in $x$. We note that $\Lf (\ev_x)=e_x$. Hence Exercise \ref{ex:regular} 5.\, yields $\ev_x \circ h = \evol_G \circ (e_x)_* \circ \Theta$ for all $x \in K$. Since the evaluations separate the points on $C^\infty (K,G)$, this implies $h = \evol_{C^\infty (K,G)}$.
\end{proof}

Similar to regularity being hereditary, current groups inherit the property of being locally exponential from their target Lie group. We leave this as Exercise \ref{ex:currentgroup} 6. In the following sections we discuss two special cases of current groups and their subgroups: loop groups and groups of gauge transformations (associated to a principal bundle). 

\subsection*{Loop groups} \addcontentsline{toc}{subsection}{Loop groups}
If $K = \SSS^1$ the current group $LG \coloneq C^\infty (\SSS^1,G)$ is better known as a \emph{loop group}\index{loop group} (cf.\ \cite{MR900587}).  Much is known about loop groups and their representation theory. We mention that they are in particular connected to the representation theory of Kac-Moody Lie algebras and to quantum field theory, cf.~\cite{Schm10}. 
\begin{rem}
  In the literature also the group $C(\SSS^1,G)$ of continuous loops is often called the loop group $LG$ of $G$. For us the loop group will nevertheless consist always of smooth loops.
\end{rem}

 We shall discuss two canonical subgroups of $LG$ for finite-dimensional $G$:
\begin{setup}
  The canonical group morphism $I \colon G \rightarrow C^\infty (\SSS^1,G), g \mapsto (k\mapsto g)$ identifies $G$ with the subgroup of constant loops.
 As $C^\infty (\SSS^1,G)$ is a canonical manifold by Exercise \ref{ex:currentgroup} 3., smoothness of $I$ is equivalent to smoothness of $I^\wedge \colon G\times \SSS^1\rightarrow G, I^\wedge (g,k) = g$ (which is obvious). Further, a local argument in charts around $(x,k) \in G\times \SSS^1$ shows that the partial derivative of $I^\wedge$ identifies the derivative of $I$. Thus by \Cref{setup:curves:tan} we obtain
 $$T I (v_g) = (k \mapsto v_g \in T_g G) \in C^\infty (\SSS^1,TG) \cong TC^\infty (\SSS^1,G), g\in G, v_g \in T_gG$$ 
 and this map is clearly injective for every $g$. Thus $I$ is infinitesimally injective and since $G$ is finite-dimensional, it is an immersion. Moreover, $I$ is a Lie group morphism with smooth inverse: Due to Exercise \ref{ex:currentgroup} 2 c) the evaluation maps $\ev_{k} \colon LG \rightarrow G, k \in \SSS^1$ are Lie group morphisms. Choosing for example $h \coloneq \left.\ev_{\one_{\SSS^1}}\right|_{I(G)}$, we obviously have $h\circ I = \id_G$, whence $I$ is a topological embedding onto its image. Hence $I$ is a smooth embedding and if we identify $G$ with $I(G)$, we can think of $G$ as a  Lie subgroup of $LG$ (see \cite[Lemma 1.13]{Glofun}).
 
 Define now the group of all loops starting at $\one_G$ as $\Omega G \coloneq \{f \in LG \mid f(\one_{\SSS^1}) = \one_G\}$.
 Since $\ev_{\one_{\SSS^1}} \colon C^\infty (\SSS^1,G) \rightarrow G$ is a Lie group morphism and a submersion with $\Omega G = \ev_{\one_{\SSS^1}}^{-1}(\one_G)$, we see that $\Omega G$ is a split submanifold of $LG$, whence a Lie subgroup. Summarising,
 $$\begin{tikzcd}
\one \arrow[rr] & &  \Omega G \arrow[rr, hookrightarrow] & & LG \arrow[rr, yshift = 3pt, twoheadrightarrow, "\ev_{\one_{\SSS^1}}"] && G \ar[rr] \ar[ll, yshift=-3pt, "I"]&& \one
\end{tikzcd}$$
is a split sequence of Lie groups, whence $LG \cong \Omega G \rtimes G$. We finally note that this allows one to define the \emph{fundamental homogeneous space}   
 $$LG/G =\{[h] \mid f \in [h] \text{ if } f = g\cdot h \text{ for some } g \in G\} ( \cong \Omega G),$$
 which plays an important r\^{o}le in the theory of loop groups (see \cite{MR900587}).
\end{setup}

\subsection*{Groups of gauge transformations}\addcontentsline{toc}{subsection}{Groups of gauge transformations}

We briefly discuss an important class of infinite-dimensional Lie groups also arising in physics. In extremely broad strokes the setup is as follows: Fix a manifold $M$. In physics this would correspond to the space-time. Physical fields in the space-time $M$ are described as sections of certain bundles $P \rightarrow M$ called principal bundles (see \Cref{defn:principal_bdl} below). A premier example is here the electromagnetic field. In this formulation the famous Maxwell equations, cf.\ \Cref{ex:Maxwell}, can be interpreted as differential equations for a certain connection on the principal bundle. We will not describe connections on principal bundles here, but point out that the geometric data of the connection corresponds to the potential of the field. However, the symmetry group of the space of connections on the bundle is the so called group of gauge transformations. We shall describe it as an infinite-dimensional Lie group and identify it as a subgroup of the current groups studied earlier. Let us begin by recalling the notion of a principal bundle (much more information on principal bundles can be found in the usual finite-dimensional literature \cite{Baum14,Hus93}).

\begin{defn}\label{defn:principal_bdl}
 Let $G$ be a Lie group with a smooth right action $\rho \colon E \times G \rightarrow E$. Assume that the quotient $p \colon E \rightarrow M \coloneq E/G$ is a smooth locally trivial fibre bundle with typical fibre $F$, i.e.\ $p$ is a smooth map such that there is an open cover $(U_i)_{i \in I}$ of $M$ and $\rho$-equivariant diffeomorphisms (\emph{bundle trivialisations}) $\kappa_i \colon p^{-1}(U_i) \rightarrow U \times F$ conjugating $p$ to the projection. The quadruple $(E,p,M,F)$ is a \emph{principal $G$-bundle}\index{principal bundle} if the action $\rho$ is simply transitive, i.e.\ for each $f_0 \in F \cong p^{-1}(x)$ we have $G \rightarrow F, g \mapsto f_0g$ is a diffeomorphism.
\end{defn}

The group $G$ in the definition of the principal bundle is also called the \emph{structure group}\index{principal bundle!structure group} of the principal bundle. In the physics literature it is customary to call the structure group the \emph{gauge group}\index{principal bundle!gauge group} of the principal bundle. We will follow the physics terminology and warn the reader that in the literature, the term ``gauge group'' is also used for the group of gauge transformations we are about to define.

\begin{rem}
 Note that for a principal $G$-bundle the fibre $F$ is diffeomorphic to the gauge group $G$, but since $F$ lacks a preferred choice of unit element, there is no preferred group structure on the fibre $F$ (one also says that $F$ is a \emph{$G$-torsor}\footnote{See \url{https://math.ucr.edu/home/baez/torsors.html} for more examples and explanations.}).
\end{rem}

\begin{ex}[The trivial example]\label{ex:Maxwell}
 If $M$ is a smooth manifold and $G$ a Lie group, then the trivial bundle $\text{pr}_M \colon M \times G \rightarrow M$ is a principal bundle. 
 
 While on first sight nothing interesting happens here, trivial principal bundles appear in interesting physical applications. As an example, we would like to mention Maxwell's equations from electromagnetics. Fix a contractible $U \opn \R^3$. The electrical field $E$ and the magnetic field $H$ and the current $J$ in $U$ can be described by time dependent vector fields on $U$. Denote by $\rho$ the electrical charge and by $c$ the speed of light. Then Maxwell's equations can be expressed in differential operator notation as
 $$\text{rot}E = -\frac{1}{c}\frac{\partial H}{\partial t} , \quad \text{div} H = 0 , \quad \text{rot} H = \frac{1}{c} \frac{\partial E}{\partial t} + \frac{4\pi}{c} J, \quad \text{div} E = 4\pi \rho.$$
 Now one can show that the Maxwell equations can be interpreted as differential equations for the curvature of a $\SSS^1$-connection on the trivial $\SSS^1$-bundle $(U \times \R) \times \SSS^1 \rightarrow U \times \R$. We refer to \cite[Section 7.1]{Baum14} for the derivation and more information.  
\end{ex}

Interpreting the Maxwell equations in the language of principal bundles turns out to be a fruitful idea. Replacing the trivial bundle by a bundle with a (non-abelian) structure group one arrives at gauge theories such as Yang-Mills theory or Chern-Simons theory.

\begin{ex}\label{ex:homogeneous}
  Let $K$ be a compact Lie group and $G$ a closed Lie subgroup of $K$, then one can show that the quotient $M \coloneq K/G$ is again a manifold such that the canonical mapping $q \colon K \rightarrow K/G=M$ becomes a submersion (this is a so called \emph{homogeneous space}).\index{homogeneous space} Then $(K,q,M,G)$ is a principal $G$-bundle, \cite[Chapter I Theorem 4.3]{MR1410059}.  
  
  As a more concrete example, consider the compact Lie group $\text{SO}_3 (\R)$ (see \Cref{ex:findimLIE} 2.). Recall from \cite[1.2.A]{DaK00} that $\text{SO}_3 (\R)$ can be identified with the $3$-dimensional unit sphere $\SSS^3$ in $\R^4$. Moreover, the spherical group $\SSS^1$ embeds into $\text{SO}_3 (\R)$ as the rotation subgroup $\left\{\begin{bmatrix} \cos (\alpha) & -\sin (\alpha) & 0 \\ \sin (\alpha) & \cos (\alpha) & 0 \\ 0 & 0 &1\end{bmatrix}, \alpha \in \R \right\}$. This identification allows us to interpret multiplication with $\SSS^1$-elements as rotations on $\SSS^3$ and it is not hard to see that the quotient, $\SSS^3 / \SSS^1 = \text{SO}_3 (\R) /\SSS^1$ is diffeomorphic to the $2$-dimensional sphere $\SSS^2$. The resulting homogeneous space 
  $$\begin{tikzcd}
\SSS^1 \arrow[rr, hookrightarrow] & & \SSS^3 \arrow[rr, twoheadrightarrow, "q"] && \SSS^2
\end{tikzcd}$$
is known as the Hopf fibration in differential topology, cf.\ \cite[Example 3.14]{sharpe97}. Note that the Hopf fibration is not isomorphic (as an $\SSS^1$-principal bundle) to the trivial $\SSS^1$-bundle over $\SSS^2$, \cite[p.56]{DaK00} .
\end{ex}

\begin{defn}\label{defn:gaugegroup}
 Let $(E,p,M,F)$ be a principal $G$-bundle. Then define the \emph{group of gauge transformations}\index{group of gauge transformations} of the bundle as follows
 $$\text{Gau}(E) \coloneq \{\varphi \in \Diff (E) \mid p\circ \varphi = p,\ \forall g \in G,\ \varphi \circ \rho(\cdot,g) = \rho(\cdot,g) \circ \varphi \}.$$
\end{defn}

\begin{setup}\label{gaugeidentification}
 If $\varphi$ is a gauge transformation we can identify $\varphi(e)=\rho(e,f(e))$ for some smooth function $f \colon E\rightarrow G$ and the relation $\varphi \circ\rho(\cdot,g) = \rho(\cdot,g) \circ \varphi$ then yield 
 \begin{align}\label{gauge}
  f(\rho(e,g)) = g^{-1}f(e)g,\quad \forall e \in E, g \in G
 \end{align}
 Conversely, every smooth function $f \colon E \rightarrow G$ satisfying \eqref{gauge} defines a gauge transformation via $\varphi_f(e)\coloneq \rho(e,f(e))$. We will show in Exercise \ref{ex:currentgroup} that for
 \begin{align} 
 C^\infty (E,G)^G\coloneq \{f \in C^\infty (E,G) \mid \forall e\in E, g\in G, f(\rho(e,g))=g^{-1}f(e)g\} \notag \\
  \text{the map } C^\infty (E,G)^G \rightarrow \text{Gau}(E),\quad f \mapsto \varphi_f, \ \varphi_f(e)=\rho(e,f(e)) \label{groupmorph} 
  \end{align}
 is a group isomorphism. Hence the group of gauge transformations can naturally be identified as a subgroup of the current group.
\end{setup}

\begin{ex}
 If $\text{pr}_M \colon M \times G \rightarrow M$ is a trivial principal $G$-bundle, it is easy to see that $\text{Gau}(M\times G) \cong C^\infty (M,G)$. So if $M$ is compact, the group of gauge transformations inherits a Lie group structure from the current group in this case.
 
 For the trivial $\SSS^1$-principal bundle connected to Maxwell's equations, \Cref{ex:Maxwell}, the group of gauge transformations (aka the current group in this case) acts on $\SSS^1$-connections on the trivial bundle (connections in the same orbit are called gauge equivalent). As already mentioned, the Maxwell equations can be interpreted as a differential equation for the curvature of a $\SSS^1$-connection. See \cite[Section 7.1]{Baum14}.  
\end{ex}

\begin{rem}\label{rem:gaugegroup}
 If $G$ is a locally exponential Lie group and $M$ a compact manifold, then the group of gauge transformations $\text{Gau}(E)$ of a smooth principal $G$-bundle $(E,p,M,F)$ carries a natural Lie group structure turning it into a locally exponential Lie group, \cite{wockel}. Indeed one can show that the group of gauge transformations then becomes a Lie subgroup of a finite product of suitable current groups \cite{MR2353707}. 
 
 Alternatively, the Lie group structure of the group of gauge transformations can be derived by identifying the group of gauge transformations with the vertical bisections of a gauge groupoid, see \Cref{rem:vbis} and \Cref{ex:gaugegpd}.
\end{rem}

\begin{Exercise}\label{ex:currentgroup}  \vspace{-\baselineskip}
 \Question We establish some auxiliary results for the proof of \Cref{loc:Liegrp}:
  \subQuestion Let $G$ be a group which is also a manifold. Show that $G$ is a Lie group if and only if the mapping $\delta \colon G \times G \rightarrow G, (g,h) \mapsto gh^{-1}$ is smooth.
  \subQuestion Let $G$ be a Lie group (actually topological group would be enough for the following) and $V \opn G$ with $\one \in V$. Show that there is $\one \in W \opn V$ with $W^{-1}=W$ and $W\cdot W \subseteq V$ (can you generalise this to $W^n \subseteq V, n\in \N$?).
  \subQuestion Let $G,H$ be Lie groups and $f \colon G \rightarrow H$ be a group homomorphism. Prove that $f$ is a Lie group morphism if and only if there exists an open $\one_G$-neighborhood $U$ such that $f|_U$ is smooth.
  \subQuestion Let $G$ be a group with two manifold structures $\mathcal{S}$ and $\mathcal{T}$ turning $G$ into a Lie group. Show that $(G,\mathcal{S})$ and $(G,\mathcal{T})$ are diffeomorphic (as manifolds and then also as Lie groups) if and only if there is an open $\one$-neighborhood on which the two manifold structures coincide. 
 \Question Verify several details from the proof of \Cref{thm:currentgroup}:
 \subQuestion Let $K \subseteq G$ be a compact subset of a Lie group $G$ and $\one \in U \opn G$. Show that there is an open $\one$-neighborhood $W$ and an open $K$-neighborhood $P$ such that $pWp^{-1} \subseteq U$ for all $p\in P$.
 \subQuestion Verify that the final topology with respect to the charts $(\varphi_g)_{g \in G}$ from Step 3 of the proof is Hausdorff.
 \subQuestion Show that the point evaluations  $\ev_x \colon C^\infty (K,G) \rightarrow G, \gamma \mapsto \gamma(x)$ are Lie group morphisms.\\
 {\tiny \textbf{Hint:} For a group homomorphisms it suffices to check smoothness in an identity neighborhood.}
  \Question Prove that the current group $C^\infty (K,G)$ is a canonical manifold.\\ 
 {\tiny \textbf{Hint:} For a Lie group, it suffices to construct an open identity neighborhood whose manifold structure is canonical.}
 \Question By \Cref{setup:locadd:Lie} every Lie group $G$ admits a local addition. Show that the manifold structure on $C^\infty (K,G)$ constructed in \Cref{setup:mfdstruct} coincides with the one from \Cref{thm:currentgroup}.
 \Question Show that $C^\infty (\N , \SSS^1) = \prod_{n \in \N} \SSS^1$ is not locally contractible, i.e.\, show that if $N$ is a neighborhood of a point $x \in \prod_{n \in \N} \SSS^1$ then there is no continuous mapping $h \colon [0,1] \times N \rightarrow N$ such that $h(0,\cdot)=\id_N$ and $h(0,y)=x, \forall y \in N$. 
 \Question Let $G$ be a locally exponential Lie group. Show that then also $C^\infty (K,G)$ is locally exponential.
 \\
 {\tiny \textbf{Hint:} Study \Cref{prop:reglift} and note in addition that $\exp_G$ induces a chart for $G$. }
 \Question Show that the map $C^\infty (E,G)^G \rightarrow \text{Gau}(E), \quad f \mapsto \varphi_f$, \eqref{groupmorph} from \Cref{gaugeidentification} is a group isomorphism.
 \end{Exercise}
  

\chapter{Lifting geometry to mapping spaces II: (weak) Riemannian metrics}\label{RiemGeo}\copyrightnotice
In this section we will discuss Riemannian metrics on infinite-dimensional spaces. Particular emphasis will be placed on the new challenges which arise on infinite-dimensional spaces.
\begin{tcolorbox}[colback=white,colframe=blue!75!black,title=General assumption for this chapter]
To ensure the existence of integrals we shall always assume that the manifolds in this chapter are modelled on (Mackey) complete locally convex spaces.
\end{tcolorbox}

\section{Weak and strong Riemannian metrics}

Riemannian metrics comes in several flavours on infinite-dimensional spaces which are not present in the finite-dimensional setting. 
The strongest flavour (as we shall see) is the notion of a strong Riemannian metric which is treated in classical monographs such as \cite{Lang,MR1330918}.

\begin{defn}
 Let $M$ be a manifold modelled on a locally convex space $E$. A \emph{weak Riemannian metric}\index{Riemannian metric!weak} $g$ is a smooth map 
 $$g \colon TM \oplus TM \rightarrow \R, (v_x,w_x) \mapsto g_x (v_x,w_x)$$
 (where $TM\oplus TM$ is the Whitney sum, cf.\ \Cref{WSum}) satisfying
 \begin{enumerate}
  \item $g_x \coloneq g|_{T_xM\times T_xM}$ is symmetric and bilinear for all $x \in M$,
  \item $g_x(v,v)\geq 0$ for all $v \in T_x M$ with $g_x(v,v)=0$ if and only if $v=0$.
 \end{enumerate}
 If a weak Riemannian metric satisfies in addition 
 \begin{enumerate}
  \item[(c)] The topology of the inner product space $(T_xM , g_x)$ coincides with the topology of $T_xM$ as a subspace of $TM$. 
 \end{enumerate}
We say that $g$ is a \emph{strong Riemannian metric}.\index{Riemannian metric!strong} A manifold with a weak (/strong) Riemannian metric will be called a \emph{weak (/strong) Riemannian manifold}.
\end{defn}

\begin{ex}\label{ex:Hilbert}
 Every Hilbert space $(H,\langle \cdot,\cdot\rangle)$ becomes a strong Riemannian manifold using the identifications $TH=H\times H$, $TH\oplus TH = H^3$ and setting $g_c (v,w) \coloneq \langle v,w\rangle$.
\end{ex}

\begin{ex}\label{ex:L2Hilbert}
 Let $(H,\langle \cdot , \cdot\rangle)$ be a Hilbert space. The locally convex space $C^\infty ([0,1],H)$ with the compact-open $C^\infty$-topology is a \Frechet space (but not a Banach- or Hilbert space!).
  Consider the $L^2$-inner product on this space as 
 $$\langle f,g\rangle_{L^2} \coloneq \int_0^1 \langle f(t),g(t)\rangle \mathrm{d} t.$$
 This is a bilinear map on $C^\infty ([0,1],H)$. By construction of the compact open $C^\infty$-topology, the inclusion $C^\infty ([0,1],H) \rightarrow C([0,1],H)_{\text{c.o}}$ is continuous linear and it is not hard to prove that the mapping $\beta \colon C([0,1],H)^2 \rightarrow H, (f,g) \mapsto \int_0^1 \langle f(t),g(t)\rangle \mathrm{d}t$ is continuous bilinear. In conclusion, the $L^2$-inner product is a continuous bilinear form on $C^\infty ([0,1],H)$.
 Interpreting the locally convex space as a manifold we have $$TC^\infty ([0,1],H) = C^\infty ([0,1],H) \times C^\infty ([0,1],H) \cong C^\infty ([0,1],H\times H) =C^\infty ([0,1],TH).$$
 We obtain an isomorphism $TC^\infty ([0,1],H)\oplus TC^\infty ([0,1],H) \cong C^\infty ([0,1], H^3)$ (cf.\, also \Cref{Whitneyiso}) which transforms $f_c,g_c \in T_c C^\infty ([0,1],H)$ into a triple $(c,f,g)$ of $H$-valued functions. Thus
 $$g_{L^2} \colon TC^\infty ([0,1],H)\oplus TC^\infty ([0,1],H) \cong C^\infty ([0,1],H^3) \rightarrow \R ,\quad (c,f,g) \mapsto \langle f,g\rangle_{L^2}$$ is a weak Riemannian metric, called the \emph{$L^2$-metric}.\index{Riemannian metric!$L^2$-metric} 
 Note that the $L^2$-metric is not a strong Riemannian metric as the topology of the inner product space is the one induced by the inclusion $C^\infty ([0,1],H) \subseteq L^2([0,1],H)$ where the space on the right-hand side is the space of all square integrable functions with the $L^2$-topology.
\end{ex}

\begin{setup}\label{musicaliso}
 If $(M,g)$ is a weak Riemannian manifold, there is a injective linear map 
 $$\flat \colon TM \rightarrow T^*M = \bigcup_{x \in M} L(T_xM, \R),\quad  T_xM \ni v \mapsto g_x (v,\cdot),$$
 where $T^*M$ is the dual tangent bundle, cf.\ \Cref{rem:dualTM}.\index{cotangent bundle}\index{vector bundle!dual bundle} If $M$ is a finite-dimensional manifold, then it follows by counting dimensions that $\flat$ is (fibre-wise) an isomorphism and thus $\flat$ is an isomorphism of vector bundles, cf.\ \Cref{rem:dualTM}. Similarly, if $g$ is a strong Riemannian metric, $\flat$ is an isomorphism of vector bundles, see \Cref{prop:strongmetric} below. 
 The inverse of $\flat$ in these situations is often denoted by $\sharp$ and the pair of isomorphisms is known as the musical isomorphisms of a Riemannian manifold. For infinite-dimensional manifolds, the map $\flat$ will in general not be surjective. 
\end{setup}

The next result yields a useful characterisation of strong Riemannian metrics (which is the definition of a strong Riemannian metric in the classical texts \cite{Lang,MR1330918}). We just mention that for the proof some tools from functional analysis are required (which we cite but do not provide a detailed review of).

\begin{prop}\label{prop:strongmetric}
 Let $(M,g)$ be a weak Riemannian manifold. The following conditions are equivalent.
 \begin{enumerate}
  \item $g$ is a strong Riemannian metric,
  \item $M$ is a Hilbert manifold and $\flat \colon TM \rightarrow T^\ast M$ is surjective,
  \item $M$ is a Hilbert manifold and $\flat \colon TM \rightarrow T^\ast M$ is a vector bundle isomorphism.
 \end{enumerate}
 In particular, on every finite-dimensional manifold $M$, a weak Riemannian metric is automatically strong. 
\end{prop}

\begin{proof} \textbf{Step 1:} \emph{strong Riemannian metric implies Hilbert and surjective $\flat$.}
 If $g$ is a strong Riemannian metric, every tangent space $T_x M$ is a Hilbert space (since it is Mackey-complete cf.\ \Cref{Mackey-remark}). As the tangent spaces are isomorphic to the modelling spaces of $M$, we see that $M$ is a Hilbert manifold. Now the surjectivity of $\flat$ is a consequence of the Riesz representation theorem \cite[11.9]{MaV97}.\\
 \textbf{Step 2:} \emph{$\flat$ surjective on a Hilbert manifold implies that $\flat$ is a bundle isomorphism}. 
 Note first that since $M$ is a Hilbert manifold, there is a canonical smooth structure on the dual bundle, cf.\ \Cref{rem:dualTM}, whence it makes sense to consider $\flat$ as a smooth map between $TM$ and $T^\ast M$. By construction we know that $\flat$ induces continuous vector space isomorphisms $T_xM \rightarrow (T_xM)^\ast$ for every tangent space. Hence the open mapping theorem \cite[I. 2.11]{Rudin} shows that $\flat$ is fibre-wise an isomorphism, which together with \cite[III. Proposition 1.3]{Lang} shows that $\flat$ is a vector bundle isomorphism.\\
 \textbf{Step 3:} \emph{Hilbert manifold and $\flat$ being a bundle isomorphism imply that $g$ is strong.}
 Consider $x \in M$ and note that as $g$ is continuous on $TM \oplus TM$, $g_x$ is continuous with respect to the Hilbert space topology of the fibre. Let us denote the Hilbert space norm on $T_xM$ by $\lVert \cdot \rVert)$). Then the ball $B_1^{g_x} (0)  \coloneq \{y \in T_xM \mid g_x(y,y)<1\}$ is by construction $B_1^{g_x} (0)$ a disc $0$-neighborhood (see \Cref{defn:convx}) and the norm $\lVert \cdot \rVert_{g_x}$ induced by $g_x$ (aka. the Minkowski functional\index{Minkowski functional} of the ball, \Cref{Mink:disc0nbhd}) is continuous. As the unit ball in $(T_xM,\lVert \cdot \rVert)$ is bounded this shows already that there is a constant $K> 0$ with $\lVert\cdot\rVert \leq K \lVert \cdot \rVert_{g_x}$. To see that the norms are equivalent, we have to prove that $B_{g_x}^1(0)$ is bounded in the Hilbert space topology. 
 As $\flat$ is surjective, every bounded linear functional on $T_xM$ is of the form $\flat (v)$ for some $v\in T_xM$. Applying the Cauchy-Schwarz inequality, we derive
 $$\sup_{y \in B_1^{g_x}(0)} |\flat(v)(y)| =\sup_{y \in B_1^{g_x}(0)}|g_x(v,y)|\leq \lVert v\rVert_{g_x}.$$
 We conclude that every bounded linear functional is bounded on $B_{1}^{g_x}(0)$. This property is called weakly bounded and it is known that on a Hilbert space, every weakly bounded set in a locally convex space is bounded, \cite[Theorem 3.13]{Rudin}. 
 Summing up, $g_x$ induces the Hilbert space topology of $T_xM$ and since $x$ was arbitrary, this shows that $g$ is a strong Riemannian metric.
 \end{proof}

\begin{ex}\label{ex:invL2}
 The space of immersions $\Imm(\SSS^1,\R^d)$ is an open subset of $C^\infty (\SSS^1,\R^d)$ by \Cref{lem:immsub_opn}. We endow it with a weak Riemannian metric (an invariant version of the $L^2$-metric, \Cref{sect:L2metric})\index{Riemannian metric!invariant $L^2$-metric}
 $$g_c (h,k) \coloneq \int_{\SSS^1} \langle h(\theta),k(\theta)\rangle \lVert \dot{c}\rVert \mathrm{d}\theta .$$
 Recall that since $\Imm (\SSS^1,\R^d) \opn C^\infty (\SSS^1,\R^d)$ we have an identification $T_c \Imm(\SSS^1,\R^d) = C^\infty (\SSS^1,\R^d)$ (cf.\ \Cref{setup:curves:tan}). Working in the identification we immediately see that in the identification the musical morphism reduces to the map: 
 $$\flat_c (h) = \lVert \dot{c} \rVert \cdot h, h \in C^\infty (\SSS^1,\R^d),$$
 where the dot signifies pointwise multiplication. Hence the image of $T_c \Imm (\SSS^1,\R^d)$ under $\flat_c$ can be identified as $C^\infty (\SSS^1,\R^d)$. However, the (topological) dual space of $C^\infty (\SSS^1,\R^d)$ is the larger space $\mathcal{D}'(\SSS^1)^d$ of $\R^d$-valued distributions (see \cite[Section 3]{MR2744150}).\index{dual space} 
\end{ex}

\begin{ex}\label{Ex:HSphereRiem} \index{Hilbert sphere}
 Let $(H,\langle \cdot ,\cdot \rangle)$ be a Hilbert space with $S_H = \{x \in H \mid \lVert x\rVert =1\}$ (\Cref{ex: HSphere}). The Hilbert sphere is a strong Riemannian manifold with the induced metric $g_x (v,w) \coloneq \langle v,w\rangle$ (where $T_x S_H = \{v \in H\mid \langle v,x\rangle =0\}$).
\end{ex}

 The distinction between strong and weak Riemannian metrics has far reaching consequences (see for example the \Cref{sect:L2metric} below on geodesic distance).

\begin{Exercise}\label{ex:metricex}\vspace{-\baselineskip}
 \Question Verify \Cref{ex:Hilbert}: Every Hilbert space $(H,\langle\cdot,\cdot\rangle)$ is a strong Riemannian manifold.
 \Question Show that the mapping $\mu \colon C([0,1],H)^2 \rightarrow \R, (f,g) \mapsto \langle f,g\rangle_{L^2}$ is continuous bilinear (whence smooth), see \Cref{ex:L2Hilbert}.\\
 {\tiny \textbf{Hint:} $\int_0^1 \colon C([0,1],\R) \rightarrow \R, f \mapsto \int_0^1 f(t) \mathrm{d}t$ is linear. Exploit that $C([0,1],\R)$ is a Banach space to prove continuity via usual integral estimates.}
 \Question Verify that the Hilbert sphere is a strong Riemannian manifold  (\Cref{Ex:HSphereRiem}).
 \Question We shall treat Riemannian metrics on spaces of smooth functions on the sphere. \\
  {\tiny \textbf{Hint:} If you are not familiar with integration on manifolds: Use $\theta \colon [0,2\pi ] \rightarrow \SSS^1, t \mapsto (\cos(t),\sin(t))$ to reduce the integral to a usual integral, see note at the beginning of \Cref{sect:shapeanalysis} and compare \cite[Chapter 16]{Lee13}.} 
  \subQuestion Let $(M,g)$ be a strong Riemannian metric. Use \Cref{Whitneyiso} to show that the $L^2$-metric 
  $$g^{L^2}_c (h,k) \coloneq \int_{\SSS^1} g_{c(\theta)}(h(\theta),k(\theta)) \mathrm{d}\theta$$
  defines a weak Riemannian metric on $C^\infty (\SSS^1,M)$.\\
  \subQuestion Verify that the metric in \Cref{ex:invL2} is a weak Riemannian metric.
\end{Exercise}

\section{The geodesic distance on a Riemannian manifold}\label{sect:L2metric}

\begin{tcolorbox}[colback=white,colframe=blue!75!black,title=General assumptions]
In this section $(M,g)$ will always denote a weak Riemannian manifold. For convenience we shall always assume that $M$ is connected.
\end{tcolorbox}

Having a Riemannian metric at our disposal, we can define the length of curves:

\begin{defn}
 Let $c \colon [a,b] \rightarrow M$ be a piecewise $C^1$-curve.\index{curve!piecewise $C^1$}\footnote{i.e.\ there exists a partition of $[a,b]$ into subintervals such that on each of them the curve restricts to a $C^1$-curve.} Then we define the \emph{length}\index{curve!length of a} and the \emph{energy}\index{curve!energy of a} of $c$ as
 \begin{align*}
  \text{Len}(c) &\coloneq \int_a^b \sqrt{g_{c(t)}(\dot{c}(t),\dot{c}(t))}\mathrm{d}t \\
  \text{En}(c) &\coloneq \frac{1}{2}\int_a^b g_{c(t)} (\dot{c}(t),\dot{c}(t))\mathrm{d}t
 \end{align*}
 For $x,y \in M$ define then 
 $$\Gamma (x,y) \coloneq \{c \colon [0,1] \rightarrow M \mid c(0)=x,c(1)=y,\ \text{ and } c\text{ is piecewise } C^1\}.$$
 Finally, we define the \emph{geodesic distance}\index{geodesic distance} between points $x,y \in M$ as 
 $$\dist (x,y) \coloneq \inf_{c \in \Gamma(x,y)} \text{Len} (c) = \inf_{c \in \Gamma(x,y)} \int_0^1\sqrt{g_{c(t)}(\dot{c}(t),\dot{c}(t))}\mathrm{d}t. $$
\end{defn}

Due to the Cauchy-Schwarz inequality, for curves $c \colon [a,b] \rightarrow M$ we find (cf.\ \cite[Proposition 1.8.7]{MR1330918}) that
\begin{align}\label{Energyest}
\text{Len}(c)^2 \leq 2 \text{En}(c)(b-a)\qquad (\text{with equality if and only if } \dot{c} \text{ is constant}). 
\end{align}

\begin{rem}\label{rem:intervallindep}
 Note that since every interval $[a,b]$ is diffeomorphic to $[0,1]$, the chain rule implies that the definition of $\Gamma(x,y)$ and of the geodesic distance does not depend on $[0,1]$. It is only a convenient choice for us and we shall ignore this choice in the construction of paths to avoid cumbersome reparametrisation arguments.
\end{rem}

\begin{lem}\label{lem:pseudodist}
 The geodesic distance is a pseudo-distance, i.e.
 \begin{enumerate}
  \item $\dist(x,y) \geq 0$ for $x,y \in M$,
  \item $\dist (x,y) = \dist (y,x), \forall x,y \in M$,
  \item $\dist (x,z) \leq \dist (x,y) + \dist (y,z), \forall x,y,z \in M$
 \end{enumerate}
\end{lem}
The proof of \Cref{lem:pseudodist} is left as an exercise.
On strong Riemannian manifolds, the geodesic distance is also \emph{point separating}, i.e.\ in addition to the properties from \Cref{lem:pseudodist} it satisfies 
\begin{enumerate}
 \item[(d)] $\dist (x,y) \neq 0$ for all $x,y\in M$ with $x\neq y$. 
\end{enumerate}
Moreover, one can prove the following result for strong Riemannian metrics
\begin{thm}[{\cite[Theorem 1.9.5]{MR1330918}}]
 Let $(M,g)$ be a strong Riemannian metric. The function $\dist \colon M \times M \rightarrow \R$ defines a metric on $M$. The topology derived from $\dist$ coincides with the given topology of $M$.
\end{thm}

However, in general \textbf{on infinite-dimensional manifolds} the geodesic distance might not be a metric. Indeed it might be non point separating and even stronger, the geodesic distance might be vanishing.\footnote{We say the geodesic distance is \emph{non-vanishing}\index{geodesic distance!non-vanishing} if there exist $x,y \in M$ such that $\dist(x,y) \neq 0$. So every point separating geodesic distance is non-vanishing but not vice versa.} Note that if $c$ is a path connecting $x \neq y$, then $\text{Len} (c) >0$ (or in other words $\text{Len}(c)=0$ implies that the path is constant). Thus the vanishing of $\dist (x,y)$ means that we can find arbitrarily short paths connecting the two points. We showcase this in the following example.

\begin{ex}[{\cite{MaT19}}]\label{ex:Hilbertdistance_van}
 Consider the space $(\ell^2, \langle \cdot ,\cdot\rangle)$ of all square-summable real sequences, cf.\ \Cref{closedsubgrp}. The map $A \colon \ell^2 \rightarrow \ell^2, (x_n)_{n\in \N} \mapsto (\tfrac{1}{n^3}x_n)_{n\in \N}$ is continuous linear and induces a bilinear symmetric map $B \colon \ell^2 \times \ell^2 \rightarrow \R, B(\mathbf{x},\mathbf{y}) \coloneq \langle \mathbf{x},A\mathbf{y}\rangle$. Identifying the tangent spaces of $\ell^2$, we obtain a weak Riemannian metric via 
 $$g \colon T\ell^2 \oplus T\ell^2 = \bigcup_{\mathbf{p} \in \ell^2} \ell^2 \times \ell^2 \rightarrow \R,\quad T_{\mathbf{p}} \ell^2 \ni (\mathbf{x},\mathbf{y}) \mapsto  e^{-\lVert \mathbf{p}\rVert^2}B(\mathbf{x},\mathbf{y}).$$
 We will now prove that the weak Riemannian manifold $(\ell^2,g)$ has vanishing geodesic distance, i.e.\ that for every $\mathbf{p} \neq \mathbf{q}$ in $\ell^2$ we have $\dist (\mathbf{p},\mathbf{q})=0$. To this end, let $\mathbf{e}_n$ be the sequence with $1$ in the $n$th place and zeroes everywhere else. We construct a path from $\mathbf{p}$ to $\mathbf{q}$ via
 \begin{align*}
  c_n \colon [0,1] \rightarrow \ell^2, t \mapsto \begin{cases}
                                                  \mathbf{p}+3tn\mathbf{e}_n & t \in [0,1/3] \\
                                                  \mathbf{p}+n\mathbf{e}_n + (3t-1)(\mathbf{q}-\mathbf{p}) & t \in [1/3,2/3]\\
                                                  \mathbf{q}+(3-3t)n\mathbf{e}_n & t \in [2/3,1]
                                                 \end{cases}
 \end{align*}
 By construction $c_n$ is a piecewise linear curve connecting $\mathbf{p}$ to $\mathbf{q}$ and passing through $\mathbf{p}+n\mathbf{e}_n$ and $\mathbf{q}+n\mathbf{e}_n$ on the way.
 We claim that $\text{Len}(c_n) \rightarrow 0$ as $n\rightarrow \infty$ and thus $\dist (\mathbf{p},\mathbf{q})=0$. To see this let us observe that 
 $$c_n' (t) = 3n\mathbf{e}_n, t \in [0,1/3[\quad c_n' (t) = 3(\mathbf{q}-\mathbf{p}), t \in ]1/3,2/3[,\quad c_n' (t)=-3n\mathbf{e}_n,\ t \in ]2/3,1].$$
 Moreover, since $\mathbf{p}, \mathbf{q} \in \ell^2$ there is $N>0$ such that every component of $\mathbf{p}$ and $\mathbf{q}$ with $n>N$ satisfies $|\pi_n (\mathbf{p})|, |\pi_n (\mathbf{q})| < 1/3$ (here $\pi_n$ is the projection on the $n$th component). Hence we see that for every such $n$ we have $\pi_n (\mathbf{p}+t(\mathbf{q}-\mathbf{p}) >-1$
 We now estimate the length $\text{Len}(c_n )$:
 \begin{align*}
  &\int_0^1 c_n(t) \mathrm{d}t = 3\left(\int_0^{1/3} e^{-\lVert c_n (t)\rVert^2} \sqrt{B(n\mathbf{e}_n, n\mathbf{e}_n)} \mathrm{d}t + \int_{1/3}^{2/3} e^{-\lVert c_n (t)\rVert^2} \underbrace{\sqrt{B(\mathbf{q}-\mathbf{p},\mathbf{q}-\mathbf{p})}}_{\equalscolon K >0} \mathrm{d}t \right.\\ 
  &\hspace{3cm} \left. \qquad + \int_{2/3}^1 e^{-\lVert c_n (t)\rVert^2} \sqrt{B(-n\mathbf{e}_n,- n\mathbf{e}_n)} \mathrm{d}t \right)\\
  \leq& 3\left(\int_0^{1/3} \sqrt{\langle n\mathbf{e}_n, \frac{1}{n^2}\mathbf{e}_n)} \mathrm{d}t + K\int_{1/3}^{2/3} e^{-(n^2+\lVert \mathbf{p}+(3t-1)(\mathbf{q}-\mathbf{p})\rVert^2 + 2n \langle \mathbf{e}_n,  \mathbf{p}+(3t-1)(\mathbf{q}-\mathbf{p})\rangle)}\mathrm{d}t\right. 
  \\
  & \hspace{4cm}\left. +\int_{2/3}^{1} \sqrt{B(-n\mathbf{e}_n,- n\mathbf{e}_n)} \mathrm{d}t  \right)\\
  \leq & \frac{1}{\sqrt{n}} + 3K \int_{1/3}^{2/3} e^{-n^2-2n\pi_n (\mathbf{p}+t(\mathbf{q}-\mathbf{p}))} \mathrm{d}t + \frac{1}{\sqrt{n}} \stackrel{n > N}{\leq} \frac{2}{\sqrt{n}} + Ke^{-n^2+2n}  
 \end{align*}
 We conclude that the length of the curve $c_n$ converges to zero as $n\rightarrow \infty$ whence the geodesic distance vanishes.
\end{ex}

Another interesting example in this regard is the invariant $L^2$-metric on the group of circle diffeomorphisms. Before we state this example, let us exhibit a general construction principle for (weak) Riemannian metrics on Lie groups:

\begin{setup}[Invariant metrics on a Lie group]\label{setup:invmetric}
 Let $G$ be a (possibly infinite-dimensional) Lie group with Lie algebra $\Lf (G)$. Assume that $\langle \colon ,\colon \rangle \colon \Lf (G) \times \Lf (G) \rightarrow \R$ is a continuous inner product on the Lie algebra. 
 Then we define the \emph{right invariant (weak) Riemannian metric}\index{Riemannian metric!right/left invariant} via the following formula
 \begin{align}\label{inv-metric}
  \langle V, W \rangle_{g} \coloneq \langle T\rho_g^{-1}(V),T\rho_g^{-1}(W)),\quad \forall V,W \in T_g G.
 \end{align}
 Here $\rho_g$ is the right translation by $g$ and we remark that due to the smoothness of the group operations the resulting metric is indeed a (weak) Riemannian metric. By construction, the right invariant metric is invariant under the right action of the Lie group $G$ on $TG$ via right multiplication.
 
 Note that by replacing every $\rho_g$ in \eqref{inv-metric} by the left translation $\lambda_g$, we can obtain the \emph{left invariant (weak) Riemannian metric} associated to the given inner product. 
\end{setup}

\begin{ex}[Right-invariant $L^2$-metric on $\Diff(\SSS^1)$]\label{MMvanish}
 We consider $\Diff(\SSS^1)$ again as an open subset of $C^\infty (\SSS^1,\SSS^1)$. Recall from \Cref{ex:Diffgp} that this manifold structure turns the diffeomorphism group into a Lie group. Moreover, $T_\varphi\Diff(\SSS^1) = C^\infty_\varphi (\SSS^1,T\SSS^1) \cong C^\infty_\varphi (\SSS^1, \SSS^1 \times \R) \cong C^\infty (\SSS^1,\R)$.\footnote{Here we exploit that $\SSS^1$ is a Lie group, whence the tangent bundle $T\SSS^1 \cong \SSS^1 \times \R$ is trivial.}
 
 On $C^\infty (\SSS^1,\R)$ we have the inner product (where we refer to the discussion in the beginning of \Cref{sect:shapeanalysis} for the meaning of the integral):
 $$\langle u,v\rangle_{L^2} \coloneq \int_{\SSS^1} f(\theta)g(\theta)\mathrm{d}\theta.$$ Plugging this into \eqref{inv-metric} we obtain the \emph{(right)invariant $L^2$-metric}\index{Riemannian metric!invariant $L^2$-metric}
 $$g^{L^2,\text{inv}}_\varphi (u,v) \coloneq \langle u\circ \varphi^{-1},v\circ \varphi^{-1})_{L^2}$$
 Due to a theorem by Michor and Mumford, the geodesic distance with respect to this metric vanishes, cf.\, e.g.\, \cite[Theorem 4]{MR2434729}. 
\end{ex}

\begin{Exercise}\vspace{-\baselineskip}
 \Question Establish the estimate \eqref{Energyest}.
 \Question Prove \Cref{lem:pseudodist} and verify that the pseudodistance does not depend on the choice of interval (\Cref{rem:intervallindep}).
 \Question We supply the details for \Cref{ex:Hilbertdistance_van}: 
 \subQuestion Show that the map $A$ makes sense and is linear and continuous and thus the map $B$ is bilinear and symmetric.
 \subQuestion Prove that $g$ is a weak Riemannian metric on $\ell^2$.
 \Question Prove that the construction of right (or left) invariant metrics on a Lie group from \Cref{setup:invmetric} yields a weak Riemannian metric.\\
 {\tiny \textbf{Hint:} To check smoothness of the metric use \Cref{tangentLie} to identify the Whitney sum $TG \oplus TG$.}
 \end{Exercise}

\subsection*{Geodesics on infinite-dimensional manifolds (informal discussion)}
To get a better understanding of the distance on weak Riemannian metrics, it seems useful to study curves ``of shortest length'' between two points the so called \emph{geodesics}.\index{geodesic} Before we study geodesics in the next section, we discuss some aspects in an informal way:

If the manifold is a Hilbert space $(H,\langle \cdot,\cdot\rangle)$ viewed as a strong Riemannian manifold, \Cref{ex:Hilbert}, the curve of shortest length between $a,b \in H$ is the straight line 
\begin{align}\label{shortest}
\R \ni t \mapsto \gamma(t)\coloneq (b-a)t+a \in H. 
\end{align}
Note that this curve also satisfies $\text{dist}(\gamma(t),\gamma(s)) = \lVert b-a\rVert |t-s|$, whence it realises the shortest distance between any two given points on the line. On a manifold, we would like to compute curves which satisfy the same property at least locally, i.e.\, in a neighborhood of each point the curve realises the shortest distance for every pair of points on the curve.\footnote{The sphere $\SSS^1$ is a Riemannian manifold by \Cref{Ex:HSphereRiem}. For $x \in \SSS^1$ consider a closed curve running in a great circle from $x$ around the sphere. Then this curve realises the shortest distance from $x$ to another point $y$ until it passes the point antipodal to $x$. However, as long as we restrict to an open neighborhood which does not contain points antipodal to each other, the curve realises the shortest distance from one point to the other.} 
Due to \eqref{Energyest}, one can equivalently describe geodesics between $p$ and $q$ as curves of minimal energy, i.e.\ extrema of the energy $\text{En}$ restricted to $\Gamma(p,q)$. Hence to find geodesics we consider the derivative of the energy. Working locally in a chart $(U,\varphi)$ of $M$ (and suppressing most identifications in the notation, see \Cref{lem:energydifferential}) this yields for the derivative $d\text{En}(c;h)$ the formula
$$\int_0^1 \frac{1}{2} d_1g_U(c,c'(t),c'(t);h)) - d_1g(c(t),h(t),c'(t);c'(t)))-g_U(c(t),h(t),c''(t)) \mathrm{d}t.$$
To find the geodesics, one would now have to isolate $h$ in the expression to rewrite the differential in the form $\int_0^1 g_c(\ldots, h) \mathrm{d}t$ and extract the geodesic equation. In general this is not possible, as the $\flat$-map is not an isomorphism whence the existence of geodesics for weak Riemannian metrics is a priori unclear (cf.~also \Cref{prop:strongmetric}). 

We remark that the geodesics of weak Riemannian metrics are also of independent interest in the context of Euler-Arnold theory (see \Cref{sect:EAtheory}). There certain partial differential equations can be interpreted as geodesic equations on infinite-dimensional manifolds.

\begin{ex}[Inviscid Burgers equation]\label{Burger1} One can show \cite[3.2]{MR2434729}, that geodesics with respect to the invariant $L^2$-metric from \Cref{MMvanish} correspond to solutions of the inviscid Burgers equation (also known as Hopf equation)\index{Riemannian metric!invariant $L^2$-metric}
$$u_t + 3u u_x=0 \text{  (subscripts denoting partial derivatives)}$$
We shall investigate a similar situation in \Cref{ex:inv-burger} below.
\end{ex}
The observant reader should have noted that we are talking about geodesics (which are the solutions of the inviscid Burgers equation) for a weak Riemannian metric which was described in \Cref{MMvanish} as having vanishing geodesic distance. It might be tempting to think that this implies that all geodesics must be constant (since only these curves have length $0$ and geodesics were described as smooth curves (locally) minimizing the length between their endpoints). This however is wrong (the reader might either consult the PDE literature, or observe (see \Cref{ex:inv-burger}) that the related equation $u_t + uu_x=0$ is the geodesic equation of a weak Riemannian metric with non-vanishing geodesic distance) as the Burgers equation admits non-constant solutions, so what is wrong here? 
For strong Riemannian metrics one can show that every geodesic is (locally) length minimizing. However there are also other characterisations of geodesics (which we will discuss in the next section) which coincide with our informal definition for strong Riemannian metrics. In the weak setting the situation is (as always) more complicated.

\begin{Exercise}\vspace{-\baselineskip} \label{exercise:length}
 \Question We verify some details concerning geodesics in Hilbert space
 \subQuestion Show that the path $c(t) = (b-a)t+a, t \in [0,1]$, \eqref{shortest}, realises the shortest length with respect to all $C^1$-curves connecting $a$ to $b$ in the Hilbert space $H$.\\
 {\tiny \textbf{Hint:} Prove first for a path consisting of two line segments meeting at a point $p$ that its length is longer then the length of $c$ if $p$ does not lie on $c$. Use then that $C^1$-paths are rectifiable.}  
 \subQuestion Prove the distance property claimed after \eqref{shortest}: For all $s,t \in \R$, $\text{dist}(\gamma(s),\gamma(t))= \lVert b-a\rVert |t-s|$. Note that this is (up to reparametrisation to unit speed) the geodesic property for curves in metric spaces, cf.\, \cite[Definition 1.3]{MR1744486}.
\end{Exercise}

\section{Geodesics, sprays and covariant derivatives}\label{sect:geod_spray}

In this section, we consider three objects associated to a Riemannian metric (see e.g.~\cite{Lang,MR960687}). This will enable us to study geodesics on Riemannian manifolds. The idea is that every (strong) Riemannian metric induces a covariant derivative, a metric spray and a connector. These can be used to conveniently describe geodesics.
The main point is to introduce the concepts of spray and covariant derivative while we relegate many details of the constructions to the literature, e.g.~\cite[Chapter VIII]{Lang}.

\begin{tcolorbox}[colback=white,colframe=blue!75!black,title=General assumptions]
In this section $(M,g)$ will always denote a \textbf{strong} Riemannian manifold. Thus there will be no need to worry about the existence of solutions to certain differential equations. Note that for weak Riemannian metrics similar computations are in principle possible, but require the careful checking of several technical details.\footnote{We shall do this in \Cref{sect:shapeanalysis} for the $L^2$-metric on $C^\infty (\SSS^1,M)$.}
\end{tcolorbox}

We define first certain vector fields on the tangent manifold which are called sprays. For any spray, geodesics can be defined and if the spray is the metric spray associated to a Riemannian metric, these geodesics will turn out to be the geodesics of the Riemannian metric.

\begin{defn}[Spray]
 A vector field $S \in \mathcal{V} (TM)$ is said to be of \emph{second order}\index{vector field!second order} if
 $$T \pi_M (S(v))=v \text{ for all } v \in TM.$$
 For each $t \in \R$ we denote by $t_{TM} \colon TM \rightarrow TM$ the vector bundle morphism which
is given in each fibre by multiplication with $t$. Now a second order vector field $S$ is a \emph{spray}\index{spray} if
\begin{align}
 \label{spray}
S(tv) = T (t_{TM})(t \cdot S(v)) \text{ for all } t \in \R, v \in TM.
\end{align}
\end{defn}

To understand the meaning of \eqref{spray}, let us localise in a chart (Warning the next identities hold only up to identifications in charts, which we will suppress in the notation!). 

\begin{setup}\label{local:spray}Choose $U \opn M$ such that $TU \cong U \times E$ (where $E$ is the model space of $M$). Then $TTU \cong (U\times E) \times (E\times E)$. Now if $S$ is a second order vector field, its restriction to $U$ is given for $(u,V) \in U \times E $ by $S(u,V) =((u,V),V,S_2(u,V))$. If $S$ is a spray, the equation \eqref{spray} reads on $U$ as follows: 
$$S(u,tV) = (u,tV,tV,S_{U,2} (u,tV))= (u,tV,tV,t^2S_{U,2}(u,V))$$
The map $S_{U,2}$ is thus quadratic with respect to scalar multiplication in the fibre. Furthermore, this implies together with Exercise \ref{ex:spray} 1.~that $S_{U,2} (x,v) = \frac{1}{2}d_2^2S_2 ((x,0);(v,v))$.

We define $B \colon M \rightarrow C^\infty (TM\oplus TM,TM)$ as the map which is locally on a chart domain $(U,\varphi)$ given by the symmetric bilinear map $B_U(x;v,w) \coloneq d_2^2S_{U,2} ((x,0);(v,w))$. Associated to this bilinear map is the quadratic form $Q_U (x,v) \coloneq B_U (x;v,v)$.\footnote{We recall that due to the identity $B_U (x;v,w) = \frac{1}{2} (Q_U (x,v+w)-Q_U(x,v)-Q_U(x,w))$ the quadratic form carries exactly the same information as the bilinear form.} Note that in finite-dimensional Riemannian geometry $-B$ is represented by the so called \emph{Christoffel symbols},\index{Christoffel symbol} \cite[p.213-214]{Lang}.
\end{setup}

\begin{setup}[Integral curves of second order vector fields]
 Let $S$ be a smooth second order vector field on the manifold $M$ and assume that $S$ admits integral curves (i.e.\ $C^1$-curves $\beta \colon J \rightarrow TM$ with $S(\beta)=\dot{\beta}$\footnote{If the manifold $M$ is modelled on a Banach space, such curves always exist due to the standard ODE solution theory. Beyond Banach spaces existence of such curves needs to be established for the special case at hand.}. Note that an integral curve $\beta \colon J \rightarrow TM$ of a second order vector field $S$ satisfies $\dot{(\pi_M \circ \beta)}(t) = \beta(t)$. 
\end{setup}

\begin{defn}\label{defn:geodesic}
 A $C^2$-curve $\alpha \colon J \rightarrow M$ is a \emph{geodesic}\index{geodesic!of a spray} of the spray $S$, if $\dot{\alpha} \colon J \rightarrow TM$ is an integral curve of $S$, or equivalently, if $\alpha$ satisfies the \emph{geodesic equation}\index{geodesic equation}
 $$\ddot{\alpha}(t) = S\left(\dot{\alpha}(t)\right), \quad \forall t \in J.$$
 Note that $\dot{\alpha}(t) =T_t \alpha (1) \in TM$. Equivalently, the geodesic equation becomes (in local coordinates) the equation $\ddot{\alpha}(t) = B_{\alpha(t)}(\dot{\alpha}(t),\dot{\alpha}(t))$.
\end{defn}

\begin{ex}\label{ex:sprayassoc}
 Let $(M,g)$ be a strong Riemannian manifold. Then $g$ induces a spray $S^g \colon TM \rightarrow T^2M$ which we describe locally on a chart domain $U\opn M, TU = U \times E, T^2U = (U \times E)\times E \times E$ (again suppressing the identifications!). Namely, we think of the local representative $g_U \colon U \times E \times E \rightarrow \R, (x,v,w) \mapsto g_x (v,w)$ of the metric as a mapping with three components. Then we define the associated spray via $S_U^g(x,v) \coloneq ((x,v),(v,\Gamma_U (x,v)))$ for $v \in T_x M$, where the quadratic form $\Gamma$ is the unique map which satisfies
 \begin{align}\label{spray-formula}
  g_U (x,\Gamma_U (x,v),w) = \frac{1}{2} d_1 g_U (x,v,v;w) -d_1g_U (x,v,w;v),\quad \forall v,w \in T_xU.
 \end{align}
 Here $d_1g_U$ denotes the partial derivative with respect to the first component (cf.~\Cref{prop:rpd}). We leave the verification that the local mappings $S_U^g$ yield a spray $S^g$ on $TM$ to the reader (Exercise \ref{ex:spray} 2.) Note that the formula \eqref{spray-formula} makes sense for weak Riemannian metrics, but since $\flat$ is in general not surjective for a weak Riemannian metric (cf.~\Cref{prop:strongmetric}), it is a priori not clear whether the condition defines a unique map $\Gamma_U$.
\end{ex}

\begin{defn}[Metric spray]\label{defn:metricspray}
 Let $(M,g)$ be a weak Riemannian manifold. A spray $S$ is called \emph{metric spray}\index{spray!metric} if its associated quadratic form $Q$ satisfies for each chart $(U,\varphi)$ the formula \eqref{spray-formula}, i.e.
 \begin{align}\label{metricspray:Qform}
   g_U (x,Q_U (x,v),w) = \frac{1}{2} d_1 g_U (x,v,v;w) -d_1g_U (x,v,w;v),\quad \forall v,w \in T_xU.
 \end{align}
\end{defn}
 Let us remark that the metric spray $S^g$ associated to a Riemannian metric $g$ can be interpreted in the sense of Hamiltonian mechanics: It can be shown (see \cite[p. 493]{MR960687}) that the spray $S^g$ is the Hamiltonian vector field on $TM$ associated with the kinetic energy function $e \colon TM \rightarrow \R , v_x \mapsto \tfrac{1}{2}g_x (v_x,v_x)$. This has several interesting consequences, such as conservation of energy by the geodesics (again we refer to \cite[Supplement 8.1.B]{MR960687} for more information). We shall return to the relation between energy and geodesics in \Cref{sect:EAtheory}.
 
 The next example shows that there is not necessarily a metric spray. It was pointed out to me by C.\ Maor and we urge the reader to compare it with \Cref{thm:metricsprayident} below.

 \begin{ex}[A weak Riemannian metric without metric spray]\label{nometricspray}
  We return to the weak Riemannian metric on the Hilbert space $(\ell^2.\langle \cdot,\cdot\rangle)$ from \Cref{ex:Hilbertdistance_van}:
  $$g \colon T\ell^2 \oplus T\ell^2 = \bigcup_{\mathbf{p} \in \ell^2} \ell^2 \times \ell^2 \rightarrow \R,\quad T_{\mathbf{p}} \ell^2 \ni (\mathbf{x},\mathbf{y}) \mapsto  e^{-\lVert \mathbf{p}\rVert^2}\sum_{n=1}^\infty \frac{\mathbf{x}_n \mathbf{y}_n}{n^3}.$$
  where the subscripts denote elements of a sequence. A quick computation shows that
  $$d_1 g(\mathbf{p},\mathbf{x},\mathbf{y};\mathbf{w}) = -2\langle \mathbf{p},\mathbf{w}\rangle g(\mathbf{p},\mathbf{x},\mathbf{y})=-2g(\mathbf{p}, (n^3\mathbf{p}_n)_{n\in \N}, \mathbf{w})\sum_{n=1}^\infty \frac{\mathbf{x}_n \mathbf{y}_n}{n^3}$$
  where the last equality only makes sense if the sequence $(n^3\mathbf{p}_n)_{n\in \N}$ is contained in $\ell^2$. Plugging this identity into the right-hand side of \eqref{metricspray:Qform}, we see that 
  \begin{align*}
   &\frac{1}{2} d_1 g(\mathbf{p},\mathbf{x},\mathbf{x};\mathbf{w}) -d_1g (\mathbf{p},\mathbf{x},\mathbf{w};\mathbf{x}) = - g(\mathbf{p}, (n^3\mathbf{p}_n)_{n\in \N}, \mathbf{w})\sum_{n=1}^\infty \frac{\mathbf{x}_n^2}{n^2} + 2 \langle \mathbf{p},\mathbf{x}\rangle g(\mathbf{p}, \mathbf{x}, \mathbf{w})\\
   =& g(\mathbf{p}, 2 \langle \mathbf{p},\mathbf{x}\rangle\mathbf{x} -\left(\sum_{n=1}^\infty \mathbf{x}_n/n^3\right) (n^3\mathbf{p}_n)_{n\in \N},\mathbf{w}).
  \end{align*}
We conclude that if a spray existed, its quadratic form needs to be given by
$$Q(\mathbf{p},\mathbf{x}) = 2 \langle \mathbf{p},\mathbf{x}\rangle\mathbf{x} -\left(\sum_{n=1}^\infty \mathbf{x}_n/n^3\right)(n^3\mathbf{p}_n)_{n\in \N}$$ and this expression is ill defined if $(n^3\mathbf{p}_n)_{n\in \N}$ is not contained in $\ell^2$. Hence $g$ does not admit a metric spray.
  \end{ex}

\begin{Exercise}\label{ex:spray} \vspace{-\baselineskip}
 \Question Let $U \opn E$, where $E$ is a locally convex space, $0\in U$ and $f \colon U \rightarrow E$ is smooth. Prove that if $f$ is \emph{locally homogeneous of order $p$}, i.e.\, $f(tx)=t^pf(x), \forall (t,x) \in \R \times U,$ such that $tx\in U$, then $f(x)= \frac{1}{p!}d^pf(0;x,\ldots,x)$.
  \Question Show that the local formula \eqref{spray-formula} defines a spray on a chart domain $U$. Then show that for any other chart $V$, the change of charts relates the local formulae obtained in this way to each other, i.e.~the sprays can be combined to define a spray on $TM$. 
 \Question Let $S$ be a spray on a manifold $M$, $(U,\varphi)$ and $(V,\psi)$ be two charts of $M$ with change of charts $\tau \coloneq \psi \circ \varphi^{-1}$. Prove the following change of charts formulae for the local expression of the spray and the associated bilinear form (cf.\, \Cref{local:spray}):
 \begin{align}
  S_{V,2} (\tau(x),d\tau(x;v)) &= d^2\tau (x;v,v)+d\tau(x;S_{U,2}(x,v))\\
  B_V (\tau(x); d\tau(x;v),d\tau(x;w)) &= d^2\tau (x;v,w) + d\tau (x;B_U(x;v,w)) \label{locbilform}
 \end{align}
 \Question Define the subspace $S \coloneq \{\mathbf{x} \in \ell^2 \mid (n^3\mathbf{x}_n)_{n \in \N} \in \ell^2\}$ of $\ell^2$. Endow $S$ with the restriction of the weak Riemannian metric  
 $g_{\mathbf{p}} (\mathbf{x},\mathbf{y}) =  e^{-\lVert \mathbf{p}\rVert^2}\sum_{n=1}^\infty \frac{\mathbf{x}_n \mathbf{y}_n}{n^3}$ from \Cref{ex:Hilbertdistance_van}. Show that the resulting weak Riemannian metric
 \subQuestion has vanishing geodesic distance.
 \subQuestion admits a metric spray and deduce that the (non-)existence of a (smooth) metric spray does not imply the non-degeneracy of the geodesic distance.\\
 {\footnotesize \textbf{Remark:} Again I am indebted to C.\ Maor for pointing this example out to me.}
\end{Exercise}

\subsection*{Covariant derivatives}
On a locally convex space $E$, we can identify vector fields with functions, whence for two vector fields $X,Y \in C^\infty (E,E)$, we can differentiate $Y$ in the direction of $X$ via $X.Y (v) = dY(v,X(v))$ (cf.\, \Cref{chapter:LA:VF}). This yields again a smooth function from $E$ to $E$. On a manifold $M$, the corresponding construction for vector fields would be $TY \circ X \colon M \rightarrow T^2M$ which is obviously \emph{not} a vector field. Hence to define a derivative taking two vector fields to vector fields an additional structure is needed.

\begin{defn}
 A \emph{covariant derivative}\index{derivative!covariant} $\nabla \colon \mathcal{V}(M) \times \mathcal{V} (M) \rightarrow \mathcal{V}(M), (X,Y) \mapsto \nabla_XY$ is a bilinear map satisfying the properties
 \begin{enumerate}
  \item for $f \in C^\infty (M,\R), X,Y \in \mathcal{V}(M)$ we have 
  \begin{enumerate}
  \item[1.] $\nabla_{fX}Y =f\nabla_XY$,
  \item[2.] $\nabla_X(fY)= \mathcal{L}_X (f) Y+f\nabla_XY$ (where $\mathcal{L}_X(f)$ is the Lie derivation \eqref{Liederiv}). 
  \end{enumerate}
  \item $\nabla_XY-\nabla_YX = \LB[X,Y]$ for all $X,Y \in \mathcal{V}(M)$.
 \end{enumerate}
\end{defn}

We shall show now that every spray induces a covariant derivative.
Recall from \Cref{chapter:LA:VF} that for a chart $(\varphi,U)$ we denote by $X_\varphi$ the local representative of a vector field $X \in \mathcal{V}(M)$.

\begin{prop}\label{assoc:covD}
 Given a spray $S$ on $M$, there exists a unique covariant derivative $\nabla$ such that in a chart $(\varphi,U)$, the derivative is given by the local formula
 \begin{align}\label{loc:covD}
  (\nabla_X Y)_\varphi (x) = X_\varphi . Y_\varphi (x) - B_U (x; X_\varphi (x),Y_\varphi(x))
 \end{align}
 We call $\nabla$ the \emph{covariant derivative associated to the spray}\index{spray!associated covariant derivative} $S$. Suppressing the indices, the above formula reads $\nabla_XY = X.Y-B(X,Y)$.
\end{prop}

\begin{proof}
 Define $\nabla_XY$ locally over $U$ via the formula \eqref{loc:covD}. In Exercise \ref{ex:covd} 2.\, we shall see that this formula has all properties of a covariant derivative for vector fields on $U$ (and all of the defining properties of a covariant derivative localise on $U$!). Obviously we can repeat the construction for every chart in an atlas $\mathcal{A}$. It suffices now to prove that the family $((\nabla_XY)_\varphi)_{\varphi \in \mathcal{A}}$ induces a vector field, i.e.\, in view of \Cref{glueinglemma}, it suffices to prove that the local representatives are related by the change of charts. We will check this for $\tau \coloneq \psi\circ \varphi^{-1}$, i.e.\, we prove $d\tau \circ (\nabla_XY)_\varphi = (\nabla_XY)_\psi \circ \tau$. Note first that by construction $Y_\psi \circ \tau = d\tau \circ Y_\varphi$ and thus \eqref{ident:doublederiv} yields 
 \begin{align}
  \label{eq:chchhigh}
d(Y_\psi \circ \tau)(x;v) = d(d\tau (x;Y_\varphi)(x;v) = d^2\tau (x;Y_\varphi(x),v)+d\tau(x,dY_\varphi(x;v))
 \end{align}
 We now apply the change of charts formulae for the spray and associated bilinear form from Exercise \ref{ex:spray} 3.:
 \begin{align*}
 & (\nabla_{X} Y)_\psi(\tau(z)) = dY_\psi (\tau(z);d\tau(z;X_\varphi(z))-B_V (\tau(z);d\tau(z;X_\varphi(z)),d\tau(z;Y_\varphi(z)))\\
  \stackrel{\eqref{locbilform}}{=}& d(Y_\psi \circ \tau)(z;X_\varphi(z))-d^2\tau (z;Y_\varphi(z),X_\varphi(z))-d\tau(z;B_U(z;X_\varphi(z),Y_\varphi(z)))\\
  \stackrel{\eqref{eq:chchhigh}}{=}& d^2\tau (z;Y_\varphi(z),X_\varphi(z)) +d\tau(x,dY_\varphi(x;X_\varphi (z)))\\
 &-d^2\tau (z;Y_\varphi(z),X_\varphi(z))-d\tau(z;B_U(z;X_\varphi(z),Y_\varphi(z)))\\
 \stackrel{\hphantom{\eqref{eq:chchhigh}}}{=}& d\tau (z;(\nabla_XY)_\varphi(z))
 \end{align*}
where we have exploited that the second derivative $d^2\tau(z;\cdot)$ is symmetric by Schwarz theorem, Exercise \ref{Ex:Bastiani} 3.
\end{proof}

\begin{rem}\label{rem:cov_op}
 For later use, we observe that the covariant derivative depends only on the values of the field in whose direction we derivate: Let $X,Y$ be vector fields and $\nabla$ be a covariant derivative associated to a spray $S$. Then the formula \eqref{loc:covD} shows that $\nabla_X Y (p)$ depends only on $X(p)$ but not on the vector field $X$ in a neighborhood of $p$. 
\end{rem}

As soon as we have a covariant derivative associated to a spray one can define an associated curvature tensor. We recall the definition here for later use:

\begin{defn}\label{defn:curvature}
 Let $M$ be a manifold with a covariant derivative $\nabla$. For vector fields $X,Y,Z \in \mathcal{V} (M)$ we define the linear map 
 \begin{align*}
  R(X,Y) \colon& \mathcal{V} (M) \rightarrow \mathcal{V}(M) \text{ given by the formula }  \\
 R(X,Y)Z \coloneq& \nabla_X \nabla_Y Z -\nabla_Y \nabla_X Z - \nabla_{\LB[X,Y]}Z.
 \end{align*}
 Then one can show that $R$ is a trilinear map in the variables $X,Y,Z$ called the \emph{curvature}\index{curvature} associated to the covariant derivative, cf.~Exercise \ref{ex:covd} 6. If $S$ a spray inducing $\nabla$ one also says that $R$ is the curvature of $S$. Similarly if $\nabla$ is the metric derivative of a (weak) Riemannian metric $g$, we say that $R$ is the curvature of the metric $g$.
\end{defn}

Curvature (associated to the metric spray) is a fundamental invariant of a Riemannian manifold \cite{MR1330918,Lang,GaHaL04}. Here we mention only that the curvature of certain infinite-dimensional (weak and strong) Riemannian manifolds plays a crucial r\^{o}le in important applications of infinite-dimensional geometry such as Arnold's result \cite{MR202082} on the practical impossibility of long term weather forecasts. Also there is an interesting divide between the curvature of strong and weak Riemannian metrics: If $(M,g)$ is a strong Riemannian manifold, the curvature is always (locally) bounded (this follows from the fact that it can be represented as a smooth section into a suitable tensor bundle). However, there are examples of weak Riemannian manifolds with covariant derivative, such that the curvature is unbounded (wrt.\ the norm induced by the metric) locally as well as a multilinear operator on the tangent space at a single point (see Exercise \ref{ex:covd} 7.).

\begin{setup}
 Similar to the construction of a covariant derivative, every spray induces a bundle morphism $K \colon T^2M \rightarrow TM$ which is locally (on a chart domain $(U,\varphi)$) given by 
 \begin{equation}\label{connectionformula}
  K(x,y,v,w) \coloneq (x,w-B_U(x;y,v)). 
 \end{equation}
 This bundle morphism is a linear connection\footnote{A linear connection for a bundle $p\colon E\rightarrow M$ is given by a bundle map $k \colon TE \rightarrow E$ over the bundle projection which is given in bundle trivialisations by bilinear maps as in \eqref{connectionformula}. See e.g.\, \cite[1.5]{MR1330918}.} called the \emph{connector} of the spray.\index{spray!connector} 
 By definition of the connector, the associated covariant derivative is $\nabla_X Y = K \circ TY \circ X$. 
\end{setup}

\begin{rem}
 If the manifold admits smooth cut-off functions (e.g.\ if $M$ is finite-dimensional, cf.\, also \Cref{smooth:bumpfun}) one can show, \cite[VII, \S 2 Theorem 2.4]{Lang}, that every covariant derivative is associated to a smooth spray as in \Cref{assoc:covD}.
\end{rem}

\begin{setup}\label{setup:metricderiv}
 A covariant derivative $\nabla$ on a Riemannian manifold $(M,g)$ is called \emph{metric derivative}\index{derivative!metric derivative} if for all $X,Y,Z \in \mathcal{V} (M)$ the following equation holds 
 $$X.g(Y,Z) = g(\nabla_XY,Z) + g(Y,\nabla_XZ).$$ 
 Here $X.f = df \circ X$ is the derivative of $f \colon M \rightarrow \R$ in the direction of the field $X$. 
\end{setup}

\begin{thm}[{\cite[VIII. \S 4 Theorem 4.1 and 4.2]{Lang}}]\label{thm:metricsprayident}
 Let $(M,g)$ be a strong Riemannian manifold. Then $g$ admits a metric derivative and there exists a unique spray $S \colon TM \rightarrow T^2M$, the \emph{metric spray},\index{spray!metric spray} whose associated covariant derivative is the metric derivative.
\end{thm}

\begin{rem}
 To every (strong) Riemannian metric, there is an associated metric spray and a metric derivative. The metric spray turns out to be the one we computed in \Cref{local:spray} (whence it coincides with the metric spray in \Cref{defn:metricspray}). For a weak Riemannian metric we have no guarantee that a metric spray exists as \Cref{nometricspray} shows, but if a metric spray exists, its associated covariant derivative is the metric derivative.
\end{rem}

 We investigate some examples of Riemannian manifolds for which these objects can be computed explicitly.

\begin{ex}[The trivial example]\label{setup:trivial}
 Let $(H,\langle \cdot,\cdot\rangle)$ be a Hilbert space considered as a strong Riemannian manifold. Due to Exercise \ref{ex:covd} 1., the covariant derivative is $\nabla_XY=X.Y$. This implies that the bilinear form $B$ of the associated spray (aka the Christoffel symbols) need to vanish, we see that the connector and the metric spray associated to the metric are 
 \begin{align*}
  K \colon T^2H \cong H^4 & \rightarrow TH\cong H^2, (x,u,v,w) \mapsto (x,w)\\
  S \colon TH \cong H^2 & \rightarrow T^2H \cong H^4, (x,u) \mapsto (x,u,u,0)
 \end{align*}
 Finally, from the formula for the covariant derivative one sees that the curvature $R$ identically vanishes on $H$ (one also says that $H$ is flat).
\end{ex}

\begin{ex}[The submanifold example]\label{ex:submfd:covd}
Recall from \Cref{Ex:HSphereRiem} that the Hilbert sphere $S_H$ of a Hilbert space $(H,\langle \cdot,\cdot\rangle)$ is a submanifold of $H$. This structure turns $S_H$ into strong Riemannian manifold with respect to the pullback metric 
$$g_x (V,W) \coloneq \langle V,W\rangle \text{ for },W \in T_x S_H = \{y \in H \mid \langle x,y\rangle =0\}.$$
 We shall now describe the covariant derivative of this Riemannian metric. Define the smooth map $\text{pr} \colon S_H \times H \rightarrow H, (p,v) \mapsto v - \langle p,v\rangle p$ and note that for fixed $p \in S_H$, $\text{pr}_p\coloneq \text{pr}(p,\cdot)$ is just the orthogonal projector onto the tangent space $T_p S_H$.
 Since the tangent space $T_p S_H$ has been identified with the orthogonal complement of $p$ in $H$, we may identify vector fields on $S_H$ with smooth maps $X \colon S_H \rightarrow H$ which satisfy $\langle X(p),p\rangle =0 , \forall p \in S_H$. Using these identifications, we can then define define a map 
 \begin{align}\label{cov:submfd}
  \nabla^{S_H} \colon \mathcal{V}(S_H)^2 \rightarrow \mathcal{V}(S_H), \nabla^{S_H}_X Y (p)\coloneq \text{pr}_p (dY (p,X(p))).
 \end{align}
 Its a straight forward computation (Exercise \ref{ex:covd} 3.) that $\nabla^{S_H}$ defines a covariant derivative on $S_H$. Let us show now that it is the metric derivative of the pullback metric $g$.
 Pick vector fields $X,Y,Z \in \mathcal{V}(S_H)$. We can now compute as follows:
 \begin{align*}
   X.g(Y,Z)(p) &= X.\langle Y,Z\rangle (p) = d\langle Y,Z\rangle (p;X(p)) \\ 
   &= \langle dY(p;X(p));Z(p)\rangle + \langle Y(p), dZ(p;X(p))\rangle\\ 
   &= g(\nabla^{S_H}_X Y,Z) (p) + g(Y,\nabla^{S_H}_X Z)(p)
 \end{align*}
 where the last equality follows from the fact that vector fields on $S_H$ take their image in the orthogonal complement of the base point. Thus $\nabla^{S_H}$ is the metric derivative of the pullback metric. The connector and the metric spray are now simply the restrictions of the ones from \Cref{setup:trivial}.
 There is also an associated formula (called the Gauss equations) relating the curvature of the sphere to the one of the flat ambient space (see \cite[Example 1.11.6]{MR1330918}). 
\end{ex}

The formulae for covariant derivative, spray and connector in \Cref{ex:submfd:covd} might seem ad hoc and indeed there is not much to see due to the simplicity of the data of the ambient space. However, the whole procedure turns out to be a special case of a formula which relates the covariant derivative of an isometrically immersed submanifold to the covariant derivative of the ambient manifold (cf.~\cite[Theorem 1.10.3]{MR1330918}). For the general notion, one needs to define the covariant derivative of vector fields along smooth maps. For now, we shall define this concept only for the special case of a covariant derivative for vector fields along curves (but see \Cref{defn:covalong}).

\begin{setup}
For a $C^1$-curve $c\colon [a,b] \rightarrow M$ we say a $C^1$ curve $\alpha \colon [a,b] \rightarrow TM$ \emph{lifts} $c$ if $\pi_M \circ \alpha = c$. Denote by $\text{Lift}(c)$ the \emph{set of all lifts} of $c$. Note that the pointwise operations turn $\text{Lift}(c)$ into a vector space on which the smooth functions $C^\infty ([a,b],\R)$ act by pointwise multiplication.
\end{setup}

If $X$ is a (smooth) vector field on $M$, then for every $C^1$-curve $c \colon [a,b] \rightarrow M$, the curve $X \circ c$ is a lift of $c$. Note however, that not every lift of a curve needs to arise as the composition of a vector field and a curve (for example, if the curve intersects itself, there can be lifts taking different values for the different time parameters associated to the intersection).
\begin{setup}\label{setup:can:lift}
 Consider a $C^1$-curve $c \colon [a,b] \rightarrow M$ and a chart $(U,\varphi)$ of $M$. We define a local representative $c_U \coloneq \varphi \circ c|_{c^{-1}(U)}$ of $c$ in the chart $\varphi$. For curves with $\pi_M \circ \alpha = c$ of $M$, we also define the \emph{principal part}\index{curve!principal part of a} with respect to the chart $(U,\varphi)$. Namely, we set $T\varphi \circ \alpha|_{\alpha^{-1}(TU)} = (\varphi\circ c|_{\alpha^{-1}(TU)},\alpha_U) = (c_U (t),\alpha_U (t))$ for some $C^1$-map $\alpha_U$. 
 
 Furthermore, define for $c$ a curve $\dot{c} \colon [a,b] \rightarrow TM$ with the property $\pi_{TM} \circ \dot{c} = c$ as follows: In any chart $(U,\varphi)$ the principal part of $\dot{c}$ is $(\dot{c})_U (t) \coloneq (\varphi \circ c)'(t) =(c_U)'(t)$. Obviously the definition does not depend on the choice of charts and we note that if $c$ is a $C^2$-curve, then $\dot{c} \in \text{Lift} (c)$. We will later use the same notation for mappings from the circle $\SSS^1$ with values in a manifold. Note that (up to a harmless identification) the new definition will coincide with the one here, cf.\ \Cref{sect:L2}.
\end{setup}

\begin{prop}[{\cite[VIII, \S 3 Theorem 3.1]{Lang}}]\label{liftinglemma}
 Let $S$ be a spray on $M$ with associate bilinear form $B$ and $c \in C^2([a,b],M)$, then there exists a unique linear map 
 $$\nabla_{\dot{c}} \colon \mathrm{Lift}(c) \rightarrow \{\gamma \in C([a,b],TM) \mid \pi_M \circ \gamma = c\}$$
 which in a chart $(U,\varphi)$ is given by the formula
 \begin{align}\label{locform:deriv}
  (\nabla_{\dot{c}} \alpha)_U (t) = \alpha_U'(t)-B_U (c_U(t);c_U'(t),\alpha_U(t)).
 \end{align}
 Furthermore, $\nabla_{\dot{c}}$ acts as a derivation on multiplication with $C^1$-functions $\varphi$, i.e.
 $$\nabla_{\dot{c}}(\varphi \alpha) = \varphi'(t)\nabla_{\dot{c}}\alpha (t)+\varphi(t)\nabla_{\dot{c}(t)}\alpha(t).$$
\end{prop}

\begin{proof}
 The proof is similar to \Cref{assoc:covD} and left as Exercise \ref{ex:covd} 5.
\end{proof}

\begin{setup}
 For a $C^2$-curve $c$, a lift $\gamma$ is $c$\emph{-parallel}\index{curve!$c$-parallel} if $\nabla_{\dot{c}} \alpha =0$. In a chart $(U,\varphi)$ this is equivalent to $\alpha_U'(t) = B_U(c_U(t);c_U'(t), \alpha_U(t))$. Due to \Cref{defn:geodesic} $c$ is a geodesic for the spray $S$ if and only if $\dot{c}$ is $c$-parallel, i.e.\ if and only if it satisfies the \emph{geodesic equation}\index{geodesic equation}
 $$\nabla_{\dot{c}}\dot{c} =0.$$ 
 \end{setup}

 \begin{ex}
  From Exercise \ref{exercise:length} we see that a lift $\alpha \colon [a,b] \rightarrow H \times H$ of a curve $c$ in a Hilbert space $H$ is $c$-parallel if and only if its principal part $\text{pr}_2 \circ \alpha$ is constant, i.e. the lift corresponds to a curve which is parallel to $c$ in the usual sense of the word. 
 \end{ex}

From the informal discussion at the end of the last section we know that a geodesic for a Riemannian manifold should (at least locally) be the shortest path between points which are not far apart.\index{geodesic} From the presentation in this section this is not apparent. A full proof requires more techniques,\footnote{In \Cref{Energy:vs:geodesic} we shall see that a curve is geodesic if and only if it extremises energy. As energy bounds length, a geodesics extremises the length (i.e.\ it is locally of minimal length).}   whence we just state:

\begin{thm}[{\cite[Theorem 1.9.3]{MR1330918}}]
 Let $(M,g)$ be a strong Riemannian manifold, $p \in M$. Then there is an open $p$-neighborhood $U_p$ and a constant $\eta >0$ such that every geodesic starting in $U_p$ of length $<\eta$ is a curve of minimal length between its end points.\index{geodesic}
\end{thm}

Thus a geodesic is always at least locally a curve of minimal length among all curves connecting two (close enough) points on the geodesic.

\begin{Exercise}\label{ex:covd} \vspace{-\baselineskip}
\Question Show that for a Hilbert space $(H,\langle \cdot , \cdot \rangle)$ considered as a strong Riemannian manifold, the metric derivative is given by the usual derivative, i.e.\, $\nabla_X Y = X.Y = dY \circ (\id, X)$. Deduce that geodesics in this case are of the form $at+b, a,b\in H$ (cf.\ Exercise \ref{exercise:length} 2.)
 \\ 
 {\tiny \textbf{Hint:} The metric derivative is unique. If it is $X.Y$, then the bilinear map $B$ of the associated metric spray needs to vanish.}
 \Question Let $(U,\varphi)$ be a chart for the manifold $M$. 
  Show that the local formula \eqref{loc:covD} induces a covariant derivative on $\mathcal{V}(U)$.\\
 {\tiny \textbf{Hint:} Review \Cref{chapter:LA:VF} to prove the statement on the Lie bracket.}
 \Question Show that \eqref{cov:submfd} defines a covariant derivative on $S_H$.
 \Question Consider a Hilbert space $(H,\langle \cdot,\cdot\rangle)$ in the canonical way as a strong Riemannian manifold. Show that for a curve $c \in C^2 ([a,b],H)$ the covariant derivative $\nabla_{\dot{c}} f = \dot{f}_H$ for all $f = (c,f_H) \in \text{Lift} (c)$. Deduce that geodesics in $(H,\langle \cdot,\cdot\rangle)$ are lines in $H$. 
 \Question Use Exercise \ref{ex:spray} 3.~to work out a proof for \Cref{liftinglemma}. 
 \Question Let $M$ be a manifold with spray $S$ and associated covariant derivative $\nabla$. Work locally in a chart $(U,\varphi)$ of $M$, where we write $X_\varphi, Y_\varphi, Z_\varphi$ for the local representatives of vector fields and $B_U$ for the bilinear form associated to the spray. 
 \subQuestion Establish the following local formula for the curvature \index{curvature}
 \begin{align*}
  (R(X,Y)Z)_\varphi &= B_U (B_U(Y_\varphi , Z_\varphi) , X_\varphi)) - B_U (B_U(X_\varphi , Z_\varphi) , Y_\varphi))\\ & \quad + d_1 B_U (X_\varphi,Z_\varphi;Y_\varphi) -d_1B_U (Y_\varphi,Z_\varphi;X_\varphi). 
 \end{align*}
 {\footnotesize\textbf{Hint:} Recall that $B_U$ takes three arguments $B_U(x;X_\varphi(x),Y_\varphi(x))$ and use \Cref{assoc:covD}.}
 \subQuestion Deduce that $R \colon \mathcal{V}(M)^3 \rightarrow \mathcal{V}(M)$ is a trilinear map which satisfies\index{curvature!Bianchi identity}
 $$R(X,Y)Z + R(Y,Z)X+R(Z,X)Y = 0  \text{     (Bianchi identity)}.$$
 \Question Consider again the subspace $S \subseteq \ell^2$ with the weak Riemannian metric from Exercise \ref{ex:spray} 4.. Derive an explicit formula for the curvature $R(X,Y)Z$ and show that the curvature at the point $s = 0$ is unbounded.\\
 {\footnotesize \textbf{Hint:} Everything is local. Since $S \subseteq \ell^2$ is a subspace of a Hilbert space, $S$ admits smooth bump functions. Thus every vector can be continued to a smooth vector field and we write the curvature now for vectors (understanding that they can be continued to vector fields). Let $e_i$ be the vector in $\ell^2$ whose only non-zero entry is $1$ in the $i$th place. Set $u_i = \tfrac{1}{\sqrt{i^3}}e_i$ and work out a formula for $g_0 (R(u_i,u_j)u_j|_{s=0},u_i)$ by exploiting that the $e_i$ are orthonormal wrt.~to $g_0$. What happens for $i,j\rightarrow \infty$?} 
\end{Exercise}

\subsection*{Weak Riemannian metrics with and without metric derivative}

For a weak Riemannian metric one can not expect in general that there exists a metric derivative associated to the metric. 
An example of a weak Riemannian metric without an associated covariant derivative can be found in \cite[p.12]{MR3265197}: For a Sobolev-type right invariant metric on a certain subgroup of the diffeomorphism group $\Diff(\R)$, the geodesic equation and the covariant derivative do not exist on the subgroup.\footnote{We will not discuss any details here as this would require at least a discussion of manifold structures on function spaces with non-compact domain. It is worth remarking though that the geodesic equation of this metric is related to the non-periodic Hunter-Saxton equation, see \cite[p.12 Theorem]{MR3265197}.}

\begin{rem}
To remedy the problem that the categorisation of weak and strong Riemannian metrics is not sharp enough to capture existence of covariant derivatives and related important structures for weak Riemannian metrics (see \Cref{sect:shapeanalysis} for an example of a weak Riemannian metric which admits a covariant derivative, connector and metric spray). several authors have proposed a finer classification of infinite-dimensional Riemannian metrics.

We mention here the following concepts and refer the interested readers to the original sources for more information.
\begin{itemize}
  \item Micheli, Michor and Mumford define in \cite{MR3098790} a \emph{robust Riemannian metric}\index{Riemannian metric!robust} as a weak Riemannian metric $g$ on $M$ such that
  \begin{enumerate}
   \item the associated metric derivative exists, and
   \item the Hilbert space completions $\overline{T_xM}^{g_x}$ form a smooth vector bundle $\bigcup_{x \in M} \overline{T_xM}^{g_x}$ whose trivialisations extend the trivialisations of the bundle $TM$. 
  \end{enumerate}
  In particular, this setting is strong enough to enable certain calculations of curvature for the weak Riemannian manifold.
 \item Stacey strengthens the notion of a weak Riemannian structure in \cite{stacey2008construct} via an additional structure: As observed in \Cref{musicaliso}, one of the main difference between the strong and the weak setting is the failure of the mapping $\flat \colon TM \rightarrow T^\ast M$ to be an isomorphism. While this can not be directly remedied, requiring a so called \emph{co-orthogonal structure}\index{Riemannian metric!co-orthogonal structure} allows one to obtain a map replacing the inverse $\sharp$ (which does not exist) of $\flat$ by a mapping with dense image. Exploiting the additional structure it is possible to define a Dirac operator on loop spaces. 
 \end{itemize}
\end{rem}

We refer to \Cref{sect:L2} for computations of metric derivatives for the (weak) $L^2$- and $H^1$-metrics.
 \section{Geodesic completeness and the Hopf-Rinow theorem}\label{sect:HopfRinow}

In this section we investigate geodesic completeness of infinite-dimensional Riemannian manifolds and the Hopf-Rinow theorem.
We will see that this theorem fails in infinite-dimensional geometry as it is built on top of (local) compactness of the underlying manifold. Let us first recall some definitions.

\begin{defn}\label{defn:completeness}
 Let $(M,g)$ be a (weak) Riemannian manifold. Then $M$ 
 \begin{enumerate}
  \item is \emph{metrically complete}, if the geodesic distance turns $M$ into a complete metric space (i.e.\, Cauchy sequences with respect to the geodesic distance converge).
  \item is \emph{geodesically complete} if every geodesic can be continued for all time.
  \item \emph{has minimizing geodesics} if for every two points in the same connected component of $M$, there exists as length minimizing geodesic $c$ connecting the points (i.e.\, for $a,b$ there is a geodesic $c$ with $\text{Len}(c) = \text{dist} (a,b)$)
 \end{enumerate}
\end{defn}

In finite-dimensions the Hopf-Rinow theorem\index{Hopf-Rinow theorem} holds (see \cite[Theorem 2.1.3]{MR1330918}):

\begin{thm}[Hopf-Rinow] \label{thm:Hopf-Rinow}
 Let $(M,g)$ be a \emph{finite-dimensional} Riemannian manifold. Then $M$ is metrically complete if and only if it is geodesically complete. Moreover, if $M$ is metrically complete it has minimizing geodesics.
\end{thm}

In infinite-dimensional settings \Cref{thm:Hopf-Rinow} does not hold. What remains true however is that metrical completeness of a strong Riemannian manifold implies geodesic completeness. 

\begin{proof}[Metrically complete implies geodesically complete]
Let $c_X \colon J \rightarrow M$ be a geodesic with $\dot{c}(0)=X \in T_p M$. We argue by contradiction and assume that $\sup J = t^+ < \infty$. We can reparametrise $c_X$ such that $\text{Len}(c_X|_{[r,s]}=|r-s|$ for all $r,s \in J$. Pick a Cauchy sequence $(t_n)_{n \in \N}$ with $t_n \nearrow t^+$. As 
$$\text{Len} (c_X(t_n),c_X(t_m)) \leq |t_n-t_m|$$
we see that the points $(c(t_n))_{n\in \N}$ form a Cauchy-sequence with respect to the geodesic distance, whence they converge towards some limit $q \in M$. Pick now $\varepsilon >0$ so small that $B_{\varepsilon}^{\text{dist}}(q)$ is the domain of Riemannian normal coordinates \cite[VIII, \S 6. Theorem 6.4]{Lang}. For $t^+-t_0 < \varepsilon /2$ there is a geodesic $\gamma$ with $\dot{\gamma}(0)= \dot{c}(t_0)$. Note that by uniqueness of geodesics $\gamma(t) = c(t+t_0)$. Since $\gamma$ is contained in the normal coordinates around $q$, the domain of definition of $\gamma$ contains at least $[-\varepsilon /2,\varepsilon/2]$, whence $c$ can be extended beyond $t^+$.
\end{proof}

 In infinite-dimensions the following example shows that metric and geodesic completeness do not imply existence of minimising geodesics. This is mainly a consequence of the lack of local compactness as the generalised version of the Hopf-Rinow theorem (in the context of metric spaces) shows, cf.\ \cite[Proposition 3.7]{MR1744486}. 
 
 \begin{ex}[{Grossman's ellipsoid, \cite{McA65}}]\label{GmEllipsoid} \index{Grossman's ellipsoid}
  Let $H= \ell^2$ be the Hilbert space of square summable sequences with the orthonormal basis $(\mathbf{e}_n)_{n\in \N}$ (where $\mathbf{e}_n$ is the sequence with a $1$ in the $n$th place and $0$s everywhere else). Recall that the inner product on $\ell^2$ is $\langle (x_n)_n , (y_n)_n\rangle = \sum_n x_ny_n$. We define $a_1 = 1$ and $a_n = 1+2^{-n}$ for $n\geq 2$ and consider the ellipsoid
  $$E \coloneq \left\{(x_n)_{n \in \N} \in \ell^2\middle| \sum_{n \in \N} \frac{x_n^2}{a_n^2} =1\right\}.$$
  Defining the smooth diffeomorphism $F \colon H \rightarrow H, (x_n)_{n \in \N} \mapsto (a_n x_n)_{n\in \N}$ we see that the ellipsoid is the image of the Hilbert sphere $E = F(S_H)$. As the Hilbert sphere is a submanifold of $H$ by \Cref{ex: HSphere}, so is the ellipsoid by Exercise \ref{ex:HR} 2. We endow the ellipsoid with the strong Riemannian metric induced by the embedding $E \subseteq H$ whence it becomes a strong Riemannian manifold. Note that also $S_H$ is a Riemannian manifold with respect to the induced metric. Due to \cite[1.10.13 (iii)]{MR1330918} geodesics in this Riemannian manifold are given by great circles and this shows that $S_H$ is complete (we shall study such a great circle in the next step). Note that by definition of $F$ we have $F(r\mathbf{e}_1) = r\mathbf{e}_1 \forall r \in \R$. Moreover, if $\gamma$ is a path in $S_H$ connecting $\mathbf{e}_1$ to $-\mathbf{e}_1$, then $F\circ \gamma$ is a path connecting $\mathbf{e}_1$ and $-\mathbf{e}_1$ in $E$. As $F$ is a diffeomorphism every path connecting $\mathbf{e}_1$ and $-\mathbf{e}_1$ in $E$ arises in this way. Now due to Exercise \ref{ex:HR} d), we can apply \cite[VIII \S 6, Theorem 6.9]{Lang} which implies that also $E$ is complete (and geodesically complete).
  
   We shall now prove that $\text{dist}_E (\mathbf{e}_1,-\mathbf{e}_1)=\pi$, but there exists no path realising the minimal distance.
   Moreover, we shall prove in Exercise \ref{ex:HR} 2. that the half circle $\gamma (t) = \cos(\pi t)\mathbf{e}_1 + \sin(\pi t) \mathbf{e}_2, t \in [0,1]$ is a geodesic connecting $\mathbf{e}_1$ and $-\mathbf{e}_1$, whence $\text{dist} (\mathbf{e}_1,-\mathbf{e}_1)$ in $S_H$ is $\pi$.  Consider now an arbitrary $\gamma \colon [0,1] \rightarrow H, \gamma(t) = (\gamma_n (t))_{n\in \N}$ in $S_H$ connecting $\mathbf{e}_1$ and $-\mathbf{e}_1$. For the length of the paths in $S_H$ and $E$ we obtain as a special case of Exercise \ref{ex:HR} 3. the following
  \begin{align}\label{estimate:ellipsoid}
  \pi \leq  \text{Len}(\gamma) = \int_0^1 \sqrt{\sum_{n\in \N} \dot{\gamma}_n(t)^2} \mathrm{d}t \leq \int_0^1 \sqrt{\sum_{n\in \N} a_n^2\dot{\gamma}_n(t)^2} =\text{Len}(F(\gamma)) 
  \end{align}
  By definition of $F$, we have equality $\text{Len}(\gamma)=\text{Len}(F\gamma)$ if and only if $\dot{\gamma}_n(t) =0$ for $n \geq 2$. However, the only curve starting in $\mathbf{e}_1$ satisfying this is the constant curve. Hence we have for all curves joining $\mathbf{e}_1$ and $-\mathbf{e}_1$ the strict inequality $\text{Len}(F\circ \gamma) >\pi$. Considering now the half ellipse $\gamma_n(t) = F(\cos (\pi t)\mathbf{e}_1 + \sin(\pi t)\mathbf{e}_n), t \in [0,1]$ joining $\mathbf{e}_1$ and $-\mathbf{e}_1$ in the $(\mathbf{e}_1,\mathbf{e}_n)$-plane, then 
 $$\text{Len}(\gamma_n) = \int_0^1\pi \sqrt{\sin^2(\pi t)+(1+2^{-n})^2\cos^2(\pi t)}\mathrm{d}t\leq \sqrt{1+2^{-n}}\pi \rightarrow \pi = \text{dist}_E(\mathbf{e}_1,-\mathbf{e}_1).$$
 \end{ex}

 We have now seen the failure of the Hopf-Rinow theorem in infinite-dimensions. Note that explicit counter examples are known for the other relations (between completeness, geodesic completeness and existence of minimizing geodesics) established in the Hopf-Rinow theorem. See \cite{MR400283,MR1432537} for more information.
 
 \begin{rem}\label{rem:HopfRinow_seminegative}
  It should be mentioned that parts of the Hopf-Rinow theorem can be salvaged also in the infinite-dimensional setting. However, then certain assumptions on the curvature of the Riemannian manifold are needed. For a manifold with seminegative curvature geodesic completeness implies completeness, see e.g.~\cite[IX. \S 3]{Lang}.
 \end{rem}

\begin{Exercise}\label{ex:HR} \vspace{-\baselineskip}
 \Question Let $c$ be a geodesic a Riemannian manifold $(M,g)$ with covariant derivative $\nabla$.
 \subQuestion Show that if $\varphi \colon \R \rightarrow \R, t \mapsto at+b$ for $a,b \in \R$, then also the reparametrisation $c \circ \varphi$ is a geodesic.  
 \subQuestion Let $\text{Len} (c|_{[a,b]})=L|a-b|$. Show that there is a reparametrisation $\psi$ such that $\text{Len} (c \circ \psi|_{[c,d]})=|c-d|$ for all $c,d$ in the domain of the reparametrised geodesic.
 \Question We check various details for \Cref{GmEllipsoid}.
 \subQuestion Show that the mapping $F \colon H \rightarrow H,\ F(\sum x_ne_n) = \sum a_n x_n e_n$ is a smooth map with smooth inverse.\\
 {\tiny \textbf{Hint:} Observe that $F$ is linear and use that limits and series can be exchanged if the series is dominated by a convergent majorant.}
 \subQuestion Let $S \subseteq M$ be a split submanifold and $F \colon M \rightarrow N$ a diffeomorphism. Show that $F(S)$ is a split submanifold of $N$.\\
 {\tiny \textbf{Hint:} By \cite[Lemma 1.13]{Glofun} $S \subseteq M$ is a submanifold if and only if the inclusion $S \rightarrow M$ is an immersion which is a topological embedding onto its image (an $f$ satisfying these properties is usually called embedding).}
 \subQuestion The covariant derivative on $S_H$ as a submanifold of $H$ with the induced metric is given as $\nabla = \text{pr} \circ \nabla^H$, where $\nabla^H$ is the covariant derivative on $H$ and $\text{pr}$ the projector onto the tangent space of $S_H$ (see \Cref{ex:submfd:covd}). Prove that the half-circle $\gamma(t) = \cos(\pi t)e_1 + \sin(\pi t)e_2$ is a geodesic in $S_H$.
  \subQuestion Show that for all $x, v \in H$ we have $\lVert T_xF(v)\rVert \geq \lVert v\rVert$.
 \Question Generalise \eqref{estimate:ellipsoid} in the following way: If $f \colon (M,g) \rightarrow (N,h)$ is a map between Riemannian manifolds such that there exists a constant $C >0$ with
 $$ \sqrt{h_{f(\pi(v))}(Tf(v),Tf(v))} \geq C \sqrt{g_{\pi(v)}(v,v)} \qquad \text{ for all }v \in TM,$$ 
 show that for a piecewise $C^1$-path $\gamma \colon [a,b] \rightarrow M$ we have $\text{Len}(f\circ \gamma) \geq C \text{Len}(\gamma)$. 
 \end{Exercise}

\chapter{Weak Riemannian metric with application in shape analysis}\label{sect:shapeanalysis}\copyrightnotice

In this chapter we study in detail the (weak) $L^2$-metric on spaces of smooth mappings. Its importance stems from the fact that this metric and its siblings, the Sobolev $H^s$-metrics are prevalent in shape analysis. It will be essential for us that geodesics with respect to the $L^2$-metric can explicitly be computed. Before we look into the specifics, let us clarify what we mean here by shape and shape analysis.\index{shape analysis}
Shape analysis seeks to classify, compare and analyse shapes. As a mathematical discipline, shape analysis goes back to the classical works by D`Arcy Thompson \cite{Tomp42} (originally published in 1917). In recent years there has been an explosion of applications of shape analysis to diverse areas such as computer vision \cite{CaEaS16}, medical imaging, registration of radar images and many more (see \cite{BaBaM14} for an expository article). There are different mathematical settings as to what is meant by a shape and what kind of data describes it. Popular choices are for example \begin{itemize}
            \item points, 
            \item curves (or surfaces) in euclidean space of manifolds, 
            \item level sets of functions, or
            \item images.
\end{itemize}
Another typical feature in (geometric) shape analysis is that one wants to remove superfluous information from the data. For example, in the comparison of shapes, rotations, translations, scalings and reflections are typically disregarded as being inessential differences. Conveniently, these inessential differences can mostly be described by actions of suitable Lie groups (such as the rotation and the diffeomorphism groups). This hints at the general process of constructing (an infinite-dimensional) manifold of shapes: One starts out with an (infinite-dimensional) manifold of data (e.g.~smooth curves) called the pre-shape space. Then the undesirable information is removed by quotienting out suitable group actions (e.g.~if reparametrisation invariance of the shapes is desired quotients out a suitable diffeomorphism group). The resulting quotient is then called shape space and one seeks to construct suitable tools (such as a Riemannian metric) to compare, classify and analyse the objects in shape space.

In the present chapter we will restrict our attention to shapes which arise as images of smooth curves which take their values in $\R^2$. The pre-shape space will thus be the infinite-dimensional manifold of smooth immersions (from the circle) with values in the two-dimensional space. As the objective is to compare images of these curves, we need to remove the specific parametrisation of the immersion. Hence we pass to shape space by quotienting out an action of the diffeomorphism group on the immersions (modelling the reparametrisation). Our aim is then to construct a suitable Riemannian metric on shape space which will allow us to compare shapes using the geodesic distance induced by it.

\section{The $L^2$-metric and its cousins}\label{sect:L2}
Having now discussed several ideas to remedy the problems in the (in general) ill behaved weak Riemannian setting, we consider now several weak Riemannian metrics which admit metric derivatives, spray etc. The metrics which we will consider are on one hand the $L^2$-metric. We will see that the $L^2$-metric admits a spray, a connector and a covariant derivative albeit being only a weak Riemannian metric. This metric will play a decisive r\^{o}le in the investigation of shape analysis in \Cref{sect:shapeanalysis}. The construction for the $L^2$-metric follows the argument first presented in \cite{BruL2}. 
Finally, we will briefly describe a Sobolev type metric whose covariant derivative can explicitly be given and is of independent interest.
Before we begin, let us set some conventions concerning the integration of functions on manifolds:

\begin{tcolorbox}[colback=white,colframe=blue!75!black,title=On integration of functions on $\SSS^1$]
The unit circle $\SSS^1$ can be parametrised (up to the double endpoint) by $\theta \colon [0,2\pi] \rightarrow \SSS^1, t \mapsto (\cos(t),\sin(t))$. We abuse now notation and denote by $\theta$ both the parameter and the parametrisation. If you have not seen integration theory on submanifolds of $\R^d$ (e.g.\ \cite[XVI.]{Lang}), this implies that for a continuous $f \colon \SSS^1 \rightarrow \R^d$, the integral on $\SSS^1$ satisfies
$$\int_{\SSS^1} f(\theta) \mathrm{d} \theta = \int_0^{2\pi} f(\cos(t),\sin(t)) \mathrm{d}t,$$
whence it can be computed as a usual one-dimensional integral. Further, for differentiable maps $c \colon \SSS^1 \rightarrow \R$ we write $c^\prime (\theta) \coloneq T_\theta c(1)$ where $1 \in T_\theta \SSS^1 \cong \R$ (the notation was previously reserved for curves and we justify it as $T_{\theta(t)} c(1)$ is (up to the linear isomorphism $T_t \theta$) given by the curve differential $(c\circ \theta)^\prime$) 
\end{tcolorbox}

\begin{rem}
 None of the techniques employed in this chapter depend on the compact manifold $\SSS^1$. We can thus skip a discussion of integration against a volume form on a compact manifold. However, we remark that all of the results in this section carry over if we replace $\SSS^1$ by an arbitrary compact manifold (we invite the reader to check this for themselves).  
\end{rem}

As in \Cref{ex:L2Hilbert} consider the space $C^\infty (\SSS^1,M)$ for a strong Riemannian manifold $(M,g)$ with the $L^2$-metric\index{Riemannian metric!$L^2$-metric}
\begin{align}\label{l2metric}
 g^{L^2}_c (f,g) = \int_{\SSS^1} g_{c(\theta)} (f(\theta),g(\theta)) \mathrm{d}\theta \qquad f,g \in T_c C^\infty (\SSS^1,M).
\end{align}
We shall show that this weak Riemannian metric admits a metric spray and a metric derivative. It will turn out that all the relevant object can be lifted to from the target manifold to the manifold of mappings.

 \begin{setup}\label{setup:Conlift}
 Let $(M,g)$ be a strong Riemannian manifold with metric spray $S$, connector $K$ and metric derivative $\nabla$. Since $(M,g)$ is a strong Riemannian manifold, it admits a local addition \cite[Lemma 10.2]{Mic}, whence the manifold structure on $C^\infty (\SSS^1,M)$ is canonical and in addition $T C^\infty(\SSS^1,M) \cong C^\infty (\SSS^1,TM)$ and $T^2C^\infty (\SSS^1,M) \cong C^\infty (\SSS^1,T^2M)$ (cf.\, \Cref{App:canmfdmap}). 
 Then the pushforward of the spray and the connector 
 \begin{align*}
  K_* \colon C^\infty (\SSS^1,T^2M) &\rightarrow C^\infty (\SSS^1,TM),\quad q \mapsto K\circ q\\
  S_* \colon C^\infty (\SSS^1,TM) &\rightarrow C^\infty (\SSS^1,T^2M),\quad v \mapsto S\circ v
 \end{align*}
 are smooth mappings by \Cref{la-reu}. Moreover, the identification of the tangent bundles immediately shows that $S_*$ is a spray on $C^\infty (\SSS^1,M)$ and that $K_*$ is a connector.
 \end{setup}
 
 \begin{prop}\label{prop:tediousdetails}
  In the situation of \Cref{setup:Conlift}, the connector associated to $S_*$ is $K_*$.
 \end{prop}

 \begin{proof}
  By definition, the connector of $S_*$ is uniquely determined by the associated map $B_{S_*}$. If we can show that $B_{S_*}$ coincides with the pushforward $B_*$ of the map $B$ associated to $S$, then the definition of the connectors yield that the connector of $S_*$ is $K_*$.
  To compute $B_{S_*}$ we need to compute $d_2^2 (S_*)_{O,2} ((x,0);\cdot)$, where $(S_*)_{O,2}$ is a local representative of $S_*$ (cf.\, \Cref{local:spray}). Since the bundle trivialisations are complicated for manifolds of mappings, we apply instead the exponential law to see that instead we can just take partial derivatives of 
  $$(S_*)^\wedge \colon C^\infty (\SSS^1,TM) \times \SSS^1 \rightarrow T^2M, \quad (h,\theta) \mapsto S(h(\theta)).$$
  Observe that we only need to take partial derivatives with respect to the first component thus $\SSS^1$ is simply a parameter set and the formula holds since it can be checked pointwise (see Exercise \ref{EX:L2} 1.).
  \end{proof}

 We use the connector $K$ of the covariant derivative on $(M,g)$ to define a covariant derivative for vector fields along smooth maps. 
 
  \begin{defn}\label{defn:covalong}
   Let $N$ be a smooth manifold and $F \colon N \rightarrow M$ be a smooth map to a Riemannian manifold $(M,g)$. A mapping $s \in C^\infty_F (N,TM) = \{s \in C^\infty (N,TM) \mid \pi_M \circ s = F\}$ is called a \emph{vector field along $F$}. \index{vector field!along a smooth map}
   Assume that $(M,g)$ admits a metric spray with associated connector $K$. For $X \in \mathcal{V}(N)$, we define
  \begin{align}\label{newcovd}
   \nabla_X^gs \coloneq K \circ Ts \circ X \colon N \rightarrow TM.
  \end{align}
  \end{defn}
  Note that the construction in \Cref{defn:covalong} is a generalised version of \Cref{liftinglemma}. If $\nabla$ is the metric derivative of a Riemannian manifold which admits a metric spray then the \emph{covariant derivative along}\index{derivative!covariant, along a map} $F \colon N \rightarrow M$ also satisfies a version of \Cref{setup:metricderiv} (cf.\ e.g.\ \cite[Proposition 1.8.14]{MR1330918}):
  \begin{align}\label{covalong:Riemannian}
   X.g(Y,Z) = g(\nabla_X Y,Z) + g(Y,\nabla_X Z), \qquad Y,Z \text{ vector field along } F, X \in \mathcal{V}(N)
  \end{align}

  We will exploit the exponential law for the canonical manifold $C^\infty (\SSS^1,M)$. To $s \in C^\infty (N,C^\infty(\SSS^1,TM))$ we associated the map $s^\wedge \in C^\infty (N\times \SSS^1,M)$. 
  
  \begin{rem}
   We face a problem as in \Cref{ex:Liealgdiffeo}: The exponential law suggests to work with vector fields on $N \times \SSS^1$,  while the covariant derivative in \eqref{newcovd} is only defined for vector fields on $N$. Luckily, vector fields on product manifolds are products of vector fields on the parts. Hence we extend $X \in \mathcal{V}(N)$ to a vector field on $N \times \SSS^1$ by the zero vector field on $\SSS^1$ via $X \times \mathbf{0}_{\SSS^1} \in \mathcal{V} (N\times \SSS^1)$. This allows us to obtain vector fields on the correct manifold on which we can now extend the covariant derivative of $N$.
  \end{rem}
  
 \begin{setup}[A covariant derivative on $C^\infty (\SSS^1,M)$]\label{defn:covd}
   Choosing $N = C^\infty (\SSS^1,M)$ we define via \eqref{newcovd} a map $\nabla^g_{X \times \mathbf{0}_{\SSS^1}} s^\wedge \in C^\infty (C^\infty (\SSS^1,M) \times \SSS^1,TM)$ and set
  \begin{align}
  \nabla_Xs \coloneq (\nabla^g_{X \times \mathbf{0}_{\SSS^1}} s^\wedge)^\vee \in C^\infty (C^\infty (\SSS^1,M),C^\infty (\SSS^1,TM)). \label{infcovd}\end{align}
  Observe that the identification $TC^\infty (\SSS^1,M) \cong C^\infty (\SSS^1,TM)$ allows us to identify $s \in \mathcal{V}(C^\infty (\SSS^1,M)) \subseteq C^\infty (C^\infty (\SSS^1,M),C^\infty (\SSS^1,TM))$. Computing with the help of Exercise \ref{EX:L2} 2.\, and Exercise \ref{ex:explaw} 2., 
  \begin{align*}
   \pi_{C^\infty (\SSS^1,M)} \circ \nabla_X s &= (\pi_{M})_*(\nabla^g_{X \times \mathbf{0}_{\SSS^1}} s^\wedge)^\vee = (\pi_M \circ (\nabla^g_{X \times \mathbf{0}_{\SSS^1}} s^\wedge)^\vee \\
   &= (\pi_M \circ K \circ T(s^\wedge) \circ (X \times \mathbf{0}_{\SSS^1}))^\vee = (\pi_M \circ K \circ (Ts \circ X)^\wedge)^\vee\\
   &= (\pi_M \circ K)_* (Ts \circ X) = \id_{C^\infty (\SSS^1,M)}.
  \end{align*}
  Hence \eqref{infcovd} induces a bilinear map $\nabla \colon \mathcal{V}(C^\infty (\SSS^1,M))^2 \rightarrow \mathcal{V}(C^\infty (\SSS^1,M))$. 
 \end{setup}
 
 The construction \Cref{defn:covd} certainly looks quite messy, however, we remark that it is indeed a very natural covariant derivative we obtain in this way. Namely, for any smooth curve $c \colon [a,b] \rightarrow C^\infty (\SSS^1,M)$ and smooth lift $\alpha \in \text{Lift} (c)$ we apply the exponential law to obtain smooth maps $c^\wedge \colon [a,b] \times \SSS^1 \rightarrow M$ and $\alpha^\wedge \colon [a,b] \times \SSS^1 \rightarrow TM$. Now the covariant derivative $\nabla$ from \Cref{defn:covd} is related to the covariant derivative $\nabla^g$ on $M$ by the following formula
 \begin{align}\label{L2:covd_curves}
  \nabla_{\dot{c}} \alpha(\cdot)(x) = \nabla^g_{\dot{c}^\wedge (\cdot, x)} \alpha (\cdot,x) \qquad \forall x\in \SSS^1.
 \end{align}
 We relegate the verification of \eqref{L2:covd_curves} to Exercise \ref{EX:L2} 2. However, the point is that despite the technical difficulties in defining the covariant derivative, it can be viewed just as a lifting of the covariant derivative of the target manifold $M$.
 
 \begin{prop}\label{prop:associatedconnector}
  For the map $\nabla \colon \mathcal{V}(C^\infty (\SSS^1,M))^2 \rightarrow \mathcal{V}(C^\infty (\SSS^1,M))$ and $X,Y \in \mathcal{V}(C^\infty (\SSS^1,M)$, the formula 
  \begin{align}\label{connectorL2}
\nabla_X Y = K_* \circ TY \circ X
  \end{align}
  holds. Since $K_*$ is a linear connector, $\nabla$ is a covariant derivative (with associated connector $K_* \colon C^\infty (\SSS^1,T^2M) \rightarrow C^\infty (\SSS^1,TM)$).
  In particular, \eqref{infcovd} is the covariant derivative associated to the spray $S_*$.
 \end{prop}

 \begin{proof}
  By \Cref{prop:tediousdetails} the connector $K_*$ is the associated connector to the spray $S_*$. Hence it suffices thus to prove \eqref{connectorL2}. We have essentially done this already as the exponential law and Exercise \ref{EX:L2} 2.\, yield
  \begin{align*}
  (\nabla_Xs)^\wedge &= \nabla^g_{X \times \mathbf{0}_{\SSS^1}} s^\wedge = K \circ Ts^\wedge \circ (X \times \mathbf{0}_{\SSS^1}) = K \circ (Ts \circ X)^\wedge \\
  &= (K_* \circ Ts \circ X)^\wedge. \qedhere
 \end{align*}
 \end{proof}

 \begin{prop}\label{prop:covdident}
  Let $(H,\langle \cdot , \cdot \rangle)$ be a Hilbert space with a strong Riemannian metric $g$ (not necessarily $g = \langle \cdot,\cdot \rangle)$. For the $L^2$-metric \eqref{l2metric} on $C^\infty (\SSS^1,H)$, the metric derivative is the covariant derivative $\nabla$ from \Cref{defn:covd}.
 \end{prop}
 
 \begin{proof}
  Recall that $TH \oplus TH = H \times H \times H$ and we can consider $g$ as a smooth map of three variables. Lets agree that the first component represents the base point in the bundle. Since $S$ is the spray of the metric $g$, the associated bilinear form $B$ satisfies for all $X,Y,Z \in H$ the relation
  \begin{align}\label{eq:bilin}
   -2 g(p,(B (p;X,Y),Z) = d_1g (p,Y,Z;X)+d_1g(p,Z,X;Y) -d_1 g(p,X,Y;Z)
  \end{align}
  see \cite[VIII, \S 4 Theorem 4.2]{Lang} or \cite[Theorem 1.8.11]{MR1330918}. Hence if we can compute that an identity as \eqref{eq:bilin} holds for the bilinear form associated to $S_*$ (which is given as pushforward with $B$) with respect to $g^{L^2}$ does this imply that $g^{L^2}$ admits an associated bilinear form (aka Christoffel symbols). Hence if we know the Christoffel symbols, we can compute the connector of giving the metric derivative. In this case, \Cref{defn:covd} shows that the connector is $K_*$, whence $S_*$ is the metric spray and $\nabla$ the metric derivative of the $L^2$-metric (cf.\ also \Cref{thm:metricsprayident}).
  
  Let us now establish the desired analogue of \eqref{eq:bilin} for the $L^2$-metric. The integral operator $\int_{\SSS^1} \colon C^\infty (\SSS^1, \R) \rightarrow \R$ is continuous linear. Up to the identification \Cref{Whitneyiso}, the derivative of $g^{L^2} = \int_{\SSS^1} \circ g_*$ in a direction is thus given by the pointwise derivative, i.e.
  $$d_1g^{L^2} (c,h,k;\xi) = \int_{\SSS^1} d_1g(c(\theta),h(\theta),k(\theta);\xi(\theta)) \mathrm{d} \theta.$$
  We apply this observation to the right hand side of \eqref{eq:bilin} and recall from \Cref{prop:tediousdetails} that the associated bilinear form $B_{S_*}$ of $S_*$ is given by the pushforward of the associated bilinear form $B$ of $S$. Together this yields
  \begin{align*}
   &d_1g^{L^2} (c,Y,Z;X)+d_1g^{L^2}(c,Z,X;Y) -d_1 g^{L^2}(c,X,Y;Z)\\ =& \int_{\SSS^1} d_1g (c,Y,Z;X)+d_1g(c,Z,X;Y) -d_1 g(c,X,Y;Z) d\theta \\ =& \int_{\SSS^1} -2g(c(\theta), B (c(\theta);X(\theta),y(\theta),Z(\theta))\mathrm{d} \theta = -2g^{L^2} (c,B_{S_*}(c;X,Y),Z).\qedhere
  \end{align*}
 \end{proof}

 \begin{rem}
  \begin{enumerate}
  \item Using the point evaluations of $TC^\infty (\SSS^1,M) \cong C^\infty (\SSS^1,TM)$ it is possible to directly describe the metric derivative of the $L^2$-metric by looking at it pointwise evaluated. While this allows one to avoid the construction in \Cref{defn:covd} one is then left with lots of localisation arguments to establish \Cref{prop:covdident}. We refer to \cite[Lemma 2.1]{YaRaT15} for more information.
   \item All of the above computations for spray, connector and covariant derivative did not exploit that the source manifold was $\SSS^1$. Thus via the same proof one can obtain a spray, connector and covariant derivative for the $L^2$-metric on $C^\infty (K,M)$ for any compact manifold $K$.
   \item Since \eqref{eq:bilin} can be formulated in any local chart, a more involved argument works for any strong Riemannian manifold $M$. However, using the Nash embedding theorem, \Cref{prop:covdident} generalises directly to all finite-dimensional Riemannian manifolds as \cite[Theorem 4.1]{BruL2} shows. 
  \end{enumerate} \label{rem:L2deriv}  
 \end{rem}
 Since we have identified the covariant derivative of the $L^2$-metric we can describe geodesics of the $L^2$-metric: A geodesic between $f,g \in C^\infty (\SSS^1,H)$ is a path of curves which is pointwise for every $\theta \in \SSS^1$ a geodesic in $H$ between $f(\theta)$ and $g(\theta)$. Let us explicitly compute this in the case where the inner product of the Hilbert space gives us the Riemannian metric:
 
 \begin{ex}\label{ex:trivialex:cont}
  Consider $(H,\langle \cdot ,\cdot \rangle)$ as a strong Riemannian manifold. Recall that the metric spray and the covariant derivative were computed for this metric in \Cref{setup:trivial}: With the identification $T^k H = H^{k+1}, k\in \N$ we obtain $S(x,v) = (x,v,v,0)$ and $\nabla_XY = X.Y$ for suitable paths and their lifts $\nabla_{\dot{c}} h =\dot{h}$. 
  Endow now $C^\infty (\SSS^1,H)$ with the $L^2$-metric and pick $p \in C^\infty (\SSS^1,H)$. We compute now the geodesics $c \colon J \rightarrow C^\infty (\SSS^1,H)$ starting at $p$ with derivative $\dot{c} =X$. Identify now the first and second derivatives as $\dot{c}(t) = (c(t),c'(t)) \in C^\infty (\SSS^1,H\times H)$ and $\ddot{c}(t) = (c(t),c'(t),c'(t), c''(t))$. The exponential law allows us to compute the geodesic equation with respect to the spray as
  $$(c(t),c'(t),c',c''(t))^\wedge = \left(\ddot{c}(t)\right)^\wedge = (S_*(c(t),c'(t)))^\wedge = S\left(c^\wedge (t,\cdot),\frac{\partial}{\partial t}c^\wedge (t,\cdot)\right).$$
  Evaluating in $\theta \in \SSS^1$ one immediately sees that $c''(t) = 0, \forall t$. Hence, from the usual rules of calculus we deduce that $c(t)(\theta) = tX(\theta) +p(\theta)$.
  A geodesic between $f,g \colon \SSS^1 \rightarrow H$ does thus always exist and is described by a map $\gamma \colon [0,1] \times \SSS^1 \rightarrow H$ such that for every $\theta \in \SSS^1$ we have $\gamma(t,\theta) = (1-t)f(\theta) + tg(\theta)$. In other words geodesics in this example are given by pointwise linear interpolation between the two functions.\index{geodesic!of the $L^2$-metric}
 \end{ex}
 
 We shall further investigate the geodesic equation of a generalised version of the $L^2$-metric on the open submanifold $\Diff (M) \opn C^\infty(M,M)$ in \Cref{sect:EAtheory}.
 However, we can also use our knowledge of the $L^2$-metric to derive other interesting examples of metric derivatives for certain weak Riemannian metrics. Note however, that the following example requires a more in depth knowledge of Riemannian geometry (e.g.~the Hodge-Laplacian, \Cref{defn:classicdiffop} and curvature \Cref{defn:curvature}).

\begin{ex}[{\cite{YaRaT15}}]
 Let $(M,g)$ be a finite-dimensional Riemannian manifold with metric derivative $\nabla^g$. We recall that every (finite-dimensional) Riemannian manifold has a (Hodge-)Laplacian $\Delta = \mathrm{d} \mathrm{d}^\ast + \mathrm{d}^\ast \mathrm{d}$ associated to the metric (cf.~\cite[p.423]{Lang}) and a curvature tensor $R$ (see \Cref{defn:curvature}).
 
 Now consider the loop space $C^\infty (\SSS^1,M)$. By \Cref{rem:L2deriv} 3.~we know that the $L^2$-metric admits a metric derivative $\nabla$. We will now use the notation of lifts and covariant derivative along curves on $[0,2\pi]$ in the context of curves on $\SSS^1$ (implicitly identifying elements in $LM$ with curves on $[0,2\pi]$ by composing with $(\sin,\cos) \colon [0,2\pi] \rightarrow \SSS^1$). Hence for a smooth curve $\gamma \in LM$ we have $T_\gamma C^\infty (\SSS^1,M) = \text{Lift} (\gamma) \cap C^\infty (\SSS^1,TM)$. We can then endow every tangent space of $C^\infty (\SSS^1,M)$ with the $H^1$-inner product\index{Riemannian metric!$H^1$-metric}
 \begin{align}\label{H1_IP}
  g^{H^1}_\gamma(X,Y) \coloneq \frac{1}{2\pi} \int_0^{2\pi} g_{\gamma(\theta)}\left((1+\Delta)X(\theta) ,Y(\theta)\right) \mathrm{d}\theta
 \end{align}
 Rewriting the codifferential in the above formula and exploiting duality with respect to the metric $g$, one can show that the inner product \eqref{H1_IP} describes the sum of the $L^2$-inner products of the lift of $\gamma$ and its first derivative. We can thus leverage Exercise \ref{ex:metricex} 4 a) where we have seen that the $L^2$-metric is a (weak) Riemannian metric. Differentiation is continuous linear (on the tangent space of $C^\infty (\SSS^1,M)$), whence the $H^1$-metric is a weak Riemannian metric. Now a (non-trivial!) computation shows that the metric derivative of the $H^1$-metric is intimately connected to the curvature tensor, the metric derivative of $g$ and the metric derivative of the $L^2$-metric. Namely, \cite[Theorem 2.2]{YaRaT15} provides for $X,Y \in T_\gamma C^\infty (\SSS^1,M)$ the following formula for the metric derivative $\nabla^{H^1}_X Y(\gamma)$:
 \begin{align*}
 \nabla_X Y+ \frac{1}{2} (1+\Delta)^{-1}&\left(-\nabla^g_{\dot{c}}(R(X, \dot{c})Y) - R(X, \dot{c})\nabla^g_{\dot{c}} Y- \nabla^g_{\dot{c}}(R(Y,\dot{c}))X  \right.\\
 &\ \left. -R(Y,\dot{c})\nabla^g_{\dot{c}}X + R(X,\nabla^{g}_{\dot{c}} Y) \dot{c} - R(\nabla^g_{\dot{c}}X,Y)\dot{c}\right).
 \end{align*}
 where $\nabla^g_{\dot{c}}(R(X, \dot{c})Y)$ denotes the lift of $\gamma$ whose value at $\theta \in \SSS^1$ is given by the formula
 $-\nabla^g_{\dot{c}(\theta)}(R(X(\theta), \dot{c}(\theta))Y(\theta))$. While the above formula looks daunting and we do not attempt to unravel its meaning here, we would like to mention that it can be used to connect the Riemannian geometry of the $H^1$-metric to pseudodifferential operators acting on a trivial bundle over the circle. This link then yields information on Chern-Simmons classes on the tangent bundle of the loop space. We refer the interested reader to \cite{YaRaT15} for more information.
\end{ex}

 \begin{Exercise} \label{EX:L2}\vspace{-\baselineskip}
  \Question Let $S \colon TM \rightarrow T^2M$ be a spray on $M$ with $K \colon T^2 M \rightarrow TM$ its associated connector.
  \subQuestion Prove that $S_*$ is a spray on $C^\infty (\SSS^1,M)$. \\
  {\tiny \textbf{Hint:} Use that $TC^\infty (\SSS^1,M) \cong C^\infty (\SSS^1,TM)$ identifies $T(p_*)$ with $(Tp)_*$ for smooth maps. For the quadratic condition review the effect of the diffeomorphism on the fibres.}
  \subQuestion Check that $K_*$ is a connector.
  \subQuestion Show that the second derivative of the vertical part of $S_*$ is the pushforward of the second derivative of the vertical part of $S$ (cf.\, \Cref{prop:tediousdetails}).
  \Question Prove for $s \in C^\infty (N,C^\infty (\SSS^1,M))$ and $X \in \mathcal{V}(N)$ the identity
  $$Ts \circ X = (Ts^\wedge \circ (X \times \mathbf{0}_{\SSS^1}))^\vee.$$
  Furthermore, establish the formula \eqref{L2:covd_curves}: $\nabla_{\dot{c}} \alpha(\cdot)(x) = \nabla^g_{\dot{c}^\wedge (\cdot, x)} \alpha^\wedge (\cdot,x),\ \forall x\in \SSS^1$.
  \Question Continue \Cref{ex:trivialex:cont} and compute for $C^\infty (\SSS^1,H)$ with the $L^2$-metric an explicit form of the geodesic equation $\nabla_{\dot{c}}\dot{c}=0$. Deduce again that a geodesic $c \colon J \rightarrow C^\infty (\SSS^1,H)$ is given for each $\theta \in \SSS^1$ by the affine linear map $c(t)(\theta) = tc'(0)(\theta) + c(0)(\theta)$.
  \Question Let $(M,g)$ be a Riemannian manifold with a metric spray and metric derivative $\nabla$. Work in a local chart to establish the identity $X.g(Y,Z) = g(\nabla_X Y,Z) + g(Y,\nabla_X Z)$, \eqref{covalong:Riemannian}, for the covariant derivative along a smooth map $F \colon N \rightarrow M$.
  \end{Exercise}

  \setboolean{firstanswerofthechapter}{true}
 \begin{Answer}[number={\ref{EX:L2} 1. c)}] 
  \emph{We check the details of \Cref{prop:tediousdetails} and show that the second derivative of the vertical part of $S_*$ is given by the pushforward of the second derivative of the vertical part of $S$.}\\[.15em]
  
  In \Cref{prop:tediousdetails} we have already seen that we can compute instead the partial derivative (with respect to the first variable) of 
  $$(S_*)^\wedge \colon C^\infty (\SSS^1,TM) \times \SSS_1 \rightarrow T^2M,\quad (h,\theta) \mapsto S(h(\theta)).$$
  Pick $\theta \in \SSS^1$ and $h \in C^\infty (\SSS^1,TM)$ together with a chart $(U,\varphi)$ of $M$ such that $h(\theta) \in TU \cong U \times E$ (for $E$ a locally convex space). By continuity, there is a compact set $L$ such that $h \in \lfloor L,TU\rfloor$ and $\theta \in L^\circ$. For a curve $c \colon ]-\varepsilon,\varepsilon[ \rightarrow \lfloor L,TU\rfloor$ 
  apply again the exponential law \Cref{thm:explaw} to see that $c(t)|_{L^\circ} = T\varphi^{-1} c_E(t)$ for some smooth $c_E  \colon ]-\varepsilon,\varepsilon[ \rightarrow C^\infty (L^\circ, U \times E)$.
  Plugging this into $(S_*)^\wedge$ we derivate to obtain 
  \begin{align*}T^2 \varphi \circ (S_*)^\wedge (T\varphi^{-1})_* \circ c_E(t),\theta) &= T^2 \varphi S(T\varphi^{-1}(c_E(t)(\theta))\\ &= (\pi_{M} \circ h(\theta),c_E(t)(\theta),c_E(t)(\theta),S_{U,2} (h(\theta),c_E(t)(\theta))).
  \end{align*}
  Here $S_{U,2}$ is the non-trivial vector part of the spray $S$. We conclude that after projecting onto the fourth component and after fixing the parameter $\theta$, the second derivative of the vertical part of $S_*$ can be identified with the pushforward of the second derivative of $S_{U,2}$. This proves that $B_{S_*}$ is the pushforward of $B$.
 \end{Answer}
 \setboolean{firstanswerofthechapter}{false}
 \begin{Answer}[number={\ref{EX:L2} 2)}] 
  \emph{We establish the formula \eqref{L2:covd_curves}: $\nabla_{\dot{c}} \alpha(\cdot)(x) = \nabla^g_{\dot{c}^\wedge (\cdot, x)} \alpha^\wedge (\cdot,x),\ \forall x\in \SSS^1$.}\\[.15em]
  
  The trick is to avoid at all costs working in charts of the manifold of mappings $C^\infty (\SSS^1,M)$. However, the object $\nabla_{\dot{c}} \alpha$ is defined via the local formula \eqref{locform:deriv}. Taking a look at the local formula for the connector \eqref{connectionformula} we see that $\nabla_{\dot{c}} \alpha = K_* (\dot{\alpha})$, where $K_*$ is the connector associated to the covariant derivative $\nabla$. As already shown in the notation (and proved in \Cref{prop:associatedconnector}), the connector of $\nabla$ is the pushforward of the connector $K$ associated to the Riemannian metric $g$ on $M$ (and thus also to the covariant derivative $\nabla^g$). 
  Setting in now these relations we obtain the desired identity
  \begin{align*}
   \nabla_{\dot{c}} \alpha (t) (x) = K_*(\dot{\alpha} (t))(x) = (K\circ \alpha^\wedge) (t,x) = \nabla_{c^\wedge (t,x)} \alpha^\wedge (t,x).
  \end{align*}

 \end{Answer}
 \section{Shape analysis via the square root velocity transform}\label{sect:SRVT}
 
 We return to the announced application of the $L^2$-metric to shape analysis. As we have seen in the introduction we seek to construct a shape space together with a Riemannian structure which will allow us to compare its elements using geodesics and geodesic distance. We begin by defining the necessary spaces and metrics.
\begin{setup}
 Define the \emph{pre-Shape space} of closed curves\index{pre-shape space}
 $$\mathcal{P} \coloneq \Imm (\SSS^1,\R^2) = \{c \in C^\infty (\SSS^1,\R^2) \mid \dot{c}(t) \neq 0, \forall t \in [0,1]\}.$$
 In Exercise \ref{Ex:cinftytop} 3. and \Cref{ex:invL2} we have seen that $\mathcal{P}$ is an open subset of $C^\infty (\SSS^1,\R^2)$ which becomes a weak Riemannian manifold with respect to the Riemannian metric
 $$g_c(f,g) = \int_{\SSS^1} \langle f(\theta),g(\theta)\rangle \lVert \dot{c}(\theta)\rVert \mathrm{d} \theta.$$
 We are actually interested in the images of elements in $\mathcal{P}$ and want to identify all curves which yield the same image up to a reparametrisation. To model the reparametrisation, consider the group
 $$\Diff (\SSS^1) \coloneq \{\varphi \in C^\infty(\SSS^1,\SSS^1) \mid \varphi \text{ is bijective with smooth inverse}\}$$
 of diffeomorphisms of $\SSS^1$. The group acts on $\mathcal{P}$ via reparametrisation, i.e.\ 
 \begin{align*}
  \gamma \colon \textstyle \Diff (\SSS^1) \times \mathcal{P}\displaystyle \rightarrow \mathcal{P},\quad (\varphi,c) \mapsto c\circ \varphi 
 \end{align*}
 is a Lie group action, cf.\ \Cref{ex:Diffgp}. We can thus define the \emph{shape space}\index{shape space} $\mathcal{S} = \mathcal{P}/\Diff(\SSS^1)$ as the quotient of $\mathcal{P}$ with respect to the Lie group action. One can show that $\mathcal{S}$ is almost a manifold\footnote{There exist singularities, turning $\mathcal{S}$ into an orbifold. See \Cref{ex:ill:behaved} for an example. The existence of singularities is usually ignored in shape analysis, as an open dense subset of $\mathcal{S}$ is a manifold, such that the projection restricts on this set to a submersion.} and since the Riemannian metric is invariant under the reparametrisation action, the Riemannian metric induces a Riemannian metric on $\mathcal{S}$. Unfortunately, Michor and Mumford \cite[3.10]{MaM06} have shown that $g$ has vanishing geodesic distance (cf.\, also \Cref{MMvanish}), whence any attempt to compare shapes by computing their geodesic distance has to fail. 
\end{setup}
The defect of the weak Riemannian metric can be solved by incorporating derivatives in the definition of the Riemannian metric. This leads to the notion of a family of metric called $H^1$-metrics (the name indicates that the associated strong Riemannian manifold consists of Sobolev $H^1$-functions). 

\begin{setup}[An elastic inner product, \cite{JaMaS07}]\label{setup:elastic}
 Let $c \in \Imm(\SSS^1,\R^2)$ with $\dot{c}=(c,c^\prime)$. Then we define $u_c(\theta) \coloneq c^\prime(\theta)/\lVert \dot{c}(\theta)\rVert$ and the arc length derivative $D_{c,\theta} (h) = h^\prime/\lVert \dot{c}\rVert$. Then we define an inner product on $T_c \Imm(\SSS^1,\R^2) = C^\infty (\SSS^1,\R^2)$ 
 \begin{align}\label{elastic}
 \begin{split}G_c (h,k) \coloneq \int_{\SSS^1} \frac{1}{4}& \langle D_{c,\theta} h,u_c\rangle \langle D_{c,\theta} k,u_c\rangle \\ & + \langle D_{c,\theta}h-u_c \langle D_{c,\theta} h,u_c\rangle,  D_{c,\theta}k - u_c \langle D_{c,\theta} k,u_c\rangle\rangle \lVert \dot{c}\rVert \mathrm{d}\theta.
  \end{split}
 \end{align}
 This inner product is called \emph{elastic inner product} as the first term measures stretching in the direction of $c$, while the second term measures bending of the curve $c$. Note that due to its construction, the elastic inner products are invariant under the reparametrisation action of $\Diff(\SSS^1)$ on $\mathcal{P}$.
\end{setup}

It is not hard to see that these inner products yield a weak Riemannian metric on the pre-Shape space $\mathcal{P}$. We will derive this only for a smaller submanifold using the so called \emph{square root velocity transform (SRVT)}:

\begin{defn}\label{defn:SRVT}
 Define the mappings
 \begin{align*}
  \mathcal{R} \colon \mathcal{P} =\Imm(\SSS^1,\R^2) &\rightarrow \{q \in C^\infty (\SSS^1,\R^2) \mid q(\theta)\neq 0 \forall \theta \in \SSS^1\} = C^\infty (\SSS^1, \R^2 \setminus \{0\})\\
  c &\mapsto \mathcal{R}(c)(t) \coloneq c^\prime(t)/\sqrt{\lVert \dot{c}(t)\rVert},\\
  \mathcal{R}^{-1} \colon C^\infty (\SSS^1,\R^2\setminus \{0\}) &\rightarrow \mathcal{P},\\
  q &\mapsto \left((\cos (t),\sin(t)) \mapsto \int_{0}^t q(\cos (s),\sin(s)) \cdot \lVert q(\cos(s),\sin(s)\rVert \mathrm{d} s \right)
  \end{align*}
 We call $\mathcal{R}$ the \emph{square root velocity transform (SRVT)}. \index{square root velocity transform} \index{SRVT!see square root velocity transform}
\end{defn}

Using that $\mathcal{P} \opn C^\infty (\SSS^1,\R^2)$ and that pushforwards of smooth mappings are smooth maps between canonical manifolds of mappings (cf.~\Cref{sect:smoothmappingspaces}), it is not hard to see that $\mathcal{R}$ and $\mathcal{R}^{-1}$ are smooth. 
 Moreover, $\mathcal{R} \circ \mathcal{R}^{-1}(q)=q$, but since $\mathcal{R}$ involves differentiation, we loose information on the starting point of the curve and have $\mathcal{R}^{-1}\circ\mathcal{R}(c)=c$ if and only if $c(\cos(0),\sin(0))=0$, i.e.\, the curve starts at the origin. As we are interested in shapes, it will be irrelevant as to where in $\R^2$ the shape is located. With other words, we can restrict to the submanifold 
$$\mathcal{P}_* \coloneq \{c \in \mathcal{P} \mid c(\cos(0),\sin(0))=0\}$$
of all immersions starting at the origin. We will see in Exercise \ref{ex:SRVT} 1.~that the SRVT induces a diffeomorphism between $\mathcal{P}_*$ and the manifold $C^\infty (\SSS^1,\R^2\setminus \{0\})$. Recall now the notion of pullback of a Riemannian metric

\begin{defn}
 Let $(M,g)$ be a (weak) Riemannian manifold and $\varphi \colon N \rightarrow M$ be an immersion. Then $N$ can be made a (weak) Riemannian manifold with respect to the \emph{pullback metric}\index{Riemannian metric!pullback metric} defined via 
 $$(\varphi^*g)_m (v,w) \coloneq g_{\varphi(m)} (T_m \varphi(v),T_m\varphi(w)).$$
\end{defn}

\begin{ex}
 If $U \opn M$ and $(M,g)$ is a weak Riemannian manifold, then the inclusion $\iota \colon U \rightarrow M$ is an immersion and the restriction of $g$ to $U$ coincides with the pullback metric obtained from $\iota$. In particular an open subset of a weak Riemannian manifold, such as $C^\infty (\SSS^1,\R^2\setminus \{0\}) \opn C^\infty (\SSS^1,\R^2)$ (with the weak $L^2$-metric), becomes a weak Riemannian manifold by restriction.
\end{ex}

\begin{prop}\label{prop:elastic_pullback}
 The pullback metric of the non-invariant $L^2$-metric on $C^\infty (\SSS^1,\R^2)$ from \Cref{ex:L2Hilbert} with respect to the SRVT $\mathcal{R}$ is the \emph{elastic metric}\index{Riemannian metric!elastic metric} described by \eqref{elastic} on each tangent space of $\mathcal{P}_*$.\footnote{The statement remains valid if we replace $\R^2$ by an arbitrary Hilbert space of dimension $\geq 2$.}
\end{prop}

\begin{proof}
 By construction, the square root velocity transform $\mathcal{R}$ is the composition of the derivative operator $D \colon \text{Imm} (\SSS^1,\R^2) \rightarrow C^{\infty} (\SSS^2,\R^2 \setminus \{0\}), q \mapsto (\theta \mapsto q^\prime(\theta) = T_\theta c (1))$ and the scaling map $\text{sc}\colon C^\infty (\SSS^1,\R\setminus \{0\}) \rightarrow C^\infty (\SSS^1,\R\setminus \{0\}), f \mapsto f / \sqrt{\lVert f\rVert}$. Thus by the chain rule we have $T_c \mathcal{R} = T\text{sc} \circ T_c D$ and to arrive at the desired formula, we have to compute the derivatives of these two mappings. To this end, we exploit that $TC^\infty (\SSS^1,\R^2) = C^\infty (\SSS^1,T\R^2) = C^\infty (\SSS^1,\R^2 \times \R^2)$ and since $\text{Imm} (\SSS^1,\R^2) \opn C^\infty (\SSS^1,\R^2)$ we can similarly identify the tangent bundle of the immersions. Now arguments as in Exercise \ref{EX11} 2 c) show that the differential operator $D$ is continuous linear, whence for an element $(c,V)$ in $T_c C^\infty (\SSS^1,\R^2) = \{(c,V) \in C^\infty (\SSS^1,\R^2)^2\}$ we obtain $T_c D (V) = V^\prime$. In Exercise \ref{ex:SRVT} 2.~we shall show that the derivative of the scaling map is 
 \begin{align}\label{deriv:scalingmap}
  T_q\text{sc} (Z) = \frac{Z}{\sqrt{\lVert q\rVert}} - \frac{1}{2 \sqrt{\lVert q\rVert}^5} \langle Z,q \rangle q
 \end{align}
 To obtain the derivative of the SRVT at $(c,V)$, we simply have to replace $q$ with $c^\prime$ and $Z$ with $V^\prime$. This yields the following formula:
 \begin{align}
  &(\mathcal{R}^* \langle \cdot , \cdot\rangle_{L^2})_c (V,W) = \langle T_c\mathcal{R}(V),T_c\mathcal{R}(W) \rangle_{L^2} =  \int_{\SSS^1} \langle T_{c^\prime}\text{sc}(V^\prime)(\theta),T_{c^\prime}\text{sc}(W)(\theta)\rangle \mathrm{d} \theta \notag \\
  =&  \int_{\SSS^1} \left\langle \frac{V^\prime}{\sqrt{\lVert c^\prime\rVert}} - \frac{1}{2 \sqrt{\lVert c^\prime\rVert}^5} \langle V^\prime,c^\prime \rangle c^\prime, \frac{W^\prime}{\sqrt{\lVert c^\prime\rVert}} - \frac{1}{2 \sqrt{\lVert c^\prime\rVert}^5} \langle W^\prime,c^\prime \rangle c^\prime\right\rangle (\theta)\mathrm{d} \theta \label{eq:pullback_elastic}
 \end{align}
 Since the inner product is bilinear, we can factor out terms of the form $\lVert c^\prime\rVert$ and replace $V^\prime, W^\prime$ and $c^\prime$ with their rescaled versions (see \Cref{setup:elastic}). Now an easy but tedious computation shows that the pullback metric \eqref{eq:pullback_elastic} coincides with the elastic metric \eqref{elastic}.
 \end{proof}

Hence the elastic metric can be understood by studying the $L^2$-metric on the manifold $C^\infty(\SSS^1,\R^2\setminus \{0\})$.
However, as this is just an open subset of the (weak) Riemannian manifold $(C^\infty (\SSS^1,\R^2),\langle \cdot,\cdot\rangle^{L^2})$ we already know the spray, connector, covariant derivative and geodesics of this metric from our discussion of the $L^2$-metric. It is important to observe that the elastic metric pulls back to the non-invariant $L^2$-metric. For the non-invariant $L^2$-metric, the geodesic distance does not vanish (compare this with the invariant $L^2$-metric \Cref{ex:invL2}) and the pullback metric (i.e.~the elastic metric) is invariant under reparametrisation.

A natural extension for the vector space valued shape spaces discussed in \Cref{prop:elastic_pullback} are Lie group valued shape spaces. Here a shape is (up to quotienting out the reparametrisation action) an element of the loop group $C^\infty (\SSS^1,G)$. If $G$ is a Hilbert Lie group, the construction of the square root velocity transform can be adapted to this more general setting by using the logarithmic derivative of Lie group valued mappings. For details, we refer to Exercise \ref{ex:SRVT} 4. Similar techniques have been used in \cite{CaEaS18} for shape analysis on homogeneous spaces. 

\begin{ex}[Sample application: Motion capturing (see e.g.~\cite{CaEaS16})]
 Assume that we have motion capturing data of e.g.~a human walking, given by a number of time dependent datapoints in $\R^3$. The associated virtual character is then modelled as a skeleton to which these datapoints represent positions of certain parts. Shifting the focus from the position to their relative position, we can interpret the positions of every part of the skeleton as an angle between neighbouring parts. Now angles in $\R^3$ can be identified with rotations, i.e.~elements in the Lie group $\text{SO}(3)$ of rotations of $\R^3$. Thus if we do not impose constraints on the allowed angles (which can lead to unnatural movements, but is generally fine when working with real motion capturing data), we can think of motion capturing data as a smooth curve with values in a product of copies of $\text{SO}(3)$ (the number depends on the number of data points which move relative to each other). 
 In \cite{CaEaS16} numerical algorithms for automatic interpolation and transformation of motion capturing data have been constructed which exploit the above point of view.
 \end{ex}

\begin{rem}
 Note that the geodesic distance of the $L^2$-metric on $C^\infty (\SSS^1,H)$ does not vanish and is indeed positive for any two shapes which are not equal (where $H$ is a Hilbert space). Indeed the geodesic distance just coincides with the $L^2$-distance of the curves. The situation gets more complicated in the space $C^\infty (\SSS^1,H\setminus \{0\})$ and in particular for $H=\R^2$(which is not simply connected, so an element $\gamma$ for which $0$ is contained in a bounded connected component of $\R^2\setminus \gamma (\SSS^1)$ can not be connected by a continuous curve in $C^\infty (\SSS^1,\R^2\setminus \{0\})$ to an element for which $0$ is not contained in such a component).
 
 While this does not happen for open shapes (i.e.\, elements of $\Imm([0,1],\R^2)$) the following problem is even more significant when it comes to applications in shape analysis: As geodesics are not allowed to pass through $0$, the $L^2$-distance of two functions is \textbf{not} the geodesic distance even if there exist continuous paths between them. If the linear interpolation between two points $c(\theta)$ and $d(\theta)$ passes through $0$, then the geodesic from the ambient space $C^\infty (\SSS^1,\R^2)$ does not exist in the smaller space. Instead the geodesic distance then needs to be computed using curves which "move one shape around the hole". This leads to a geodesic distance which is strictly larger then the $L^2$-distance. In numerical applications this is often just plainly ignored.
\end{rem}

That the geodesic distance on $C^\infty(\SSS^1,\R^2\setminus\{0\})$ diverges from the easily computed $L^2$-distance is a consequence of the space being incomplete. Indeed since we have "drilled a hole" by excluding a point, geodesics passing through that point can only exist up to the time they enter. As the $L^2$ distance moves points on the image of functions along straight lines interpolating between them, the $L^2$-distance fails to give the geodesic distance for points lying on opposite sides of the excluded point.

\begin{Exercise}\label{ex:SRVT} \vspace{-\baselineskip}
 \Question Establish some properties of the square root velocity transform $\mathcal{R}$, \Cref{defn:SRVT}.\\
 {\tiny \textbf{Hint:} Use the substitution rule for integrals on submanifolds; in the case at hand, this works just as in usual interval case.}\\ 
 Show that
 \subQuestion $\mathcal{R}$ and $\mathcal{R}^{-1}$ are smooth maps with $\mathcal{R}\circ \mathcal{R}^{-1}=\id_{C^\infty (\SSS^1,\R^2\setminus \{0\}}$. 
 \subQuestion $\mathcal{P}_*=\{c \in \mathcal{P} \mid c(\cos(0),\sin(0))=0\}$ is a closed submanifold of $\mathcal{P}$. \subQuestion $\mathcal{R}$ and $\mathcal{R}^{-1}$ induce a diffeomorphism between the manifolds $\mathcal{P}_*$ and $C^\infty (\SSS^1,\R^2\setminus \{0\})$
 \subQuestion If $\varphi \in \Diff(\SSS^1)$ with $\varphi^\prime (\theta) >0, \forall \theta \in \SSS^1$, then $\mathcal{R}(c\circ \varphi) = \varphi^\prime \cdot \mathcal{R}(c)\circ \varphi $.
 \Question Let $(E,\lVert \cdot \rVert)$ be a Hilbert space with $\lVert v\rVert^2 = \langle v,v\rangle $ and $K$ a compact manifold. Prove that the scaling map $\text{sc} \colon C^\infty (K,E\setminus \{0\}) \rightarrow C^\infty (K,E\setminus \{0\}), q \mapsto q/\sqrt{\lVert q\rVert}$ is smooth with tangent map given by the formula \eqref{deriv:scalingmap}:
 $$T_q \text{sc} (W) =  \frac{W}{\sqrt{\lVert q\rVert}} - \frac{1}{2 \sqrt{\lVert q\rVert}^5} \langle W,q \rangle q.$$
 {\footnotesize \textbf{Hint}: Use the exponential law together with the canonical identification of the tangent bundles.} 
 \Question Show that the elastic metric \eqref{elastic} is invariant under reparametrisation with elements $\varphi$ in $\Diff (\SSS^1)$ which satisfy $T_\theta \varphi(1) >0, \forall \theta \in \SSS^1$. 
 \Question Let $G$ be a Hilbert Lie group, i.e.\ $\Lf(G)$ is a Hilbert space with inner product $\langle \cdot,\cdot\rangle$. Then we define a square root velocity transform on the subset of immersions of the loop group
 $$\mathcal{R} \colon \text{Imm} (\SSS^1, G) \rightarrow C^\infty (\SSS^1,\Lf(G) \setminus \{0\}),\quad c \mapsto \delta^r (c) / \sqrt{\lVert \delta^r (c)\rVert},$$
 where $\delta^r$ is the right logarithmic derivative.
 \subQuestion Show that $\mathcal{R}$ is a smooth diffeomorphism (what is its inverse?). 
 \subQuestion Compute a formula for the pullback of the $L^2$-metric on $C^\infty (\SSS^1,\Lf(G)\setminus \{0\})$.\\
 This metric is known as the elastic metric on the Lie group valued immersions and it can be used in computer animation and motion capturing applications. See \cite{CaEaS16} for more information.
 \end{Exercise}

 \begin{Answer}[number={\ref{ex:SRVT} 3.}]
  \emph{The elastic metric \eqref{elastic} is invariant under reparametrisations with elements $\varphi$ in $\Diff (\SSS^1)$ which satisfy $T_\theta \varphi(1) >0, \forall \theta \in \SSS^1$.}\\[.15em]
  
  We have seen in \Cref{prop:elastic_pullback} that the elastic metric is the pullback of the $L^2$-metric via the SRVT. In Exercise \ref{ex:SRVT} 1. c) we have seen that for a diffeomorphism $\varphi$ in $\Diff (\SSS^1)$ which satisfy $T_\theta \varphi(1) >0, \forall \theta \in \SSS^1$ one has $\mathcal{R}(c\circ \varphi) = \dot{\varphi} \cdot \mathcal{R}(c)\circ \varphi $.
  Plugging this in the $L^2$-inner product, we see that invariance follows from the usual transformation rule for integrals.
 \end{Answer}

\chapter{Connecting finite, infinite-dimensional and higher geometry}\copyrightnotice

In this section we will highlight the an interesting connection between finite and infinite-dimensional differential geometry. To this end we shall consider in \Cref{subsect:groupoids} elements from ``higher geometry'', so called Lie groupoids. The moniker ``higher geometry'' stems from the fact that in the language of category theory, these objects form higher categories. We shall not explore higher categories or its connection to differential geometry in this book (but the reader might consult \cite{MaZ15} or the general introduction \cite{Baez97}). In previous sections we have discussed how finite-dimensional manifolds and geometric structures give rise to infinite-dimensional structures such as Lie groups (e.g.~the diffeomorphims and groups of gauge transformations) and Riemannian metrics (such as the $L^2$-metric from shape analysis).
While we have seen that every manifold determines an (in general infinite-dimensional) group of diffeomorphisms, we turn this observation now on its head and ask: \\
\emph{Can we recognise the underlying finite-dimensional geometric structure from the infinite-dimensional object?}

\section{Diffeomorphism groups determine their manifolds}\label{sect:determination}
Let us examine this question for the diffeomorphism group.\index{diffeomorphism group} We have already seen that to every compact manifold we can associate the infinite-dimensional diffeomorphism group. Conversely, Takens \cite{takens79} has shown that the diffeomorphism group identifies (up to diffeomorphism) the underlying manifold. Namely, we have the following theorem (which we cite here from the much more general statement of \cite{MR693972}):

\begin{thm}[Takens '79/Filipkiewicz '82/Banyaga '88/Rubin '89]\label{thm:recognition}
 If $M,N$ are smooth compact, connected manifolds such that $\phi \colon \Diff (M) \rightarrow \Diff(N)$ is a group isomorphism then there exists a diffeomorphism $\phi \colon M \rightarrow N$ such that $\Phi (\gamma) = \phi \circ \gamma \circ \phi^{-1}$.
\end{thm}

A full proof of \Cref{thm:recognition} would leads us too far astray, but it is possible to highlight certain aspects of the proof which are of special interest for us with regard to the question if finite-dimensional objects can be recognised from their associated infinite-dimensional objects. Before we begin, let us recall two concepts for diffeomorphism groups:

\begin{defn}
Let $M$ be a smooth and compact manifold. 
\begin{itemize}
 \item For $x_0 \in M$, the \emph{stabiliser}\index{diffeomorphism group!stabiliser of a point} is defined as
$$S_{x_0}\Diff(M) \coloneq \{ h \in \Diff (M) \mid h(x_0)=x_0\}.$$
 \item The group $\Diff(M)$ \emph{acts $n$-transitive}\index{diffeomorphism group!acting $n$-transitive} on the manifold $M$ for $n\in \N$ if for any two sets $\{x_1,\ldots , x_n\}, \{y_1,\ldots,y_n\}$ of non-repeating points in $M$ there is $h \in \Diff (M)$ with $h(x_i)=y_i, i\in \{1,\ldots, n\}$.
\end{itemize}
\end{defn}
We shall show how a group isomorphism of diffeomorphism groups mapping stabilisers to each other induces a diffeomorphism on the base manifold. The first step towards this is \cite[Lemma 1]{Ban88} (also compare \cite{Ryb95}):

\begin{lem}\label{lem:ind_homeo}
 Let $M, N$ be two connected smooth manifolds and $\phi \colon \Diff (M) \rightarrow \Diff (N)$ a group isomorphism such that the following holds 
 \begin{enumerate}
  \item For $K \in \{M,N\}$, $\Diff(K)$ acts $n$-transitively for $n\in \{1,2\}$, and  
  \item for each $x_0 \in M$ there exists $y_0 \in N$ such that $\phi (S_{x_0}\Diff(M)) = S_{y_0}\Diff(N)$. \label{cond:B}
 \end{enumerate}
 Then there is a unique homeomorphism $\omega \colon M \rightarrow N$ such that $\phi (f) = \omega f \omega^{-1}$.
 \end{lem}

 \begin{proof} \textbf{Step 1:} \emph{Construction of the homeomorphism $\omega$.}
  Fix a pair of points $x_0$ and $y_0$ as in condition \ref{cond:B}. Since $\Diff(M)$ and $\Diff (N)$ are $1$-transitive, we see that 
  \begin{align*}
   \ev_{x_0} \colon \Diff(M) \rightarrow M,\quad h \mapsto h(x_0) \text{ and } \ev_{y_0} \colon \Diff (N) \rightarrow N,\quad g \mapsto g(y_0)
  \end{align*}
 are surjective with $\ev_{x_0}^{-1}(x_0)=S_{x_0} \Diff (M)$ (and similarly for $y_0$). We can thus choose for every $x \in M$ an (in general) non-unique diffeomorphism $h_x \in \Diff(M)$ such that $h_x (x_0)=x$. Now if $\tilde{h}$ is another diffeomorphism such that $\tilde{h}(x_0)=x$ we see that $\tilde{h}^{-1}h_x(x_0) =x_0$, whence $\tilde{h}^{-1}h_x \in S_{x_0} \Diff (M)$. As by assumption, $\phi(S_{x_0} \Diff (M)) = S_{y_0}\Diff (N)$, this implies $\phi (\tilde{h}^{-1}h_x) = \phi(\tilde{h}^{-1})\phi(h_x) \in S_{y_0}\Diff(N)$. Indeed, the value is independent of the choice of $h_x$, whence we obtain a well-defined map 
 $$\omega \colon M \rightarrow N,\quad x \mapsto \ev_{x_0}(\phi(h_x))=\phi(h_x)(y_0)$$
 as this mapping does not depend on the choice of $h_x$. Again by the above, the map $\omega$ is a bijection which is even a homeomorphism (details will be checked in Exercise \ref{Ex:diffrec}).
 
 \textbf{Step 2:} \emph{$\omega$ induces $\phi$.} Let $y \in N$ and $h \in \Diff (N)$ with $h(y_0)=y$ and $x = \phi^{-1}(h)(x_0)$. By construction we have $\omega (x)=y$. If $f \in \Diff (M)$ we pick $g \in \Diff (M)$ with $g(x_0)=f(x)$. Then $f^{-1}(g(x_0))=x=\phi^{-1}(h)(x_0)$ and thus $g^{-1}f\phi^{-1}(h) \in S_{x_0} \Diff(M)$.
 We deduce that $(\phi(g))^{-1}\circ \phi(f) \circ h \in S_{y_0} \Diff (N)$ or in other words $\phi(f)\circ h(y_0)=\phi(g)(y_0)$. Now $h(y_0)=y=\omega(x)$ and $\phi (g)(y_0)=\omega (f(x))$ (as $g(x_0)=f(x)$. We deduce that 
 $$\phi(f)(\omega(x))=\omega (f(x)), \quad \phi(f)\circ \omega = \omega \circ f,$$
 and using that $\omega$ is bijective this yields $\phi(f) = \omega \circ f \circ \omega^{-1}$.
 
 \textbf{Step 3:} \emph{$\omega$ is unique.}
 Assume that there is another homeomorphism $\tilde{\omega} \colon M \rightarrow N$ inducing $\phi$. Then 
 $$\omega \circ f \circ \omega^{-1} = \phi (f) = \tilde{\omega} \circ f \circ \tilde{\omega}, \quad \forall f \in \Diff (M).$$
 In other words we have for $\rho \coloneq \tilde{\omega}^{-1} \circ \omega$ that $\rho \circ f \circ \rho^{-1} = f , \forall f\in \Diff (M)$. Arguing by contradiction we assume that $\rho \neq \id_M$. Then there exists $x \in M$ with $y = \rho(x) \neq x$. Pick $z \in M\setminus\{x,y\}$. Now $\Diff(M)$ is $2$-transitive, whence there is $f \in \Diff (M)$ with $f(x)=x$ and $f(y)=z$. We see that 
 $$\rho \circ f \circ \rho^{-1} (y) = \rho(f(x))=\rho(x)=y\neq z = f(y).$$
 However, this contradicts $\rho \circ f \circ \rho^{-1} =f$, whence we must have $\tilde{\omega} = \omega$.
 \end{proof}

\begin{proof}[{Sketch of the proof of \Cref{thm:recognition}}]
Every group isomorphism $\phi \colon \Diff(M) \rightarrow \Diff (N)$ satisfies condition \ref{cond:B} in the statement of \Cref{lem:ind_homeo} (this is far from trivial, cf.\ \cite{MR693972}),
Moreover, the group of smooth diffeomorphisms acts $n$-transitively for every $n \in \N$ if $\dim M > 1$, see \cite{MR1319441}. Thus we can apply \Cref{lem:ind_homeo} to obtain a homeomorphism $\omega \colon M \rightarrow N$. 

We prove that $\omega$ is a diffeomorphism under the assumption that $\phi \colon \Diff (M) \rightarrow \Diff(N)$ is a Lie group isomorphism. Note that the statement of the theorem is much stronger as, a priori, $\phi$ need not even be continuous. However, in this case one needs a deep result on Lie group actions on manifolds, see \cite[Step 3 on p.173]{MR693972}. A posteriori this implies that any group isomorphism $\Diff(M) \rightarrow \Diff (N)$ is already a Lie group isomorphism.

So let us assume that $\phi$ is a Lie group isomorphism and let us study the composition $\omega \circ \ev_{x_0} \colon \Diff (M) \rightarrow N$. By step 2 of the proof of \Cref{lem:ind_homeo} we can rewrite this as
$$\omega \circ \ev_{x_0} (h) = \ev_{y_0} (\phi (h)) \quad (\text{and conversely } \omega^{-1}\circ \ev_{y_0} = \ev_{x_0} \circ \phi^{-1}).$$
Now Exercise \ref{ex:canmfd} 3.a) shows that $\ev_{x_0}$ and $\ev_{y_0}$ are smooth surjective submersions. Hence the smoothness of the right hand side together with Exercise \ref{Ex:submersion} 5. shows that $\omega$ and $\omega^{-1}$ are smooth.
\end{proof}

We have seen that diffeomorphism groups determine (up to diffeomorphism) their underlying manifold uniquely. In the next section we shall discuss objects, so called Lie groupoids, which can be used to describe many finite-dimensional geometric structures and which admit a similar connection to infinite-dimensional groups.

\begin{Exercise}\label{Ex:diffrec}   \vspace{-\baselineskip}
 \Question We are working in the setting of \Cref{lem:ind_homeo} and let $\phi \colon \Diff(M) \rightarrow \Diff (N)$ be a group isomorphism which maps the stabiliser $S_{x_0}\Diff(M)$ to the stabiliser $S_{y_0}\Diff(N)$. 
 \subQuestion Show that the mapping $\omega \colon M \rightarrow N, x \mapsto \phi(h_x)(y_0)$ is a bijection (where $h_x \in \Diff (M)$ with $h_x(x_0)=x$-
 \subQuestion It is well known that $\Diff(M)$ satisfies the following condition: \emph{For any non-empty connected $U \opn M$ and $x \in U$, there exists $h \in \Diff (M) \setminus \{\id_M \}$ such that $\overline{\{y \in M \mid h(x)\neq x\}} \subseteq U$ and $x$ is contained in the interior of $\overline{\{y \in M \mid h(x)\neq x\}}$}. 
 
 Let $f \in \Diff (M)$ and define $\text{Fix}(f) \coloneq \{x \in M \mid f(x)=x\}$. Show that the set $\mathcal{B} \coloneq \{M\setminus \text{Fix}(f) \mid f \in \Diff (M)\}$ is a basis for the topology on $M$.
 \subQuestion Show that $\text{Fix} (\phi(g)) = \omega (\text{Fix}(g))$ holds and conclude that $\omega$ is a homeomorphism. 
\end{Exercise}

\section{Lie groupoids and their bisections}\label{subsect:groupoids}

In this section we consider a generalisation of Lie groups called Lie groupoids. These objects allow one to treat constructions in differential geometry as differentiable objects. For example quotients of manifolds modulo Lie group actions may fail to be manifolds. However, one can encode them using suitable Lie groupoids. We motivate the construction with the following example of an ill-behaved quotient.

\begin{ex}\label{ex:ill:behaved}
 In \Cref{sect:shapeanalysis}, we studied shape spaces which arise as quotients of manifolds of mappings modulo an action of the diffeomorphism group.
 Namely, we considered the canonical action of the Lie group $\Diff (\SSS^1)$ on the open submanifold $\text{Imm}(\SSS^1,\R^2) \opn C^\infty (\SSS^1,\R^2)$ (see \Cref{ex:Diffgp} and \Cref{lem:immsub_opn}) via precomposition 
 $$p \colon \text{Imm}(\SSS^1,\R^2) \times \Diff (\SSS^1) \rightarrow \text{Imm}(\SSS^1,\R^2), \quad (f,\varphi) \mapsto f \circ \varphi.$$
 The shape space $\mathcal{S} \coloneq \text{Imm}(\SSS^1,\R^2) /\Diff (\SSS^1)$ is then the quotient modulo the action. It inherits a natural topology which however does not turn $\mathcal{S}$ into a manifold. As the action $p$ is not free, the quotient has singular points in which one fails to obtain charts. An example for such a point is the image of the immersion of the circle into $\R^2$ given by the map tracing out the circle in ``double speed''
 $$c \colon \SSS^1 \rightarrow \R^2,\quad e^{i\theta} \mapsto e^{i2\theta}.$$
 Since $c$ traces the circle twice we see that $p(c,\varphi) = c \circ \varphi = c$ for the diffeomorphism $\varphi \colon \SSS^1 \rightarrow \SSS^1, e^{i\theta} \mapsto e^{i(\theta + \pi)}$ and we deduce that the immersion $c$ has a non-trivial stabiliser under the action $\rho$. Thus in particular, the quotient $\mathcal{S}$ is not a manifold (though every singularity is mild in the sense that it is generated by a finite group, i.e.~one obtains an infinite-dimensional orbifold\index{orbifold}, cf.\ \cite[Section 7.3]{Mic20}). As manifolds are the basic setting for differential geometry, one needs to pass to the subset generated by the free immersions (i.e.\, those immersions with trivial stabiliser).These form indeed a dense open subset which is a manifold.  
 \end{ex}

The last example exhibits that quotients of Lie group actions will in general not be manifolds. While in this special example one could still say a lot about the structure of the quotient, it shows that quotient constructions with manifolds are in general very badly behaved (not only in infinite dimensions). Thus we would like to avoid quotients and obtain an object which contains the same information as the quotient: a (Lie) groupoid. There are many literature accounts for the basic theory of (finite-dimensional) Lie groupoids such as \cite{Mackenzie05,Mein17}. While the finite-dimensional examples will be most important for us as they describe geometric constructions, the concept of a Lie groupoid can also be formulated in the infinite-dimensional context (as the notion of submersion makes sense in this setting).  

\begin{defn}[Groupoid]\label{defn:Liegpd}
 Let $G, M$ be two sets with surjective maps $\src,\trg \colon G \rightarrow M$ (\emph{source} and \emph{target}) and a partial multiplication $\mathbf{m} \colon G \times G\supseteq (\src,\trg)^{-1}(M \times M ) \rightarrow G, (a,b) \mapsto ab$ which satisfies:
 \begin{enumerate}
  \item $\src(ab) = \src(b)$ and $\trg(ab)=\trg(a)$, and $(ab)c=a(bc)$
  \item \emph{identity section} $\one \colon M \rightarrow G$ with $\one(\trg(g))g=g$ and $g\one(\src(g))=g$ for all $g \in G$, 
  \item \emph{inverses} $\forall g \in G$ there is $g^{-1} \in G$ with $g^{-1}g=\one(\src(g))$ and $gg^{-1}=\one(\trg(g))$.
 \end{enumerate}
 We call $G$ (or $\mathcal{G} = (G\toto M)$) a \emph{groupoid} and the set $M$ is called the \emph{set of units}. If $G,M$ are smooth manifolds, such that the structure maps $\src,\trg$ are smooth submersions and $\mathbf{m},\one$ and the inversion map $\mathbf{i} \colon G \rightarrow G, \mathbf{i}(g)=g^{-1}$ are smooth maps, we say that $\mathcal{G} = (G\toto M)$ is a \emph{Lie groupoid}.\index{Lie groupoid}
\end{defn}

\begin{tcolorbox}[colback=white,colframe=blue!75!black,title=Standard notation for Lie groupoids]
Throughout this section (if nothing else is said), we write $\cG =(G\toto M)$ for a Lie groupoid with structure maps $\one, \src,\trg, \mathbf{m}, \mathbf{i}$ as in the definition of a groupoid.
\end{tcolorbox}

\begin{rem}
 In this book manifolds are required to be Hausdorff. Thus \Cref{defn:Liegpd} excludes by design Lie groupoids $\mathcal{G} = (G \toto M)$ whose space of arrows $G$ is not Hausdorff. A broad (and important) class of Lie groupoids with non-Hausdorff space of arrows are the so called foliation groupoids, arising from the treatment of foliations in a groupoid framework, see \cite{MaM03}. In principle many of the results presented here are also valid in the non-Hausdorff setting, cf.~e.g.~\cite{Ryb02}.
\end{rem}

A useful mental image to keep in mind is to picture the units of the groupoid as dots connected by arrows which represent the elements of the groupoid which are not units. This is illustrated in \Cref{fig:gpd} below. Obviously two arrows can then only be composed if one of them ends where the other starts. This picture also immediately shows in which way a groupoid generalises the concept of a group: it can possess more units and its elements are not necessarily composable. 

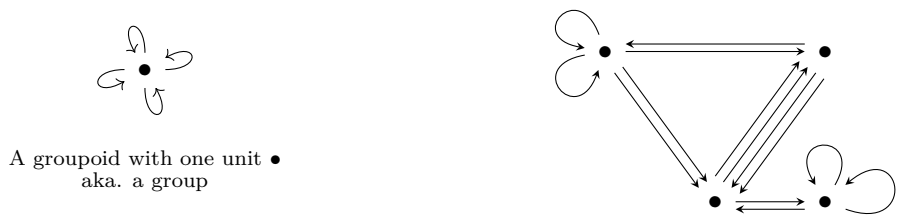
\begin{figure}[h]\label{fig:gpd}
 $$
 \begin{tikzcd}
  \bullet \arrow[out=0,in=30,loop,swap]
  \arrow[out=90,in=120,loop,swap]
  \arrow[out=180,in=210,loop,swap]
  \arrow[out=270,in=300,loop,swap] \\
  \substack{\text{A groupoid with one unit $\bullet$}\\ \text{aka. a group}}
\end{tikzcd} \hspace{3cm} \begin{tikzcd}[arrow style=tikz,>=stealth,row sep=4em]
\bullet \arrow[rr]
  \arrow[out=110,in=170,loop] 
   \arrow[out=190,in=250,loop] 
  \arrow[dr,shift left=.4ex]
  \arrow[dr,shift right=.4ex,swap]
&& \bullet\arrow[ll, shift right=.6ex] 
\arrow[dl,shift left=1.2ex]
  \arrow[dl,shift left=.4ex]&\\
& \bullet \arrow[ur,shift left=.4ex]
  \arrow[ur,shift left=1.2ex]
  \arrow[r]&
  \bullet \arrow[l,shift left =.6ex] 
  \arrow[out=-20,in=40,loop] 
   \arrow[out=60,in=120,loop] 
\end{tikzcd}
 $$ \caption{Picturing groups and groupoids. In the right picture we suppressed all arrows between the two nodes with looping arrows. As arrows tracing a path from one node to the other can always be composed, a picture of all groupoid elements would also need to represent these arrows.}
\end{figure}
\begin{defn}
For $G\toto M$ a groupoid, and $a \in M$ a unit, we consider the fibres $\src^{-1}(a)$ and $\trg^{-1}(a)$ of all arrows starting, respectively ending at $a$. The intersection $G_a \coloneq \src^{-1}(a) \cap \trg^{-1}(a)$ forms a group, called the \emph{vertex group}\index{Lie groupoid!vertex group} at $a$ of the groupoid. 
\end{defn}
 If $(G\toto M)$ is a Lie groupoid $\src^{-1}(a)$ and $\trg^{-1}(a)$ are submanifolds of $G$ by \Cref{cor:preimagesubmfd}. A natural question is then as to whether the vertex groups $G_a$ inherit a Lie group structure. In general (for arbitrary infinite-dimensional Lie groupoids) this question is still open as the finite-dimensional argument \cite[Corollary 1.4.11]{Mackenzie05} establishing the Lie group structure breaks down in infinite dimensions. However, it has recently been proven in \cite{BaGaJaP19} that if $G,M$ are Banach manifolds, then the vertex groups of $G \toto M$ are (Banach) Lie groups
Before we now finally give examples for Lie groupoids, let us first define groups which will play similar r\^{o}les as the diffeomorphism group in the previous section in relation to a manifold.

\begin{defn}
The \emph{group of bisections}\index{group of bisections} $\Bis (\cG)$ of $\cG$ is given as the set of smooth maps
 $\sigma \colon M \rightarrow G$ such that $\src \circ \sigma =\id_M$ and $\trg \circ \sigma \colon M \rightarrow M$ is a diffeomorphism. The group structure is given by the product
 \begin{equation}\label{eq: BISGP1}
  (\sigma \star \tau ) (x) \coloneq \sigma ((\trg \circ \tau)(x))\tau(x)\text{ for }  x \in M.
 \end{equation}
 The identity section $\one \colon M \rightarrow G$ is the neutral element and the inverse of $\sigma$ is
 \begin{equation}\label{eq: BISGP2}
  \sigma^{-1} (x) \coloneq \mathbf{i}( \sigma ((\trg \circ\sigma)^{-1} (x)))\text{ for } x \in M.
 \end{equation}
\end{defn}
 The definition of bisection is not symmetric with respect to source and target. This lack of symmetry can be avoided by defining a bisection as a set (see \cite[p.23]{Mackenzie05}). However, this point of view does not fit well into the function space perspective we take. Thus we shall stick with the asymmetric definition. 
 
Lie groupoids simultaneously generalise Lie groups and (differentiable) equivalence relations. To emphasise this, we recall the following standard examples (cf.~\cite{Mackenzie05}).
\begin{ex}In the following, we denote by $\{\bullet\}$ the one-point manifold. 
\begin{enumerate}
 \item Let $G$ be a Lie group. Then the Lie group structure yields a Lie groupoid $G \toto \{\bullet\}$, i.e.\ a Lie group is a Lie groupoid and conversely every Lie groupoid whose set of units contains only one element is a Lie group.
 In this case $\Bis (G \toto \{\bullet\})=G$
 \item Let $\pi \colon M \rightarrow N$ be a submersion. Then the fibre product of $M$ with itself gives rise to a Lie groupoid $M \times_N M \toto M$. Its source and target maps are given as $\src = \text{pr}_2$ and $\trg = \text{pr}_1$. Multiplication is then given by concatenation $(m,n) \cdot (n,k) \coloneq (m,k)$. Using \Cref{lem:pullbacksubm} it is not hard to see that this construction yields a Lie groupoid encoding the equivalence relation $x \sim y \Leftrightarrow \pi(x)=\pi(y)$. We mention two special cases of this construction:  
 \begin{itemize} 
  \item If $\pi = \id_M$,  we obtain the \emph{unit groupoid}\index{Lie groupoid!unit groupoid} $\mathfrak{u} (M) \coloneq (M \toto M)$ (where all structure mappings are the identity). Clearly $\Bis (\mathfrak{u}(M))=\{\id_M\}$.
  \item The map $\pi \colon M \rightarrow \{\bullet\}$ yields the \emph{pair groupoid},\index{Lie groupoid!pair groupoid} $\mathfrak{p}(M) = (M\times M \toto M)$. We shall see in Exercise \ref{Ex:Groupoid} 2. that $\Bis(\mathfrak{p}(M))\cong \Diff(M)$.
 \end{itemize}
 \end{enumerate}
 \end{ex}
 
 \begin{ex}
  Consider a (left) Lie group action $\alpha \colon G \times M \rightarrow M, (g,m) \mapsto g.m$. We form the \emph{action groupoid}\index{Lie groupoid!action groupoid} $\mathcal{A}_\alpha = (G \times M \toto M)$, where $\src (g,m) \coloneq m$ and $\trg (g,m) \coloneq \alpha(g,m)$. Now multiplication is defined as $(g,hm)\cdot (h,m)\coloneq(gh,m)$.
  The associated bisection group can be identified as 
  \begin{align}\label{bis:act}
   \Bis (\mathcal{A}_\alpha) = \{\sigma \in C^\infty (M,G) \mid m \mapsto \alpha(\sigma (m),m) \text{ is a diffeomorphism}\}. 
  \end{align}
  Now $C^\infty (M,G) \rightarrow C^\infty (M,M), f \mapsto \alpha_* (f \times \id_{M})$ is smooth. Continuity of this map together with $\Diff (M) \opn C^\infty (M,M)$ yields $\Bis (\cA_\alpha) \opn C^\infty (M,G)$. However, the bisections are not am open subgroup of the current group $C^\infty (M,G)$ (see \Cref{sect:currentgp}) as the multiplication is $\sigma \star \tau (m) = \sigma (\tau(m).m).m$ instead of the pointwise product.
  
  It is important to note that an action groupoid contains the same information as the group action and the quotient space. So instead of the ill-behaved quotient of the Lie group action 
  $$p \colon \text{Imm} (\SSS^1, \R^2) \times \Diff (\SSS^1) \rightarrow \text{Imm} (\SSS^1,\R^2),$$
  from \Cref{ex:ill:behaved} one could instead work with the (infinite-dimensional) Lie groupoid $\mathcal{A}_p = (\text{Imm} (\SSS^1, \R^2) \times \Diff (\SSS^1) \toto \text{Imm} (\SSS^1,\R^2))$ and carry out the geometric analysis on the groupoid instead of the quotient shape space. Since this quotient is of interest in shape analysis, Riemannian structures compatible with the Lie groupoid structure would then be needed to replace the metric on the quotient space. A suitable concept for such metrics has been worked out in \cite{dHaF18}.
 \end{ex}
 
 \begin{rem}
  Another interesting perspective on Lie groupoids is that they model symmetries which cannot be described by a global group action. A class of groupoids which fits well to this theme are the orbifold atlas groupoids, \cite{MaP97}. Recall that an \emph{orbifold}\index{orbifold} is a manifold with mild singularities, i.e.\ a Hausdorff space which is locally homeomorphic to a manifold modulo a finite group of diffeomorphisms. The key point here is that the local group acting is allowed to change. As a visual example consider the sphere $\SSS^2$ where the upper half is rotatet around the north pole by a rotation group of order $p$, while the lower half is rotated around the south pole by a rotation group of order $q$ and the results are glued together. Topologically the space is still $\SSS^2$ but the manifold structure breaks down at the two fixed points. It has been shown that these structures are equivalent to certain Lie groupoids. We refrain from discussing the rather technical details and refer instead to \cite{MaP97,MaM03} for a detailed account.
 \end{rem}

 The concept of a Lie groupoid carries over without any changes to infinite-dimensional settings (using submersions as defined in \Cref{sect:subm}). For example the action Lie groupoid modelling \eqref{ex:ill:behaved} is infinite-dimensional. Let us mention further examples: 
 In \cite{BaGaJaP19} Lie groupoids modelled on Banach spaces were studied. These arise naturally in studying certain pseudo-inverses in $C^\ast$-algebras.     
 As a more concrete example of a genuine infinite-dimensional Lie groupoid consider the following:
 
 \begin{ex}
  Let $\cG = (G\toto M)$ be a finite-dimensional Lie groupoid and $K$ a compact manifold. Then the pushforwards of the groupoid operations yield a Lie groupoid, called the \emph{current groupoid}\index{Lie groupoid!current groupoid} $C^\infty (K,\mathcal{G}) \coloneq (C^\infty (K,G) \toto C^\infty (K,M))$. It is an easy exercise (Exercise \ref{Ex:Groupoid} 3.) to verify that the current groupoid is a Lie groupoid. The theory for such groupoids was developed in \cite{AaGaS18}. There it was shown that current groupoids inherit many structural properties from the finite-dimensional target groupoids.
 \end{ex}

 However, it should be noted that from the rich theory available for finite-dimensional Lie groupoids, \cite{Mackenzie05}, virtually nothing is known for Lie groupoids modelled on general locally convex spaces. It is for example unclear as to whether the vertex groups always inherit a Lie group structure from the ambient groupoid.
 
 \begin{ex}\label{ex:gaugegpd}
  Let $(E,p,M,F)$ be a principal $G$-bundle where $E,M$ are Banach manifolds and $G$ is a Banach Lie group (see \Cref{defn:principal_bdl}). Denote by $e\cdot g$ the right-$G$ action on $E$ and consider the diagonal $G$-action $(e,f)\cdot g\coloneq (e\cdot g,f\cdot g)$ on $E\times E$. Then the quotient $Q\coloneq (E\times E) /G$ is a manifold (with the unique structure turning the quotient map into a submersion). We obtain a Lie groupoid, called the \emph{Gauge groupoid}\index{Lie groupoid!gauge groupoid} $\text{Gauge}(E) = (Q \toto M)$, associated to the principal $G$-bundle. The groupoid is given by the following source, target and identity maps (cf.~Exercise \ref{Ex:Groupoid} 5.)
 $$\src ([e,f])\coloneq p(f), \quad \trg ([e,f]) \coloneq p(e), \quad \one (m) \coloneq [u,u]\ (\text{where } u \in p^{-1}(x) \text{ is arbitrary}).$$
 Then we use the difference map $\delta \colon E \times_M E\rightarrow G, [e\cdot g,e] \mapsto g$ to define the multiplication 
 $$\mathbf{m} ([e,f],[\tilde{e},\tilde{f}])\coloneq [e,\tilde{f}\delta(f,\tilde{e})], \text{ where } (f,\tilde{e}) \in E \times_M E.$$
 So we can associate to every principal bundle (of Banach manifolds)\index{principal bundle} a Lie groupoid. Conversely, one can show that the gauge groupoid uniquely identifies the principal bundle (see \cite[Proposition 1.3.5]{Mackenzie05}). Hence a gauge groupoid contains the same information as a principal bundle. 
 The bisection group $\Bis (\text{Gauge}(E))$ is the automorphism group
 $$\text{Aut} (E,p,M) = \{f \in \Diff (E) \mid p \circ f \in \Diff (M), f(v\cdot g)=f(v)\cdot g, \forall v\in E, g\in G\}.$$
 This group is known to be an infinite-dimensional Lie group (cf.~\cite{AaCaMaM89}) which contains the group of gauge transformations from \Cref{defn:gaugegroup} as a (proper) Lie subgroup. 
 \end{ex}

 We have now seen in several examples that Lie groupoids can be used to formulate concepts from finite-dimensional differential geometry such as Lie group actions and principal bundles. Moreover, they come with an associated group, the bisection group, which in some instances can be identified with infinite-dimensional Lie groups. The next proposition shows that this is no accident.
 
 \begin{prop}\label{prop:bis:Lie}
  Assume that for a Lie groupoid $\cG$, $G$ is finite-dimensional and $M$ is compact. Then $\Bis (\cG)$ is a Lie group and $\trg_* \colon \Bis (\cG) \rightarrow \Diff (M), \sigma \mapsto \trg \circ \sigma$ is a Lie group morphism.
 \end{prop}

 \begin{proof}
  Recall from \Cref{ex:Diffgp} that $\Diff (M)$ is an open submanifold of $C^\infty(M,M)$. Further, the pushforward $\trg_* \colon C^\infty (M,G) \rightarrow C^\infty (M,M), f \mapsto \trg \circ f$ is smooth by \Cref{la-reu}. Since $\src \colon G \rightarrow M$ is a submersion, the Stacey-Roberts Lemma \Cref{lem:SR} asserts that $\src_* \colon C^\infty (M,G) \rightarrow C^\infty (M,M)$ is a submersion, whence the restriction $\theta \coloneq \src_*|_{\trg_*^{-1} (\Diff (M))}$ is a submersion. We deduce that $\theta^{-1} (\id_M) = \Bis (\cG)$ is a submanifold of $C^\infty (M,G)$.
To see that this manifold structure turns $\Bis (\cG)$ into a Lie group, we rewrite the formulae \eqref{eq: BISGP1} and \eqref{eq: BISGP2} as follows:
\begin{align*}
\sigma \star \tau = \mathbf{m}_* \big(\text{Comp} (\sigma, \trg_* (\tau))\big), \tau) \qquad \sigma^{-1} = \one_* \circ \text{Comp} (\sigma , \iota \circ \trg_* (\sigma))
\end{align*}
where $\mathbf{m}$ is groupoid multiplication, $\mathbf{i}$ groupoid inversion and $\iota$ the inversion in the Lie group $\Diff (M)$ (cf.\ \Cref{ex:Diffgp}). Since $M$ is compact, pushforwards and the composition map are smooth by \Cref{smooth:fullcomp}. In conclusion the group operations are smooth as composition of smooth mappings.

As $\Bis (\cG) \subseteq C^\infty (M,G)$ is a submanifold, the smoothness of $\trg_*$ on $\Bis (\cG)$ follows from the smoothness of pushforwards on manifolds of mappings, \Cref{la-reu}. To see that $\trg_*$ is a group morphism, we observe that 
$$(\trg_* (\sigma \star \tau )(x) = \trg (\sigma (\trg (\tau(x)))\tau(x)) = \trg (\sigma (\trg (\tau(x)))) = (\trg_* (\sigma) \circ \trg_* (\tau)) (x). \qedhere$$
 \end{proof}

  \begin{rem}
  The assumptions on $\mathcal{G}$ in the formulation of \Cref{prop:bis:Lie} are superfluous. The same proof (see \cite[Proposition 1.3]{AaS19b})works for any finite-dimensional Lie groupoid (dropping the compactness assumption on $M$), while in \cite[Theorem A]{SaW} a proof for compact $M$ but infinite-dimensional $G$ was given (thus dropping the assumption on $G$). The latter proof is believed to generalise to non-compact $M$ (and infinite-dimensional $G$).
  \end{rem}

 \begin{rem}
  The Lie group structure of the bisections turns $\Bis (\mathfrak{p}(M))\cong \Diff (M)$ into an isomorphism of Lie groups. Note however that \Cref{prop:bis:Lie} cannot replace the classical construction of the Lie group structure on $\Diff (M)$ as we exploited this structure already in the proof of the proposition.
 \end{rem}

 \begin{rem}\label{rem:vbis}
  The kernel of the Lie group morphism $\trg_* \colon \Bis (\cG) \rightarrow \Diff (M)$ is the \emph{group of vertical bisections}\index{group of bisections!vertical bisections}
  $\vBis (\cG)$. Under certain assumptions on the Lie groupoid, it was shown in \cite{Sch20} that the vertical bisections form an infinite-dimensional Lie group.
  
  If $\cG$ is a gauge groupoid of some principal bundle, the vertical bisections coincide with the group of gauge transformations of the bundle. Moreover, in this case, the Lie group structure of the vertical bisections coincides with the Lie group structure on the group of gauge transformations, \Cref{rem:gaugegroup}.
 \end{rem}

 Note that similar to the diffeomorphism group acting via evaluation on the underlying manifold, there is a canonical smooth action of the bisection group on the manifold of arrows of the groupoid.
 \begin{lem}\label{canon:act}
  The evaluation map induces a Lie group action 
  $$\gamma \colon \Bis (\cG) \times G \rightarrow G, \quad (\sigma, g) \mapsto \sigma (\trg (g)) \cdot g$$
 \end{lem}
 \begin{proof}
 Setting in the definition, it is immediately clear that $\gamma$ is a group action, Now rewrite $\gamma$ as $\gamma (\sigma,g) = \mathbf{m} (\ev (\sigma , \trg (g)),g),\quad \sigma \in \Bis (\cG), g \in G.$
 Exploiting the smoothness of the evaluation map, \Cref{base-cano}, we see that the action is smooth.
 \end{proof}

 With the help of the action one can identify the Lie algebra of the diffeomorphism group (cf.~Exercise \ref{Ex:Groupoid} 6.). We will focus here on global aspects of the theory and thus do not go into the details of the construction. However, it should be remarked that the Lie algebra of the bisection group is closely connected to the infinitesimal level of Lie groupoid theory. To make sense of this, let us mention that every Lie groupoid admits an infinitesimal object called a Lie algebroid (its r\^{o}le is similar to that of a Lie algebra associated to a Lie group). A Lie algebroid is a vector bundle together with certain additional structures. Equivalently, a Lie algebroid can be formulated as a special type of Lie algebra, called Lie-Rinehard algebra. In the present case, the Lie-Rinehard algebra turns out to be the Lie algebra of the bisection group. This is left as Exercise \ref{Ex:Groupoid} 7.\ and we refer to \cite{SaW,Mackenzie05} for more information.
 
\begin{Exercise}\label{Ex:Groupoid}   \vspace{-\baselineskip}
 \Question Let $\mathcal{G} = (G\toto M)$ be a Lie groupoid. Show that
 \subQuestion the domain of the multiplication $\mathbf{m}$ is a smooth manifold (whence it makes sense to require it to be smooth in the definition of a Lie groupoid),
 \subQuestion the unit map $\one \colon M \rightarrow G$ is a section of $\src$ and $\trg$ and as a consequence, $\one$ is a smooth embedding (i.e.~an immersion which is a homeomorphism onto its image).
 \subQuestion if only $\src$ is a submersion, so is $\trg$ (vice versa if $\trg$ is a submersion, so is $\src$). Hence the submersion requirements in the definition of Lie groupoid can be weakened.
 \Question Let $M$ be a manifold and $\alpha \colon G \times M \rightarrow M$ a Lie group action.
 \subQuestion Show that the bisection group of the pair groupoid $\mathfrak{p}(M)$ is isomorphic (as a group) to the diffeomorphism group $\Diff (M)$.
 \subQuestion Assume in addition that $M$ is compact. Show that the group isomorphism from a) becomes a Lie group isomorphism where the Lie group structure of $\Bis (\cG)$ is as in \Cref{prop:bis:Lie} and the one on $\Diff (M)$ as in \Cref{ex:Diffgp}. 
 \subQuestion  Work out the bisection group $\Bis (\mathcal{A}_\alpha)$ of the associated action groupoid. When is this group isomorphic to the current group $C^\infty (M,G)$? 
 \Question Use the Stacey-Roberts Lemma (\Cref{lem:SR}) to prove that the current groupoid $C^\infty (K,\cG)$ is a Lie groupoid for a finite-dimensional Lie groupoid $\cG$.
 \Question Let $G \toto M$ be a Lie groupoid such that $G,M$ are manifolds modelled on Banach spaces. Show that the multiplication map $\mathbf{m} \colon G \times_M G \rightarrow G$ is a submersion. Deduce that the multiplication in every Banach Lie group is a submersion.
 \\
 {\footnotesize \textbf{Hint}: Since we are in the Banach setting, a submersion is a mapping which admits smooth local sections (see Exercise \ref{Ex:submersion} 5.)}
 \Question Let $(E,p,M,F)$ be a principal $G$-bundle where $E,M$ are Banach manifolds and $G$ is a Banach Lie group. We check that the associated gauge groupoid is a Lie groupoid. Show that 
 \subQuestion one can construct a manifold structure on the quotient $(E\times E)/G$ turning the quotient map $E \times E \rightarrow (E\times E) /G$ into a submersion.\\ 
 {\footnotesize \textbf{Hint}: Cover $M$ by domains of sections of the submersion $p$ and use the sections to construct charts for the manifold. For the submersion use Exercise \ref{Ex:Groupoid} 4.}
 \subQuestion the structure maps $\src,\trg,\one, \mathbf{m}$ are smooth and that $\src, \trg$ are submersion. Conclude that the gauge groupoid is a Lie groupoid.\\
  {\footnotesize \textbf{Hint}: Assume $p$ is a surjective submersion and $q$ a smooth map between Banach manifolds. Then if $q\circ p$ is a submersion so is $q$, \cite[Proposition 4.1.5]{MR1173211}.}
  \Question Let $\mathcal{G} = (G\toto M)$ be a Lie groupoid. Show that by applying the tangent functor $T$ to every manifold and structure map of $\mathcal{G}$, one obtains a Lie groupoid $T\mathcal{G}$. One calls $T\mathcal{G}$ the \emph{tangent (prolongation) groupoid}\index{Lie groupoid!tangent groupoid} of $\mathcal{G}$. 
 \Question In this exercise we identify the Lie algebra of the bisections $\Bis (\cG)$ as a Lie algebra of sections of a certain vector bundle. Note that this is precisely the algebra of sections induced by the Lie algebroid $\Lf (\cG)$ associated to $\cG$. The Lie algebroid is the infinitesimal object associated to $\cG$ (similar to the Lie algebra associated to a Lie group). As we have no need for a discussion of Lie algebroids, we will not discuss it but refer instead to  \cite[3.5]{Mackenzie05}.
 \subQuestion Exploit that $\src$ is a submersion, and use submersion charts to show that 
 $$T^{\src} G \coloneq \bigcup_{g \in G} T_g\src^{-1} (\src{g}) = \bigcup_{g \in G} \text{ker} T_g \src.$$
 is a submanifold of $TG$ and even a subvector bundle of $TG$.
 \subQuestion Use Exercise \ref{Ex:submersion} 3.~to show that 
 $$T_{\one} \Bis (\cG) = \text{ker} T_{\one} \src_* = \{f \in C^\infty (M,TG) \mid f(m) \in T_{\one (m)} G \text{ for all } m\in M\}.$$
 and deduce that $T_{\one} \Bis (\cG) \cong \Gamma (\one^{*}T^{\src} G)$ as vector spaces (where the right hand denotes sections of the pullback bundle $T^{\src} G$.
 \subQuestion Observe that $\mathbf{m} (g,\cdot) \colon \src^{-1} (\src{g}) \rightarrow \src (g), h \mapsto gh$ is smooth for every $g\in G$ and show that every $X \in \Gamma(\one^* T^{\src}G)$ extends to a vector field on $G$ via the formula
 $$ \overrightarrow{X} (g) \coloneq T (\mathbf{m}(g,\cdot))(X(\trg (g))).$$
 Prove that $X= \overrightarrow{X} \circ \one$, whence the linear map $\Gamma(\one^* T^{\src}G) \rightarrow \mathcal{V}(G), X \mapsto \overrightarrow{X}$ must be injective and we can define a Lie bracket on $\Gamma(\one^* T^{\src} G)$ via $\LB[X,Y] \coloneq \LB[\overrightarrow{X},\overrightarrow{Y}] \circ \one$ (where the Lie bracket on the right is the Lie bracket of vector fields).  
 \subQuestion Adapt the proof to bisection groups, i.e.~show that if $X^R$ is a right-invariant vector field on $\Bis (\cG)$ then the vector field $X^R \times 0_G \in \mathcal{V} (\Bis (\cG) \times G)$ is $\gamma$-related to $\overrightarrow{X(\one)}$. Deduce from this that the Lie bracket can be identified with the negative of the bracket form c).
 \Question Let $\cG_1,\cG_2$ be Lie groupoids a \emph{morphism of Lie groupoids} is a pair of smooth maps $F \colon G_1 \rightarrow G_2$ and $f \colon M_1 \rightarrow M_2$ such that $\src_2 \circ F = f\circ \src_1$, $\trg_2 \circ F = f \circ \trg_1$ and $F(gh) = F(g)F(h)$ (whenever, $g,h \in G_1$ are composable). If $f = \id_{M_1}$, we say $F$ is a \emph{morphism over the identity}. Show that every morphism $F$ over the identity induces a Lie group morphism $F_\ast \colon \Bis (\cG_1) \rightarrow \Bis (\cG_2), \sigma \mapsto F\circ \sigma$.\\ 
 {\footnotesize \textbf{Remark:} So far avoided Lie groupoid morphisms as Lie groupoids and general morphisms exhibit a more complicated interplay as they form a $2$-category (thus the moniker ``higher geometry``). We will not delve into the details of this construction.}  \end{Exercise}

 \setboolean{firstanswerofthechapter}{true}
 \begin{Answer}[number={\ref{Ex:Groupoid} 3.}]
  \emph{We show that for a Banach Lie groupoid $G \toto M$ the multiplication map $\mathbf{m} \colon G \times_M G \rightarrow G$ is a submersion.}\\
  \textbf{Remark}: I do not know whether this proof (which was shared with me by D.M.~Roberts (Adelaide))or even the corresponding statement generalises beyond the Banach setting. \\[.5em]
  
  We exploit the following characterisation of submersions between Banach manifolds: The map $\mathbf{m}$ is a submersion if and only if it admits local sections, i.e.~for every $(g_1,g_2) \in G \times_M G$ exists a smooth map $\varphi \colon U \rightarrow G \times_M G$ such that $g\coloneq \mathbf{m}(g_1,g_2) \in U$, $\varphi (g)=(g_1,g_2)$ and $\mathbf{m} \circ \varphi = \id_{U}$ (cf.~\cite[Proposition 4.1.13]{MR1173211}). Thus we fix $g_1,g_2$ and $g$ as above and write $\varphi = (\varphi_1,\varphi_2) \in C^{\infty}(G,G \times G)$ (exploiting that $G\times_M G$ is a submanifold of the cartesian product, whence it suffices to obtain two smooth maps with values in $G$ such that their combination takes values in the fibre-product). Now exploiting the groupoid structure we observe that $g_2 = g_1^{-1} \cdot g$. Ignoring for a moment that multiplication is not globally defined, we see that for any smooth map $\varphi_1$ with $\varphi_1 (g)=g_1$, the smoothness of the groupoid operations yields a smooth map $\varphi_2$ via 
  \begin{align}\label{eq:secondcomp}
   \varphi_2 (x)= \varphi_1 (x)^{-1}\cdot x = \mathbf{m}(\mathbf{i}(\varphi_1 (x)),x).
  \end{align}
  Setting in $x=g$ we immediately see that $\varphi_2(g)=g_2$. However, to make the formula \eqref{eq:secondcomp} well defined we need to require that $\src (\varphi_1 (x)^{-1}) = \trg (x)$. Since inversion intertwines source and target this yields $\trg \circ \varphi_1 = \trg$. We deduce that it suffices to construct a certain smooth map $\varphi_1 \colon U \rightarrow G$ on some neighborhood $U$ of $g$ with $\trg \circ \varphi_1 = \trg$ and $\varphi_1 (g)=g_1$.
  
  Set $y \coloneq \trg (g)$ and observe that $\trg (g)=\trg (g_1g_2) =\trg(g_1)$. Since $\trg$ is a submersion, there is an open neighborhood $O_y$ of $y$ together with a smooth section $\sigma \colon O_y \rightarrow G$ such that $\sigma (y) = g_1$ and $\trg \circ \sigma =\id_{O_y}$. We can now choose an open neighborhood $g \in U$ such that $\trg (U) \subseteq O_y$ and define $\varphi_1 \colon U \rightarrow G, x \mapsto \sigma \circ \trg (x)$. Then $\varphi_1$ is smooth and satisfies $\varphi_1 (g)=g_1$ and $\trg \circ \varphi_1 = \trg$. We conclude that $\mathbf{m}$ admits local sections and is thus a submersion.
 \end{Answer}
 \setboolean{firstanswerofthechapter}{false}
 
 \begin{Answer}[number={\ref{Ex:Groupoid} 4.}] 
  \emph{We check that the Gauge groupoid associated to a principal $G$-bundle $(E,p,M,F)$ is a Lie groupoid if $E,M$ and $G$ are Banach manifolds.}\\[.5em]
  
  a) We begin with the construction of charts for $(E\times E)/G$. Let $(U_{i})_{i\in I}$ be an open cover of $M$ such that there exist smooth local sections $\sigma_{i}\colon U_{i} \rightarrow  E$ of $p$. This yields an atlas $(U_i,\kappa_i)_{i \in I}$ of local trivialisations of the bundle $p \colon E \rightarrow M$ which are given by
 \begin{equation*}
  \kappa_i \colon p^{-1} (U_i) \rightarrow U_i \times G,\quad x \mapsto (p (x), \mathrm{d}(\sigma_{i}(p(x)), x))
 \end{equation*}
 with $\mathrm{d} \colon E \times_M E \rightarrow G$, $(x,y) \mapsto x^{-1}\cdot y$. Here we use $x^{-1}\cdot y$ as the suggestive notation for the element $g\in G$ that satisfies $x\cdot g=y$.
 
 The local trivialisations commute with the right $G$-action on $E$ since
 \begin{equation*}
  \kappa_i (x\cdot g) = (p(x \cdot g), \mathrm{d} (\sigma_{i}(p(x\cdot g)),x\cdot g)=(p(x), \mathrm{d} (\sigma_{i}(p(x)),x)\cdot g.
 \end{equation*}
 In particular, the trivialisations descent to manifold charts for the arrow manifold of the gauge groupoid:
 \begin{align*}
  K_{ij} \colon (p^{-1}(U_i) \times p^{-1}(U_j))/G &\rightarrow U_i
  \times U_j \times G ,\\ [ x_1,x_2] &\mapsto (p (x_1), p (x_2) ,
  \mathrm{d}( \sigma_{i}(p( x_{1})),x_{1}) \mathrm{d} (\sigma_{j}(p(x_{2})),x_{2})^{-1}).
 \end{align*}
 To see that the projection is a submersion, it suffices to prove this locally in charts. In the trivialisations and the charts, the quotient becomes the map
 $$(U_i \times G) \times (U_j \times G) \rightarrow U_i\times U_j \times G , \quad \left((u_i,g_i), (u_j,g_j)\right) \mapsto (u_i,u_j,g_ig_j^{-1}).$$
 While we have the identity in the $u$-components, the $G$ component is the composition of inversion in the second component with the Lie group multiplication. Inversion in a Lie group is a diffeomorphism, while the multiplication in the Banach Lie group $G$ is a submersion by Exercise \ref{Ex:Groupoid} 4. Now the composition of submersions is a submersion by Exercise \ref{Ex:submersion} 1. and by Exercise \ref{Ex:submersion} 2., the quotient map is a submersion.
 
 b) Smoothness of the mappings follow from Exercise \ref{Ex:submersion} 6. by composing them with the submersion $q\colon E\times E \rightarrow (E\times E)/G$ and observing that the resulting mappings are smooth on $E\times E$. In particular, $\src \circ q = p$ is a surjective submersion, whence by \cite[Proposition 4.1.5]{MR1173211}, $\src$ is a surjective submersion.
 \end{Answer}

\section{(Re-)construction of a Lie groupoid from its bisections}

We shall now consider whether a Lie groupoid is determined by its group of bisections or can even be reconstructed from this group.
For the reconstruction, we consider again the Lie group action of the bisections on the manifold of arrows. This action turns out to be a submersion.

\begin{prop}\label{prop:canon:subm}
 Let $\cG$ be a finite-dimensional Lie groupoid. Then the following mappings are submersions:
 \begin{enumerate}
  \item $\gamma_m \colon \Bis (\cG) \rightarrow \src^{-1}(m), \quad \sigma \mapsto \sigma (m), \quad \forall m \in M$,
    \item $\ev \colon \Bis (\cG) \times M \rightarrow G, \quad \sigma \mapsto \sigma (m)$,
  \item $\gamma \colon \Bis (\cG) \times G \rightarrow G,\quad (\sigma, g) \mapsto \sigma (\trg (g))\cdot g$.
 \end{enumerate}
 \end{prop}

 \begin{proof}
  (a) In Exercise \ref{ex:canmfd} 4. we saw that the tangent of $\ev_m \colon C^\infty (M,G) \rightarrow G, f\mapsto f(m)$ is given by $T_f\ev_m \colon C^\infty_f (M,TG) \rightarrow T_{f(m)}G, h \mapsto h(m)$. By assumption $\src^{-1}(m) \subseteq G$ is a finite-dimensional manifold, whence \Cref{rel:subm:imm} shows that it suffices to prove that for each $\sigma \in \Bis (\cG)$ the tangent map $T_\sigma \Bis (\cG) \rightarrow T_{\sigma (m)} \src^{-1} (m)$ is surjective.\footnote{While the above formula immediately shows that $\ev_m$ is a submersion, we cannot directly conclude this for $\gamma_m$ without identifying the subspace of $C^\infty_\sigma (M,G)$ identified with $T_\sigma \Bis (\cG)$}. By construction $(\src_*)^{-1}(\id_M) \cap (\trg_*)^{-1}(\Diff (M)) = \Bis (\cG) \subseteq C^\infty (M,G)$ and since $\src_*$ is a submersion we have with Exercise \ref{Ex:submersion} 3.~ and arguments as in Exercise \ref{Ex:Groupoid} 6.b) that 
  $$T_\sigma \Bis (\cG) = \text{ker} T_\sigma (\src_*)  = \text{ker} (T\src)_*|_{T_\sigma C^\infty (M,G)}\cong \Gamma (\sigma^* T^{\src} G) \subseteq C^\infty_{\sigma} (M,G).$$
  This shows that $T_\sigma \gamma_m = T_\sigma \ev_m|_{T_\sigma \Bis_{\cG}}$ is surjective as every element of $T_{\sigma(m)}^{\src} G$ can be written as $X(\sigma (m))$ for some $X \in \Gamma (T^{\src}G) \cong \Gamma (\sigma^*T^{\src}G)$ (see Exercise \ref{Ex:quotientgpd} 2.). We deduce that $\gamma_m$ is a submersion.
  
  b) The proof turns out to be quite involved and involves a reduction step to the case already dealt with in a). Note that this is not obvious as $T\src^{-1}(m)$ is in general properly contained in $TG$. For similar reasons we cannot deduce the submersion property of $\ev$ from the submersion property of $\ev \colon C^\infty (M,G) \times M \rightarrow G$, For the proof and the necessary details we refer to \cite[Proposition 2.8 and Corollary 2.10]{SaW16}.  
  
  c) Note first that we can write $\gamma (\sigma , g) = \mathbf{m} (\ev(\sigma , \trg (g)),g)$
  and both $\trg $ and $\mathbf{m}$ are submersions (cf.~Exercise \ref{Ex:Groupoid} 4.). Hence $\gamma$ is a submersion as $\ev \colon \Bis (\cG) \times M \rightarrow G$ is a submersion by b).
 \end{proof}

We have now a submersion $\ev \colon \Bis (\cG) \times M \rightarrow G$ which is a Lie group action of the bisections on the arrow manifold of the groupoid we constructed the bisections from. 
Using that the unit embeds $M$ as a submanifold of $G$ and $\trg_*(\Bis (\cG)) \subseteq \Diff (M)$. This observation allows us to prolong the action of the bisections to an action on $M$:
$$A \colon \Bis (\cG) \times M \rightarrow M ,\quad (\sigma , m ) \mapsto \trg (\sigma (m)).$$
We will now show that the action groupoid constructed from this action determines (under certain conditions) the Lie groupoid $\cG$.

\begin{defn}
 Let $\cG = (G\toto M)$ be a Lie groupoid and $A \colon \Bis (\cG) \times M \rightarrow M$ the canonical Lie group action of the bisections on the units.
 Then we call the action groupoid $\Bis (\cG) \ltimes M \coloneq (\Bis (\cG) \times M \toto M)$ the \emph{bisection action groupoid}.\index{Lie groupoid!bisection action groupoid} 
 Furthermore, the map $\ev \colon \Bis (\cG) \times M \rightarrow G$ induces a Lie groupoid morphism $\mathbf{ev}$ over the identity\footnote{A Lie groupoid morphism over the identity is a smooth map $f \colon G \rightarrow G*$ (for groupoids $G \toto M$ and $G' \toto M$) which relates the structural maps of the groupoids, i.e.~$\src' \circ f = \src$, $\trg \circ f = \trg$, $\mathbf{m}'\circ (f,f) = f \circ \mathbf{m}$ and $\mathbf{i}'\circ f = f \circ \mathbf{i}$. We verify the conditions for $\ev$ in Exercise \ref{Ex:quotientgpd} 2.} 
\end{defn}

The question is now of course, whether the bisection action groupoid completely determines the Lie groupoid from which the bisections were constructed. 
In general this will not be the case as there will not be enough bisections to obtain all elements in the arrow manifold of a Lie groupoid.

\begin{ex}\label{ex:notenough}
Let $M$ be a compact manifold with two connected components $M = M_1 \sqcup M_2$ such that $M_1$ and $M_2$ are \textbf{not} diffeomorphic. We have seen in the last chapter that for the pair groupoid $\mathfrak{p}(M) = (M \times M \toto M)$ the bisection group can be identified as $\Bis (\mathfrak{p}(M))\cong \Diff (M)$. Consider now an element $(m_1,m_2) \in M$ such that $m_1 \in M_1$ and $M_2$. If there were a bisection $\sigma$ such that $\sigma (m_2) = (m_1,m_2)$ this would imply that there must be a diffeomorphism $\phi \colon M \rightarrow M$ such that $\phi (m_2) = \phi (m_1)$. As this entails $\phi (M_2) = M_1$ (since diffeomorphisms permute the connected components of a manifold), this is clearly impossible. We conclude that for every pair such that the elements come from different components, there cannot be a bisection through this element of the Lie groupoid.  
\end{ex}

So if there should be any hope that the bisection group identifies the Lie groupoid from which they were derived, we need to require that there are enough bisections in the following sense:
\begin{defn}
 A Lie groupoid $\cG = (G\toto M)$ is said to have \emph{enough bisections}\index{Lie groupoid!enough bisections} if for every $g \in G$ there exists a bisection $\sigma_g \in \Bis (\cG)$ with $\sigma_g (\src(g))=g$.
\end{defn}

Fortunately, sufficient conditions for a Lie groupoid to possess enough bisections are known. Indeed it turns out that the deficiency pointed out in \Cref{ex:notenough} is caused by a lack of connectedness.
This can be remedied by requiring that the groupoid $\cG$ is \emph{source connected}, i.e.~for every $m \in M$ the source fibre $\src^{-1}(m)$ is connected.

\begin{rem}
 If $\cG$ is the pair groupoid of a manifold, $\Bis (\cG) \cong \Diff (M)$, the groupoid is source connected if and only if $M$ is connected. Our next result will entail that $\Diff (M)$ acts transitively on $M$. We remark that connectedness of $M$ was required in the statement of the Takens/Filipkiewicz result in \Cref{sect:determination}. Note that transitivity of the group action was an essential ingredient in the proof of the result.
\end{rem}

\begin{lem}
 If $\cG$ is source connected, then $\cG$ has enough bisections.
\end{lem}

\begin{proof}
  The image $\cU \coloneq \ev (\Bis (\cG) \times M)$ contains the image of the object inclusion $\one$, i.e.\ $\one (m) \in \cU$ for all $m \in M$.
  Define for $m \in M$ the set $\cU_m = \cU \cap \src^{-1} (m)$ and note that $\cU_m = \ev (\Bis (\cG) \times \{m\})$.
  Clearly the set $\cU$ contains $\one (M)$ and forms a subgroupoid $\cU \toto M$ of $\cG$ of $G \toto M$ (see Exercise \ref{Ex:quotientgpd} 3.). As submersions are open maps, $\cU$ is an open Lie subgroupoid of $\cG$. Now As $\ev_m \colon \Bis (\cG) \rightarrow \src^{-1} (m)$ is a submersion by \Cref{prop:canon:subm} a) we infer that $\cU_m$ is an open subset of $\src^{-1} (m)$. 
  However, $\cU_m$ is also closed: The complement $\src^{-1}(m) \setminus \cU_m$ is the union $\bigcup_{g \in \src^{-1}(m)\setminus \cU_m}\cU_{\trg (g)}\cdot g$. As $\cU_{\trg (m)} \opn \src^{-1}(\trg(g))$, we see that $\cU_{\trg (g)}\cdot g$ is open, whence $\cU_m$ is also closed. We deduce that the closed and open set $\cU_m \subseteq \src^{-1}(m)$ equals $\src^{-1}(m)$ as $\cG$ is source connected.
  \end{proof}
If $\cG$ has enough bisections, the evaluation map $\ev \colon \Bis (\cG) \times M \rightarrow G$ becomes a surjective submersion. Hence the Lie groupoid structure of $\cG$ is completely determined by the group of bisections. One can moreover show that the original groupoid is a quotient groupoid of the bisection action groupoid in this case (see \cite[Theorem B]{SaW16}). As we do not wish to introduce deeper concepts in groupoid theory, we do not go into details concerning this result. The main upshot however is that for a Lie groupoid with enough bisections, the groupoid is uniquely determined by the action of the bisection group.

\begin{rem}
 The results presented so far in this section are a reconstruction result. This means that starting from a Lie groupoid, we can recover the Lie groupoid (under certain topological assumptions) as the quotient of a Lie groupoid we cook up from the action of the bisection group. 
 One can of course ask whether there are construction results for Lie groupoids which do not require to start from a Lie groupoid. Instead, one would like to start from an action of a suitable infinite-dimensional group and construct a Lie groupoid, such that in the case we started with the canonical action of a bisection group, one would recover the Lie groupoid. At least partial answers to these questions exist. We refer to \cite{SaW16}, where transitive pairs, i.e.\ a principal bundle version of Klein geometries, \cite[Chapter 3]{sharpe97}, are proposed as a starting point for a construction result.  
\end{rem}

\begin{Exercise}\label{Ex:quotientgpd}   \vspace{-\baselineskip}
 \Question Let $\cG = (G\toto M)$ be a Lie groupoid. Show that the image of the canonical action $E\coloneq \ev (\Bis (\cG) \times M) \subseteq G$ is closed under multiplication and inversion in $G$. Hence with the induced structure maps we obtain a subgroupoid $E \toto M$ of $\cG$.
  \Question  Prove that $\ev \colon \Bis (\cG) \times M \rightarrow G$ induces a Lie groupoid morphism $\Bis (\cG) \ltimes M \rightarrow G$ over the identity.
 \Question Let $\pi \colon E \rightarrow M$ be a finite-dimensional vector bundle. Show that for every $e \in E$ there is $X^e\in \Gamma (E)$ with $X^e(\pi(e)) = e$.\\
 {\footnotesize \textbf{Hint:} Construct locally in trivialisations using bump functions. Note that the assumption of being finite-dimensional can be replaced by requiring the existence of suitable bump functions.}
\end{Exercise}

\chapter{Euler-Arnold theory: PDE via geometry} \label{sect:EAtheory} \copyrightnotice

In this chapter we shall give an introduction to Euler-Arnold theory for partial differential equations (PDEs). The main idea of this theory is to reinterpret certain PDEs as smooth ordinary differential equations (ODEs) on infinite-dimensional manifolds. One advantage of this idea is that the usual solution theory for ODEs can be used to establish properties for the PDE under consideration. This principle has been successfully applied to a variety of PDE arising for example in hydrodynamics. Among these are the Euler equations for an ideal fluid, the Camassa-Holm equation, the Hunter-Saxton and the inviscid Burgers equation. We refer to \cite[p.34]{MR2456522} for a much longer list of physically relevant PDE which fit into this setting. 

As in the rest of the book, we shall only work with smooth functions. This is rather unnatural for solutions of partial differential equations but allows us to avoid spaces and manifolds of finitely often differentiable mappings. From the theoretical point of view this is problematic (at least considering the results we are after) and we will comment on the ``correct setting'' in the end of this chapter.
Before we begin, recall the relation between the energy of a curve and it being a geodesic (cf.\ \Cref{sect:L2metric}).  

\begin{defn}\label{defn:smoothvariation}
 Let $M$ be a manifold and for $x,y \in M$ we define the closed submanifold 
 $$C^\infty_{x,y}([0,1],M) \coloneq \{c\in C^\infty ([0,1],M) \mid c(0)=x, c(1)=y\} \subseteq C^\infty([0,1],M),$$ of curves from $x$ to $y$ (cf.\ Exercise \ref{Ex:EA:prelim} 1.).
 A smooth map $p \colon ]-\delta,\delta[ \times [0,1] \rightarrow M$ is a \emph{smooth variation}\index{smooth variation} of $c \in C^\infty_{x,y}([0,1],M)$ if $p(0,t) = c(t), \forall t \in [0,1]$ and $p(s,\cdot) \in C^\infty_{x,y}([0,1],M), \forall s \in ]-\delta,\delta[$.  
\end{defn}

With the notation in place, we can now formulate the following standard result from Riemannian geometry for weak Riemannian metrics which admit a metric spray.

\begin{prop}\label{Energy:vs:geodesic}
 Let $(M,g)$ be a (weak) Riemannian manifold which admits a metric spray $S$. Then $c \colon [0,1] \rightarrow M$ is a geodesic\index{geodesic} if and only if it extremises the energy  $$\mathrm{En}(c) =  \frac{1}{2}\int_a^b g_{c(t)} (\dot{c}(t),\dot{c}(t))\mathrm{d}t,$$
 i.e.\ if for each smooth variation $p\colon ]-\delta,\delta[\times [0,1] \rightarrow M$ of $c$ we have  $\left.\frac{d}{d s}\right|_{s=0} \mathrm{En} (p(s,\cdot)) = 0 $. 
\end{prop}

\begin{proof}
 Pick a smooth variation $q$ of a curve $c$ with $h(t) \coloneq \left.\frac{\partial}{\partial s}\right|_{s=0} q(s,t)$. We compute the derivative by exploiting the formula for the derivative of the energy from \Cref{lem:energydifferential} in a local chart $(U,\varphi)$. Suppressing again the identifications in the notation we find for $\left.\frac{d}{d s}\right|_{s=0} \mathrm{En} (q(s))$ the formula
 \begin{align*}
   &\int_0^1 \frac{1}{2} d_1g_U (c,c'(t),c'(t);h) - d_1g_U(c(t),h(t),c'(t);c'(t))-g_U (c(t),h(t),c''(t)) \mathrm{d}t\\ 
  \stackrel{\eqref{metricspray:Qform}}{=} &\int_0^1 g_U (c(t),B_U (c(t),c'(t)),c'(t),h(t)) - g_U (c(t),h(t),c''(t))\mathrm{d}t\\ 
  \stackrel{\hphantom{\eqref{metricspray:Qform}}}{=}& \int_0^1 g_U (c(t),B_U (c(t),c'(t),c'(t)) - c''(t),h(t)) \mathrm{d}t \stackrel{\eqref{locform:deriv}}{=} \int_0^1 -g_U (h(t),\nabla_{\dot{c}(t)} \dot{c}(t),h(t)) \mathrm{d}t.
 \end{align*}
 Hence $g_U (h(t),\nabla_{\dot{c}(t)} \dot{c}(t),h(t))$ needs to vanish for every $h$. Since the only element which gets annulled by all $g(h,\cdot)$ is $0$, we conclude that $\left.\frac{d}{d s}\right|_{s=0} \mathrm{En} (p(s,\cdot)) = 0$ holds for all smooth variations if and only if $\nabla_{\dot{c}} \dot{c} = 0$, i.e.\ $c$ is a geodesic.
\end{proof}

\begin{rem}
 \Cref{Energy:vs:geodesic} shows that geodesics are critical points of the energy. In \Cref{sect:geod_spray} we have taken the perspective that geodesics are locally length minimizing, i.e.~critical points of the length. However, the energy depends on the parametrisation while the length does not. To reconcile this, note that the Cauchy-Schwarz inequality yields equality of energy and length if $g_c(\dot{c},\dot{c})$ is constant. In Exercise \ref{Ex:EA:prelim} 2. we will see that every geodesic satisfies this property, whence our definition of geodesic comes with a preferred parametrisation which makes both points of view equivalent. 
\end{rem}

The variational approach will enable us to identify the geodesic equations of infinite-dimensional Riemannian metrics. Before we turn to these results, let us fix some notation for this and later sections.

\begin{tcolorbox}[colback=white,colframe=blue!75!black,title=Concerning partial derivatives]
 Let $p \colon I \times J \rightarrow M$ be a smooth variation. Assume that $\nabla$ is the metric derivative of the Riemannian manifold $(M,g)$. For $s \in I$ and $t \in J$ we denote by $\partial s$ and $\partial t$ the constant unit vector field (on $I$ resp.\ on $J$).
For $p$, we denote by $\frac{\partial}{\partial s} p(s,t) \coloneq Tp(s,t)(1,0)$ and $\frac{\partial}{\partial t} p(s,t) \coloneq Tp(s,t)(0,1)$ the partial derivatives as vector fields along $p$. Hence we can consider e.g.\ $\frac{\nabla}{\partial s} \frac{\partial}{\partial t} p$ as a vector field along $p$ (in the sense of \Cref{defn:covalong}). 
Moreover, we note that as $\nabla$ is the metric derivative (whence torsion free), \cite[Proposition 1.5.8. i)]{MR1330918} implies that 
\begin{align}\label{commuting:derivatives}
 \frac{\nabla}{\partial s} \frac{\partial}{\partial t} p = \frac{\nabla}{\partial t} \frac{\partial}{\partial s} p.
\end{align}
\end{tcolorbox}

\begin{ex}[The inviscid Burgers equation]\label{ex:inv-burger}
 Recall from \Cref{cor:diffopn} that the diffeomorphism group $\Diff (\SSS^1) \opn C^\infty (\SSS^1,\SSS^1)$ is an open submanifold. Fix some Riemannian metric $g$ on $\SSS^1$. We exploit that $\SSS^1 \subseteq \R^2$ is an embedded submanifold and apply \Cref{rem:L2deriv} c) to deduce from \Cref{prop:covdident} that the $L^2$-metric\index{Riemannian metric!$L^2$-metric} on $C^\infty (\SSS^1,\SSS^1)$ admits a metric spray. The same then holds to its restriction to $\Diff (\SSS^1)$:
 $$g^{L^2} (X ,Y ) = \int_\SSS^1 g_{\varphi (\theta)}( X(\theta), Y(\theta)) \mathrm{d}\theta.$$
 Where we exploited the identification $T_\varphi \Diff (\SSS^1) \cong \{X \circ \varphi \mid X \in \mathcal{V}(\SSS^1)\}$.
 Pick now a smooth variation $c \colon ]-\delta , \delta [ \times [0,1] \rightarrow \Diff (\SSS^1)$ of some curve (which we also denote by $c$). Set $v (t) \coloneq \left.\frac{\partial}{\partial s}\right|_{s=0} c(s,t)$ and note that taking the derivative with respect to $s$ coincides with taking the derivative with respect to the unit vector field $\partial s$. We can now compute the variation of the energy
 \begin{align*}
  \left.\frac{d}{d s}\right|_{s=0} \mathrm{En} (c(s,\cdot)) &\stackrel{\hphantom{\eqref{covalong:Riemannian}}}{=} \frac{1}{2}\int_0^1 \int_\SSS^1 \left.\frac{\partial}{\partial s}\right|_{s=0}g\left(\frac{\partial}{\partial t} c(s,t)(x), \frac{\partial}{\partial t} c(s,t)(x)\right) \mathrm{d}\theta \mathrm{d}t\\
  &\stackrel{\eqref{covalong:Riemannian}}{=}\int_0^1 \int_\SSS^1 \left.g\left(\frac{\nabla}{\partial s}\frac{\partial}{\partial t} c(s,t)(x), \frac{\partial}{\partial t} c(s,t)(x)\right)\right|_{s=0} \mathrm{d}\theta \mathrm{d}t\\
  &\stackrel{\eqref{commuting:derivatives}}{=} \int_0^1 g^{L^2}\left. \left(\frac{\nabla}{\partial t} \frac{\partial}{\partial s} c(s,t)(x), \frac{\partial}{\partial t}c(s,t)(x) \right)\right|_{s=0}\mathrm{d}t \\ 
  &\stackrel{\hphantom{\eqref{covalong:Riemannian}}}{=} -\int_0^1 g^{L^2} \left(v(t), \frac{\nabla}{\partial t}\frac{\partial}{\partial t}  c(t)\right) \mathrm{d}t
 \end{align*}
 where the last equality is due to a usual integration by parts argument.
 In particular, we see that if $c \colon [0,1] \rightarrow \Diff (\SSS^1)$ should be a geodesic, it must satisfy the pointwise equation 
 \begin{align}\label{eq:BUrgercond}
  \frac{\nabla}{\partial t} \frac{\partial}{\partial t} c^\wedge(t,\theta) = 0,\quad \forall \theta \in \SSS^1. 
 \end{align}
 To connect this equation to objects on the finite-dimensional manifold $\SSS^1$, we construct a (time-dependent) vector field on $\SSS^1$ from the curve of diffeomorphisms $c(t)$ by setting
 \begin{align}\label{rel:recon}
  u(t,\theta) = \frac{\partial}{\partial t} c^\wedge (t, c^{-1} (t,\theta))
 \end{align}
 where the inverse is the inverse in $\Diff (M)$. In other words, $c(t)(\theta)$ is the flow of the time-dependent vector field $u$, i.e.\ $\frac{\partial}{\partial t} c(t)(\theta) = u(t,c(t)(\theta))$. Plugging this definition into the left hand side of \eqref{eq:BUrgercond}, the chain rule and a quick computation yields the following statement. 
 \begin{lem}\label{lem:cov_deriv_calc}
 Let $c\colon [0,1] \rightarrow \Diff (\SSS^1)$ be smooth and the flow of the time-dependent vector field $u$ on $\SSS^1$, cf.\ \eqref{rel:recon}. Then
  \begin{align}\label{Burgers:Euler}
 \frac{\nabla}{\partial t} \frac{\partial}{\partial t} c^\wedge(t,\theta) &= \frac{\nabla}{\partial t} (u(t,c^\wedge (t,\theta))) = \frac{\partial u}{\partial t}(t,c^\wedge(t,\theta)) + \nabla_u u (t,c^\wedge(t,\theta),
\end{align}
where $\nabla_u u$ denotes the covariant derivative of $u$ against itself for fixed $t$ and we interpret $\frac{\partial u}{\partial t} (t,\theta)$ as a partial derivative of $u (\cdot, \theta)$ in $T_\theta \SSS^1$ for every fixed $\theta$.
  \end{lem}
  
  Note that \eqref{Burgers:Euler} is central to the idea of Euler-Arnold theory (whence we promoted it to its own lemma) and holds in similar form if one replaces $\SSS^1$ by an arbitrary smooth compact manifold $M$.
  To distinguish the interpretation of $\frac{\partial u}{\partial t} (t,\theta)$ from the usual partial derivative of a smooth variation, let us write $\partial_t u$ for this derivative. We conclude from \Cref{lem:cov_deriv_calc} and \eqref{eq:BUrgercond} that $c$ is a geodesic of the $L^2$-metric if and only if the associated vector field $u$ solves the inviscid Burgers (or Hopf) equation
 \begin{align}\label{Burgersmfd}
\partial_t u + \nabla_u u =0.  
 \end{align}
 Burgers equation is a partial differential equation ($\nabla_u u$ takes derivatives of $u$) which is miraculously equivalent to an ordinary differential equation (the geodesic equation) on the infinite-dimensional group $\Diff (\SSS^1)$. 
 Note that \eqref{Burgersmfd} is connected to the classical Burgers equation $u_t + 3uu_x=0$ (subscripts denoting partial derivatives) from \Cref{Burger1}. This will be made explicit in Exercise \ref{Ex:EA:prelim} 3.: Since $\SSS^1 \subseteq \R^2$ is an embedded submanifold, we endow $\SSS^1$ with the pullback metric induced by the euclidean inner product $g_x(v,w) \coloneq \langle v,w\rangle$ (by identifying $T_x \SSS^1 \subseteq \R^2$). Working out the covariant derivatives a canonical identification shows that the Burgers equation \eqref{Burgers:Euler} coincides with the classical Burgers equation $u_t + uu_x =0$ (which is up to scaling equivalent to \Cref{ex:inv-burger}).
\end{ex} 

\begin{rem}
 The derivation of Burgers equation as a geodesic equation on $\SSS^1$ did not exploit any special structure of $\SSS^1$. It just enabled us to make sense of the integrals without recourse to integration against volume forms.
 Thus the same argument carries over without any change to an arbitrary compact manifold $M$. There the inviscid Burgers equation 
 $$\partial_t u + \nabla_u u =0$$
 makes sense (with respect to the covariant derivative induced by the metric) and is the geodesic equation of the $L^2$-metric (cf.\ \eqref{generalL2} below) on $\Diff (M)$.
\end{rem}

In the next section we will systematically investigate the mechanism to associate a geodesic equation to certain partial differential equations. 

\begin{Exercise}\label{Ex:EA:prelim}\vspace{-\baselineskip}
 \Question Let $(M,g)$ be a (weak) Riemannian manifold (with metric derivative $\nabla$) and denote by $C^\infty ([0,1],M)$ the space of smooth curves with the manifold structure from \Cref{sect:mfdofcurves}. Fix $x,y \in M$. Show that 
 \subQuestion $C^\infty_{x,y} ([0,1],M) =\{c\in C^\infty ([0,1],M) \mid c(0)=x, c(1)=y\}$ is a closed submanifold of $C^\infty ([0,1],M)$.\\
 {\footnotesize \textbf{Hint:} Consider a canonical chart for the manifold of mappings. Show that the model space splits for every $c\in C^\infty_{x,y} ([0,1],M)$.}
  \subQuestion $p \colon ]-\delta,\delta[\times [0,1]\rightarrow M$ is a smooth variation of $c \in C^\infty_{x,y}([0,1],M)$ if and only if $p^\vee \colon ]-\delta,\delta[\rightarrow C^\infty_{x,y} ([0,1],M)$ is smooth.
  \subQuestion the energy $\mathrm{En}$ restricts to a smooth function on $C^\infty_{x,y} ([0,1],M)$ and prove the following analogue of \Cref{Energy:vs:geodesic}: A curve $c$ is a geodesic connecting $x$ and $y$ if and only if $d\mathrm{En} (c;\cdot)$ vanishes on $C^\infty_{x,y}([0,1],M)$.
  \Question Let $(M,g)$ be a weak Riemannian manifold which allows a metric spray $S$ and an associated metric derivative $\nabla$. Show that a geodesic $c \colon [0,1] \rightarrow M$ is a curve of constant speed, i.e. $g_c (\dot{c},\dot{c})$ is constant.\\
  {\footnotesize \textbf{Hint:} Use \Cref{setup:metricderiv} to show that  $\frac{d}{dt}g_c (\dot{c},\dot{c})$ vanishes}
  \Question Show that in \Cref{ex:inv-burger} we can rewrite \eqref{Burgersmfd} in traditional notation as
  $$u_t + uu_x = 0,\quad \text{ where } u \colon ]-\delta ,\delta[ \times [0,2\pi] \rightarrow \R.$$
  In addition show that in this setting the flow $\eta$ of $u$ satisfies then $\frac{\partial^2}{\partial t^2} \eta = 0$.\\
  {\footnotesize \textbf{Hint:} Use that the tangent bundle of $\SSS^1$ is trivial together with \Cref{ex:submfd:covd}.}
  \Question Prove \Cref{lem:cov_deriv_calc}.
 \end{Exercise}

\section{The Euler equation for an ideal fluid}
We shall exhibit the general principle first for the classical example considered by Arnold \cite{MR202082}: the Euler equation for an inviscid incompressible fluid on a Riemannian manifold. It describes the development of a fluid occupying the manifold $M$ under certain assumptions. Let us first fix some notation.
\begin{tcolorbox}[colback=white,colframe=blue!75!black,title=Conventions]
In this section we will denote by $(M,g)$ a compact (thus finite-dimensional) orientable Riemannian manifold.
\begin{itemize}
 \item For a (time-dependent) vector field $u$ we write $\partial_t u(t,x)$ for the partial $t$-derivative in $T_x M$ (thus not taking values in $T(TM)$!).  
 \item Since $M$ is orientable, it admits a volume form $\mu$ induced by $g$. Further, we denote by $\divr X$ the \emph{divergence}\index{vector field!divergence} of a vector field $X \in \mathcal{V} (M)$ and by $\grad f$ the \emph{gradient}\index{gradient} of a smooth function $f \colon M \rightarrow \R$ (see \Cref{supp:diff-findim}).
\end{itemize}
\end{tcolorbox}

We will not derive Eulers equations here from first principles, but refer to \cite{modin2019geometric} for an account together with a  history of the problem.

\begin{setup}[Euler equation]
 The \emph{Euler equation} for an incompressible fluid is 
 \begin{align} \label{Euler:eq}
  \begin{cases}
   \partial_t u(t,m) + \nabla_u u (t,m) = -\grad p &\\
   \divr u(t,\cdot) = 0, & \forall t  \\
   u(0,\cdot) = u_0 & \text{with } \divr u_0 = 0 
  \end{cases}
 \end{align}
where the function $p \colon \R \times M \rightarrow \R$ is interpreted as 'pressure'. Euler's equation searches for a (time-dependent) vector field $u$ on $M$. The condition, $\divr u(t,\cdot) = 0$, i.e.~that $u$ is divergence-free, is the condition enforcing the incompressibility of the fluid. 

In \eqref{Euler:eq} we seek a vector field, whence one says that the equation is in \emph{Eulerian form}.\index{Euler equation!Eulerian form}
\end{setup}

\begin{rem}
 Apart from the incompressibility condition, Euler's equation is similar to Burgers equation. Indeed the only difference on the PDE level is the right hand side which is given by the gradient of a pressure function. We will see later that the gradient acts as a Lagrange multiplier enforcing the incompressibility condition (in general the term $\nabla_u u$ will not be divergence-free).
\end{rem}

Again Euler's equation, just like Burgers equation, is formulated on a finite-dimensional manifold and has a priori nothing to do with infinite-dimensional geometry. However, we will change the perspective to uncover the connection to infinite-dimensional geometry. Again the idea is similar to what we did for Burgers equation (but we will now start with the vector field rather than a flow).

\subsection*{From the Eulerian to the Lagrangian perspective}
  Let us consider a time-dependent smooth vector field $u \colon I \times M \rightarrow TM$ on a compact interval $I$. Then recall (e.g.~from \cite[IV. \S 1]{Lang}, where we exploit that $M$ is compact) that the flow for $u$ is a mapping $\eta \colon I \times M \rightarrow M$ such that
\begin{align}\label{Euler:eq:Lag}
 u (t,\eta(t,m)) = \frac{\partial}{\partial t} \eta (t,m), \forall t \in I, m \in M.
\end{align}
Furthermore, it is well known that for each $t \in I$ the flow $u(t,\cdot)$ is a diffeomorphism of $M$. The equation \eqref{Euler:eq:Lag} or the equivalent equation \eqref{rel:recon} are sometimes called the \emph{reconstruction equation}. Observe now that instead of constructing a vector field $u$ which solves the Euler equation \eqref{Euler:eq}, we can instead construct its flow. Searching for the flow whose associated vector field solves the PDE, is called the \emph{Lagrangian perspective} on the PDE.
If now $u$ is divergence-free, $\divr u = 0$, this implies that $\eta_\ast \mu = \mu$, i.e.\ for every $t$, the diffeomorphism $\eta(t,\cdot)$ leaves the volume form invariant. As $\eta$ is smooth, the exponential law, \Cref{thm:explaw}, allows us to reinterpret the flow $\eta \colon J \times M \rightarrow M$ as a smooth curve $\eta^\vee \colon J \rightarrow \Diff (M) \opn C^\infty (M,M)$. Now the incompressibility condition shows that $\eta^\vee$ takes its values in the closed Lie subgroup $\Diff_\mu (M) \subseteq \Diff (M)$ of volume-preserving diffeomorphisms.

Our aim is again to connect the finite-dimensional PDE to the infinite-dimensional Riemannian geometry induced by the $L^2$-metric. Let us briefly recall its definition:

\begin{defn}[$L^2$-metric on the diffeomorphism group]
Let $(M,g)$ be a compact Riemannian manifold\footnote{Readers unfamiliar with integration on manifolds may safely replace $M$ in the following by $\SSS^1$. However, this does not simplify any of the argument.}. We define a weak Riemannian metric on $\Diff (M)$ via 
\begin{equation}\label{generalL2}
 g^{L_2}_\varphi (X \circ \varphi, Y \circ \varphi) = \int_M g_{\varphi (m)} (X(\varphi(m)),Y(\varphi(m)))\mathrm{d}\mu (m).
\end{equation}
Here $X,Y \in \mathcal{V} (M), \varphi \in \Diff(M)$ and we exploited that $\Diff (M)$ is a Lie group, whence its tangent bundle is trivial, i.e.\ $T\Diff (M) \cong \Diff(M) \times \mathcal{V} (M)$ (where the diffeomorphism is induced by right translation). It follows directly from the rules of integration that $g^{L^2}$ is a right-invariant Riemannian metric\index{Riemannian metric!invariant $L^2$-metric} (see \Cref{setup:invmetric} on the subgroup $\Diff_\mu (M)$ of volume-preserving diffeomorphisms (but not on $\Diff (M)$!).\index{diffeomorphism group!volume-preserving} Moreover, in Exercise \ref{Ex:EA:Euler} 2. we will see that the weak Riemannian metric admits a metric spray and a covariant derivative. 
With more work, one can also establish this for the restriction of the $L^2$-metric to the closed Lie subgroup $\Diff_\mu (M)$ (cf.\ \cite[Theorem 9.6]{EM70}).
\end{defn}

We have seen above that for vector fields solving Eulers equations, the associated flow yields a curve into the group of volume-preserving diffeomorphisms. Since $g^{L^2}$ induces a weak Riemannian metric on $\Diff_\mu (M)$ which admits a metric spray, we can compute geodesics as curves which extremise the energy. This allows us to derive a differential equation for the flow corresponding to vector field solutions of \eqref{Euler:eq}. As in the Burgers case, we need that $g^{L^2} (h,\frac{\nabla}{\partial t}\frac{\partial}{\partial t} \eta)$ vanishes for every $h \in T\Diff_\mu (M)$. We know that the tangent spaces of $\Diff_\mu (M)$ is (up to a shift) given by divergence-free vector fields. Now due to the Helmholtz decomposition \Cref{Helmholtz}, elements which are $L^2$-orthogonal for every $h$ are gradients of functions. Thus if $\frac{\nabla}{\partial t}\frac{\partial}{\partial t} \eta$ is a gradient then the inner product vanishes and the curve $\eta$ extremises the energy. Conversely, if we assume that $u$ solves \eqref{Euler:eq} and denote by $\eta$ its flow, we compute the derivative as follows
\begin{align}\label{Euler:eq:Lagrange}  \begin{split}
 \frac{\nabla}{\partial t}\frac{\partial}{\partial t} \eta (t,m) &= ((\partial_t u) + \nabla_u u))(t,\eta(t)) = -\grad p (t,\eta(t,m)). 
 \end{split}
\end{align}
In other words, the Helmholtz decomposition shows that a flow $\eta$ extremises the energy if and only if its associated vector field solves Eulers equation \eqref{Euler:eq}. Hence the Euler equation can equivalently be formulated as the following set of ordinary differential equations on the volume-preserving diffeomorphisms.

\begin{setup}[Lagrangian formulation of the Euler problem]
Find $\eta(t,\cdot) \in \Diff (M)$ such that
 \begin{align}\label{Euler:Lagrangeproblem}
  \begin{cases}
    \frac{\nabla}{\partial t} \frac{\partial}{\partial t} \eta (t,x) = - \grad p (t,\eta(t,x))\\ 
    \eta (t,\cdot)_\ast \mu = \mu &  \forall t \\
    \eta(0,\cdot) = \id_M
   \end{cases}
 \end{align}
 These equation \eqref{Euler:Lagrangeproblem} are called \emph{Euler equations in Lagrangian form}.\index{Euler equation!Lagrangian form}
\end{setup}

We achieved our goal to rewrite Eulers equation as a differential equation on an infinite-dimensional manifold. However, we have not yet exploited that the metric and the equation are right-invariant (with respect to the group multiplication). In the next section we will investigate these properties and connect the Lagrangian formulation to the geometry of the Lie group at hand. This will lead (among other things) to another derivation of the geodesic equation as the Euler equation. While this might on first sight look like a superfluous exercise (after all we already know that Eulers equations can be rewritten as the geodesic equation) we wish to point out that this property is crucial for the investigation of PDE in the Euler-Arnold framework we present here.

\section{Euler-Poincar\'{e} equations on a Lie group}
Comparing the Lagrangian version of Eulers equations \eqref{Euler:Lagrangeproblem} and the $L^2$-metric, it is immediate that all terms arise by right shifting objects. Moreover, as the tangent of the right shift with a diffeomorphism is just precomposition with the diffeomorphism, we obtain for a curve $c$ with values in the volume-preserving diffeomorphisms the formula 
\begin{align}\label{inv:Energy}
\mathrm{En} (c) = \int_0^1 \int_M g_m(\dot{c}(t) \circ c(t)^{-1}(m),\dot{c}(t)t  \circ c(t)^{-1}(m)) \mathrm{d}\mu(m)\mathrm{d}t . 
\end{align}
Hence we can compute the energy using the $L^2$-inner product on the Lie algebra. We shall see that the geometry of the Lie group is tightly connected to the geodesics by virtue of the Riemannian metric being right-invariant. 
To state Eulers equations as a geodesic equation, we need to understand derivatives of right shifted variations.

\begin{lem}\label{lem:commutator:smoothvar}
 Let $G$ be a Lie group and $p \colon ]-\delta,\delta[ \times [0,1] \rightarrow G$ a smooth variation. We identify $TG \cong G \times \Lf (G)$ by right multiplication (cf.~\Cref{tangentLie}) and define:
 \begin{align*}
  X_p \colon  ]-\delta,\delta[ \times [0,1] &\rightarrow \Lf (G),\quad (s,t) \mapsto \pr_2 \left(T\rho_{p(s,t)^{-1}} \frac{\partial}{\partial t} p(s,t)\right),\\
  Y_p \colon  ]-\delta,\delta[ \times [0,1] &\rightarrow \Lf (G),\quad (s,t) \mapsto \pr_2 \left(T\rho_{p(s,t)^{-1}} \frac{\partial}{\partial s} p(s,t)\right).
 \end{align*}
 Then the mixed derivatives of the \emph{right shifted variations}\index{smooth variation!right shifted} are related as follows
 \begin{align}\label{eq:comm:mixed}
 \frac{\partial}{\partial s} X_p - \frac{\partial}{\partial t} Y_p   = -\LB[X_p,Y_p]  
 \end{align}
\end{lem}

\begin{proof}
It suffices to establish \eqref{eq:comm:mixed} pointwise for every pair $(s_0,t_0)$. Define the smooth map $\tilde{p}(s,t) \coloneq p(s,t)\cdot p(s_0,t_0)^{-1}$. Then $\tilde{p}(s_0,t_0)=\one_G$ and a quick calculation shows that we have $\pr_2 \left(T\rho_{\tilde{p}(s,t)^{-1}} \left(\frac{\partial}{\partial t} \tilde{p}(s,t)\right)\right) = X_p(s,t)$.
A similar identity holds for $Y_p$ and we may thus assume without loss of generality for the proof that $p(s_0,t_0) = \one_G$.

We work locally and pick a chart $\varphi \colon G \supseteq U \rightarrow V \subseteq \Lf (G)$ such that $\varphi (\one_G) = 0$ and $T_{\one_G} \varphi =\id_{\Lf (G)}$. As in \Cref{setup:locmult} we define a local multiplication $v \ast w \coloneq \varphi(\varphi^{-1}(v)\varphi^{-1}(w))$ for all elements $v,w \in \Lf (G)$ near enough to zero. For an element $v$ we define (if it is close enough to $0$) the inverse $v \coloneq \varphi (\varphi^{-1}(v)^{-1}$. Choose an open $0$-neighborhood $\Omega$ such that $\ast$ is defined on $\Omega \times \Omega$ and $\Omega$ is symmetric (i.e.\ $v \in \Omega$ implies $v^{-1} \in \Omega$).

By construction, we have for all $(s,t)$ in a neighborhood of $(s_0,t_0)$ that $q \coloneq \varphi \circ p$ makes sense and takes values in $\Omega$. Note that also $q^{-1} = \varphi (p^{-1})$ takes values in $\Omega$ for all such $(s,t)$.
Employing now the rule on partial differentials, \Cref{prop:rpd}, shows that $X_p(s,t) = d_1 \ast \left(q(s,t),q^{-1}(s,t);\frac{\partial}{\partial t} q(s,t)\right)$. Specialising to $(s_0,t_0)$, we see that $X_p(s_0,t_0) = \frac{\partial}{\partial t} q(s_0,t_0)$. Similar identities hold for $Y_p$ by exchanging $t$ and $s$.
Compute the second derivative using the rule on partial differentials twice (where again the situation is symmetric in $s$ and $t$):
\begin{align*}
 &\frac{\partial}{\partial s} X_p(s,t) = \frac{\partial}{\partial s} d_1 \ast \left(q,q^{-1};\frac{\partial}{\partial t} q\right)\\ 
 =&d_1^2\ast \left(q,q^{-1};\frac{\partial}{\partial t} q,\frac{\partial}{\partial s} q\right) + d_2 (d_1\ast) \left(q,q^{-1};\frac{\partial}{\partial t}q , \frac{\partial}{\partial s} q^{-1}\right) + d_1 \ast \left(q,q^{-1};\frac{\partial^2}{\partial t \partial s} q\right)
\end{align*}
Due to Schwarz' rule, we see that the first and third term in the above formula are completely symmetric in $t$ and $s$. Hence these terms will not contribute to the difference \eqref{eq:comm:mixed}. Let us now compute the differential of the inverse:
\begin{align}
 \frac{\partial}{\partial s} q^{-1} &= \pr_2 T \varphi \left(\frac{\partial}{\partial s} p(s,t)^{-1}\right) \stackrel{\eqref{Tiota}}{=} -\pr_2 T \varphi T\lambda_{p(s,t)^{-1}}T\rho_{p(s,t)^{-1}}\left(\frac{\partial}{\partial s} p(s,t)\right) \label{inverseformula}
\end{align}
In particular,\eqref{inverseformula} reduces for $(s_0,t_0)$ to $Y_p (s_0,t_0)$. Likewise for $\tfrac{\partial}{\partial t} q^{-1} (s_0,t_0)$ we obtain $X_p (s_0,t_0)$. Note that by construction $q(s_0,t_0) = 0 = q^{-1}(s_0,t_0)$. Hence we can deduce now that the difference $\left(\frac{\partial}{\partial s} X_p -  \frac{\partial}{\partial t} Y_p\right)(s_0,t_0)$ is given as
\begin{align*}
 &\left(d_2 (d_1\ast) \left(q,q^{-1};\frac{\partial}{\partial t}q ,\frac{\partial}{\partial s} q^{-1}\right) - d_2 (d_1\ast) \left(q,q^{-1};\frac{\partial}{\partial s}q , \frac{\partial}{\partial t} q^{-1}\right)\right)(s_0,t_0)\\
 \stackrel{\eqref{inverseformula}}{=}& -\left(d_2 (d_1\ast) \left(0,0;X_p (s_0,t_0) , Y_p(s_0,t_0)\right) + d_2 (d_1\ast) \left(0,0;Y_p (s_0,t_0) ,X_p (s_0,t_0)\right)\right)\\
  \stackrel{\hphantom{\eqref{inverseformula}}}{=}&\left.\frac{\partial^2}{\partial s \partial t}\right|_{s,t=0} \left( tY_p(s_0,t_0) \ast sX_p(s_0,t_0) - sX_p(s_0,t_0) \ast tY_p(s_0,t_0) \right)\\
 \stackrel{\Cref{setup:locmult}}{=}& - \LB[X_p(s_0,t_0),Y_p(s_0,t_0)]
\end{align*}
\end{proof}

We are now in a position to establish Arnold's classical result on the Euler equation via geometry on the Lie group of volume-preserving diffeomorphisms:

\begin{thm}[Arnold]\label{Arnolds_theorem}
 Let $(M,g)$ be a compact Riemannian manifold and consider a curve $\varphi \colon [0,1] \rightarrow \Diff_\mu (M)$. Then $\varphi$ is a geodesic of the $L^2$-metric (i.e.\ the restriction of \eqref{generalL2} to $\Diff_\mu (M)$) if and only if 
 $$u \coloneq \dot{\varphi} \circ \varphi^{-1} \in \mathcal{V}_\mu (M)$$
 solves the Euler equation \eqref{Euler:eq}, i.e.\ for some function $p \colon [0,1] \times M \rightarrow \R$ we have 
 $$\partial_t u + \nabla_u u = - \grad p.$$
\end{thm}

\begin{proof}
 Let $\varphi (s,t)$ be a smooth variation of $\varphi$ and $u(s,t) \coloneq \frac{\partial}{\partial t}\varphi (s,t) \circ \varphi(s,t)^{-1}$, i.e.\ $u(t)=u(0,t)$. Set $h(t) = \left.\frac{\partial}{\partial s}\varphi (s,t) \circ \varphi(s,t)^{-1}\right|_{s=0}$ and note that by picking different smooth variations, $h(t)$ can be chosen to be an arbitrary curve in $\mathcal{V}_\mu (M)$ with vanishing end points. We take now the derivative of the energy and compute with \eqref{inv:Energy} 
 \begin{align}
  \left.\frac{d}{ds}\right|_{s=0}\mathrm{En} (\varphi(s,\cdot)) &\stackrel{\hphantom{\eqref{eq:comm:mixed}}}{=}  \int_0^1 \int_M g_m \left(u(t)(m),\left.\frac{\partial}{\partial s}\right|_{s=0} u(s,t)(m)\right) \mathrm{d}\mu (m) \mathrm{d}t \notag\\
  &\stackrel{\eqref{eq:comm:mixed}}{=}  \int_0^1 g^{L^2}\left(u(t),\frac{d}{dt} h(t)-\LB[u(t),h(t)]\right) \mathrm{d}t \label{eq:energy:variation}
\end{align}
Recall that the Lie bracket in \eqref{eq:energy:variation} is the Lie bracket of $\mathcal{V}_\mu (M)$. As this is a subalgebra of the Lie algebra of $\Diff (M)$, \Cref{ex:Liealgdiffeo} implies that it is the negative of the commutator bracket of vector fields, whence $\nabla_u h - \nabla_h u = -\LB[u,h]$. Replace now the Lie bracket and apply Exercise \ref{ex:classical} 5. to $g(u,-\nabla_h u)$ to see that \eqref{eq:energy:variation} yields: 
\begin{align*}
&\int_0^1 g^{L^2} \left(u(t),\frac{d}{dt} h(t)+\nabla_{u(t)}h(t)-\nabla_{h(t)} u(t) \right) \mathrm{d}t\\
 =&\int_0^1 \left(g^{L^2} \left(u(t),\frac{d}{dt} h(t)+\nabla_{u(t)} h(t)\right) \underbrace{-\frac{1}{2} g^{L^2}(\grad g(u(t),u(t)),h(t))}_{=0 \text{ by \Cref{Helmholtz} as $h(t)\in \mathcal{V}_\mu (M)$}}\right) \mathrm{d}t\\
 =&\int_0^1 g^{L^2} \left(u(t),\frac{d}{dt} h(t)+\nabla_{u(t)} h(t))\right)\mathrm{d}t
\end{align*}
We continue with integration by parts with respect to $t$ and the identity $g(u,\nabla_u h) = g(u,\grad g(u,h))-g(\nabla_u u ,h)$ (see Exercise \ref{ex:classical} 5.). Together with the Helmholtz decomposition \Cref{Helmholtz} the above equation is equal to
\begin{align*}
 -\int_0^1 g^{L^2} \left(\frac{d}{dt}u(t) +\nabla_{u(t)} u(t), h(t)\right) \mathrm{d}t
\end{align*}
Hence if $\varphi$ extremises the energy, we see that $-(\partial_tu + \nabla_u u)$ must be $L^2$-orthogonal to every curve (with vanishing end points) in $\mathcal{V}_\mu(M)$. By the Helmholtz decomposition, this happens if and only if there is some function $p$ (determined up to a constant) such that $\partial_t u + \nabla_u u = -\grad p$.  
\end{proof}

We have seen now that the geometry of the group of volume-preserving diffeomorphisms can be exploited to identify the Euler equation as a geodesic equation. This connection is typical for PDE and their associated geodesic equations which are amenable to Arnold's approach. Indeed there is one last reformulation of the Euler equation on the Lie group $\Diff_\mu (M)$ which needs to be mentioned here as it exhibits the connection between invariant Riemannian metric and Lie group more explicitly.

\begin{rem}[The Euler equation as an Euler-Poincar\'{e} equation on $\Diff_\mu (M)$]
 Our aim is to identify the geodesic equation as a so called \emph{Euler-Poincar\'{e} equation} on $\Diff_\mu (M)$. For this, let us start more general: Let $G$ be a regular Lie group with Lie algebra $(\Lf (G),\LB )$ and $\langle \cdot ,\cdot \rangle$ a continuous inner product on $\Lf (G)$. Assume that we with to compute geodesics for the right-invariant metric induced by the choice of inner product. Arguing as for the Euler equation, we see that a curve $\varpi \colon [0,1] \rightarrow G$ extremises the energy if and only if if the expression \eqref{eq:energy:variation} vanishes, i.e.\ in the notation of Exercise \ref{Ex:LA} 11. we must have
 \begin{align*}
  0= \int_0^1 \left\langle u(t),\frac{d}{dt} h(t)-\LB[u(t),h(t)]\right\rangle\mathrm{d}t =  \int_0^1\left\langle u(t),\frac{d}{dt} h(t)-\text{ad}_{u(t)}(h(t))\right\rangle \mathrm{d}t,
 \end{align*}
 where $u(t) = \delta^r \varphi (t)$ (the right logarithmic derivative of $\varphi$). Assume now that for all $x \in \Lf (G)$ there exists an adjoint $\text{ad}_x^{\top}$ for the linear operator $\text{ad}_x$ with respect to the inner product $\langle \cdot, \cdot \rangle$, i.e. $\langle \text{ad}_x^\top (y),z\rangle = \langle y,\text{ad}_x (z)\rangle$. Applying again integration by parts we see that $\varphi$ is a geodesic if and only if its right logarithmic derivative satisfies the \emph{Euler-Poincar\'{e} equation}\footnote{The name goes back to honour Henri Poincar\'{e} who formulated in \cite{poin01} differential equations for mechanical systems on (finite dimensional) Lie groups in the presented form.} \index{Euler-Poincar\'{e} equation}
 $$\frac{d}{dt} \delta^r \varphi = -\text{ad}^\top_{\delta^r \varphi} (\delta^r \varphi).$$
 Thus we have derived yet another expression which is equivalent to the geodesic equation and by the previous results also to the Euler equation of an incompressible fluid if $G = \Diff_\mu (M)$ and the inner product is the $L^2$-inner product. Observe that the Euler-Poincar\'{e} equation is a differential equation on the Lie algebra. The Euler-Poincar\'{e} equation reduces the geodesic equation to the Lie algebra and shows that the geometry of the Lie group and the Riemannian geometry of a right-invariant metric are closely intertwined. We will not discuss this fruitful perspective on the Euler equation. However, there are accounts of the general mechanism with many examples available in the literature. The interested reader is referred for example to \cite{Viz08,modin2019geometric} and \cite[II.3]{MR2456522}.
\end{rem}
\newpage

\begin{Exercise}\vspace{-\baselineskip} \label{Ex:EA:Euler}
 \Question Let $(M,g)$ be a compact Riemannian manifold. Show that the $L^2$-metric \eqref{generalL2} on $\Diff (M)$ restricts to a right-invariant Riemannian metric on $\Diff_\mu (M)$.
 \Question We let again $(M,g)$ be a compact Riemannian manifold and consider the $L^2$-metric $g^{L^2}$ \eqref{generalL2} on $\Diff (M)$. Let $S$ be the metric spray of $g$ and $K$ the associated connector. The aim of this exercise is to prove that $g^{L^2}$ admits a metric spray, connector and covariant derivative by exploiting the right invariance of the metric.\\
 {\footnotesize \textbf{Remark:} The proof for this statement for the $L^2$-metric on $C^\infty (\SSS^1,M)$ from \Cref{sect:L2} can be adapted to the present situation. However, we will follow in this exercise the classical argument of \cite{EM70} which highlights the use of invariance properties.}
 \subQuestion Show that the pushforward $S_*$ is a spray on $\Diff (M)$ and the pushforward $K_*$ is a connector on $\Diff(M)$.
 \subQuestion Define the covariant derivative $\nabla^{L^2}_X Y = K_* \circ TY \circ X$ associated to the connector $K$ (i.e.\ $X,Y$ are vector fields on $\Diff (M)$).
              Work out $\nabla^{L^2}_X Y$ for right-invariant vector fields on $\Diff (M)$ (i.e.\ vector fields $X(\varphi) = X(\id) \circ \varphi$, where $X(\id) \in \mathcal{V} (M) = \Lf (\Diff (M)$). Verify then that $\nabla^{L^2}$ satisfies the properties of the metric derivative associated to $g^{L^2}$ for all right-invariant vector fields.
 \subQuestion Establish that $\nabla^{L^2}$ is the metric derivative of $g^{L^2}$ and deduce that $K_*$ is the connector and $S_*$ the metric spray associated to $g^{L^2}$.\\
 {\footnotesize \textbf{Hint:} Exploit that for every vector field $X \in \mathcal{V} (\Diff (M))$ and $\varphi \in \Diff (M)$ there exists a right-invariant vector field $X^R$ such that $X^R (\varphi) = X(\varphi)$.} 
 \Question Supply the necessary details for the proof of \Cref{lem:commutator:smoothvar}. Show in particular, that $\pr_2 \left(T\rho_{\tilde{p}(s,t)^{-1}} \left(\frac{\partial}{\partial s} \tilde{p}(s,t)\right)\right) = X_p(s,t)$, $X_p = d_1\ast (q,q^{-1};\frac{\partial}{\partial s} q)$ and $X_p(s_0,t_0) = \frac{\partial}{\partial s} q(s_0,t_0)$ and \eqref{inverseformula} reduces to $X_p (s_0,t_0)$.
\end{Exercise}

\setboolean{firstanswerofthechapter}{true}
\begin{Answer}[number={\ref{Ex:EA:Euler} 1.}] 
 \emph{Let $(M,g)$ be a Riemannian metric with volume form $\mu$. We show that the $L^2$-metric is a right-invariant Riemannian metric on $\Diff_\mu (M)$
} \\[.5em]

We have already seen (for $M = \SSS^1$ but the general case is similar) that the $L^2$-inner product $\langle X,Y\rangle_{L^2} \coloneq \int_M g_m (X(m),Y(m))\mathrm{d}\mu (m)$ is a continuous inner product on $\mathcal{V} (M)$.
From the formula for the tangent of the right multiplication in $\Diff (M)$ (cf.\ \Cref{thetamp}), we see that the right-invariant metric induced by the $L^2$-inner product is given by
\begin{align*}
 \langle X \circ \varphi , Y \circ \varphi\rangle_\varphi = \langle X \circ \varphi \circ \varphi^{-1}, Y\circ \varphi \circ \varphi^{-1}\rangle_{L^2} = \langle X , Y \rangle_{L^2}
\end{align*}
Let us assume now that $\varphi$ is a volume-preserving diffeomorphism. Then by diffeomorphism invariance of the integral (see e.g.\ \cite[Proposition 16.6]{Lee13}) we derive now that 
\begin{align*}
  \langle X \circ \varphi , Y \circ \varphi\rangle_\varphi= \langle X , Y \rangle_{L^2} = \int_M \varphi^* g(X,Y) \mathrm{d}\mu = \int_M g_{\varphi(m)}(X\circ \varphi (m),Y\circ \varphi(m))\mathrm{d}\mu(m)
\end{align*}
Thus for every volume-preserving diffeomorphism, the $L^2$-metric coincides with the right-invariant Riemannian metric induced by the $L^2$-inner product on the vector fields. We conclude that the $L^2$-metric is right-invariant on the subgroup $\Diff_\mu (M)$.   
\end{Answer}
\setboolean{firstanswerofthechapter}{false}

\begin{Answer}[number={\ref{Ex:EA:Euler} 1.}] 
 \emph{For $(M,g)$ a compact Riemannian manifold consider the $L^2$-metric $g^{L^2}$ \eqref{generalL2} on $\Diff (M)$. Let $S$ be the metric spray of $g$ and $K$ the associated connector. (a) Then $S_*$ and $K_*$ define a spray and a connector on $\Diff (M)$. (b)-(c) $\nabla^{L^2}_XY \coloneq K_* \circ TY \circ X$ defines the metric derivative on $g^{L^2}$.} \\[.5em]
 
 (a) As $\Diff (M) \opn C^\infty (M,M)$ and $C^\infty (M,M)$ is a canonical manifold, the pushforwards are smooth. That they from a spray and a connector follows directly from the identification of $T^k\Diff (M) \opn T^kC^\infty (M,M) \cong C^\infty (M,T^kM)$.\\
 (b)-(c) Details for this proof are recorded in \cite[Proof of Theorem 9.1]{EM70}.
\end{Answer}

\section{An outlook on Euler-Arnold theory}

In the last section we have seen how the Euler equation of an incompressible fluid can be identified as a geodesic equation of a right-invariant metric on an infinite-dimensional Lie group.
This equation can in turn be rephrased as the Euler-Poincar\'{e} equation on the Lie algebra, highlighting the close connection of Lie group and Riemannian geometry. 
In the present section we will discuss several applications of the theory developed so far. The section concludes with a discussion of the problems one faces to carry out the program sketched.

\begin{setup}[Applications of Euler-Arnold theory]\label{EA:theory:application}
 As we have seen, certain PDE, such as the Euler equation of an incompressible fluid, can be rewritten as geodesic equations on an infinite-dimensional manifold. Moreover, the mechanism is reversible, i.e.\ solutions to the original PDE correspond to geodesics. Vice versa solutions to the geodesic equation of certain (weak) Riemannian metrics yield solutions to partial differential equations. This has the following immediate applications.
 \begin{enumerate}
  \item \textbf{Existence, uniqueness and parameter dependence of solutions}. Geodesics are solution to ordinary differential equations. To establish properties of their solutions (such as local existence of unique solutions) one can hope to apply the usual toolbox for ordinary differential equations. Transporting solutions back to the finite-dimensional world, this will yield local existence and uniqueness for solutions of the PDE. Historically, this was how the existence and uniqueness problem for Eulers equations of an incompressible fluid was first solved in the general case in \cite{EM70}.\footnote{The emphasis is here on general. Some results were known for special cases previous to the treatment in loc.cit.} The caveat here is that ODE tools break down beyond Banach manifolds and one requires a technical analysis to make the program work (see \Cref{EA:theory:problems} below). 
\end{enumerate}
\end{setup}

Before we continue let us recall the concept of sectional curvature.
\begin{setup}\label{sectionalcurvature}
 Let $(M,g)$ be a (weak) Riemannian manifold with covariant derivative $\nabla$ and curvature $R$. Then for two linearly independent vectors $u,v \in T_m M$ spanning a $2$-dimensional subspace $\sigma$ the \emph{sectional curvature}\index{curvature!sectional} is defined as  
  $$K(\sigma) \coloneq \frac{g_x(R(v,w)w,v)}{g_x(v,v)g_x(w,w)-(g_x(v,w))^2}$$
  (where we actually evaluate the curvature $R$ in (local) vector fields $V,W$ with $V(x)=v$ and $W(x)=w$). 
  
As a concrete example, endow the diffeomorphism group $\Diff (M)$ with the weak $L^2$-metric. Every vector $V \in T_\varphi \Diff (M)$ can be expressed as the value of a right-invariant vector field $X_V$ on $\Diff (M)$. Hence it suffices to compute the sectional curvature using right-invariant vector fields. The formula for the covariant derivative of the $L^2$-metric then shows that for a subspace $\sigma$ generated by the orthonormal elements $\{X_V(\varphi) , Y_V(\varphi)\}$ the sectional curvature is given as 
$$K^{L^2} (\sigma) = g^{L^2} (R^{L^2}(X_V,Y_V)Y_V,X_V) = \int_M K(X_V (\id)(m),Y_V(\id)(m)) \mathrm{d}\varphi^*(\mu)(m),$$
where $K$ is the sectional curvature on $M$ (computed with respect to the space spanned by the vectors $X_V(\id)(m)$ and $Y_V(\id)(m)$, \cite[6.4]{Smo07}).
\end{setup}
Continuing with our review of Euler-Arnold theory from \Cref{EA:theory:application}:
\begin{enumerate}
  \item[(b)] \textbf{Geometric tools for PDE analysis}. It is well known from finite-dimensional Riemannian geometry that curvature controls the behaviour of geodesics (see e.g.\ \cite[Chapter 5]{doC92} and note also the connection to the Hopf-Rinow theorem in \Cref{rem:HopfRinow_seminegative}). The point is that for positive sectional curvature, geodesics starting at the same point with slight variation of the initial velocity tend to converge towards each other while for negative sectional curvature they diverge (this can be made explicit as in the finite-dimensional case, but we will not discuss the details here). In the context of partial differential equations these properties can be interpreted as stability of solutions under perturbations of initial conditions. 
  
  Indeed Arnold showed in \cite{MR202082} that the sectional curvature of the $L^2$-metric on $\Diff_\mu (M)$ is negative in almost all directions. So nearby fluid regions will typically diverge exponentially fast from each other. This analysis applies in particular to partial differential equations employed in weather forecasts. So infinite-dimensional Riemannian geometry shows that reliable long term weather forecasts are practically impossible.
 \end{enumerate}

\begin{rem}
 There is also a beautiful connection of the Euler-Arnold equations to ideas from Hamiltonian mechanics on $\Diff (M)$. The differential geometric context for this are (weakly) symplectic structures on $\Diff (M)$ and we refer to \cite[Section 6 and 7]{Smo07} for a discussion. 
\end{rem}

 In the present chapter we have only seen the mechanism applied to the Burgers and Eulers equation. There are many more PDE which typically arise in hydrodynamics and can be treated in the same framework. Equations which are amenable to this treatment are nowadays called \emph{Euler-Arnold equations}.\index{Euler-Arnold equation} We refer to \cite{MR2456522} for an extensive list but mention explicitly the Camassa-Holm equation, the Hunter-Saxton and the Korteweg-deVries (KdV) equation as PDE belonging to this class. 
Since the Hunter-Saxton equation admits a beautiful geometric interpretation, we will now briefly discuss a few more details related to this equation and its geometric treatment.

\begin{ex}[The Hunter-Saxton equation]
 We consider the \emph{periodic Hunter-Saxton equation}\index{Hunter-Saxton equation} on the circle $\SSS^1$. The task is to find a time-dependent vector field $u \colon [0,T[ \times \SSS^1 \rightarrow \R$ which satisfies the following equation
 \begin{equation}\label{HS:equation}
  u_{t xx}+2u_xu_{xx}+uu_{xxx}=0,
 \end{equation}
 where subscripts denote again partial derivatives with respect to time ($t$) or $x\in \SSS^1$. We will see in Exercise \ref{Ex:EA:general} 1 that the Hunter-Saxton equation is the geodesic equation of the right-invariant $\dot{H}^1$-semimetric.\index{Riemannian metric!$\dot{H}^1$-semimetric} To describe this semimetric, we recall from \Cref{sect:L2} the notation $U^\prime \coloneq T_\theta U (1)$ and define the inner product   
 \begin{equation}\label{H1-semimetric}
  g^{\dot{H}^1}_{\id}(U ,V) = \frac{1}{4}\int_{\SSS^1} U^\prime(\theta) V^\prime(\theta) \mathrm{d} \theta \qquad \text{on } C^\infty (\SSS^1,\R).
 \end{equation}
 Then the $\dot{H}^1$-semimetric is the right-invariant semimetric induced by \eqref{H1-semimetric}. Note that it is a semimetric as constant vector fields are annihilated. In \Cref{ex:badexponential} we saw that the constant vector fields generate the group of rotations $\text{Rot} (\SSS^1)$. Hence there are two possibilities to obtain a (weak) Riemannian metric: One can work with the quotient manifold $\Diff(\SSS^1)/\text{Rot} (\SSS^1)$ (see \cite{Len08} for a detailed discussion) or one has to fix a subgroup containing only the trivial rotation. We consider the induced weak Riemannian metric on the Lie subgroup
 \begin{align}\label{fixinggroup}
  D_0 \coloneq \{\phi \in \Diff (\SSS^1) \mid \phi (\theta(0)) = \theta(0)\},
 \end{align}
 where $\theta \colon [0,2\pi] \rightarrow \R^2, t \mapsto (\cos (t),\sin(t)) $ is the canonical parametrisation of the circle. To ease the computations we follow Lenells \cite{Len08} and will in the following always identify diffeomorphisms and vector fields of $\SSS^1$ as periodic mappings $[0,2\pi] \rightarrow \R$. This allows one to prove that the Hunter-Saxton equation exhibits a fascinating geometric feature discovered in \cite{Len07}: The group $D_0$ can be identified as a convex subset of a sphere and this embedding is a Riemannian isometry relating the $\dot{H}^1$-metric to the $L^2$-metric. Indeed this embedding is surprisingly simple, as it is given by
 $$\Psi \colon D_0 \rightarrow C^\infty (\SSS^1,\R) ,\quad \Psi (\varphi) = \sqrt{\varphi^\prime}.$$
 Its image becomes the convex set $\{f(\theta)>0,\ \forall \theta \in \SSS^1, \int_{\SSS^1} |f(\theta)|^2 \mathrm{d}\theta = 1\} \subseteq C^\infty (\SSS^1,\R)$. We omit the details here and refer instead to the exposition in \cite{Len07}. 

 The geometric content of this observation is that we can transform the Hunter-Saxton equation to a geodesic equation on the $L^2$-sphere in $C^\infty (\SSS^1,\R)$. Now solutions to the geodesic equation on the $L^2$-sphere (with respect to the natural $L^2$-metric) can explicitly be computed: We have seen in \Cref{GmEllipsoid} that geodesics of the $L^2$-sphere (seen as the unit sphere of the Hilbert space $L^2 (\SSS^1,\R)$) are given by great circles.  
\end{ex}

While the Hunter-Saxton equation can be interpreted as the geodesic equation on an infinite-dimensional sphere in a Hilbert space, our approach so far has been to consider this equation as an equation on the space $C^\infty (\SSS^1,\R)$ which is not a Hilbert space. This is quite unnatural for two reasons: From the perspective of the PDE this approach will only allow solutions which are smooth in space, whereas it is often of interest to have much less regular solutions, e.g.\ solutions which are only finitely often differentiable in space. The geometric perspective allows us to connect the PDE to an ODE which we then have to solve. However, on $\Diff (M)$ this presents a problem as we will discuss now.

\begin{setup}[Returning to the Banach and Hilbert setting]\label{EA:theory:problems}
 In \Cref{EA:theory:application} we listed as an advantage of the Euler-Arnold approach to PDEs that (local) existence and uniqueness of solutions to these PDE can be obtained by methods for ordinary differential equations (ODE) on infinite-dimensional manifolds. Unfortunately the manifolds we have been working in this chapter are submanifolds of the manifolds of mappings $C^\infty (K,M)$ (where $K$ is a compact manifold. These manifolds are never (except in trivial cases) Banach manifolds, so there are no black-box techniques for ODE as \Cref{Diffeq:beyond} shows. Thus the elegant theory developed so far misses an essential analytic ingredient to solve the ODEs occurring.
 
 The solution to this problem is in principle simple (if one glosses over the technical details): Replace $C^\infty$ functions by finitely often differentiable ones. Note that this dovetails nicely with the problem statement from the PDE side we mentioned earlier. By going to finitely often differentiable functions, we allow solutions to the PDE which are much less regular in space. From the perspective of infinite-dimensional manifolds, one can prove that  
 $$C^\infty (K,M) = \bigcap_{k \in \N_0} C^k (K,M) = \bigcap_{s \in \N_0} H^s (K,M).$$
 where the manifolds in the middle are Banach manifolds and the spaces on the right even Hilbert manifolds. Here $H^s(K,M)$ denotes all mappings of Sobolev $H^s$-type (meaning that their weak derivatives are in $L^2$). We refrain from defining Sobolev spaces on manifolds as there are several subtle points involved in their construction. Instead we remark that they admit a manifold structure similar to the manifolds of mappings we constructed in \Cref{sect:smoothmappingspaces}. For more information we refer the reader to the detailed exposition in \cite{IaKaT13}. 
 
 Conveniently, also Sobolev type groups of diffeomorphisms $\Diff^{H^s}(K)$ exist and one can even prove that $\Diff (K) = \displaystyle \lim_{s \rightarrow \infty} \Diff^{H^s}(K)$ as a projective limit in the category of manifolds. This structure is called an ILH-Lie group, see \cite{Omo74}, and coincides with the Lie group structure of $\Diff (K)$ constructed in \Cref{ex:Diffgp}. The point of the construction is of course that one can work on the Hilbert manifold of Sobolev morphisms. Unfortunately, due to Omori's theorem \cite{MR0579603},\index{Omori's theorem} the groups $\Diff^{H^s} (K)$ can not be Lie groups. They are however manifolds and topological groups such that the left multiplication is continuous but not differentiable (cf.\ \eqref{eq:commdiag} to see that the derivative looses orders of differentiability). However, right multiplication is still smooth and one obtains a so called \emph{half-Lie group},\index{half-Lie group} \cite{MaN18}. This leads to several analytic problems which need to be solved to establish smoothness of the associated metric sprays (cf.\ e.g.\ \cite{EM70}). 
 We refer the reader to the literature for more details as these problems are beyond the scope of this chapter.
\end{setup}

 At this point, the reader should be suitably equipped to understand the classical research literature on these topics. For example \cite{EM70,Ebin15} present the theory for the Euler equation of an incompressible fluid. Also the monograph \cite{MR2456522} presents an excellent overview of the theory together with many pointers towards the literature. We conclude this chapter with a short remark on some more recent developments in Euler-Arnold theory.
 
 \begin{rem}
  Euler-Arnold theory is still an active area of research. Among the many recent results, I like to point out several which I find particularly interesting:
  \begin{enumerate}
   \item Classical Euler-Arnold theory works with right-invariant Riemannian metrics on (subgroups of) diffeomorphism groups. In \cite{BaM20} it has been shown that the approach works also for Riemannian metrics on $\Diff(M)$ which are only invariant with respect to $\Diff_\mu (M)$. Thus a whole new family of PDE, such as certain shallow water equations, can be treated by Euler-Arnold methods.
   \item Instead of a purely deterministic PDE one can apply the mechanism to a stochastic partial differential equation. Stochastic versions of Eulers equations for an incompressible fluid have recently been considered as models for data driven hydrodynamics (modelling uncertainty in the data). For the Euler equation of an incompressible fluid, \cite{MaMaS19} works out the necessary details to make the Euler-Arnold machinery work in the stochastic setting.
   \item There is a connection between Euler-Arnold theory, the differential geometry of diffeomorphism groups and optimal mass transport. This is based on the observation that there is a submersion $\pi \colon \Diff (M) \rightarrow \text{Dens}(M)$ from the diffeomorphism group to the manifold of densities on $M$. This links the Riemannian metrics on $\Diff(M)$ to the Wassertein metric from optimal transport. An introduction to these topics can be found in \cite[Appendix A.5]{MR2456522}.
  \end{enumerate}

 \end{rem}

\begin{Exercise}\vspace{-\baselineskip} \label{Ex:EA:general}
 \Question We consider the group $\Diff (\SSS^1)$ and the $\dot{H}^1$-semimetric \eqref{H1-semimetric}. Show that
 \subQuestion identifying diffeomorphisms with periodic mappings we can identify $D_0$ with $\{u+\id \mid u \colon [0,2\pi] \rightarrow \R,\quad \frac{d}{dx} u >-1, u(0)=0=u(2\pi)\}$.
 \subQuestion $E = \{u \in C^\infty ([0,2\pi],\R) \mid u(0)=0\}$ is a closed subspace of $C^\infty ([0,2\pi],\R)$ and $D_0$ is diffeomorphic to an open subset of $\id +E \subseteq E$.
 \subQuestion the continuous linear operator $A (u)\coloneq -u''$ induces an isomorphism from $E$ to $F \coloneq \{f \in C^\infty ([0,2\pi],\R) \mid \int_0^{2\pi} f(x) \mathrm{d}x =0\}$. Moreover, show that (up to identification) $g^{\dot{H}^1}_{\id} (U,V) = \int_{\SSS^1} UA(V)\mathrm{d}\theta$.
 \subQuestion the group $D_0$ from \eqref{fixinggroup} is a Lie subgroup of $\Diff (\SSS^1)$ and its Lie algebra can be identified as $\Lf (D_0) = \{f \in C^\infty (\SSS^1,\R) \mid \int_{\SSS^1} f(\theta) \mathrm{d}\theta =0\}$.
 \subQuestion \eqref{H1-semimetric} induces a right-invariant weak Riemannian metric on $D_0$.
 \subQuestion a curve $\varphi$ extremises the energy of the $\dot{H}^1$-semimetric \eqref{H1-semimetric} if and only if $u = \frac{\partial}{\partial t} \varphi \circ \varphi^{-1}$ satisfies the Hunter-Saxton equation \eqref{HS:equation}. 
 \Question Prove the claims on the sectional curvature of the $L^2$-metric on $\Diff (M)$ from \Cref{sectionalcurvature}.
\end{Exercise}

\chapter{The geometry of rough paths} \label{sect:rough} \copyrightnotice

In this chapter we will discuss the (infinite-dimensional) geometric framework for rough paths and their signature. Rough path theory originated in the 90's with the work of T.\ Lyons, see e.g.\ \cite{Lyo98}. 
It seeks to establish a theory of integrals and differential equations driven by rough signals. For example, one is interested in the controlled ordinary differential equations of the following type 
\begin{align}\label{cont:ODE}
 y_t' = f (y_t) + g(y_t) X_t'.
\end{align}
Here the subscripts track the time parameter $t$, $X$ is an input path with values in $\mathbb{R}^d$ and $y$ is the output with values in $\mathbb{R}^e$. We will for this exposition assume that all derivatives and integrals needed exist (e.g.\ if $X$ is a smooth path). Finally, $f,g$ are non-linear functions with values in $\R^e$ and $\mathcal{L}(\R^d,\R^e)$, respectively. Focussing on the control term in \eqref{cont:ODE}, let us consider a simple approximation to the solution in the case that $f=0$, i.e.\ $y_t' = g(y_t)X_t'$. For example the first order Euler method gives us the following approximation for the components of $y$:
\begin{align*}
 y_t^i - y_s^i \approx g^i(y_s) \int_s^t dX^i = \lim_{|P|\rightarrow 0} \sum_{[t_j,t_{j+1}] \in P} \left(X^i_{t_{j+1}}-X^i_{t_j}\right), \quad i=1,\ldots ,e. 
\end{align*}
Here superscripts denote components of maps, the integral is defined via a Riemann-Stieltjes sum\index{Riemann-Stieltjes sum} and we think of $g(y_s)$ as a matrix (selecting columns and rows appropriately). The information needed for the approximation is the integral of $X$. In general we would like better approximations (or even a solution), so it is natural to increase the order of the approximations. To obtain the desired formula one applies a Taylor expansion to obtain the following $2$nd order Euler approximation for $i = 1,\ldots ,e$
\begin{align}\label{2step-Euler}
 y_t^i -y_s^i \approx g^i(y_s) \int_s^t dX^i + \sum_{k = 1}^e\sum_{j=1}^d g_k^i(y_s) \frac{\partial}{\partial k} g^j(y_s) \int_s^t \int_s^r \mathrm{d}X^i\mathrm{d}X^j.
\end{align}
Thus higher-order approximation requires knowledge about (mixed) iterated integrals of the path $X$ against itself. So far we have tacitly assumed that the necessary paths and derivatives exist (e.g.\ that all the objects in question are smooth). 

Weakening this requirement, the objects we are interested in are iterated integrals of H\"{o}lder continuous paths with values in $\R^d$.\footnote{The basic theory extends to Banach space valued paths but this requires more technical efforts such as, for example, the introduction of tensor norms.}  
For convenience we will assume throughout this chapter that we are dealing with paths on the interval $[0,1]$. This is no restriction since we can always reparametrise H\"{o}lder paths on any $[0,1]$ to obtain corresponding paths on $[0,1]$ (though this changes the H\"{o}lder norm). Assume we have two continuous mappings $X \colon [0,1] \rightarrow \R^d$ and $Y \colon [0,1] \rightarrow \mathcal{L}(\R^d,\R^e)$, where the continuous linear maps $\mathcal{L}(\R^d,\R^e)$ have been endowed with the operator norm. We would like to define an integral now of $Y$ against the path $X$ and these integrals should yield a continuous map $(X,Y) \mapsto \int Y \mathrm{d}X$. 
Setting $X_{s,t} \coloneq X_t-X_s$ we could try to define the integral as a limit of Riemann-Stieltjes sums,\index{Riemann-Stieltjes sum}
$$\int_0^1 Y(t)\mathrm{d}X(t) \coloneq \lim_{|\mathcal{P}| \rightarrow 0} \sum_{[s,t] \in \mathcal{P}} Y(s)X_{s,t},$$
where $\mathcal{P}$ is a partition of $[0,1]$ and the limit takes the mesh size to $0$. The resulting integral is called the \emph{Young integral}\index{Young integral} and Young has shown in \cite{Young} that the Riemann-Stieltjes sum converges if $X$ is an $\alpha$-H\"{o}lder path\footnote{The notion of $\alpha$-H\"{o}lder paths is recalled in \Cref{sec:roughintro} below.} and $Y$ is a $\beta$-H\"{o}lder path such that $\alpha + \beta > 1$. This result is sharp as one can construct examples of paths such that the sum becomes ill defined. 
So in general, there is no hope for an integration theory which allows us to integrate arbitrarily rough paths against each other. However, the key insight of rough path theory is that the regularity assumption $\alpha + \beta >1$ from Young's theorem can be circumvented if additional structure is added to the paths. Thus a rough integral can be built if we enhance H\"{o}lder continuous paths with additional information to so called rough paths. The point is that this information can be chosen for irregular paths such as Brownian motion (which is known to be $\alpha$-H\"o{}lder for $\alpha \in ]0,1/2[$). 
Here rough path theory excels at clever estimates for the integrals appearing. In the present chapter we will focus on the geometric side of the picture and leave the hard analytic estimates to the rough path literature. Our first aim is to develop a convenient geometric framework to record iterated integrals.

\begin{tcolorbox}[colback=white,colframe=blue!75!black,title=Notation for increments and iterated integrals]
We frequently encounter increments of continuous paths $X \colon [0,1] \rightarrow \R^d$:
\begin{itemize}
 \item $X_t \coloneq X(t)$ and 
 \item $X_{s,t} \coloneq X_t-X_s$ for the increment and $t,s \in [0,1]$.
 \end{itemize}
 With enough differentiability of $X$, we can define iterated integrals of $X$ against itself. For $0\leq s \leq t\leq 1$ we integrate componentwise in $\R^d \otimes \R^d \cong \R^{d^2}$ and set
 $$\int_s^t \int_s^r \mathrm{d}X \otimes \mathrm{d}X \coloneq \int_s^t X_{s,r} \otimes \mathrm{d}X \coloneq \int_s^t X_{s,r} \otimes \mathrm{d}X_r \coloneq \int_s^t X_{s,r} \otimes X_r' \mathrm{d}r.$$
 The notation suppresses indices with the understanding that the objects are matrices whose components are iterated integrals. For higher iterated integrals this will quickly become impractical whence we shall use tensor notation instead.
\end{tcolorbox}

\section{Iterated integrals and the tensor algebra}

In this section we consider the tensor algebra as a continuous inverse algebra. As seen in the introduction to this chapter, we are interested in iterated integrals of the components of a path $X \colon [0,1] \rightarrow \R^d$. For example, the second iterated integrals \eqref{2step-Euler} yield a matrix object whose components we can also conveniently be recorded using tensor notation:
$$\int_s^t \int_s^r \mathrm{d}X^i \otimes \mathrm{d}X^j \coloneq \int_s^t \left(\int_s^r \mathrm{d}X^i\right) \mathrm{d}X^j e_i\otimes e_j,$$
where $e_i,e_j$ are standard basis vectors of $\R^d$. Iterating for higher orders, we can write the resulting integrals as elements in an iterated tensor product. Note that on the basic levels tensors are just a bookkeeping device to track the components integrated against each other. Indeed the canonical identification $\R^d \otimes \R^d \cong \R^{d\times d}$ maps $e_i \otimes e_j$ to the component in $i$th row and $j$th column. Hence, detailed knowledge on tensor products is not needed and we refer to \cite[Chapter 5]{MR960687} for an introduction. 

\begin{defn}[(truncated) Tensor algebra]
 For $d\in \N$ and $k \in \N$ we set 
 $$(\R^d)^{\otimes k} \coloneq \underbrace{\R^d \otimes \R^d \otimes \cdots \otimes \R^d}_{k \text{times}} \text{ and } (\R^d)^{\otimes 0}\coloneq\R.$$
 Write $e_i, 1\leq i \leq d$ for the standard basis vectors of $\R^d$ and recall that the products $e_{i_1} \otimes \cdots \otimes e_{i_k}$ form a base for $(\R^{d})^{\otimes_k}$ whence this space is isomorphic to $\R^{d^k}$. 
 Then we define for $N \in \N \cup \{\infty\}$ the \emph{(truncated) tensor algebra}\footnote{The tensor algebra discussed in this section differed from what is usually called the tensor algebra. In the literature, the tensor algebra usually denotes the direct sum of iterated tensor products (i.e.\ finite sequences of iterated tensors). From this perspective, the tensor algebra $\mathcal{T}^\infty (\R^d)$ should rather be called the completed   tensor algebra.}\index{tensor algebra}
 $$\mathcal{T}^N(\R^d) \coloneq \prod_{k=0}^N (\R^d)^{\otimes k}.$$
 Elements in $\mathcal{T}^N(\R^d)$ will be denoted as sequences $(x_k)_{k < N+1}$ (where $\infty+1=\infty$). An element concentrated in the $k$th factor $(\R^d)^{\otimes k}$ is called \emph{homogeneous of degree} $k$.
 Then the algebra product is given by 
 \begin{align}\label{algebraprod}
  (x_k)_{k < N+1} \otimes (y_k)_{k < N+1} \coloneq \left(\sum_{n+m=k} x_n \otimes y_m\right)_{k <N+1}.
 \end{align}
 Since tensor products over $\R$ with elements of $\R$ contract, i.e.\ $\lambda \otimes v = \lambda v$, we see that the tensor algebra has a unit element $\one = (1,0,0,\ldots) \in \mathcal{T}^N(\R^d)$ (i.e.\ $\one$ is homogeneous of degree $0$) and the map $\pi_0^N \colon \mathcal{T}^N (\R^d) \rightarrow (\R^d)^{\otimes 0}$ is an algebra morphism (and in particular the homogeneous elements of degree $0$ form a subalgebra of the (truncated) tensor algebra). Note however, that the algebra is almost always non-commutative (Exercise \ref{ex:tensor} 2.).  
 \end{defn}
 
 In the introduction we have seen that iterated integrals can be conveniently identified with elements in the tensor algebra. From this point of view it would be enough to treat the tensor algebra as locally convex space which simply stores information. However, it turns out that iterated integrals satisfy several natural identities which can be expressed using the product in the tensor algebra. For example, if $X \colon [0,1] \rightarrow \R^d$ is a smooth path and $\mathbf{X}_{s,t} \coloneq \int_s^t X_{s,r}\otimes \mathrm{d}X \in \R^d \otimes \R^d$ its iterated integral, it is easy to see that these satisfy \emph{Chen's relation}\index{Chen's relation}
 \begin{align}\label{simple:Chen}
  \mathbf{X}_{s,t} - \mathbf{X}_{s,u} - \mathbf{X}_{u,t} = X_{s,u} \otimes X_{u,t}, \quad \forall u \in [s,t]
 \end{align}
To get a feeling for these identities and for the truncated tensor algebra we recommend Exercise \ref{ex:tensor} 2. below. Summing up, we should be interested in the algebra structure of the tensor algebra as well. Our next result discusses the topological structure of these algebras.
 
 \begin{lem}\label{lem:top:CIA}
  Fix $N \in \N\cup\{\infty\}$ and $d \in \N$.
  \begin{enumerate}
   \item An element $a$ in $\mathcal{T}^N(\R^d)$ is invertible in the tensor algebra if and only if the associated element of degree $0$, $a_0 \coloneq \pi_0^N(a)$ is invertible.  
   \item The algebra $\mathcal{T}^N (\R^d)$ is a continuous inverse algebra (CIA).
  Moreover if $N< \infty$, then $\mathcal{T}^N (\R^d)$ is a Banach algebra and if $N=\infty$, $\mathcal{T}^N (\R^d)$ is a \Frechet algebra. 
  \end{enumerate}
 \end{lem}

\begin{proof}
 As a first step, let us topologise $\mathcal{T}^N(\R^d)$. 
 
 \textbf{If $N< \infty$,} we obtain a finite dimensional algebra and thus it is a Banach algebra and a CIA (alternatively see Exercise \ref{ex:tensor} 1.).

\textbf{If $N = \infty$,} we observe first that $\mathcal{T}^\infty(\R^d)$ is a countable product of finite dimensional locally convex spaces, whence it is a \Frechet space. Note that this topology and the Banach topology for $N< \infty$ turn the projection $\pi_0^N$ into a continuous map.
 Exploiting the product topology, we see that the product $\mathcal{T}^\infty(\R^d) \times \mathcal{T}^\infty(\R^d) \rightarrow \mathcal{T}^\infty(\R^d)$ is continuous if and only the component maps 
 $P_k \colon \mathcal{T}^\infty(\R^d) \times \mathcal{T}^\infty(\R^d) \rightarrow (\R^d)^{\otimes k}, k\in \N$ are continuous. However, as the algebra product \eqref{algebraprod} respects the homogeneous degree of elements, it is clear that $P_k$ factors through the continuous inclusion $\mathcal{T}^k(\R^d)$ and the algebra product of $\mathcal{T}^k(\R^d)$. Both the inclusion and the algebra product are continuous, whence $P_k$ and consequently the product on $\mathcal{T}^\infty (\R^d)$ is continuous. We deduce that $\mathcal{T}^\infty (\R^d)$ is a locally convex algebra. To see that it is also a CIA, let us prove first the claim on invertibility of elements.

 Let $a \in \mathcal{T}^N(\R^d)$ and write $a = a_0 +b$, where $a_0 = \pi_0^N(a)$ and $b = a-a_0$. Since $\pi_0^N$ is an algebra homomorphism, $a_0$ must be invertible if $a$ is invertible.
 For the converse we observe that $a$ is invertible if and only if $a_0^{-1}a=1+a_0^{-1}b$ is invertible. Plugging $a_0^{-1}a$ into the \emph{Neumann inversion formula}\index{Neumann inversion formula} (cf.\ \cite[Theorem II.1.11]{MR1787146})
 \begin{align}\label{neumann_inverse}
  (1+X)^{-1} = \sum_{k=0}^\infty (-1)^kX^{\otimes k}
 \end{align}
 we obtain the desired inverse $a_0^{-1}(1+b)^{-1}$ if the series converges. For this we observe that $1-a_0^{-1}a$ has no part which is homogeneous of degree $0$. Taking products of this element with itself, we only obtain contributions by homogeneous parts of higher degrees. Hence, if $N< \infty$ and we truncate, the series is just a polynomial. For $N=\infty$ we see similarly that after projecting to the factors $(\R^d)^{\otimes k}$ we again obtain only a polynomial. Hence also in this case the series converges as it (trivially) converges in every factor. Thus $a_0^{-1}a$ is invertible and the inverse depends continuously on $a$. Moreover, the set of units $\mathcal{T}^\infty (\R^d)^\times = (\pi_0^\infty)^{-1}(\R\setminus \{0\})$ is open. Summing up this proves that $\mathcal{T}^\infty (\R^d)$ is a CIA.
 \end{proof}
 
 \begin{cor}\label{unit}
  For every $N \in \N \cup\{\infty\}$ and $d\in \N$. The group $(\mathcal{T}^N(\R^d))^\times =\{v\in \mathcal{T}^N(\R^d)\mid \pi_0^N(v) \neq 0\}$ is a regular locally exponential Lie group whose Lie algebra is $\mathcal{T}^N(\R^d)$ with the commutator bracket. Moreover, its Lie group exponential is given by  
  $$\exp_\otimes  \colon \mathcal{T}^N(\R^d) \rightarrow (\mathcal{T}^N(\R^d))^\times,\quad \exp_\otimes (v) = \sum_{k=0}^\infty \frac{v^{\otimes k}}{k!},$$
  where $v^{\otimes k}$ is the $k$th power of $v$ with respect to the algebra product. 
 \end{cor}

 \begin{proof}
  The Lie group structure was established in \Cref{ex:CIAunit} and the Lie algebra computed in Exercise \ref{Ex:LA} 8. Note that due to \Cref{lem:top:CIA} the (truncated) tensor algebra is complete, whence its unit group is regular by \Cref{ex:CIAregular}. There we also mentioned that the evolution map $\Evol$ is given by the Volterra series \eqref{Volterraseries}. Plugging a constant path $t\mapsto v$ into the Volterra series\index{Volterra series} e immediately obtain that $\exp_\otimes$ is the Lie group exponential (check this!). If $N < \infty$ the CIA is finite dimensional, whence $(\mathcal{T}^N(\R^d))^\times$ is a finite dimensional Lie group and thus locally exponential. The case $N=\infty$ is much more involved and we refer to \cite[Theorem 5.6]{MR2997582} for details. Note however, that we will establish the convergence of the Lie group exponential and its inverse on a certain closed subspace in \Cref{Lgpexp:onsubspace} below. 
 \end{proof}

 Returning briefly to iterated integrals of a smooth path $X \colon [0,1] \rightarrow \R^d$, we shall now investigate iterated integrals of the path against itself. For this identify $\R^d$ with the homogeneous elements of degree $1$. Thus $X$ becomes a smooth curve to the tensor algebra. However, for reasons which will become apparent in a moment, we are more interested in the smooth curve $DX (t) \coloneq (0,X' (t),0,0,\ldots) \in \mathcal{T}^N(\R^d)$. Applying the evolution to $DX$ (viewed as a Lie algebra valued path) and sorting the result of the Volterra series by degree we obtain 
 $$\Evol (DX)(t) = \left( t \mapsto \left(1,X_{0,t}, \int_0^t \int_0^r \mathrm{d}X\otimes \mathrm{d}X_r, \int_0^t \int_0^{r_2} \int_0^{r_1} \mathrm{d} X \otimes \mathrm{d}X \otimes \mathrm{d}X, \ldots \right)\right)$$
 where we either truncate at $N < \infty$ or take the full tensor series for $N=\infty$. We have now used the evolution of the Lie groups to define a well known object in the theory of rough paths:
 
 \begin{defn}[Signature of a smooth path]
  Let $X \colon [0,1] \rightarrow \R^d$ be a smooth path. Then we define the signature \emph{$N$-step signature}\index{signature of a smooth path}
  \begin{align}\label{sig:defn}
   S_N (X)_{s,t} \coloneq \left(1,X_{s,t}, \int_s^t \int_s^{r_1}\mathrm{d}X \otimes \mathrm{d}X, \int_s^t \int_s^{r_2}\int_s^{r_1}\mathrm{d}X \otimes \mathrm{d}X \otimes \mathrm{d}X, \ldots \right).
  \end{align}
  If $N=\infty$ we also write shorter $S(X)\coloneq S_\infty (X)$ and say that $S(X)$ is the \emph{signature of the smooth path} $X$.  
 \end{defn}
 
 As $S_N(X)$ is the evolution of a smooth path it satisfies for every $N \in \N \cup \{\infty\}$ the Lie type differential equation 
 \begin{align}\label{sig:diffeq}
  \begin{cases}
    \frac{d}{dt}S_N(X)_{0,t} = S_N(X)_{0,t} \otimes (dX)_t & t \in [0,1],\\
    S_N(X)_{0,0}= \one . &
   \end{cases}
 \end{align}
 Before we continue, it is important to stress that the signature can not only be defined for a smooth path. We shall see later (see \Cref{sec:roughintro}) that the signature exists for arbitrary rough paths.
 The signature of a path has turned out to be an immensely important object in the theory of rough paths and its applications. For example, the signature can be computed in advance for paths of interest and can then applied in numerical analysis or for machine learning purposes. We refer to \cite{chevyrev2016primer} for an introduction. 
 
 We will generalise \eqref{sig:diffeq} in Exercise \ref{ex:tensor} 4. and show that $S_N(X)_{s,t}$ can be recovered directly from $\Evol (DX)(t) = S_N (X)_{0,t}$ by virtue of Chen's relation $S_{N}(X)_{s,t} = S_N(X)_{s,u} \otimes S_N(X)_{u,t}$ for all $s<u<t$. However, our investigation so far also hints at the fact that the unit group of the tensor algebra is much larger than needed and contains many elements which will not turn out to be signatures of smooth paths. For example, the second level of the signature is a matrix (viewed as an element in $\R^d \otimes \R^d$). Its symmetric part is fixed by the shuffle of the level $1$ part of the signature with itself (Chen's relation). Hence the second level uniquely is determined by its antisymmetric part which is in the stochastical theory interpreted as the Levy-area. To capture these non-linear constraints, one restricts to a certain Lie subgroup of the unit group which expresses the geometric features of the signature. For this let us first study the restriction of the Lie group exponential.

 \begin{lem}\label{Lgpexp:onsubspace}
  Let $N \in \N \cup \{\infty\}$ and $d \in \N$ and set $\mathcal{I}_N \coloneq (\pi^N_0)^{-1}(0)$. Then $\mathcal{I}_N$ is a Lie algebra ideal in $\mathcal{T}^N (\R^d)$ and the following map are mutually inverse smooth diffeomorphisms
  \begin{align*}
   \exp_N \colon \mathcal{I}_N &\rightarrow \one + \mathcal{I}_N,\quad v \mapsto \sum_{0\leq n\leq N} \frac{X^{\otimes n}}{n!}\\
   \log_N \colon \one + \mathcal{I}_N &\rightarrow \mathcal{I}_N, \quad 1+v \mapsto \sum_{0\leq n \leq N} (-1)^{n+1} \frac{Y^{\otimes n}}{n}
  \end{align*}
 \end{lem}

 \begin{proof}
  Since the Lie bracket is given by the commutator, it is clear that $\mathcal{I}_N = (\pi_0^N)^{-1}(0)$ is a Lie ideal (i.e.\ $\LB[v,w] \in \mathcal{I}_N$ if either $v$ or $w$ is in $\mathcal{I}_N$. Note first that since $v \in \mathcal{I}_N$ we have $\pi_0^N(v)=0$ and thus $v^{\otimes k}$ does not contain contributions by homogeneous elements of degree less then $k$. In particular, we see that for every degree, the series $\exp_N$ and $\log_N$ reduce to polynomials in $v$. Thus both mappings are well defined smooth mappings to $\mathcal{T}^N(\R^d)$. Inserting and rearranging the formal power series into each other shows that $\exp_N$ and $\log_N$ are mutually inverse, whence diffeomorphisms.  
 \end{proof}

 Note that $\one + \mathcal{I}_N$ is a subgroup of $\mathcal{T}^N(\R^d)^\times$ and a closed submanifold (as a closed affine subspace of $\mathcal{T}^N(\R^d)$. Thus in particular it is a closed Lie subgroup of the unit group.
 Since the exponential series yields the Lie group exponential of the unit group, we can interpret \Cref{Lgpexp:onsubspace} as the statement that the Lie group exponential of $\one + \mathcal{I}_N$ is a diffeomorphism from the Lie algebra onto the group. However, the Lie group we are after is yet a smaller Lie subgroup of $\one + \mathcal{I}_N$ which nevertheless contains all signatures. 
 
 \begin{defn}
  Let $N \in \N \cup \{\infty\}, d\in \N$. Then we define $\mathfrak{g}^N(\R^d)$ as the smallest closed Lie subalgebra generated by the homogeneous elements of degree $1$ in $\mathcal{T}^N(\R^d)$.
  Explicitly, these algebras are constructed as follows: Set $\mathcal{P}^1 (\R^d) \coloneq \R^d \subseteq \mathcal{T}^N (\R^d)$ (via the canonical identification). Then define recursively
  \begin{align*}
   \mathcal{P}^{n+1}(\R^d) &\coloneq \mathcal{P}^n (\R^d) + \text{span} \{\LB[x,y] \mid x \in \R^d , y \in \mathcal{P}^{n}(\R^d)\}\\
   \mathcal{P}^\infty (\R^d) &\coloneq \left\{(0,P_1,P_2,\ldots ) \middle|\, \forall i \in \N, P_i \in (\R^d)^{\otimes i} \cap \mathcal{P}^n(\R^d) \text{ for some } n\in \N\right\}
  \end{align*}
  Elements in $\mathcal{P}^n(\R^d)$ are called \emph{Lie polynomials}\index{Lie polynomial} while elements in $\mathcal{P}^\infty (\R^d)$ are called \emph{Lie series}.\index{Lie series}\footnote{The name hails from the custom to write a Lie series in the form of a formal power series instead of the sequence representation we have chosen. We refer to \cite[Chapter 1]{Reut93} for more information.} Then we set $\mathfrak{g}^N (\R^d ) \coloneq \overline{\mathcal{P}^N(\R^d)}$ (where the bar denotes topological closure) and observe that these spaces are closed Lie subalgebras of $\mathcal{I}_N$ for $N \in \N \cup\{\infty\}$. 
 \end{defn}

 If we take now the image of $\mathfrak{g}^N(\R^d)$ under the exponential map $\exp_N$ we obtain a closed subset $G^N(\R^d) \coloneq \exp_N (\mathfrak{g}^N(\R^d))$ of the unit group. It is non-trivial to see that $G^N (\R^d)$ forms a group under tensor multiplication. The classical proof for this fact employs the Baker-Campbell-Hausdorff series as an essential tool. As this would lead us too far from our objects of interests, we will import this result and investigate just its differentiability.
 
 \begin{prop}\label{prop:GNlocexp}
  The set $G^N (\R^d)$ is a closed Lie subgroup of the unit group $(\mathcal{T}^N(\R^d))^\times$ for all $N \in \N \cup \{\infty\}$. Moreover, this structure turns it into a (locally) exponential Lie group.
 \end{prop}

 \begin{proof}
  We have seen already that $G^N(\R^d)$ is a closed subset.  It is a subgroup by \cite[Corollary 3.3]{Reut93}. To see that it is a Lie group, we have to show that it is a submanifold. Exploit that $\mathcal{T}^N(\R^d)$ is locally exponential (see \Cref{unit}). Hence there is a $0$-neighborhood $V \opn \mathcal{T}^N(\R^d)$ such that $\exp_\otimes$ restricts to a diffeomorphism on $V$. By construction we have $\exp_\otimes (\mathfrak{g}^N(\R^d)) = \exp_N (\mathfrak{g}^N(\R^d))=G^N(\R^d)$, whence $\exp_\otimes (V \cap \mathfrak{g}^N (\R^d)=\exp_\otimes (V)\cap G^N(\R^d)$ yields a submanifold chart $(\varphi,V)$ for $G^N(\R^d)$ around the identity. Exploiting that $\mathcal{T}^N(\R^d)^\times$ is a Lie group, we see that a submanifold atlas for $G^N(\R^d)$ is then given by $\varphi_g \colon V \rightarrow G^N(\R^d), x \mapsto g\varphi(x),\ g \in G^N(\R^d)$. We conclude that $G^N(\R^d)$ is a closed Lie subgroup of the unit group with Lie algebra $\mathfrak{g}^N(\R^d)$. To see that it is locally exponential, it suffices to notice that $\exp_\otimes$ restricts to the Lie group exponential of $G^N(\R^d)$. This is due to the fact that the Lie group exponential is defined via solution to certain differential equations. Given initial values in the closed subspace $\mathfrak{g}^N (\R^d)$, the solutions of the equation in the unit group already stay in $G^N(\R^d)$, whence they solve the differential equation in the subgroup.
  \end{proof}
  
 The algebraic arguments in this section were just cited from the literature as we wished to keep the exposition simple. However, the algebraic structure is the key to understanding the Lie groups at hand, since, as a consequence of \Cref{Lgpexp:onsubspace}, the Lie groups are globally diffeomorphic to their Lie algebra via the exponential map. For the next result, we assume familiarity with projective limits (see e.g.\ \cite[Chapter 1]{HaM07}).
 
 \begin{prop}\label{prop:proj-lim}
  Let $m,n \in \N \cup \{\infty\}$ such that $m \geq n$.
  \begin{enumerate}
   \item Then the canonical projection $\pi^m_n \colon \mathcal{T}^m (\R^d) \rightarrow \mathcal{T}^n (\R^d)$ is a morphism of locally convex algebras, which restricts to a Lie group morphism $p^m_n \colon G^m (\R^d) \rightarrow G^n(\R^d)$.
   \item We obtain a commutative diagram of Lie groups and their associated Lie algebras
    \begin{equation}\label{proj-system}
  \begin{tikzcd}
   \{0\} \times \R^d = \mathfrak{g}^1 (\R^d) \ar[d,"\exp_1"] & \mathfrak{g}^2 (\R^d)  \ar[l,"\Lf (p_1^2)"] \ar[d,"\exp_2"] & \mathfrak{g}^3 (\R^d)  \ar[l, "\Lf (p_2^3)"] \ar[d, "\exp_3"]& \ar[l] \cdots \\
   \{1\} \times \R^d = G^1 (\R^d) & G^2 (\R^d) \ar[l, "p_1^2"] & \ar[l,"p_2^3"] G^3 (\R^d) & \ar[l] \cdots
  \end{tikzcd}
 \end{equation}
 where the upper row is a projective system of locally convex Lie algebras whose limit is $\mathfrak{g}^\infty (\R^d)$.   
  \end{enumerate}
 \end{prop}

 \begin{proof}
  Due to the definition of the product in the tensor algebra it is clear that the $\pi^m_n$ are algebra morphisms. As $\pi_n^m$ is just the projection from a product onto some of its components, it is continuous in the product topology, whence a morphism of locally convex algebras. Note that this entails that $\pi^m_n$ restricts to a morphism of locally convex Lie algebras $q_n^m \colon \mathfrak{g}^m (\R^d) \rightarrow \mathfrak{g}^n(\R^d)$ and a morphism of Lie groups $p_n^m \colon G^m(\R^d) \rightarrow G^n(\R^d)$ (it is smooth as the restriction of a continuous linear map to closed submanifolds). Since $\pi_n^m$ is linear, we have $\Lf (p_m^n) = T_{\one} \pi_n^m|_{\mathfrak{g}^\infty (\R^d)} = q_n^m$. Thus the naturality of the exponential map \eqref{naturalexp}, yields $p_n^m \circ \exp_m = \exp_n \circ \Lf (p_m^n)$. Summing up, this proves (a) and establishes the commutativity of \eqref{proj-system}.
  
  For part (b) let us note that (a) establishes that the upper row of \eqref{proj-system} is a projective system of locally convex Lie algebras (as $\Lf (p_i^j) \circ \Lf (p_k^i) = \Lf(p_k^j)$ and these mappings are continuous morphisms of Lie algebras). Consider now $x \in \mathfrak{g}^\infty(\R^d)$. Its projection $\pi_n^\infty (x)$ is contained in $\mathfrak{g}^n(\R^d)$ (this is clear for a Lie series and follows by considering converging sequences for the elements in the closure since the Lie algebras $\mathfrak{g}^n(\R^d)$ are closed). Since $\pi^\infty_n$ restricts to the Lie algebra morphism $\Lf (p_n^\infty)$ on $\mathfrak{g}^\infty (\R^d)$, this implies that $\mathfrak{g}^\infty (\R^d)$ is the projective limit of the projective system in the category of Lie algebras. In addition, the product topology on $\mathcal{T}^\infty(\R^d)$ is the projective limit of the locally convex spaces $\mathcal{T}^n(\R^d)$, whence the topology on $\mathfrak{g}^\infty (\R^d)$ is the locally convex projective limit topology induced by the projective system. In conclusion, $\mathfrak{g}^\infty (\R^d)$ is the projective limit in the category of locally convex Lie algebra. 
 \end{proof}

  \begin{rem}[Projective limits of finite dimensional Lie groups]\label{pro-Lie:rem}  \index{pro-Lie group}
  In \Cref{prop:proj-lim} we exploited that the projective limit of locally convex Lie algebras can be described as the projective limit of Lie algebras with the (locally convex) projective limit topology. 
  Due to the commutativity of \eqref{proj-system}, also the lower row forms a projective system of finite dimensional Lie groups. Projective limits for Lie groups may not exist (while they always exist in the category of  topological groups). Topological groups which are projective limits of finite dimensional Lie groups are called \emph{pro-Lie groups}, \cite{HaM07}. Hence \eqref{proj-system} shows that $G^\infty (\R^d)$ is a pro-Lie group which is simultaneously a Lie group. This situation has been studied in \cite{HaN09}. It is worth mentioning that groups with these properties inherit a surprising amount of structure from the finite dimensional Lie groups which were used in their construction. 
 \end{rem}
  Summing up there the truncated groups are closely connected to their projective limit $G^\infty (\R^d)$. Moreover, (truncated) signatures of smooth paths extend naturally to the projective limit. In the next section we will review the concept of a rough path, which, in a certain sense, generalises the signature for paths of low regularity. For this, the geometry of the groups $G^N(\R^d)$ will be instrumental.  
 
\begin{Exercise}\label{ex:tensor}   \vspace{-\baselineskip}
 \Question Declare a Hilbert space structure on $\mathcal{T}^N(\R^d)$ for $N \in \N$ by defining the canonical basis elements $e_{i_1} \otimes e_{i_2} \otimes \cdots \otimes e_{i_k}$ to be orthonormal. Show that
 \subQuestion the norm corresponding to the inner product satisfies $\lVert v\otimes w\rVert \leq \lVert v\rVert \cdot \lVert w \rVert$ and $\lVert v\otimes w \rVert = \lVert w\otimes v\rVert$. Deduce that $\mathcal{T}^N(\R^d)$ becomes a Banach algebra.
 \subQuestion identifying the homogeneous elements of degree $k$ with elements $\R^{d^k}$ by sending the canonical bases to each other, the resulting isomorphism is an isometry of Hilbert spaces. 
  \Question Consider the step-$2$ truncated tensor algebra $\mathcal{T}^2 (\R^d) = \R \times \R^d \times \R^d \otimes \R^d$. 
 \subQuestion Show that the multipliction of the truncated tensor algebra is given by 
 \begin{align*}
  (a,b,c) \cdot (x,y,z) = (ax,ay+xb,az+xc+b\otimes y) \\
  \text{ and } (1,b,c)^{-1} = (1,-b,-c + b \otimes b)
 \end{align*}
 Deduce that if $d\neq 1$, the product is not commutative.
 \subQuestion Let $X \colon [0,1] \rightarrow \R^d$ be a smooth path and $\mathbb{X}_{s,t} \coloneq \int_s^t X_{s,r}\otimes X_r' \mathrm{d}r$. Define $\mathbf{X}_{s,t} \coloneq (1,X_{s,t},\mathbb{X}_{s,t}) \in \mathcal{T}^2(\R^d)$. Establish that $X$ satisfies Chen's relation\index{Chen's relation} \eqref{simple:Chen} 
 $$ \mathbb{X}_{s,t} - \mathbb{X}_{s,u} - \mathbb{X}_{u,t} = X_{s,u} \otimes X_{u,t}, \quad \forall u \in [s,t].$$
 Prove then that $\mathbf{X}_{s,t} = \mathbf{X}_{s,u} \otimes \mathbf{X}_{u,t}$ (also called Chen's relation). 
 \Question Show that for  $N \in \N$ the algebra $\mathcal{T}^N(\R^d)$ is a quotient of $\mathcal{T}^\infty (\R^d)$ modulo the algebra ideal $\mathcal{I}_N = \{(x_k)_{k \in \N} \in \mathcal{T}^\infty (\R^d)\mid x_1 = \cdots = x_N = 0\}$. 
 \Question Let $X\colon [0,1] \rightarrow \R^d$ be a smooth path, $N \in \N \cup \{\infty\}$ and $DX \colon [0,1] \rightarrow \mathcal{T}^N(\R^d), t\mapsto  (0,X_t',0,\ldots)$. Show that
 \subQuestion for fixed $s$ the signature satisfies the differential equation 
 $$\begin{cases}
    \frac{\mathrm{d}}{\mathrm{d}t} S_N (X)_{s,t} = S_N(X)_{s,t} \otimes (DX)_t & s < t \leq 1 \\
    S_N(X)_{s,s}=\one =(1,0,\ldots) &
   \end{cases}
$$
 \subQuestion the signature satisfies Chen's relation 
 $$S_N (X)_{s,t} = S_N(X)_{s,u}\otimes S_N(X)_{u,t}, \qquad 0\leq s\leq u \leq t \leq 1$$
 {\footnotesize \textbf{Hint:} It suffices to prove this for every projection of $S_N (X)$ to $(\R^d)^{\otimes k}$ and consider the iterated integral of $\mathrm{d}X_{r_1} \otimes \cdots \otimes \mathrm{d}X_{r_k}$ over the simplex $\Delta^N=\{s<r_1 < r_2 < \cdots < r_n < t\}$.}
 \subQuestion Deduce that $S_N(X)_{s,t} = S_N (X)_{0,s}^{-1}\otimes S_N(X)_{0,t}$, whence the signature can be recovered from the curve $\Evol (DX)$ via the group operations.
 \Question Let $a \in \mathcal{T}^N (\R^d)$, $N\in \N \cup \{\infty\}$ such that $a_0 \coloneq \pi^N(a) \neq 0$. Write $a= a_0 (\one +b)$ and prove that the Neumann inverse \eqref{neumann_inverse} yields $a^{-1} = a_0^{-1}(1+b)^{-1}=a_0^{-1}\sum_{k=1}^\infty (-1)^{k}b^{\otimes k}$.  
\end{Exercise}

  \setboolean{firstanswerofthechapter}{true}
\begin{Answer}[number={\ref{ex:tensor} 4.}] 
 \emph{Let $N \in \N \cup\{\infty\}, d \in \N$ and $X \colon [0,1] \rightarrow \R^d$ be a smooth path. Define $DX \colon [0,1] \rightarrow \mathcal{T}^N(\R^d), t\mapsto  (0,X_t',0,\ldots)$. We will then 
 \begin{enumerate}
  \item show that $\frac{\mathrm{d}}{\mathrm{d}t} S_N (X)_{s,t} = S_N(X)_{s,t} \otimes (DX)_t, s < t \leq 1, S_N(X)_{s,s} = \one$.
  \item establish Chen's relation 
 $$S_N (X)_{s,t} = S_N(X)_{s,u}\otimes S_N(X)_{u,t}, \qquad 0\leq s\leq u \leq t \leq 1$$
 \item thus $S_N(X)_{s,t} = S_N (X)_{0,s}^{-1}\otimes S_N(X)_{0,t}$ holds.
 \end{enumerate}
}

Note first that the claims will follow for $N \in \N \cup \{\infty\}$ if we can establish the identity for the projection to every finite degree $k \in \N_0$.
\begin{enumerate}
 \item For $s=0$ the claim is just \eqref{sig:diffeq}. For arbitrary $s$ the claim follows from inspecting the Volterra series. Namely, projecting to homogeneous elements of degree $k \geq 1$, the component of $S_N(X)_{s,t}$ is  \begin{align*}
  &\int_s^t \int_s^{r_{k-1}}\cdots \int_s^{r_1} \mathrm{d}X_{r_1} \otimes \cdots \otimes \mathrm{d}X_{r_k} \\
 =& \int_{s}^t \left(\int_{s}^{r_{k-1}} \cdots \int_s^{r_1} \mathrm{d}X_{r_1} \otimes \cdots \otimes \mathrm{d}X_{r_{k-1}} \right) \otimes \mathrm{d}X_{r_k} = \int_{s}^t \pi_{k-1}^N \left(S_N(X)_{s,r_k}\right) \otimes \mathrm{d}X_{r_k}
 \end{align*}
 In other words, the signature satisfies the integral equation $S_N(X)_{s,t} = 1+ \int_s^t S_N(X)_{s,r} \otimes \mathrm{d}X_r$ in $\mathcal{T}^N (\R^d)$, whence it solves the desired ODE.
 \item We proceed by induction on $k$. Note that the identity is trivially true for $k=0$ since it reads $1=1\cdot 1$. Assume now that we have established the claim now for every $s < u < t \in [0,1]$ and $\ell \leq k$, whence $S_k (X)_{s,t}= S_k (X)_{s,u}\otimes S_k(X)_{u,t}$.
We work now in the truncated tensor algebra $\mathcal{T}^{k+1}(\R^d)$ (and remark that the following identities hold precisely by truncating after degree $k+1$)
\begin{align*}
 S_{k+1}(X)_{s,u} = 1 + \int_s^u S_{k+1}(X)_{s,r} \otimes \mathrm{d}X_r = \int_s^u S_k (X)_{s,r}\otimes \mathrm{d}X_r\\
 S_{k+1}(X)_{s,u} \otimes \int_u^t S_N (X)_{u,r} \otimes \mathrm{d}X_r = S_{k}(X)_{s,u} \otimes \int_u^t S_k{X}_{u,r}\otimes \mathrm{d}X_r
\end{align*}
Applying the induction hypothesis to split the $S_k(X)$ for $s<u<r<t$, we obtain
\begin{align*}
 S_{k+1}(X)_{s,t} &= 1 + \int_s^u S_k(X)_{s,r} \mathrm{d}X_r + \int_u^t S_{k} (X)_{s,u}\otimes S_k (X)_{u,r} \otimes \mathrm{d}X_r \\
 &= S_{k+1}(X)_{s,u} + S_{k+1}(X)_{s,u} \otimes \left(\int_u^t S_{k}(X)_{t,r} \otimes \mathrm{d}X_r\right)\\
 &=S_{k+1}(X)_{s,u} \otimes \left(1+ (S_{k+1}(X)_{u,t}-1)\right) = S_{k+1}(X)_{s,u}\otimes S_{k+1}(X)_{u,t}.
\end{align*}
\item Multiplying Chen's relation for $S_N(X)_{0,t}$ from the left with the inverse of $S_{N}(X)_{0,s}$ immediately yields the desired identity.
\end{enumerate}
\end{Answer}
\setboolean{firstanswerofthechapter}{false}

\section{A rough introduction to rough paths}\label{sec:roughintro}

In this section we will recall the notion of a rough path. The main idea of rough path theory is that paths of much lower regularity than being smooth can be augmented with extra information replacing the iterated integrals we studied in the last section. Indeed the basic idea is to declare the signature to be the object of interest and define signature like objects in the tensor algebra. While we will present the basic theory of rough paths, we recommend one of the excellent introductions to rough path theory available (see e.g.\ \cite{FaV10,FaH20}) for more in depth information.

As a starting point let us formalise the properties observed for the signature in the last section.

\begin{defn}
 Let $\Delta \coloneq \{(s,t) \in [0,1] \mid 0\leq s \leq t \leq 1\}$ be the \emph{standard simplex},\index{standard simples ($\Delta$)} $N \in \N \cup\{\infty\}$ and $d\in \N$. We call a map
 $$\bX \colon \Delta \rightarrow \mathcal{T}^N (\R^d),\quad \bX_{s,t}=(X^0_{s,t},X^1_{s,t},X^2_{s,t},\ldots)$$
 \emph{multiplicative functional}\index{multiplicative functional}\footnote{There is a deeper story going on which motivates the term ``multiplicative functional''. We will briefly discuss this in \Cref{sect:shuffle} below.} of degree $N$ if $X^0_{s,t}= 1$ for all $(s,t) \in \Delta$ and the map satisfies \index{Chen's relation}
 \begin{align}\label{full:chen}
 \text{\emph{Chen's relation}}\qquad \bX_{s,t} = \bX_{s,u}\otimes \bX_{u,t} \quad \forall s,u,t\in [0,1], s\leq u \leq t
 \end{align}
\end{defn}
Note that Chen's relation for the first two (non-trivial) components of $\bX$ reduces to (cf.\ Exercise \ref{ex:tensor} 2.)
\begin{align}\label{Chen:lvl2}
 X^1_{s,t} = X_{s,u}^1 + X_{u,t}^1, \qquad X_{s,t}^2 = X_{s,u}^2+X_{u,t}^2+X_{s,u}^1 \otimes X_{u,t}^1.
\end{align}
Moreover, we have seen in the last chapter that every multiplicative functional is invertible in the tensor algebra, whence Chen's relation entails
$\bX_{s,t} = \bX_{0,s}^{-1}\otimes \bX_{0,t}$. So instead of a multiplicative functional, we may think of the path $[0,1] \rightarrow \mathcal{T}^N(\R^d), t \mapsto \bX_{0,t}$. Hence the term rough path will make more sense once we define it.
For the first level however, this translates to the statement that there is a path $X \colon [0,1] \rightarrow \R^d$ with $X^{1}_{s,t} = X_t - X_s$. 

\begin{ex}\label{ex:multfunc.}
 In the previous section we have already seen that the signature of a smooth path yields a multiplicative functional. For example, we can augment the zero-path in different ways to generate multiplicative functionals:
 $\bX_{s,t} \coloneq (1,0,(t-s)w)$ is a multiplicative functional of degree $2$ for any $w \in \R^d \otimes \R^d$. Note that it coincides with the level $2$-signature of the zero path only if $w$ is the zero element.
 More generally, for any function $F \colon [0,1] \rightarrow (\R^d)^{\otimes N}$ the map $\bX_{s,t} = (1,0,\ldots,0, F(t)-F(x)) \in \mathcal{T}^N(\R^d)$ is a multiplicative functional (see Exercise \ref{Ex:rough} 1.).
\end{ex}

Now we need to add an analytic condition to the algebraic objects we just defined. In this case we wish to consider rough signals in the framework of H\"{o}lder regular path. 

\begin{defn}\label{defn:Hoelder} \index{H\"{o}lder continuity}
 Let $X \colon [0,1] \rightarrow E$ be a continuous map with values in a Banach space $(E,\lVert \cdot \rVert)$. We say $X$ is an $\alpha$-H\"{o}lder path for $0 < \alpha < 1$ if
 $$\lVert X \rVert_\alpha \coloneq \sup_{\substack{t,s \in [0,1],\\ t \neq s}} \frac{\lVert X_t-X_s\rVert}{|t-s|^\alpha} <\infty.$$
 Denote by $\mathcal{C}^{\alpha}([0,1],E)$ \emph{the space of all $\alpha$-H\"{o}lder continuous functions}.\index{space!of H\"{o}lder continuous functions} 
\end{defn}
It is well known (see Exercise \ref{Ex:rough} 2. below) that  $\mathcal{C}^{\alpha}([0,1],E)$ is a Banach space on which $\lVert \cdot \rVert_\alpha$ is a continuous seminorm (it annihilates all constant paths).

\begin{rem}
 Again we will restrict here to paths on the interval $[0,1]$. The results carry over to any interval $[0,1]$ by composing the H\"{o}lder paths with a reparametrisation of the interval. Note however, that the H\"{o}lder norm is \textbf{not} invariant under reparametrisation. This is one reason why in the rough paths literature one often considers the (more or less) equivalent formulation via $p$-variation paths for $p>1$. The $p$-variation norm is invariant under reparametrisation, cf.\ \cite{FaV10}.
\end{rem}

Due to Young's theorem, it is not possible to augment an $\alpha$-H\"{o}lder path with iterated integrals against itself if $\alpha < 1/2$. The core idea of rough path theory is to augment an $\alpha$-H\"{o}lder path with additional information to make integration against it feasible and circumvent the restrictions of Young integration theory. To this end we augment the concept of a multiplicative functional with a H\"{o}lder condition. 

\begin{defn}
Fix $\alpha \in ]0,1[$ and $N \geq \lfloor 1/\alpha\rfloor$. An $\R^d$-valued \emph{$\alpha$-rough path}\index{rough path} consists of a mulitplicative functional $\bX \coloneq (1,X^1,X^{2}, X^{3},\ldots,X^{N} )$ such that the $k$th component $X^{k}$ is ``$k$-times $\alpha$ H{\"o}lder continuous'', i.e.
\begin{align}\label{Holder continuity}
 \lVert X^{k}_{s,t}\rVert \lesssim  |t-s|^{k\alpha} \text{  where \textquotedblleft} \lesssim \text{\textquotedblright means inequality up to a constant.} 
\end{align}
If in addition the $\alpha$-rough path takes values in $G^N(\R^d)$, we call $\bX$ a \emph{weakly geometric rough path}.\index{rough path!weakly geometric} 
We define $\mathrm{RP}^{\alpha}([0,1],\R^d)$ as the \emph{set of weakly geometric $\alpha$-rough paths}\index{rough path!set of weakly geometric} (we shall see in \Cref{thm:Lyons-lift} below that it makes sense to drop $N$ from the notation for $\mathrm{RP}^{\alpha}([0,1],\R^d)$).
\end{defn}

The higher levels of a rough path are not given by increments of a function, whence they do not become constant even if $k$-times $\alpha$-H\"{o}lder means that the H\"{o}lder index satisfies $k\alpha >1$.
Lyon's original concept of rough path does not require the rough path to take its image in the group $G^N(\R^d)$. However, it has turned out that the general notion is not sufficient to solve non-linear rough differential equations (see e.g.\ \cite{FaV10,FaH20}). For this purpose one should consider weakly geometric rough paths and we will do this in the rest of the section. To ease notation we shall (unless we explicitly say otherwise) only consider weakly geometric rough paths and simply call them ``rough paths''. Let us show that the cut-off level $N$ in the definition of a rough path is unimportant as long as $N \geq \lfloor 1/\alpha\rfloor$.

\begin{thm}[{Lyon's lifting theorem, \cite[Theorem 9.5]{FaV10}}]\label{thm:Lyons-lift} \index{Lyon's lifting theorem}
 Let $\alpha \in ]0,1[$ and $\lfloor 1/\alpha \rfloor \leq n$. Assume that $\bX \colon \Delta \rightarrow G^n(\R^d)$ is an $\alpha$-rough path. Then there exists a unique $\alpha$-rough path $\bX^{n+1} \colon \Delta \rightarrow G^{n+1} (\R^d)$ extending $\bX$, i.e.\ if $\pi_n^{n+1} \colon \mathcal{T}^{n+1} (\R^d) \rightarrow \mathcal{T}^{n} (\R^d)$ is the canonical projection, then $\pi_n^{n+1} (\bX^{n+1}) = \bX$. A rough path extending $\bX$ in this way is called a \emph{Lyons lift}\index{Lyons lift} of $\bX$.
 
 Thus every $\alpha$-rough path with values in $G^n(\R^d)$ extends to an $\alpha$-rough path with values in $G^m (\R^d), m\geq n \geq \lfloor 1/\alpha\rfloor$
\end{thm}
As a consequence of \Cref{thm:Lyons-lift}, every $\alpha$-rough path can be obtained as a restriction of an $\alpha$-rough path with values in $G^\infty (\R^d)$, or conversely as an extension of an $\alpha$-rough path with values in $G^{\lfloor 1/\alpha\rfloor}(\R^d)$. The remarkable fact here is of course that the extension is uniquely determined once information up to level $\lfloor 1/\alpha\rfloor$ is available. 

\begin{ex}
We have seen that for a smooth path $X$ and $N \in \N \cup \{\infty\}$, the signature $S_N (X)$ is a multiplicative functional. More generally, if we start with an $\alpha$-H\"{o}lder path for $\alpha > 1/2$, we can compute its signature using iterated Young integrals. It is then a consequence of Young's inequality \cite{Young} that the resulting multiplicative functional is indeed an $\alpha$-rough path (cf. \cite[Section 4]{FaH20} for a detailed discussion).
\end{ex}

For paths of lower regularity, the idea is that the components $X^k$ of a rough path $\bX$ replace the (in general ill defined) iterated integrals of the path giving the first level increments. Due to its importance, let us stress it here again explicitly: If $\bX = (1,X^1_{s,t}, X^{2}_{s,t}))$ is an $\alpha$-rough path for $\alpha \in ]1/3,1/2[$ and $X^1_{s,t} = X_t -X_s$, one should interpret the second level as 
$$\int_s^t X_{u,s} \mathrm{d} X_u \coloneq X^2_{s,t},\qquad 0\leq s\leq t\leq 1,$$
where the left hand side is defined via the right hand and not the other way around! 
The levels of an $\alpha$-rough path thus encode similar information as the iterated integrals in the signature of a smooth path.  

It is not obvious at all whether a given $\alpha$-H\"{o}lder path with values in $\R^d$ can be enhanced to yield an $\alpha$-rough path. Due to a result by Lyons and Victoire this this can be achieved in many case through an abstract extension result. If $1/\alpha \not\in \N$, every $\alpha$-H\"{o}lder path can be enhanced (non uniquely) to an $\alpha$-H\"{o}lder rough path (this is the Lyons-Victoire lifting theorem \cite[Theorem 1]{LaV07}). As an example for the non-uniqueness of the extension let us mention the following example for rough paths arising from Brownian motion.

\begin{ex}[Brownian motion as a rough path]\label{ex:BM} \index{rough path!Brownian motion}
 Brownian motion models the (random) movement of a particle in a fluid. It can be modelled as a stochastic process $B \colon \Omega \times [0,1] \rightarrow \R^d$, with independent Gaussian increments and continuous sample paths. Here $(\Omega, \mathcal{F},\mathbb{P})$ is a probability space. However we refer to the stochastic literature for explanations and more details. 
One can show that for partitions $P$ of $[0,1]$, the Riemann sum  
$$
\int_0^1 B_r \otimes d^{\theta} B_r = \lim_{ | P | \rightarrow 0} \sum_{ P = (t_i) } B_{ t_i + \theta( t_{i+1} - t_i) } \otimes B_{t_i,t_{i+1}} 
$$
converges in $L^2(\Omega, \R^{d \times d})$ for mesh size $|P|$ converging to $0$ and $\theta \in [0,1]$. Contrary to usual Riemann-Lebesgue integration theory, the integral $\int B_r \otimes d^{\theta}B_r$ is \emph{not} independent of $\theta$. The choice $\theta = 0$ leaves the martingale structure invariant and is referred to as the It\^{o} integral, whereas $\theta = \frac12$ is compatible with regular calculus rules and is referred to as the Stratonovich integral. It is well known that $B$ is an $\alpha$-H\"{o}lder path for every $\alpha \in ]1/3,1/2[$ and if we write $B^{It\hat{o}}$ and $B^{Strat}$ for the It\^{o} and the Stratonovich integrals, we obtain two $\alpha$-rough paths $(1,B,B^{It\^{o}})$ and $(1,B,B^{Strat})$. Note that the rough path obtained via the It\^{o} integral does not take its values in $G^2(\R^d)$ (i.e.\ it is not a weakly geometric rough path), while the one obtained from Stratonovich integration takes values in $G^2 (\R^d)$ and is weakly geometric. We refer to \cite[Chapter 3]{FaH20} for a detailed discussion of these examples.
\end{ex}

The reader may wonder now, in which sense a rough path is a path, i.e.\ can we interpreted it as a H\"{o}lder continuous path $[0,1] \rightarrow G^N(\R^d)$? For this we have to consider the metric geometry of the groups $G^N(\R^d)$ for $N<\infty$. However, as a first step to make sense of the following constructions, we need the following result (whose full proof we postpone to \Cref{rem:fullproof}). 

\begin{lem}\label{inverse_RP}
 Let $\bX \colon \Delta  \rightarrow G^N(\R^d), \bX = (1,X^1,X^2,\ldots)$ be an $\alpha$-rough path. Taking the pointwise inverse, the map $\bX^{-1} \colon \Delta \rightarrow G^N(\R^d), t \mapsto (\bX_t)^{-1}=(1,X^{-1}_t,X^{-2}_t,\ldots)$ is graded $k\alpha$-H\"{o}lder, i.e. $\lVert X^{-k}\rVert_{k\alpha} < \infty$ for all $k \in \N$
 \end{lem}

\begin{proof}
We will only prove the case where $N=2$ \textbf{and} $d=2$. Without more techniques from \Cref{sect:shuffle} the general case turns out to be quite involved.
Note that for $N=2$ due to Exercise \ref{ex:tensor} 2.a) we have $\bX^{-1} = (1,-X^1,-X^{2}+X^1 \otimes X^1)$. Thus the degree $1$ component is $\alpha$-H\"{o}lder since $X^1$ is $\alpha$-H\"{o}lder. Immediately, we see that  the degree two component might not be $2\alpha$-H\"{o}lder, as $X^1 \otimes X^1$ is a product of $\alpha$-H\"{o}lder functions. To circumvent this we express $\bX$ with the help of the standard basis $e_1,e_2$ of $\R^2$ as $\bX = (1, x_1 e_1 + x_2 e_2 , \sum_{1\leq i,j\leq 2}y_{ij} e_i \otimes e_j)$. Working out the constraints imposed by $G^2(\R^2) = \exp_2(\mathfrak{g}^2 (\R`2)$ (the reader should check this!) we find that 
$$x_1x_2 = y_{12}+y_{21}, \quad x_1^2 = 2 y_{11}, \quad x_2^2 = 2y_{22}.$$
Plugging this into the inversion formula we find for the degree $2$ component:
\begin{align*}
 X^{-2} &= - \sum_{1\leq i,j\leq 2}y_{ij} e_i \otimes e_j + x_1^2e_1\otimes e_1 + x_1x_2e_1\otimes e_2 + x_1x_2 e_2\otimes e_1 + x_2^2 e_2\otimes e_2\\
 &= y_11 e_1\otimes e_1 +y_22 e_2 \otimes e_2 + y_{12}e_2\times e_1 + y_{21}e_1\otimes e_2.
\end{align*}
Thus the second component is $2\alpha$-H\"{o}lder as it contains only contributions from the $2\alpha$-H\"{o}lder maps $y_{ij}$ comprising the second level of $\bX$. This happens in general, as the non-linear constraints imposed by $G^N(\R^d)$ allow to rewrite the $k$th degree component of $\bX^{-1}$ as a sum of the functions comprising the $k$th degree component of $\bX$ (whence they preserve the H\"{o}lder condition). 
\end{proof}

The discussion of the $N=2,d=2$ case reveals that \Cref{inverse_RP} will in general be false if the rough path is not weakly geometric. We define now a suitable metric on the group $G^N(\R^d)$ for which every $\alpha$-rough path will correspond to a H\"{o}lder function. 

\begin{defn}
 Consider for $N \in \N$ the subset $\mathcal{T}_1^N(\R^d) = \{x \in \mathcal{T}^N(\R^d) \mid \pi_0^N (x) = 1\}$. Let $x = (1,x^1,x^2,\ldots,x^N) \in \mathcal{T}_1^N(\R^d)$, then we define
\begin{align}\label{norm:homo}
 |x|\coloneq \max_{k=1,\ldots , N} \{(k!\lVert x^k\rVert)^{1/k}\} +\max_{k=1,\ldots , N} \{(k!\lVert (x^{-1})^k\rVert)^{1/k}\},
\end{align}
where $\lVert \cdot \rVert$ is the norm we have chosen on $\mathcal{T}^N(\R^d)$. Then we set $\rho_N (x,y) \coloneq |x^{-1} \otimes y|$. Identifying $\mathcal{T}_1^N(\R^d)$ with the vector space $\mathcal{I}_N$ (by subtracting $\one$), it is easy to see that $|\cdot |$ becomes a norm and $\rho_N$ induces a left invariant, symmetric and subadditive metric on $G^N (\R^d)$ (see Exercise \ref{Ex:rough} 4. for the details). 
\end{defn}

The metric $\rho_N$ turns the group $G^N(\R^d)$ into a homogeneous group in the sense of \cite{FaS82}.\footnote{The groups $G^N(\R^d)$ possess many strong structural properties which are exploited in rough path theory. They are homogeneous groups and connected and simply connected nilpotent Lie groups (i.e.\ Carnot groups, which are studied in sub-Riemannian geometry, cf.\ \cite{LDaZ21}). The details are beyond the scope of this chapter and we refer the reader to the literature.}
Now if $\bX$ is an $\alpha$-rough path (with values in $G^N(\R^d)$), then the path $x_t \coloneq \bX_{0,t}$ satisfies
\begin{align}\label{Holderpath}
 \lVert x\rVert_\alpha^{\rho_N} \coloneq \sup_{\substack{t\neq s\\ s,t \in [0,1]}} \frac{\rho_N (x_s,x_t)}{|t-s|^\alpha} = \sup_{\substack{t\neq s\\ s,t \in [0,1]}} \frac{\rho_N (\one,\bX_{s,t})}{|t-s|^\alpha} < \infty.
\end{align}
In other words, an $\alpha$-rough path is an $\alpha$-H\"{o}lder continuous path with values in metric space $(G^N(\R^d),\rho_N)$. This statement hinges on the rough path taking values in $G^N(\R^d)$ (i.e.\ the path being weakly geometric) as we need H\"{o}lder continuity of the pointwise inverse $\bX^{-1}$. 

Indeed due to Chen's relation this is an equivalent point of view, up to forgetting the starting point of the path in $\R^d$.
Unfortunately this point of view is limited to the truncated groups, as the metric \eqref{norm:homo} does not extend to $G^\infty (\R^d)$. 

\begin{rem}
 In \Cref{pro-Lie:rem} we saw that $G^\infty (\R^d)$ is the projective limit (as a Lie group) of the truncated groups $G^N(\R^d)$. Moreover, this Lie group is modelled on the \Frechet space $\mathfrak{g}^\infty (\R^d)$. So in view of the characterisation of rough paths as H\"{o}lder continuous paths in the truncated groups, the question is of course whether we can use the metric on $\mathfrak{g}^\infty (\R^d)$ to construct a suitable metric $\text{dist}$ which allows us to cast $\alpha$-rough path as an $\alpha$-H\"{o}lder path with values in the metric space $(G^\infty(\R^d),\text{dist})$.
 
 This problem was considered in \cite{LDaZ21} starting from the \emph{Carnot-Caratheodory metric} on $G^N(\R^d), N <\infty$. It is defined as $d_{\mathrm{CC}}^N(y, z) \coloneq d_{\mathrm{CC}}^N(\one, y^{-1} \otimes z)$ and
$$d_{\mathrm{CC}}^N(\one, y) \coloneq \inf \left\{ \int_0^1 \| X_t'\| \, \mathrm{d}t \middle| X \in C([0,1], \R^d), \begin{subarray}{c} X_0 = 0, \, \text{$X_t$ has bounded variation } \\ \\ \text{ and } y = S_N (X)_{0,1} \end{subarray}  \right\}, $$
where the signature of the bounded variation path is defined as in \eqref{sig:defn} using (iterated) Riemann-Stieltjes integrals. We will abbreviate this metric as the CC-metric and note that it is not straight forward to prove that it is a metric. Moreover, one can show that the CC-metric is equivalent to the metric $\rho_N$ (see \cite[7.5.4]{FaV10} for details on the CC-metric). 
Now the one can argue that the metric $d_\infty$ should correspond to the metric on the projective limit of the system  $(G^N(\R^d),\one,d_{\text{CC}}^N)$ in the category of (pointed) metric spaces (with morphisms given by submetries). However, one of the main results of \cite{LDaZ21} is the somewhat surprising insight that the limiting object $(G_\infty,\one,d_\infty)$ can not be a topological group in the topology induced by $d_\infty$ (let us note that $G_\infty \neq G^\infty (\R^d)$).

Hence there seems to be no straight forward way to construct a metric on the projective limit which simultaneously captures the desired convergence and geometry, is left-invariant and turns the projective limit into a Lie group. 
\end{rem}

Summing up, rough paths can naturally be identified as maps with values in the infinite-dimensional Lie group $G^\infty (\R^d)$. This group is closely connected to the truncated groups $G^N(\R^d)$ and inherits many properties from the projective system of these groups. However, there are important geometric properties connected to the sub-Riemannian geometry of the groups $G^N(\R^d)$, which have no counterpart on the infinite-dimensional group.
This leaves us at an uncomfortable situation: If it is enough to work with the finite dimensional groups, why bother with the more complicated situation of $G^\infty (\R^d)$?

One reason to care about the infinite-dimensional group is that it hosts all rough paths regardless of their regularity (recall that an $\alpha$-rough path can only be defined on the group $G^N(\R^d)$ where $N \geq \lfloor 1/\alpha\rfloor$). It would be interesting to understand the geometry and manifold structure of the set of all elements which can be reached by $\alpha$-rough paths (the construction of tangent spaces to $\alpha$-rough paths for $\alpha \in ]1/3,1/2[$ in \cite{QaT11} can be understood as a step in this direction). Another reason might be that one could be interested in different flavours of rough paths such as Gubinelli's branched rough paths, \cite{Gub10}. The properties of branched rough paths require a different geometry and one has to replace the groups $G^N(\R^d)$. We shall describe the general construction in the next section (but mention that for the branched rough paths the construction actually yields isomorphic groups due to a deep algebraic result).

\begin{Exercise}\label{Ex:rough}   \vspace{-\baselineskip}
 \Question Let $N \in \N$ and $\bX,\bY \colon \Delta \rightarrow \mathcal{T}^N(\R^d)$ be multiplicative functionals which agree up to degree $N-1$ (i.e.\ their projections onto the components up to $m$th level are equal). Prove then that
 \subQuestion The difference function $F_{s,t} = X^N_{s,t} - Y_{s,t}^N \in (\R^d)^{\otimes N}$ is additive in the sense that for all $s\leq u\leq t$ one has $F_{s,t}= F_{s,u}+F_{u,t}$.
  \subQuestion if $F\colon \Delta \rightarrow (\R^d)^{\otimes N}$ is an additive function (in the above sense), then also $(s,t) \mapsto \bX_{s,t} + F_{s,t}$ is an additive functional (where addition is addition in the tensor algebra).
 \Question We consider the $\alpha$-H\"{o}lder space $\mathcal{C}^{\alpha}([0,1],E)$ for $n \in \N$, $0 < \alpha< 1$ for $(E,\lVert \cdot \rVert)$ a Banach space. Show that the map $$\lVert x \rVert_\alpha \coloneq \sup_{\substack{t,s \in [0,1],\\ t \neq s}} \frac{\lVert x_t-x_s\rVert}{|t-s|^\alpha}$$
 is a semi-norm on $\mathcal{C}^{\alpha}([0,1],E)$ which is not a norm. Then deduce that
 $$\lVert x\rVert_{\mathcal{C}^{\alpha}} \coloneq \sup_{t \in [0,1]} \lVert x_t\rVert + \lVert x\rVert_\alpha$$
 is a norm turning $\mathcal{C}^{\alpha}([0,1],E)$ into a Banach space.
 \Question Assume that a path $x \colon [0,1] \rightarrow E$ with values in a Banach space is $\alpha$-H\"o{}lder in the sense of \Cref{defn:Hoelder}. If $\alpha >1$, show that $x$ is differentiable and constant.
 \Question Consider for $N \in \N$ the subset $\mathcal{T}_1^N (\R^d) = \{x \in \mathcal{T}^N(\R^d) \mid \pi_0(x) = 1\}$ (i.e. all elements in the tensor algebra whose zeroth degree term is $1$ (recall that these are invertible!) set $\rho_N (x,y) \coloneq |x^{-1}\otimes y|$ where as in \eqref{norm:homo} we have
 $$|(1,x^1,x^2,\ldots,x^n)| =  \max_{k=1,\ldots , N} \{(k!\lVert x^k\rVert)^{1/k}\} +\max_{k=1,\ldots , N} \{(k!\lVert ((x^{-1})^k\rVert)^{1/k}\}.$$
 \subQuestion Show that $|\cdot|$ is subadditive, i.e.\ $|x\otimes x'| \leq |x|+|x'|$.
 \subQuestion For $\lambda \in \R$ define the dilation $\delta_\lambda \colon \mathcal{T}_1^N (\R^d) \rightarrow \mathcal{T}_1^N (\R^d), (x^k)_{0\leq k\leq N} \mapsto (\lambda^kx^k)_{0\leq k\leq N}$. Show that $|\delta_\lambda (x)|=|\lambda| |x|$
 \subQuestion Show that $\rho_N$ is a left invariant metric on $G^N(\R^d)$, i.e.\ $\rho_N (b\otimes x , b \otimes y) =\rho_N(x,y)$.
 \subQuestion Let $N \geq \lfloor 1/\alpha \rfloor$. Prove that a multiplicative functional $X \colon \Delta \rightarrow G^N(\R^d)$ is an $\alpha$-rough path if and only if \eqref{Holderpath} is satisfied.\\
 {\footnotesize \textbf{Hint}: You will need \Cref{inverse_RP} and can use (without a proof) that the $k$th component of $\bX^{-1}$ arises by applying a linear map to the $k$th component of $\bX$.}
 \end{Exercise}

\section{Rough paths and the shuffle algebra} \label{sect:shuffle}

In this section we will broaden the scope of the investigation. The reason for this is twofold. On one hand, it allows us to give some further motivation to the question why infinite-dimensional differential geometry is of relevance in rough path theory. On the other hand, there is a pleasing framework which generalises the construction of the groups $G^\infty (\R^d)$ and leads to a whole class of infinite-dimensional Lie groups which have recently been found in a variety of mathematical contexts.

To start, consider again a smooth path $X \colon [0,1] \rightarrow \R^d$ with components $X^i, i=1,\ldots d$. Now the iterated integrals comprising the signature $S_N(X)$ satisfy several algebraic conditions. For example, for $S_2 (X)_{s,t} = (1,X_{s,t},\int_s^tX_{s,r}\mathrm{d}X)$ the product rule of ordinary calculus yields
\begin{align}\label{product-rule}
X_{s,t}^i \cdot X_{s,t}^j =  \int_s^t X_{s,r}^i \mathrm{d}X^j + \int_s^t X_{s,r}^j \mathrm{d}X^i = \int_s^t \int_s^r \mathrm{d}X^i \mathrm{d}X^j +  \int_s^t \int_s^r \mathrm{d}X^j \mathrm{d}X^i
\end{align}
Similar identities for higher order iterated integrals and this is ultimately responsible for the non-linear set $G^N(\R^d)$ being the correct state space of a rough path. The question is of course how these identities and the associated combinatorics can be conveniently expressed. These questions lead to the so called shuffle product.

\begin{defn}[{Shuffle algebra, \cite{Reut93}}]\label{shufflealgebra} \index{shuffle algebra}
 Consider the set $\mathcal{A} = \{1,2,\ldots,d\}$, which is called in this context an \emph{alphabet} and its elements \emph{letters}. By concatenation we can construct words from the letters and the set $\mathcal{A}^\ast$ of all words including the empty word $\emptyset$. We construct now the shuffle algebra $\text{Sh} (\mathcal{A}) = \R \mathcal{A}^\ast$ as the vector space generated by the words over $\mathcal{A}$. Note that as a locally convex space $\text{Sh}(\mathcal{A})$ is isomorphic to the direct sum of countably many copies of $\R$, see \Cref{ex:boxtop}. Moreover, $\mathcal{A}^\ast$ and thus also $\text{Sh}(\mathcal{A})$ is graded by word-length, i.e.\ if $w = a_1\cdots a_n \in \mathcal{A}^\ast$ for letters $a_i$, we set $|w|=n$ and say that an element of $\text{Sh}(\mathcal{A})$ is homogeneous of degree $n \in \N_0$ if it is a linear combination of words of length $n$. For the algebra structure let $a,b \in \mathcal{A}$ and $u,w \in \mathcal{A}^\ast$. Then the \emph{shuffle product} is defined recursively by
    \begin{displaymath}
    \emptyset \shuffle w  = w \shuffle \emptyset = w, \quad  (au) \shuffle (bw) := a(u\shuffle (bw)) + b((au) \shuffle w)
    \end{displaymath}
 We will see in Exercise \ref{Ex:Hopfexp} 1. that it extends to a continuous product on $\text{Sh}(\mathcal{A})$, whence $\text{Sh}(\mathcal{A})$ becomes a unital locally convex algebra, called the \emph{shuffle algebra}.\footnote{The shuffle algebra is actually a graded bialgebra. The coproduct is given by deconcatenation of words
    \begin{displaymath}
    \delta \colon\text{Sh} (\mathcal{A}) \rightarrow\text{Sh} (\mathcal{A}) \otimes \text{Sh} (\mathcal{A}),\quad  \delta (a_1 a_2 \cdots a_n) = \sum_{i=0}^n a_1  a_2 \cdots a_i \otimes a_{i+1} \cdots a_{i+2} \cdots a_n , \quad a_1,\ldots, a_n \in \mathcal{A}.
    \end{displaymath}
 As the subspace of homogeneous elements of degree $0$ is $1$-dimensional, the graded bialgebra $\text{Sh} (\mathcal{A})$ becomes a Hopf algebra (see \cite{Man08} for an introduction). This means that $\text{Sh} (\mathcal{A})$ admits an antipode, i.e.\ a mapping $S\colon \text{Sh} (\mathcal{A}) \rightarrow \text{Sh} (\mathcal{A})$ connecting the algebra and coalgebra structures and given by $$S (a_1 a_2 \ldots a_n) = (-1)^n a_n a_{n-1} \ldots a_1, \quad \forall a_1,\ldots a_n \in \mathcal{A}.$$}
\end{defn}
It turns out that the shuffle product encodes the combinatorics keeping track of the non-linear identities in the signature. For this we define for a smooth path $X \colon [0,1]\rightarrow \R^d$ with (truncated) signature $S_N (X), N\in \N \cup \{\infty\}$ the evaluation 
$$\langle S_N (X) , w \rangle \coloneq \int_s^t\int_s^{t_{n-1}} \cdots \int_s^{t_1} \mathrm{d}X^{a_1} \mathrm{d}X^{a_2} \cdots \mathrm{d}X^{a_m},\quad \text{for } w= a_1 \cdots a_n \in \mathcal{A}^\ast, |w|\leq N.$$
With this notation in place, one can show (see Exercise \ref{Ex:Hopfexp} 2.) that the chain rule of ordinary calculus implies
\begin{align}\label{signature:shuffle}
 \langle S_N (X)_{s,t},w\rangle \langle S_{N}(X)_{s,t},u\rangle = \langle S_N(X)_{s,t},w\shuffle u\rangle, \quad \forall w,u \in \mathcal{A}^\ast.
\end{align}
Thus in view of \eqref{signature:shuffle}, we see that the chain rule induces some non-linear constraints to the signature. Indeed these constraints are responsible for the signature to take its values in $G^N(\R^d)$ instead of the full tensor algebra. The role these identities play will become clearer once we change our point of view slightly. For this we identify the $k$th level of the signature $S_\infty (X)$ as a mapping into $(\R^d)^{\otimes k} \cong \R^{d^k}$ whose components are given by the maps $\Delta \rightarrow \langle S_\infty (X),w\rangle$ for all $w \in \mathcal{A}^\ast$ with $|w|=k$. Hence for every pair $(s,t) \in \Delta$ we can identify the signature with a functional 
$$\langle S_\infty (X)_{s,t}, \cdot \rangle \colon \text{Sh}(\mathcal{A}) \rightarrow \R,\quad \mathcal{A}^\ast \ni w \mapsto \langle S_\infty (X)_{s,t},w\rangle ,$$
In Exercise \ref{Ex:Hopfexp} 2. we will show that $\langle S_\infty (X)_{s,t},\cdot)$ is continuous, i.e.\ it takes its values in the continuous dual $(\text{Sh}(\mathcal{A}))'\cong \prod_{w \in \mathcal{A}^\ast} \R$ (this follows from Exercise \ref{Ex:final} 2. as the dual of a locally convex direct sum of copies of $\R$ is the direct product, \cite[Proposition 24.3]{MaV97}). From this identification of the dual one deduces at once that the signature gives rise to a continuous map $\langle S_\infty (X),\cdot \rangle \colon \Delta \rightarrow (\text{Sh}(\mathcal{A}))'$. Elements in the image of $\langle S_\infty (X)_{s,t},\cdot\rangle$ map the algebra product to the multiplication in $\R$ by \eqref{signature:shuffle}, whence they are algebra morphisms. 

Algebra morphism from the shuffle algebra are called \emph{characters}\index{character of an algebra} of the algebra $\text{Sh}(\mathcal{A})$. We have already mentioned that the shuffle algebra carries more structure and is indeed a Hopf algebra. A \emph{Hopf algebra}\index{Hopf algebra} $A$ is an algebra which is simultaneously a coalgebra $A \rightarrow A\otimes A$ such that the algebra and coalgebra structure are connected via the so called antipode $S \colon A \rightarrow A$. We refer to e.g.\ \cite{Man08} for the full definition and many examples. Now, for a Hopf algebra such as the shuffle algebra, the characters form a group $\mathcal{G}(\text{Sh}(\mathcal{A}),\R)$ under the convolution product \index{character group of a Hopf algebra}
$$\varphi \star \psi (w) \coloneq \varphi \otimes \psi (\delta(w)), \text{ where } \delta \text{ is the deconcatenation coproduct}.$$
Note that we exploit here $\R \otimes \R \cong \R$. Inversion in the character group is given by precomposition with the antipode, i.e.\ $\varphi^{-1} = \varphi \circ S$ (cf.\ e.g.\ \cite[Lemma 2.3]{BaDaS16}). 
Summing up, we have just proved the following:

\begin{lem}
 For a smooth path $X \colon [0,1] \rightarrow \R^d$ the signature induces a continuous map 
 $$\langle S_\infty (X),\cdot\rangle \colon \Delta \rightarrow \mathcal{G}(\mathrm{Sh}(\mathcal{A}),\R) \subseteq \mathrm{Sh}(\mathcal{A})', \quad (s,t) \mapsto \langle S_\infty (X)_{s,t},\cdot\rangle$$
 with values in the character group of the Hopf algebra $\mathrm{Sh}(\mathcal{A})$.
\end{lem}
More general, if we consider the $k$th level of an element $\bX = (1,X^1,X^2,\ldots) \in G^\infty (\R^d)$ we see that $X^k \in (\R^d)^{\otimes k}$ can be written as $X^k = \sum_{w \in \mathcal{A}^\ast, w=a_1 \cdots a_k} \alpha_{\bX} (w) e_{a_1} \otimes \cdots e_{a_k}$, where $e_{a_\ell}$ is the standard basis vector labelled by $a_\ell \in \mathcal{A} = \{1,\ldots ,d \}$. With some work on the algebra, one can prove the following:

\begin{lem}[{\cite[Theorem 2.6]{Ree58}}]
 Let $\mathcal{A} = \{1,\ldots , d\}$ and for $\bX = (1,X^1,X^2,\ldots)$ we write $X^k = \sum_{w \in \mathcal{A}^\ast, w=a_1 \cdots a_k} \alpha_{\bX} (w) e_{a_1} \otimes \cdots e_{a_k}$. Then the following map is a well defined group isomorphism
 \begin{equation}\label{groupiso} \begin{split}
 \Psi\colon G^\infty (\R^d) \rightarrow & \ \mathcal{G}(\mathrm{Sh}(\mathcal{A}),\R),\quad \bX = (1,X^1,X^2,\ldots) \mapsto \psi_\bX, \\ 
  \textup{ where } & \psi_\bX \textup{ is defined on $\mathcal{A}^\ast$ as } \psi_\bX (w) = \alpha_\bX (w),\quad \forall w \in \mathcal{A}^\ast.
\end{split}
\end{equation}
Moreover, a similar statement holds for the character group of the truncated shuffle algebra and the groups $G^N(\R^d), N \in \N$. 
\end{lem}
Turning now to the topological side, $G^\infty (\R^d)$ is a Lie group with respect to the subspace topology induced by the tensor algebra. As a locally convex space, the tensor algebra is isomorphic to a countable product of copies of the reals. We can also topologise the character group $\mathcal{G}(\text{Sh}(\mathcal{A}),\R) \subseteq \text{Sh}(\mathcal{A})' = \prod_{w \in \mathcal{A}^\ast} \R$ with the subspace topology. This topology even turns the character group into an infinite-dimensional Lie group, \cite[Theorem A]{BaDaS16}. Thus we have the following additional statement:

\begin{prop}\label{prop:RPidentify}
 The group isomorphism \eqref{groupiso} is an isomorphism of Lie groups. Thus there is a bijection between $\alpha$-rough path $\bX$ with values in $G^\infty (\R^d)$ and continuous maps 
 $$\Psi_\bX \colon \Delta \rightarrow \mathcal{G} (\mathrm{Sh}(\mathcal{A}),\R), (t,s) \mapsto \psi_{\bX_{s,t}} $$
 such that for every $w \in \mathcal{A}^\ast$ the component path $\Delta \rightarrow \R, (t,s) \mapsto \psi_{\bX_{s,t}} (w)$ satisfies the ``$k\alpha$-H\"o{}lder condition''
 \begin{align}\label{kHoelder:character}
  \sup_{\substack{(s,t)\in \Delta\\ s\neq t}} \frac{| \psi_{\bX_{s,t}}(w)|}{|t-s|^{k\alpha}} < \infty.
 \end{align}
\end{prop}

\begin{proof}
 In Exercise \ref{Ex:Hopfexp} 3. below you will show that \eqref{groupiso} is the restriction of an isomorphism of locally convex spaces, whence an isomorphism of topological groups. Now by \Cref{prop:GNlocexp} the group $G^\infty (\R^d)$ is a Lie group and a submanifold of the tensor algebra. From \cite[Theorem B]{BaDaS16} we know that also $\mathcal{G}(\text{Sh}(\mathcal{A}),\R)$ is Lie group and a submanifold of the dual space of the shuffle algebra. Hence we deduce from \Cref{lem:submfd:initial} that \eqref{groupiso} is already smooth and thus an isomorphism of Lie groups. 
 
 Every $\alpha$-rough path $\bX$ with values in $G^\infty (\R^d)$ is a continuous $G^\infty (\R^d)$-valued map (since its components are continuous in the product topology $G^\infty(\R^d) \subseteq \mathcal{T}^\infty (\R^d)$) and the $k$th level satisfies a $k\alpha$-H\"{o}lder condition. Hence $\Psi_\bX \coloneq \Psi \circ \bX \colon \Delta \rightarrow \mathcal{G}(\text{Sh}(\mathcal{A}),\R)$ is continuous. Up to changing the H\"{o}lder constant, we can change the norm on $\R^d)^{\otimes k } \cong \R^{d^k}$ to the maximum norm. This shows that the components $\Psi_\bX(w)$ inherit any H\"{o}lder condition the $k$th level of $\bX$ satisfies, i.e.\ \eqref{kHoelder:character} holds. Vice versa if \eqref{kHoelder:character} holds for every $w\in \mathcal{A}^\ast$ with $|w|=k$, then the $k$th level of $\bX$ satisfies a $k\alpha$-H\"{o}lder condition. This establishes the claimed bijection. 
\end{proof}

 \Cref{prop:RPidentify} yields a nice interpretation of a (weakly geometric) rough path as a ``multiplicative functional'' on the shuffle algebra.\index{multiplicative functional} Of course this was not intended when T.\ Lyons originally coined the term multiplicative functional. 

\begin{rem}\label{rem:fullproof}
We are finally able to give a complete and easy proof for \Cref{inverse_RP}: In view of \Cref{prop:RPidentify}, a rough path $\bX$ takes its values in $\mathcal{G}(\mathrm{Sh}(\mathcal{A}),\R)$ and the group structure is isomorphic to the one of $G^\infty(\R^d)$. Hence the pointwise inverse $\bX^{-1}$ can be computed via the group structure. Now it is a well known fact about characters of a Hopf algebra, that for a character $\phi$, its inverse in the character group is given by $\phi \circ S$, where $S$ is the antipode of the Hopf algebra, \cite[Lemma 2.3]{BaDaS16}. Thus in the case of the shuffle Hopf algebra, this yields for an arbitrary linear combination of the words of length $k$ the formula  
$$S\left(\sum_{i_1,\ldots i_k \in\{1,\ldots ,d\}} a_{(i_1,\ldots , i_k)} i_1 \cdots i_k\right) = \sum_{i_1,\ldots i_k \in\{1,\ldots ,d\}} (-1)^ka_{(i_1,\ldots , i_k)} i_k \cdots i_1.$$
Now the $k$th component of $\bX$ corresponds to a linear combination of basis element $e_{i_1} \otimes e_{i_2} \otimes \cdots \otimes e_{i_k}$. The isomorphism $\mathcal{G} (\mathrm{Sh}(A),\R) \cong G^\infty (\R^d)$ identifies $i_1i_2 \ldots i_k$ with $e_{i_1}\otimes \cdots \otimes e_{i_k}$, whence the inversion formula shows the $k$th component of the pointwise inverse $\bX^{-1}$ arises by permuting coefficients (and multiplying by $(-1)^k$). Hence the $k$th component of $\bX^{-1}$ is $k\alpha$-H\"{o}lder if the $k$th component of $\bX$ is so.
\end{rem}

\begin{Exercise}\label{Ex:Hopfexp} \vspace{-\baselineskip}
 \Question Let $\mathcal{A} = \{1,\ldots , d\}$ be a finite alphabet and $\mathcal{A}^\ast$ the monoid of words generated by $\mathcal{A}$. Recall that $\text{Sh}(\mathcal{A}) = \R\mathcal{A}^\ast$ is the vector space generated by $\mathcal{A}^\ast$. Show that 
 \subQuestion for all words $(u\shuffle v)\shuffle w = u \shuffle (v\shuffle w)$, $u\shuffle v = v \shuffle u$ and $|u \shuffle v| = |u|+|v|$ hold. 
 \subQuestion one can bilinearly extend the shuffle product to $\shuffle \colon \text{Sh}(\mathcal{A}) \times  \text{Sh}(\mathcal{A}) \rightarrow  \text{Sh}(\mathcal{A})$. 
 \subQuestion the box-topology from \Cref{ex:boxtop} makes $\shuffle \colon \text{Sh}(\mathcal{A}) \times\text{Sh}(\mathcal{A}) \rightarrow \text{Sh}(\mathcal{A})$ continuous and $(\text{Sh}(\mathcal{A}),\shuffle)$ a locally convex algebra. \\
 {\footnotesize \textbf{Hint:} By Exercise \ref{Ex:final} 2. it suffices to show that for every pair of words $v,w$ the shuffle induces a continuous map $\R \times \R \rightarrow \text{Sh}(\mathcal{A}), (a,b) \mapsto (a v)\shuffle (bw)$.}
\Question Let $X \colon [0,1] \rightarrow \R^d$ be a smooth path with signature $S_N (X)$ (for some $N \in \N \cup \{\infty\}$). Let $w = a_1\cdots a_n \in \mathcal{A}^\ast$ 
 Define $\langle S_N(X)_{s,t} , w\rangle = \int_s^t \mathrm{d}X^{a_1} \mathrm{d}X^{a_2} \cdots \mathrm{d}X^{a_n}$. Show that
 \subQuestion \eqref{signature:shuffle} holds, i.e. $\langle S_N (X)_{s,t},w\rangle \langle S_{N}(X)_{s,t},u\rangle = \langle S_N(X)_{s,t},w\shuffle u\rangle, \quad \forall w,u \in \mathcal{A}^\ast.$
 {\footnotesize \textbf{Hint:} Use induction by word length and the chain rule. For $|w|=1$, the claim becomes \eqref{product-rule}.}
 \subQuestion for every pair $(s,t) \in \Delta$ the linear map $\langle S_\infty (X)_{s,t}, \cdot \rangle \colon  \text{Sh}(\mathcal{A}) \rightarrow \R$ is continuous, where we endow the shuffle algebra with the box-topology from \Cref{ex:boxtop}.
 \subQuestion the signature gives rise to a continuous map 
 $\langle S_\infty (X),\cdot \rangle \colon \Delta \rightarrow (\text{Sh}(\mathcal{A}))'\cong \prod_{w\in \mathcal{A}^\ast} \R.$
 {\footnotesize \textbf{Hint:} The identification is given by $\text{Sh}(\mathcal{A}) \ni f \mapsto (f(w))_{w \in \mathcal{A}^\ast} \in \prod_{w\in \mathcal{A}^\ast}\R$.}
 \Question Show that there is a canonical isomorphism $\text{Sh}(\mathcal{A})' = \prod_{n\in \N} \R \cong \mathcal{T}^\infty (\R^d)$ of locally convex spaces if $\mathcal{A}=\{1,\ldots, d\}$. Deduce that this restricts to  \eqref{groupiso}, whence $\eqref{groupiso}$ is a morphism of topological) groups if we endow $G^\infty (\R^d)$ with its Lie group topology and $\mathcal{G}(\text{Sh}(\mathcal{A}),\R)$ with the subspace topology induced by embedding it in $\text{Sh}(\mathcal{A})'$. 
\end{Exercise}
 
 \section{The grand geometric picture (rough paths and beyond)} \label{sect:beyond_rough}

 In the last section we have seen that rough paths can be understood as certain continuous paths with values in the character group of the shuffle Hopf algebra.
 We will now extend the focus slightly to different types of rough paths. For this, recall that the signature of a smooth path had to satisfy certain algebraic identities which connected it to the shuffle algebra and its character group. For smooth paths this was a consequence of the chain rule of ordinary calculus. However, there are also many interesting objects, which do not satisfy a classical chain rule. One example for this is the It\^{o}-integral from stochastic calculus (nevertheless we constructed the rough path It\^{o}-lift of Brownian motion in \Cref{ex:BM}). However, this motivates the idea to relax the requirement of \eqref{signature:shuffle} and record iterated integrals of a path $X \colon [0,1] \rightarrow \R^d$ of the form 
\begin{align}\label{branchedintegral}
 \int_s^t \left(\int_s^r \mathrm{d}X^i\right) \left( \int_s^r \mathrm{d}X^j\right) \mathrm{d}X^k.
\end{align}
 For the rough paths we have considered so long, the iterated integral \eqref{branchedintegral} (respectively the corresponding combination of levels of the rough path) simplifies via the shuffle identity to an iterated integral (resp.\ level of the rough path) we have already recorded. Relaxing the requirement \eqref{signature:shuffle} necessitates that we record objects corresponding to \eqref{branchedintegral} as additional information.
 Following this approach one arrives at Gubinelli's concept of a \emph{branched rough path}, \cite{Gub10}.\index{rough path!branched} Branched rough paths satisfy different algebraic identities than the rough paths we have been considering (which are called \emph{weakly geometric rough paths} in the literature, \cite{FaH20}). By dropping \eqref{signature:shuffle}, branched rough paths will not take values in the character group of the shuffle algebra. However, this can easily be remedied by replacing the shuffle algebra with another Hopf algebra. This leads us to the following concept.
 
 \begin{defn}
  Let $\mathcal{H}$ be a graded and connected Hopf algebra\footnote{Graded and connected means that $\mathcal{H} = \bigoplus_{n \in \N_0} \mathcal{H}_n$ as vector spaces with $\mathcal{H}_0 \cong \R$ and the algebra and coalgebra structures are compatible with the grading, see \cite{BaDaS16}. Recall that an element $x \in \mathcal{H}$ is called homogeneous of degree $|x|=n$ if $x \in \mathcal{H}_n$.} and $(\mathcal{G}(\mathcal{H},\R), \star)$ its character group. Fix a basis $B$ of the vector space $\mathcal{H}$ consisting of homogeneous elements. Then an $\alpha$-rough path over $\mathcal{H}$ is a map
  $$\bX \colon \Delta \rightarrow \mathcal{G}(\mathcal{H},\R)$$
  which satisfies Chen's relation
  $$\bX_{s,u} \star \bX_{u,t} = \bX_{s,t},\qquad s \leq u \leq t \in [0,1],$$
  and for every $h \in B$ we have the graded H\"{o}lder condition $|\langle \bX_{s,t} , h\rangle| \lesssim |t-s|^{\alpha|h|}$ (where $\langle \cdot ,h\rangle$ is evaluation in $h$).
  If $\text{dim} \mathcal{H}_1 = d$ and $B \cap \mathcal{H}_1 = \{e_1, \ldots, e_d\}$, an $\mathcal{H}$-rough path $\bX$ \emph{lifts} $x \colon [0,1] \rightarrow \R^d, x(t)=(x_1(t),\ldots , x_d(t))$ if $\langle \bX_{s,t},e_i\rangle = x_i(t)-x_i(s), i=1,\ldots, d$.
   \end{defn}
  
  Since the Hopf algebra $\mathcal{H}$ is graded, one can also define truncated versions of the rough paths by looking at the character group of truncated Hopf algebra.
  Note that the choice of basis $B$ is part of the definition of an $\mathcal{H}$-rough path and the graded H\"{o}lder condition will depend on the choice of grading and the base.
  
\begin{ex}
 The concept of an  $\alpha$-rough path over $\mathcal{H}$ is quite flexible and we illustrate this with the following list of different types of rough paths appearing in the literature:
 \begin{itemize}
  \item For $\mathcal{H} = \text{Sh}(\mathcal{A})$ and $B = \mathcal{A}^\ast$, we recover the notion of a (weakly) geometric rough path as discussed in \Cref{sec:roughintro} and \Cref{sect:shuffle}.
  \item If $\mathcal{H}$ is the Butcher-Connes-Kreimer Hopf algebra\index{Butcher-Connes-Kreimer algebra} of (decorated) rooted trees, \cite[Example 4.6]{BaDaS16} and $B$ the base of all forests, one obtains branched rough paths, \cite{Gub10}.
  \item For the Munthe--Kaas-Wright Hopf algebra and $B$ the base consisting of planar forests, one obtains the notion of a planarly branched rough path. These objects generalise the concept of a branched rough path to homogeneous spaces. See \cite{CaEFaMaMK20} for more information.
 \end{itemize}
\end{ex}
While switching the Hopf algebra allows to treat different concepts of rough paths, the general geometric theory stays the same for the different types of rough paths which can be treated in this generalised setting.
For example, there is a generalised version of the Lyons-Victoire lifting theorem (see \cite[Theorem 3.4]{TaZ20}) and the character groups can be endowed with an infinite-dimensional Lie group structure for which every rough path becomes a continuous map (see \cite[Theorem A]{BaDaS16}). Thus the theory for weakly geometric rough paths from \Cref{sect:shuffle} carries over to the generalised concept of rough path over $\mathcal{H}$.
However, there are also new geometric features which appear in this generalised setting. For example, the following result hints at interesting geometry: 

\begin{prop}[{\cite[Theorem 1.2]{TaZ20}}]\label{prop:homogeneous}
 For every $\alpha \in ]0,1[$ with $\alpha^{-1} \not \in \N$ the $\alpha$-branched rough paths form a homogeneous space under the action of a vector space of $\alpha$-H\"{o}lder functions.\index{rough path!branched} 
\end{prop}

The term ``homogeneous space''\index{homogeneous space} in \Cref{prop:homogeneous} means here, that there is a transitive and free group action $G \times X \rightarrow X$ of a vector space $G$ of H\"{o}lder continuous functions on the space of branched rough paths $X$. The canonical manifold structures on the spaces and groups should turn the branched rough paths into an infinite-dimensional homogeneous space in the sense of differential geometry (see \Cref{ex:homogeneous}), however this question was, to the best of my knowledge, not yet investigated. While there is a lot more which could be said about rough paths and their interplay with (finite and infinite dimensional) geometry we will not go into more details here. Instead, let us briefly point out a general theme underlying the idea to identify rough paths as paths into the character group of a suitable Hopf algebra.

\begin{setup}[Elements in the character group as formal power series]\label{formal:powerseries}
 Assume that we have a graded Hopf algebra $\mathcal{H} = \bigoplus_{n \in \N_0} \mathcal{H}_n$ with $B$ a vector space base of $\mathcal{H}$ consisting only of homogeneous elements. Instead of thinking of an element in  the (continuous) dual space $\mathcal{H}' = \prod_{n\in \N_0} \mathcal{H}_n'$ as a linear map, we can identify it with its values on the base $B$ and write it as a \emph{formal power series}\index{formal power series} 
 \begin{align}\label{formal:power}
  \psi = \sum_{w \in B} \langle \psi , w\rangle w, \qquad \text{with } \langle \psi , w\rangle \coloneq \psi (w)
 \end{align}
 Now elements in the character group $\mathcal{G} (\mathcal{H},\R)$ are algebra morphisms. Hence it suffices to record their values on a set $C$ which generates the algebra $\mathcal{H}$ in the sense that $\mathcal{H}$ is isomorphic (as an algebra) to the (not necessarily commutative) polynomial algebra $\R \langle C\rangle$. If we have chosen a set $C$ generating $\mathcal{H}$ in this sense, it suffices to record instead of the formal power series \eqref{formal:power} the series 
 \begin{align}\label{reduced:formalpower}
  \psi = \sum_{w \in C} \langle \psi , w \rangle w , \qquad \text{ if } \psi \in \mathcal{G} (\mathcal{H},\R).
 \end{align}
\end{setup}

\begin{ex} \index{shuffle algebra}
 For the shuffle algebra $\text{Sh} (\mathcal{A})$ we can represent a functional as the formal power series 
 $$\psi = \sum_{w\in \mathcal{A}^\ast} \langle \psi , w \rangle w$$
 these series are also called \emph{word series} and they are studied in relation to dynamical systems and numerical integration, \cite{MaSZ17}. Note that the shuffle algebra can be interpreted as a commutative polynomial algebra generated by the Lyndon words, \cite[Chapter 5]{Reut93}.  
\end{ex}

While \Cref{formal:powerseries} seems to be only a notational trick to write a functional as a series this idea leads to a rich source of examples. The idea one should have when looking at the series expansions \eqref{formal:power} and \eqref{reduced:formalpower} is that these series represent some kind of Taylor expansion.\footnote{For rough paths this is most easily seen in the context of controlled rough paths, cf.\ \cite[Section 4.6]{FaH20}.} It is then the role of the Hopf algebra $\mathcal{H}$ to describe the combinatorics for the objects which define the Taylor expansion. Groups of (formal) power series modelled as characters on suitable Hopf algebras arise in a variety of contexts. Beyond rough paths here one should mention the following applications: Hairers's regularity structures for stochastic partial differential equations (see \cite[Chapter 13]{FaH20}), word series in numerical analysis \cite{MaSZ17}, Chen-Fliess series in control theory \cite{gray2021continuity} and Connes-Kreimers approach to the renormalisation of quantum field theories, \cite{Man08}. In these applications, the differential geometry of the character group is of interest. As a concrete example, recall the connection to numerical integrators.

\begin{ex}
 Consider the (time independent) ordinary differential equation 
 \begin{align}\label{ODE:develop}
  y'(t) = F(y), \text{ where } F \colon \R^d \rightarrow \R^d \text{ is a vector field.}
 \end{align}
 Assume that we are trying to compute a power series solution to \eqref{ODE:develop}. Then we compute derivatives of $y$ via the chain rule: 
 $$y'=F(y), y''= dF(y;F(y)), y''' = d^2 F (y;F(y),F(y))+ dF(y;dF(y;F(y))),\ldots.$$ 
 Fixing a starting point $y_0$ we only need iterated differentials of $F$ and information on where derivatives were inserted in the arguments of the differential. The combinatorics can be handled by encoding the information via rooted trees, i.e.\ as finite graphs with special node called the root (in the following displayed as the nethermost node):
 \tikzstyle dtree=[grow'=up,sibling distance=2mm,level distance=2mm,thick]
\tikzstyle dtree node=[scale=0.3,shape=circle,very thin,draw]
\tikzstyle dtree black node=[style=dtree node,fill=black]
 $$\begin{tikzpicture}[dtree]
    \node[dtree black node] {}
    ;
  \end{tikzpicture}, \quad \begin{tikzpicture}[dtree]
  \node[dtree black node] {}
  child { node[dtree black node] {} }
  ;
\end{tikzpicture}, \quad 
\begin{tikzpicture}[dtree]
  \node[dtree black node] {}
  child { node[dtree black node] {} }
  child { node[dtree black node] {} }
  ;
\end{tikzpicture},
\quad 
\begin{tikzpicture}[dtree]
  \node[dtree black node] {}
  child { node[dtree black node] {} child { node[dtree black node] {} }}
  ;
\end{tikzpicture}, \quad 
\begin{tikzpicture}[dtree]
  \node[dtree black node] {}
  child { node[dtree black node] {} }
  child { node[dtree black node] {} child { node[dtree black node] {} }}
  ;
\end{tikzpicture},\quad
\begin{tikzpicture}[dtree]
  \node[dtree black node] {}
  child { node[dtree black node] {} }
  child { node[dtree black node] {} }
  child { node[dtree black node] {} }
  ;
\end{tikzpicture}, \ldots
$$
 Formally, we write $\emptyset$ for the empty tree. Every rooted tree $\tau$ can be written recursively as $\tau = [\tau_1 ,\ldots,\tau_m]$, where the $\tau_i$ are trees whose roots are graftet to a common new root. For example $\left[\begin{tikzpicture}[dtree]
    \node[dtree black node] {}
    ;
  \end{tikzpicture},\begin{tikzpicture}[dtree]
    \node[dtree black node] {}
    ;
  \end{tikzpicture}\right] = \begin{tikzpicture}[dtree]
  \node[dtree black node] {}
  child { node[dtree black node] {} }
  child { node[dtree black node] {} }
  ;
\end{tikzpicture}$. Then the iterated differentials of $F$ can recursively be encoded via the \emph{elementary differentials} defined as  $E_F (\bullet,y) \coloneq F(y)$ and for a rooted tree $\tau$ we set 
 $$E_F (\tau)(y) \coloneq d^n F (y;E_F (\tau_1,y),\ldots, E_F(\tau_n,y)) \text{ for } \tau = [\tau_1, \ldots, \tau_n]$$ 
 This leads to a formal power series, called a \emph{B-series} 
 \begin{align}\label{Bseries}
  B_F (\psi,y,h) \coloneq y + \sum_{\tau \text{ rooted tree}} \frac{h^{|\tau|}}{\sigma(\tau)} \psi (\tau) E_F(\tau,y),
 \end{align}
 where $h \in \R, a(\tau) \in \R$, $|\tau|$ is the number of nodes in the tree $\tau$ and $\sigma(\tau)$ is a certain symmetry factor. B-series model certain numerical solutions to \eqref{ODE:develop} where the parameter $h$ is interpreted as the step-size of the numerical method. One identifies the B-series \eqref{Bseries} with its coefficients $\{\psi(\tau)\}_{\tau \text{ rooted tree}}$. As the rooted trees generate the Butcher-Connes-Kreimer Hopf algebra, $\psi$ extends to a character of this Hopf algebra. The character group of the Butcher-Connes-Kreimer Hopf algebra is in numerical analysis known at the \emph{Butcher group},\index{Butcher group} \cite{BaS17}. Its elements encode numerical integration schemes such as Runge-Kutta methods and the group product models composition of schemes.  
\end{ex}
 From the perspective of numerical analysis, the Butcher group is a convenient tool as its algebraic and differential geometric structure is of interest in the analysis of numerical integrators. Unfortunately, it is also a very large group with many elements which do not correspond to any (locally) convergent integration scheme. This leads to subgroups which admit a stronger topology turning them into Lie groups while still containing all elements of interest, \cite{DaS20}. In the context of rough paths, a similar problem is the question whether there is a differentiable structure on the subgroup generated by all $\alpha$-rough paths for some fixed $\alpha \in ]0,1[$.

\appendix
\chapter{A primer on topological vector spaces and locally convex spaces}\label{AppA}\copyrightnotice

This section contains some auxiliary results on topological vector spaces and locally convex spaces in particular. Most of the results are standard and can be found in textbooks such as \cite{MaV97,Jar81}. Note that for some of the results (e.g.\ \Cref{prop:uniquetop}) in this Appendix, it is essential that we only consider Hausdorff topological vector spaces. Since we are only working with real vector spaces, some of the proofs simplify substantially (compare to the general proofs for $\R$ and $\mathbb{C}$, see \cite[Chapter 1]{Rudin}). 

\section{Basic material on topological vector spaces}
A vector space with a Hausdorff topology making vector addition and scalar multiplication continuous is called a topological vector space or TVS (cf.\ \Cref{defn:tvs}).\index{space!topological vector space}
Note that a morphism of TVS is a continuous linear map. In particular, two TVS are isomorphic (as TVS) if they are isomorphic as vector spaces and the isomorphism is a homeomorphism. 
\begin{tcolorbox}[colback=white,colframe=blue!75!black,title=Conventions]
 Let $U,V$ be subsets of a (topological) vector space $E$, $s \in \R$ and $I \subseteq \R$. Then we define 
 \begin{align*}
  U+V \coloneq \{z = u+v\mid u\in U, v\in V\}, \quad sU\coloneq \{z=su\mid u\in U\}, \quad I\cdot U \coloneq \bigcup_{s\in I } sU
 \end{align*}
\end{tcolorbox}

\begin{defn}\label{def:TVS}
 Let $(E,\mathcal{T})$ be a TVS and $U$ a subset of $E$. We say that $U$ is
 \begin{enumerate}
 \item a \emph{$0$-neighbourhood} if $U$ is a neighbourhood of $0$.
\item \emph{bounded} if for every $0$-neighbourhood $V$ there is $s>0$ with $U \subseteq sV$.
\end{enumerate}
In general topological vector spaces sequences are not sufficient to test for example continuity. Instead one would need nets to test for continuity and a complete topological vector space should be defined in terms of convergence of Cauchy nets (see \cite[p. 258]{MaV97}). However, for our calculus we usually do not need complete spaces and the limits we consider can always be described in terms of sequential limits. Thus we do not go into details here and stay in the realm of the more familiar sequences.
A sequence $(x_n)_{n\in\N} \subseteq E$ is called
\begin{enumerate}
\item[(c)]  \emph{Cauchy sequence} if for every $0$-neighbourhood $V \subseteq E$ there exists $N \in \N$ such that 
 $$x_n - x_m \in V \quad \forall n,m \geq N$$ 
 \item[(d)] \emph{Mackey-Cauchy sequence} if there exists a bounded subset $B \subseteq E$ and a family $m_{k,l} \in \N $ for $k,l \in \N$ such that 
 $$m_{k,l} (x_k-x_l) \in B, \quad  \forall k,l \in \N$$
 and such that for every $R >0$ there is $N \in \N$ with $m_{k,l} > R$ if $k,l > N$ (i.e.~$m_{k,l}\rightarrow \infty$). Note that every Mackey-Cauchy sequence is a Cauchy sequence.
\end{enumerate}
Now we say that the topological vector space $(E,\mathcal{T})$ is 
\begin{enumerate}
 \item[(e)] \emph{sequentially complete} if every Cauchy sequence in $E$ converges.  
 \item[(f)] \emph{Mackey complete} if every Mackey-Cauchy sequence in $E$ converges.\index{Mackey complete}
 \end{enumerate} 
 Mackey completeness as per (f) can be shown (see \cite[Theorem 2.14]{KM97}) to be equivalent to the notion from \Cref{Mackeycomplete}.
\end{defn}

\begin{lem}\label{lem:spec_subsets}
 Let $E$ be a topological vector space and $U \subseteq E$ a $0$-neighbourhood. Then the following holds
 \begin{enumerate}
  \item[(a)] For each $x \in E$ the translation $\lambda_x \colon E \rightarrow E, y \mapsto x+y$ is a homeomorphism.
  \item[(b)] for each $r \in \R \setminus \{0\}$, scaling $s_r \colon E \rightarrow E, x \mapsto rx$ is a homeomorphisms.
  \item[(c)] $U$ contains a \emph{balanced} $0$-neighbourhood $V$, i.e.\ $tV \subseteq V$ for each $|t|\leq 1$.
  \item[(d)] $U$ contains a $0$-neighbourhood $W$ such that $W+W \subseteq U$.
  \item[(e)] If $\mathcal{B}$ is a basis of $0$-neighbourhoods, then for each $x \in E$ the set $\{x+W \mid W \in \mathcal{B} \}$ is a basis of $x$-neighbourhoods.
  \item[(f)] Each $0$-neighbourhood contains a closed $0$-neighbourhood.
  \item[(g)] If $K \subseteq E$ is compact and $U \opn E$ with $K \subseteq U$, then there exists $0 \in W \opn E$ such that $K + W \subseteq U$.
 \end{enumerate}
 \end{lem}
\begin{proof}
 \textbf{(a-b) and (e)} The maps $\lambda_x$ and $s_r$ have inverses $\lambda_{-x}$ and $s_{1/r}$. Thus the claim is clear from the definition of topological vector spaces. Since translations are homeomorphisms (a) implies (e).\\
 \textbf{(c)} By continuity of the scalar multiplication $\mu \colon \R \times E \rightarrow E$, the set $\mu^{-1}(U)$ is open. Thus we can find $(-\varepsilon,\varepsilon) \times W \subseteq \mu^{-1}(U)$ and thus $V\coloneq (-\varepsilon ,\varepsilon)W = \mu((-\varepsilon, \varepsilon)\times W) \subseteq U$. Then $[-1,1]V=V$ and $V$ is a balanced $0$-neighbourhood contained in $U$.\\
 \textbf{(d)} As addition $\alpha \colon E \times E \rightarrow E$ is continuous with $\alpha(0,0)=0$, the preimage $\alpha^{-1}(U\times U)$ is a $(0,0)$-neighbourhood. We can thus find $W_1,W_2 \subseteq E$ $0$-neighbourhoods in $E$ such that $W_1 \times W_2 \subseteq \alpha^{-1}(U)$. Then $W\coloneq W_1 \cap W_2 \subseteq U$ satisfies $W+W \subseteq U$.\\
 \textbf{(f)} We conclude from (c) and (d) that there is a $0$-neighbourhood $V$ with $V-V \subseteq U$. For $w \in \overline{V}$ (closure), the set $w + V$ is a $w$-neighbourhood (by (e)). Hence we can pick $v_1\in V$ such that $v_1 \in w+V$, i.e.\ $v_1 = w+v_2$ for some $v_2 \in V$. But then $w = v_1-v_2 \in U$, whence $\overline{V} \subseteq U$.\\
 \textbf{(g)} For every $x \in K$ we can pick by (e) a $0$-neighbourhood $V_x$ such that $x + V_x \subseteq U$. By (d) there is $0 \in W_x \opn E$ with $W_x + W_x \subseteq V_x$. Then $(x+W_x)_{x \in K}$ is an open cover of $K$ and by compactness we can choose a finite subset $F\subseteq K$ with $K \subseteq \bigcup_{x \in F} (x+W_x)$. Then $W\coloneq \bigcap_{x \in F} W_x$ is an open $0$-neighbourhood. For $y \in K$, there exists $x \in F$ such that $y \in x+W_x$. Then $y+W\subseteq x+W_x+w \subseteq x + V_x \subseteq U$. As $y$ was arbitrary $K +W \subseteq U$. 
\end{proof}

\begin{prop}[{\cite[I Theorem 1.22]{Rudin}}]\label{prop:compfindim}
 If $E$ is a topological vector space which contain a compact $0$-neighbourhood, then $E$ is finite-dimensional.
\end{prop}

\begin{prop}[{Uniqueness of topology, \cite[Theorem 9.1]{trev06}}]\label{prop:uniquetop}
 If $E$ is a finite dimensional topological vector space of dimension $d$, then $E \cong \R^d$ as topological vector spaces, where $\R^d$ carries the usual norm topology.
\end{prop}

\begin{lem}\label{lem:linmapzero}
 Let $f \colon E \rightarrow F$ be a linear map between topological vector spaces. Then $f$ is continuous (open) if and only if it is continuous (open) in $0$.
\end{lem}

\begin{proof}
 Clearly the conditions are necessary. To prove sufficiency, we assume that $f$ is continuous in $0$. Pick $x \in E$ and observe that the translations $\tau_{-x}(y)=y-x$ and $\tau_{f(x)}(z)=f(x)$ are continuous. Thus $f(y)=\tau_{f(x)}\circ f\circ \tau_{-x}$ is continuous in $x$. since $x$ was arbitrary, $f$ is continuous in every point.
 
 Let now $f$ be open in $0$ and $U$ an open set. For $x \in U$, $\tau_{-x}(U)$ is an open $0$-neighbourhood, whence $\tau_{f(x)} \circ f \circ \tau_{-x}(U)=f(U)$ is open.
\end{proof}

En route towards locally convex spaces let us first recall some results on convex sets.

		\begin{defn}\label{defn:convx}
			A subset $S$ of a topological vector space $E$ is said to be
		\begin{itemize}
		    \item \emph{absorbent} if $E= \bigcup_{t >0} tS$.
			\item \emph{convex} if for all $x,y \in S$ and $t \in [0,1]$, the linear combination $t x + (1-t)y \in S$
			\item a \emph{disc} if it is convex and balanced, i.e.\ for all $x \in S$, $|\lambda|\leq 1, \lambda \in \R$, $\lambda x \in S$.
		\end{itemize}
		\end{defn} 
			\begin{lem}\label{lem:nbhd:abs}
			Every $0$-neighbourhood of a topological vector space is absorbent.
		\end{lem}
		\begin{proof}
			Let $U$ be a $0$-neighbourhood of the topological vector space $E$, and let $x_0 \in E$. Since scalar multiplication is continuous there exists some neighbourhood $V$ of $x_0$ and $\delta > 0$ such that for all $x \in V$, $\lambda x \in U$ when $|\lambda| < \delta \, , \, \lambda \in \R$. Especially $\lambda x_0 \in U$, and thus $\{x_0\} \subset \lambda^{-1}U $  when $|\lambda|  \leq \delta /2$ that is $|\lambda^{-1} |\geq 2/\delta$. 
		\end{proof}
	\begin{ex}	
	 The ball $B_r (x) \coloneq \{y\in E \mid \lVert x-y\rVert < r\}$ in a normed space $(E,\|\cdot\|)$ is convex (and a disc if $x=0$).
	\end{ex}

	\begin{lem}\label{lem:intcvx}
		The interior $A^\circ$ of a convex set $A$ is convex. 
	\end{lem}
	\begin{proof}
		Let $x, y \in A^\circ$. By \Cref{lem:spec_subsets} (c) there is some balanced $0$-neighbourhood $U$, such that $x + U \opn A^\circ$ and $y + U \opn A^\circ$ are neighbourhoods of $x$ and $y$ contained in $A$. For any $z = tx + (1-t)y,\ t \in [0,1]$ and $u \in U$ we have
		$$
		z + u = tx + (1-t)y + tu + (1-t)u = t (x + u) + (1-t)(y + u). 
		$$
		As $A$ is convex, $z+u \in A$ and thus $z+U \opn A$. Hence $tx+(1-t)y \in A^\circ, \forall t \in [0,1]$.
	\end{proof}

    \begin{lem}\label{lem:discnbhd}
     If $N$ is a convex $0$-neighbourhood in a topological vector space, then $N$ contains an open disc.
    \end{lem}

    \begin{proof}
     Consider first the set $M \coloneq -N \cap N$. If $|\lambda| \leq 1$ we see that $\lambda M = (-\lambda)N \cap \lambda N$. Now as $0 \in N$ and $N$ is convex, we have that $-\lambda N, \lambda N \subseteq N$. In particular, $\lambda M \subseteq M$ for all $|\lambda|\leq 1$, i.e.~$M$ is balanced. By \Cref{lem:spec_subsets} (c), we can find a balanced $0$-neighbourhood $U \subseteq N$. As $U$ is balanced we see that $U \subseteq -N \cap N$. Hence the interior $V$ of $M=-N\cap N$ is a convex $0$-neighbourhood (by \Cref{lem:intcvx} as it is the interior of an intersection of convex sets, Exercise \ref{Ex:TVS} 3.). Now the interior of a balanced set is again balanced (Exercise \ref{Ex:TVS} 1.), whence $V$ is a disc.
    \end{proof}

    \begin{Exercise}\label{Ex:TVS}  \vspace{-\baselineskip}
 \Question Let $B$ be a balanced subset of a topological vector space. Show that then also the interior of $B$ is balanced.
  \Question Let $U \opn E$ be a bounded $0$-neighbourhood in a topological vector space $E$. Show that every $0$-neighbourhood contains a set of the form $\{rU\}_{r\in ]0,\infty[}$.\\ 
 {\tiny \textbf{Hint:} Use \Cref{lem:nbhd:abs} together with the fact that $U$ is bounded.} 
 \Question Show that the intersection $C\cap D$ of two convex sets $C$ and $D$ is convex.
  \end{Exercise}
    \section{Seminorms and convex sets}
    
    In the main text we have defined locally convex spaces using seminorms. In this section we shall review seminorms and in particular their connection to convex sets (thus justifying the name ``locally convex space''). 
        \begin{defn}
				A family $\cal P$ of seminorms on a vector space $E$ is said to be \emph{separating} if for each $x \in E$, $p(x) \neq 0$ for at least one $p\in \cal P$. \index{seminorm!separating family of}
			\end{defn}
    
	\begin{prop}\label{prop:semi-norm}
		Let $E$ be a vector space and $\big(p_i \big)_{i\in I}$ a separating family of seminorms on $E$. Then a Hausdorff vector topology is generated by the subbase
			\begin{equation}\label{subbase}
				B_{i, \epsilon}(x_0)  \coloneq \{x \in E : p_i(x- x_0) < \epsilon \} , \quad   i \in I, \, \epsilon > 0, \, x_0 \in E.
			\end{equation}
        Thus $(E,\{p_i\}_I)$ is a locally convex space and the topology contains a 0-neighbourhood basis of convex sets. Finally, each  $p_i$ is continuous with respect to the locally convex topology.
	\end{prop}
		\begin{proof}
			Let us first note that the subbase \eqref{subbase} generates the initial topology induced by the family $\{q_i \colon E \rightarrow E/\text{ker}p_i\}_{i\in I}$, where the right hand side is endowed with the normed topology induced by $p_i$ (cf. Exercise \ref{Ex:AppA} 1.)
			Let $U$ be 0-neighbourhood. Then $U$ contains some finite intersection $\bigcap B_{i, \epsilon}(x)$, which is convex since the seminorm balls are convex and intersections preserve convexity, Exercise \ref{Ex:TVS} 3. Thus every $0$-neighbourhood contains a convex $0$-neighbourhood. 
	     	For the Hausdorff property we choose for $x_1,x_2\in E$ a seminorm $p_i$ such that $0 < p_i(x_1 - x_2)$. Set $\delta =  p_i(x_1 - x_2)/3$. Now if 
			$z \in B_{i, \delta}(x_2) \cap B_{i, \delta}(x_1)$ were non-empty we must have 
			\begin{equation*}
				0 < \delta= p_i(x_1 - x_2) \leq p_i(x_1- z) + p_i(z -x_2) \leq \frac{2}{3} \delta
			\end{equation*}
			which is absurd. Therefore the intersection must be empty and $E$ is Hausdorff.
			
			We have continuity of addition since for each
			\begin{equation*}
			U_{x + y} = B_{i_1, \epsilon_1}(x_0 + y_0) \cap B_{i_2, \epsilon_2}(x_0 + y_0)\cap \dots \cap B_{i_n, \epsilon_n}(x_0 + y_0)
			\end{equation*}
			the neighbourhoods 
			\begin{equation*}
				U_{x_0} = B_{i_1, \epsilon_1/2}(x_0) \cap B_{i_2, \epsilon_2/2}(x_0)\cap \dots \cap B_{i_n, \epsilon_n/2}(x_0)
			\end{equation*}
			and
			\begin{equation*}
			U_{y_0} = B_{i_1, \epsilon_1/2}(y_0) \cap B_{i_2, \epsilon_2/2}(y_0) \cap \dots  \cap B_{i_n, \epsilon_n/2}(y_0)
			\end{equation*}
			satisfies $U_x + U_y \subset U_{x + y}$ since for any $x +y \in B_{i, \epsilon/2}(x) + B_{i,\epsilon/2}(y)$
			\begin{equation*}
				p_i(x +y) \leq p_i(x) + p_i(y) < \epsilon/ 2 + \epsilon/2 = \epsilon
			\end{equation*}
			For the continuity of scalar multiplication consider the neighbourhood
			\begin{equation*}
				U_{\lambda_0 x_0} = B_{i_1, \epsilon_1}(\lambda_0 x_0) \cap B_{i_2, \epsilon_2}(\lambda_0 x_0)\cap \dots \cap B_{i_n, \epsilon_n}(\lambda_0 x_0)
			\end{equation*}
			 Let $|\lambda - \lambda_0| < \delta$, and $x \in B_{i, \delta}(x_0)$ then
			\begin{align*}
			p_i(\lambda x - \lambda_0x_o) & = 
			p_i((\lambda - \lambda_0)x + \lambda_0(x - x_0)) \\ & \leq 
			|\lambda - \lambda_0| p_i(x) + |\lambda_0| p_i(x- x_0) \\ & <
			\delta(p_i(x-x_0) + p_i(x_0)) +   |\lambda_0|p_i(x- x_0) \\ & <
			\delta (\delta + p_i(x_0)+ |\lambda_0|) < \epsilon
			\end{align*}
			if $\delta$ is small enough. So we can find $\delta_1, \dots , \delta _n$ and $\delta = \min\{\delta_1, \dots , \delta _n \}$ such that
			\begin{equation*}
			(\lambda, x) \in (\lambda_0-\delta,\lambda_0+ \delta) \times U_{x_0} \implies \lambda x\in U_{\lambda_0 x_0} 
			\end{equation*}
			To see that $p_i$ is continuous, let $(a,b)$ be an open interval in $[0,\infty)$, then $p_i^{-1} (a, b) = (E\setminus \overline B _{i, a}(0)) \cap B_{i, b}(0)$ is open being a finite intersection of open sets, when 
			$
				\overline B _{i, a}(0) \coloneqq  \{ x \in E : p_i(x) \leq a \} \, . 
			$
			Since $p_i^{-1}([0,b[) =B_{i,b}(0)$ we deduce that $p_i$ is continuous. 
			\end{proof}
		
		\begin{defn}
		 	Let $\big(p_i)_{i \in I}$ be a family of seminorms. We say the family \index{seminorm!basis condition}
		 	\begin{itemize}
		 	 \item satisfies the \emph{basis condition} if for each two seminorms $p_i$ and $p_j$, there exists a third seminorm $p_k$ and $C>0$ such that 
		 	\begin{equation*}
		 	\max\{p_i(x) , p_j(x)\} \leq C p_k(x), \quad \forall x \in E . 
		 	\end{equation*} 
		 	\item is called a \emph{fundamental system} of seminorms, if it generates the topology on $E$ and satisfies the basis condition.\index{seminorm!fundamental system of}
		 	\end{itemize}
		 \end{defn}
			Note that for a fundamental system of seminorms the subbase \eqref{subbase} is a basis for the topology it generates.
	
	\begin{ex}\label{ex:co_seminorm}
	 Consider again the space of smooth functions $C^\infty ([0,1],\R)$ with the \Frechet\ topology induced by the seminorms $\lVert f\rVert_n \coloneq \sup_{0\leq k\leq n} \sup_{x \in [0,1]} \left|\frac{\mathrm{d}^k}{\mathrm{d}x^k}f(x)\right|$. For two of these seminorms we obviously have 
     $$\max\{\lVert f \rVert_n,\lVert f \rVert_m\} \leq \lVert, \quad f \rVert_{\max \{n,m\}} \forall f \in C^\infty ([0,1],\R).$$
     Hence these seminorms form a fundamental system of seminorms and their $r$-balls form a basis of the topology called the compact open $C^\infty$-topology.\index{compact open $C^\infty$-topology}
	\end{ex}

	Let us associate now to every disc a seminorm. The upshot will be that one can equivalently define a locally convex space as a topological vector space with a $0$-neighbourhood base consisting of convex sets.
	
	\begin{defn}
		For a vector space $E$ and a disc $A$ in $E$, define the \emph{Minkowski functional},\index{Minkowski functional} $p_{A}\colon E \to \R, p_{A}(x) \coloneq \inf \{t > 0 \mid x \in tA\} \ , $
		where $\inf \emptyset \coloneq \infty$
	\end{defn}
	\begin{lem}\label{Mink:disc0nbhd}
		If $U$ is a disc $0$-neighbourhood in the locally convex space $E$, then the Minkowski functional $p_U$ is a continuous seminorm on $E$.
	\end{lem}
	\begin{proof}
		By \Cref{lem:nbhd:abs}, $p_{U}(x) \in [0,\infty)$ for all $x \in E$.  
		For the triangle inequality, let $x, y \in E$, then if $x \in tU$ and $y \in sU$
		\begin{equation*}
			\frac{1}{t + s} (x +y) = \frac{t}{t + s}\frac{x}{t} + \frac{s}{t + s}\frac{y}{s} \in U
		\end{equation*}
		or rather $x + y \in (t+s)U$. Therefore $p_U({x+y})\leq  t + s = p_U({x}) + p_U({y})$. Multiplicative scalars can be taken out of the seminorm due to $U$ being balanced: If $\lambda \in \R \setminus \{0\}$,
		$\lambda x = \frac{\lambda}{|\lambda|}|\lambda| x \in tU$  if and only if $|\lambda|x = \frac{|\lambda|}{\lambda} \lambda x \in tU$. Therefore $p_U({\lambda x }) = \inf\{ t > 0 : \lambda x \in tU \} = \inf\{t >0 : |\lambda| x \in t U \} = |\lambda|p_U(x)$. Continuity of the seminorm follows from $p_U^{-1}(]a,b[)=(E\setminus(\overline{aU}))\cap \bigcup_{0\leq t \leq b} tU$ is open for all $a,b \in [0,\infty[$.
\end{proof}
As \Cref{lem:discnbhd} shows, every convex $0$-neighbourhood gives rise to a disc and these give rise to seminorms by \Cref{Mink:disc0nbhd}. Thus an equivalent definition of a locally convex space (fitting the name better, cf.~\cite[Theorem 1.34 and Remark 1.38]{Rudin}) is:
\begin{defn}
 A Hausdorff topological vector space is a locally convex space if it contains a $0$-neighbourhood basis of convex sets.\index{space!locally convex}
\end{defn}

Finally, let us recall Kolmogorov`s normability criterion, which gives a (necessary and sufficient!) condition for a topological vector space to be normable, i.e.~the vector topology coincides with the topology induced by some norm.

\begin{thm}[Kolmogorov`s normability criterion]\label{thm:Kolm_crit} \index{Kolmogorov's normability criterion}
 A topological vector space $E$ is normable if and only if $E$ has a bounded and convex $0$-neighbourhood.
\end{thm}

\begin{proof}
 The criterion is necessary as in a normed space, the unit ball is a convex bounded $0$-neighbourhood. 
 
 Let now $E$ be a topological vector space with a bounded and convex $0$-neighbourhood $N$. By \Cref{lem:discnbhd} we can pick an open disc $U \subseteq N$. Note that $U$ is also bounded and define $\lVert x\rVert \coloneq p_U (x)$ for $x \in E$, where $p_U$ is the Minkowski functional associated to $U$. Now due to Exercise \ref{Ex:TVS} 2.~the sets $sU,\ s \in ]0,\infty[$ form a neighbourhood base of $0$ in $E$. If $x \neq 0$ is an element of $E$, the Hausdorff property implies that there exists $s>0$ with $x \not \in sU$, i.e.~$\lVert x\rVert \geq s >0$. We deduce from \Cref{Mink:disc0nbhd} that $\lVert \cdot \rVert$ is a (continuous) norm on $E$. Hence the norm topology induced by $\lVert \cdot \rVert$ is coarser then the original topology. Conversely, recall that since $U$ is open, we have $\{x\in E \mid \lVert x\rVert <s \} = sU$. As the $sU$ form a neighbourhood base, this shows that the norm topology coincides with the original topology, whence $E$ is normable.  
\end{proof} 

The Kolmogorov normability criterion allows us to describe the pathology occurring for dual space of topological vector spaces beyond Banach spaces.

\begin{prop}\label{prop:dual_norm}
 Let $E$ be a locally convex space and 
 $$E' \coloneq \{\lambda \colon E \rightarrow \R\mid \lambda \text{ is continuous and linear}\}$$ be its dual space.\index{dual space}
 If $E'$ is a topological vector space such that the evaluation map\\ $\ev \colon E \times E'\rightarrow \R , (x,\lambda) \mapsto \lambda (x)$ is continuous, then $E$ is normable. 
\end{prop}

\begin{proof}
 Assume that $E`$ is a topological vector space such that $\ev$ is continuous. Then there are $0$-neighbourhoods $U \opn E$ and $V\opn E'$ such that $\ev (U \times V) \subseteq [-1,1]$. Since $V$ is absorbent by \Cref{lem:nbhd:abs} this implies that every continuous linear functional is bounded on $U$. Now \cite[Theorem 3.18]{Rudin} yields that $U$ is already bounded. Shrinking $U$, we may assume that $U$ is convex and bounded. Hence Kolmogorov`s criterion \Cref{thm:Kolm_crit} shows that $E$ must be normable. 
\end{proof}

Recall from Exercise \ref{Ex:BFrecht} 1. that if $E$ is normable, the evaluation map on the dual space is indeed continuous with respect to the operator norm on the dual space.
\begin{Exercise}\label{Ex:AppA}  \vspace{-\baselineskip}

\Question Let $E$ be a vector space and $p$ a seminorm on $E$.
 \begin{enumerate}
  \item[(a)] Show that $\text{ker} p \coloneq \{x \in E \mid p(x)=0\}$ is a vector subspace of $E$.
  \item[(b)] Prove that $\lVert q(x)\rVert \coloneq \inf_{y \in \text{ker} p} p(x+y)$ defines a norm on the quotient, where $q$ is the (surjective!) quotient map $q\colon E \rightarrow E/\text{ker}p$.
 \end{enumerate}
 \Question Let $E$ be a locally convex space whose topology $\mathcal{T}$ is generated by a family $\mathcal{P}$ of seminorms. Show that if $q$ is a continuous seminorm on $E$, then the topology generated by $\cal P \cup \{ \mathrm{q}\}$ is equal to $\mathcal{T}$.
 \Question Show that every locally convex space admits a fundamental system of seminorms.\\ 
 {\tiny \textbf{Hint:} Show that $\max \{p,q\}$ is a continuous seminorm and use the previous exercise.} 
 \end{Exercise}
\section{Subspaces of locally convex spaces}\label{App:complemented}

In this section we recall some material on subspaces of locally convex spaces.

\begin{defn}\index{space!complemented subspace}
 A vector subspace $F \subseteq E$ of a locally convex space is called \emph{complemented} if there exists a locally convex space $X$ such that $E \cong E\times X$ (isomorphic as locally convex spaces).
\end{defn}

\begin{lem}\label{lem:complemented}
 A subspace $F \subseteq E$ is complemented if and only if there exists a continuous linear map $\pi \colon E \rightarrow E$ with $\pi(E)=F$ and $\pi \circ \pi = \pi$. Further, a complemented subspace is always closed. We call $\pi$ a \emph{continuous projection}.\index{continuous projection}
\end{lem}

\begin{proof}
 Let $F$ be complemented with isomorphism $\varphi \colon E \rightarrow F \times X$, then $\pi \coloneq \varphi^{-1}\circ p_1\circ \varphi$ with $p_1 \colon F \times X \rightarrow F \times X, (f,x) \mapsto (f,0)$ is a continuous projection.
 
 Conversely, let $\pi \colon E \rightarrow E$ be a continuous projection with $\pi (E)=F$. Then $X \coloneq\text{ker} \pi$ is a closed subspace of $E$ and $F \times X \rightarrow E, (f,x) \mapsto f+x$ is a continuous linear map with continuous inverse $e \mapsto (\pi(e),e-\pi(e))$.
 
 If $F$ is complemented, we have an associated continuous projection and see that $F = \text{ker} (\id_E-\pi)$ is closed.
\end{proof}

\begin{ex}
 Finite-dimensional subspaces of locally convex spaces are always complemented, \cite[Lemma 4.21]{Rudin} (thus all subspaces of a finite-dimensional locally convex space are complemented). More general, every closed subspace of a Hilbert space is complemented, \cite[Theorem 12.4]{Rudin}. 
 
 Note however, that complemented subspaces (e.g. of Banach spaces) may be rare. Indeed one can prove that a Banach space for which every closed subspace is complemented, must already be a Hilbert space (see \cite{LaT71}). Moreover, there are examples of infinite-dimensional Banach spaces whose only complemented subspaces are finite dimensional.
\end{ex}

\begin{ex}\label{ex: notcomplemented}
 Consider the Banach space $c_0$ of all (real) sequences converging to $0$ as a subspace of the Banach space $\ell^\infty$ of all bounded real sequences, with the norm 
 $$\lVert (x_1,x_2,\ldots)\rVert_\infty = \sup_{n \in \N} |x_n|.$$ Then $c_0$ is not complemented in $\ell^\infty$. The proof is however more involved, whence we defer to \cite[Satz IV.6.5]{MR1787146}.
\end{ex}

\begin{Exercise}  \vspace{-\baselineskip}
 \Question Let $(E_i)_{i \in I}$ be a family of locally convex spaces. Prove that $\prod_{i \in I}E_i \coloneq \{(x_i)_I \mid x_i \in E_i\}$ with componentwise addition and scalar multiplication and endowed with the product topology is a locally convex space.
 \Question Show that $F \subseteq E$ is complemented if and only if the projection $q\colon E \rightarrow E/F$ has a continuous linear right inverse $\sigma \colon E/F \rightarrow E$ (i.e.\ $q\circ \sigma = \id_{E/F}$).
\end{Exercise}

\section{On smooth bump functions}\label{smooth:bumpfun}

In finite-dimensional differential geometry one uses commonly local to global arguments employing smooth bump functions (also sometime called cut-off functions) and partitions of unity. This strategy fails in general due to a lack of bump functions. We briefly discuss the problem and refer to \cite[Chapter III]{KM97} for more information. 
\begin{defn}
 For a map  $f \colon V \rightarrow F$ with $V \opn E$ and $E,F$ locally convex spaces the \emph{carrier}\index{carrier of a function} of $f$ is the set $\text{carr}(f)\coloneq \{x \in V \mid f(x) \neq 0\}$. As usual the \emph{support}\index{support of a function} of $f$ is defined to be the closure of the carrier.
\end{defn}
\begin{defn}
 Let $E$ be a locally convex space, $x \in E$ and $U \subseteq E$ an $x$-neighbourhood.
 \begin{enumerate}
  \item A $C^k$-map $f \colon E \rightarrow [0,\infty[$ for $k \in \N_0 \cup \{\infty\}$ is a $C^k$-\emph{bump function}\index{bump function} with carrier in $U$ if $\text{carr}(f) \subseteq U$ and $f(x)=1$. 
  \item if $(U_i)_{i\in I}$ is an open cover of $E$, we say that a family $(f_i \colon E \rightarrow [0,\infty[)_{i\in I}$ is a $C^k$\emph{-partition of unity}\index{partition of unity} (subordinate to the cover) if every $f_i$ is $C^k$ with carrier in $U_i$ and $\sum_i f_i(x) = 1, \forall x \in E$. 
 \end{enumerate}
\end{defn}

\begin{defn}
Let $E$ be a locally convex space and $k \in \N_0\cup \{\infty\}$. We say that $E$ is
\begin{enumerate}
 \item  \emph{$C^k$-regular},\index{space!$C^k$-regular} if for any neighbourhood $U$ of a point $x$ there exists a $C^k$-bump function with carrier in $U$ and $f(x)=1$.
 \item \emph{$C^k$-paracompact},\index{space!$C^k$-paracompact}\footnote{$C^0$-paracompact is equivalent to the usual notion of paracompactness due to \cite[Theorem 5.1.9]{Eng89}.} if every open cover admits a $C^k$-partitions of unity.  
\end{enumerate}
\end{defn}
Obviously similar definitions make sense if we consider instead of a locally convex space a manifold modelled on locally convex spaces. While the existence of partitions of unity hinges on topological properties of the manifold (paracompactness!), the smoothness of these partitions depends only on the availability of bump functions on the model space. 

\begin{setup}[Typical Local-to-global-argument]\index{local-to-global-argument}
 Let $M$ be a manifold which admits smooth partitions of unity subordinate to open covers. Assume we have an object defined for every chart, to illustrate this we choose smooth Riemannian metrics on $TU \cong U \times E$, i.e. $g_U \colon U\times H \times H \rightarrow \R, (u,h,k) \mapsto \langle h,k\rangle_u$. Using the chart we transport it back to the manifold, i.e.\ $g_\varphi \colon TU\oplus TU \rightarrow \R, (v,w) \mapsto \langle T \varphi(v),T\varphi(w)\rangle_{\varphi (\pi(v))}$. Now choose a smooth partition of unity $p_\varphi$ subordinate to the open covering $(U,\varphi)_{\varphi \in \mathcal{A}}$. Then 
 $$g \colon TM \oplus TM \rightarrow \R ,\quad (v,w) \mapsto \sum_{\varphi \in \mathcal{A}} p_\varphi (\pi_M(v)) g_\varphi(v,w)$$
 is a Riemannian metric on $M$. Note that if $g_\varphi (v,w)$ is not defined, $p_\varphi(\pi_M(v))$ is zero so the definition makes sense. 
\end{setup}

Recall (e.g.\, from \cite[Section 2.2]{Hir76}) that every finite-dimensional space is $C^\infty$-regular. It is also $C^\infty$-paracompact as a result by Toru\'{n}czyk (see \cite[Corollary 16.16]{KM97}) shows that every Hilbert space is $C^\infty$-paracompact. We shall now recall some results about $C^k$-regularity of locally convex (and in particular Banach spaces).

\begin{prop}[{\cite{MR0198492}}]
 A locally convex space is $C^k$-regular if and only if the topology is initial with respect to the functions $C^k(E,\R)$. 
\end{prop}

\begin{proof}
 The initial topology with respect to the $C^k$-functions is generated by the subbase
 $$f^{-1}(]a,b[), \qquad f \in C^{k}(E,\R), a,b \in \R \cup\{\pm \infty\}.$$
 Hence it is clear that if $E$ is $C^k$-regular that the topology is initial with respect to the $C^k$ functions. 
 For the converse consider $x \in U$ where $U$ is open in the initial topology. Then we find $a_1,\ldots, a_n, b_1,\ldots ,b_n$ and $f_1,\ldots ,f_n \in C^k (E,\R)$ for $n \in \N$ such that $x \in \bigcap_{1\leq i \leq n} f_i^{-1}(]a_i,b_i[)$. Adjusting choices of the $f_i$, we may assume without loss of generality that $a_i = - \varepsilon_i$ and $b_i = \varepsilon_i$ for some $\varepsilon_i >0$. Since $\R$ is $C^\infty$-regular, we pick $h \in C^\infty (\R,\R)$ with $h(0)=1$ and $h(t)=0,\ \forall |t|\geq 1$. Then $f \colon E \rightarrow \R, y\mapsto \prod_{1 \leq i \leq n} h(f_i(x)/\varepsilon_i)$ is a bump function with carrier in $U$.
\end{proof}

For a Banach space the existence of $C^k$-bump functions is tied to differentiability of the norm.
\begin{defn}\index{rough norm}
Let $(E,\lVert \cdot \rVert)$ be a normed space. The norm $\lVert \cdot \rVert \colon E \rightarrow \R$ is \emph{rough} if there
exists an $\varepsilon > 0$ such that for every $x \in E$ with $\lVert x\rVert = 1$ there exists $v \in E$ with $\lVert v\rVert =1$ and 
$$\limsup_{t \searrow 0} \frac{\lVert x + tv\rVert + \lVert x - tv\rVert -2}{t} \geq \varepsilon.$$ 
If a Banach space is $C^1$-regular then it does not admit a rough norm (see \cite[14.11]{KM97}).
\end{defn}

One can prove that the Banach spaces $(C([0,1],\R),\lVert \cdot \rVert_{\infty})$ (i.e.\, the continuous functions with the compact open topology) and $(\ell_1, \lVert \cdot \rVert_1)$ (see \cite[13.11 and 13.12]{KM97}) have rough norms, whence they are not even $C^1$-regular. 
On the other hand, since nuclear \Frechet\, spaces are $C^\infty$-regular, the space $C^\infty ([0,1],\R)$ is $C^\infty$-regular. Similar statements then hold for spaces of smooth sections into bundles. Again we refer to \cite{KM97}.

\section{Inverse function theorem beyond Banach spaces}\label{invfunction}
Before we conclude this chapter on locally convex spaces, let us briefly discuss (the lack of) an important tool from calculus which is driving many basic results in (finite dimensional) differential geometry.  Many basic existence results and constructions in finite dimensional differential geometry are more or less direct consequences of the inverse function theorem, the constant rank theorem and its cousin the implicit function theorem. Note that the inverse function theorem and its cousin the implicit function theorem still hold in Banach spaces, \cite[I, \S 5]{Lang} (while the constant rank theorem is already more delicate, see \cite{MR1173211}). Beyond Banach spaces this breaks down as the following example due to Hamilton \cite[I. 5.5.1]{MR656198} shows (cf.\ also \Cref{setup:naiveimmersion}):
 
 \begin{ex}\label{example:diffop}
 Consider the \Frechet space $C^\infty([-1,1],\R)$\index{Fr\'{e}chet space} of all smooth functions from $[-1,1]$ to $\R$.\index{compact open $C^\infty$-topology}\footnote{The \Frechet space structure is given pointwise operations with the compact open $C^\infty$-topology. It is defined via the metric $d(f,g) \coloneq \sum_{i=0}^\infty \frac{\lVert f-g\rVert_i}{2^i}$, where $\lVert f\rVert_i$ is the supremum norm of the $i$th derivative.} together with the differential operator 
 $$P \colon C^\infty ([-1,1],\R) \rightarrow C^\infty ([-1,1],\R),\quad P(c)(x) \coloneq c(x)-xc(x)c'(x)$$
 A computation shows that $P$ is $C^\infty$ with derivative $dP(c,g)(x)= g(x)-xg(x)c'(x)-xc(x)g'(x)$, i.e.\ for $c\equiv0$ the derivative is the identity. However, the image of $P$ is no zero-neighbourhood in $C^\infty ([-1,1],\R)$ as it does not contain any of the functions $g_n (x) \coloneq \tfrac{1}{n}+\tfrac{x^n}{n!}$ for $n\in \N$ (but $\lim_{n\rightarrow \infty} g_n = 0$ in  $C^\infty ([-1,1],\R)$). In conclusion, the inverse function theorem does not hold for $P$.
 \end{ex}  

To give a more geometric example, the exponential map of a Lie group (modelled e.g.\ on a \Frechet space) might fail to even be a local diffeomorphism around the identity. For example, this happens for diffeomorphism groups, cf.\ \Cref{ex:badexponential}.

\begin{setup}[How to recover an inverse function theorem]
The calculi discussed so far are too weak to provide an inverse function theorem on their own. If one has more information (such as metric estimates in the \Frechet setting) there are inverse function theorems which can apply in more general situations.
The most famous one is certainly the Nash-Moser inverse function theorem \cite{MR656198} which works with so called tame maps on tame spaces.
Further generalised theorems are available in the framework of M{\"u}llers bounded geometry \cite{MR2463806} and Gl{\"o}ckners inverse function theorems, cf.\ \cite{Glo07,MR2269430} and the references therein.
To keep the exposition short we do not provide details here. Note though that the generalisations mentioned require specific settings or certain estimates which are often hard to check in applications.
\end{setup}  
 
 A consequence of the lack of an inverse function theorem is that in infinite-dimensional differential geometry one needs to be careful when considering the notions of immersion and submersion (cf.\ \Cref{sect:subm}).
 Further, there is no general solution theory for ordinary differential equations beyond Banach spaces (even for linear differential equations!).
\begin{Exercise}\label{Ex:diffop}  \vspace{-\baselineskip}
 \Question Fill in the missing details for \Cref{example:diffop}: Show that the differential operator $P$ is differentiable and compute its derivative. For $g_n (x) = \tfrac{1}{n}+\tfrac{x^n}{n!}$ show that $g_n \rightarrow 0$ in the compact open $C^\infty$-topology. Let $f$ be a smooth map on $[-1,1]$. Develop $f$ and $P(f)$ into their Taylor series at $0$. Then show that $P(f)=g_n$ is impossible. 
\end{Exercise}

\section{Differential equations beyond Banach spaces}\label{Diffeq:beyond}
Here we exhibit several examples of ill posed differential and integral equation on locally convex spaces. 

\begin{ex}[No solutions in incomplete spaces (Dahmen)]\label{Dahmensexample}
 Consider the mapping 
 $$h \colon [0,1] \rightarrow C^\infty ([1,2],\R),\quad t \mapsto (x \mapsto x^t).$$
 We apply the exponential law to $h$. Note however that \Cref{thm:explaw} is not quite sufficient as $[0,1]$ and $[1,2]$ are manifolds with boundary, whence we refer to the more general exponential law \cite[Theorem A]{MR3342623}. Now $h$ is smooth if and only if $h^\wedge \colon [0,1] \times [1,2] \rightarrow \R , (t,x) \mapsto x^t = \exp (\ln (x) t)$ is smooth, i.e.~$h^\wedge$ is smooth on the interior of the square $[0,1] \times [1,2]$ such that the derivatives extend continuously onto the boundary. This is a trivial calculation. Since the derivative of $h$ corresponds via the exponential law to the partial derivative of $h^\wedge$ (cf.\ \Cref{lem:fwedge_vector}), we see that $dh(t;y)(x) = (d_1 h^\wedge(t,\cdot;y))^\vee (x) = y\ln(x)x^t$. 
 Let us define now two subspaces of $C^\infty ([1,2],\R)$ as the locally convex spaces generated by the image of $h$ and $h' = dh(\cdot ; 1)$: 
 $$E \coloneq \text{span} \{h(t) \mid t \in [0,1]\}, \qquad F \coloneq \text{span} \{h'(t) \mid t \in [0,1]\}.$$ 
 By construction $h(t) \not \in F$ and $h'(t) \not \in E$. This entails that  $h|^{E} \colon [0,1] \rightarrow E \subseteq C^\infty ([1,2],\R)$ is not differentiable and shows that sequential closedness is indispensable in \Cref{lem:seq-closed}. As a consequence both subspaces are not closed and in particular not (Mackey) complete.  
 We see that the (trivial) differential equation $\gamma'(t) = h'(t)$ or equivalently the (weak) integral equation $\gamma (t) = \int_0^1 h'(t
 )\mathrm{d}t$ does not admit a solution in $F$. 
\end{ex}

It should not come as a surprise that in the absence of suitable completeness properties differential and integral equations may be ill posed. However, even in complete spaces relative benign (e.g.~linear) differential equations do not admit solutions.

\begin{ex}[{\cite[I.5.6.1]{MR656198}}]
Consider the \Frechet space $C^\infty ([0,1],\R)$ of smooth functions endowed with the compact open $C^\infty$-topology. Recall that the differential operator $D \colon C^\infty ([0,1],\R) \rightarrow C^\infty ([0,1],\R) ,\ f \mapsto f' = df(\cdot;1)$ is continuous and linear. Consider a solution $f \colon ]-\varepsilon,\varepsilon[ \rightarrow C^\infty ([0,1],\R)$ of the linear differential equation
\begin{align}\label{eq:lindiff}
 \begin{cases}\frac{\mathrm{d}}{\mathrm{d} t} f(t) &= D(f) = f',\\
  f(0)&=f_0
 \end{cases}
\end{align}
Apply the exponential law \cite[Theorem A]{MR3342623} to the smooth map $f$. We obtain a smooth map $f^\wedge \colon ]\varepsilon , \varepsilon[ \times [0,1] \rightarrow \R, (t,x) \mapsto f(t)(x)$ such that \eqref{eq:lindiff} is equivalent to the partial differential equation
\begin{align}\label{assocPDE}
 \partial_t f^\wedge(t,x) = \partial_x f^\wedge (t,x) \quad f(0,x)=f_0(x)
\end{align}
By the Whitney extension theorem \cite{Whi34} there is a (non unique) extension $F_0 \in C^\infty ([0,2],\R)$ of $f_0$ and it is easy to see that the function $f(t,x) \coloneq F_0 (x+t), t,x \in [0,1]$ solves \eqref{assocPDE}. Since the extension $F_0$ is non unique, the solution to \eqref{eq:lindiff} is non unique (albeit we study a linear ordinary differential equation with smooth right hand side!).
\end{ex}

A related example is given by the heat equation on the circle
\begin{ex}[{Heat equation on $\SSS^1$, \cite[Example 6.3]{Mil82}}]\label{ex:heat}
 The heat equation on the circle $\SSS^1$ is given by 
 \begin{align*}
  \partial_t f(t,\theta) = \partial_\theta^2 f(t,\theta),
 \end{align*}
 where $\partial_\theta$ is the derivative on $\SSS^1$. Again the derivative induces a continuous linear derivative operator $D \colon C^\infty (\SSS^1,\R) \rightarrow C^\infty (\SSS^1,\R), f \mapsto \frac{\mathrm{d}}{\mathrm{d}\theta} f$, whence the heat equation can be understood as an ordinary differential equation on $C^\infty (\SSS^1,\R)$
 We do not go into the details concerning solutions of this equation. However, the reader may want to refer to \Cref{sect:EAtheory} for examples of partial differential equations which are treated by similar techniques as ordinary differential equations on infinite-dimensional manifolds.
\end{ex}

The following examples are due to Milnor (see \cite[Example 6.1 and 6.2]{Mil82}):
\begin{ex}[Too many solutions]
 Let $\R^\N$ be the \Frechet space of real valued sequences (with the topology induced by identifying $\R^\N \cong \prod_{n \in \N} \R$) and define the \emph{left shift}\begin{align*}
    \Lambda \colon \R^\N \rightarrow \R^\N ,\quad  (x_1,x_2,x_3,\ldots )   \mapsto (x_2,x_3,\ldots).
\end{align*}
 Then $\Lambda$ is continuous and linear and the differential equation $\mathbf{y}'(t) = \Lambda(\mathbf{y})(t)$ reduces to the system of equations $y_i'(t)=y_{i+1} (t), i\in \N$. For every initial value this system has infinitely many solutions. 
 
 To see this consider the initial condition $\mathbf{y}(0)=\mathbf{0}\in \R^\N$. If $\varphi \in C^\infty (\R,\R)$ vanishes in a neighbourhood of $0$, then we set $y_1(t)=\varphi(t)$ and $y_{n}(t) = \frac{\mathrm{d}^{n-1}}{\mathrm{d}t^{n-1}}\varphi (t)$ for $n \geq 2$ to obtain a solution $\mathbf{y}(t)=(y_1(t),y_2(t),\ldots)$ of the equation with $\mathbf{y}(0)=\mathbf{0}$.
\end{ex}

\begin{ex}[No solutions]\label{ex:boxtop}
 The space $E\coloneq \{(x_n)_{n \in \N} \in \R^\N \mid \text{ almost all } x_n = 0\}$ is locally convex with respect to the box topology (i.e.\ the topology whose basis is given by the sets $\prod_{n\in \N} U_n, U_n \opn \R$, cf.\ Exercise \ref{Ex:boxtop} 2.). Define the \emph{right shift}
 \begin{align*}
   R\colon E\rightarrow E ,\quad  (x_1,x_2,x_3,\ldots )   \mapsto (0, x_1,x_2,x_3,\ldots )
\end{align*}
Then it is not hard to see that $R$ is continuous linear and we consider the initial value problem 
$$\begin{cases}
   \mathbf{y}' &= R(\mathbf{y})(t)\\
   \mathbf{y}(0) &= (1,0,0\ldots)
  \end{cases}
$$
Note that for a prospective solution $\mathbf{y}(t)=(y_1(t),y_2(t),\ldots)$ the differential equation yields $y_1'(t) = 0$, whence $y_1 (t) \equiv 1$ by the initial condition. Then $y_2'(t) = y_1(t) = 1$. Integrating, we see that inductively $y$ is a solution if $y_i(t) = \frac{t^{i-1}}{(i-1)!}$. However, for $t\neq 0$ this sequence has infinitely many terms not equal to $0$ and thus does not exist in $E$, i.e. the initial value problem (given by a linear differential equation with smooth right hand side!) does not have any solution in $E$.
\end{ex}

\begin{Exercise}\label{Ex:boxtop}  \vspace{-\baselineskip}
 \Question Review \Cref{sect:explaw} and \Cref{thm:explaw} to work out the details of the identification of the PDE heat equation with the ODE on $C^\infty (\SSS^1,\R)$ in \Cref{ex:heat}.
 \Question Consider the space $E\coloneq \{(x_n)_{n \in \N} \in \R^\N \mid \text{ almost all } x_n = 0\}$ with the box topology (i.e.\ the topology induced by $E \cap \prod_{n\in \N} U_i$, where $U_i \opn \R$. Show that 
 \subQuestion $E$ is a locally convex space and the box topology is properly finer then the subspace topology induced by $E \subseteq \R^\N$. Then show that $R$ is continuous.
 \subQuestion every base of $0$-neighbourhoods in $E$ is necessarily uncountable, so $E$ can not be a metrisable space. 
 \end{Exercise}
\setboolean{firstanswerofthechapter}{true}
\begin{Answer}[number={\ref{Ex:boxtop} 3.~a)}] 
 \emph{We show that the box topology turns  $E\coloneq \{(x_n)_{n \in \N} \in \R^\N \mid \text{ almost all } x_n = 0\}$ into a locally convex space and makes the right shift continuous}\\[.5em]
 
 Note that we have $\lambda(\prod_n U_n) +(\prod_n V_n) = \prod_{n \in \N}(\lambda U_n + V_n)$ for $U_n , V_n \opn \R$ and $\lambda \in \R$. If $(x_n)+(y_n) \in \prod_n U_n$ we exploit that $\R$ is a topological vector space to find for every $n \in \N$ $x_n \in A_n \opn \R$ and $y_n \in B_n \opn \R$ with $A_n + B_n \subseteq U_n$. We conclude that $\left(\prod_n A_n\right) + \left(\prod_n B_n\right) \subseteq \left(\prod_n U_n\right)$ and vector addition is continuous. For scalar multiplication we note that if $\lambda \cdot (x_n) \in \prod_n U_n$ we can exploit that $x_n \neq 0$ for only finitely many $n$, to construct an open $\lambda$-neighbourhood $\lambda \in O \opn \R$ together with $(x_n) \in \prod_n V_n$ such that $O \cdot \prod_{n} V_n \subseteq \prod_n U_n$. Hence scalar multiplication is continuous and $E$ is a TVS. Let now $\prod_n U_n$ be a $0$-neighbourhood. This implies that every $U_n$ is a zero neighbourhood in $\R$. Since $\R$ is locally convex, for every $n$ there is a convex $0$-neighbourhood $C_n \subseteq U_n$. Then $\prod_n C_n \subseteq \prod_n U_n$ is a convex $0$-neighbourhood and thus $E$ is locally convex. The right shift is continuous as $R^{-1}(\prod_n U_n)=\prod_n U_{n+1}$ if $0\in U_1$ and $\emptyset$ otherwise. 
 \end{Answer}
 \setboolean{firstanswerofthechapter}{false}
 
 \section{Another approach to calculus: Convenient calculus}\label{convenient:calculus}

Convenient calculus was introduced by Fr\"ohlicher and Kriegl \cite{MR961256} (see \cite{KM97} for an introduction). To understand the basic idea, we recall a theorem by Boman:
 
 \begin{thm}[Boman's theorem, {\cite{MR0237728}}] \index{Boman's theorem} A map $f \colon \R^d \rightarrow \R$, $d\in \N$ is smooth if and only if for each smooth curve $c \colon \R \rightarrow \R^d$ the curve $f \circ c$ is smooth.\footnote{We emphasise here that smoothness can be tested against smooth curves, while this becomes false for finite orders of differentiability.} 
 \end{thm}

Now smoothness of curves $c \colon \R \rightarrow E$ with values in a locally convex space $E$ is canonically defined (e.g.\ via Definition \ref{defn: smoothcurve}). Hence smooth curves can be used to define (conveniently) smooth maps on locally convex spaces which are Mackey complete\footnote{See \Cref{Mackeycomplete}. Mackey completeness is weaker then sequential completeness.}. Following the usual lingo of convenient calculus, a locally convex space which is Mackey-complete is called \emph{convenient vector space}.\index{space!convenient vector space}\index{Mackey complete}
 
 \begin{defn}\index{convenient calculus} \index{convenient smooth map}
 Let $E,F$ be convenient vector spaces.
 \begin{enumerate}
 \item Write $c^\infty E$ for $E$ endowed with the final topology with respect to all smooth curves $\R \rightarrow E$. We call this topology the $c^\infty$\emph{-topology} and its open sets $c^\infty$\emph{-open}.
 \item Let $U \subseteq E$ be $c^\infty$-open,
$f \colon U \rightarrow F$ a map. Then $f$ is \emph{convenient smooth} or $C^\infty_{\text{conv}}$ if $f\circ c \colon \R \rightarrow F$ is a smooth curve for every smooth curve $c\colon \R \rightarrow U$.
\end{enumerate} 
 \end{defn}

Obviously, the chain rule holds for conveniently smooth mappings and one can define and study derivatives. Once again manifolds, tangent spaces etc.\ make sense. Further, Boman's theorem asserts that between finite dimensional spaces convenient smooth coincides with \Frechet smooth (more on this in \ref{setup: conv:bas} below).
 
 \begin{rem}[Bornology vs.\ topology]
 By now, the reader will have wondered why one defines $C^\infty_{\text{conv}}$-maps on $c^\infty$-open subsets instead of using the native topology on $E$. 
 The reason for this is that differentiability of curves into a locally convex space does not depend on the topology of $E$, but rather on the bounded sets, the \emph{bornology} of $E$.\index{bornology} One can show that smoothness of curves is a bornological concept, and this is captured by the $c^\infty$-topology (which is finer then the native topology but induces the same bornology).
 
The $c^\infty$-topology is somewhat delicate to handle. For example, $c^\infty (E\times F) \neq c^\infty E \times c^\infty F$ (in general) and $c^\infty E$ will not be a topological vector space. However, it can be shown that for a \Frechet space the $c^\infty$-topology coincides with the \Frechet space topology, cf.\ \cite[Section I.4]{KM97}. 
 \end{rem}

By definition a conveniently smooth map is continuous with respect to the $c^\infty$-topology on $E$. However, as the $c^\infty$-topology is finer then the native locally convex topology conveniently smooth maps may fail to be continuous with respect to the native topologies.
To stress it once more: \textbf{Differentiability in the convenient calculus is \emph{not built on top of continuity}}.  
Having now defined two notions of calculus we will clarify their relation in the next section and discuss some of their properties.

\subsection*{Bastiani vs.\ convenient calculus}

 We have already seen in the last sections that both the convenient and the Bastiani calculus yield the well known concept of (\Frechet) smooth maps between finite dimensional vector spaces. To clarify the relation between Bastiani and convenient calculus observe that the definition of smooth curves into a locally convex space coincides in both calculi. The Bastiani chain rule, \Cref{chainrule}, yields the following.\index{Bastiani calculus}
 
 \begin{lem}
 Let $E,F$ be convenient vector spaces, $U \opn E$ and $f\colon E \supseteq U \rightarrow F$ a $C^\infty$-map. Then $f$ is $C^\infty_{\text{conv}}$.
 \end{lem} 
 
 Thus (completeness properties aside) Bastiani smoothness is the stronger and more restrictive concept which enforces continuity with respect to the native topologies. However, on \Frechet spaces both calculi coincide \cite[Theorem 12.4]{MR2069671}.
 
 \begin{setup}\label{setup: conv:bas}
 Let $E$, $F$ be convenient spaces, $U \opn E$, then the differentiability classes of a map $f \colon U \rightarrow F$ are related as follows:
\begin{equation}\label{diag: rel:conv:Bas}
 \begin{tikzcd}
C^\infty_{\text{conv}} \arrow[rr, bend right, Rightarrow, "E \text{ \Frechet}"'] \arrow[rr, Leftarrow]
& & C^\infty \arrow[rr, Leftrightarrow,"E {,} F\text{ normed}"'] & &  FC^\infty 
\end{tikzcd} 
\end{equation}
The dividing line between convenient calculus and Bastiani calculus is continuity, see \cite{MR2213538} for examples of discontinuous conveniently smooth maps. Also cf.\ \cite[Theorem 4.11]{KM97} for more information on spaces on which the concepts coincide.
 \end{setup}   
One may ask oneself now if there is one calculus which is preferable over the other.

\begin{setup}[Bastiani calculus is more convenient than convenient calculus]
A major difference between Bastiani and convenient calculus is continuity.
Arguably continuity with respect to the native topologies, as in the Bastiani calculus is desirable for smooth maps. In particular, the infinite-dimensional spaces and manifolds often have an intrinsic topology one would rather like to preserve instead of having to deal with the somewhat delicate $c^\infty$-topology.
This is one reason why in infinite-dimensional Lie theory, Bastiani calculus is prevalent (continuity allows one to use techniques from topological group theory, such as the local to global result for Lie groups \Cref{loc:Liegrp}). 

In addition one can often conveniently establish Bastiani smoothness using induction arguments over the order of differentiability\footnote{Here one should mention that a similar but somewhat more involved notion of $k$-times Lipschitz differentiable mappings exists in the convenient setting.} and one can interpret this (together with the continuity) as an argument for the naturality of Bastiani calculus.
\end{setup} 

\begin{setup}[Convenient calculus is more convenient than Bastiani calculus]
Discarding continuity for smooth maps might not be as exotic after all since smoothness of curves is a bornological and not a topological property. Further, it leads to a convenient category of spaces with smooth maps:

To explain this, consider sets $X,Y,Z$ and denote by $Z^X$ the set of all maps from $X$ to $Z$. Then $f \in Z^{X \times Y}$ induces a map $f^\vee \in (Z^{Y})^X$ via $f^\vee (x)(y)= f(x,y)$. The resulting bijection $Z^{X \times Y} \cong (Z^Y)^X$ is known as the exponential law.\smallskip

\textbf{Exponential law for convenient smooth maps}\emph{
Let $E,F,G$ be locally convex vector spaces, $U \subseteq E, V\subseteq F$ $c^\infty$-open. Then the spaces of convenient smooth maps admit a locally convex topology such that there is a linear convenient smooth diffeomorphism}
$$C^\infty_{\text{conv}} (U\times V , G) \cong C^\infty_{\text{conv}} (U, C^\infty_\text{conv} (V,G)),\quad f \mapsto f^\vee$$ 
\emph{In particular, $f^\vee \colon U \rightarrow C^\infty_\mathrm{conv} (V,G)$ is $C^\infty_{\mathrm{conv}}$ if and only if $f\colon U\times V \rightarrow G$ is} $C^\infty_{\text{conv}}$.\smallskip

The Exponential law is an immensely important tool, simplifying many proofs (e.g.\ that the diffeomorphism group is a Lie group). It also establishes cartesian closedness of the category of convenient vector spaces with convenient smooth maps.\footnote{The "convenient" in convenient calculus, references Steenrod's "A convenient category of topological spaces", \cite{MR0210075} in which a cartesian closed category of topological spaces is built.} Note that the Exponential law and cartesian closedness do not hold in the Bastiani setting.\footnote{Albeit rudiments of an Exponential law exist in the Bastiani setting as \Cref{sect:explaw} shows. See also \cite{MR3342623} for a stronger version.} We remark though that cartesian closedness in the convenient setting is a statement about the category of convenient vector spaces. The result does not carry over to the category of manifolds modelled on convenient vector spaces which turns out to be \emph{not cartesian closed}. To obtain a cartesian closed category of manifolds it seems to be unavoidable to pass to even more general concepts such as diffeological spaces (see \cite{MR3025051}). 
\end{setup}
 
 Summing up, there is no clear cut answer to the question which calculus is preferable. It will depend on the application or use one has in mind whether one goes with the stronger notion of Bastiani calculus or the weaker convenient calculus (which often has more convenient tools).

\chapter{Basics from topology}\label{app:topo}\copyrightnotice

We are assuming that the reader is familiar with the basic concepts of topology, we review some concepts used for the readers convenience. All of the concepts in this appendix are covered in standard textbooks on topology, e.g. \cite{Eng89}. We mention again that we write $U \opn X$ if $U$ is an open subset of a topological space $X$.

\section{Initial and final topologies}
Let $X$ be a set and $(f_i)_{i\in I}$ be a family of mappings $f_i \colon X \rightarrow Y_i$ to topological spaces $Y_i$. The coarsest topology on $X$ making each $f_i$ continuous is called the \emph{initial topology on $X$ with respect to the mappings $(f_i)_{i\in I}$.}\index{initial topology}
Dually, if $\{g_i \colon Y_i \rightarrow X\}_{i \in I}$ is a family of mappings from topological spaces to $X$ there is a finest topology making all $g_i$ continuous. This topology is called the \emph{final topology}\index{final topology} with respect to the $g_i$.

\begin{rem}\label{rem:inittop}
 Note that the sets $\bigcap_{i \in F} f_i^{-1} (U_i),$ where $F \subseteq I$ is finite and $U_i \opn Y_i$ form a basis of the initial topology topology (even if we restrict our choice of $U_i$ to a basis of the topology of $Y_i$).
\end{rem}

\begin{lem}
The initial topology on $X$ with respect to a family $(f_i)_{i\in I}$ is the unique topology which satisfies:
$g\colon Z \rightarrow X$ is continuous if and only if $f_i \circ g \colon Z \rightarrow Y_i$ is continuous for every $i\in I$
\end{lem}

\begin{proof}
 Assume first that $X$ carries the initial topology $\mathcal{I}$. Then clearly if $g$ is continuous $f_i \circ g$ is continuous for every $i\in I$. Conversely, since $\mathcal{S}\coloneq \{f_i^{-1}(V_i) \mid V_i \opn Y_i\}$ is a subbase for $\mathcal{I}$ by \Cref{rem:inittop}, we see that $f_i \circ g$ continuous implies that $g^{-1}(W)$ is open in $Z$ for each $W \in \mathcal{S}$. Thus $g$ is continuous. We conclude that $\mathcal{I}$ has the claimed property.

 Let now $\mathcal{T}$ be another topology which satisfies the above property. Then the identity maps $\id \colon (X,\mathcal{I}) \rightarrow (X,\mathcal{T})$ and $\id \colon (X,\mathcal{T}) \rightarrow (X,\mathcal{I})$ are continuous, whence both topologies coincide.
\end{proof}

\begin{ex}
 If $X$ is a subset of a topological space $Y$, the \emph{induced}, or \emph{subspace topology} is the initial topology with respect to the inclusion $\iota \colon X \rightarrow Y, x\mapsto x$.
\end{ex}

\begin{ex}
 We always endow the cartesian product $X \coloneq \prod_{i\in I} X_i$ of a family $(X_i)_{i\in I}$ of topological spaces with the \emph{product topology}. This topology is the initial topology with respect to the family $(\text{pr}_i)_{i\in I}$ of projections $\text{pr}_j ((x_i)_I) \coloneq x_j$. 
\end{ex}

  \begin{Exercise} \label{Ex:final} \vspace{-\baselineskip}
   \Question Let $X$ be a topological space endowed with the final topology with respect to a family of mappings $f_i \colon X_i \rightarrow X$. Show that $g\colon X \rightarrow Y$ is continuous if and only if $g\circ f_i$ is continuous for all $i\in I$.
   \Question Let $(E_i)_{i\in I}$ be locally convex spaces and $E \coloneq \{(x_i)_{i\in I} \in \prod_{i \in I} E_i\mid \text{almost all } x_i =0\}$. Endow $E$ with the final topology (see \Cref{app:topo}) with respect to the family of inclusions $\iota_j\colon E_j \rightarrow E, x \mapsto (x_i)_{i \in I}$ with $x_j = x$ and $x_i= 0$ if $i\neq j$. Show that 
   \subQuestion the resulting topology is the box topology, i.e.\ the topology generated by the base of sets $E \cap \prod_{i \in I} U_i$, where for every $i\in I$, $U_i$ runs through a topological base of $E_i$.
   \subQuestion if $I$ is a \textbf{countable} set and $f_i \colon E_i \rightarrow F$ is a continuous linear map to a locally convex space $F$, then there exists a unique continuous linear map $f \colon E \rightarrow F$ with $f \circ \iota_j = f_j$. This proves in particular that the box topology turns $E$ into the direct locally convex sum of the spaces $E_i$.\\ 
   {\footnotesize \textbf{Hint:} For continuity of $f$ consider the preimage of a $0$-neighborhood $V$ in $F$. Construct inductively a sequence $(V_n)_{n \in \N}$ of $0$-neighborhoods with $V_n + V_n \subseteq V_{n+1}$. }
   \subQuestion Let $I$ be uncountable and $E_i = \R$ for all $i\in I$ Show that the summation map $s \colon \{(x_i)_i \in \R^I \mid x_i = 0 \text{ for almost all }i\} \rightarrow \R, (x_i) \mapsto \sum_{i\in I} x_i$ is discontinuous in the box topology. Thus the box topology is properly coarser than the direct sum topology in this case.
  \end{Exercise}

\section{The compact open topology}\label{sect:cptopntop}

Let $X,Y$ be (Hausdorff) topological spaces and $C(X,Y)$ the set of all continuous mappings from $X$ to $Y$.
We define a topology on $C(X,Y)$ by declaring a subbase consisting of the following sets
$$\lfloor K , U \rfloor \coloneq \{f \in C(X,Y) \mid f(K) \subseteq U\},\qquad K \subseteq X \text{compact}, U \opn Y.$$
The resulting topology is called the \emph{compact open topology}\index{compact open topology} and we write $C(X,Y)_{\text{c.o.}}$ for the set of continuous mappings with this topology.\footnote{If you prefer a video walkthrough covering (parts of ) the material in this appendix, then take a look at \url{https://www.youtube.com/watch?v=vGs-C9eEdJ0}.}

\begin{rem}\label{rem:pointeval}
 As singletons are compact, we see that the inclusion map 
 $$C(X,Y)_{\text{c.o.}} \rightarrow Y^X \coloneq \prod_{x \in X} Y$$
 is continuous if we equip the right hand side with the product topology. Thus $C(X,Y)_{\text{c.o.}}$ will again be Hausdorff and the point evaluations $\ev_x \colon C(X,Y)_{\text{c.o.}} \rightarrow Y, \ev_x (f) \coloneq f(x)$ are continuous for every $x \in X$.
\end{rem}

We will now show that if the target is a locally convex spaces the compact open coincides with the topology of compact convergence.

\begin{setup}\label{setup:compact_convergence}
 Let $(E,\{p_i\}_{i\in I})$ be a locally convex space and $K \subseteq X$ be a compact subset of a topological space. Then we define a seminorm on $C(X,E)_{\text{c.o.}}$ via
 $$\lVert f \rVert_{p_i,K} (f) \coloneq \sup_{x \in K} p_i (f(x)).$$
 Note that these seminorms are separating, since for $\gamma \in C(X,E)$ with $\gamma \neq 0$ we find $x \in X$ with $\gamma(x)\neq 0$ and thus $\lVert \gamma \rVert_{p_i , \{x\}} \neq 0$ for some $i\in I$.
 The locally convex topology generated by all seminorms $(\lVert \cdot \rVert_{p_i,K})_{i\in I , K \subseteq X \text{compact}}$ is called \emph{topology of compact convergence}.\index{topology of compact convergence}
 \end{setup}

 \begin{lem}\label{lem:co_vs_uniform}
 Let $X$ be a topological space and $E$ a locally convex space, then the compact-open topology coincides with the topology of compact convergence on $C(X,E)$. 
 As a consequence, $C(X,E)_{\text{c.o.}}$ is again a locally convex space if $E$ is locally convex.
 \end{lem}

 \begin{proof}
  First note that every seminorm $\lVert \cdot \rVert_{K,p}$ for $K \subseteq X$ compact and $p$ a continuous seminorm on $E$ we have
  \begin{align}\label{identity}
  \{f \in C(X,E)\mid \lVert f\rVert_{K,p} < r\} = \lfloor K , \{ x \in E \mid p(x)<r\}\rfloor.
  \end{align}
  To see that the topology of compact convergence is finer then the compact open topology, it suffices to prove that all sets $\lfloor K,U\rfloor$ with $K \subseteq X$ compact and $U \opn E$ are open in that topology.
  Let now $\gamma \in \lfloor K , U \rfloor$. Then $\gamma(K) \subseteq U \opn E$ is compact and by \Cref{lem:spec_subsets} (g) there is an open $0$-neighborhood $W \opn E$ such that $\gamma(K) + W \subseteq U$. We choose a seminorm $p$ such that $V \coloneq \{ x \in E \mid p(x)<r\} \subseteq W$. Then $\gamma + \lfloor K, V\rfloor$ is an open neighborhood of $\gamma$ in the topology of compact convergence by \eqref{identity}. Moreover, $\gamma + \lfloor K, V\rfloor$ is contained in $\lfloor K,U\rfloor$ since $\gamma(x)+\eta(x) \in \gamma(K) + V \subseteq U$ for all $\eta \in \lfloor K , V\rfloor$ and $x \in K$. This shows that $\lfloor K,U\rfloor$ is a $\gamma$-neighborhood in the topology of compact convergence, whence open.
  
  Conversely, let $\gamma \in C(X,E)$ due to \Cref{lem:spec_subsets} (e) and \eqref{identity}, we see that the sets $\gamma + \lfloor K, \{y\in E \mid p(y)<r\}\rfloor$ form a basis of open $\gamma$-neighborhoods in the topology of compact convergence. Thus it suffices to prove that these sets are open in the compact open topology. For this choose $0 \in V \opn E$ such that $V - V \subseteq \{y\in E \mid p(y)<r\}$. As $\gamma|_K$ is continuous, we find for every $x \in K$ an $x$-neighborhood $K_x \opn K$ with $\gamma(K_x) \subseteq \gamma(x)+V$. Using compactness we choose a finite set $F\subseteq K$ with $K = \bigcup_{x\in F}K_x$.
  We will now show that 
  $$\Omega_\gamma \coloneq \bigcap_{x \in F} \lfloor K_x , \gamma(x) + V\rfloor \subseteq \gamma + \lfloor K, \{y\in E \mid p(y)<r\}\rfloor.$$
  If $\eta \in \Omega_\gamma$, then we find for every $y \in K$ a $x \in F$ with $y \in K_x$. Thus $(\gamma-\eta)(y) \in \gamma(x) + V - (\gamma(x) +V) \subseteq V -V \subseteq U$ and thus $\gamma - \eta \in \lfloor K, \{y\in E \mid p(y)<r\}\rfloor$. Now $\Omega_\gamma$ is open in the compact-open topology and contains $\gamma$ by choice of the $K_x$. Thus both topologies coincide.
  \end{proof}

  \begin{lem}\label{lem:pushforward}
   Let $A, X,Y,Z$ be topological spaces and $h \colon A \rightarrow X, f \colon Y \rightarrow Z$ be continuous maps. Then the \emph{pushforward} and the \emph{pullback} map \index{pushforward} \index{pullback} 
   \begin{align*}
    f_* \colon C(X,Y)_{\text{c.o.}} \rightarrow C(X,Z)_{\text{c.o.}},\quad g \mapsto f\circ g\\
    h^* \colon C(X,Y)_{\text{c.o.}} \rightarrow C(A,Y)_{\text{c.o.}},\quad g \mapsto g\circ h
   \end{align*}
   are continuous.
  \end{lem}

  \begin{proof}
  We begin with the pushforward $f_*$ and show that $f_*^{-1}(\lfloor K,U\rfloor)$ is open for every $K \subseteq X$ compact and $U \opn Y$. For $\gamma \in C(X,Y)$ we see that 
   \begin{align*}
   f_*(\gamma) = f\circ \gamma \in \lfloor K,U\rfloor \Leftrightarrow f(\gamma(K))\subseteq U \Leftrightarrow \gamma(K) \subseteq f^{-1}(U) \Leftrightarrow \gamma \in \lfloor K, f^{-1}(U)\rfloor
   \end{align*}
  Thus $f_*^{-1}(\lfloor K,U\rfloor) =\lfloor K,f^{-1}(U)\rfloor \opn C(X,Y)_{\text{c.o}}$ by continuity of $f$.
  
  Now for the continuity of the pullback $h^*$, pick $L \subseteq A$ compact and $V \opn X$. Then for $\gamma \in C(X,Y)$, we have $h^*(\gamma)= \gamma \circ h \in \lfloor L,V\rfloor$ if and only if $\gamma(h(L))\subseteq V$. In other word, if and only if $\gamma \in \lfloor h(L),V\rfloor$. By continuity of $h$, $h(L)$ is compact and thus $(f^*)^{-1}(\lfloor L,V\rfloor)=\lfloor h(L),V\rfloor \opn C(X,Y)_{\text{c.o.}}$ and the pullback is continuous. 
  \end{proof}

  \begin{lem}
   If $X,Y,Z$ are topological spaces and $\text{pr}_Y \colon Y\times Z \rightarrow Y$ and $\mathrm{pr}_Z \colon Y\times Z \rightarrow Z$ the canonical projections then the map 
   $$\Theta \colon C(X,Y\times Z)_{\text{c.o.}} \rightarrow C(X,Y)_{\text{c.o.}}\times C(X,Z)_{\text{c.o.}},\quad \gamma \mapsto \left((\mathrm{pr}_Y)_*(\gamma),(\mathrm{pr}_Z)_*(\gamma)\right),$$
   is a homeomorphism, whence we can identify $C(X,Y\times Z)_{\text{c.o.}} \cong C(X,Y)_{\text{c.o.}} \times C(X,Z)_{\text{c.o.}}$. 
  \end{lem}
  
  \begin{proof}
   By \Cref{lem:pushforward} the map $\Theta$ is continuous. Clearly $\Theta$ is a bijection and thus we only need to show that it takes open sets in a subbase to open sets to see that $\Theta$ is a homeomorphism. 
   Now open rectangles $U\times V, U \opn Y, V\opn Z$ form a subbase of the product topology, whence Exercise \ref{Ex:cptopn} shows that the sets $\lfloor K, U\times V\rfloor$, with $K \subseteq X$ compact form a subbase of the topology on $C(X,Y\times Z)$. Now $\Theta (\lfloor K, U\times V\rfloor) = \lfloor K, U\rfloor \times \lfloor K,V\rfloor$ is open in $C(X,Y)_{\text{c.o.}} \times C(X,Z)_{\text{c.o.}}$, whence $\Theta$ is open. 
  \end{proof}
   
\begin{lem}\label{lem:eval} \index{evaluation map}
 Let $X,Y$ be topological spaces. If $X$ is locally compact, the evaluation map 
 $$\ev \colon C(X,Y)_{\text{c.o.}} \times X \rightarrow Y, \quad (f,x)\mapsto f(x)$$
 is continuous
\end{lem}

\begin{proof}
 For $U \opn Y$ we will show that $\ev^{-1}(U)$ is open in $C(X,Y)_{\text{c.o.}} \times X$.
 To this end, let $(\gamma,x) \in \ev^{-1}(U)$, i.e. $\gamma(x) \in U$. By continuity of $\gamma$ and local compactness of $X$ there is a compact $x$-neighborhood $K$ such that $\gamma(K)\subseteq U$. Thus $\lfloor K,U\rfloor \times K$ is a $(\gamma,x)$-neighborhood such that $\ev (\lfloor K,U\rfloor \times K) \subseteq U$. We see that $\ev^{-1}(U)$ is open and the evaluation is continuous. 
\end{proof}

\begin{prop}
 Let $X,Y,Z$ be topological spaces such that $Y$ is locally compact. Then the composition map 
 $$\Comp \colon C(X,Y)_{\text{c.o.}} \times C(Y,Z)_{\text{c.o.}} \rightarrow C(X,Z)_{\text{c.o.}},\quad (f,g) \mapsto g\circ f$$
 is continuous
\end{prop}

\begin{proof}
 Let $K \subseteq X$ be compact and $U \opn Z$. Pick $(\gamma,\eta)$ with $\text{Comp}(\gamma,\eta) \in \lfloor K,U\rfloor$. We have $\eta(\gamma(K))\subseteq U$. By Exercise \ref{Ex:cptopn} 2. we can pick a compact neighborhood $L$ of $\gamma(K)$ in $\eta^{-1}(U)$ and set $W \coloneq L^\circ$ (interior). Then $\lfloor K,V\rfloor \times \lfloor L,U\rfloor$ is a neighborhood of $(\gamma,\eta)$ which is contained in $\text{Comp}^{-1}(\lfloor K,U\rfloor)$ by construction. Thus the composition is continuous.
\end{proof}

We will now consider continuous mappings on cartesian products. If $f \colon X \times Y \rightarrow Z$ is continuous we can form for every $x \in X$ a mapping $f(x,\cdot) \colon Y \rightarrow Z y \mapsto f(x,y)$. Since $f$ is continuous every partial map $f(x,\cdot)$ is continuous. It turns out that the mapping which assigns to each $x \in X$ the partial map $f(x,\cdot)$ is continuous as a mapping into the space of continuous functions.

 \begin{prop}\label{prop:C0expohalf}
  Let $X,Y,Z$ be topological spaces and $f\colon X \times Y \rightarrow Z$ be continuous. Then $f^\vee \colon X \rightarrow C(Y,Z)_{\text{c.o.}}, f^\vee(x) \coloneq f(x,\cdot)$ is continuous.
 \end{prop}

\begin{proof}
 Consider $\lfloor K,U\rfloor \subseteq C(Y,Z)_{\text{c.o.}}$. We then compute
 \begin{align*}
  (f^\vee)^{-1}(\lfloor K,U\rfloor) &= \{x \in X \mid f(x,y) \in U \forall y \in K\} = \{x \in X \mid f(\{x\} \times K ) \subseteq U\}\\
  & = \{x \in X \mid \{x\}\times K \subseteq f^{-1}(U)\}
 \end{align*}
 Now $f^{-1}(U)$ is an open neighborhood of the compact set $\{x\} \times K$ in the cartesian product $X \times Y$. From \cite[Lemma 3.1.15]{Eng89} we deduce that there are $A \opn X$ and $B \opn Y$ such that $\{x\}\times K \subseteq A\times B \subseteq f^{-1}(U)$. This shows that $(f^\vee)^{-1}(\lfloor K, U \rfloor)$ is a neighborhood for every $x$ contained in it and thus an open set. Since sets of the form $\lfloor K,U\rfloor$ form a subbase of the compact open topology, $f^\vee$ is continuous.
 \end{proof}

 \begin{prop}[Exponential law for the compact open topology]\index{exponential law}\label{prop:C0expo}
  Let $X,Y,Z$ be topological spaces and $Y$ be locally compact. Then a mapping $f \colon X \rightarrow C(Y,Z)_{\text{c.o.}}$ is continuous if and only if the map
  $$f^\wedge \colon X \times Y \rightarrow Z, \quad (x,y) \mapsto f(x)(y)$$
  is continuous.
 \end{prop}

 \begin{proof}
  If $f$ is continuous, then $f^\wedge (x,y) = \ev (f(x),y)$ is continuous by \Cref{lem:eval}. Conversely, assume that $f^\wedge$ is continuous, then a quick calculation shows that $f= (f^\wedge)^\vee$, whence $f$ is continuous by \Cref{prop:C0expohalf}
 \end{proof}
 
 It is worth noting that local compactness is a crucial ingredient to obtain the exponential law. Indeed one can prove that under some requirement to the topological spaces involved, the exponential law can only hold if the topological space $Y$ is locally compact. See \cite[Exercise 3.4.A]{Eng89} for details.
  \begin{Exercise} \label{Ex:cptopn} \vspace{-\baselineskip}
   \Question Prove that if $X$ carries the final topology with respect to a family $g_i \colon Y_i \rightarrow X, i\in I$, then $f \colon X \rightarrow Z$ is continuous if and only if $f\circ g_i$ is continuous for every $i\in I$.
   \Question Let $X,Y$ be topological spaces and $\mathcal{S}$ be a subbase for the topology on $Y$. Show that the sets $\lfloor K,V\rfloor$ with $K \subseteq X$ compact and $V \in \mathcal{S}$ form a subbase for the compact open topology on $C(X,Y)$.
  \Question Let $X$ be a locally compact topological space and $K \subseteq X$ compact with $K \subseteq U \opn X$. Show that there is a compact set $L \subseteq U$ whose interior contains $K$.
   \Question Let $K$ be a compact topological space and $\Omega \opn K \times Y$. Prove that the set 
   $$\Omega' \coloneq \{f \in C(K,Y) \mid \text{graph} (f) \subseteq \Omega \}$$
   is open in $C(K,Y)_{\text{c.o}}$. Here $\text{graph}(f) = \{(x,f(x)) \mid x\in K\} \subseteq K \times Y$. \\
   {\tiny \textbf{Hint:} If $f \in \Omega'$, then for every $x \in K$ there are open $U_x \opn K, V_x \opn Y$ with $(x,f(x)) \in U_x \times V_x \opn \Omega$.} 
  \end{Exercise}
  \setboolean{firstanswerofthechapter}{true}
 \begin{Answer}[number={\ref{Ex:cptopn} 4.}]
  \emph{We show that the set} $\Omega' \coloneq \{f \in C(K,Y) \mid \text{graph} (f) \subseteq \Omega \}$\emph{ is open in the compact-open topology if $K$ is compact and $\Omega \opn K \times Y$.} \\[.15em]
  
  Let us show that $\Omega'$ is a neighborhood for each $f \in \Omega'$. Since $\Omega$ is open in $K\times Y$ with the product topology, we find for each $x \in K$ open subsets $U_x \opn K, V_x \opn Y$ with $(x,f(x)) \in U_x \times V_x \opn \Omega$. Shrinking the $U_x$ we may assume that also $\overline{U}_x \times V_x \subseteq \Omega$ and $f(\overline{U}_x) \subseteq V_x$.
  By compactness $\overline{U}_x$ is compact and we can cover $K$ with finitely many of the $U_x$, say $K = \bigcup_{1\leq k \leq n} U_{x_k}$. Then by construction $N_f \coloneq \bigcap_{1\leq k \leq n} \lfloor \overline{U}_{x_k},V_{x_k}\rfloor$ is open in the compact open topology and $f \in N_f$. Moreover, if $h \in N_f$, then we have for $x \in U_{x_k}$ that $(x,h(x)) \in U_x \times V_x \subseteq \Omega$. Since the $U_x$ cover $K$, we deduce that $h \in \Omega'$ and thus $N_f \subseteq \Omega'$.
 \end{Answer}
\setboolean{firstanswerofthechapter}{false}

\chapter{Canonical manifold of mappings}\label{App:canmfdmap}\copyrightnotice
This appendix sketches the construction of a canonical manifold of mappings structure for smooth mappings between (finite-dimensional) manifolds (for details see \cite[Appendix A]{AaGaS18}). Before we begin, let us reconsider for a moment the locally convex space $C^\infty (M,E)$ ($M$ a manifold and $E$ a locally convex space, cf.\ \Cref{sect:top:smoothmaps}). The topology and vector space structure allow us to compare two smooth maps $f$ and $g$ by measuring their difference $f-g$ on compact sets. As a general manifold $N$ lacks an addition we can not mimic this construction (albeit the topology still makes sense!). On first sight it might be tempting to think that one could use the charts of $N$ to construct charts for $C^\infty (M,N)$. However, if $N$ does not admit an atlas with only one chart, there will be smooth mappings whose image is not contained in one chart. Thus the charts of $N$ turn out to be not very useful. Instead one needs to find a replacement of the vector space addition to construct a way in which ``charts vary smoothly'' over $N$. 

\section{Local additions}

In this section we define first a replacement for vector additions which allow us to construct a canonical manifold structure on spaces of mappings.

\begin{defn}\label{def-loa}
Let $M$ be a smooth manifold. A \emph{local addition on $M$}\index{local addition} is a smooth map
\begin{displaymath}
\Sigma \colon TM \subseteq U \rightarrow M,
\end{displaymath}
defined on an open neighborhood $U \opn TM$ of the \emph{zero-section}\index{zero-section} of the tangent bundle
$0_M\coloneq \{0_p\in T_pM \mid p\in M\}$
such that $\Sigma(0_p)=p$ for all $p\in M$,
$$
U'\coloneq \{(\pi_{M}(v),\Sigma(v))\mid v\in U\}
$$
is open in $M\times M$ (where $\pi_{M}\colon TM\to M$ is the bundle projection) and the mapping $\theta \coloneq(\pi_{M},\Sigma)\colon U \to U'$ is a $C^\infty$-diffeomorphism.
\end{defn}

 \begin{setup}\label{setup:locadd:Lie}
   Let $G$ be a Lie group, then $G$ admits a local addition. To see this, let $\varphi \colon U \rightarrow V$ be a chart of $G$ with $\one \in U$ and $\varphi(\one)=0$. Then $\tilde{U} \coloneq T_{0}\varphi^{-1}(V)$ is open in $T_{\one} G$ and we define $\alpha_{\one} \colon \tilde{U} \rightarrow U, v \mapsto \varphi^{-1}(T_{\one} \varphi (v))$. Note that $\alpha_{\one}$ is a diffeomorphism. Now the tangent bundle of $G$ is trivial, i.e.\, $TG \cong G \times T_{\one} G$, \Cref{tangentLie}, and we obtain an open set $\Omega \coloneq \bigcup_{g \in G} T\lambda_g (\tilde{U}) \cong G \times \tilde{U} \opn G \times T_{\one} G$.
   Define the smooth map
   \begin{align*}
    \Sigma \colon \Omega \rightarrow G, v \mapsto \pi_G (v)\cdot \alpha_{\one} (T\lambda_{\pi_G(v)^{-1}}(v)).
   \end{align*}
  We note that $\Omega$ is a neighborhood of the zero section and $\Sigma (0_g) = g\cdot \alpha_{\one} (0_{\one}) = g\cdot \one =g$. Finally, $(\pi_G , \Sigma) \colon \Omega \rightarrow G \times G$ is a diffeomorphism (onto an open set) with inverse $(\pi_G , \Sigma)^{-1} (g,h) = T\lambda_{g}(\alpha_{\one}^{-1}(g^{-1}h))$.
  \end{setup}
  
  \begin{rem}\label{Riemlocadd}
   If $(M,g)$ is a strong Riemannian manifold (see \Cref{RiemGeo}), then the Riemannian exponential map $\exp_g$ of $g$ induces a local addition on $M$, \cite[Lemma 10.1]{Mic}.
  \end{rem}

  We leave the proof of the following statement as Exercise \ref{Ex:locadd} 1.
  
\begin{lem}[{\cite[Lemma 7.5]{SaW} or \cite[10.11]{Mic}}] \label{lem:newlocadd}
 Let $M$ be a manifold which admits a local addition $\Sigma$. Then $TM$ admits a local addition (it is $T\Sigma \circ \kappa$, where $\kappa$ is the flip of the double tangent bundle, Exercise \ref{Ex:locadd} 1.).
\end{lem}

With the help of a local addition we can pull back functions near a given function:
\begin{setup}
 Let $f \in C^\infty (K, M)$ and assume that $M$ admits a local addition $\Sigma \colon  TM \supseteq \Omega \rightarrow M$. If $g \in C^\infty (K,M)$ such that $$(f,g) (K) \subseteq (\pi_M,\Sigma)(\Omega) \opn M \times M,$$ we can define the mapping 
 $$\Phi_f (g) \coloneq (\pi_M,\Sigma)^{-1}\circ (f,g) \in C^\infty (K,TM).$$
 A quick computation shows that $\pi_M \circ \Phi_f(g) = f$. The idea is that $\Phi_f (g)$ yields at every $x \in K$ a vector in $T_{f(x)}M$ which gets mapped by the local addition to $g(x)$. Thus the mapping $\Phi_f(g)$ measures the difference between $f$ and $g$ (similar to $f-g$ in the vector space case). Note that $\Phi_f$ takes its values in the subspace
 $$C^\infty_f (K,TM) \coloneq \left\{h \in C^\infty (K,TM) \mid \pi_M \circ h =f\right\}$$
 of mappings over the given map $f$.
\end{setup}
  
  To make mappings of the form $\Phi_f$ into charts for $C^\infty (K,M)$ we need to study first spaces of the form $C^\infty_f(K,TM)$. In particular, we need these spaces to be locally convex spaces and will showcase this in the next section. Relating these sets to spaces of sections of certain vector bundles. Note however, that the space of functions $C^\infty_f (K,TM)$ depends heavily on the function $f$. In general, these spaces will NOT be isomorphic if we change the base function $f$. 

\begin{Exercise} \vspace{-\baselineskip} \label{Ex:locadd}
 \Question Let $M$ be a manifold, $TM$ its tangent bundle and $T^2M=T(TM)$ the double tangent bundle. Assume that $M$ admits a local addition $\Sigma \colon TM \supseteq U \rightarrow M$.
 \subQuestion Show that there is a bundle isomorphism $\kappa \colon T^2M\rightarrow T^2M$ over the identity such that in local coordinates we have (up to identification) the identity $\kappa (x,v,u,w) = (x,u,v,w)$. We call $\kappa$ the \emph{(canonical) flip} of the double tangent bundle.\index{flip of the double tangent bundle}
 \subQuestion Prove that $T\Sigma \circ \kappa$ is a local addition for $TM$. Explain then why $T\Sigma$ is not a local addition.
\end{Exercise}

\section{Vector bundles and their sections}\label{sect:VB}

Recall that a \emph{vector bundle}\index{vector bundle} is a pair of manifolds $E$ (\emph{total space}) and $M$ (\emph{base (space)}) together with a surjective submersion $p \colon E \rightarrow M$ such that for every $x \in M$ the \emph{fibre} $E_x \coloneq p^{-1}(x) \subseteq E$ is a real vector space and there is $x \in U_x \opn M$ and a diffeomorphism $\kappa_x \colon U_x \times F_x \rightarrow p^{-1}(U_x)$, a \emph{bundle trivialisation} such that
\begin{itemize}
 \item $p\circ \kappa_x (y,v) = y, \forall y \in U_x$
 \item for each $y \in U_x$, the map $\kappa_x (y,\cdot) \colon F_x \rightarrow p^{-1}(y)$ is a vector space isomorphism.
\end{itemize}
 Finally, for two vector bundles $p_i \colon E_i \rightarrow M_i, i=1,2$, a pair of smooth maps $F \colon E_1 \rightarrow E_2$ and $f \colon M_1 \rightarrow M_2$ is called \emph{vector bundle morphism}\index{vector bundle morphism} if $p_2 \circ F = f\circ p_1$ and $F|_{p_1^{-1}(x)}^{p_2^{-1}(f(x))}$ is linear for every $x \in M_1$ (we also say that $F$ is a bundle morphism over $f$).
 
 \begin{setup}[Sections of vector bundles]
  For a vector bundle $p \colon E \rightarrow M$, a smooth map $f \in C^\infty (M,E)$ is called a \emph{smooth section}\index{vector bundle!smooth section} if $p\circ f = \id_M$. We denote the \emph{set of all sections}\index{vector bundle!set of all sections} by
  $$\Gamma(E) \coloneq \{f \in C^\infty (M,E) \mid p\circ f = \id_M\}.$$
 \end{setup}
 
 Since every fibre of a vector bundle is a vector space, pointwise addition and scalar multiplication of sections makes sense and $\Gamma(E)$ becomes a vector space. This space can be topologised as a locally convex space as follows.
 
 \begin{setup}\label{topo:sect}
   Let $p\colon E \rightarrow M$ be a vector bundle. We pick a family of bundle trivialisations $\kappa_i \colon p^{-1}(U_i) \rightarrow U_i \times F_i, i\in  I$ which cover $E$ (i.e.\, the $U_i$ cover $M$). If we denote by $\text{pr}_{F_i} \colon U_\varphi \times F_i \rightarrow F_i$ the canonical projection, then the \emph{principal part} of the section $X$ in the trivialisation $\kappa_i$ is defined as
   $$X_{\kappa_i} \coloneq \text{pr}_{F_i} \circ \kappa_i \circ X|_{U_i} \in C^\infty (U_i,F_i).$$
   Note that this entails that with respect to the bundle trivialisation we have 
   $$\kappa_i \circ X = (\id_{U_i}, X_{\kappa_i}), \quad \forall X \in \Gamma(E).$$
   We declare a topology on $\Gamma(E)$ as the initial topology with respect to the map 
   $$\mathcal{I}_E \colon \Gamma(E) \rightarrow \prod_{i \in I} C^\infty (U_i, F_i),\quad  X \mapsto (X_{\kappa_i})_{i \in I},$$
 where the factors on the right hand side carry the compact open $C^\infty$-topology. Now by definition of the vector space operations we have $\mathcal{I}_E(X+rU) = \mathcal{I}_E (X) + r\mathcal{I}_E(U)$. Note that since the right hand side is a locally convex space by \Cref{prop:lcvx_mappingsp}, this shows that also $\Gamma(E)$ is a locally convex space. We shall see that this structure does not depend on the choice of trivialisations in Exercise \ref{Ex:bundles} below.
 \end{setup}

 Now to model spaces of smooth mappings, we need a certain type of bundles constructed from a map and a vector bundle. While we study this situation we recall two ways to construct new vector bundles from given vector bundles.
 
\begin{setup}[Pullback bundles and their sections]\label{bundleconv}
Let $M$ be a smooth manifold, and $K$ be a manifold.\index{vector bundle!pullback}
If $p\colon E\to M$ is a smooth vector bundle over~$M$
and $f\colon K\to M$ is a smooth map, then
$$
f^*(E)\coloneq\bigcup_{x\in K}\{x\}\times E_{f(x)}
$$
is a split submanifold of $K\times E$ (as it locally looks like $\mbox{graph}(f)\times E_x$ inside $K\times M\times E_x$ around points in $\{x\}\times E_x$). We endow $f^*(E)$ with this submanifold structure.
Together with the natural vector space structure on $\{x\}\times E_{f(x)}\cong E_{f(x)}$ and the map $p_f\colon f^*(E)\rightarrow K$, $(x,y)\mapsto x$, 
we obtain a vector bundle $f^*(E)$ over~$K$, the so-called \emph{pullback of $E$ along~$f$}. For each local trivialization $\theta=(p|_{E|_U},\theta_2) \colon E|_U\to U\times F$ of~$E$ and $W:=f^{-1}(U)$, the map
$$
f^*(E)|_W \to W\times F,\quad (x,y)\mapsto (x,\theta_2(y))
$$
is a local trivialization of $f^*(E)$. We endow
$$
C^\infty_f(K,E)\coloneq\{\tau\in C^\infty(K,E)\mid \pi\circ \tau=f\}
$$
with the topology induced by $C^\infty(K,E)$. With pointwise operations,
$C^\infty_f(K,E)$ is a vector space and the map
$$
\Psi\colon \Gamma(f^*(E))\to C^\infty_f(K,E),\;\, \sigma\mapsto \text{pr}_2\circ \, \sigma
$$
is a bijection with inverse $\tau\mapsto (\id_K,\tau)$.
As $(\text{pr}_2)_*\colon C^\infty(K,K\times E)\to C^\infty(K,E)$ is a continuous map and also $\tau\mapsto (\id_K,\tau)\in
C^\infty(K,K) \times C^\infty(K,E)\cong C^\infty(K,K\times E)$ is continuous,
we deduce that $C^\infty_f(K,E)$ is a locally convex topological vector space and
$\Psi$ is an isomorphism of topological vector spaces.
If we wish to emphasize the dependence on~$E$, we also write $C^\infty_f(K,E)(E)$
instead of $C^\infty_f(K,E)$.
\end{setup}

\begin{setup}[Whitney sum of bundles]\label{WSum} \index{vector bundle!Whitney sum}
Let $p_i \colon E_i \rightarrow M, i=1,2$ be vector bundles over the same base $M$. Then we can form the \emph{direct product}\index{vector bundle!direct product} of these vector bundles, which is the vector bundle 
$$p_1 \times p_2 \colon E_1 \times E_2 \rightarrow M \times M$$
over the product manifold $M\times M$. Consider the diagonal map $\Delta \colon M \rightarrow M \times M, m \mapsto (m,m)$. 
Then the \emph{Whitney sum} of $p_1$ and $p_2$ is defined as the bundle
$$p_1 \oplus p_2 \colon E_1 \oplus E_2 \coloneq \Delta^* (E_1\times E_2) \rightarrow M.$$
Note that by construction we have as fibres $(E_1 \oplus E_2)_x = (E_1)_x \times (E_2)_x$ and if $\kappa_i$ is a bundle trivialisation of $E_i$ over a common open set $U \opn M$, then the restriction of $\kappa_1 \times \kappa_2$ to $(p_1 \oplus p_2)^{-1}(U)$ is a bundle trivialisation of the Whitney sum. 
\end{setup}

We just mention that there is a version of the exponential law for spaces of sections, cf.\, \cite[Appendix A]{AaGaS18}.

\begin{Exercise} \vspace{-\baselineskip} \label{Ex:bundles}
 \Question Show that the mapping $\mathcal{I}_E \colon \Gamma(E) \rightarrow \prod_{i \in I} C^\infty (U_i,F)$ from  \Cref{topo:sect} has a closed image. Then prove that the topology does not depend on the choice of local trivialisations, i.e.\, if $\mathcal{B} \coloneq \{\nu_j\}_{j \in J}$ is another family of trivialisations covering $E$, then the topologies induced by $\mathcal{I}_E$ and $\mathcal{I}_{E,\mathcal{B}}$ coincide.
 \\
 {\tiny \textbf{Hint:} If $\mathcal{A},\mathcal{B}$ are families of trivialisations, then also $\mathcal{A} \cup \mathcal{B}$ is such a family. Thus we may assume without loss of generality that $\mathcal{A} \subseteq \mathcal{B}$ and it suffices to prove that the topology induced by $\mathcal{B}$ can not be finer then the one induced by $\mathcal{A}$. To prove this adapt \Cref{lem:scattering}.}
 \Question Verify that the pullback bundle in \Cref{bundleconv} forms a split submanifold of $K \times E$.\\ 
 {\tiny \textbf{Hint:} Construct submanifold charts by hand as for the graph of a function.}
 \Question Show that for two vector bundles $p_i \colon E_i \rightarrow M$ there is a canonical isomorphism of locally convex spaces $\Gamma(E_1) \times \Gamma(E_2) \cong \Gamma(E_1 \oplus E_2)$.
\end{Exercise}

\setboolean{firstanswerofthechapter}{true}
\begin{Answer}[number= {\ref{Ex:bundles} 2.}]
 We show that the pullback bundle $f^*(E)$ is a split submanifold of $K \times E$. The idea is similar to the proof that the graph of a smooth function is a split submanifold. 
 The proof is essentially \Cref{lem:pullbacksubm}: By definition of a fibre bundle $p \colon E \rightarrow M$ is a submersion. In the notation of \Cref{lem:pullbacksubm}, we have $f^*(E) = K \times_M E$ is a split submanifold of $K \times E$ for each $f \in C^\infty (K,M)$.
\end{Answer}
\setboolean{firstanswerofthechapter}{false}
\section{Construction of the manifold structure}\label{app:constmfd}
\begin{tcolorbox}[colback=white,colframe=blue!75!black,title=General assumptions]
We let $K$ be a compact manifold and $M$ be a smooth manifold which admits a local addition $\Sigma \colon TM\supseteq U \rightarrow M$ such that 
\begin{itemize} 
 \item $\Sigma$ is \emph{normalised}\index{local addition!normalised}, i.e.\ $T_{0_p}(\Sigma|_{T_pM}) = \id_{T_pM}$ for all $p \in M$, (a manifold with local addition has a normalised local addition, see \cite[Lemma A.14]{AaGaS18}),
 \item $\theta \coloneq (\pi_M, \Sigma) \colon U \rightarrow U'$ is a diffeomorphism
\end{itemize}
\end{tcolorbox}

\begin{setup}[Manifold structure on $C^\infty (K,M)$]\label{setup:mfdstruct}
For $f\in C^\infty(K,M)$, the locally convex space of $C^\infty$-sections
of $f^*(TM)$ can be identified with
\[
C^\infty_f(K,E)=\{\tau\in C^\infty(K,TM)\mid \pi_M\circ \tau=f\},
\]
with the topology induced by $C^\infty (K,TM)$. 
Use notation as in Definition~\ref{def-loa}. Then $O_f:=C^\infty_f(K,E)\cap C^\infty(K,U)$
is an open subset of~$C^\infty_f(K,E)$,
\[
O_f':=\{g\in C^\infty(K,M)\mid (f,g)(K)\subseteq U'\}
\]
is an open subset\footnote{The proof is similar to Exercise \ref{Ex:cptopn} 4.} of $C^\infty(K,M)$ and the map
$
\phi_f\colon O_f\to O'_f,\quad\tau\mapsto \Sigma\circ \tau
$
is a homeomorphism with inverse $g\mapsto \theta^{-1}\circ (f,g)$.
By the preceding, if also $h\in C^\infty(K,M)$, then $\phi_h^{-1}\circ \phi_f$
has an open domain; it is smooth there, since (by the exponential law \Cref{thm:explaw}),
we only need to observe that the map
\begin{equation}\label{eq1}
(\tau,x)  \mapsto  (\phi_h^{-1}\circ\phi_f)(\tau)(x)=\theta^{-1}(h(x),\Sigma(\tau(x)))
\end{equation}
is $C^{\infty}$. Hence $C^\infty(K,M)$ has a smooth manifold structure
such that each of the maps $\phi_f^{-1}$ is a local chart.
\end{setup}

We prove that the manifold structure on $C^\infty (K,M)$ is canonical and thus by Lemma \ref{base-cano} (b) the construction \ref{setup:mfdstruct} is independent of the choice of local addition.

\begin{lem}\label{la:cano:mfdmap}
The manifold structure on $C^\infty (K,M)$ constructed in \ref{setup:mfdstruct} is canonical.\index{canonical manifold of mappings}
\end{lem}
\begin{proof}
We first show that the evaluation map $\ev\colon C^\infty(K,M)\times K\to M$
is $C^{\infty}$. It suffices to show that $\ev(\phi_f(\sigma),x)$ is $C^{\infty}$
in $(\sigma,x)\in O_f\times K$ for all $f\in C^\infty(K,M)$.
But
\[
\ev(\phi_f(\sigma),x)=\Sigma(\sigma(x))=\Sigma(\ev(\sigma,x)),
\]
where $\ev\colon C^\infty_f(K,E)\times K\to f^*(TM)$, $(\sigma,x)\mapsto\sigma(x)$
is $C^{\infty}$ (cf.\ \cite[Proposition 3.20]{MR3342623}).
Now let $h\colon L\to C^\infty(K,M)$ be a map, where $L$ is a manifold. If $h$ is~$C^\infty$, then $h^\wedge=\ev\circ (h\times\id_K)$
is~$C^{\infty}$.
If, conversely, $h^\wedge$ is $C^{\infty}$,
then $h$ is continuous as a map to $C(K,M)$ with the compact open topology (see \Cref{prop:C0expohalf}) and $h(x)=h^\wedge(x,\cdot)\in C^\infty(K,M)$ for each $x\in L$.
Given $x\in L$, let $f:=h(x)$. Then 
$$\psi_f \colon C(K,M)\to C(K,M)\times C(K,M)\cong C(K,M\times M),\quad  g\mapsto (f,g)$$
 is a continuous map and $W\coloneq h^{-1}(O_f')=(\psi_f\circ h)^{-1}(C(K,U'))$ is an open $x$-neighbourhood in~$L$. Since
\begin{align*}
((\phi_f)^{-1}\circ h|_W)^\wedge(y,z)=(\theta^{-1}\circ (f,h(y)))(z)=\theta^{-1}(f(z),h^\wedge(y,z))
\end{align*}
is $C^{\infty}$, the map $\phi_f^{-1}\circ h|_W$ (and hence also $h|_W$)
is~$C^\infty$ (apply \Cref{thm:explaw} to the spaces $C^\infty (U_i,F)$ containing the principal parts of the sections).
\end{proof}

\begin{setup}\label{setup:curves:tan}
Let $\gamma \in C^\infty (K,M)$ and view $T_\gamma C^\infty (K,M)$ as a set of equivalence classes of smooth curves $c \colon ]-\varepsilon, \varepsilon[ \rightarrow C^\infty (K,M), c(0)=\gamma$.  As the manifold structure is canonical, $c$ is smooth if and only if $c^\wedge \colon ]-\varepsilon,\varepsilon[ \times K \rightarrow M$ is smooth.
Hence for the canonical chart $\phi_\gamma\colon O_\gamma\to O_\gamma'\subseteq C^\infty(K,M)$, the map $T_0\phi_\gamma \colon C^\infty_\gamma (K,TM) \rightarrow \to T_\gamma C^\infty (K,M)$ is an isomorphism of TVS. For $x\in K$ denote by $\varepsilon_x$ be the point evaluation in $x$. Since $\Sigma$ is normalised we obtain 
\begin{align*}
T\varepsilon_xT\phi_f(0,\tau)&=T\varepsilon_x([t\mapsto \Sigma\circ (t\tau)]) =[t\mapsto \Sigma(t\tau(x))]\\
&= [t\mapsto\Sigma|_{T_{f(x)}M}(t\tau(x))]
=T\Sigma|_{T_{f(x)}M}(\tau(x))=\tau(x),
\end{align*} 
Summing up, this implies that for each fibre there is a linear bijection
\begin{equation}\label{eq:PHI}
\Phi_\gamma \colon T_\gamma C^\infty (K,M) \rightarrow C^\infty_\gamma (K,TM),\quad  [c] \mapsto (k \mapsto [t \mapsto c^\wedge (t,k)]).
\end{equation}
We will now sketch the proof that the fibre maps \eqref{eq:PHI} induce a bundle isomorphism
$$\Phi_M\colon TC^\infty (K,M) \rightarrow  C^\infty (K,TM),\quad   T_\gamma C^\infty (K,M) \ni V \mapsto \Phi_\gamma (V),$$
such that the following diagram commutes  
\begin{displaymath}
\begin{tikzcd}
TC^\infty (K,M) \ar[rr,"\Phi_M"] \arrow[rd, swap,"\pi_{C^\infty (K,M)}"] & & C^\infty (K,TM) \ar[ld,"(\pi_M)_*"]  \\ 
& C^\infty (K,M) &
\end{tikzcd} 
\end{displaymath}
\textbf{(Sketch of the proof)} (cf.\ \cite[Appendix A]{AaGaS18}) If $\lambda_p \colon T_p M \rightarrow TM$ is the inclusion and $\kappa \colon T^2M \rightarrow T^2M$ the canonical flip, then $\Theta \colon TM \oplus TM \rightarrow \pi_{TM}^{-1} (0_M) \subseteq T^2M, \Theta(v,w) = \kappa (T\lambda_{\pi(v)}(v,w))$ is a bundle isomorphism. Let $0\colon M \rightarrow TM$ be the zero-section\index{zero-section}, then $\Theta$ induces a diffeomorphism $$\Theta_\gamma\colon O_\gamma\to O_{0\circ \gamma},\quad \eta \mapsto \Theta \circ (0\circ \gamma, \eta).$$
From the local addition $\Sigma$ we construct a local addition $\widetilde{\Sigma}$ on $TM$ and consider the charts $\phi_{0\circ \gamma}$ on $C^\infty (K,TM)$. Then the sets $S_\gamma \coloneq T\phi_\gamma (O_\gamma\times C^\infty_\gamma (K,TM))$ form an open cover of $T(C^\infty(K,M))$ for $\gamma\in C^\infty(K,M)$. We deduce that the sets $\Phi_M(S_\gamma)$ form a cover of $C^\infty(K,TM)$ by sets which are open as $\Phi_M(S_\gamma) =(\phi_{0\circ \gamma}\circ \phi_\gamma)(O_\gamma\times C^\infty_\gamma (M,TM))=\phi_{0\circ \gamma}(O_{0\circ \gamma})$. Hence we can check that $\Phi_M$ restricts to a $C^\infty$-diffeomorphism on these open sets, i.e.
\begin{displaymath}
\Phi\circ T\phi_\gamma=\phi_{0\circ \gamma}\circ \Theta_\gamma
\end{displaymath}
for each $\gamma\in C^\infty(K,M)$ (as all other mappings in the formula are smooth diffeomorphisms). Now we can rewrite $\Phi_M(T\phi_\gamma(\sigma,\tau))$ as
\begin{align*}
& ([t\mapsto \Sigma(\sigma(x)+t\tau(x))])_{x\in K}
= ([t\mapsto (\Sigma\circ \lambda_{\gamma(x)})(\sigma(x)+t\tau(x))])_{x\in K}\\
=& (T(\Sigma\circ \lambda_{\gamma(x)})(\sigma(x),\tau(x)))_{x\in K}
= (\Sigma_{TM}((\kappa \circ T\lambda_{\gamma(x)})(\sigma(x),\tau(x))))_{x\in K}\\
=& ((\Sigma_{TM}\circ \Theta_\gamma)(\sigma,\tau)(x))_{x\in K}
=(\phi_{0\circ \gamma}\circ \Theta_\gamma)(\sigma,\tau). 
\end{align*}
Thus $\Phi_M$ is a $C^\infty$-diffeomorphism.
\end{setup}

\begin{setup}[Smooth maps into the Whitney sum over strong Riemannian manifolds]\label{Whitneyiso}\index{vector bundle!Whitney sum} \label{Whitneysum}
 By \Cref{lem:newlocadd}, $TM$ admits a local addition whose product yields a local addition on the product manifold. Thus $C^\infty (K,M), C^\infty (K,TM), C^\infty (K,TM \times TM)$ are canonical manifolds. Taking the Whitney sum of $(\pi_M)_* \colon C^\infty (K,TM) \rightarrow C^\infty (K,M)$ (\Cref{setup:curves:tan}) with itself, we obtain the bundle $C^\infty (K,TM) \oplus C^\infty (K,TM)$. Our aim is to identify it with the bundle $C^\infty (K,TM\oplus TM)$. 
 Observe that $C^\infty (K,TM)\oplus C^\infty (K,TM)$ is a split submanifold of $C^\infty (K,M) \times C^\infty (K,TM)^2$. Now the factors of the product are canonical manifolds. Thus \Cref{base-cano} yields a diffeomorphism 
 $$C^\infty (K,M) \times C^\infty (K,TM)^2 \cong C^\infty(K,M\times (TM)^2),$$
 which takes the split submanifold $C^\infty (K,TM)\oplus C^\infty(K,TM)$ to $C^\infty (K,TM\oplus TM)$. As diffeomorphisms preserve split submanifolds, see Exercise \ref{ex:HR} 2., $C^\infty (K,TM\oplus TM)$ must be a split submanifold of $C^\infty (K,M\times (TM)^2)$. Finally, \Cref{base-cano} (c) shows that $C^\infty (K,TM\oplus TM)$ is a canonical manifold diffeomorphic to $TC^\infty (K,M)\oplus TC^\infty (K,M)$.
  \end{setup}
  \begin{rem}
By uniqueness of canonical manifolds, $C^\infty (K,TM\oplus TM)$ from \Cref{Whitneysum} coincides with the manifold structure we could have obtained via a local addition on $TM\oplus TM$.
The same proof works if we only assume that $C^\infty (K,M)$ and $C^\infty (K,TM)$ are canonical manifolds (without assuming that $M$ has a local addition).
 \end{rem}
\begin{Exercise}\label{Ex:tangent}
 \Question Work out the missing details in the sketch of the proof in \Cref{setup:curves:tan}.
\end{Exercise}

\section{Manifolds of curves and the energy of a curve}\label{sect:mfdofcurves}

In this appendix we consider the manifold structure on spaces of curves on a compact interval. The reason for this is that we defined for a (weak) Riemannian manifold $(M,g)$ the energy $\text{En} \colon C^\infty ([0,1] , M) \rightarrow \R$ of a curve and would like to differentiate $\text{En}$ to find geodesics. Hence a manifold structure on the space of curves $C^\infty ([0,1],M)$ is needed. Many details of the construction will be left to the reader as Exercise \ref{Ex:curvemanifolds}. Moreover, we will not systematically introduce tangent bundles for manifolds with boundary such as $[0,1]$ (see e.g.\ \cite{Mic}). Thus the compact open $C^\infty$-topology needs to be defined without recourse to tangent bundles.

\begin{setup}\label{curves:coCinfty}
 Let $M$ be a (possibly infinite-dimensional) manifold. Let $c \colon [0,1] \rightarrow M$ be a smooth curve and $K \subseteq [0,1]$ compact such that $c(K) \subseteq U$, where $(U,\varphi)$ is a chart of $M$. If $\varphi (U) \opn E$ for the locally convex space $E$, we pick a seminorm $\lVert \cdot \rVert$ on $E$ and define a $C^k$-neighborhood $N^k(c,K,(U,\varphi),\lVert \cdot\rVert,\varepsilon)$ as the set
 $$\left\{g \in C^\infty ([0,1],M) \middle| g(K) \subseteq U, \sup_{0\leq \ell \leq k} \sup_{x\in K} \left\lVert \frac{\mathrm{d}^\ell}{\mathrm{d} x^k} (\varphi \circ g - \varphi \circ c)(x)\right\rVert < \varepsilon\right\}.$$
 Then the family $N^k (c,K,(U,\varphi),\lVert \cdot \rVert, \varepsilon)$, where $c \in C^\infty ([0,1] ,M)$, $K \subseteq [0,1]$ compact, $\lVert \cdot \rVert$ is a continuous seminorm of $E$ and $\varepsilon >0$, forms the base of a topology on $C^\infty ([0,1],M)$ called the \emph{compact open $C^\infty$-topology}.\index{compact open $C^\infty$-topology}  If $M=E$ is a locally convex space, then also $C^\infty ([0,1],E)$ is a locally convex space\footnote{For $M = \R$, we have described this structure already in \Cref{example:lcs}.} and the exponential law \Cref{thm:explaw} carries over to $C^\infty ([0,1],O)$, $O \opn E$.
\end{setup}

\begin{prop}\label{prop:curvesmfd}
 Let $(M,\Sigma)$ be a manifold with local addition $\Sigma$ and topologise the space $C^\infty([0,1],M)$ with the compact open $C^\infty$-topology. Then $C^\infty ([0,1],M)$ is a canonical manifold and we have $T C^\infty ([0,1],M) \cong C^\infty ([0,1],TM)$. 
\end{prop}

\begin{lem}\label{lem:energydifferential}
 Let $(M,g)$ be a (weak) Riemannian manifold such that $M$ and $TM$ admit local additions. Then the \emph{energy} \index{curve!energy of a} 
 $$\mathrm{En} \colon C^\infty ([0,1],M) \rightarrow \R , \quad c \mapsto  \frac{1}{2} \int_0^1 g_{c(t)} (\dot{c}(t),\dot{c}(t))\mathrm{d}t$$
 is smooth. We can express its derivative in a local chart $(U,\varphi)$ (suppressing most identifications) as
 $$d\mathrm{En}(c;h) = \int_0^1 \frac{1}{2} d_1g_U(c,c'(t),c'(t);h)) - d_1g(c(t),h(t),c'(t);c'(t)))-g_U(c(t),h(t),c''(t)) \mathrm{d}t,$$
 where we view $g$ as a map of three arguments, $g_U(c,a,b)$ and $h \in T_c C^\infty ([0,1],M) \cong \{g \in C^\infty ([0,1],TM)\mid \pi\circ g = c\}$ with $h(0)=0_{c(0)}$ and $h(1)=0_{c(1)}$.
 \end{lem}

\begin{proof}
 Since $M$ and $TM$ admit local additions, Exercise \ref{Ex:curvemanifolds} 4. implies that both $C^\infty ([0,1],M)$ and $C^\infty ([0,1],TM)$ are canonical manifolds. Applying the exponential law, $C^\infty ([0,1],M) \rightarrow C^\infty ([0,1],TM), c \mapsto \dot{c} = (t \mapsto T_t c (1))$ is smooth if and only if $C^\infty ([0,1],M) \times [0,1] \rightarrow TM, (c,t)\mapsto  T_t c(1)$ is smooth. We check this locally in a neighborhood of a $c$: Pick a chart $(U,\varphi)$ of $M$ and $[a,b] \subseteq [0,1]$ with $c([a,b]) \subseteq U$. As the topology on $C^\infty([0,1],M)$ is finer then the compact open topology, there exists a whole neighborhood of curves $g$ with $g([a,b])\subseteq U$. Cover $[0,1]$ by compact intervals which $c$ maps into a chart domain and work locally. To keep the notation simple, we will assume that $c([0,1]) \subseteq U$ or in other words, assume without loss of generality that $M \opn E$ for some locally convex vector space $E$. Thus we need to prove that $C^\infty ([0,1],M) \times [0,1] \rightarrow M \times E, (c,t) \mapsto (c(t),c'(t))=(\ev(c,t),\ev (c',t))$ is smooth. The evaluation map is smooth on canonical manifolds and the mapping $C^\infty ([0,1],M) \rightarrow C^\infty ([0,1],E), c \mapsto c'$ is the restriction of the continuous linear map $C^\infty ([0,1],E) \rightarrow C^\infty ([0,1],E), c \mapsto c'$ to the open subset $C^\infty ([0,1],M) \opn C^\infty ([0,1],E)$ (here we exploit the compact open $C^\infty$-topology). We conclude that the mapping $C^\infty ([0,1],M) \rightarrow C^\infty ([0,1],TM), c \mapsto \dot{c}$ is smooth.
 Since Exercise \ref{Ex:curvemanifolds} 6. identifies the Whitney sums, we deduce from $(\dot{c},\dot{c}) \in TM \oplus TM$ that $C^\infty ([0,1],M) \rightarrow C^\infty ([0,1],\R), c\mapsto g_* (\dot{c},\dot{c})$ is smooth (as pushforwards are smooth on canonical manifolds). However, $\tfrac{1}{2}\int_0^1 \colon C^\infty ([0,1],\R) \rightarrow \R$ is continuous linear, whence $\text{En}$ can be written as a composition of smooth mappings and is thus smooth.
 
 To compute the derivative of the energy, we work in a local chart $(U,\varphi)$ of $M$ (though we will only label $g$ and suppress the other identifications): Then the metric $g$ becomes a map of three arguments $g_U$ which is bilinear in the last two. Recall that the vector component of $\dot{c}$ is $c'$. Now by choice of $h$, there is a smooth curve $q \colon ]-\varepsilon,\varepsilon[ \rightarrow C^\infty ([0,1],M)$ such that $\left.\frac{\partial}{\partial s}\right|_{s=0} q(s) = h$ and $q(t)(0)=c(0)$ and $q(t)(1)=c(1)$ for all $t$ (a smooth variation, cf.\ also \Cref{defn:smoothvariation} for the meaning of the partial derivative). Then we compute with the help of the exponential law (Exercise \ref{Ex:curvemanifolds} 3.):
 \begin{align*}
  d\mathrm{En}(c;h) &= \left.\frac{d}{ds}\right|_{s=0} \mathrm{En}(q(s))= \frac{1}{2}\int_0^1 \left.\frac{d}{ds}\right|_{s=0} g_U\left(q(s)(t),\frac{d}{dt} q(s)(t),\frac{d}{dt} q(s)(t)\right)\mathrm{d}t \\
  &= \int_0^1 \frac{1}{2} d_1g_U (c(t),c'(t),c'(t);h(t))+g_U\left(c(t),\left. \frac{d}{d s}\right|_{s=0} \frac{d}{dt} q(s)(t),c'(t)\right)\mathrm{d}t \\
  &= \int_0^1 \frac{1}{2} d_1g_U (c(t),c'(t),c'(t);h(t)) - d_1g_U \left(c(t),\left. \frac{d}{d s}\right|_{s=0} q(s)(t),c'(t);c'(t)\right) \\ & \hspace{6cm}-g_U\left(c(t),\left. \frac{d}{d s}\right|_{s=0} q(s)(t),c''(t)\right) \mathrm{d}t\\
  &= \int_0^1 \frac{1}{2} d_1g_U (c(t),c'(t),c'(t);h(t)) - d_1g_U(c(t),h(t),c'(t);c'(t)) \\ & \hspace{6cm}-g_U(c(t),h(t),c''(t)) \mathrm{d}t
 \end{align*}
 In passing from the second to the third line we used integration by parts together with the fact that $\left. \frac{d}{d s}\right|_{s=0} q(s,t)$ vanishes at the endpoints of the interval. 
\end{proof}

Almost all of the terms in the formula for the derivative of the energy in \Cref{lem:energydifferential} can be globalised to the Riemannian manifold. Derivation exploits of course that we work locally and the second derivative of $c$ needs to be taken (from the perspective of the Riemannian manifold $M$) in the fibre over $c(t)$. This already hints at the connection of this formula to the covariant derivative (which however was not yet needed).

\begin{Exercise}\label{Ex:curvemanifolds}  \vspace{-\baselineskip}
 \Question Prove that the neighborhoods defined in \Cref{curves:coCinfty} form the base of a topology. 
 \Question Let $E$ be a locally convex space. Show that the compact open $C^\infty$-topology turns $C^\infty ([0,1],E)$ into a locally convex space. Show then that for $M= \R$ this topology coincides with the compact open $C^\infty$-topology from \Cref{example:lcs}.
 \Question Establish a variant of the exponential law \Cref{thm:explaw} for manifolds of smooth mappings on $[0,1]$ (with values in open sets of locally convex spaces).
 \Question Generalise \Cref{setup:mfdstruct} to prove \Cref{prop:curvesmfd}. Then show that $C^\infty ([0,1],M)$ is a canonical manifold.
 \Question Follow the argument in \Cref{setup:curves:tan} to prove that $TC^\infty ([0,1],M)$ can naturally be identified with $C^\infty ([0,1],TM)$.
 \Question Adapt the argument in \Cref{Whitneyiso} to establish an isomorphism $C^\infty([0,1],TM\oplus TM) \cong C^\infty([0,1],TM) \oplus C^\infty ([0,1],TM)$ if $M$ and $TM$ admit local additions.
\end{Exercise}

\chapter{Vector fields and their Lie bracket}\label{chapter:LA:VF}\copyrightnotice

In this section we recall the construction of the Lie algebra of vector fields.

\begin{ex}
 The tangent bundle $\pi_M \colon TM \rightarrow M$ is a vector bundle (bundle trivialisations are given by the canonical charts $T\varphi$). A smooth section of the tangent bundle is called (smooth) \emph{vector field}\index{vector field} and we write shorter $\mathcal{V}(M) \coloneq \Gamma (TM)$ for the vector space of all vector fields.\index{vector field!vector space of}
\end{ex}

\begin{ex}
 If $U \opn E$ in a locally convex space, we have $TU = U \times E$ and $\pi_U \colon U \times E \rightarrow U, (u,e)\mapsto u$. Thus a vector field of $U$ can be written as $X=(X_U,X_E) \colon U \rightarrow U \times E$ and we must have $X_U=\id_U$. Hence a vector field on $U$ is uniquely determined by the smooth map $X_E \in C^\infty (U,E)$.
\end{ex}

\begin{setup}\label{local derivation}
 If $M$ is a manifold and $(\phi,U_\phi)$ a manifold chart, then we have an analogue of $X_E$ on $U_\phi$ for $X \in \mathcal{V}(M)$: Clearly $T\phi \circ X \circ \phi^{-1} = (\id_{V_\phi}, X_\phi)$ for the smooth map $X_\phi \coloneq d\phi\circ X\circ \phi^{-1} \colon V_\phi \rightarrow E$. We call $X_\phi$ the \emph{local representative of} $X$ or \emph{the principal part of $X$}with respect to the chart $\phi$.\index{vector field!local representative} \index{vector field!principal part}
 
 For later use, consider a vector field $X \in \mathcal{V}(M)$ and a smooth function $f \colon M \rightarrow F$, where $F$ is a locally convex space. Then we define a function $X.f \in C^\infty (M,F)$ via
 \begin{align}\label{eq:Liederiv:smoo}
X.f (m) \coloneq df \circ X (m) = \text{pr}_2Tf \circ X (m).
\end{align}
\end{setup}

\begin{setup}\label{lcvx:space}
 Similar to \Cref{topo:sect} we topologise $\mathcal{V}(M)$: Pick an atlas $\mathcal{A}$ of $M$ whose charts we denote by $\varphi \colon U_\varphi \rightarrow V_\varphi \subseteq E_\varphi$. Then we declare the topology to be the initial topology with respect to the map 
 $$\kappa \colon \mathcal{V}(M) \rightarrow \prod_{\varphi \in \mathcal{A}} C^\infty (V_\varphi, E_\varphi), X \mapsto (X_\varphi)_{\varphi \in X},$$
 where the factors on the right hand side carry the compact open $C^\infty$-topology. In particular, this topology turns the vector fields into a locally convex space.
\end{setup}

We will use the notion of integral curves and flows for vector fields, whence we recall the definition of these objects.

\begin{setup}\label{VF:flow}
 Let $X \in \mathcal{V}(M)$, we say a $C^1$-curve $c \colon [a,b] \rightarrow M$ is an \emph{integral curve}\index{vector field!integral curve} for $X$ if for every $y \in [a,b]$ the curve satisfies $\dot{c}(t) = X(c(t))$.
 
 If $M$ is a Banach manifold, it follows from the theory of ordinary differential equations, \cite[IV.]{Lang},that for every $m \in M$ there exists an integral curve $c_m$ of $X$ on some open interval $J_m \coloneq ]-\varepsilon,\varepsilon[$ such that $c_m(0)=m$. Moreover, the \emph{flow}\index{vector field!flow}
 $$\text{Fl}^X \colon \bigcup_{m \in M}\{m\} \times J_m \rightarrow M, (m,t) \mapsto c_m(t),$$
 defines a continuous map on some open subset of $M \times \R$. If $M$ is modelled on a locally convex space, existence of integral curves and flows is not automatic, cf.\ \Cref{Diffeq:beyond}.
\end{setup}

\begin{defn}
 Let $f \colon M \rightarrow N$ be smooth. We call the vector fields $X \in \mathcal{V}(M), Y \in \mathcal{V}(N)$ \emph{$f$-related} if $Y \circ f = Tf \circ X$.\index{vector field!related}
\end{defn}

\begin{lem}\label{glueinglemma}
 Let $M$ be a manifold modelled on a locally convex space $E$ with atlas $\mathcal{A}$. Let $(X_\phi)_{\phi \in \mathcal{A}}$ be a family of smooth maps $X_\phi \colon V_\phi \rightarrow E$ such that every pair $X_\phi, X_\psi$ is $\psi \circ \phi^{-1}$ related on $\phi(U_\psi\cap U_\phi)$. Then there is a unique vector field $X \in \mathcal{V}(M)$ whose local representatives coincide with the $X_\phi$
\end{lem}

\begin{proof}
 Define $X \colon M \rightarrow TM, p \mapsto T\phi^{-1} (\phi(p),X_\phi (\phi (p)))$ for $p\in U_\phi$. Since the maps $X_\phi, X_\psi$ are related by the change of charts on the overlap $U_\phi \cap U_\psi$, the mapping is well defined. By construction it is smooth and a vector field.
\end{proof}

\begin{setup}
For principal parts of vector fields $X,Y$ on $U \opn E$ write $X_E.Y_E(z)\coloneq dY_E \circ X \coloneq  dY_E (z;X_E(z))$. Define $$\LB[X,Y] \coloneq X.Y-Y.X. \qquad X,Y \in C^\infty (U,E).$$
We will see in the following that the bracket of principal parts of vector fields gives rise to a Lie bracket of vector fields.\index{vector field!Lie bracket} 
\end{setup}

\begin{lem}\label{lem:locrelated}
 Let $U \opn E$, $V \opn F$ be open in locally convex spaces and $f \in C^\infty (U,V)$, $X_1,X_2 \in C^\infty (U,E)$ and $Y_1,Y_2 \in C^\infty (V,F)$. Assume that $X_i$ is $f$-related to $Y_i$ for $i=1,2$, then $\LB[X_1,X_2]$ is $f$-related to $\LB[Y_1,Y_2]$.
\end{lem}

\begin{proof}
 Using the chain rule, \eqref{ident:doublederiv} and relatedness we obtain
 \begin{align}\label{localrelation}
  df(x,dX_i(x;v))=dY_i(f(x),df(x;v))-d^2f(x;X_i(x),v) \quad i=1,2 (x,v)\in U\times E
 \end{align}
 We use this relation together with relatedness to obtain
 \begin{align*}
  df(x;\LB[X_1,X_2](x))=&df(x;dX_2 (x,X_1(x)))-df(x;dX_1(x;X_2(x))) \\ 
  =& dY_2 (f(x);df(x;X_1(x)))-d^2f(x;X_2(x),X_1(x))\\ 
  &-dY_1(f(x);df(x;X_2(x)))+d^2f(x;X_1(x),X_2(x))\\
  =& dY_2 (f(x);Y_1(f(x)))-dY_1(f(x);Y_2(f(x)))=\LB[Y_1,Y_2](f(x))
 \end{align*}
where the second order terms cancel by Schwartz' theorem.
\end{proof}

Before we establish now the Lie algebra properties, let us recall a general definition useful for our purpose.

\begin{defn}\label{defn:derivations}
 Let $(A,\cdot)$ be an associative algebra, then the linear mappings $L(A,A)$ form a Lie algebra under the commutator bracket $\LB[\phi,\psi] \coloneq \phi \circ \psi - \psi\circ \phi$, \Cref{ex:assocLie} (where ``$\circ$'' is the usual composition of linear maps). A mapping $\phi \in L(A,A)$ is called \emph{derivation}\index{derivation of an algebra} of the algebra $A$ if it satisfies the Leibniz rule
 $$\phi (a\cdot b) = \phi(a) \cdot b+a \cdot \phi(b) \quad \forall a,b \in A.$$
 We denote by $\text{der} (A)$ the \emph{set of all derivations} of $A$ and note that it forms a Lie subalgebra of $(L(A,A),\LB )$. (As no topology is involved this will in general not be a locally convex Lie algebra.)
\end{defn}
 For $E$ a locally convex space, $U \opn E$ and $X \in \mathcal{V}(U)$ define the \emph{Lie derivative}\index{derivative!Lie derivative}
 \begin{align}\label{Liederiv}
  \mathcal{L}_X(f)\coloneq df\circ X = df (\id_U,X_E) \text{ for } f \in C^\infty (U,\R).
 \end{align}
 By definition $\mathcal{L}_X (f) = X.f$ in the special case that $f$ is real valued. The reason for the new notation and name will become apparent from the following observations (cf.\ also \Cref{defn:Lie derivative}): The pointwise multiplication turns $C^\infty (U,\R)$ into an associative algebra. Then $\mathcal{L}_X$ is linear in $f$. Thus
\begin{align}\label{LD}
 \mathcal{L}_X (f\cdot g) = \mathcal{L}_X (f)\cdot g +f\cdot \mathcal{L}_X (g).
\end{align}
With other words, $\mathcal{L}_X$ is a derivation of the algebra $C^\infty (U,\R)$.

\begin{lem}\label{lem:locLie}
 Let $U \opn E$ in a locally convex space
 \begin{enumerate}
 \item $\mathcal{L}_{\LB[X,Y]} = \mathcal{L}_X \circ \mathcal{L}_Y - \mathcal{L}_Y \circ \mathcal{L}_X$
 \item The map $\mathcal{L} \colon C^\infty (U,E) \rightarrow \text{\textup{der}}(C^\infty (U,E)), X \mapsto \mathcal{L}_X$ is linear and injective
 \item The map $\LB \colon C^\infty (U,E) \times C^\infty (U,E) \rightarrow C^\infty (U,E), (X,Y) \mapsto \LB[X,Y] = X.Y-Y.X$ turns the space $C^\infty (U,E)$ into a Lie algebra.
 \end{enumerate}
\end{lem}

\begin{proof}
 \begin{enumerate}
  \item From \eqref{ident:doublederiv} we deduce 
  $$\mathcal{L}_X(\mathcal{L}_Y(f))=d^2f(x;Y(x),X(x))+df(x;dY(x;X(x))).$$
  Using the formula also for $X$ and $Y$ interchanged, we see that the second order terms cancel by Schwartz' theorem and thus 
  $$\LB[\mathcal{L}_X,\mathcal{L}_Y](f)(x) = \mathcal{L}_{\LB[X,Y]}(f)(x).$$
  \item $\mathcal{L}_X$ is linear in $X$ as $df(x;\cdot)$ is. Thus it suffices to prove that the kernel of $\mathcal{L}$ is trivial. Let $X \in C^\infty (U,E)$ be a map with $X(x)\neq 0$ for some $x \in U$. By the Hahn-Banach theorem \Cref{thm:HahnBanach}, we find $\lambda \in E'$ with $\lambda(X(x))\neq 0$. Then $\mathcal{L}_X(\lambda)(x)=d\lambda(x,X(x))=\lambda(X(x))\neq 0$ and thus $\mathcal{L}_X \neq 0$.
  \item Clearly $\LB$ is bilinear, whence $(C^\infty(U,E),\LB )$ is an algebra. Now $\LB[X,X]=X.X-X.X=0$. Recall that in the Jacobi identity, the entries of the iterated Lie bracket are cyclically permuted. We write shorter $\sum_{\text{cycl}} \LB[X,\LB[Y,Z]]$ for this and thus have to check that this expression vanishes for all $X,Y,Z \in C^\infty(U,E)$. However,
  \begin{align*}
   \mathcal{L}\left(\sum_{\text{cycl}} \LB[X,\LB[Y,Z]]\right) =\sum_{\text{cycl}} \LB[\mathcal{L}_X,\LB[\mathcal{L}_Y,\mathcal{L}_Z]] =0 
  \end{align*}
 where we have used linearity of $\mathcal{L}$, (a), (b) and the fact that the derivations form a Lie algebra. Since $\mathcal{L}$ is injective by (b), we see that the Jacobi identity holds. \qedhere
 \end{enumerate}
\end{proof}

Finally, we show that the Lie bracket of vector fields is continuous if the space $E$ is finite dimensional.

\begin{lem}\label{lem:loccvxLieFD}
 Let $E$ be a finite dimensional space and $U \opn E$. Then the Lie bracket 
 $$\LB \colon C^\infty (U,E) \times C^\infty (U,E) \rightarrow C^\infty (U,E)$$
 is continuous. Hence $(C^\infty (U,E),\LB )$ is a locally convex Lie algebra.\index{Lie algebra!of vector fields}
\end{lem}

\begin{proof}
 Note that $C^\infty (U,E)$ is a locally convex space with respect to the compact open $C^\infty$-topology, \Cref{prop:lcvx_mappingsp}. 
 To establish continuity of the Lie bracket, we deduce from Lemma \Cref{lem:fwedge_vector} that it suffices to establish continuity of adjoint map 
 $$p \colon C^\infty (U,E) \times C^\infty (U,E) \times U\rightarrow E,\quad (X,Y,u) \mapsto dY(u;X(u)).$$
 Recall that the compact open $C^\infty$-topology is initial with respect to the mappings $d^k \colon C^\infty (U,E) \rightarrow C(U\times E^k,E)_{\text{c.o}}, f\mapsto d^kf$.
 Hence the map $d \colon C^\infty (U,E) \rightarrow C(U\times E,E)$ is continuous. We can thus write the adjoint map as a composition of continuous mappings (cf.\ \Cref{lem:eval}, which uses that $U$ is locally compact, i.e.\ $E$ finite dimensional)
 $p (f,g) = \ev(d(f),u,\ev(g,u))$, whence the Lie bracket is continuous.
\end{proof}

\begin{cor}\label{cor:VFlcvx:LA}
 Let $M$ be a finite dimensional manifold. Then $(\mathcal{V}(M),\LB )$ is a locally convex Lie algebra.
\end{cor}

\begin{proof}
That the vector fields form a Lie algebra is checked in Exercise \ref{Ex:LA} 3. Recall from \Cref{lcvx:space} that the vector fields were topologised as a subspace of a product of spaces of the form $C^\infty (U,E)$ where $U \opn M$. By construction of the Lie bracket of two vector fields, the bracket is given by a local formula on chart domains $U$. Hence it suffices to establish continuity of the local formula on the spaces $C^\infty (U,E)$. This was exactly the content of \Cref{lem:loccvxLieFD}
 \end{proof}

\begin{rem}\label{rem:noloccvxLA}
 In general, the Lie algebra of vector fields $\mathcal{V} (M)$ will not be a locally convex Lie algebra if $M$ is an infinite-dimensional manifold. Indeed, it can be shown that \Cref{lem:loccvxLieFD} becomes false beyond the realm of Banach spaces. To see this, let $U \opn E$ be an open subset of a non-normable space. We consider the subalgebra 
 $$\mathcal{A} = \{X_{A,b} \in C^\infty (U,E) \mid \forall v \in E,\ X_{A,b} (v) = Av+b, \text{ for } A \in L (E,E), b \in E\}$$
 of affine vector fields. By construction we can identify $\mathcal{A} \cong L (E,E) \times E$. Here the subspace topology on induced by the compact open $C^\infty$-topology of $C^\infty (U,E)$ on $\mathcal{A}$ is the product topology, where $E$ carries its natural locally convex topology and the space of continuous linear mappings $L(E,E)$ is endowed with the compact open topology (i.e.~the topology induced by the embedding $L(E,E) \subseteq C_{\text{c.o.}}(E,E)$). Indeed the latter fact is irrelevant for us, we are only interested in the fact that this topology turns $L(E,E)$ into a topological vector space. Now, the Lie bracket of $C^\infty (U,E)$ induces the Lie bracket
 \begin{align*}
  \LB[X_{A,b}, X_{C,d}](v) = (A \circ C - C \circ A)(v) + (A(d) - C(b))
 \end{align*}
on the affine vector fields (these facts are left as Exercise \ref{Ex:VFLA} 3.). To see that this Lie bracket is in general not continuous, it suffices to note that the evaluation map $L(E,E) \times E \rightarrow E,\ (A,v) \mapsto A(v)$ is discontinuous. For this we pick $0 \neq v \in E$ and consider the mapping
\begin{align}\label{map:jemb}
 j\colon E^\prime=L(E,\R) \rightarrow L(E,E) ,\quad \lambda \mapsto (x \mapsto \lambda(x)\cdot v)
\end{align} 
If we endow the dual space $E^\prime$ with the compact open topology (again the subspace topology of $E^\prime\subseteq C_{\text{c.o.}}(E,\R)$) then $E^\prime$ becomes a topological vector space and $j$ continuous. However,  \Cref{prop:dual_norm} shows that the evaluation map $E^\prime\times E \rightarrow \R$ is discontinuous for every topological vector space $E$ which is not normable. As $j$ and scalar multiplication in $E$ are continuous, this implies that the evaluation of $L(E,E)$ must be discontinuous if $E$ is not normable. We deduce that the Lie bracket on $C^\infty (U,E)$ must be discontinuous if $E$ is not normable.\footnote{Even stronger one can show that the evaluation must be discontinuous on $L(E,E)$ with the compact open topology for all infinite-dimensional spaces $E$, cf.\ \cite[Remark I.5.3]{Neeb06} for an exposition.}
\end{rem}

\begin{Exercise} \vspace{-\baselineskip} \label{Ex:VFLA}
 \Question Show that the construction of the topology for $\mathcal{V}(M)$ in \Cref{lcvx:space} is just a special case of \Cref{topo:sect}.
 \Question Let $A$ be an associative algebra. Show that the set of derivations $\text{der}(A)$ (see \Cref{defn:derivations}) forms a Lie subalgebra of $(L(A,A),\LB )$, where the bracket is given by the commutator bracket $\LB[f,g] = f\circ g - g\circ f$ of linear maps.\index{derivation of an algebra}
  \Question We provide the missing details in \Cref{rem:noloccvxLA}. To this end let $U \opn E$ in a locally convex space and endow $C^\infty (U,E)$ with the compact-open $C^\infty$-topology (i.e.~the topology induced by the embedding $L(E,E) \subseteq C_{\text{c.o.}}(E,E)$). We consider the affine vector fields  $$\mathcal{A} = \{X_{A,b} \in C^\infty (U,E) \mid \forall v \in E,\ X_{A,b} (v) = Av+b, \text{ for } A \in L (E,E), b \in E\}$$
  and identify $\mathcal{A} = L (E,E) \times E$ (where $L (E,E)$ denote continuous linear maps). Show that 
 \subQuestion the subspace topology on $\mathcal{A}$ is the product topology of the compact open topology on $L(E,E)$ and the locally convex topology of $E$, 
 \subQuestion the Lie bracket on $C^\infty (U,E)$ induces the Lie bracket $$\LB[X_{A,b}, X_{C,d}] = (A \circ C - C \circ A) + (A(d) - C(b))\text{ on }\mathcal{A}.$$
 \subQuestion if we endow the dual space $E^\prime$ with the compact open topology (i.e. the subspace topology of $E^\prime\subseteq C_{\text{c.o.}}(E,\R)$), then $E^\prime$ is a topological vector space and the map $j$ from \eqref{map:jemb} becomes continuous.
\end{Exercise}

\chapter{Differential forms on infinite-dimensional manifolds}\copyrightnotice \label{app:diff_form}

In this appendix we give a short introduction to differential forms on infinite-dimensional manifolds. For more information on differential forms on infinite-dimensional manifolds and their application we refer the interested reader to \cite{MR891612,GNprep}. The main difference between the finite-dimensional (or Banach) and our setting, is that it is in general impossible to interpret differential forms as (smooth) sections into certain bundles of linear forms. The reason for this is again that the topology on spaces of linear forms breaks down beyond the Banach setting (cf.\ \Cref{prop:dual_norm}). Even worse, the many equivalent ways to define differential forms in finite dimensions become inequivalent in the infinite-dimensional setting (see \cite[Section 33]{KM97} for a thorough discussion of this phenomenon). Most notably, there is no useful way to describe differential forms as a sum of differential forms coming from a local coordinate system. 

We begin with the definition of a differential form. This definition is geared towards avoiding any reference to topologies on spaces of linear mappings. This again is a continuity problem and in the inequivalent convenient setting of global analysis, differential forms can be described as sections in suitable bundles, cf.~\cite[33.22 Remark]{KM97}. Furthermore, we need to avoid arguments involving the existence of (smooth) bump functions (which in general do not exist, see \Cref{smooth:bumpfun}).

\begin{defn}\label{defn:differentialform}
 Let $M$ be a manifold and $E$ be a locally convex space and $p\in \N_0$. An \emph{$E$-valued $p$-form}\index{differential form} $\omega$ on $M$ is a function $\omega$ which associates to each $x \in M$ a $p$-linear alternating map $\omega_x \colon (T_x M)^p \rightarrow E$ such that for each chart $(U,\varphi)$ of $M$, the map 
 \begin{align}\label{defn:diffform}
  \omega_\varphi \colon V_\varphi \times F_\varphi^p \rightarrow E,\quad \omega_\varphi (x,v_1,\ldots, v_p) \coloneq \omega_{\varphi^{-1}(x)} (T_x \varphi^{-1}(v_1),\ldots , T_x \varphi^{-1}(v_p))
 \end{align}
 is smooth. We write $\Omega^p (M,E)$ for the \emph{space of smooth $E$-valued $p$-forms}\index{space!of smooth $p$-forms} on $M$. Note that $\Omega^0 (M,E) = C^\infty (M,E)$.
\end{defn}

\begin{ex}
 For $p=0$, we have already seen that smooth functions are differential forms. If $f \in C^\infty (M,E)$, then the derivative $df \colon TM \rightarrow E, v \mapsto \text{pr}_2 \circ Tf(v)$ is a smooth $E$-valued  $1$-form. 
\end{ex}

Constructing the derivative of a smooth function can be generalised to a differential on the space of $p$-forms, the so called \emph{exterior differential}\index{exterior differential} 
$$\mathrm{d} \colon \Omega^p(M,E) \rightarrow \Omega^{p+1} (M,E)$$ 
which we discuss now. On an infinite-dimensional manifold there is no generalisation of local coordinates in a vector basis. Hence the finite-dimensional approach defining the exterior differential in a local coordinate frame is not available. Recall some standard notation useful in the present context: 
In \eqref{eq:Liederiv:smoo} we defined a derivative $X.f$ of a smooth function $f$ in the direction of $X$. If $\omega \in \Omega^p (M,E)$ and $U \opn M$, we define for vector fields $X_1,\ldots X_p \in \mathcal{V} (U)$ a smooth map $$\omega (X_1,\ldots,x_p) \colon U \rightarrow E,\quad m \mapsto \omega_m (X_1 (m),\ldots X_p (m)).$$ 
Finally, we write $\omega (X_0, \ldots, \hat{X}_i, \ldots, X_p)$ to indicate that the $i$th component is to be omitted from the formula.

\begin{prop}\label{ext:deriv}
 For $\omega \in \Omega^p (M,E)$ there exists a smooth $p+1$-form $\mathrm{d}\omega \in \Omega^{p+1}(M,E)$ which for any $U \opn M$ and vector fields $X_1,\ldots , X_p \in \mathcal{V} (U)$ satisfies 
 \begin{align}
  \mathrm{d}\omega (X_0,X_1,\ldots , X_p) (m) =& \sum_{i=0}^p (-1)^{i} (X_i.\omega(X_0,\ldots,\hat{X}_i,\ldots,X_p))(m) \notag \\ &+ \sum_{i <j}(-1)^{i+j} \omega (\LB[X_i,X_j], X_0,\ldots,\hat{X}_i, \ldots \hat{X}_j , \ldots,X_p)(m) \label{ext:diff}
 \end{align}
\end{prop}

\begin{proof}
 Consider $m \in M$ and $v_1,\ldots, v_p \in T_mM$. To define $(\mathrm{d}\omega )_m (v_1,\ldots,v_p)$ we pick an open neighborhood $U$ of $m$ together with vector fields $X_i \in \mathcal{V}(U)$ such that $X_i (m)=v_i, i=0,1,2,\ldots, p$. Note that such vector fields always exist as we can take the constant vector fields in a chart neighborhood (in particular, the definition does not require us to globalise these fields, which would require bump functions which may not exist). Then
 \begin{align}\label{loc:defn}
  (\mathrm{d}\omega )_m (v_0,v_1,\ldots,v_p) \coloneq \mathrm{d}\omega (X_0,X_1,\ldots,X_p) (m),
 \end{align}
where the right hand side has been defined via \eqref{ext:diff} for our choice of vector fields. 

\textbf{Step 1:} $(\mathrm{d}\omega)_m (v_1,\ldots,v_p)$ does not depend on the choice of vector fields in \eqref{loc:defn}.
We have to show that the expression \eqref{ext:diff} becomes $0$ if $X_k (m) = 0$ for at least one $k$. Assuming that $X_k$ vanishes in $m$, we may without loss of generality assume that we are working in local coordinates. We will suppress the chart identification in the formulae and also identify each vector field $X_i$ with its principal part on some $U \opn F$ (where $F$ is a locally convex space). Exploit now that $\omega$ is alternating and linear to see that the contributions in \eqref{ext:diff} which do not directly vanish are 
\begin{align}
 \sum_{i \neq k}^p (-1)^i (X_i.\omega(X_0,\ldots,\hat{X}_i,\ldots,X_p))(m) \label{first:part} \\
 + \sum_{i < k} (-1)^{i+k} \omega (\LB[X_i,X_k],X_0,\ldots,\hat{X}_i,\ldots , \hat{X}_k, \ldots , X_p)(m) \label{higher:part} \\
 + \sum_{k < i} (-1)^{i+k} \omega (\LB[X_k,X_i],X_0,\ldots,\hat{X}_k,\ldots , \hat{X}_i, \ldots , X_p)(m) \label{lower:part}
\end{align}
Apply the definition of the differential form $\omega$ on the open subset of a locally convex space \eqref{defn:diffform}. In this presentation $\omega$ is a function of $p+1$-variables and $p$-linear in the last $p$-variables. Hence we can compute the derivative for a summand in \eqref{first:part} explicitly as:
\begin{align*}
 &X_i.\omega(X_0,\ldots,\hat{X}_i,\ldots,X_p)(m)\\ =& d_1\omega(m,X_1(m),\ldots \hat{X}_i (m),\ldots,X_p(m);X_i(m)) \\ &+ \sum_{j<i} \omega_m (X_0(m),\ldots, dX_j (m; X_i(m)),\ldots,\hat{X}_i (m),\ldots, X_p(m))\\
 &+ \sum_{i<j} \omega_m (X_0(m),\ldots, \hat{X}_i (m), dX_j (m; X_i(m)),\ldots,,\ldots, X_p(m))
\end{align*}
As $X_k$ vanishes, we see that for every $i>k$ only 
$$\omega_m (X_0(m),\ldots,dX_k(m;X_i(m)),\ldots ,\hat{X}_i(m),\ldots, X_p(m))$$
survives and as $X_k (m)=0$, we have 
$$dX_k (m,X_i(m))=dX_k (m,X_i(m))- dX_i (m,X_k(m)) = \LB[X_i,X_k](m)$$
Using that $\omega$ is alternating, we see that 
\begin{align*}
 &(-1)^k\omega(\LB[X_k,X_i], X_0,\ldots ,\hat{X}_k,\ldots, \hat{X}_i ,\ldots, X_p)(m)\\
 =& - \omega_m (X_0(m),\ldots, dX_k (m,X_i(m)),\ldots,\hat{X}_i (m),\ldots, X_p(m))
\end{align*}
hence the corresponding terms in \eqref{first:part} and \eqref{higher:part} cancel. Similar arguments show that this is also happens for the parts in \eqref{first:part} and \eqref{lower:part} if $i < k$. We conclude that \eqref{ext:diff} vanishes at a point if one of the vector fields vanishes at the point. This shows in particular that \eqref{loc:defn} is independent of the choices of vector fields.

\textbf{Step 2:} $\mathrm{d}\omega$ is a smooth $p+1$-form.
For smoothness we work again locally in a chart as in Step 1 and pick all vector fields $X_i$ to be constant. As the Lie bracket of constant vector fields vanishes, \eqref{loc:defn} reduces to 
\begin{align}\label{loc:form2}
 (\mathrm{d}\omega)_m(v_0,\ldots,v_p) = \sum_{i=0}^p (-1)^{i} d_1\omega(m,v_0,\ldots, \hat{v}_i,\ldots, v_p;v_i).
\end{align}
Now by definition $\omega$ induces a smooth function in all charts, whence we see that also $\mathrm{d}\omega$ is smooth in $(m,v_0,v_1,\ldots,v_o)$.
To see that $(\mathrm{d}\omega)_m$ is alternating, observe that the summands in \eqref{loc:form2} are alternating. Assume now that $v_i = v_j$ for some $i <j$ its easy to see that \eqref{loc:form2} vanishes (we leave this as Exercise \ref{Ex:differentialform} 1.).
\end{proof}

We thus obtain for every $p \geq 0$ an exterior differential on the space of $p$-forms. The usual proof (cf.~\cite[Theorem 33.18]{KM97} or \cite[V. Proposition 3.3]{Lang} then shows that $\mathrm{d}^2 = \mathrm{d}  \circ \mathrm{d} = 0$ in every degree. Hence as in the finite-dimensional (or the Banach) setting, the exterior differential gives rise to a cochain complex of differential forms 
$$C^\infty (M,E) = \Omega^0 (M,E) \xrightarrow{\mathrm{d}} \Omega^1(M,E) \xrightarrow{\mathrm{d}} \Omega^3 (M,E)  \xrightarrow{\mathrm{d}} \cdots.$$
Starting from this complex, one can define and study de Rham cohomology on the (infinite-dimensional) manifold $M$. We will not pursue this route here and refer instead to \cite{MR891612} or \cite[Chapter 34]{KM97} for more information.

Differential forms of higher order are typically constructed using the wedge product. The definition is as in the finite-dimensional setting (note however that there are several conventions as to the coefficients, we chose to follow \cite{Lang})

\begin{defn}\label{defn:wedge}
 Let $E_i, i=1,2,3$ be locally convex spaces and $\beta \colon E_1 \times E_2 \rightarrow E_3$ be a continuous bilinear map. Fix $p,q \in \N_0$ and denote by $S_{p+q}$ the symmetric group of all permutations of $\{1,2,\ldots p+q\}$. For $\omega \in \Omega^p (M,E_1)$ and $\eta \in \Omega^q (M,E_2)$ define the \emph{wedge product}\index{wedge product} $\omega \wedge \eta \in \Omega^{p+q}(M,E_3)$ via $(\omega \wedge \eta)_x \coloneq \omega_x \wedge \eta_x$ for $x \in M$, where 
 \begin{align*}
  (\omega_x\wedge \eta_x)(v_1,\ldots,v_{p+q}) \coloneq \frac{1}{p!q!} \sum_{\sigma \in S_{p+q}} \text{sgn} (\sigma) \beta(\omega_x (v_{\sigma (1)},\ldots,v_{\sigma (p)}) , \eta_x (v_{\sigma (p+1)}, \ldots , v_{\sigma (p+q)})
 \end{align*}
Then  $$\wedge \colon \Omega^p(M,E_1) \times \Omega^q(M,E_2) \rightarrow \Omega^{p+q}(M,E_3), \quad (\omega,\eta)\mapsto \omega \wedge \eta$$
is a bilinear map.
\end{defn}

\begin{ex}\label{ex:scalarmult}
 If $s \colon \R \times E \rightarrow E$ is the scalar multiplication and $f\in C^\infty (M,\R)$, then the wedge product of $f$ and $\omega \in \Omega^p(M,E)$ is given by
 $(f\wedge \omega)_x = f(x)\omega_x$. This is usually abbreviated by $f\omega \coloneq f\wedge \omega$ and its easy to see that $\Omega^p(M,E)$ becomes a $C^\infty (M,\R)$-module. 
\end{ex}

\begin{ex}\label{Lie:wedge}
 Let $(E,\LB )$ be a locally convex Lie algebra. Then $\LB \colon E \times E \rightarrow E$ is bilinear and we can construct the wedge product $\wedge$ with respect to the Lie bracket.
 For this special situation we define for $\omega \in \Omega^p (M,E)$ and $\eta \in \Omega^q (M,E)$ the bracket
 $$\LB[\omega,\eta]_\wedge \coloneq \omega \wedge \eta$$
\end{ex}

We will now define several standard operations on differential forms such as the pullback of $p$-forms by smooth mappings.

\begin{defn}\label{defn:pullback}
 Let $\varphi \colon M \rightarrow N$ be a smooth map between manifolds. Then we define for $\omega \in \Omega^p(N,E)$ a $p$-form $\varphi^\omega \in \Omega^p (M,E)$, the \emph{pullback of $\omega$ by $\varphi$}\index{differential form!pullback} via
 $$(\varphi^*\omega)_x (v_1,\ldots v_p) \coloneq \omega_{\varphi(x)} (T_x \varphi (v_1) ,\ldots T_x \varphi (v_p))$$
\end{defn}

Due to the chain rule we immediately have the following rules for the computation of pullbacks 
\begin{setup}
The following rules hold for smooth maps and $p$-forms
$$\id_M^* \omega = \omega, \quad \varphi_1^* (\varphi_2^* \omega) = (\varphi_2 \circ \varphi_1)^*, \quad \varphi^* (\omega \wedge \eta) = \varphi^*\omega \wedge \varphi^* \eta$$
 If $p=0$, i.e.~$\omega = f \in C^\infty (M,E)$, then $\varphi^* f = f\circ \varphi$ and we recover the pullback discussed in the context of manifolds of mappings.
\end{setup}

Finally, we define the Lie derivative of a differential form by a vector field. Before we begin, note that the definition of the Lie derivative has to diverge from the usual definition on finite-dimensional or Banach manifolds. 
This is due to the fact that the common description of the Lie derivative (see e.g.\ \cite[V. \S 2]{Lang}) uses the differential of a flow of a vector field. However, as flows of vector fields are the solutions to certain ordinary differential equations, it is unclear whether the flow of a vector field would exist on the more general manifolds we consider (see \Cref{Diffeq:beyond} for a discussion of this problem). Nevertheless, \cite[V.5 Proposition 5.1]{Lang} shows that for Banach manifolds the following definition coincides with the classical one involving flows. 

\begin{defn}\label{defn:Lie derivative}
 Let $M$ be a manifold and $E$ a locally convex space. For $X \in \mathcal{V} (M)$ and $\omega \in \Omega^p (M,E), p\in \N_0$ we define the \emph{Lie derivative}\index{derivative!Lie derivative}\index{differential form!Lie derivative} $\mathcal{L}_Y \omega \in \Omega^p (M,E)$ as follows
 \begin{align*}
  (\mathcal{L}_Y \omega)_m (v_1,\ldots,v_p) &= Y.\omega(X_1,\ldots,X_p)(m) -\sum_{j=1}^p \omega (X_1 ,\ldots, \LB[Y,X_j],\ldots X_p)(m)\\
  &= Y.\omega(X_1,\ldots,X_p)(m) +\sum_{j=1}^p (-1)^j\omega (\LB[Y,X_j], X_1 ,\ldots, \hat{X}_j,\ldots X_p)(m)
 \end{align*}
 where the $X_i$ are smooth vector fields defined in a neighborhood of $m$ such that $X_i (m) =v_i$. That the Lie derivative is well defined will be checked in Exercise \ref{Ex:differentialform} 4.
\end{defn}

Note that for $\omega \in \Omega^0 (M,E) = C^\infty (M,E)$ the formula of the Lie derivative reduces to $\mathcal{L}_Y\omega = d\omega \circ Y$. This was precisely the formula for the Lie derivative described in \Cref{defn:derivations} for functions.

\begin{Exercise}\label{Ex:differentialform} \vspace{-\baselineskip}%
 \Question Check that the exterior differential $\mathrm{d}\omega$ of a $p$-form is an alternating $p+1$-form.
 \Question Check the details from \Cref{defn:wedge}: Show that
 \subQuestion the wedge product of a $p$-form and a $q$-form is indeed a $p+q$-form,
 \subQuestion the wedge product defines a bilinear map between spaces of differential forms,
 \subQuestion that $\Omega^p(M,E)$ is a $C^\infty (M,\R)$-module (cf.~\Cref{ex:scalarmult}),
 \subQuestion for $p=q=1$ we have $\omega_x \wedge \eta_x (v_1,v_2) = \beta(\omega_x(v_1),\eta_x(v_2))-\beta(\omega_x (v_2),\eta_x(v_2))$.
 \Question Prove that for $\omega \in \Omega^p (M,E)$, $\eta \in \Omega^q (M,F)$ and any wedge product, the following formula holds 
 $$\mathrm{d} (\omega \wedge \eta) = (\mathrm{d}\omega) \wedge \eta + (-1)^p \omega \wedge (\mathrm{d}\eta).$$
 Furthermore, show that for $f\colon N \rightarrow M$ smooth, we have
  $$f^*\mathrm{d}\omega = \mathrm{d}f^*\omega.$$
 {\footnotesize \textbf{Hint:} The assertions are local, whence they can be solved using the local formula for the exterior differential.}
 \Question In this exercise we let $X \in \mathcal{V}(M)$ and $\omega \in \Omega^p (M,E), p \in \N_0$ for $M$ a manifold and $E$ a locally convex space.
 Show that the definition of the Lie derivative $\mathcal{L}_Y \omega$ does not depend on the choice of vector fields $X_i$ in \Cref{defn:Lie derivative}. Conclude that $\mathcal{L}_Y \omega$ is a smooth $p$-form.\\ 
 {\footnotesize \textbf{Hint:} It suffices to show that $\mathcal{L}_Y\omega$ vanishes if $X_i (m)=0$ and this can be checked locally}
 \end{Exercise}

\section{The Maurer-Cartan form on a Lie group}\label{App:MC}
For Lie groups there are two important differential forms induced by the Lie group structure:

\begin{ex}\label{ex:MCform}
 Let $G$ be a Lie group with Lie algebra $\Lf (G)$. then we define the \emph{right Maurer-Cartan form}\index{Maurer-Cartan form} $\kappa^r \in \Omega^1(G,\Lf(G))$ via 
 $$(\kappa^r)_g \colon T_g G \rightarrow \Lf (G), v \mapsto T_g \rho_{g^{-1}} (v)$$
 where $\rho_{g^{-1}} (h) =hg^{-1}$. 
 
 If now $f \in C^\infty (M,G)$, we can define its right logarithmic derivative via 
 $$\delta^{r} f \colon M \rightarrow \Lf (G),\quad \delta^r f \coloneq f^*\kappa^r.$$
 Similarly, one can define the \emph{left Maurer-Cartan form} $\kappa^\ell \in \Omega^1(G,\Lf (G))$ and a left logarithmic derivative by replacing right multiplication with left multiplication in the definition of $\kappa^r$.
 This generalises the construction of the logarithmic derivatives for curves from \Cref{setup:logderiv}. One can show (cf.~Exercise \ref{Ex:MC} 1.) that the left logarithmic derivative of any function satisfies the right Maurer-Cartan equation 
 \begin{align}\label{eq:MC}
  \mathrm{d}\delta^\ell f + \frac{1}{2}\LB[\delta^\ell f,\delta^\ell f]_\wedge = 0
 \end{align}
 where $\LB[\delta^\ell f,\delta^\ell f]_\wedge = \delta^\ell f \wedge \delta^\ell f$ for the wedge product induced by the Lie bracket.
\end{ex}

\begin{defn}
 Let $G$ be a Lie group with Lie algebra $\Lf(G)$ and $\omega \in \Omega^1 (M,\Lf (G))$ for some smooth manifold $M$.
 Then $\omega$ is called \begin{enumerate}
                          \item \emph{integrable}\index{differential form!(locally) integrable} if there exists $f \in C^\infty (M,G)$ with $\omega = \delta^\ell f$.
                          \item \emph{locally integrable} if for every $m \in M$ there is $m \in U \opn M$ such that $\omega|_U$ is integrable.
                         \end{enumerate}
                         
\end{defn}
  Equivalently we could have defined (local) integrability using the right logarithmic derivative.
  With this definition it is possible to formulate the fundamental theorem for Lie group-valued functions with values in regular Lie groups, see \Cref{sect:regularLie}.

 \begin{prop}[Fundamental theorem for Lie group-valued functions]\label{prop:funthm:Lie}
 Let $G$ be a regular Lie group with Lie algebra $\Lf (G)$ and $\omega \in \Omega^{1}(M,\Lf(G))$. If $\omega$ satisfies the right Maurer-Cartan equation \eqref{eq:MC}, then $\omega$ is locally integrable. If in addition, $M$ is simply connected, then $\omega$ is integrable. 
 \end{prop}
 The proof of \Cref{prop:funthm:Lie} needs concepts (e.g.~connections on principal bundles) which we will not introduce here. Instead we refer the interested reader either to the classical proofs for the finite-dimensional setting, e.g.~\cite[3.\S 6-7]{sharpe97} or to the infinite-dimensional sources \cite[Theorem III.2.1]{Neeb06} or \cite[Theorem 40.2]{KM97}.
 
 \begin{lem}\label{lem:uniq_logderiv}
  Let $M$ be a connected manifold and $\varphi , \psi \in C^\infty (M,G)$, where $G$ is a Lie group. \index{derivative!left logarithmic}
  Then $\delta^\ell \varphi = \delta^\ell \psi$ is equivalent to the existence of $g \in G$ with $\varphi = g \cdot \psi$.
 \end{lem}

 \begin{proof}
  If there exists $g \in G$ with $\varphi = g\cdot \psi$, then a straight forward calculation shows that $\delta^\ell \varphi = \delta^\ell \psi$ holds. Assume conversely that $\delta^\ell \varphi =\delta^\ell \psi$. Then define the map $\Psi \coloneq \varphi \cdot \psi^{-1}$ (where the product and inverse are taken pointwise in $G$). Then Exercise \ref{Ex:MC} 3. yields
  $$\delta (\gamma) = T_e \lambda_{\psi} T_e \rho_{\psi^{-1}} (\delta^\ell \varphi - \delta^\ell \psi ) = 0.$$
  or in other words, the map $\gamma$ is locally constant by \Cref{cor:loc_const}. As $M$ is connected, we conclude that $g \coloneq \gamma (m) \in G$ (for any $m \in M$) satisfies $g\circ \psi = \varphi$.
 \end{proof}

 \begin{prop}[Lie II for regular Lie groups]\label{LieII:reg}
  Let $G,H$ be Lie groups with Lie algebras $\Lf(G)$ and $\Lf(H)$, respectively. Let $f\colon \Lf(G) \rightarrow \Lf(H)$ be a morphism of locally convex Lie algebras. If $H$ is a regular Lie group and $G$ is connected and simply connected, then there exists a unique morphism of Lie groups $\varphi \colon G\rightarrow H$ with $\Lf(\varphi) = f$.
 \end{prop}
 
 \begin{proof}
 Since $f$ is continuous we can consider the smooth $1$-form $\alpha \coloneq f \circ \kappa^\ell \in \Omega^{1}(G,\Lf (H)$, where $\kappa^\ell$ is the left Maurer-Cartan form on $G$. Moreover, since the Maurer-Cartan form is left invariant, so is $\alpha$. We consider the Maurer-Cartan equation on $H$ (to mark this we label the bracket operation $\LB^H\wedge$ by $H$). As the wedge is induced by the Lie bracket on $\Lf(H)$, we can exploit that $f$ is a Lie algebra morphism and compute with Exercise \ref{Ex:differentialform} 3. as follows: 
 \begin{align*}
  \mathrm{d} \alpha + \frac{1}{2}\LB[\alpha,\alpha]_\wedge^H = f\circ (\mathrm{d} \kappa^\ell + \frac{1}{2}\LB[\kappa^\ell,\kappa^\ell]_\wedge) = 0,
 \end{align*}
 where the unlabelled bracket is the one induced by the Lie bracket of $\Lf(G)$ and we exploited that the Maurer-Cartan form satisfies the Maurer-Cartan equation on $G$ by Exercise \ref{Ex:MC} 1. Now the fundamental theorem \Cref{prop:funthm:Lie} implies that there is a mapping $\varphi \colon G \rightarrow H$ with $\delta^\ell \varphi = \alpha$. Fixing $\varphi (e_G) = e_H$, this mapping is unique by \Cref{lem:uniq_logderiv}. Consider now $g \in G$ and $\varphi \circ \lambda_g$. Then we pick $v \in T_k G$ and evaluate the differential form
 $$\delta^\ell (\varphi \circ \lambda_g)(v) = T\lambda_{\varphi (gk)^{-1}} T_{gk}\varphi T \lambda_{g} (v) = \delta^{\ell} \varphi (T\lambda_g (v)) = \lambda_g^{\ast} \alpha = \alpha.$$
 Applying now again \Cref{lem:uniq_logderiv} the maps $\varphi \circ \lambda_g$ and $\varphi$ differ only by left translation with an element which we compute as $\varphi \circ \lambda (e_G) = \varphi (g)$. In other words $\varphi (gk) = \varphi (g) \varphi (k)$ for all $k \in G$. Since $g \in G$ was arbitrary, we see that $\varphi$ is indeed a morphism of Lie groups.
 \end{proof}

 \begin{Exercise}\label{Ex:MC} \vspace{-\baselineskip}%
 \Question Let $G$ be a Lie group with Lie algebra $\Lf(G)$ and left Maurer-Cartan form $\kappa^\ell$. Show that $\kappa^\ell$
 \subQuestion is an $\Lf (G)$-valued differential form on $G$ which is left invariant in the sense that $\lambda_g^*\kappa^\ell = \kappa^\ell$ for each $g \in G$ (where $\lambda_g (h) \coloneq gh$,
 \subQuestion satisfies the right Maurer-Cartan equation 
 $$\mathrm{d}\kappa^\ell + \frac{1}{2} \LB[\kappa^\ell, \kappa^\ell]_\wedge =0 .$$
 {\footnotesize \textbf{Hint:} Compute the exterior derivative locally using \eqref{loc:defn}. It suffices to prove the formula using left invariant vector fields (why?)\\ 
  \textbf{Remark:} The ``right Maurer-Cartan equation'' is related to the right principal action of $G$ on itself by multiplication. There is also a corresponding left Maurer-Cartan equation for the right Maurer-Cartan form, where the bracket in the equation gets a negative sign.}
 \subQuestion Deduce that for a smooth function $f$, also $\delta^\ell f$ satisfies the right Maurer-Cartan equation.
 \Question Let $\varphi \colon G \rightarrow H$ be a morphism of Lie groups and $\kappa^\ell$ the (left) Maurer-Cartan form on $G$. Shows that 
 \subQuestion $\delta^\ell \varphi = \Lf (\varphi) \circ \kappa^\ell$
 \subQuestion if $\psi \colon G \rightarrow H$ is another Lie group morphism with $\delta^\ell \varphi = \delta^\ell \psi$, then $\varphi = \psi$.
 \Question Let $f,g \in C^\infty (M,G)$ be smooth maps into a Lie group. Establish the following quotient rule for the left logarithmic derivative 
 $$\delta^{\ell} (f\cdot g^{-1})(m) = T_{\one_G} \lambda_{g(m)} T_{\one_G} \rho_{g^{-1}(m)} (\delta^\ell f(m) - \delta^\ell g(m)) = \Ad_{g(m)} (\delta^\ell f(m)- \delta^\ell g(m)),$$
 where $\one_G \in G$ is the identity element, products and inverses are taken pointwise and $\lambda$ (resp. $\rho$) denote left (resp.~right) multiplication in the Lie group.\\ 
 {\footnotesize \textbf{Hint:} Apply \Cref{tangentLie}.}
 \end{Exercise}

\section{Supplement: Volume form and classical differential operators}\label{supp:diff-findim}
 
In this short supplement we will record some well known facts on differential forms on finite-dimensional (compact) manifolds. Many of these notions are needed in the \Cref{sect:EAtheory} and we recall them for the readers convenience. Thus detailed proofs will in general be omitted in this section. However, all of these results are readily available in the standard finite-dimensional literature (which we will reference).

\begin{tcolorbox}[colback=white,colframe=blue!75!black,title=Conventions]
We fix $(M,g)$ a compact (thus finite-dimensional) and connected  manifold with Riemannian metric $g$. Furthermore, we denote by $d = \dim M$ the dimension of the model space of $M$.
\end{tcolorbox}

Let us first recall that there is another canonical way to define differential forms.
\begin{setup}[{Differential forms as sections, cf.\ \cite[1.4]{MR1330918}, \cite[Section 6]{MR960687}}]\label{setup:diffform:findim} \index{differential form}
 Starting with the tangent bundle $TM$ we can construct the bundle of alternating $k$-forms $\mathcal{A}^k (M) \rightarrow M$ for $k \geq 0$. 
 The fibre of $\mathcal{A}^k(M)$ over $x \in M$ is given by the space of alternating $k$-linear mappings $(T_x M)^k \rightarrow \R$ which we denote by $L^k_a (T_x M , \R)$ and topologise as a subspace of the $k$-linear mappings (which carry the usual norm topology induced by the operator norm for $k$-linear maps, cf.\ \cite[I. \S 2]{Lang}). Further, every chart $(U,\varphi)$ of $M$ induces a vector bundle trivialisation of $\mathcal{A}^k (M)$ over $U$ via $$\kappa_\varphi (x, \omega) \coloneq (\varphi (x), \omega \circ (T_x\varphi^{-1} \times T_x \varphi^{-1} \times \cdots \times T_x \varphi^{-1}).$$  
 Comparing the construction with \eqref{defn:diffform} it becomes clear that for the finite-dimensional manifold $M$ (indeed for any Banach manifold) differential $k$-forms are just smooth sections of $\mathcal{A}^k (M)$. In other words, we obtain
 $$\Omega^k (M) = \Gamma (\mathcal{A}^k (M)) , \quad \forall k \in \N_0.$$
 Moreover, this allows us to topologise the space of differential $k$-forms as a locally convex vector space via \Cref{topo:sect}. Namely if we pick a family of open sets $(U_i)_{i\in I}$ which covers $M$ such that on each $U_i$ there is a bundle trivialisation $\kappa_i \colon \mathcal{A}^k(M)|_{U_i} \rightarrow U_i \times \R^d, V \mapsto (\kappa^1_i (x),\kappa^2_i(x))$, then the mapping
 $$\Omega^k (M) \rightarrow \prod_{i\in I} C^\infty (U_i,L_a^k (\R^d,\R)),\quad \omega \mapsto \omega \circ (\kappa^1_i , \kappa^2_i , \ldots,\kappa^2_i)$$
 is an embedding of $\Omega^k (M)$ as a closed locally convex subspace of the product on the right hand side.  
\end{setup}

Recall that the dimension of the spaces $L_a^k (T_x M , \R)$ depends on the dimension $d= \dim T_x M$. In particular, $L_a^d (T_x M,\R)$ is one-dimensional and a differential form $\mu \in \Omega^d (M)$ which vanishes nowhere is called a \emph{volume form}. \index{differential form!volume form}
On Riemannian manifolds there is a convenient way to construct a volume form associated to the Riemannian metric. 
Recall that a manifold $M$ is orientable, if it admits an atlas $(U_i,\varphi_i)_{i\in I}$ such that the Jacobians of all change of charts are positive. One can prove, \cite[Theorem 1.127]{GaHaL04}, that a Riemannian manifold is orientable if and only if it admits a volume form $\mu$ (induced by the Riemannian metric). Volume forms are the tool of choice to define integration on manifolds, see \cite[Part III]{Lang} or \cite[Section 7]{MR960687}. In particular, we can define the $L^2$-metric on $\mathcal{V} (M)$ in the presence of a volume form as follows
\begin{align}\label{general:L2}
 g_{L^2} (X,Y) = \int_M g (X,Y) \mathrm{d} \mu \qquad X,Y \in \mathcal{V}(M).  
\end{align}
In previous chapters we often considered only integration on $\SSS^1$ since a global parametrisation allowed us to hide the dependence on a volume form and \eqref{general:L2} reduces for $M=\SSS^1$ to \eqref{l2metric}.

\subsection*{Classical differential operators on a Riemannian manifold}
We will now assume that there is a volume form $\mu$ associated to the Riemannian metric on $M$.
Let us then recall the following classical differential operators on $M$.

\begin{defn}[{\cite[Sections 6.5 and 7.5]{MR960687}}]\label{defn:classicdiffop} For a compact Riemannian manifold (orientable in case we need a volume form) we will consider the following differential operators.
\begin{itemize}
 \item For a vector field $X$, there is a unique smooth function $\divr X \colon M \rightarrow \R$, the \emph{divergence}\index{vector field!divergence} of $X$ such that 
$$\mathcal{L}_X \mu = (\divr X) \mu.$$
 \item If $f \in C^\infty (M,\R)$ we exploit that the Riemannian metric induces an isomorphism $TM \cong T^\ast M$ (see \Cref{prop:strongmetric}). Thus the following formula uniquely determines the \emph{gradient of $f$ with respect to} $g$:\index{gradient}
 $$g_m(\grad f(m), v_m) = df (v_m) \quad \forall m \in M, v_m \in T_m M.$$
 \item the \emph{(Hodge-)Laplacian} $\Delta = \mathrm{d} \mathrm{d}^\ast + \mathrm{d}^\ast \mathrm{d}$ associated to the metric (cf.~\cite[p.423]{Lang}). Here $\mathrm{d}$ is the exterior differential from \Cref{ext:deriv} and $\mathrm{d}^\ast$ is the codifferential defined via the Hodge star (this is a finite-dimensional construction which depends on the Riemannian metric $g$, we refer to \cite[p.~464]{Lee13} for more information).  
\end{itemize}
\end{defn}

Having defined the necessary differential operators we recall two decompositions which are for example relevant in geometric hydrodynamics:

\begin{prop}[{Helmholtz Decomposition, cf.\ \cite[Lemma 1.2]{modin2019geometric}}]\label{Helmholtz} \index{vector field!Helmholtz decomposition}
 Let $(M,g)$ be a compact oriented Riemannian manifold with volume form $\mu$ and $X \in \mathcal{V} (M)$. Then there exist $V \in \mathcal{V} (M)$ and $f \in C^\infty (M,\R)$ such that
 $$X = V + \grad f \qquad \text{and } \divr V =0.$$
 Moreover, $V$ and $\grad f$ are orthogonal with respect to the $L^2$-metric \eqref{general:L2}, i.e.\ 
 $$g_{L^2} (V,\grad f) = \int_M g(V,\grad f) \mathrm{d} \mu = 0$$
\end{prop}

Note that since the differential $\mathrm{d}$ and the codifferential $\mathrm{d}^\ast$ make sense for arbitrary $k$-forms, we can also extend the Hodge-Laplacian to $k$-forms. This induces the Hodge decomposition of $k$-forms, see e.g.\ \cite[Proposition 8.2]{MR2744150}. We will not recall it here, but would like to mention that it is an important ingredient to establish the Lie group structure of the groups of volume preserving diffeomorphisms and the symplectomorphism group, \Cref{ex:geom_subgroup}. As we have now the necessary notation in place, let us very briefly sketch the idea of the proof

\begin{setup}[{Submanifold structure of volume preserving diffeomorphisms (Sketch)}] \label{volumepreser} \index{diffeomorphism group!volume preserving}
 Define the map $\Psi_\mu \colon \Diff(M) \rightarrow \Omega^d (M), \phi \mapsto \phi^\ast \mu$. Since $M$ is compact, we endow $\Omega^d (M)$ via \Cref{setup:diffform:findim} with a locally convex vector space structure. With some work one can show that $\Psi_\mu$ is a smooth map with derivative $T_\phi \Psi_\mu (V_\phi) = \phi^\ast (\mathcal{L}_{V_\phi} \omega)$ (this follows somewhat similarly to the proof that the pullback with smooth functions is smooth). 
 Then one needs to prove that $\Psi_\mu$ is a submersion onto the cohomology class $[\mu] = \mu + \mathrm{d}\Omega^{d-1}(M) \subseteq \Omega^d(M)$ of $\mu$. The proof uses the Hodge decomposition of $d$-forms to construct a splitting of the kernel of $\Psi_\mu$. Further, one needs to work in a Sobolev completion of $\Diff (M)$, whence this is beyond the techniques we are developing in this book. We refer the interested reader to \cite{Smo07,EM70}. Then the volume preserving diffeomorphism group is simply the preimage $\Psi_\mu^{-1}(\mu)$ of the singleton $\mu$. In particular, $\Diff_\mu (M)$ is a submanifold and thus a Lie subgroup of $\Diff (M)$.
 
 Note that it is apparent from the derivative of $\Psi_\mu$ and Exercise \ref{Ex:submersion} 4.\ that the Lie algebra of $\Diff_\mu (K)$ is $\mathcal{V}_\mu (M) = \{X \in \mathcal{V} (M) \mid \divr X = 0\}$ the Lie algebra of divergence free vector fields.\index{Lie algebra!of divergence free vector fields}
\end{setup}

 \begin{Exercise} \vspace{-\baselineskip} \label{ex:classical}
  \Question Show that the structure described in \Cref{setup:diffform:findim} yields a vector bundle $\mathcal{A}^k(M) \rightarrow M$.
  \Question Prove that the characterisation of differential forms via the bundle in \Cref{setup:diffform:findim} coincides with the one from \Cref{defn:differentialform}.
  \Question Consider $\R^d$ as a Riemannian manifold with the standard euclidean metric. Convince yourself that $\divr, \grad$ and $\Delta$ are ``the usual'' differential operators from vector calculus in this case.
  \Question Work out the details for \Cref{volumepreser} (note that this requires the Hodge Decomposition Theorem \cite[Theorem 7.5.3]{MR960687}): Show that
  \subQuestion $\Psi_\mu$ is smooth with surjective derivative,
  \subQuestion show that the kernel of $T_\eta \Psi_\mu$ is a split subspace of $T_\eta \Diff (M)$.\\
  {\footnotesize \textbf{Remark:} If $\Diff (M)$ were a Banach manifold the above would imply that $\Psi_\mu$ is a submersion. This is one reason why manifolds of finitely often differentiable mappings enter the picture here: The same statements as in the $C^\infty$-case can be proven and these manifolds turn out to be Banach manifolds.}  
  \Question Let $(M,g)$ be a Riemannian manifold with metric derivative $\nabla$. Show that for vector fields $X,Y, Z \in \mathcal{V} (M)$ the following formula holds  
  $$g(Z,\grad g(X,Y)) = g(\nabla_Z X,Y) + g(X,\nabla_Z Y).$$
  Then deduce that this implies $g(\nabla_X Y , Y) = \frac{1}{2} g(X,\grad g(Y,Y))$.
 \end{Exercise}
\setboolean{firstanswerofthechapter}{true}
\begin{Answer}[number={\ref{ex:classical} 5.}] 
 \emph{For a Riemannian manifold $(M,g)$ with metric derivative $\nabla$, we establish the formula
 $$g(Z,\grad g(X,Y)) = g(\nabla_Z X,Y) + g(X,\nabla_Z Y) \text{ for } X,Y,Z \in \mathcal{V} (M).$$
}
From the definition of the gradient and the compatibility of the metric derivative with $g$ we deduce that
\begin{align*}
 g(Z,\grad g(X,Y)) = dg(X,Y)(\cdot;Z) = Z.g(X,Y) = g(\nabla_Z X,Y) + g(X,\nabla_Z Y)
\end{align*}
\end{Answer}
\setboolean{firstanswerofthechapter}{false}

\chapter{Solutions to selected Exercises}
\shipoutAnswer
\addcontentsline{toc}{chapter}{References}
\bibliography{Schmeding_ItidGeo}

\begin{thebibliography}{ACMM89}
\providecommand{\url}[1]{\texttt{#1}}
\providecommand{\urlprefix}{URL }
\expandafter\ifx\csname urlstyle\endcsname\relax
  \providecommand{\doi}[1]{doi:\discretionary{}{}{}#1}\else
  \providecommand{\doi}{doi:\discretionary{}{}{}\begingroup
  \urlstyle{rm}\Url}\fi
\providecommand{\eprint}[2][]{\url{#2}}

\bibitem[AC09]{zbMATH05649883}
{Agrachev}, A.~A. and {Caponigro}, M.
\newblock \emph{{Controllability on the group of diffeomorphisms}}.
\newblock {Ann. Inst. Henri Poincar\'e, Anal. Non Lin\'eaire} \textbf{26}
  (2009)(6):2503--2509.
\newblock \doi{10.1016/j.anihpc.2009.07.003}

\bibitem[ACMM89]{AaCaMaM89}
{Abbati}, M.~C., {Cirelli}, R., {Mania'}, A. and {Michor}, P.
\newblock \emph{{The Lie group of automorphisms of a principal bundle}}.
\newblock {J. Geom. Phys.} \textbf{6} (1989)(2):215--235

\bibitem[AGS20]{AaGaS18}
Amiri, H., Gl{\"o}ckner, H. and Schmeding, A.
\newblock \emph{{L}ie groupoids of mappings taking values in a {L}ie groupoid}.
\newblock Archivum Mathematicum \textbf{56} (2020):307--356.
\newblock \doi{10.5817/AM2020-5-307}

\bibitem[AMR88]{MR960687}
Abraham, R., Marsden, J.~E. and Ratiu, T.
\newblock \emph{Manifolds, tensor analysis, and applications}, \emph{Applied
  Mathematical Sciences}, vol.~75 (Springer-Verlag, New York, 1988), second
  edn.
\newblock \doi{10.1007/978-1-4612-1029-0}

\bibitem[Arn66]{MR202082}
Arnold, V.
\newblock \emph{Sur la g\'{e}om\'{e}trie diff\'{e}rentielle des groupes de
  {L}ie de dimension infinie et ses applications \`a l'hydrodynamique des
  fluides parfaits}.
\newblock Ann. Inst. Fourier (Grenoble) \textbf{16} (1966)(fasc., fasc.
  1):319--361

\bibitem[AS15]{MR3342623}
Alzaareer, H. and Schmeding, A.
\newblock \emph{Differentiable mappings on products with different degrees of
  differentiability in the two factors}.
\newblock Expo. Math. \textbf{33} (2015)(2):184--222.
\newblock \doi{10.1016/j.exmath.2014.07.002}

\bibitem[AS19a]{AaS19}
Amiri, H. and Schmeding, A.
\newblock \emph{A differentiable monoid of smooth maps on {L}ie groupoids}.
\newblock J. Lie Theory \textbf{29} (2019)(4):1167--1192.
\newblock \doi{10.1007/s10455-019-09650-3}

\bibitem[AS19b]{AaS19b}
{Amiri}, H. and {Schmeding}, A.
\newblock \emph{{Linking Lie groupoid representations and representations of
  infinite-dimensional Lie groups}}.
\newblock {Ann. Global Anal. Geom.} \textbf{55} (2019)(4):749--775

\bibitem[Atk75]{MR400283}
Atkin, C.~J.
\newblock \emph{The {H}opf-{R}inow theorem is false in infinite dimensions}.
\newblock Bull. London Math. Soc. \textbf{7} (1975)(3):261--266.
\newblock \doi{10.1112/blms/7.3.261}

\bibitem[Atk97]{MR1432537}
Atkin, C.~J.
\newblock \emph{Geodesic and metric completeness in infinite dimensions}.
\newblock Hokkaido Math. J. \textbf{26} (1997)(1):1--61.
\newblock \doi{10.14492/hokmj/1351257804}

\bibitem[{Bae}97]{Baez97}
{Baez}, J.~C.
\newblock \emph{{An introduction to \(n\)-categories}}.
\newblock In \emph{{Category theory and computer science. 7th international
  conference, CTCS '97, Santa Margherita Ligure, Italy, September 4--6, 1997.
  Proceedings}}, pp. 1--33 (Berlin: Springer, 1997)

\bibitem[{Ban}88]{Ban88}
{Banyaga}, A.
\newblock \emph{{On isomorphic classical diffeomorphism groups. II}}.
\newblock {J. Differ. Geom.} \textbf{28} (1988)(1):23--35

\bibitem[Bas64]{bastiani}
Bastiani, A.
\newblock \emph{Applications diff\'erentiables et vari\'et\'es
  diff\'erentiables de dimension infinie}.
\newblock J. Analyse Math. \textbf{13} (1964):1--114

\bibitem[{Bau}14]{Baum14}
{Baum}, H.
\newblock \emph{{Eichfeldtheorie. Eine Einf\"uhrung in die
  Differentialgeometrie auf Faserb\"undeln}} (Heidelberg: Springer Spektrum,
  2014), 2nd revised ed. edn.

\bibitem[BBM14a]{MR3265197}
Bauer, M., Bruveris, M. and Michor, P.~W.
\newblock \emph{Homogeneous {S}obolev metric of order one on diffeomorphism
  groups on real line}.
\newblock J. Nonlinear Sci. \textbf{24} (2014)(5):769--808.
\newblock \doi{10.1007/s00332-014-9204-y}

\bibitem[BBM14b]{BaBaM14}
{Bauer}, M., {Bruveris}, M. and {Michor}, P.~W.
\newblock \emph{{Overview of the geometries of shape spaces and diffeomorphism
  groups}}.
\newblock {J. Math. Imaging Vis.} \textbf{50} (2014)(1-2):60--97

\bibitem[BD95]{MR1410059}
Br\"{o}cker, T. and tom Dieck, T.
\newblock \emph{Representations of compact {L}ie groups}, \emph{Graduate Texts
  in Mathematics}, vol.~98 (Springer-Verlag, New York, 1995).
\newblock Translated from the German manuscript, Corrected reprint of the 1985
  translation

\bibitem[BDS16]{BaDaS16}
{Bogfjellmo}, G., {Dahmen}, R. and {Schmeding}, A.
\newblock \emph{{Character groups of Hopf algebras as infinite-dimensional Lie
  groups}}.
\newblock {Ann. Inst. Fourier} \textbf{66} (2016)(5):2101--2155.
\newblock \doi{10.5802/aif.3059}

\bibitem[Beg87]{MR891612}
Beggs, E.~J.
\newblock \emph{The de {R}ham complex on infinite-dimensional manifolds}.
\newblock Quart. J. Math. Oxford Ser. (2) \textbf{38} (1987)(150):131--154.
\newblock \doi{10.1093/qmath/38.2.131}

\bibitem[BF66]{MR0198492}
Bonic, R. and Frampton, J.
\newblock \emph{Smooth functions on {B}anach manifolds}.
\newblock J. Math. Mech. \textbf{15} (1966):877--898

\bibitem[BGJP19]{BaGaJaP19}
{Belti\c{t}\u{a}}, D., {Goli\'nski}, T., {Jakimowicz}, G. and {Pelletier}, F.
\newblock \emph{{Banach-Lie groupoids and generalized inversion}}.
\newblock {J. Funct. Anal.} \textbf{276} (2019)(5):1528--1574

\bibitem[BGN04]{MR2069671}
Bertram, W., Gl{\"o}ckner, H. and Neeb, K.-H.
\newblock \emph{Differential calculus over general base fields and rings}.
\newblock Expo. Math. \textbf{22} (2004)(3):213--282.
\newblock \doi{10.1016/S0723-0869(04)80006-9}

\bibitem[BGT18]{BaGaT18}
Beltiţă, D., Goliński, T. and Tumpach, A.-B.
\newblock \emph{Queer poisson brackets}.
\newblock Journal of Geometry and Physics \textbf{132} (2018):358--362.
\newblock \doi{https://doi.org/10.1016/j.geomphys.2018.06.013}

\bibitem[BH99]{MR1744486}
Bridson, M.~R. and Haefliger, A.
\newblock \emph{Metric spaces of non-positive curvature}, \emph{Grundlehren der
  Mathematischen Wissenschaften [Fundamental Principles of Mathematical
  Sciences]}, vol. 319 (Springer-Verlag, Berlin, 1999).
\newblock \doi{10.1007/978-3-662-12494-9}

\bibitem[BM20]{BaM20}
{Bauer}, M. and {Modin}, K.
\newblock \emph{{Semi-invariant Riemannian metrics in hydrodynamics}}.
\newblock {Calc. Var. Partial Differ. Equ.} \textbf{59} (2020)(2):25.
\newblock \doi{10.1007/s00526-020-1722-x}.
\newblock Id/No 65

\bibitem[Bom67]{MR0237728}
Boman, J.
\newblock \emph{Differentiability of a function and of its compositions with
  functions of one variable}.
\newblock Math. Scand. \textbf{20} (1967):249--268.
\newblock \doi{10.7146/math.scand.a-10835}

\bibitem[Bou98]{MR1728312}
Bourbaki, N.
\newblock \emph{Lie groups and {L}ie algebras. {C}hapters 1--3}.
\newblock Elements of Mathematics (Berlin) (Springer-Verlag, Berlin, 1998).
\newblock Translated from the French, Reprint of the 1989 English translation

\bibitem[Bru18]{BruL2}
Bruveris, M.
\newblock \emph{{The $L^2$-metric on $C^\infty (M,N)$}} 2018.
\newblock \eprint{arxiv:1804.00577.pdf}

\bibitem[Bru19]{BruSC}
Bruveris, M.
\newblock \emph{Riemannian Geometry for Shape Analysis and Computational
  Anatomy}, pp. 15--44.
\newblock Lecture Notes Series, Institute for Mathematical Sciences, National
  University of Singapore (2019).
\newblock \doi{10.1142/9789811200137_0002}.
\newblock
  \eprint{https://www.worldscientific.com/doi/pdf/10.1142/9789811200137_0002}

\bibitem[BS17]{BaS17}
{Bogfjellmo}, G. and {Schmeding}, A.
\newblock \emph{{The Lie group structure of the Butcher group}}.
\newblock {Found. Comput. Math.} \textbf{17} (2017)(1):127--159.
\newblock \doi{10.1007/s10208-015-9285-5}

\bibitem[BS18]{BaS18}
{Bogfjellmo}, G. and {Schmeding}, A.
\newblock \emph{{The geometry of characters of Hopf algebras}}.
\newblock In \emph{{Computation and combinatorics in dynamics, stochastics and
  control. The Abel symposium, Rosendal, Norway, August 16--19, 2016. Selected
  papers}}, pp. 159--185 (Cham: Springer, 2018)

\bibitem[CEMM20]{CaEFaMaMK20}
{Curry}, C., {Ebrahimi-Fard}, K., {Manchon}, D. and {Munthe-Kaas}, H.~Z.
\newblock \emph{{Planarly branched rough paths and rough differential equations
  on homogeneous spaces}}.
\newblock {J. Differ. Equations} \textbf{269} (2020)(11):9740--9782.
\newblock \doi{10.1016/j.jde.2020.06.058}

\bibitem[CES16]{CaEaS16}
Celledoni, E., Eslitzbichler, M. and Schmeding, A.
\newblock \emph{Shape analysis on {L}ie groups with applications in computer
  animation}.
\newblock J. Geom. Mech. \textbf{8} (2016)(3):273--304.
\newblock \doi{10.3934/jgm.2016008}

\bibitem[CES18]{CaEaS18}
{Celledoni}, E., {Eidnes}, S. and {Schmeding}, A.
\newblock \emph{{Shape analysis on homogeneous spaces: a generalised SRVT
  framework}}.
\newblock In \emph{{Computation and combinatorics in dynamics, stochastics and
  control. The Abel symposium, Rosendal, Norway, August 16--19, 2016. Selected
  papers}}, pp. 187--220 (Cham: Springer, 2018)

\bibitem[CK16]{chevyrev2016primer}
Chevyrev, I. and Kormilitzin, A.
\newblock \emph{A primer on the signature method in machine learning} 2016.
\newblock \eprint{arXiv:1603.03788}

\bibitem[dF18]{dHaF18}
{del Hoyo}, M. and {Fernandes}, R.~L.
\newblock \emph{{Riemannian metrics on Lie groupoids}}.
\newblock {J. Reine Angew. Math.} \textbf{735} (2018):143--173

\bibitem[DK00]{DaK00}
{Duistermaat}, J.~J. and {Kolk}, J. A.~C.
\newblock \emph{{Lie groups}} (Berlin: Springer, 2000)

\bibitem[{do }92]{doC92}
{do Carmo}, M.
\newblock \emph{{Riemannian geometry. Translated from the Portuguese by Francis
  Flaherty}} (Boston, MA etc.: Birkh\"auser, 1992)

\bibitem[{Dob}95]{Dob95}
{Dobrowolski}, T.
\newblock \emph{{Every infinite-dimensional Hilbert space is real-analytically
  isomorphic with its unit sphere}}.
\newblock {J. Funct. Anal.} \textbf{134} (1995)(2):350--362

\bibitem[DS20]{DaS20}
{Dahmen}, R. and {Schmeding}, A.
\newblock \emph{{Lie groups of controlled characters of combinatorial Hopf
  algebras}}.
\newblock {Ann. Inst. Henri Poincar\'e D, Comb. Phys. Interact. (AIHPD)}
  \textbf{7} (2020)(3):395--456.
\newblock \doi{10.4171/AIHPD/90}

\bibitem[{Ebi}15]{Ebin15}
{Ebin}, D.~G.
\newblock \emph{{Groups of diffeomorphisms and fluid motion: reprise}}.
\newblock In \emph{Geometry, mechanics, and dynamics. The legacy of Jerry
  Marsden. Selected papers presented at a focus program, Fields Institute for
  Research in Mathematical Sciences, Toronto, Canada, July 2012}, pp. 99--105
  (New York, NY: Springer, 2015).
\newblock \doi{10.1007/978-1-4939-2441-7_6}

\bibitem[Eel66]{Eel66}
Eells, Jr., J.
\newblock \emph{A setting for global analysis}.
\newblock Bull. Amer. Math. Soc. \textbf{72} (1966):751--807.
\newblock \doi{10.1090/S0002-9904-1966-11558-6}

\bibitem[EM70]{EM70}
Ebin, D.~G. and Marsden, J.
\newblock \emph{Groups of diffeomorphisms and the motion of an incompressible
  fluid}.
\newblock Ann. of Math. (2) \textbf{92} (1970):102--163.
\newblock \doi{10.2307/1970699}

\bibitem[Eng89]{Eng89}
Engelking, R.
\newblock \emph{General topology}, \emph{Sigma Series in Pure Mathematics},
  vol.~6 (Heldermann Verlag, Berlin, 1989), second edn.
\newblock Translated from the Polish by the author

\bibitem[EP20]{EaP2}
Eftekharinasab, K. and Petrusenko, V.
\newblock \emph{{Finslerian geodesics on Fr\'{e}chet manifolds}}.
\newblock Bulletin of the Transilvania University of Brasov, SERIES III -
  MATEMATICS, INFORMATICS, PHYSICS \textbf{13(62)} (2020)(1):129–152.
\newblock \doi{10.31926/but.mif.2020.13.62.1.11}

\bibitem[FH20]{FaH20}
{Friz}, P.~K. and {Hairer}, M.
\newblock \emph{{A course on rough paths. With an introduction to regularity
  structures. 2nd updated edition}} (Cham: Springer, 2020), 2nd updated edition
  edn.

\bibitem[Fil82]{MR693972}
Filipkiewicz, R.~P.
\newblock \emph{Isomorphisms between diffeomorphism groups}.
\newblock Ergodic Theory Dynam. Systems \textbf{2} (1982)(2):159--171 (1983).
\newblock \doi{10.1017/s0143385700001486}

\bibitem[FK88]{MR961256}
Fr{\"o}licher, A. and Kriegl, A.
\newblock \emph{Linear spaces and differentiation theory}.
\newblock Pure and Applied Mathematics (New York) (John Wiley \& Sons, Ltd.,
  Chichester, 1988).
\newblock A Wiley-Interscience Publication

\bibitem[FS82]{FaS82}
{Folland}, G.~B. and {Stein}, E.~M.
\newblock \emph{{Hardy spaces on homogeneous groups}}, vol.~28 (Princeton
  University Press, Princeton, NJ, 1982)

\bibitem[FV10]{FaV10}
{Friz}, P.~K. and {Victoir}, N.~B.
\newblock \emph{{Multidimensional stochastic processes as rough paths. Theory
  and applications.}}, vol. 120 (Cambridge: Cambridge University Press, 2010).
\newblock \doi{10.1017/CBO9780511845079}

\bibitem[GHL04]{GaHaL04}
{Gallot}, S., {Hulin}, D. and {Lafontaine}, J.
\newblock \emph{{Riemannian geometry. 3rd ed}} (Berlin: Springer, 2004), 3rd
  ed. edn.

\bibitem[Gl{\"o}02]{MR1911979}
Gl{\"o}ckner, H.
\newblock \emph{Infinite-dimensional {L}ie groups without completeness
  restrictions}.
\newblock In \emph{Geometry and analysis on finite- and infinite-dimensional
  {L}ie groups (Bedlewo, 2000)}, \emph{Banach Center Publ.}, vol.~55, pp.
  43--59 (Polish Acad. Sci. Inst. Math., Warsaw, 2002).
\newblock \doi{10.4064/bc55-0-3}

\bibitem[Gl{\"o}06a]{MR2213538}
Gl{\"o}ckner, H.
\newblock \emph{Discontinuous non-linear mappings on locally convex direct
  limits}.
\newblock Publ. Math. Debrecen \textbf{68} (2006)(1-2):1--13

\bibitem[Gl{\"o}06b]{MR2269430}
Gl{\"o}ckner, H.
\newblock \emph{Implicit functions from topological vector spaces to {B}anach
  spaces}.
\newblock Israel J. Math. \textbf{155} (2006):205--252.
\newblock \doi{10.1007/BF02773955}

\bibitem[Gl{\"o}07]{Glo07}
Gl{\"o}ckner, H.
\newblock \emph{{Implicit functions from topological vector spaces to Fréchet
  spaces in the presence of metric estimates}} 2007.
\newblock \urlprefix\url{http://arxiv.org/abs/math/0612673v5}.
\newblock \eprint{arXiv:math/0612673v5}

\bibitem[Gl{\"o}16]{Glofun}
Gl{\"o}ckner, H.
\newblock \emph{{Fundamentals of submersions and immersions between
  infinite-dimensional manifolds}} 2016.
\newblock \eprint{arXiv:1502.05795v4}

\bibitem[GMV15]{zbMATH06551456}
{Grong}, E., {Markina}, I. and {Vasil'ev}, A.
\newblock \emph{{Sub-Riemannian geometry on infinite-dimensional manifolds}}.
\newblock {J. Geom. Anal.} \textbf{25} (2015)(4):2474--2515.
\newblock \doi{10.1007/s12220-014-9523-0}

\bibitem[GN]{GNprep}
Gl{\"o}ckner, H. and Neeb, K.-H.
\newblock \emph{{Infinite-dimensional Lie groups}}.
\newblock In preparation

\bibitem[GN12]{MR2997582}
Gl\"{o}ckner, H. and Neeb, K.-H.
\newblock \emph{When unit groups of continuous inverse algebras are regular
  {L}ie groups}.
\newblock Studia Math. \textbf{211} (2012)(2):95--109.
\newblock \doi{10.4064/sm211-2-1}

\bibitem[GPS21]{gray2021continuity}
Gray, W.~S., Palmstr\o{}m, M. and Schmeding, A.
\newblock \emph{Continuity of formal power series products in nonlinear control
  theory} 2021.
\newblock \eprint{2102.05732}

\bibitem[GS21]{GaS21}
Gl\"{o}ckner, H. and Schmeding, A.
\newblock \emph{Manifolds of mappings on cartesian products} 2021.
\newblock \eprint{2109.01804}

\bibitem[{Gub}10]{Gub10}
{Gubinelli}, M.
\newblock \emph{{Ramification of rough paths}}.
\newblock {J. Differ. Equations} \textbf{248} (2010)(4):693--721.
\newblock \doi{10.1016/j.jde.2009.11.015}

\bibitem[Ham82]{MR656198}
Hamilton, R.~S.
\newblock \emph{The inverse function theorem of {N}ash and {M}oser}.
\newblock Bull. Amer. Math. Soc. (N.S.) \textbf{7} (1982)(1):65--222.
\newblock \doi{10.1090/S0273-0979-1982-15004-2}

\bibitem[Hir94]{Hir76}
Hirsch, M.~W.
\newblock \emph{Differential topology}, \emph{Graduate Texts in Mathematics},
  vol.~33 (Springer-Verlag, New York, 1994).
\newblock Corrected reprint of the 1976 original

\bibitem[HM07]{HaM07}
{Hofmann}, K.~H. and {Morris}, S.~A.
\newblock \emph{{The Lie theory of connected pro-Lie groups. A structure theory
  for pro-Lie algebras, pro-Lie groups, and connected locally compact groups}},
  vol.~2 (Z\"urich: European Mathematical Society (EMS), 2007).
\newblock \doi{10.4171/032}

\bibitem[HN09]{HaN09}
{Hofmann}, K.~H. and {Neeb}, K.-H.
\newblock \emph{{Pro-Lie groups which are infinite-dimensional Lie groups}}.
\newblock {Math. Proc. Camb. Philos. Soc.} \textbf{146} (2009)(2):351--378.
\newblock \doi{10.1017/S030500410800128X}

\bibitem[HN12]{HaN12}
Hilgert, J. and Neeb, K.-H.
\newblock \emph{Structure and geometry of {L}ie groups}.
\newblock Springer Monographs in Mathematics (Springer, New York, 2012).
\newblock \doi{10.1007/978-0-387-84794-8}

\bibitem[HS17]{HaS17}
{Hjelle}, E.~O. and {Schmeding}, A.
\newblock \emph{{Strong topologies for spaces of smooth maps with
  infinite-dimensional target}}.
\newblock {Expo. Math.} \textbf{35} (2017)(1):13--53.
\newblock \doi{10.1016/j.exmath.2016.07.004}

\bibitem[{Hus}93]{Hus93}
{Husemoller}, D.~H.
\newblock \emph{{Fibre bundles. 3rd ed}}, vol.~20 (Berlin: Springer-Verlag,
  1993), 3rd ed. edn.

\bibitem[IKT13]{IaKaT13}
{Inci}, H., {Kappeler}, T. and {Topalov}, P.
\newblock \emph{{On the regularity of the composition of diffeomorphisms}},
  vol. 1062 (Providence, RI: American Mathematical Society (AMS), 2013).
\newblock \doi{10.1090/S0065-9266-2013-00676-4}

\bibitem[IZ13]{MR3025051}
Iglesias-Zemmour, P.
\newblock \emph{Diffeology}, \emph{Mathematical Surveys and Monographs}, vol.
  185 (American Mathematical Society, Providence, RI, 2013).
\newblock \doi{10.1090/surv/185}

\bibitem[Jar81]{Jar81}
Jarchow, H.
\newblock \emph{Locally convex spaces} (B. G. Teubner, Stuttgart, 1981).
\newblock Mathematische Leitf\"{a}den. [Mathematical Textbooks]

\bibitem[Kel74]{MR0440592}
Keller, H.~H.
\newblock \emph{Differential calculus in locally convex spaces}.
\newblock Lecture Notes in Mathematics, Vol. 417 (Springer-Verlag, Berlin-New
  York, 1974)

\bibitem[Kli95]{MR1330918}
Klingenberg, W. P.~A.
\newblock \emph{Riemannian geometry}, \emph{De Gruyter Studies in Mathematics},
  vol.~1 (Walter de Gruyter \& Co., Berlin, 1995), second edn.
\newblock \doi{10.1515/9783110905120}

\bibitem[KM97]{KM97}
Kriegl, A. and Michor, P.~W.
\newblock \emph{The convenient setting of global analysis}, \emph{Mathematical
  Surveys and Monographs}, vol.~53 (American Mathematical Society, Providence,
  RI, 1997)

\bibitem[Kol08]{MR2434729}
Kolev, B.
\newblock \emph{Geometric differences between the {B}urgers and the
  {C}amassa-{H}olm equations}.
\newblock J. Nonlinear Math. Phys. \textbf{15} (2008)(suppl. 2):116--132.
\newblock \doi{10.2991/jnmp.2008.15.s2.9}

\bibitem[KW09]{MR2456522}
Khesin, B. and Wendt, R.
\newblock \emph{The geometry of infinite-dimensional groups}, \emph{Ergebnisse
  der Mathematik und ihrer Grenzgebiete. 3. Folge. A Series of Modern Surveys
  in Mathematics [Results in Mathematics and Related Areas. 3rd Series. A
  Series of Modern Surveys in Mathematics]}, vol.~51 (Springer-Verlag, Berlin,
  2009)

\bibitem[Lan99]{Lang}
Lang, S.
\newblock \emph{Fundamentals of differential geometry}, \emph{Graduate Texts in
  Mathematics}, vol. 191 (Springer-Verlag, New York, 1999).
\newblock \doi{10.1007/978-1-4612-0541-8}

\bibitem[{Lar}19]{Laro19}
{Larotonda}, G.
\newblock \emph{{Metric geometry of infinite-dimensional Lie groups and their
  homogeneous spaces}}.
\newblock {Forum Math.} \textbf{31} (2019)(6):1567--1605.
\newblock \doi{10.1515/forum-2019-0127}

\bibitem[Lee13]{Lee13}
Lee, J.~M.
\newblock \emph{Introduction to smooth manifolds}, \emph{Graduate Texts in
  Mathematics}, vol. 218 (Springer, New York, 2013), second edn.

\bibitem[{Len}07]{Len07}
{Lenells}, J.
\newblock \emph{{The Hunter-Saxton equation describes the geodesic flow on a
  sphere}}.
\newblock {J. Geom. Phys.} \textbf{57} (2007)(10):2049--2064.
\newblock \doi{10.1016/j.geomphys.2007.05.003}

\bibitem[{Len}08]{Len08}
{Lenells}, J.
\newblock \emph{{The Hunter-Saxton equation: A geometric approach}}.
\newblock {SIAM J. Math. Anal.} \textbf{40} (2008)(1):266--277.
\newblock \doi{10.1137/050647451}

\bibitem[LT71]{LaT71}
{Lindenstrauss}, J. and {Tzafriri}, L.
\newblock \emph{{On the complemented subspaces problem}}.
\newblock {Isr. J. Math.} \textbf{9} (1971):263--269.
\newblock \doi{10.1007/BF02771592}

\bibitem[LV07]{LaV07}
{Lyons}, T. and {Victoir}, N.
\newblock \emph{{An extension theorem to rough paths}}.
\newblock {Ann. Inst. Henri Poincar\'e, Anal. Non Lin\'eaire} \textbf{24}
  (2007)(5):835--847.
\newblock \doi{10.1016/j.anihpc.2006.07.004}

\bibitem[{Lyo}98]{Lyo98}
{Lyons}, T.~J.
\newblock \emph{{Differential equations driven by rough signals}}.
\newblock {Rev. Mat. Iberoam.} \textbf{14} (1998)(2):215--310.
\newblock \doi{10.4171/RMI/240}

\bibitem[LZ21]{LDaZ21}
{Le Donne}, E. and {Z\"ust}, R.
\newblock \emph{{Space of signatures as inverse limits of Carnot groups}}.
\newblock {ESAIM, Control Optim. Calc. Var.} \textbf{27} (2021):14.
\newblock \doi{10.1051/cocv/2021040}.
\newblock Id/No 37

\bibitem[Mac05]{Mackenzie05}
Mackenzie, K. C.~H.
\newblock \emph{General theory of {L}ie groupoids and {L}ie algebroids},
  \emph{London Mathematical Society Lecture Note Series}, vol. 213 (Cambridge
  University Press, Cambridge, 2005)

\bibitem[{Man}08]{Man08}
{Manchon}, D.
\newblock \emph{{Hopf algebras in renormalisation}}.
\newblock In \emph{Handbook of algebra. Volume 5}, pp. 365--427 (Amsterdam:
  Elsevier/Noth-Holland, 2008).
\newblock \doi{10.1016/S1570-7954(07)05007-3}

\bibitem[McA65]{McA65}
McAlpin, J.
\newblock \emph{Infinite Dimensional Manifolds and Morse theory}.
\newblock Ph.D. thesis, Columbia University 1965

\bibitem[Mei17]{Mein17}
Meinrencken, E.
\newblock \emph{{Lie groupoids and Lie algebroids}} 2017.
\newblock
  \urlprefix\url{www.math.toronto.edu/mein/teaching/MAT1341_LieGroupoids/Groupoids.pdf}

\bibitem[Mic80]{Mic}
Michor, P.~W.
\newblock \emph{Manifolds of differentiable mappings}, \emph{Shiva Mathematics
  Series}, vol.~3 (Shiva Publishing Ltd., Nantwich, 1980)

\bibitem[Mic20]{Mic20}
Michor, P.~W.
\newblock \emph{Manifolds of mappings for continuum mechanics}.
\newblock In R.~Segev and M.~Epstein (Eds.), \emph{Geometric Continuum
  Mechanics}, \emph{Advances in Continuum Mechanics}, vol.~42, pp. 3--75
  (Birkh\"auser, 2020).
\newblock \doi{10.1007/978-3-030-42683-5}

\bibitem[Mil82]{Mil82}
Milnor, J.
\newblock \emph{{On infinite dimensional Lie groups}} 1982

\bibitem[Mil84]{Mil84}
Milnor, J.
\newblock \emph{Remarks on infinite-dimensional {L}ie groups}.
\newblock In \emph{Relativity, groups and topology, {II} ({L}es {H}ouches,
  1983)}, pp. 1007--1057 (North-Holland, Amsterdam, 1984)

\bibitem[MM03]{MaM03}
{Moerdijk}, I. and {Mr\v{c}un}, J.
\newblock \emph{{Introduction to foliations and Lie groupoids}}, vol.~91
  (Cambridge: Cambridge University Press, 2003)

\bibitem[MM06]{MaM06}
Michor, P.~W. and Mumford, D.
\newblock \emph{Riemannian geometries on spaces of plane curves}.
\newblock J. Eur. Math. Soc. (JEMS) \textbf{8} (2006)(1):1--48.
\newblock \doi{10.4171/JEMS/37}

\bibitem[MMM13]{MR3098790}
Micheli, M., Michor, P.~W. and Mumford, D.
\newblock \emph{Sobolev metrics on diffeomorphism groups and the derived
  geometry of spaces of submanifolds}.
\newblock Izv. Ross. Akad. Nauk Ser. Mat. \textbf{77} (2013)(3):109--138.
\newblock \doi{10.4213/im7966}

\bibitem[MMS19]{MaMaS19}
Maurelli, M., Modin, K. and Schmeding, A.
\newblock \emph{Incompressible euler equations with stochastic forcing: a
  geometric approach} 2019.
\newblock \eprint{1909.09982}

\bibitem[MN18]{MaN18}
{Marquis}, T. and {Neeb}, K.-H.
\newblock \emph{{Half-Lie groups}}.
\newblock {Transform. Groups} \textbf{23} (2018)(3):801--840.
\newblock \doi{10.1007/s00031-018-9485-6}

\bibitem[MO92]{MR1173211}
{Margalef Roig}, J. and {Outerelo Dom{\'\i}nguez}, E.
\newblock \emph{Differential topology}, \emph{North-Holland Mathematics
  Studies}, vol. 173 (North-Holland Publishing Co., Amsterdam, 1992).
\newblock With a preface by Peter W. Michor

\bibitem[{Mod}19]{modin2019geometric}
{Modin}, K.
\newblock \emph{{Geometric hydrodynamics: from Euler, to Poincar\'e, to
  Arnold}}.
\newblock In \emph{{13th young researchers workshop on geometry, mechanics and
  control. Three mini-courses, Coimbra, Portugal, December 6--13, 2018}}, pp.
  71--91 (Coimbra: Universidade de Coimbra, Departamento de Matem\'atica, 2019)

\bibitem[MP97]{MaP97}
{Moerdijk}, I. and {Pronk}, D.~A.
\newblock \emph{{Orbifolds, sheaves and groupoids}}.
\newblock {\(K\)-Theory} \textbf{12} (1997)(1):3--21

\bibitem[MRT15]{YaRaT15}
{Maeda}, Y., {Rosenberg}, S. and {Torres-Ardila}, F.
\newblock \emph{{The geometry of loop spaces. I: \(H^{s}\)-Riemannian
  metrics}}.
\newblock {Int. J. Math.} \textbf{26} (2015)(4):26.
\newblock Id/No 1540002

\bibitem[MS17]{MaSZ17}
{Murua}, A. and {Sanz-Serna}, J.~M.
\newblock \emph{{Word series for dynamical systems and their numerical
  integrators}}.
\newblock {Found. Comput. Math.} \textbf{17} (2017)(3):675--712.
\newblock \doi{10.1007/s10208-015-9295-3}

\bibitem[MSJ07]{JaMaS07}
Mio, W., Srivastava, A. and Joshi, S.
\newblock \emph{On shape of plane elastic curves}.
\newblock International Journal of Computer Vision \textbf{73} (2007).
\newblock \doi{10.1007/s11263-006-9968-0}

\bibitem[MT20]{MaT19}
{Magnani}, V. and {Tiberio}, D.
\newblock \emph{{A remark on vanishing geodesic distances in infinite
  dimensions}}.
\newblock {Proc. Am. Math. Soc.} \textbf{148} (2020)(8):3653--3656

\bibitem[M{\"u}l08]{MR2463806}
M{\"u}ller, O.
\newblock \emph{A metric approach to {F}r\'echet geometry}.
\newblock J. Geom. Phys. \textbf{58} (2008)(11):1477--1500.
\newblock \doi{10.1016/j.geomphys.2008.06.004}

\bibitem[MV94]{MR1319441}
Michor, P.~W. and Vizman, C.
\newblock \emph{{$n$}-transitivity of certain diffeomorphism groups}.
\newblock Acta Math. Univ. Comenian. (N.S.) \textbf{63} (1994)(2):221--225

\bibitem[MV97]{MaV97}
Meise, R. and Vogt, D.
\newblock \emph{Introduction to functional analysis}, \emph{Oxford Graduate
  Texts in Mathematics}, vol.~2 (The Clarendon Press, Oxford University Press,
  New York, 1997).
\newblock Translated from the German by M. S. Ramanujan and revised by the
  authors

\bibitem[MZ15]{MaZ15}
{Meyer}, R. and {Zhu}, C.
\newblock \emph{{Groupoids in categories with pretopology}}.
\newblock {Theory Appl. Categ.} \textbf{30} (2015):1906--1998

\bibitem[Nee05]{NeeMon}
Neeb, K.-H.
\newblock \emph{{Monastir Summer School: Infinite-Dimensional Lie Groups}}
  2005.
\newblock 3rd cycle. Monastir (Tunisie)

\bibitem[Nee06]{Neeb06}
Neeb, K.-H.
\newblock \emph{Towards a {L}ie theory of locally convex groups}.
\newblock Jpn. J. Math. \textbf{1} (2006)(2):291--468.
\newblock \doi{10.1007/s11537-006-0606-y}

\bibitem[{Nee}10]{Neeb10}
{Neeb}, K.-H.
\newblock \emph{{On differentiable vectors for representations of infinite
  dimensional Lie groups}}.
\newblock {J. Funct. Anal.} \textbf{259} (2010)(11):2814--2855.
\newblock \doi{10.1016/j.jfa.2010.07.020}

\bibitem[{Omo}74]{Omo74}
{Omori}, H.
\newblock \emph{{Infinite dimensional Lie transformation groups}}, vol. 427
  (Springer, Cham, 1974)

\bibitem[Omo78]{MR0579603}
Omori, H.
\newblock \emph{On {B}anach-{L}ie groups acting on finite dimensional
  manifolds}.
\newblock T\^ohoku Math. J. (2) \textbf{30} (1978)(2):223--250.
\newblock \doi{10.2748/tmj/1178230027}

\bibitem[Omo81]{MR630634}
Omori, H.
\newblock \emph{A remark on nonenlargeable {L}ie algebras}.
\newblock J. Math. Soc. Japan \textbf{33} (1981)(4):707--710.
\newblock \doi{10.2969/jmsj/03340707}

\bibitem[{Pal}57]{Pal57}
{Palais}, R.~S.
\newblock \emph{{A global formulation of the Lie theory of transformation
  groups}}, vol.~22 (Providence, RI: American Mathematical Society (AMS),
  1957).
\newblock \doi{10.1090/memo/0022}

\bibitem[{Pal}68]{Pal68}
{Palais}, R.~S.
\newblock \emph{{Foundations of global non-linear analysis}}.
\newblock {Mathematics Lecture Note Series. New York-Amsterdam: W.A. Benjamin,
  Inc. VII, 131 p. (1968).} 1968

\bibitem[{Poi}01]{poin01}
{Poincar\'e}, H.
\newblock \emph{{Sur une forme nouvelle des \'equations de la m\'ecanique.}}
\newblock {C. R. Acad. Sci., Paris} \textbf{132} (1901):369--371

\bibitem[PS86]{MR900587}
Pressley, A. and Segal, G.
\newblock \emph{Loop groups}.
\newblock Oxford Mathematical Monographs (The Clarendon Press, Oxford
  University Press, New York, 1986).
\newblock Oxford Science Publications

\bibitem[QT11]{QaT11}
{Qian}, Z. and {Tudor}, J.
\newblock \emph{{Differential structure and flow equations on rough path
  space}}.
\newblock {Bull. Sci. Math.} \textbf{135} (2011)(6-7):695--732.
\newblock \doi{10.1016/j.bulsci.2011.07.011}

\bibitem[{Ree}58]{Ree58}
{Ree}, R.
\newblock \emph{{Lie elements and an algebra associated with shuffles}}.
\newblock {Ann. Math. (2)} \textbf{68} (1958):210--220.
\newblock \doi{10.2307/1970243}

\bibitem[{Reu}93]{Reut93}
{Reutenauer}, C.
\newblock \emph{{Free Lie algebras}}, vol.~7 (Oxford: Clarendon Press, 1993)

\bibitem[Rud91]{Rudin}
Rudin, W.
\newblock \emph{Functional analysis}.
\newblock International Series in Pure and Applied Mathematics (McGraw-Hill,
  Inc., New York, 1991), second edn.

\bibitem[{Ryb}95]{Ryb95}
{Rybicki}, T.
\newblock \emph{{Isomorphisms between groups of diffeomorphisms}}.
\newblock {Proc. Am. Math. Soc.} \textbf{123} (1995)(1):303--310

\bibitem[{Ryb}02]{Ryb02}
{Rybicki}, T.
\newblock \emph{{A Lie group structure on strict groups}}.
\newblock {Publ. Math.} \textbf{61} (2002)(3-4):533--548

\bibitem[Sch10]{Schm10}
Schmid, R.
\newblock \emph{Infinite-dimensional {L}ie groups and algebras in mathematical
  physics}.
\newblock Advances in Mathematical Physics  (2010).
\newblock \doi{10.1155/2010/280362}

\bibitem[{Sch}20]{Sch20}
{Schmeding}, A.
\newblock \emph{{The Lie group of vertical bisections of a regular Lie
  groupoid}}.
\newblock {Forum Math.} \textbf{32} (2020)(2):479--489

\bibitem[{Ser}20]{Serg20}
{Sergeev}, A.~G.
\newblock \emph{{In search of infinite-dimensional K\"ahler geometry}}.
\newblock {Russ. Math. Surv.} \textbf{75} (2020)(2):321--367

\bibitem[{Sha}97]{sharpe97}
{Sharpe}, R.~W.
\newblock \emph{{Differential geometry: Cartan's generalization of Klein's
  Erlangen program. Foreword by S. S. Chern}}, vol. 166 (Berlin: Springer,
  1997)

\bibitem[{Smo}07]{Smo07}
{Smolentsev}, N.~K.
\newblock \emph{{Diffeomorphism groups of compact manifolds}}.
\newblock {J. Math. Sci., New York} \textbf{146} (2007)(6):6213--6312

\bibitem[Sta08]{stacey2008construct}
Stacey, A.
\newblock \emph{How to construct a {D}irac operator in infinite dimensions}
  2008.
\newblock \eprint{arXiv:0809.3104}

\bibitem[Ste67]{MR0210075}
Steenrod, N.~E.
\newblock \emph{A convenient category of topological spaces}.
\newblock Michigan Math. J. \textbf{14} (1967):133--152

\bibitem[SW15]{SaW}
Schmeding, A. and Wockel, C.
\newblock \emph{The {L}ie group of bisections of a {L}ie groupoid}.
\newblock Ann. Global Anal. Geom. \textbf{48} (2015)(1):87--123.
\newblock \doi{10.1007/s10455-015-9459-z}

\bibitem[SW16]{SaW16}
Schmeding, A. and Wockel, C.
\newblock \emph{({R}e)constructing {L}ie groupoids from their bisections and
  applications to prequantisation}.
\newblock Differential Geom. Appl. \textbf{49} (2016):227--276.
\newblock \doi{10.1016/j.difgeo.2016.07.009}

\bibitem[{Tak}79]{takens79}
{Takens}, F.
\newblock \emph{{Characterization of a differentiable structure by its group of
  diffeomorphisms}}.
\newblock {Bol. Soc. Bras. Mat.} \textbf{10} (1979)(1):17--25

\bibitem[Tay11]{MR2744150}
Taylor, M.~E.
\newblock \emph{Partial differential equations {I}. {B}asic theory},
  \emph{Applied Mathematical Sciences}, vol. 115 (Springer, New York, 2011),
  second edn.
\newblock \doi{10.1007/978-1-4419-7055-8}

\bibitem[{Tho}42]{Tomp42}
{Thompson}, D.~W.
\newblock \emph{{On growth and form}}.
\newblock {New edition Cambridge University Press, Cambridge, England. X, 1116
  p. (1942).} 1942

\bibitem[Tre06]{trev06}
Treves, F.
\newblock \emph{Topological vector spaces, distributions and kernels} (Dover
  Publications, Inc., Mineola, NY, 2006).
\newblock Unabridged republication of the 1967 original

\bibitem[TZ20]{TaZ20}
{Tapia}, N. and {Zambotti}, L.
\newblock \emph{{The geometry of the space of branched rough paths}}.
\newblock {Proc. Lond. Math. Soc. (3)} \textbf{121} (2020)(2):220--251.
\newblock \doi{10.1112/plms.12311}

\bibitem[{Viz}08]{Viz08}
{Vizman}, C.
\newblock \emph{{Geodesic equations on diffeomorphism groups}}.
\newblock {SIGMA, Symmetry Integrability Geom. Methods Appl.} \textbf{4}
  (2008):paper 030, 22.
\newblock \doi{10.3842/SIGMA.2008.030}

\bibitem[{Voi}92]{Voi92}
{Voigt}, J.
\newblock \emph{{On the convex compactness property for the strong operator
  topology}}.
\newblock {Note Mat.} \textbf{12} (1992):259--269

\bibitem[Wal12]{MR2952176}
Walter, B.
\newblock \emph{Weighted diffeomorphism groups of {B}anach spaces and weighted
  mapping groups}.
\newblock Dissertationes Math. \textbf{484} (2012):128.
\newblock \doi{10.4064/dm484-0-1}

\bibitem[Wer00]{MR1787146}
Werner, D.
\newblock \emph{Funktionalanalysis} (Springer-Verlag, Berlin, 2000), extended
  edn.

\bibitem[Whi34]{Whi34}
Whitney, H.
\newblock \emph{Analytic extensions of differentiable functions defined in
  closed sets}.
\newblock Trans. Amer. Math. Soc. \textbf{36} (1934)(1):63--89.
\newblock \doi{10.2307/1989708}

\bibitem[Woc06]{wockel}
Wockel, C.
\newblock \emph{Infinite-Dimensional Lie Theory for Gauge Groups}.
\newblock Ph.D. thesis, TU Darmstadt 2006

\bibitem[Woc07]{MR2353707}
Wockel, C.
\newblock \emph{Lie group structures on symmetry groups of principal bundles}.
\newblock J. Funct. Anal. \textbf{251} (2007)(1):254--288.
\newblock \doi{10.1016/j.jfa.2007.05.016}

\bibitem[Woc14]{WocInf}
Wockel, C.
\newblock \emph{Infinite-dimensional and higher structures in differential
  geometry} 2014.
\newblock Unpublished lecture notes

\bibitem[{Wur}95]{Wurz95}
{Wurzbacher}, T.
\newblock \emph{{Symplectic geometry of the loop space of a Riemannian
  manifold}}.
\newblock {J. Geom. Phys.} \textbf{16} (1995)(4):345--384.
\newblock \doi{10.1016/0393-0440(94)00033-Z}

\bibitem[{You}36]{Young}
{Young}, L.~C.
\newblock \emph{{An inequality of the H\"older type, connected with Stieltjes
  integration}}.
\newblock {Acta Math.} \textbf{67} (1936):251--282.
\newblock \doi{10.1007/BF02401743}

\end{thebibliography}

\printindex
\end{document}